\documentclass[12pt,a4paper,twoside]{report}
\usepackage{fancyheadings}
\usepackage{supertabular}
\usepackage[dvips]{graphicx}
\usepackage{subfigure}
\usepackage{rotating}
\usepackage{mathrsfs}
\usepackage{geometry}
\usepackage{footnote}
\usepackage{amssymb,latexsym}
\usepackage{amsmath,amsthm}
\usepackage[sort&compress,comma,square,numbers]{natbib}
\renewcommand{\chaptermark}[1]%
        {\markboth{#1}{}}
\renewcommand{\sectionmark}[1]%
	{\markright{\thesection\ #1}}
\newcommand{\chapcleardoublepage}{\newpage{\pagestyle{empty}\cleardoublepage}}

\newtheorem{theo}{Theorem}[section]
\newtheorem{prop}{Proposition}[section]

\newtheorem{rem}{Remark}[section]
\newtheorem{lem}{Lemma}[section]

\newtheorem{guess8}{Example}[section]
\newcommand{\nuw}{\mbox{\boldmath$\nu$}}
\newcommand{\taux}{\mbox{\boldmath$\tau$}}
\newcommand{\etav}{\mbox{\boldmath$\eta$}}
\newcommand{\zetau}{\mbox{\boldmath$\zeta$}}

\newcommand{\Phiy}{\mbox{\boldmath$\Phi$}}
\newcommand{\wa}{\mbox{\boldmath$w$}}
\newcommand{\nb}{\mbox{\boldmath$n$}}
\newcommand{\mc}{\mbox{\boldmath$m$}}
\newcommand{\hd}{\mbox{\boldmath$h$}}
\newcommand{\fe}{\mbox{\boldmath$f$}}
\newtheorem*{prop2.2.1}{Proposition 2.2.1}
\newtheorem*{prop2.2.2}{Proposition 2.2.2}
\newtheorem*{prop2.2.3}{Proposition 2.2.3}
\newtheorem*{prop2.2.4}{Proposition 2.2.4}
\newtheorem*{thm2.3.1}{Theorem 2.3.1}
\newtheorem*{thm2.3.2}{Theorem 2.3.2}
\newtheorem*{thm2.3.3}{Theorem 2.3.3}
\newtheorem*{lem3.2.6}{Lemma 3.2.6}
\newtheorem*{lem3.2.8}{Lemma 3.2.8}
\newtheorem*{lem3.2.9}{Lemma 3.2.9}
\newtheorem*{lem3.3.1}{Lemma 3.3.1}
\newtheorem*{thm3.3.1}{Theorem 3.3.1}
\newtheorem*{lem3.4.1}{Lemma 3.4.1}
\newtheorem*{lem3.4.2}{Lemma 3.4.2}
\newtheorem*{lem3.4.3}{Lemma 3.4.3}
\newtheorem*{lem3.4.4}{Lemma 3.4.4}
\newtheorem*{lem3.5.1}{Lemma 3.5.1}
\newtheorem*{lem3.5.2}{Lemma 3.5.2}

\begin{document}

\pagenumbering{roman}
\setcounter{page}{1}

\thispagestyle{empty}
  \begin{flushright}
    {\small ~}
  \end{flushright}
  \begin{center}
    {\LARGE h-p spectral element methods for three dimensional elliptic problems
            on non-smooth domains using\\
            parallel computers}
  \end{center}
  \large
  \vspace{2.00in}
  \begin{center}
    {\Large\bf $\mbox{Akhlaq Husain}^{1}$}
  \end{center}

  \vspace{1.00in}

  \begin{figure}[!ht]
  \centering
  \includegraphics[scale = 0.60]{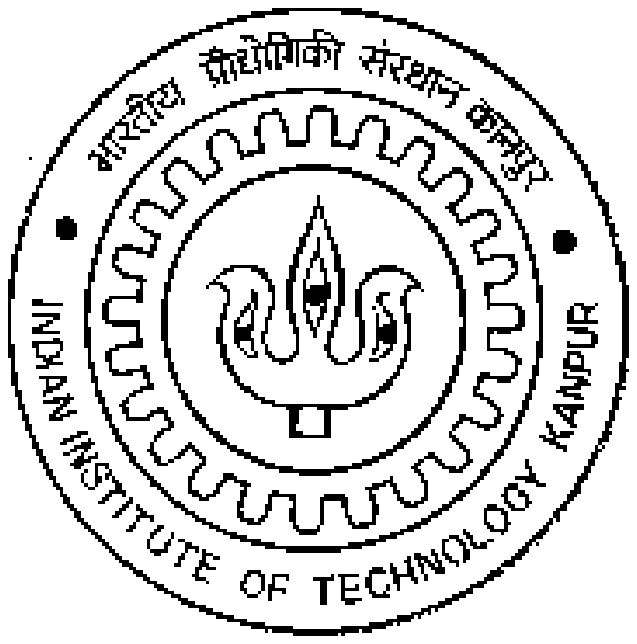}
  \end{figure}
  \vspace{0.1in}
  \begin{center}
   {\normalsize DEPARTMENT OF MATHEMATICS AND STATISTICS} \\
   {\large\bf INDIAN INSTITUTE OF TECHNOLOGY KANPUR} \\
   {\large June, 2010}
  \end{center}
\footnote[1]{This is a reprint of the PhD thesis of the author, which was submitted on $11^{\text{th}}$ June, 2010 and
defended on $14^{\text{th}}$ January, 2011 at IIT Kanpur, India}

\chapcleardoublepage

\thispagestyle{empty}
\begin{titlepage}
\begin{flushright}
    {\small ~}
  \end{flushright}
  \Large
  \begin{center}
    {\LARGE h-p spectral element methods for three dimensional elliptic problems
            on non-smooth domains using\\
            parallel computers}
  \end{center}
  \large
  \vspace{0.60in}
  \begin{center}
    {\large A Thesis Submitted \\
    in Partial Fulfillment of the Requirements \\
    for the Degree of} \\
    \vspace{0.15in}
    {\large\bf Doctor of Philosophy}\\
    \vspace{0.15in}
    {\large\it by} \\
    \vspace{0.15in}
    {\Large\bf $\mbox{Akhlaq Husain}^{2}$}
  \end{center}
  \vspace{0.40in}
  \begin{figure}[!ht]
  \centering
  \includegraphics[scale = 0.60]{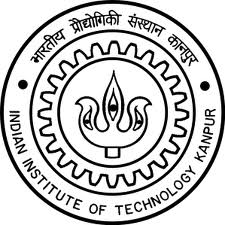}
  \end{figure}
  \vspace{0.15in}
  \begin{center}
    {\large\it to the} \\
    \vspace{0.10in}
    {\normalsize DEPARTMENT OF MATHEMATICS AND STATISTICS} \\
    {\large\bf INDIAN INSTITUTE OF TECHNOLOGY KANPUR}\\
    {\large June, 2010}
  \end{center}
\footnote{This is a reprint of the PhD thesis of the author, which was submitted on $11^{\text{th}}$ June, 2010 and defended
on $14^{\text{th}}$ January, 2011 at IIT Kanpur, India}
\footnote{Dr. Akhlaq Husain\\
LNM Institute of Information Technology\\
Rupa ki Nangal, Post-Sumel, Via-Jamdoli, Jaipur-302031, Rajasthan, INDIA}
\end{titlepage}

\chapcleardoublepage

\thispagestyle{empty}
\clearpage
\renewcommand{\baselinestretch}{1.5}
\begin{figure}[!ht]
\centering
\includegraphics[scale = 0.40]{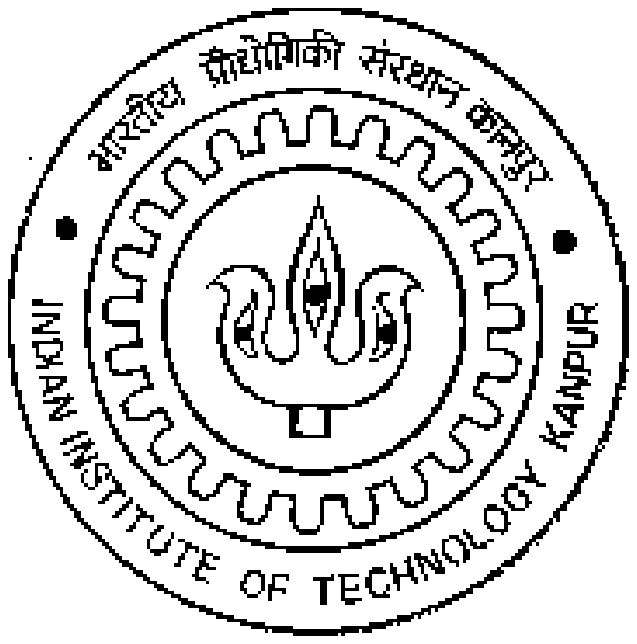}
\end{figure}
\vspace{0.3in}
\begin{center}
\Large{{\bf CERTIFICATE}}
\end{center}
\vspace{0.5in}
\normalsize
\renewcommand{\baselinestretch}{1.5}
\renewcommand{\baselinestretch}{1.5}
It is certified that the work contained in the thesis entitled
``{\bf $h-p$ spectral element methods for three dimensional elliptic
problems on non-smooth domains using parallel computers}'' by
Akhlaq Husain, has been carried out under our supervision and that
this work has not been submitted elsewhere for a degree.
\vspace{1.5in}
\renewcommand{\baselinestretch}{1.5}
\begin{flushleft}
Pravir Dutt~~~~~~~~~~~~~~~~~~~~~~~~~~~~~~~~~~~~~~~A.S. Vasudeva Murthy\\
(Professor)~~~~~~~~~~~~~~~~~~~~~~~~~~~~~~~~~~~~~~~~(Professor)\\
Department of Mathematics~~~~~~~~~~~~~~~~~TIFR Centre for Applicable\\
and Statistics~~~~~~~~~~~~~~~~~~~~~~~~~~~~~~~~~~~~Mathematics\\
Indian Institute of Technology~~~~~~~~~~~~~~Tata Institute of Fundamental Research\\
Kanpur~~~~~~~~~~~~~~~~~~~~~~~~~~~~~~~~~~~~~~~~~~~~Bangalore
\end{flushleft}
\vspace{0.5in}
\begin{flushleft}
June, $2010$ \\
\end{flushleft}
\renewcommand{\baselinestretch}{1.5}
\chapcleardoublepage
\clearpage

\chapcleardoublepage

\thispagestyle{empty}
{\huge{ \bf{ABSTRACT}}}
\vspace{1.0cm}

Elliptic partial differential equations arise in many fields of science and engineering such as
steady state distribution of heat, fluid dynamics, structural/mechanical engineering, aerospace
engineering, medical science and seismology etc.

In three dimensions it is well known that the solutions of elliptic boundary value problems have
singular behaviour near the corners and edges of the domain. The singularities which arise are
known as vertex, edge, and vertex-edge singularities. Due to the presence of singularities the
conventional numerical methods are unable to provide accurate numerical solutions and the rate of
convergence of these methods degrades. In order to improve efficiency of computations and accuracy
of the solutions, it is desirable to find efficient methods along with standard numerical
techniques such as finite element method (FEM), spectral element method (SEM) and so on.

We propose a nonconforming $h-p$ spectral element method to solve three dimensional elliptic
boundary value problems on non-smooth domains to exponential accuracy.

To overcome the singularities which arise in the neighbourhoods of the vertices, vertex-edges
and edges we use local systems of coordinates. Away from
these neighbourhoods standard Cartesian coordinates are used. In each of these neighbourhoods
we use a \emph{geometrical mesh} which becomes finer near the corners and edges. The geometrical
mesh becomes a \emph{quasi-uniform mesh} in the new system of coordinates. Hence \emph{Sobolev's
embedding theorems} and the \emph{trace theorems} for Sobolev spaces are valid for spectral
element functions defined on mesh elements in the new system of variables with a uniform constant.
We then derive \emph{differentiability estimates} in these new sets of variables and a
\emph{stability estimate}, on which our method is based, for a non-conforming $h-p$ spectral
element method.

We choose as our approximate solution the spectral element function which minimizes the sum of
a weighted squared norm of the residuals in the partial differential equations and the squared
norm of the residuals in the boundary conditions in fractional Sobolev spaces and enforce
continuity by adding a term which measures the jump in the function and its derivatives at
inter-element boundaries in fractional Sobolev norms, to the functional being minimized. The
Sobolev spaces in vertex-edge and edge neighbourhoods are anisotropic and become singular at
the corners and edges.

The method is essentially a \emph{least$-$squares collocation} method and a solution can be obtained
using \emph{Preconditioned Conjugate Gradient Method (PCGM)}. To solve the minimization problem we
need to solve the \emph{normal equations} for the \emph{least$-$squares} problem. The residuals in
the normal equations can be obtained without computing and storing \emph{mass} and \emph{stiffness}
matrices.

We solve the normal equations using a block diagonal preconditioner where each block corresponds to
the square of $H^{2}$ norm of the SEF defined on a particular element. Moreover it is shown that there
exists a diagonal preconditioner using separation of variables technique. Let $N$ denote the number of
refinements in the geometrical mesh. We shall assume that $N$ is proportional to $W$.

For problems with Dirichlet boundary conditions the condition number of the preconditioned system
is $O((lnW)^2)$, provided $W=O(e^{N^{\alpha}})$ for $\alpha<1/2$. Moreover there exists a new
preconditioner which can be diagonalized in a new set of basis functions, using separation of
variables techniques, in which each diagonal block corresponds to a different element, and hence
it can easily be inverted on each element. For Dirichlet problems the method requires $O(NlnN)$
iterations of the PCGM to obtain solution to exponential accuracy and it requires $O(N^{5}ln(N))$
operations on a parallel computer with $O(N^{2})$ processors to compute the solution. For mixed
problems with Neumann and Dirichlet boundary conditions the condition number of the preconditioned
system is $O(N^4)$, provided $W=O(e^{N^{\alpha}})$ for $\alpha<1/2$. Hence, it requires $O(N^{3})$
iterations of the PCGM to obtain solution to exponential accuracy and requires $O(N^{7})$ operations
on a parallel computer with $O(N^{2})$ processors to compute the solution.

Computational results for a number of model problems confirm the theoretical estimates obtained
for the error and computational complexity.

\chapcleardoublepage

\thispagestyle{empty}
\chapter*{Synopsis}
\addcontentsline{toc}{chapter}{Synopsis }
\noindent
\begin{center} 
\begin{tabular}{lcl}
\hline
Name of the Student & : & {\bf Akhlaq Husain}
\\
Roll Number & : & {\bf Y4108061}
\\
Degree for which submitted  & : & {\bf Ph.D.}
\\
Department & : & {\bf Mathematics}
\\
Thesis Title & : & {\bf $h-p$ Spectral Element Methods for}
\\
             &  & {\bf Three Dimensional Elliptic Problems}
\\
             &  & {\bf on Non-smooth Domains using}
\\
             &  & {\bf Parallel Computers}
\\
Thesis Supervisors & : & {\bf Dr. Pravir K. Dutt and Dr. A. S.}
\\
             &  & {\bf Vasudeva Murthy}
\\
Month and year of submission & : & {\bf June, 2010}
\\
\hline
\end{tabular}
\end{center}
\vspace{1.0cm}

Elliptic partial differential equations arise in many fields of science and engineering such as
steady state distribution of heat, fluid dynamics, structural/mechanical engineering, aerospace
engineering, medical science and seismology etc. Elliptic boundary value problems in polygonal
and polyhedral domains have been studied in many works in the literature. Among these, problems
on polyhedral domains with corners and edges have become increasingly important in the last two
decades. In many practical situations, the physical domain often has corners and edges either
due to its geometry or created by unions and intersections of simpler objects such as cylinders,
cones and spheres. Hence singularities of the solutions occur at the corners and edges, and
severely affect the regularity of the solutions.

In three dimensions it is well known that the solutions of elliptic boundary value problems have
singular behaviour near the corners and edges of the domain. The singularities which arise are
known as vertex, edge, and vertex-edge singularities. Due to the presence of singularities the
conventional numerical methods are unable to provide accurate numerical solutions and the rate of
convergence of these methods degrades. In order to improve efficiency of computations and accuracy
of the solutions, it is desirable to find efficient methods along with standard numerical
techniques such as finite element method (FEM), spectral element method (SEM) and so on. Different
approaches and methods have been attempted over the years to find accurate solutions to the elliptic
boundary value problems on polyhedrons containing singularities.


We propose a nonconforming $h-p$ spectral element method to solve three dimensional elliptic
boundary value problems on non-smooth domains to exponential accuracy.

To overcome the singularities which arise in the neighbourhoods of the vertices, vertex-edges
and edges we use local systems of coordinates. These local coordinates are modified versions
of spherical and cylindrical coordinate systems in their respective neighbourhoods. Away from
these neighbourhoods standard Cartesian coordinates are used. In each of these neighbourhoods
we use a \emph{geometrical mesh} which becomes finer near the corners and edges. The geometrical
mesh becomes a \emph{quasi-uniform mesh} in the new system of coordinates. Hence \emph{Sobolev's
embedding theorems} and the \emph{trace theorems} for Sobolev spaces are valid for spectral
element functions defined on mesh elements in the new system of variables with a uniform constant.
We then derive \emph{differentiability estimates} in these new sets of variables and a
\emph{stability estimate}, on which our method is based, for a non-conforming $h-p$ spectral
element method.

We choose as our approximate solution the spectral element function which minimizes the sum of
a weighted squared norm of the residuals in the partial differential equations and the squared
norm of the residuals in the boundary conditions in fractional Sobolev spaces and enforce
continuity by adding a term which measures the jump in the function and its derivatives at
inter-element boundaries in fractional Sobolev norms, to the functional being minimized. The
Sobolev spaces in vertex-edge and edge neighbourhoods are anisotropic and become singular at
the corners and edges.

The spectral element functions are represented by a uniform constant at all the corner elements
in vertex neighborhoods and on the corner-most elements in vertex-edge neighbourhoods which are
in the angular direction to the edges. At corner elements which are in the direction of edges
in vertex-edge neighbourhoods and at all the corner elements in edge neighbourhoods the spectral
element functions are represented as one dimensional polynomials of degree $W$ in the modified
coordinates. In all other elements in edge neighbourhoods and vertex-edge neighbourhoods the
spectral element functions are a sum of tensor products of polynomials of degree $W$ in their
respective modified coordinates. The remaining elements in the vertex neighbourhoods and the
regular region are mapped to the master cube and the spectral element functions are represented
as a sum of tensor products of polynomials of degree $W$ in $\lambda_1, \lambda_2$, and $\lambda_3$,
the transformed variables on the master cube.

The method is essentially a \emph{least$-$squares collocation} method and a solution can be obtained
using \emph{Preconditioned Conjugate Gradient Method (PCGM)}. To solve the minimization problem we
need to solve the \emph{normal equations} for the \emph{least$-$squares} problem. The residuals in
the normal equations can be obtained without computing and storing \emph{mass} and \emph{stiffness}
matrices.

We choose spectral element functions (SEF) which are non-conforming. We solve the normal equations
using a block diagonal preconditioner where each block corresponds to the square of $H^{2}$ norm of
the SEF defined on a particular element.
Let $N$ denote the number of refinements in the geometrical mesh. We shall assume that $N$ is
proportional to $W$.

For problems with Dirichlet boundary conditions the condition number of the preconditioned system
is $O((lnW)^2)$, provided $W=O(e^{N^{\alpha}})$ for $\alpha<1/2$. Moreover there exists a new
preconditioner which can be diagonalized in a new set of basis functions, using separation of
variables techniques, in which each diagonal block corresponds to a different element, and hence
it can easily be inverted on each element. For Dirichlet problems the method requires $O(NlnN)$
iterations of the PCGM to obtain solution to exponential accuracy and it requires $O(N^{5}ln(N))$
operations on a parallel computer with $O(N^{2})$ processors to compute the solution. For mixed
problems with Neumann and Dirichlet boundary conditions the condition number of the preconditioned
system is $O(N^4)$, provided $W=O(e^{N^{\alpha}})$ for $\alpha<1/2$. Hence, it requires $O(N^{3})$
iterations of the PCGM to obtain solution to exponential accuracy and requires $O(N^{7})$ operations
on a parallel computer with $O(N^{2})$ processors to compute the solution.



We mention that once we have obtained our approximate solution consisting of non-conforming
spectral element functions we can make a correction to it so that the corrected solution is
conforming and is an exponentially accurate approximation to the actual solution in the $H^1$
norm over the whole domain.

Our method works for non self-adjoint problems too. 
Computational results for a number of model problems on non-smooth domains with constant and variable
coefficients having smooth and singular solutions are presented 
which confirm the theoretical estimates obtained for the error and computational complexity.

For mixed problems rapid growth of the factor $N^{4}$ creates difficulty in parallelizing the numerical
scheme. To overcome this difficulty another version of the method may be defined in which we choose
spectral element functions to be conforming only on the wirebasket of the elements and non-conforming
elsewhere. The values of the spectral element functions at the wirebasket of the elements constitute
the set of common boundary values and an accurate approximation to the Schur complement matrix can be
computed. We plan to consider this in future work.

\chapcleardoublepage

\markboth{\small Table of  Contents}{\small Table of  Contents}
\addcontentsline{toc}{chapter}{Contents }
\thispagestyle{empty}
\tableofcontents
\chapcleardoublepage

\markboth{\small List of Figures }{\small List of Figures }
\addcontentsline{toc}{chapter}{List of Figures }
\thispagestyle{empty}
\listoffigures
\chapcleardoublepage

\markboth{\small List of Tables }{ \small List of Tables}
\addcontentsline{toc}{chapter}{List of Tables}
\thispagestyle{empty}
\listoftables
\chapcleardoublepage

\baselineskip   22.5pt 

\pagenumbering{arabic}

\chapter{Introduction}
Many stationary phenomena in science and engineering are modelled by
elliptic boundary value problems for instance, the elasticity
problem on polyhedral domains. Usually we do not have a closed form
solution of the problem. So we frequently require the numerical
solution of these problems. In structural mechanics the physical
domain often has edges, vertices, cracks and interface between
different materials. It is well known that the solutions of these
problems develop singularities due to the presence of corners and
edges in a three-dimensional domain. In the presence of
singularities, the standard numerical methods such as finite element
method (FEM) and finite difference method (FDM) fail to provide
accurate solutions and efficiency of computations. The situation is
even worse if the singularity is more severe, for example, a
vertex-edge singularity. As a result the approximation becomes
difficult and inefficient and the conventional numerical methods
yield poor convergence results for solutions to these problems. In
order to have reliable and economical approximate solutions with
optimal rate of convergence, it is desirable to find efficient and
accurate numerical techniques.

We propose \textbf{$h-p$ Spectral Element Methods for Three Dimensional Elliptic Problems on
Non-smooth Domains using Parallel Computers}. In this chapter, we give a brief review of the
existing numerical methods for such problems and discuss their computational complexity.

\section{Spectral Methods}
Spectral Element Methods (SEM) are a class of spatial discretizations that can be utilized
for solving partial differential equations. Spectral methods are one of the most accurate
methods for solving partial differential equations, among others namely Finite Difference
Methods (FDM) and Finite Element Methods (FEM).

Spectral methods are considered as higher order finite element methods due to their so called
\textit{spectral/exponential} accuracy and use of high order polynomials for computing numerical
solution. The very high accuracy of spectral methods allows us to treat problems which would
require an enormous number of grid points by finite difference or finite element methods with
much fewer degrees of freedom.

Spectral methods were proposed by Blinova~\cite{BL} and first
implemented by Silberman~\cite{SL}, but abandoned in the mid-1960s.
Orszag~\cite{ORZ} and Eliasen \emph{et al.}~\cite{EMR}
resurrected them again. The formulation of the theory of modern
spectral methods was first presented in the monograph by Gottlieb
and Orszag~\cite{GO} for the numerical solution of partial
differential equations. Multi-dimensional discretizations were
formulated as tensor products of one-dimensional constructs in
separable domains. Since then spectral methods were extended to a
broader class of problems and thoroughly analyzed in the 1980s and
entered the mainstream of scientific computation in the 1990s. The
text book of Canuto \emph{et al.}~\cite{CHQZ1} focuses on fluid
dynamics algorithms and includes both practical as well as
theoretical aspects of global spectral methods. A companion book
\textit{Spectral Methods, Fundamentals in Single Domains} by Canuto
\emph{et al.}~\cite{CHQZ2} is focused on the essential aspects of
spectral methods on separable domains. The book \textit{Spectral/hp
Element Methods for Computational Fluid dynamics}, by Karniadakis
and Sherwin~\cite{KS}, deals with many important practical aspects
of computations using spectral methods and summarizes the recent
research in the subject. In the latest book by Bochev and
Gunzburger~\cite{BOGU} the least-squares finite element method
(LSFEM) for elliptic problems have been described.

The first practitioners of spectral methods were meteorologists studying global weather modelling
and fluid dynamicists investigating isotropic turbulence. The original idea was to use truncated
Fourier series to approximate the (smooth) solution when the problem was specified with (mostly)
periodic boundary conditions. In order to tackle problems with more general boundary conditions
(Dirichlet or Neumann type), the set of (algebraic) polynomials replaced the set of truncated
series, but the characterization of the unique discrete function that would provide the numerical
solution was still achieved following the original strategy.

The key components for spectral methods are the \emph{trial functions} (also called the expansion or
approximating functions) and the \emph{test functions} (also known as weight functions). The trial
functions, which are linear combinations of suitable trial basis functions, are used to provide
the approximate representation of the solution. The test functions are used to ensure that the
differential equation and some boundary conditions are satisfied as closely as possible by the
truncated series expansion. This is achieved by minimizing, with respect to a suitable norm, the
residual produced by using the truncated expansion instead of the exact solution.
For this reason they may be viewed as a special case of the method of weighted residuals
(Finlayson and Scriven, \cite{FIN}). An equivalent requirement is that the residual satisfy
a suitable orthogonality condition with respect to each of the test functions. From this
perspective, spectral methods may be viewed as a special case of Petrov-Galerkin methods
(Zienkiewicz and Cheung \cite{ZIE}, Babu\v{s}ka and Aziz \cite{AZIZ}).

The most frequently used approximation functions (trial functions) are trigonometric polynomials,
Chebyshev polynomials, and Legendre polynomials. Generally, trigonometric polynomials are used
for periodic problems whereas Chebyshev and Legendre polynomials for non-periodic problems.
Laguerre polynomials are used for problems on semi-infinite domains and Hermite polynomials
for problems on infinite domains.

Boyd~\cite{BOY} contains a wealth of detail and advise on spectral algorithms and is an especially
good reference for problems on unbounded domains and in cylindrical and spherical co-ordinate systems.
A thorough analysis of the theoretical aspect of spectral methods for elliptic equations was provided
by Bernardi and Maday~\cite{BER}.

\section{Types of Spectral Methods}
Spectral methods can be broadly classified into two categories: the pseudo-spectral or collocation
methods and the Galerkin methods~\cite{KS}. The choice of the trial functions distinguishes between
the three early versions of spectral methods, namely, the \textit{collocation, Galerkin} and
\textit{tau} versions~\cite{CHQZ2}.

\subsection{Collocation method}
In the collocation approach the test functions are translated Dirac delta-functions centered at
special, so-called \textit{collocation points}. This approach requires the differential equation
to be satisfied exactly at the collocation points. Of course the choice of the set of collocation
points is of fundamental importance for the accuracy of the method and the number of collocation
points must be equal to the dimension of the space of approximation. Otherwise, the problem could,
in general, be over- or under-specified. The collocation points for both the differential equations and the
boundary conditions are usually the same as the physical grid points. The most effective choice
for the grid points are those that correspond to quadrature formulae of maximum precision.

The collocation approach appears to have been first used by Slater~\cite{SLAT} and by
Kantorovic~\cite{KANT} in specific applications. Frazer \emph{et al.}~\cite{FRAZ} developed
it as a general method for solving ordinary differential equations. They used a variety of trial
functions and an arbitrary distribution of collocation points. The work of Lanczos~\cite{LANC}
established for the first time that a proper choice of trial functions and distribution of
collocation points is crucial to the accuracy of the solution. The earliest applications of
the spectral collocation method to partial differential equations were made for spatially
periodic problems by Kreiss and Oliger~\cite{KRIE} (who called it the Fourier method) and
Orszag~\cite{ORSZ2} (who termed it pseudo-spectral). Here we choose different spaces of test
functions and trial functions. We approximate $u$ by
\begin{equation}\label{eq1.2.1}
u_N(x)=\sum_{i=0}^{N} a_i\phi_i(x)
\end{equation}
where ${\{\phi_i\}}_i$ is the space of trial functions.
\newline
We let the space of test functions ${\{\psi_i\}}_i$ be different and impose the orthogonality
condition
\begin{equation}\label{eq1.2.2}
(\mathcal{L} u_{N},\psi_i)=0\qquad\text{for all}\; i.
\end{equation}
Now if we choose $\psi_i=\delta(x-x_i)$ for a suitably chosen set of points ${\{x_i\}}_i$ we
obtain the set of equations
\begin{equation}\label{eq1.2.3}
(\mathcal{L} u_{N})(x_i)=0\qquad\text{for all}\; i.
\end{equation}
The points $x_i$ are chosen as the quadrature points of a Gaussian integration formula.

\subsection{Galerkin method}
The Galerkin approach~\cite{GALER}, enjoys the aesthetically pleasing feature that the trial
and the test functions are the same, and the discretization is derived from a weak form of the
mathematical problem. The test functions are, therefore, infinitely smooth functions that
individually satisfy some or all of the boundary conditions. The differential equation is
enforced by requiring that the integral of the residual times each test function be zero, after
some integration-by-parts, accounting in the process for any remaining boundary conditions. Finite
element methods customarily use this approach. Moreover, the first serious application of spectral
methods to PDE's $-$ that of Silberman~\cite{SL} for meteorological modeling $-$ was a Galerkin
method. The basic idea is to assume that the unknown function $u(x)$ can be approximated by a sum
of $N+1$ basis functions $\phi_{n}(x)$:
\begin{equation}\label{eq1.2.4}
u_N(x)=\sum_{i=0}^{N} a_i\phi_i(x).
\end{equation}
When this series is substituted into the equation
\begin{equation}\label{eq1.2.5}
\mathcal{L}u(x)=f(x)
\end{equation}
where $\mathcal{L}$ is a differential (or integral) operator, the result is the so-called
\textit{residual function} defined by
\begin{equation}\label{eq1.2.6}
\mathcal{R}^{N}(x)=\mathcal{L}u_{N}(x)-f(x).
\end{equation}
Since the residual function $\mathcal{R}^{N}(x;a_{i})$ is identically equal to zero for the exact
solution, the challenge is to choose the series coefficients $\{a_{n}\}$ so that the residual
function is minimized.
\noindent
\subsection{Tau method}
The spectral tau methods, introduced by Lanczos~\cite{LANC}, are similar to Galerkin methods in
the way differential equation is enforced. However, none of the test functions need satisfy the
boundary conditions. Hence, a supplementary set of equations is used to apply boundary conditions.
Tau methods may be viewed as a special case of the so-called Petrov-Galerkin method and are
applicable to problems with nonperiodic boundary conditions.

\section{Spectral Methods on Non-smooth Domains}
We now present a brief summary of the Section \textit{Non-Smooth Domains} in the Chapter
\textit{Diffusion Equation} of~\cite{KS}.

Unlike finite element and finite difference methods, the order of convergence of spectral methods
is not fixed and it is related to the maximum regularity (smoothness) of the solution. Spectral
methods give exponential or \textit{spectral convergence} if the solution is very smooth, i.e.
possessing a high degree of regularity. In practice, exponential convergence implies that as the
number of collocation points is doubled, the error in the numerical solution decreases by at least
two orders of magnitude and not a fixed factor as in low-order methods. However, this fast
convergence is easily lost if the solution has finite regularity or if the domain is irregular.
For example, the solution of a Helmholtz/Poisson's equation may be singular~\cite{G}.

This singularity in the solution may arise due to non-smooth domains or due to discontinuity in the
boundary conditions, or in the specified data (e.g. forcing). Here as in~\cite{G}, we assume that
all the data, as well as the boundary conditions, are smooth/analytic and that singularities are only
due to non-smoothness of the domain. First derivatives are unbounded when the angle is reflexive or
convex, and the second derivatives are unbounded when the angle is acute or obtuse. In this case, not
only the fast convergence of spectral/$hp$ discretization be destroyed, but also the numerical solution
obtained (with any standard method) may be erroneous. In general, theoretical results in three dimensions
for vertex, edge and combined vertex-edge singularities are more difficult to obtain, but work by Guo,
Babu\v{s}ka and others~\cite{BG4,BG5,BG1,BG2,BG3,G1,GOH1} has addressed these issues.

To proceed further, we consider the domain shown in Figure \ref{fig1.0} with the corner located at the
origin and with one side of the corner aligned along $x_{1}$ axis, while the other is at an angle
$\alpha\pi$, $0<\alpha<2$, in the counterclockwise direction. The solution in polar coordinates may be
expressed as
\[u(r,\theta)\propto r^{\beta}\zeta(\theta)\chi(r).\]
Here $\zeta(\theta)$ is an analytic function and $\chi(r)$ is a smooth cut-off function.
\begin{figure}[!ht]
\centering
\includegraphics[scale = 0.55]{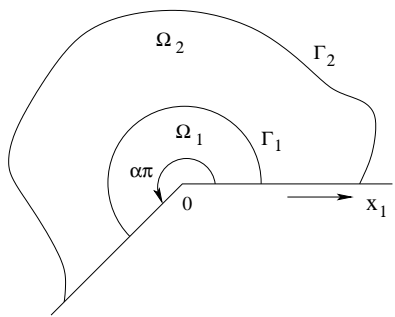}
\caption{Computational domain $\Omega$ containing a corner of angle $\theta=\alpha\pi$.}
\label{fig1.0}
\end{figure}
In this case, it is known that the spectral element solution $\tilde{u}(r,\theta)$ computed with
polynomial order $p$ in each element, satisfies
\[||u(r,\theta)-\tilde{u}(r,\theta)||_{1}\leq Cp^{-2\beta-\epsilon}\:,\]
where $\epsilon>0$ and the exponent $\beta$ depends on the angle. For many problems $\beta=1/\alpha$
(where $\theta=\alpha\pi$), and thus the convergence rate lies between
$\mathcal{O}(p^{-1})$ and $\mathcal{O}(p^{-2})$.

There are three main methods that allow us to recover, if not exponential, at least very high accuracy
for most elliptic problems. They make use of the following:
\begin{enumerate}
\item The method of gradual $h$-refinement.
\item The Method of Auxiliary Mapping (MAM) or Conformal mapping to smooth out the singularity.
\item Use of extra basis functions that contain the form of singularity referred to as
      \textit{space enrichment methods}. These require knowledge of eigenpair representation
      of the solution.
\end{enumerate}

\subsection{Gradual $h$ refinement}
This approach requires the application of good discretization strategy which is usually done using
radial or geometrically refined meshes (see~\cite{BS}) and \textit{quasi-uniform meshing}. A geometric
progression has been found to be effective with a ratio of $(\sqrt{2}-1)^{2}\approx0.17$, independent
of the strength $\beta$ of the singularity~\cite{GB,SZAB}. However, in practice a value of $0.15$ is
typically adopted.

\subsection{Method of auxiliary or conformal mapping}
Babu\v{s}ka and Oh~\cite{BABU} introduced a new approach called the Method of Auxiliary Mapping (MAM),
to deal with two dimensional elliptic boundary value problems containing corner singularities. In two
dimensions the method adopts an auxiliary mapping of the type $\phi(z)=z^{1/\alpha}$ that maps a
region where the solution is singular to a region where the transformed solution has a higher regularity.
Thus, on the mapped domain true solution has better approximation properties. The method gives highly
accurate numerical solutions without destroying the standard band structure of FEM and without increasing
the number of degrees of freedom.
The mapping is determined by the known nature of the singularity in such a way that the exact solution
of lower regularity can be transformed to a function of higher regularity. The MAM, implemented in the
$p$-version of the FEM, has proven to be successful in dealing
with all prominent singularities arising in the two dimensional case~\cite{BABU,OH1,OH2,LUCA}.
Benchmark comparisons with other well known numerical methods, reported in~\cite{LUCA}, show that
MAM is more efficient than other numerical methods.

The MAM was successfully extended and implemented for elliptic PDEs on non-smooth domains in
$\mathbb{R}^{3}$ by Guo and Oh~\cite{GOH1} and Lee \textit{et al.}~\cite{LEE}. However, unlike
the two dimensional case, in $R^{3}$ there are three different types of singularities, the vertex,
the edge and the vertex-edge. The solutions of three dimensional elliptic problems are anisotropic
in the neighbourhoods of edges and vertex-edges. Thus, the MAM techniques in three dimensions are
different from the two dimensional counterpart in theory as well as computation strategies. The
numerical results for the Poisson equations containing the vertex, the edge and the vertex-edge
singularities provided in~\cite{GOH1,LEE} show the effectiveness of the MAM in three dimensions.

We will now describe the method of auxiliary mapping for three different equations, namely
\textit{Laplace, Poisson} and \textit{Helmholtz} equation in two and three dimensional domains.
\newline
{\bf{\textit{Laplace equation in two-dimensional domains}}}
\newline
\noindent
Let us consider Laplace's equation $\Delta u=0$ with homogeneous boundary conditions, i.e. $u=0$
on $\partial\Omega$. In general, in the neighbourhood of the corner, the solution can be expressed
as~\cite{KOND}
\begin{equation}\label{eq1.7}
u(r,\theta)=\sum_{k=0}^{\infty}a_{k}\phi_{k}(r,\theta)
\end{equation}
where the coefficients $a_{k}$ are determined by the boundary conditions and
\begin{align}
\phi_{k}(r,\theta)&=\left\{\begin{array}{ccc}
r^{k/\alpha}\sin\left(\frac{k}{\alpha\theta}\right),\qquad \frac{k}{\alpha}
\:\mbox{is not an integer}\\
r^{k/\alpha}\left[\ln r\sin\left(\frac{k}{\alpha\theta}\right)
+\theta\cos\left(\frac{k}{\alpha}\right)
\right], \quad \frac{k}{\alpha}\:\mbox{is an integer}\:.
\end{array}\right.
\end{align}
For homogeneous Neumann boundary conditions the sin term is replaced by $\cos\left(k/\alpha\theta\right)$,
and for a problem with Dirichlet condition on one side and Neumann condition on the other side, it is
replaced by $\sin\left\{(k/\alpha)(\theta/2)\right\}$. Assuming that the logarithmic term does not
contribute to the solution, the mapping $z=\xi^{\alpha}$, where $z=re^{\dot{\iota}\theta}$, $\alpha=\pi/\omega$
(here $\omega$=sectoral angle), shown in Figure \ref{fig1.1},
makes the transformed solution analytic in terms of the new variables, thus
\begin{figure}[!ht]
\centering
\includegraphics[scale = 0.60]{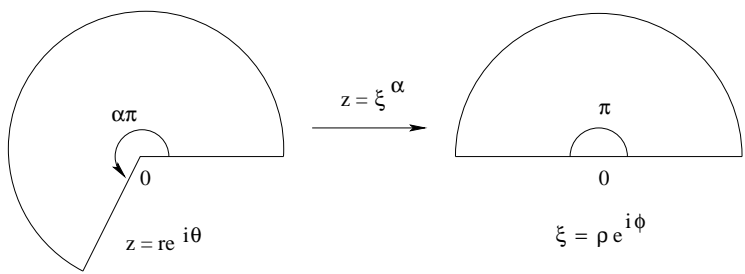}
\caption{Auxiliary mapping of a domain containing a corner to a domain with no corners.}
\label{fig1.1}
\end{figure}
\begin{equation}
u(r,\theta)=\sum_{k=0}^{\infty}a_{k}r^{k/\alpha}\sin\left(\frac{k}{\alpha\theta}\right)
\mapsto u(\rho,\phi)=\sum_{k=0}^{\infty}a_{k}\rho^{k}\sin(k\phi).\notag
\end{equation}
For the spectral/$hp$ element discretization this method was first implemented in~\cite{BABU}.
\newline
{\bf{\textit{Laplace equation in three-dimensional domains}}}
\newline
The solution of the Laplace equation in three dimensions, in the vicinity of the singularities,
can be decomposed into three different forms, depending whether it is in the neighbourhood of a
\textit{vertex}, an \textit{edge} or \textit{vertex-edge} (an intersection of the vertex and edge).
A three dimensional domain $\Omega$, is shown in Figure \ref{fig1.2}, which contains typical
three-dimensional singularities.
\begin{figure}[!ht]
\centering
\includegraphics[scale = 0.70]{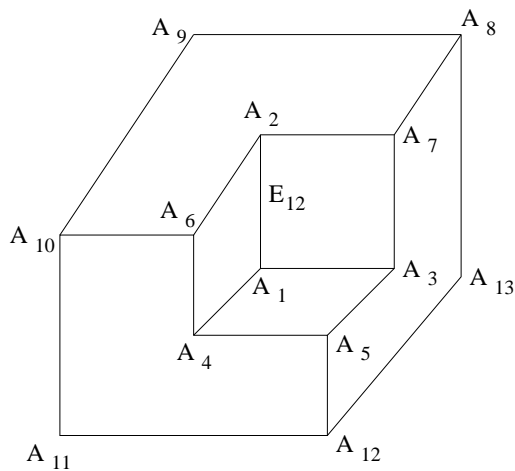}
\caption{Typical three-dimensional singularities.}
\label{fig1.2}
\end{figure}
Vertex singularities arise in the neighbourhoods of the vertices $A_{i}$, and edge singularities
arise in the neighbourhoods of the edges $E_{ij}$ ($E_{ij}$ is the edge joining vertices $A_{i}$
and $A_{j}$). Close to the vertex-edge intersection, vertex-edge singularities arise.
We assume that $\Omega$ contains only straight edges, for curved edges, we refer to~\cite{COST}.
\newline
\textbf{Neighbourhoods of singularities caused by non-smoothness of domain}
\newline
First, we give a brief introduction to the neighbourhoods of vertices, edges and vertex-edges.
We shall give a complete description of each of them in Chapter $2$. Throughout, we shall assume
that, $(x_1,x_2,x_3)$, $(\rho,\theta,\phi)$ and $(r,\theta,z)$ denote the usual Cartesian,
spherical and cylindrical coordinates respectively of points in $\Omega$.

Let $\Gamma_{i}$, $i\in\mathcal{I}=\{1,2,\cdots,I\}$ be the (open) faces,
$E_{ij}=\Gamma_{i}\cap\Gamma_{j}$ be the edges and $A_k$, $k\in\mathcal{K}=\{1,2,\cdots,K\}$ be
the vertices of $\Omega$. We also denote a vertex by $v$ and an edge by $e$. Without loss of
generality, we assume that the vertex $v$ is located at the origin and the edge coincides with
the positive $z$-axis.

By $\Omega^{e}$ we denote an \textit{edge-neighbourhood} of the edge $e=\{(r,\theta,z)|
r=0,0<z<l_{e}\}$ of length $l_{e}$, which is defined by
\begin{equation}
\Omega^{e}=\{(r,\theta,z)\in\Omega| 0<r<\epsilon,\delta_{v}<z<l_{e}-\delta_{v}\}\:.
\end{equation}
Here $\epsilon<1$ and $\delta_{v}<1$ are selected so that
$\Omega^{e}\cap\overline{\Gamma}_{i}=\emptyset$
for those faces $\Gamma_{i}$ which do not contain the edge $e$.

Consider a \textit{neighbourhood} $\Omega_{\rho_{v}}^{v}$ of the vertex $v$(=the origin) defined
by
\[\Omega_{\rho_{v}}^{v}=\{(\rho,\phi,\theta)\in\Omega|\rho<\rho_{v}\}.\]
Here $0<\rho_{v}<1$ is chosen so that $\Omega_{\rho_{v}}^{v}$ has a void intersection with
those edges which do not pass through the vertex $v$. Now we decompose $\Omega_{\rho_{v}}^{v}$
into a vertex neighbourhood and several vertex-edge neighbourhoods. We define, a
\textit{vertex-edge neighbourhood} of the vertex $v$ and edge $e$ as:
\[\Omega^{v-e}=\{(\rho,\phi,\theta)\in\Omega_{\rho_{v}}^{v}|0<\phi<\phi_{v}\}\]
where, $0<\phi_{v}$ and $\rho_{v}<1$ are selected so that
$\overline{\Omega}^{v-e_{1}}\cap\overline{\Omega}^{v-e_{2}}=v$ for distinct edges $e_{1}$ and
$e_{2}$ having $v$ as a common vertex.

A \textit{vertex-neighbourhood}, $\Omega^{v}$, of the vertex $v$ is defined by
\begin{equation}
\Omega^{v}=\Omega_{\rho_{v}}^{v}\setminus\underset{e\in\mathcal{E}^{v}}{\bigcup}\Omega^{v-e}\:,
\end{equation}
where, $\mathcal{E}^{v}$ denote the set of all edges passing through the vertex $v$.
\newline
\textbf{Intensities of vertex, edge and vertex-edge singularities}
\newline
We now give a short outline of the singular solution decomposition in the neighbourhood of a
vertex, an edge, or at a vertex-edge intersection. For corresponding computational results
see~\cite{GOH1,LEE,YOSI2}. For simplicity, we assume that in the vicinity of vertices or edges
of interest, homogeneous boundary conditions are imposed. The following asymptotic expansions of
the true solution are known:
\newline
{\bf{\textit{Vertex singularities}}}
\newline
In a vertex neighbourhood $\Omega^{v}$ of the vertex $v$, the solution has the vertex singularity
and admits an expansion of the form
\begin{equation}
u(\rho,\phi,\theta)=\sum_{l=1}^{L}a_{l}\rho^{\gamma_{l}}h_{l}(\phi,\theta)+v(\rho,\phi,\theta)\:,
\end{equation}
where $(\rho,\phi,\theta)$ denotes usual spherical coordinates with origin at the vertex $v$,
$v(\rho,\phi,\theta)\in H^{2}(\Omega^{v})$, $h_{l}(\phi,\theta)$ are analytic functions of $\phi$
and $\theta$ and are referred to as vertex eigenfunctions, and $\gamma_{l}$ is a positive real
number (see~\cite{BEAW,GRIS1,G}). In other words, in $\Omega^{v}$, the true solution has a
singularity of type $\rho^{\lambda_{1}}$, where $\lambda_{1}<1$.
\newline
{\bf{\textit{Edge singularities}}}
\newline
In an edge neighbourhood $\Omega^{e}$ of the edge $e$, the solution has the edge singularity and
can be decomposed as follows:
\begin{equation}
u(r,\theta,z)=\sum_{m=1}^{M}a_{m}(z)r^{\beta_{m}}h_{m}(\theta)+v(r,\theta,z)\:,
\end{equation}
where $(r,\theta,z)$ denotes usual cylindrical coordinates, and the functions $a_{m}(z)$ are
analytic in $z$. The functions $v(r,\theta,z)\in H^{2}(\Omega^{e})$, $h_{m}(\theta)$ are also
analytic and $h_{m}(\theta)$ are referred to as edge eigenfunctions (see~\cite{BEAW,GRIS1,WALK}).
In other words, in $\Omega^{e}$, the true solution has a singularity of type $r^{\lambda_{2}}$,
where $\lambda_{2}<1$.
\newline
{\bf{\textit{Vertex-edge singularities}}}
\newline
The most complicated decomposition of the solution arises in the case of a vertex-edge intersection.
For example, let us consider the neighbourhood where the edge $e$ approaches the vertex $v$. 
The resulting vertex-edge neighbourhood is denoted by $\Omega^{v-e}$, and the solution has a
singularity caused by the combination of edges and a vertex and has an expansion of the form
\begin{align}
u(\rho,\phi,\theta) & =\sum_{m=1}^{M}\left(\sum_{l=1}^{L}a_{ml}\rho^{\gamma_{l}}+f_{m}(\rho)\right)
(\sin\phi)^{\beta_{m}}g_{m}(\theta) \notag \\
& +\sum_{l=1}^{L}a_{l}\rho^{\gamma_{l}}h_{l}(\phi,\theta)+v(\rho,\phi,\theta)\:,
\end{align}
where the functions $f_{m}(\rho)$ are analytic in $\rho$, and $g_{m}(\theta)$, $h_{l}(\phi,\theta)$
and $v(\rho,\phi,\theta)\in H^{2}(\Omega^{v-e})$ are also analytic (see~\cite{GRIS1,G,GRIS2,WALK}).
That is, in $\Omega^{v-e}$, the true solution has a singularity of type
$\rho^{\lambda_{3}}(\sin\phi)^{\lambda_{4}}$, where $\lambda_{3}<1$ and $\lambda_{4}<1$.
\newline
{\bf{\textit{Poisson equation}}}
\newline
The situation becomes more complicated when a forcing function is introduced as in Poisson's
equation
\[-\Delta u=f(x)\:.\]
The convergence achieved through auxiliary mapping is no better than algebraic because the
solution may not be analytic after the mapping. Typically, we decompose the solution into a
homogeneous part $u^{\mathcal{H}}$, which has the singularity, and a particular part
$u^{\mathcal{P}}$, which depends on the forcing. Complications arise due to particular part
since, even if it is smooth in the original domain, it may be singular after the transformation.
For instance, consider the following transformed equation which has a singular forcing function
\begin{equation}
\Delta u=\alpha^{2}\rho^{2\alpha-2}f(\rho,\phi)\:.\notag
\end{equation}
Given the $\mathcal{O}(\rho^{-2\alpha})$ singularity, the spectral element convergence is of
the order $\mathcal{O}(\rho^{-4\alpha-\epsilon})$ for any $\epsilon>0$. In order to enhance
the convergence, we separate the two contributions so that we have an analytic contribution
in the $z$-plane and an analytic contribution in the $\xi$-plane.
\newline
{\bf{\textit{Helmholtz equation}}}
\newline
For the Helmholtz equation
\begin{equation}
\Delta u-\lambda u=f(x), \lambda>0\:,\notag
\end{equation}
the conformal mapping is an effective way of improving convergence,
although exponential convergence cannot be fully recovered. The
auxiliary mapping $z=\xi^{\alpha}$, $\alpha=\omega/\pi$,
$0<\alpha<2$ converts the Helmholtz equation to
\begin{equation}
\Delta u-\lambda\alpha^{2}\rho^{2\alpha-2}u=\alpha^{2}\rho^{2\alpha-2}f\:.\notag
\end{equation}
In terms of the original variables, the solution around the corner is
\begin{equation}
u(r,\theta)=\sum_{k=0}^{\infty}a_{k}I_{k/\alpha}(\sqrt\lambda r)\sin\left(\frac{k}{\alpha\theta}\right),
\notag
\end{equation}
for $k/\alpha$ not an integer, where $I(z)$ represent the modified Bessel function of the first kind.
After application of the mapping, the solution has the form
\begin{equation}
u(r,\theta)=\sum_{k=0}^{\infty}a_{k}\rho^{k}\sin(k\theta)\sum_{j=0}^{\infty}c_{j}\rho^{2j\alpha}\:,\notag
\end{equation}
with a leading singular term of order $\rho^{1+2\alpha}$. Therefore, the estimated convergence
rate is $\mathcal{O}(\rho^{-2-4\alpha-\epsilon})$, which in practise, is adequately fast, though
algebraic.

Similar to the Poisson equation, in three dimension there are three
different types of singularities namely, the vertex, the edge and
the combined vertex-edge. Furthermore, solutions of elliptic
problems are anisotropic in the neighbourhood of edges and
vertex-edges. However, it is possible to obtain explicitly the form
of such singularities~\cite{G1}, and thus the auxiliary mapping
technique can be effectively used in three dimensions. The
difference is that specific auxiliary mappings are required to
handle each type of singularity. Lee \textit{et al.}~\cite{LEE} have
obtained and implemented these mappings successfully for treating
all three types of singularities. The effectiveness of this method
was demonstrated on two polyhedral domains. The results with the
mapping method were superior to those obtained with the $p$-version
of FEM~\cite{SZAB} or other low-order methods.
\newline
{\bf{\textit{Singular Basis}}}
\newline
An alternative approach to using the auxiliary mapping is to use a set of supplementary basis functions
which have the leading behavior of the singularity in conjunction with the smooth basis $\Phi_{k}(x)$.
For the Helmholtz equation the leading-order singular terms are
\[r^{1/\alpha}, r^{2/\alpha} (\alpha>1/2), r^{3/\alpha} (\alpha>1),\cdots,\]
which can be included into the expansion basis. However, we can do even better by supplementing the
standard basis in the \textit{mapped} domain. The transformed solution is then
\begin{equation}
u(\rho,\phi)=\sum_{k=1}^{\infty}a_{k}I_{k/\alpha}(\sqrt\lambda\rho^{\alpha})\sin(k\phi)
=\sum_{k=1}^{\infty}\sum_{l=0}^{\infty}b_{kl}\rho^{k+2l\alpha}\sin(k\phi),
\end{equation}
and thus the leading singularities are weaker, i.e.,
\[\rho^{1+2\alpha}, \rho^{2+2\alpha} (\alpha>1/2), \rho^{1+4\alpha},\cdots\:.\]
A comparison of the effect of adding supplementary basis terms in the original and the transformed
domain is given in~\cite{KS,PATH}. It has been shown in~\cite{PATH} that with one or two terms included
in the transformed domain, a very fast convergence is obtained. To achieve the exponential accuracy we
need to include higher order terms but in the limiting case the system becomes ill-conditioned as the
augmented basis $\{\Phi_{1},\Phi_{2},\cdots,\Phi_{N},\varphi_{1},\varphi_{2},\cdots,\varphi_{M}\}$
is nearly linearly dependent if many singular functions are used and as a result the stiffness matrix
is ill-conditioned.
\newline
{\bf{\textit{Eigenpair representation: Steklov formulation}}}
\newline
A more recent method that treats singular solutions of both scalar and vector elliptic problems in
the neighbourhood of corners was presented for two-dimensional domains by Yosibash and
Szab\'{o}~\cite{SZAB2,YOSI3} and for three-dimensional domains by Yosibash in~\cite{YOSI1,YOSI2}.
As we have already seen, such singular solutions are characterized by the form $u=r^{\beta}h(\theta)$
close to the corner.

Let us illustrate the basic characteristics of the solution to the
Laplace problem over a two-dimensional domain as shown in Figure
\ref{fig1.1}. The Laplace equation over $\Omega_{2}$, with Dirichlet
boundary conditions on its boundary, is cast in cylindrical
coordinates as follows:
\begin{equation}\label{eq1.15}
\Delta u=\frac{\partial^{2}u}{\partial r^{2}}+\frac{1}{r}\frac{\partial u}{\partial r}
+\frac{1}{r^{2}}\frac{\partial^{2}u}{\partial\theta^{2}}\quad \mbox{in}\quad \Omega_{2}\:.
\end{equation}
The solution in the vicinity of the singular point is sought by separation of variables. We denote by
$h^{+}(\theta)$ and $h^{-}(\theta)$ the functions associated with the positive and negative values of
$\beta$ respectively. Although for the Laplace equation $h^{+}(\theta)\equiv h^{-}(\theta)$, for a
general elliptic equation this is not the case. Thus, the solution to (\ref{eq1.15}) admits the expansion
\begin{equation}\label{eq1.16}
u=\sum_{i=1}^{\infty}a_{i}r^{\beta_{i}}h_{i}^{+}(\theta),\: h_{i}^{+}(\theta)=\sin(\beta_{i}\theta),
\:\mbox{with}\:\: \beta_{i}=\frac{i}{\alpha}\:,
\end{equation}
where $\beta_{i}$ and $h_{i}(\theta)$ denote \textit{eigenpairs},
and these are determined uniformly by the geometry and boundary
conditions in the neighbourhood of the singular point. Notice that
if $\beta_{i}<1$ then the corresponding $i$th term in the expansion
(\ref{eq1.16}) for $\nabla u$ is unbounded as $r\rightarrow 0$. We
say that $u$ is singular at $0$ if $\nabla u$ tends to infinity as
$r\rightarrow 0$. The solution $u$ in (\ref{eq1.16}) is therefore
singular at $0$ if $\alpha>1$ for $i=1,\ldots,\emph{finitely many}$.
The coefficients $a_{i}$ depend on the boundary conditions away from
the singular points as well as the forcing term in the Poisson
equation.

For general singular points, analytical computation of eigenpairs is not practical, and numerical
approximations are usually sought. One of the most robust and efficient methods for the computation
of eigenpairs is the modified Steklov method, developed by Yosibash~\cite{YOSI2} and Yosibash and
Szab\'{o}~\cite{YOSI3}.

\subsection{Advantages and applications of spectral methods: A survey}
We now briefly describe important features of spectral methods that should be considered in
their formulation and application~\cite{GO}.
\begin{enumerate}
\item \textit{Rate of convergence} : If the solution to a problem is analytic then a properly
designed spectral method gives exponential accuracy in contrast to FDM and FEM which yield
finite-order rates of convergence. The important consequence is that spectral methods can
achieve high accuracy with little more resolution than is required to achieve moderate accuracy
in FDM and FEM.
\item \textit{Efficiency} : In order to be useful the spectral method should be as efficient
as difference methods with comparable number of degrees of freedom. The development of
collocation (pseudospectral) and Galerkin methods permits spectral methods to be implemented
with comparable efficiency to that of FDM with the same number of independent degrees of freedom.
As the required accuracy increases, the attractiveness of spectral methods increase.
\item \textit{Boundary conditions} : The mathematical features of spectral methods follow very
closely those of the partial differential equation being solved. Thus the boundary conditions
imposed on spectral approximations are normally the same as those imposed on the differential
equation. In contrast, FDM of higher order than the differential equation require additional
``boundary conditions". Many of the complications of finite-order FDM disappear with the
infinite-order-accurate spectral methods.

Another aspect of the treatment of boundary conditions by spectral methods is their high resolution
of boundary layers. For details we refer to~\cite{GO,ORSZ3}.
\item \textit{Bootstrap estimation of accuracy} : It is often possible to estimate the accuracy
of spectral computations by examination of the shape of the spectrum. Thus, in computations of
three-dimensional incompressible flows at high Reynolds numbers, the mean-square vorticity spectrum
must not increase abruptly at large wave numbers. If the vorticity spectrum decreases smoothly to
zero as wave number increases, it is safe to infer that the calculation is accurate. On the other
hand, similar criteria for finite-difference methods can be misleading.
\end{enumerate}
Let us now survey some applications of spectral methods. We shall classify the method according
to geometry and boundary conditions.
\begin{enumerate}
\item \textit{Periodic boundary conditions in Cartesian coordinates} : Here Fourier series should
be used. Spectral methods have been regularly used in three-dimensions to simulate homogeneous
turbulence using high resolution codes. Most operational codes now use pseudospectral methods
because aliasing errors are usually small.
\item \textit{Rigid boundary conditions in Cartesian coordinates} : Use of Chebyshev or Legendre
polynomials is appropriate is this case. Typical applications include numerical studies of turbulent
shear flows and boundary layer transition.
\item \textit{Rigid boundary conditions in Cylindrical geometry} : Chebyshev or Legendre polynomials
should be used in radial, Fourier series in angular, and either Fourier or Chebyshev series in the
axial direction (depending on boundary conditions).
\item \textit{Problems in spherical geometry} : Here surface harmonic expansions, generalized
Fourier series, and ``associated" Chebyshev expansions all have attractive features.
\item $Semi$-$infinite$ $or$ $infinite$ $geometry$ : In this case Chebyshev expansions are best
if the domain can be mapped or truncated to a finite domain without serious error. There are two
cases to be considered: additional boundary conditions may or may not be required at ``infinity".
If additional boundary conditions, such as radiation or outflow boundary conditions, must be
imposed on the truncated domain, then they should also be applied to the spectral method. On
the other hand, if mapping without additional boundary conditions does not introduce a singularity
in the exact equations, no boundary conditions at ``infinity" are required in the spectral
approximation.
\end{enumerate}

\subsection{Limitations of spectral methods}
The drawbacks of spectral methods are four fold.
\begin{enumerate}
\item The main drawback is their inability to handle complex geometries. Different strategies are,
however, possible to overcome this difficulty. The idea to couple domain decomposition techniques
to the spectral discretization has been successful in overcoming this drawback (S. A. Orszag \cite{ORSZ3}).
\item They are usually more difficult to program than finite difference methods.
\item They are more costly per degrees of freedom than other low-order methods.
\end{enumerate}

\section{Finite Element Method}
The Finite Element Method (FEM) is one of the most widely used numerical methods for obtaining
approximate solutions to a large variety of engineering problems. Although originally developed
for numerical solution of complex problems in structural mechanics, it has since been extended
and applied to the broad field of continuum mechanics. Though the term finite element was first
coined by Clough~\cite{CLOU} in $1960$ in a paper on plane elasticity problems, the ideas of
finite element analysis date back much further and can be traced back to the work by Alexander
Hrennikoff (1941) and Richard Courant (1942). In the early 1960s, engineers used the method for
approximate solutions of problems in a wide variety of engineering problems such as stress
analysis, fluid flow, heat transfer, and other areas. The first book on the FEM by Zienkiewicz
and Chung~\cite{ZIE} was published in 1967.

The underlying premise of the method states that a complicated domain can be sub-divided into a
series of smaller regions in which the differential equations are approximately solved. By
assembling the set of equations for each region, the behavior over the entire problem domain is
determined. Each region is referred to as an \textit{element} and the process of subdividing a
domain into a finite number of elements is referred to as \textit{discretization}. Elements are
connected at specific points, called \textit{nodes}, and the assembly process requires that the
solution be continuous along common boundaries of adjacent elements.

When more accuracy is needed, the finite element method has three different versions, namely $h$
version, $p$ version and $h-p$ version. In the $h$ version, mesh size $h$ is reduced and polynomials
of fixed degree $p$ are used to increase accuracy. In the $p$ version mesh size $h$ is kept fixed
while polynomial degree $p$ is raised. In the $h-p$ version the mesh size $h$ is reduced and the
polynomial degree $p$ is raised simultaneously within the elements either uniformly over the entire
computational domain or selectively depending on the resolution requirements.

Finite element method is a variational method of approximation,
making use of global or variational statements of physical problems
and employing the Rayleigh-Ritz-Galerkin philosophy of constructing
co-ordinate functions whose linear combinations represent the
unknown solutions. It took approximately a decade before the method
was recognized as a form of the Rayleigh-Ritz method. The relation
between these two techniques comes from considering the variational
form of a given problem. For instance, the quadratic functional
\begin{equation}\label{eq1.17}
\mathcal{F}({u})=\int_{0}^{1}[p(x)(u'(x))^{2}+q(x)(u(x))^{2}-2f(x)u(x)]dx
\end{equation}
has a minimum with respect to a variation in $u(x)$ given by the Euler equation
\begin{equation}\label{eq1.18}
-\frac{d}{dx}\left(p(x)\frac{du(x)}{dx}\right)+q(x)u(x)=f(x).
\end{equation}
Therefore, instead of solving for (\ref{eq1.18}) to determine $u(x)$, an alternative but equivalent
way is to find $u(x)$ which minimizes the functional (\ref{eq1.17}).

The Rayleigh-Ritz idea approximates the solution by a finite number of functions $u(x)=\sum_{i=1}^{N}
\alpha_{i}\Phi_{i}(x)$ to determine the unknown weights $\alpha_{i}$, which minimize the functional
(\ref{eq1.17}). In the FEM the solution is also approximated by a finite number of functions, which
are typically local in nature as opposed to the global functions used in the Rayleigh-Ritz approach.
However, the starting point for a finite element method is the differential equation (\ref{eq1.18}),
which is formulated into an integral form (also known as Galerkin formulation) so that the problem
reduces to an algebraic system of equations which can be solved numerically. This connection between
the two methods was made when it was realized that the integral form of the FEM was exactly same as
the functional form used in the Rayleigh-Ritz method.

This relation between the FEM and Rayleigh-Ritz method was very
significant and it made the finite element technique mathematically
respectable. A more general formulation is possible using the method
of weighted residuals which leads to the standard Galerkin
formulation.

\subsection{Method of weighted residuals}
The origin of the \textit{Method of Weighted Residuals} (MWR) dates back prior to development of the
finite element method. The method illustrates how the choice of different test (or weight) functions
may be used to produce many commonly used numerical methods for solving differential equations.

Suppose we have a linear differential operator $\mathcal{L}$ acting on a function $u$ as follows:
\[\mathcal{L}(u({\bf x})) = f({\bf x}) \:\mbox{for}\: {\bf x}\in\Omega.\]
We wish to approximate $u$ by a function $\tilde{u}$, which is a linear combination of basis
functions chosen from a linearly independent set. That is,
\begin{equation}\label{eq1.19}
u\approx\tilde{u}=\sum_{j=1}^{n}a_{j}\Phi_{j}.
\end{equation}
Now, when substituted into the differential operator, $\mathcal{L}$, the result of the operations
is not, in general, $f(x)$. Hence an error or residual will exist. Define
\begin{equation}
R({\bf x}) = \mathcal{L}(\tilde{u}({\bf x})) - f({\bf x}).
\end{equation}
The notion in the MWR is to force the residual to zero in some average sense over the domain
$\Omega$. That is,
\begin{equation}
\int_{\Omega} R({\bf x})w_{j}({\bf x})d{\bf x} = 0, j=1,2,\cdots,n
\end{equation}
where the number of test functions $w_{j}({\bf x})$ is exactly equal to the number of unknown
constants $a_{j}$ in $\tilde{u}$. The result is a set of $n$ algebraic equations for the unknown
constants $a_{j}$. Different choices of test (or weight) functions $w_{j}'$s give rise to
different methods. A list of most commonly used test functions and the computational method they
produce is given in Table \ref{tab1.1}. We shall describe two of them namely the Least Squares
method and the Galerkin method below. For further details of these methods we refer to~\cite{KS}.
\begin{table}[!ht]
\caption{Weight functions $w_{j}({\bf x})$ used in the method of residual and the method produced.}
\label{tab1.1}
   \begin{center}
   \begin{tabular}{lcl}
   \hline
    Test/weight function & Type of method \\
   \hline
     $w_{j}({\bf x})=\delta({\bf x}-{\bf x}_{j})$                & Collocation
     \\
     $w_{j}({\bf x})$=
                     $1$,  \mbox{inside}   $\Omega^{j}$  & Finite volume \\
     \hspace{1.4cm}  $0$,  \mbox{outside}  $\Omega^{j}$  & (sub-domain)   \\
     $w_{j}({\bf x})=\frac{\partial R}{\partial \tilde{u}_{j}}$  & Least-squares
     \\
     $w_{j}({\bf x})=\Phi_{j}$                                   & Galerkin
     \\
     $w_{j}({\bf x})=\Psi_{j}\neq\Phi_{j}$                       & Petrov-Galerkin
     \\
    \hline
   \end{tabular}
   \end{center}
\end{table}
\newline
{\bf{\textit{Least-squares method}}}
\newline
The least-squares method originates from the idea of least-squares estimation developed by Gauss.
If the integral of the square of residuals is minimized, the rationale behind the name can be seen.
In other words, a minimum of
\begin{equation}
\int_{\Omega}R^{2}({\bf x})d{\bf x}\: \notag
\end{equation}
is achieved. Therefore the weight functions for the least-squares method are just the derivatives
of the residual with respect to the unknown constants:
\begin{equation}
w_{j}({\bf x})=\frac{\partial R}{\partial a_{j}}\:. \notag
\end{equation}
This formulation using a spectral/$hp$ element discretization has recently increased in popularity
~\cite{JIAN,PROO1,PROO2}.
\newline
{\bf{\textit{Galerkin method}}}
\newline
This method (also known as Bubnov-Galerkin method) may be viewed as a modification of the least-squares
method. Rather than using the derivative of the residual with respect to the unknown $a_{j}$, the
derivative of the approximating function is used. It turns out that the weight/test functions are
same as the trial functions. That is,
\begin{equation}
w_{j}({\bf x})=\frac{\partial \tilde{u}}{\partial a_{j}}=\Phi_{j}({\bf x})\:. \notag
\end{equation}

\subsection{Advantages of FEM}
\begin{enumerate}
\item Finite element methods convert differential equations into matrix equations that are
      \textit{sparse} because only a handful of basis functions are non-zero in a given sub-interval.
      Sparsity of matrix equations allows the stiffness matrix to be stored and inverted easily
      thus saving a lot of computational cost.
\item In multi-dimensional problems, the sub-intervals become triangles (in 2D) or tetrahedra
      (in 3D) which can be fitted to irregularly-shaped geometries such as the shell of an
      automobile etc.
\item The major advantages of the FEM over FDM and other low-order methods are its built-in abilities
      to handle unstructured meshes, a rich family of element choices, and natural handling of boundary
      conditions.
\item FEM can handle a wide variety of engineering problems such as problems in solid mechanics,
      dynamics, fluids, heat conduction and electrostatics.
\end{enumerate}

\subsection{Disadvantages of FEM}
\begin{enumerate}
\item Finite element methods are not well suited for open region problems.
\item They suffer low accuracy (for a given number of degrees of freedom $N$) because each basis
      function is a polynomial of low degree.
\end{enumerate}

\subsection{Spectral vs. finite element methods}
Spectral methods are similar to finite element methods in philosophy; the major difference is that
in finite element methods we choose $\Phi_{n}(x)$ to be \textit{local} functions which are
polynomials of fixed degree and are non-zero only over a couple of sub-intervals. In contrast,
spectral methods use \textit{global} basis functions in which $\Phi_{n}(x)$ is a polynomial
(or trigonometric polynomial) of high degree which is non-zero, except at isolated points, over the
entire computational domain.

Finite element methods convert differential equations into matrix equations that are \textit{sparse}
because only a handful of basis functions are non-zero in a given sub-interval. Spectral methods
generate algebraic equations with full matrices, but in compensation, the higher order of the basis
functions give high accuracy for a given $N$. When fast iterative matrix solvers are used, spectral
methods can be much more efficient than FEM or FDM methods for many classes of problems. However,
they are most useful when the $geometry$ of the problem is fairly smooth and regular.

\textit{Spectral Element Methods} gain the best of both worlds by hybridizing spectral and finite
element methods. The domain is subdivided into elements, as in finite elements, to gain the
flexibility and matrix sparsity of finite elements. At the same time, the degree of the polynomial
$p$ in each sub-domain is sufficiently high to retain the high accuracy and low storage of spectral
methods.

\section{Non-conforming Methods}
The formulations presented so far deal with conforming elements where vertices of adjoining elements
coincide, and correspondingly a $C^{0}$ continuity condition is satisfied at the element interfaces.
However, design over complex domains often require to refine the mesh locally. For example, to resolve
the geometric singularity in a flow past a half-cylinder it is desirable to contain the mesh refinement
locally as needed and not propagate the mesh changes globally. Such local refinement is particularly
useful in increasing the computational efficiency of direct and large eddy simulations of turbulent
flows.

Non-conforming methods can be used to decompose and recompose such
complex domains into sub-domains without requiring the compatibility
between the meshes on the separate components. That is, we will no
longer require that the vertices of the adjoining elements coincide.
Instead, we will develop a framework that allows for arbitrary
connection between elements. An added advantage of this idea is that
the mesh refinement can be imposed selectively on the components
where it is required. We now present a brief summary of commonly
used non-conforming methods which allow for arbitrary connection
between elements and are highly suitable for parallel implementation
(for details see~\cite{KS}):
\newline\newline
\textit{\bf Iterative patching}

This formulation employs geometrically non-conforming elements but maintains $C^{0}$ continuity of
the global polynomial expansion.
\newline\newline
\textit{\bf Constrained approximation}

The method of constrained approximation was introduced by Oden and his associates~\cite{ODEN1,ODEN2,ODEN3}
to deal with geometrically non-conforming discretizations introduced by refinement. The main idea is
to maintain $C^{0}$ continuity across elemental interfaces by modifying the unconstrained basis
functions appropriately. In other words, the approximation space is a constrained space which is a
subset of $H^{1}(\Omega)$ for second order elliptic problems.
\newline\newline
\textit{\bf Mortar element method}

In this method $C^{0}$ continuity is no longer imposed and new weak forms of the problem are developed.
This method was first introduced by Patera and co-workers~\cite{ANAG,MAVR,BERN1}, who coined the term
`\textit{mortar element methods}' because the discretization introduces a set of functions that mortar
the brick-like elements together. The method generalizes the SEM to geometrically nonconforming
partitions, to sub-domains with different resolutions (polynomial degrees) on sub-domain interfaces and
allows for the coupling of variational discretizations of different types in non-overlapping domains,
that is, the non-conformity may be due to geometry, approximation spaces, or both.
\newline\newline
\textit{\bf Discontinuous Galerkin method}

Similar to mortar element method we do not require $C^{0}$ continuity in this method. Although original
application of most discontinuous Galerkin methods (DGM) was in solving hyperbolic problems, more recent
work has led to formulations for parabolic and elliptic problems~\cite{COCK}. In an effort to classify
all contributions made toward the use of discontinuous Galerkin methods for elliptic problems, Arnold
\textit{et al.}, first in~\cite{ARNO1} and then in more generality in~\cite{ARNO2}, published a unified
analysis of discontinuous Galerkin methods for elliptic problems.

\section{$h-p$/Spectral Element Methods on Parallel Computers}
Spectral element methods function very well on massively parallel machines. One can assign a single
large element with a high order polynomial approximation within it to a single processor. A three
dimensional element of degree $N$ roughly has $N^{3}$ internal degrees of freedom, but the number
of grid points on its boundary is $\mathcal{O}(6N^{2})$. It is these boundary values that must be
shared with other elements i.e., other processors, so that the numerical solution is continuous
everywhere. As $N$ increases, the ratio of internal grid points to the boundary grid points increases,
implying that more and more of the computations are internal to the element, and the shared boundary
values become smaller and smaller compared to the total number of unknowns. This in turn implies
inter-processor communication to be small. To do the same calculation with lower order methods, one
would need roughly eight times as many degrees of freedom in three dimensions. That would increase
the inter-processor communication load by at least a factor of four.

An exponentially accurate $h-p$/spectral element method for solving two dimensional general elliptic
problems with mixed Neumann and Dirichlet boundary conditions on non-smooth domains using parallel
computers was proposed in~\cite{DTK1,DT,DKU,DTK2,SKT1,SKT2}. To resolve the singularities which arise
at the corners an auxiliary map of the form $z=\log{\xi}$ is used along with a geometric mesh at the
corners. In a neighbourhood of the corners, modified polar coordinates $(\tau_k,\theta_k)$ are used,
where $\tau_k=\ln r_k$ and $(r_k,\theta_k)$ denote polar coordinates with origin at the vertex $A_k$.
Away from the sectoral neighbourhoods of the corners the usual coordinate system $(x_1,x_2)$ is used.

With this mesh a numerical method was proposed as follows:
\newline
Find a solution which minimizes ``\textit{the sum of the weighted squared norm of the residuals in
the partial differenital equation and the squared norm of the residuals in the boundary conditions
in fractional Sobolev norms and enforce continuity by adding a term which measures the jump in the
function and its derivatives at inter-element boundaries, in appropriate Sobolev norms.}

The method is a least-squares method and the solution can be obtained by solving the normal equations
using the preconditioned conjugate gradient method (PCGM) without computing the mass and stiffness
matrices~\cite{DT,DKU,SKT2}. Let $N$ denote the number of layers in the geometric mesh. Then the
method requires $O(N\ln N)$ iterations of the PCGM to obtain the solution to exponential accuracy.

\section{Review of Existing Work}
The $h-p$ version of the finite element method (FEM) for elliptic problems was proposed by Babu\v{s}ka
and his coworkers in the mid 80ies. They unified the hitherto largely separate developments of fixed
order ``$h$-version FEM" in the sense of Ciarlet, which achieve convergence through reduction of the
mesh size $h$, and the so-called ``\textit{spectral} (or $p$-version) FEM" achieving convergence through
increasing polynomial order $p$ on a fixed mesh. Apart from unifying these two approaches, a key new
feature of $hp$-FEM was the possibility to achieve \textit{exponential convergence} in terms of number
of degrees of freedom $N$. A method for obtaining a numerical solution to exponential accuracy for
elliptic problems in one dimension was first proposed by Babu\v{s}ka and Gui in~\cite{GB} within the
framework of $hp$-FEM. Exponential convergence results were shown for the model singular solution
$u(x)=x^{\alpha}-x\in H^{1}_{0}(\Omega)$ in $\Omega=(0,1)$. Specifically, the error was shown to be
bounded by $e^{-b\sqrt{N_{dof}}}$, $N_{dof}$ is the number of degrees of freedom. This result required
$\sigma$-geometric meshes with a \textit{fixed mesh ratio} $\sigma\in(0,1)$ while the constant $b$ in
the convergence estimate depends on the singularity exponent $\alpha$ as well as on $\sigma$. Among all
$\sigma\in(0,1)$, the optimal value was shown to be $\sigma_{opt}=(\sqrt{2}-1)^{2}\approx 0.17$.

In two dimensions, exponential convergence (an upper bound $Ce^{-b\sqrt[3]{N_{dof}}}$) on the error for
elliptic problems posed on polygonal domains was obtained by Babu\v{s}ka and Guo in the mid 80ies in a
series of landmark papers~\cite{BG6,BG7,BG8}. Key ingredients in the proof were \textit{geometric mesh}
refinements towards the corners and \textit{nonuniform elemental polynomial degrees} which increase
linearly with the elements' distance from corners.
We remark that the proof of \textit{elliptic regularity} results in terms of \textit{countably normed
spaces}, which constitutes an essential component of the exponential convergence proof, has been a major
technical achievement. This problem has also been examined by Karniadakis and Sherwin in~\cite{KS} and
Pathria and Karniadakis in~\cite{KS,PATH} in the frame work of spectral/$h-p$ element methods.
In~\cite{DTK1,DT,DKU,DTK2,SKT1,SKT2} $h-p$ spectral element methods for solving general elliptic problems
to exponential accuracy on polygonal domains using parallel computers were proposed. More recently the
case of elliptic systems was analyzed in~\cite{KDU} for non-conforming spectral element functions. Key
ingredients were use of geometric mesh at the corners and use of auxiliary mapping to remove the
singularity at the corners. Computational results for least-squares $h-p$/spectral element method for
elliptic problems on non-smooth domains with monotone singularities of the type $r^{\alpha}$ and
$r^{\alpha}\log^\delta r$ as well as the oscillating singularities of the type $r^{\alpha}\sin(\epsilon\log r)$
have been obtained in~\cite{KR}.

Starting in the 90ies, efforts to extend the analytic regularity and the $hp$-convergence analysis
of two dimensional problems to three dimensions in the frame work of finite element methods were
undertaken in~\cite{BG1,BG2,BG5,G2} and the references therein which include related works in the
subject. $hp$-version of discontinuous Galerkin finite element method ($hp$-DGFEM) for
second order elliptic problems with piecewise analytic data in three dimensional polyhedral domains
has been analyzed by Sch$\ddot{o}$tzau \textit{et al.}~\cite{SSW1,SSW2}. The method is shown to be
exponentially accurate.

\section{Review and Outline of the Thesis}
As discussed earlier, current formulations of spectral methods to solve elliptic problems on non-smooth
domains allow us to recover only algebraic convergence~\cite{CHQZ2,KS}. The method of auxiliary mapping,
which yields relatively fast convergence makes use of a conformal mapping of the form $\xi=z^{1/\alpha}$
that maps a neighbourhood of the singularity point (where the true solution is singular) onto a domain
where the true solution is smooth and has better approximation properties, i.e. we smooth out the
singularity that occurs at the corners~\cite{KS} using auxiliary map.

In this thesis, we propose an exponentially accurate $h-p$ spectral element method for elliptic
problems on non-smooth polyhedral domains in $\mathbb{R}^{3}$.

In contrast to the two dimensional case, in three dimensions the character of the singularities is
much more complex, not only because of higher dimension but also due to the different nature of the
singularities which are the vertex singularity, the edge singularity and the vertex-edge singularity.
Thus we have to distinguish between the behaviour of the solution in the neighbourhoods of the
vertices, edges and vertex-edges. Unlike the two dimensional case where weighted isotropic spaces
are used, in three dimensions we have to utilize weighted anisotropic spaces because the solution
is smooth along the edges but singular in the direction perpendicular to the edges~\cite{BG5}.
Behaviour of the solution is even more complex at the vertices where the edges are joined together
and the solution is not smooth along the edges too.

To overcome the singularities which arise in the neighbourhoods of the vertices, vertex-edges
and edges we use local systems of coordinates. These local coordinates are modified versions
of spherical and cylindrical coordinate systems in their respective neighbourhoods. Away from
these neighbourhoods standard Cartesian coordinates are used in the regular region of the
polyhedron. In each of these neighbourhoods we use a \emph{geometrical mesh} which becomes finer
near the corners and edges.

With this mesh we choose our approximate solution as the spectral element function which minimizes
the sum of a weighted squared norm of the residuals in the partial differential equations and the
squared norm of the residuals in the boundary conditions in fractional Sobolev spaces and enforce
continuity by adding a term which measures the jump in the function and its derivatives at
inter-element boundaries in fractional Sobolev norms, to the functional being minimized. The Sobolev
spaces in vertex-edge and edge neighbourhoods are anisotropic and become singular at the corners and
edges.

We then derive \textit{differentiability estimates} with respect to these new coordinates in the
neighbourhoods of vertices, edges and vertex-edges and in the regular region of the polyhedron
where we retain standard Cartesian coordinate system.

The spectral element functions are represented by a uniform constant
at all the corner elements in vertex neighborhoods and on the
corner-most elements in vertex-edge neighbourhoods which are in the
direction transverse to the edges of the polyhedron. At corner elements which are in the direction
of edges in vertex-edge neighbourhoods and at all the corner elements in edge neighbourhoods the
spectral element functions are represented as one dimensional polynomials of degree $W$ in the
modified coordinates. In all other elements in edge neighbourhoods and vertex-edge neighbourhoods
the spectral element functions are a sum of tensor products of polynomials of degree $W$ in their
respective modified coordinates. The remaining elements in the vertex neighbourhoods and the regular
region are mapped to the master cube and the spectral element functions are represented as a sum of
tensor products of polynomials of degree $W$ in $\lambda_1,\lambda_2$, and $\lambda_3$, the transformed
variables on the master cube.

A \textit{stability estimate} is then derived for the functional we minimize. We use the stability
estimate to obtain \textit{parallel preconditioners} and \textit{error estimates} for the solution
of the minimization problem. Let $N$ denote the number of refinements in the geometrical mesh. We
shall assume that $N$ in proportional to $W$. Then for problems with Dirichlet boundary conditions
the condition number of the preconditioned system is $O((lnW)^2)$ provided $W=O(e^{N^{\alpha}})$
for $\alpha<1/2$. Moreover it is shown
that there exists a new preconditioner which can be diagonalized in a new set of basis functions
using separation of variables techniques in which each diagonal block corresponds to a different
element, and hence it can easily be inverted on each element. Moreover, if the data is analytic
then the error is \textit{exponentially small} in $N$.

For mixed problems the condition number grows like $O(N^4)$. The
rapid growth of the factor $N^{4}$ creates difficulties in
parallelizing the numerical scheme. To overcome this difficulty
another version of the method may be defined in which we choose our
spectral element functions to be conforming on the wirebasket (union
of vertices and edges) of elements. It can be shown that a probing
parallel preconditioner can be defined for the minimization problem
using the stability estimates which allows the problem to decouple.
We intend to study this in the future.

The method is essentially a \emph{least-squares} collocation method and a solution can be obtained
using \emph{Preconditioned Conjugate Gradient Method (PCGM)}. To solve the minimization problem we
need to solve the \emph{normal equations} for the \emph{least$-$squares} problem. The residuals in
the normal equations can be obtained without computing and storing \emph{mass} and \emph{stiffness}
matrices.

For Dirichlet problems we use spectral element functions which are non-conforming and hence there
are no common boundary values. For problems with mixed boundary conditions the spectral element
functions are essentially non-conforming except that they are continuous only at the wirebasket
of the elements. Hence the cardinality of the set of common boundary values which is equal to the
values of the function at the wirebasket of the elements is much smaller than the cardinality of
the common boundary values for the standard finite element method and so we can compute an accurate
approximation to the Schur complement matrix. In this dissertation we examine the non-conforming
version of the method. The case when the spectral element functions are conforming on the wirebasket
will be examined in future work.

The method requires $O(NlnN)$ iterations of the PCGM to obtain the solution to exponential accuracy
and requires $O(N^{5}ln(N))$ operations on a parallel computer with $O(N^{2})$ processors for Dirichlet
problems. However, for mixed problems it would require $O(N^3)$ iterations of the PCGM to obtain
solution to exponential accuracy and $O(N^{7})$ operations on a parallel computer with $O(N^{2})$
processors.

Our method works for non self-adjoint problems too. Results of numerical simulations for a number of
model problems on non-smooth domains are presented with constant and variable coefficients including
the non self-adjoint case which confirm the theoretical estimates obtained for the error and
computational complexity.

We remark here that once we have obtained our approximate solution, consisting of non-conforming
spectral element functions, we can make a correction to the approximate solution so that the
corrected solution is conforming and the error between the corrected and exact solution is
exponentially small in $N$ in the $H^1$ norm over the whole domain.

We now briefly describe the contents of the thesis which is divided into six chapters.
\begin{description}
\item [Chapter 1] provides brief descriptions of different methods, namely spectral and finite
      element methods, with advantages and disadvantages. Basic properties of these methods
      are discussed. An overview of the existing work is also provided.

\item [Chapter 2] introduces the problem under consideration. We define various neighbourhoods
      of vertices, edges and vertex-edges and then derive the differentiability estimates and
      introduce the function spaces in these neighbourhoods which will be needed in what follows.
      The main stability estimates are stated here.

\item [Chapter 3] gives the proof of the stability theorem on which our method is based.

\item [Chapter 4] presents the numerical scheme which is based on the stability estimates of
      Chapter 2 and error estimates for the approximate solution are obtained.

\item [Chapter 5] provides preconditioning and parallelization techniques. It is shown that a
      preconditioner can be defined for the quadratic form corresponding to the minimization
      problem which allows the problem to decouple. It is also shown that there exists another
      preconditioner which can be diagonalized in a new basis using the separation of variables
      technique. Lastly, it briefly describes the steps to compute the residual.

\item [Chapter 6] gives numerical results to validate the error estimates and bounds on
      computational complexity. We also briefly outline our plan for future work here.
\end{description}

\chapcleardoublepage

\chapter{Differentiability and Stability Estimates}
\section{Introduction}
In this chapter we shall obtain differentiability estimates and stability estimates for a
general elliptic boundary value problem posed on a polyhedral domain with mixed Dirichlet and
Neumann boundary conditions.
To resolve the singularities which arise at the corners and edges namely, the vertex,
vertex-edge and edge singularities we use geometrical meshes which become finer near
corners and edges. 
In the neighbourhood of vertices, vertex-edges and edges we switch to local systems of
coordinates (auxiliary mappings). In doing this the geometrical mesh is reduced to a
quasi-uniform mesh and hence Sobolev's embedding theorems and trace theorems for Sobolev
spaces apply for functions defined on mesh elements in these new variables with a uniform
constant. Away from these neighbourhoods we retain standard Cartesian coordinates in the
regular region of the polydehron. The use of auxiliary mappings together with geometrical meshes
to overcome the singularities along corners and edges allows us to obtain the solution with
exponential accuracy. We remark that the local systems of coordinates which we use are
modified versions of spherical and cylindrical coordinates in the vertex and edge
neighbourhoods respectively and a hybridization of the two coordinate systems in the
vertex-edge neighbourhoods.

We now seek a solution as in~\cite{DT,DTK1,DTK2,SKT1,SKT2} which minimizes the sum of
the squares of a weighted squared norm of the residuals in the partial differential
equation and a fractional Sobolev norm of the residuals in the boundary conditions and
enforce continuity across inter-element boundaries by adding a term which measures the
sum of the squares of the jump in the function and its derivatives at inter-element
boundaries in appropriate Sobolev norms to the functional being minimized. Since the
residuals in the partial differential equation blow up in the neighbourhoods of vertices,
vertex-edges and edges, we have to multiply these residuals by appropriate weights in
various neighbourhoods. Anisotropic Sobolev norms are used in the neighbourhoods of edges
and vertex-edges. All these computations are done using modified system of coordinates
in the neighbourhoods of corners and edges and a global coordinate system elsewhere.

The differentiability estimates are now obtained with respect to these new coordinates
in the neighbourhoods of vertices, edges and vertex-edges and in the regular region where
standard Cartesian coordinate system is used.

We use spectral element functions which are \textit{non-conforming}. Let $N$ denote the
number of layers in the geometric mesh. The spectral element functions are represented
by a uniform constant at all the corner elements in vertex neighborhoods and on the corner-most
elements in vertex-edge neighbourhoods which are in the angular direction from the edges of
the domain. At remaining corner elements which are in the direction of the edges in edge
neighbourhoods and vertex-edge neighbourhoods the spectral element functions are represented
as polynomials which are functions of one variable. On elements away from corners and edges
in edge neighbourhoods and vertex-edge neighbourhoods these spectral element functions are
a sum of tensor products of polynomials in the modified coordinates. The remaining elements
in the vertex neighbourhoods and the regular region are mapped to the master cube $Q$ and the
approximate solution is represented as a sum of tensor products of polynomials of degree $W$
in $\lambda_1,\lambda_2$ and $\lambda_3$, the transformed variables. Here $W$ is chosen
proportional to $N$, the number of layers. A stability estimate is obtained on which our
method is based.

\section{Differentiability Estimates in Modified Coordinates}
We consider an elliptic boundary value problem posed on a polyhedron $\Omega$ in $R^3$
with mixed Neumann and Dirichlet boundary conditions:
\begin{align}\label{eq2.1}
Lw &= F    \mbox{ in} \;\Omega, \notag\\
w &= g^{[0]} \; \mbox{for}\; x \in \Gamma^{[0]}, \notag\\
\left(\frac{\partial w}{\partial\nuw}\right)_A &= g^{[1]}\;\mbox{for}\;x \in \Gamma^{[1]}.
\end{align}
It is assumed that the differential operator
\begin{equation}\label{eq2.2}
Lw(x) = \sum_{i,j=1}^3 - \frac{\partial}{\partial x_i}(a_{i,j}
w_{x_j}) + \sum_{i=1}^3 b_i w_{x_i} + c w
\end{equation}
is a strongly elliptic differential operator which satisfies the \textit{Lax-Milgram}
conditions. Moreover $a_{i,j}=a_{j,i}$ for all $i,j$ and the coefficients of the
differential operator are analytic. Let $\Gamma_i,$ $i \in {\mathcal I} = \{ 1,2,...,I\},$
be the faces of the polyhedron. Let ${\mathcal D}$ be a subset of ${\mathcal I}$
and ${\mathcal N} = {\mathcal I}\setminus{\mathcal D}.$ We impose Dirichlet boundary
conditions on the faces $\Gamma_i$, $i \in {\mathcal D}$ and Neumann boundary conditions
on the faces $\Gamma_j$, $j \in {\mathcal N}$. The data  $F$, $g^{[0]}$ and $g^{[1]}$
are analytic on each open face and $g^{[0]}$ is continuous on
$\underset{i \in {\mathcal D}}\cup\bar{\Gamma}_{i}$.

By $H^m(\Omega)$, we denote the usual Sobolev space of integer order $m\geq0$
furnished with the norm
$$||u||^2_{H^{m}(\Omega)}=\sum_{|\alpha|\leq m}||D^{\alpha}u||^2_{L^{2}(\Omega)}$$
where $\alpha=(\alpha_1,\alpha_2,\alpha_3), |\alpha|=\alpha_1+\alpha_2+\alpha_3$, and
$D^{\alpha}u=D^{\alpha_1}_{x_1}D^{\alpha_2}_{x_2}D^{\alpha_3}_{x_3}u=u_{x_{1}^{{\alpha}_{1}}
x_{2}^{{\alpha}_{2}}x_{3}^{{\alpha}_{3}}}$ is the distributional (weak) derivative of $u$.
As usual, $H^{0}(\Omega)=L^{2}(\Omega), H^{1}_{0}(\Omega)=\{u\in{L^{2}(\Omega)}: Du\in{L^{2}
(\Omega)}, u=0 \:\text{on}\: \partial\Omega\}.$ A seminorm on $H^m(\Omega)$ is given by
$$|u|^2_{H^{m}(\Omega)}=\sum_{|\alpha|= m}||D^{\alpha}u||^2_{L^{2}(\Omega)}.$$

\subsection{Differentiability estimates in modified coordinates in vertex neighbourhoods}
Let $\Omega$ denote a polyhedron in $R^3$, as shown in Figure \ref{fig2.1}.
\begin{figure}[!ht]
\centering
\includegraphics[scale = 0.60]{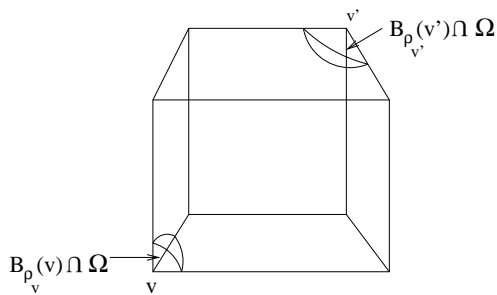}
\caption{Polyhedral domain $\Omega$ in $R^3$.}
\label{fig2.1}
\end{figure}
Let $\Gamma_{i}$, $i \in{\mathcal I} =\{1,2,...,I\},$ be the faces (open),
$S_j$,\,$j \in{\mathcal J} =\{1,2,...,J\},$ be the edges and $A_k$,
$k \in{\mathcal K}=\{1,2,...,K\},$ be the vertices of the polyhedron.

We shall also denote an edge by $e$, where $e \in {\mathcal E}=\{S_1,S_2,...,S_J\}$,
\,the set of edges, and a vertex by $v$ where $v\in{\mathcal V}=\{A_1,A_2,...,A_K\}$,
the set of vertices.
\begin{figure}[!ht]
\centering
\includegraphics[scale = 0.60]{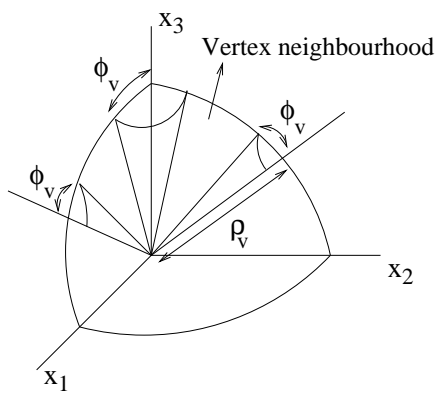}
\caption{Vertex neighbourhood $\Omega^v$.}
\label{fig2.2}
\end{figure}
Let $B_{\rho_v}(v)=\{x: dist\:(x,v)<{\rho_v}\}$. For every vertex $v$, $\rho_v$ is
chosen so small that $B_{\rho_v}(v)\cap B_{\rho_{v^{\prime}}}(v^{\prime})=\emptyset$
if the vertices $v$ and $v^{\prime}$ are distinct.

Now consider a vertex $v$ which has $n_v$ edges passing through it. We shall let the
$x_3$ axis denote one of these edges. Consider first the edge $e$ which coincides with
the $x_3$ axis. Let $\phi$ denote the angle which $x=(x_1,x_2,x_3)$ makes with the
$x_3$ axis. Let
$$\mathcal V_{\rho_v,\phi_v}(v,e)=\{x \in \Omega :\: 0< dist\,(x,v)< {\rho_v},\,
0< \phi < \phi_v\}$$ where $\phi_v$ is a constant. Let us choose $\phi_v$ sufficiently
small so that
$$\mathcal V_{\rho_v,\phi_v}(v,e^{\prime}) \bigcap \mathcal
V_{\rho_v,\phi_v}(v,e^{\prime\prime})=\emptyset$$ if $e^{\prime}$
and $e^{\prime\prime}$ are distinct edges which have $v$ as a common vertex. Now we
define $\Omega^v$, the vertex neighbourhood of the vertex $v$. Let $\mathcal E^v$ denote
the subset of $\mathcal E$, the set of edges, such that $\mathcal E^v=\{e\in\mathcal E:v$
is a vertex of $e\}$. Then
\[\Omega^v={\left(B_{\rho_v}(v)\setminus \underset{e \in\mathcal E^v}{\bigcup}
{\overline{\mathcal V_{\rho_v,\phi_v}(v,e)}}\right)}\bigcap\:\Omega\,.\]
Here $\rho_v$ and $\phi_v$ are chosen so that $\rho_v\sin(\phi_v) =Z$, a constant
for all $v \in \mathcal{V}$, the set of vertices. We now introduce a set of modified
coordinates in the vertex neighbourhood $\Omega^v$. Let
\begin{align*}
\rho &= \sqrt{{x_1}^{2}+{x_2}^{2}+{x_3}^{2}}\\
\phi &= \cos^{-1}(x_3/\rho) \\
\theta &= \tan^{-1}(x_2/x_1)
\end{align*}
denote the usual spherical coordinates in $\Omega^v$. Define
\begin{align}\label{eq2.3}
x_1^{v} &= \phi \notag\\
x_2^{v} &= \theta \notag\\
x_3^{v} &= \chi=\;\ln\rho\:.
\end{align}

Let $w_v=w(v)$, denote the value of $w$ at the vertex $v$. By $\tilde{\Omega}^v$ is
denoted the image of ${\Omega}^v$ in $x^v$ coordinates. We can now state the
differentiability estimates in these modified coordinates in vertex neighbourhoods.
The proof is based on the regularity results proved by Babu\v{s}ka and Guo in~\cite{BG1}.
These estimates are obtained when the differential operator is the Laplacian. However
they are valid for the more general situation examined in this work.

Unless otherwise stated, as in Babu\v{s}ka and Guo~\cite{BG1,BG2,BG3} we let
$w(x^{v})$, $w(x^{v-e})$, $w(x^{e})$ denote $w(x(x^{v}))$, $w(x(x^{v-e}))$,
$w(x(x^{e}))$ respectively. We shall use the same notation for the spectral
element functions $u(x^{v})$, $u(x^{v-e})$, $u(x^{e})$ etc. as well in the
ensuing sections and chapters.
\begin{prop}\label{prop2.2.1}
There exists a constant $\beta_v \in (0,1/2)$ such that for all $0<\nu\leq\rho_v$
the estimate
\begin{align}\label{eq2.4}
\underset{{\tilde{\Omega}^v\cap\{x^v:\ x_3^v\leq \,\ln(\nu)\}}}\int
\sum_{|\alpha|\leq m}e^{x_3^v}\left|\:D_{x^v}^\alpha\:(w(x^v)-w_v \right)|^2\:dx^v
\leq C\,(d^m\,m!)^2\,\nu^{(1-2\beta_v)}
\end{align}
holds for all integers $m\geq 1$. Here $C$ and $d$ denote constants and $dx^v$ denotes
a volume element in $x^v$ coordinates.
\end{prop}
\begin{proof}
The proof is provided in Appendix A.1.
\end{proof}

\subsection{Differentiability estimates in modified coordinates in edge neighbourhoods}
Let $e$ denote an edge, which for convenience we assume to coincide with the $x_3$ axis,
whose end points are the vertices $v$ and $v^{\prime}$ as shown in Figure \ref{fig2.3}.

Assume that the vertex $v$ coincides with the origin. Let the length of the edge $e$ be
$l_e$, $\delta_v=\rho_v\cos(\phi_v)$ and $\delta_{v^{\prime}}=\rho_{v^{\prime}}
\cos(\phi_{v^{\prime}})$. Let $(r,\theta,x_3)$ denote the usual cylindrical coordinates
\begin{align}
r &= \sqrt{{x_1}^2+{x_2}^2} \notag\\
\theta &= \tan^{-1}(x_2/x_1) \notag
\end{align}
and $\Omega^e$ the edge neighbourhood
$$\Omega^e = \left\{x \in \Omega:\; \delta_v < x_3 < l_e-
\delta_{v^{\prime}},\ 0< r < Z\right\}$$ as shown in Figure \ref{fig2.3}. We introduce the
modified system of coordinates
\begin{align}\label{eq2.5}
x_1^e &= \tau = \ln r \notag\\
x_2^e &= \theta \notag\\
x_3^e &= x_3\:.
\end{align}
\begin{figure}[!ht]
\centering
\includegraphics[scale = 0.65]{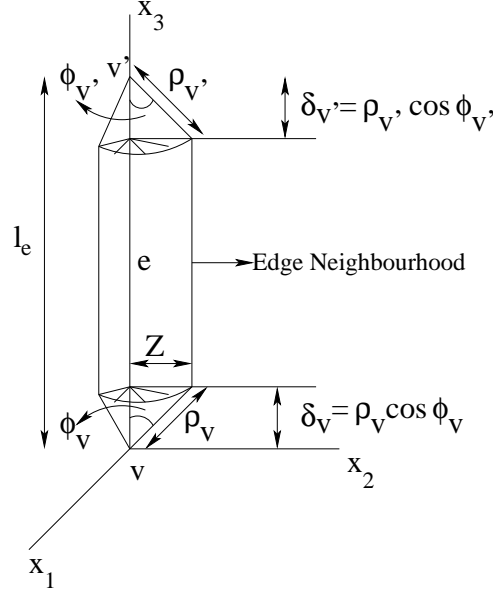}
\caption{ Edge neighbourhood $\Omega^e$. }
\label{fig2.3}
\end{figure}
Let $\tilde{\Omega}^e$ denote the image of $\Omega^e$ in $x^e$ coordinates. The
differentiability estimates for the solution $w$ in edge neighbourhoods in these
modified coordinates can now be stated.
\begin{prop}\label{prop2.2.2}
Let $s(x_3)=w\left(x_1,x_2,x_3\right)|_{\left(x_1=0,x_2=0\right)}$. Then
\begin{align}\label{eq2.6}
\int_{\delta_v}^{l_e-\delta_{v^{\prime}}}\sum_{k \leq m}\left|\:D_{x_3^e}^k
\:s(x_3^e)\:\right|^2\:dx_3^e \leq C\,(d^m\,m!)^2
\end{align}
for all integers $m\geq1$.

Moreover there exists a constant\,$\beta_e \in (0,1)$ such that for $\mu \leq Z$
\begin{align}\label{eq2.7}
\underset{\tilde{\Omega}^e\cap\{x^e:\:x_1^e< \ln\mu\}}\int\sum_{|\alpha| \leq m}
\left|\:D_{x^e}^\alpha(w(x^e)-s(x_3^e))\right|^2\,dx^e\leq\:C\,(d^m\,m!)^2
\:\mu^{2(1-\beta_e)}
\end{align}
for all integers $m\geq1$. Here $dx^e$ denotes a volume element in $x^e$ coordinates.
\end{prop}
\begin{proof}
The proof is provided in Appendix A.2.
\end{proof}

\subsection{Differentiability estimates in modified coordinates in vertex-edge
neighbourhoods} Let $e$ denote an edge, which for convenience we assume coincides
with the $x_3$ axis, and $v$ a vertex which coincides with the origin.
\newline
Then the vertex-edge neighbourhood $\Omega^{v-e}$, shown in Figure \ref{fig2.4}, is
defined as
\[\Omega^{v-e}=\left\{x\in\Omega:\:0< \phi < \phi_v,\:
0< x_3 <\delta_v=\rho_v\cos\phi_v\right\}.\]
\begin{figure}[!ht]
\centering
\includegraphics[scale = 0.60]{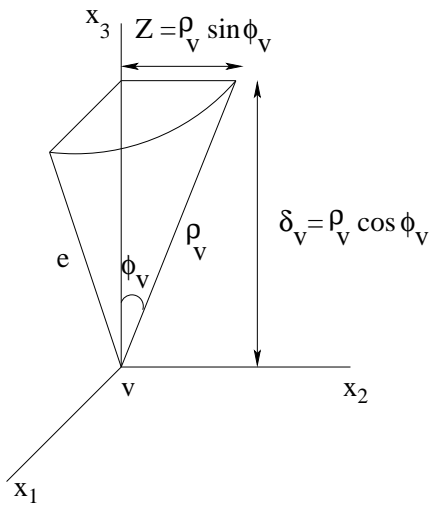}
\caption{ Vertex-edge neighbourhood $\Omega^{v-e}$. }
\label{fig2.4}
\end{figure}
We thus obtain a set of vertex-edge neighbourhoods $\Omega^{v-e}$ where
$v-e \in \mathcal V- \mathcal E$, the set of vertex-edges.

Let us introduce a set of modified coordinates in the vertex-edge neighbourhood
$\Omega^{v-e}$
\begin{align}\label{eq2.8}
x_1^{v-e} &= \psi =\: \ln(\tan\phi)  \notag\\
x_2^{v-e} &= \theta \notag\\
x_3^{v-e} &= \zeta =\: \ln{x_3}\:.
\end{align}
Let $\tilde{\Omega}^{v-e}$ denote the image of $\Omega^{v-e}$ in $x^{v-e}$
coordinates. We can now state the differentiability estimates in modified
coordinates in vertex-edge neighbourhoods.
\begin{prop}\label{prop2.2.3}
Let $w_v=w(v)$,\:the value of $w$ evaluated at the vertex $v$,\:and
$s(x_3)=w(x_1,x_2,x_3)|_{\left(x_1=0,x_2=0\right)}$. Then there exists a constant
$\beta_v \in (0,1/2)$ such that for any $0<\nu\,\leq\,\delta_v$
\begin{align}\label{eq2.9}
\underset{-{\infty}}\int^{\ln\nu} e^{x_3^{v-e}}\sum_{k\leq m}
\left|\:D_{x_3^{v-e}}^k(s(x_3^{v-e})-w_v)\right|^2\: dx_3^{v-e}
\leq\:C\,(d^m\,m!)^2\,\nu^{(1-2\beta_v)}.
\end{align}

Moreover there exists a constant $\beta_e \in (0,1)$ such that for any
$0<\alpha \leq \tan\phi_v$ and $0<\nu \leq\delta_v$
\begin{align}\label{eq2.10}
&\underset{{\tilde{\Omega}^{v-e}} \cap \{x^{v-e}:\;x_1^{v-e}<\ln\alpha,\ x_3^{v-e}<\ln\nu\}}
\int e^{x_3^{v-e}}\sum_{|\gamma|\leq m}\left|\:D_{x^{v-e}}^{\gamma}(w(x^{v-e})-s(x_3^{v-e}))
\right|^2 \:dx^{v-e} \notag\\
&\hspace{4.5cm}\leq C\,(d^m\,m!)^2\,\alpha^{2(1-\beta_e)}\,\nu^{(1-2\beta_v)}
\end{align}
for all integers $m \geq 1$. Here $dx^{v-e}$ denotes a volume element in $x^{v-e}$
coordinates.
\end{prop}
\begin{proof}
The proof is provided in Appendix A.3.
\end{proof}

\subsection{Differentiability estimates in standard coordinates in the regular
region of the polyhedron} Let $\Omega^r$ denote the portion of the polyhedron
$\Omega$ obtained after the closure of the vertex neighbourhoods, edge neighbourhoods
and vertex-edge neighbourhoods have been removed from it.
\newline
Thus let
$$\Delta=\left\{\underset{v \in \mathcal V} \bigcup{\overline{\Omega}^v}
\right\}\bigcup\left\{\underset{e \in \mathcal E}\bigcup{\overline{\Omega}^e}\right\}
\bigcup\left\{\underset{v-e \in \mathcal V-\mathcal E}\bigcup{\overline{\Omega}^{v-e}}
\right\}\,.$$
Then
\[\Omega^r=\Omega \setminus \bigtriangleup \;.\]
We denote the regular region of the polyhedron, in which the solution $w$ is analytic,
by $\Omega^r$. In $\Omega^r$ the standard coordinate system $x=(x_1,x_2,x_3)$ is retained.
The differentiability estimates in these coordinates in the regular region of the
polyhedron, are now stated.
\begin{prop}\label{prop2.2.4}
The estimate
\begin{equation}\label{eq2.11}
\underset{{\Omega^r}}\int\sum_{|\alpha|\leq m}\left|\:D_x^\alpha
w(x)\right|^2 dx \leq C\,(d^m\,m!)^2
\end{equation}
holds for all integers $m\geq 1$. Here $dx$ denotes a volume element in $x$ coordinates.
\end{prop}
\begin{proof}
The proof is provided in Appendix A.4.
\end{proof}

\subsection{Function spaces}
We need to review a set of function spaces described in~\cite{BG1}.

Let $\Omega^v$ denote the vertex neighbourhood of the vertex $v$ and $\rho=\rho(x)=dist(x,v)$
for $x\in\Omega^{v}$ and $\beta_{v}\in(0,1/2)$. We introduce a weight function as in~\cite{BG1}
as follows:
Define
$$\Phiy_{\beta_v}^{\alpha,l}(x) =
\left\{\begin{array}{ccc}
\rho^{{\beta_v}+|\alpha|-l},\ &  \mbox{for} \;|\alpha|\geq l\\
\hspace{-1.5cm} 1, & \mbox{for} \;|\alpha|< l
\end{array}\right.
$$
as in $(2.3)$ of~\cite{BG1}. Let
\begin{equation*}
{\bf H}_{\beta_v}^{k,l}(\Omega^v) =\left\{u|\:\left\|u\right\| ^2_{{\bf H}_{\beta_v}
^{k,l}(\Omega^v)} = \underset{|\alpha|\leq k} \sum
\left\|\Phiy_{\beta_v}^{\alpha,l}\:D^{\alpha} u\:\right\|^2_{L^2({\Omega}^v)}
<\infty\right\}
\end{equation*}
and
\begin{equation*}
{\bf B}^l_{\beta_v}(\Omega^v) = \left\{u|\:u \in {\bf H}_{\beta_v}^{k,l}(\Omega^v)
\:\mbox{ for all }\:k\geq l \:\:\mbox{and}\: \left\|\Phiy_{\beta_v}^{\alpha,l}
{D}^{\alpha} u\;\right\|^2_{L^2{({\Omega}^v)}}\leq C\,d^{\alpha}\,\alpha!\right\}
\end{equation*}
denote the weighted Sobolev space and countably normed space defined on $\Omega^v$ as
in~\cite{BG1}.

Let us denote by ${\bf C}^2_{\beta_v}(\Omega^v)$ a countably normed space as in~\cite{BG1,G1}
which is the set of functions $u(x) \in {\bf C}^0(\bar{\Omega}^v)$ such that for all $\alpha$,
$|\alpha| \geq 0$
\[\left|\:D_x^{\alpha}(u(x)-u(v))\:\right| \leq C\,
d^{\alpha}\,\alpha!\,\rho^{-(\beta_v+|\alpha|-{1/2})}(x).\]
Here $\bar{\Omega}^{v}$ denotes the closure of $\Omega^{v}$. Then by Theorem $5.6$
of~\cite{BG1}, ${\bf B}^2_{\beta_v}(\Omega^v)\subseteq {\bf C}^2_{\beta_v}(\Omega^v).$

We now cite an important regularity result, viz. Theorem $5.1$ of~\cite{G1} for
the solution of (\ref{eq2.1}).
\newline
\textit{There exists a unique (weak) solution $u\in H^{1}(\Omega^v)$ of (\ref{eq2.1})
which belongs to ${\bf C}^2_{\beta_v}(\Omega^v)$ with $\beta_v\in(0,1/2)$ satisfying
\begin{equation}\label{eq2.12}
\beta_v\geq 1/2-\lambda_v, \lambda_v=\frac{1}{2}\sqrt{1+4\nu^{v}_{min}}-1
\end{equation}
where $\nu^{v}_{min}$ is the smallest positive eigenvalue of the Laplace-Beltrami
operator on the spherical boundary $S^v$.}

Next, let $\Omega^{e}$ denote an edge neighbourhood of $\Omega$ and $r=r(x)=dist(x,e)$
for $x\in\Omega^{e}$ and $\beta_{e}\in(0,1)$. Define the weight function by
$$\Phiy_{\beta_e}^{\alpha,l}(x) =
\left\{\begin{array}{ccc}
r^{{\beta_e}+|\alpha^{\prime}|-l},\ &  \mbox{for} \;|\alpha^{\prime}|
=\alpha_1+\alpha_2\geq l\\
\hspace{-1.5cm} 1, &  \hspace{-2.0cm} \mbox{for} \;|\alpha|< l
\end{array}\right.
$$
as in $(2.1)$ of~\cite{BG1}. Let
\begin{equation*}
{\bf H}_{\beta_e}^{k,l}(\Omega^e)=\left\{u|\:\left\|u\right\| ^2_{{\bf H}_{\beta_e}
^{k,l}(\Omega^e)} = \underset{|\alpha|\leq k} \sum\left\|\Phiy_{\beta_e}^{\alpha,l}\:
D^{\alpha} u\:\right\|^2_{L^2{({\Omega}^e)}}<\infty\right\}
\end{equation*}
and
\begin{equation*}
{\bf B}^l_{\beta_e}(\Omega^e) = \left\{u|\:u \in{\bf H}_{\beta_e}^{k,l}(\Omega^e)
\:\mbox{ for all }\:k\geq l \:\:\mbox{and}\: \left\|\Phiy_{\beta_e}^{\alpha,l}
{D}^{\alpha} u\;\right\|^2_{L^2{({\Omega}^e)}}\leq C\,d^{\alpha}\,\alpha!
\right\}
\end{equation*}
denote the weighted Sobolev space and countably normed space defined on $\Omega^e$
as in~\cite{BG1}.

We denote by ${\bf C}^2_{\beta_e}(\Omega^e), \beta_e \in (0,1)$ a countably normed
space as in~\cite{BG1,G1}, the set of functions $u \in {\bf C}^0(\bar{\Omega}^e)$
such that for $|\alpha| \geq 0$
\[\left\|\:r^{\beta_e+\alpha_1+\alpha_2-1}D_x^{\alpha}
\left(u(x)-u(0,0,x_3)\right)\right\|_{{\bf C}^0(\bar{\Omega}^e)}
\leq C\,d^{\alpha}\,\alpha!\] and for $k\geq 0$
\[\left\|\frac{d^k}{(dx_3)^k}u(0,0,x_3)\right
\|_{{\bf C}^0(\bar{\Omega}^e\cap\{x:\:x_1=x_2=0\})}
\leq C\,d^k\,k!\:,\]
where $\bar{\Omega}^{e}$ denotes the closure of $\Omega^{e}$. Then by Theorem $5.3$
of~\cite{BG1}, ${\bf B}^2_{\beta_e}(\Omega^e)\subseteq {\bf C}^2_{\beta_e}(\Omega^e)$.

We now cite another important regularity result, viz. Theorem $3.1$ of~\cite{G1}
for the solution of (\ref{eq2.1}).
\newline
\textit{There exists a unique (weak) solution $u\in H^{1}(\Omega^e)$ of (\ref{eq2.1})
which belongs to ${\bf C}^2_{\beta_e}(\Omega^e)$ with $\beta_e\in(0,1)$ satisfying
\begin{align}\label{eq2.13}
\beta_e\geq 1-\kappa_e, \kappa_e=\left\{\begin{array}{ccc}
\frac{\pi}{2\omega_{e}}\ &  \mbox{if} \:\: \Gamma_{s}\subset\Gamma^{[0]},
 \Gamma_{t}\subset\Gamma^{[1]}\\
\frac{\pi}{\omega_{e}}, &  \hspace{-2.0cm} \mbox{otherwise}
\end{array}\right.
\end{align}
where $\Gamma_{s}$ and $\Gamma_{t}$ are such that $\Gamma_{s}\cap\Gamma_{t}=e$}.

Finally, let $\rho=\rho(x)$ and $\phi=\phi(x)$ for $x\in\Omega^{v-e}$.
We define a weight function by
$$\Phiy_{\beta_{v-e}}^{\alpha,l}(x) =
\left\{\begin{array}{ccc}
\rho^{{\beta_v}+|\alpha|-l}(\sin(\phi))^{{\beta_e}+|\alpha^{\prime}|-l},
\ &  \mbox{for} \;|\alpha^{\prime}|=\alpha_1+\alpha_2\geq l\\
\hspace{-3.0cm} \rho^{{\beta_v}+|\alpha|-l}, & \hspace{-0.9cm} \mbox{for}
\;|\alpha^{\prime}|<l\leq|\alpha|\\
\hspace{-4.2cm} 1, &  \hspace{-2.0cm} \mbox{for} \;|\alpha|< l
\end{array}\right.$$
as in $(2.2)$ of~\cite{BG1}. Let
\begin{equation*}
{\bf H}_{\beta_{v-e}}^{k,l}(\Omega^{v-e})=\left\{u|\:\left\|u\right\|^2_{{\bf H}_{\beta_{v-e}}
^{k,l}(\Omega^{v-e})} = \underset{|\alpha|\leq k} \sum\left\|\Phiy_{\beta_{v-e}}^{\alpha,l}\:
D^{\alpha} u\:\right\|^2_{L^2{({\Omega}^{v-e})}}<\infty\right\}
\end{equation*}
and
\begin{align*}
{\bf B}^l_{\beta_{v-e}}(\Omega^{v-e}) = \left\{u|\:u \in {\bf H}_{\beta_{v-e}}^{k,l}(\Omega^{v-e})
\:\:\forall\:k\geq l \: \mbox{and}\: \left\|\Phiy_{\beta_{v-e}}^{\alpha,l}
{D}^{\alpha} u\;\right\|^2_{L^2{({\Omega}^{v-e})}}\leq C
\,d^{\alpha}\,\alpha!\right\}
\end{align*}
denote the weighted Sobolev space and countably normed space defined on $\Omega^{v-e}$
as in~\cite{BG1}.

Let us denote by ${\bf C}^2_{\beta_{v-e}}(\Omega^{v-e})$,
where $\beta_{v-e} = (\beta_v,\beta_e), \ \beta_v \in(0,1/2)$ and
$\beta_e \in (0,1)$, the set of functions $u(x)\in {\bf C}^0(\bar{\Omega}^
{v-e})$ such that
\[\left\|\:\rho^{\beta_v+|\alpha|-{1/2}}\:(\sin\phi)^{\beta_e+\alpha_1+\alpha_2-1}
\:D^{\alpha}_x\left(u(x)-u(0,0,x_3)\right)\:\right\|_{{\bf C}^0({\bar{\Omega}^{v-e}})}
\leq C\,d^{\alpha}\,\alpha!\]
and
\[\left|\:|x_3|^{\beta_v+k-{1/2}} \frac{d^k}{dx_3^k}\left(u(0,0,x_3)-u(v)\right)
\right|_{{\bf C}^0\left(\bar{\Omega}^{v-e}\cap\{x:\:x_1=x_2=0 \}\right)}\leq C\,d^k\,k!\]
as described in \cite{BG1,G1}. Here $\bar{\Omega}^{v-e}$ denotes the closure of
$\Omega^{v-e}$. Now by Theorem $5.9$ of ~\cite{BG1}, ${\bf B}^2_{\beta_{v-e}}
(\Omega^{v-e})\subseteq {\bf C}^2_{\beta_{v-e}}(\Omega^{v-e})$.

We cite one last Theorem $4.1$ of~\cite{G1} for the solution of (\ref{eq2.1}).
\newline
\textit{There exists a unique (weak) solution $u\in H^{1}(\Omega^{v-e})$
of (\ref{eq2.1}) which belongs to ${\bf C}^2_{\beta_{v-e}}(\Omega^{v-e})$, where
$\beta_{v-e} = (\beta_v,\beta_e),\ \beta_v \in(0,1/2)$ and
$\beta_e \in (0,1)$ satisfying (\ref{eq2.12}) and (\ref{eq2.13}).}

\section{The Stability Theorem}
$\Omega$ is divided into a regular region $\Omega^r$, a set of vertex neighbourhoods
$\Omega^v$, where $v \in \mathcal V$, a set of edge neighbourhoods $\Omega^e$, where
$e \in \mathcal E$ and a set of vertex-edge neighbourhoods $\Omega^{v-e}$, where $v-e
\in \mathcal V - \mathcal E$. In the regular region $\Omega^r$ standard coordinates
$x=(x_1,x_2,x_3)$ are used and in the remaining regions modified coordinates are used
as has been described in Section $2$. $\Omega^r$ is divided into a set of curvilinear
hexahedrons, tetrahedrons and prisms. We impose a geometrically graded mesh in the
remaining regions which is described in this section. We remark that a tetrahedron can
always be divided into four hexahedrons~\cite{SSW1}, in the same way that a triangle
can be divided into three quadrilaterals by joining the centre of the triangle to the
midpoints of the sides (Figure~\ref{fig2.5}). Moreover a prism can be divided into three
hexahedral elements. Hence we can choose all our elements to be hexahedrons.
\begin{figure}[!ht]
\centering
\includegraphics[scale = 0.60]{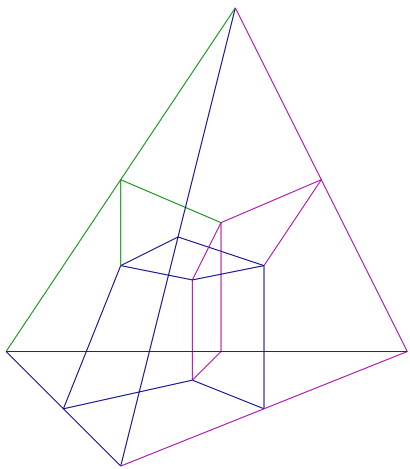}
\caption{Division of a tetrahedron into hexahedrons.}
\label{fig2.5}
\end{figure}
A set of spectral element functions are defined on the elements. In edge neighbourhoods
and vertex-edge neighbourhoods these spectral element functions are a sum of tensor products
of polynomials in the modified coordinates. Let $\{{\mathcal F_u}\}$ denote the spectral
element representation of the function $u$. We shall examine two cases. The first case is
when the spectral element functions are nonconforming. The second case is when the spectral
element functions are conforming on the wirebasket $W\!B$ of the elements, i.e. the union
of the edges and vertices of the elements. 
In both these cases the spectral element functions are nonconforming on the faces (open) of
the elements.

To state the stability theorem we need to define some quadratic forms. Let $N$ denote the
number of refinements in the geometrical mesh and $W$ denote an upper bound on the degree
of the polynomial representation of the spectral element functions. We shall define two
quadratic forms $\mathcal V^{N,W}\left(\{\mathcal F_u\}\right)$ and $\mathcal U^{N,W}
\left(\{\mathcal F_u\}\right)$.

Now
\begin{align}\label{eq2.14}
\mathcal V^{N,W}\left(\{\mathcal F_u\}\right) &=\mathcal
V_{regular}^{N,W}\left(\{\mathcal F_u\}\right)+\mathcal
V_{vertices}^{N,W}\left(\{\mathcal F_u\}\right) +\mathcal
V_{vertex-edges}^{N,W}\left(\{\mathcal F_u\}\right) \notag \\
&+\mathcal V_{edges}^{N,W}\left(\{\mathcal F_u\}\right).
\end{align}

In the same way
\begin{align}\label{eq2.15}
\mathcal U^{N,W}\left(\{\mathcal F_u\}\right)&=\mathcal
U_{regular}^{N,W}\left(\{\mathcal F_u\}\right)+\mathcal
U_{vertices}^{N,W}\left(\{\mathcal F_u\}\right)+\mathcal
U_{vertex-edges}^{N,W}\left(\{\mathcal F_u\}\right) \notag \\
&+\mathcal U_{edges}^{N,W}\left(\{\mathcal F_u\}\right).
\end{align}

Let us first consider the regular region $\Omega^r$ of $\Omega$ and define the two
quadratic forms $\mathcal V_{regular}^{N,W}\left(\{\mathcal F_u\}\right)$ and
$\mathcal U_{regular}^{N,W}\left(\{\mathcal F_u\}\right)$. The regular region
$\Omega^r$ is divided into $N_r$ curvilinear hexahedrons, tetrahedrons and prisms.
In $\Omega^r$ the standard coordinates $x=(x_1,x_2,x_3)$ are used. Let $\Omega_l^r$
be one of the elements into which $\Omega^r$ is divided, which we shall assume is
a curvilinear hexahedron to keep the exposition simple. Let $Q$ denote the standard
cube $Q=(-1,1)^3$. Then there is an analytic map $M_l^r$ from $Q$ to $\Omega_l^r$
which has an analytic inverse. Let $\Omega_l^r$ be as shown in Figure \ref{fig2.6}
and let $\{\Gamma_{l,i}^r\}_{1\leq i\leq n_{l}^{r}}$ denote its faces.
\begin{figure}[!ht]
\centering
\includegraphics[scale = 0.60]{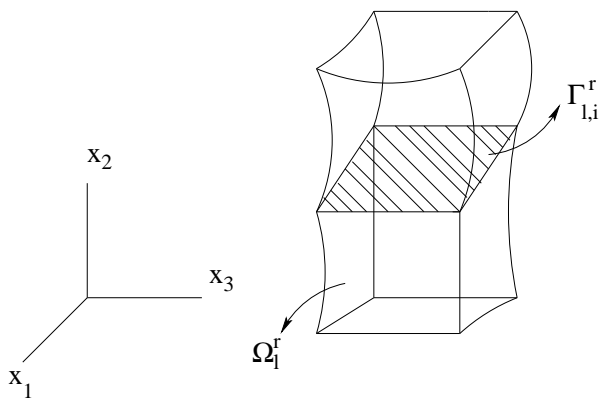}
\caption{ Elements in $\Omega^r$. }
\label{fig2.6}
\end{figure}
Now the map $M_l^r$ is of the form
$$x=X_l^r(\lambda_1,\lambda_2,\lambda_3)$$
where $(\lambda_1,\lambda_2,\lambda_3) \in Q$, the master cube. Define the spectral
element function $u_l^r$ on $\Omega_l^r$ by
\[u_l^r(\lambda) = \sum_{i=0}^W \sum_{j=0}^W \sum_{k=0}^W
\:\alpha_{i,j,k}\:\lambda_1^i\lambda_2^j\lambda_3^k.\]
Now the spectral element functions are nonconforming in the general case. Let
$\left.{[u]}\right|_{\Gamma_{l,i}^r}$ denote the jump in $u$ across the face
$\Gamma_{l,i}^r$. Let the face $\Gamma_{l,i}^r=\Gamma_{m,j}^r$ where $\Gamma_{m,j}^r$
is a face of the element $\Omega_m^r$. We may assume the face $\Gamma_{l,i}^r$
corresponds to $\lambda_3=1$ and $\Gamma_{m,j}^r$ corresponds to $\lambda_3=-1$. Then
$\left.{[u]}\right|_{\Gamma_{l,i}^r}$ is a function of only $\lambda_1$ and $\lambda_2$.

We now define
\begin{align}\label{eq2.16}
\mathcal V_{regular}^{N,W}\left(\{\mathcal F_u\}\right) &=
\sum_{l=1}^{N_r}\int_{\Omega_l^r}\left|\:Lu_l^r(x)\:\right|^2\:dx \notag\\
&+\sum_{\Gamma_{l,i}^r\subseteq\bar{\Omega}^r\setminus\partial\Omega}
\left(\|[u]\|_{0,\Gamma_{l,i}^r}^2 +\sum_{k=1}^3\left \|[u_{x_k}]
\right\|_{1/2,{\Gamma_{l,i}^r}}^2\right) \notag\\
&+\sum_{{\Gamma_{l,i}^r}\subseteq{\Gamma^{[0]}}}\left
\|u_l^r\right\|_{3/2,{\Gamma_{l,i}^r}}^2 +
\sum_{{\Gamma_{l,i}^r}\subseteq{\Gamma^{[1]}}}
\left\|\left(\frac{\partial u_l^r}{\partial
\nu}\right)_A\right\|_{1/2,{\Gamma_{l,i}^r}}^2.
\end{align}
The fractional Sobolev norms used above are as defined in~\cite{G}.

Since $\Gamma_{l,i}^r$, corresponding to $\lambda_{3} = 1$, is the image of
$S=(-1,1)^2$, or $T$ the master triangle, in $\lambda_1,\lambda_2$ coordinates
\begin{subequations}\label{eq2.17}
\begin{align}\label{eq2.17a}
\|w\|_{\sigma, \Gamma_{l,i}^r}^2 = \|w\|_{0,E}^2 + \int_E\int_E
\frac{\left(w(\lambda_1,\lambda_2)-w(\lambda_1^{\prime},
\lambda_2^{\prime})\right)^2}{\left((\lambda_1-\lambda_1^
{\prime})^2+(\lambda_2-\lambda_2^{\prime})^2\right)^{1+\sigma}}
d\lambda_1\:d\lambda_2\:d\lambda_1^{\prime}\:d\lambda_2^{\prime}
\end{align}
for $0< \sigma < 1$. Here $E$ denote either $S$ or $T$.
\newline
However, if $E$ is $S$ then we prefer to use the equivalent norm
\begin{align}\label{eq2.17b}
\|w\|_{\sigma, \Gamma_{l,i}^r}^2 = \|w\|_{0,E}^2 + \int_{-1}^{1}
\int_{-1}^{1}\int_{-1}^{1}\frac{(w(\lambda_1,\lambda_2)-w
(\lambda_1^{\prime},\lambda_2))^2}{(\lambda_1-\lambda_1^{\prime}
)^{1+2\sigma}}\,d\lambda_1 d\lambda_1^{\prime}d\lambda_2 \notag\\
+\int_{-1}^{1}\int_{-1}^{1}\int_{-1}^{1}\frac{(w(\lambda_1,
\lambda_2)-w(\lambda_1,\lambda_2^{\prime}))^2}{(\lambda_2-
\lambda_2^{\prime})^{1+2\sigma}}\,d\lambda_2 d\lambda_2^{\prime}
d\lambda_1\,.
\end{align}
\end{subequations}
Moreover
\begin{equation}\label{eq2.18}
\|w\|_{1+\sigma, \Gamma_{l,i}^r}^2 = \|w\|_{0,E}^2 +\sum_{i=1}^2 \left\|
\:\frac{\partial w}{\partial\lambda_i}\:\right\|_{\sigma, E}^2\;.
\end{equation}
Next, we define
\begin{align}\label{eq2.19}
\mathcal U_{regular}^{N,W}(\{\mathcal F_u\})=
\sum_{l=1}^{N_r}\int_{Q=(M_l^r)^{-1}(\Omega_l^r)}\underset{|\alpha|\leq 2}{\sum}
\left|\:D_\lambda^\alpha u_l^r\:\right|^2\,d\lambda\;.
\end{align}
Let $v$ be one of the vertices of $\Omega$. In Figure~\ref{fig2.7} the vertex neighbourhood,
described in Section $2.1$, is shown.
\begin{figure}[!ht]
\centering
\includegraphics[scale = 0.60]{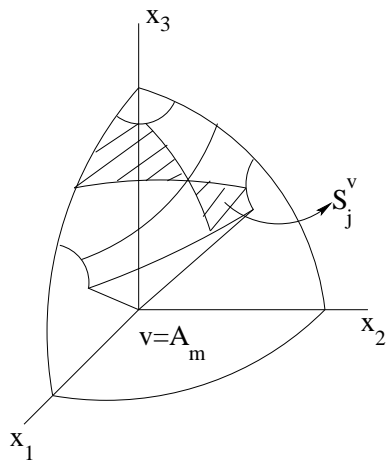}
\caption{ Mesh imposed on the spherical boundary $S^v$. }
\label{fig2.7}
\end{figure}
Let $S^v$ denote the intersection of the surface of the sphere ${B_{{\rho}_v}}(v)$ with
$\bar{\Omega}^v$, i.e.
\[S^v=\left\{x \in \bar{\Omega}^v:\:dist\:(x,v)=\rho_v\right\}.\]
We divide the surface $S^v$ into a set of triangular and quadrilateral elements as
shown in Figure \ref{fig2.8}. Let $S_j^v$ denote these elements where $1\leq j\leq I_v$.
Here $I_v$ denotes a fixed constant. Let $\mu_v$ be a positive constant less than
one which shall be used to define a geometric mesh in the vertex neighbourhood $\Omega^v$
of the vertex $v$. We now divide $\Omega^v$ into $N_v = I_v(N+1)$ curvilinear hexahedrons
and prisms $\{\Omega_l^v\}_{1 \leq l \leq N_v}$, where $\Omega_l^v$ is of the form
\[\Omega_l^v=\left\{x:\;(\phi,\theta) \in S_j^v, \ \rho_k^v<\rho<\rho_{k+1}^v\right\}\]
\begin{figure}[!ht]
\centering
\includegraphics[scale = 0.60]{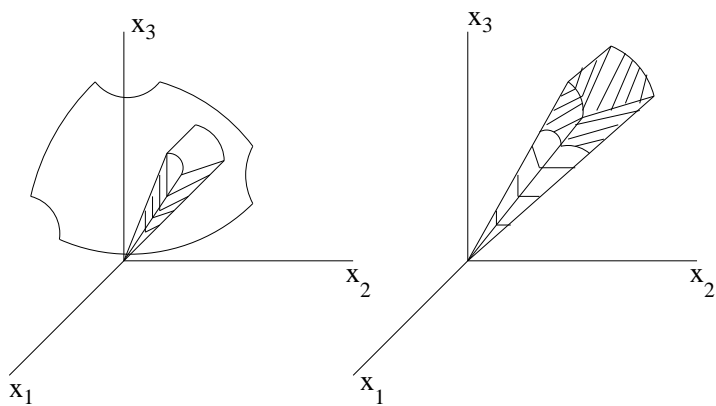}
\caption{ Geometrical mesh imposed on $\Omega^v$. }
\label{fig2.8}
\end{figure}
for $1 \leq j \leq I_v$ \ and \ $0 \leq k\leq N$. Here $\rho_k^v = \rho_v (\mu_v)^{N+1-k}$
and $0<\mu_v<1$ for ${1 \leq k \leq N+1}$. Moreover $\rho_0^v=0$. We introduce the vertex
coordinates $x^v$,
\begin{align}
x_1^v &= \phi \notag\\
x_2^v &= \theta \notag\\
x_3^v &= \chi = \;\ln \rho\;.\notag
\end{align}
Let $\tilde{\Omega}^v$ denote the image of $\Omega^v$ in $x^v$ coordinates and
$\tilde{\Omega}_l^v$ denote the image of the element $\Omega_l^v$. Then the geometric
mesh $\{\Omega_l^v\}_{1\leq l \leq N_v}$ which has been defined on $\Omega^v$, is
mapped to a quasi-uniform mesh  $\{\tilde{\Omega}_l^v\}_{1 \leq l \leq N_v}$ on
$\tilde{\Omega}^v$, except that the corner elements
$$\Omega_l^v =\left\{x:\;(\phi,\theta) \in S_j^v, \ 0< \rho <\rho_1^v\right\}$$
are mapped to the semi-infinite elements
$$\tilde{\Omega}_l^v =\left\{x^v:\; (\phi,\theta) \in S_j^v,\:
-\infty< \chi < \ln\rho_1^v\right\}\;.$$
We now specify the form of the spectral element functions $u_l^v(x^v)$ on the elements.
Consider first the case when $\tilde{\Omega}_l^v$ is a corner element of the form
$$\tilde{\Omega}_l^v =\left\{x^v:\; (\phi,\theta) \in S_j^v,\:
-\infty< \chi < \ln\rho_1^v\right\}\;.$$
In this case we define $u_l^v(x^v) = h^v$, where $h^v$ is a constant. Thus at all corner
elements the spectral element functions assume the same constant value for that corner.

Now there is an analytic map $M_l^v$ from $Q$, the master cube to $\tilde{\Omega}_l^v$,
which has an analytic inverse. Here the map $M_l^v$ is of the form
\[x^v = X_l^v(\lambda_1,\lambda_2,\lambda_3).\]
We define the spectral element function $u_l^v$ on $\tilde{\Omega}_l^v$ by
\[u_l^v(\lambda) = \sum_{t=0}^{W_l}\sum_{s=0}^{W_l}\sum_{r=0}^{W_l}\;\beta_{r,s,t}
\:\lambda_1^r\lambda_2^s\lambda_3^t\;.\]
Here $1\leq W_l\leq W$. Moreover as in~\cite{G1}, $W_l=[\mu_{1} i]$ for $1\leq i\leq N$,
where $\mu_{1}>0$ is a degree factor. Hereafter $[a]$ denotes the greatest positive integer
$\leq a$.

Let $v \in \mathcal V$ denote the vertices of $\Omega$. Define
\begin{equation}\label{eq2.20}
\mathcal V_{vertices}^{N,W}(\{\mathcal F_u\})
=\sum_{v \in \mathcal V}\mathcal V_v^{N,W}(\{\mathcal F_u\})
\end{equation}
and
\begin{equation}\label{eq2.21}
\mathcal U_{vertices}^{N,W}(\{\mathcal F_u\})
=\sum_{v \in \mathcal V} \mathcal U_v^{N,W}(\{\mathcal F_u\})\;.
\end{equation}
We now fix a vertex $v$ and define the quadratic forms
$\mathcal V_v^{N,W}(\{\mathcal F_u\})$ and $\mathcal U_v^{N,W}(\{\mathcal F_u\})$.
Consider the vertex neighbourhood $\Omega^v$ and let $\Omega_l^v$ be one of the
elements into which it is divided. Now $\Omega_l^v$ has $n_l^v$ faces
$\{\Gamma_{l,i}^v\}_{1\leq i\leq n_l^v}$. Let $\tilde{\Omega}_l^v$ be the image of
$\Omega^v_l$ and $\tilde{\Gamma}_{l,i}^v$ be the image of $\Gamma_{l,i}^v$ in $x^v$
coordinates.

Define $L^v u(x^v)$ so that
\begin{equation}\label{eq2.22}
\int_{\tilde{\Omega}_l^v}\left|\:L^v u(x^v)\:\right|^2\:dx^v =
\int_{\Omega_l^v} \rho^2\left|\:Lu(x)\:\right|^2\:dx\;.
\end{equation}
Here $dx^v$ denotes a volume element in $x^v$ coordinates and $dx$ a volume element
in $x$ coordinates. In Chapter $3$ it will be shown that
\begin{equation}\label{eq2.23}
L^v u(x^v)=-div_{x^v}\left(e^{\chi/2}\sqrt{\sin\phi}A^v\nabla_{x^v}u
\right)+\sum_{i=1}^3 \hat{b}_i^v u_{x_i^v}+\hat{c}^v u \;.
\end{equation}
In the above $A^v$ is a symmetric, positive definite matrix.

Let $\Gamma_{l,i}^v$ be one of the faces of $\Omega_l^v$ and
$\tilde{\Gamma}_{l,i}^v$ denote its image in $x^v$ coordinates.
Let $\tilde P$ be a point belonging to $\tilde{\Gamma}_{l,i}^v$
and $\nuw^v$ be the unit normal to $\tilde{\Gamma}_{l,i}^v$
at the point $\tilde P$. Then define
\begin{equation}\label{eq2.24}
\left(\frac {\partial u}{\partial \nuw^v}\right)_{A^v}
(\tilde P)= \left(\nuw^v\right)^T A^v\nabla_{x^v}u \,.
\end{equation}
Here the matrix $A^v$ is as in (\ref{eq2.21}).
\\
Let
$$R_{l,i}^v=\underset {x^v\in\tilde{\Gamma}_{l,i}^v}{sup}(e^{x_3^v}).$$
We now define
\begin{align}\label{eq2.25}
& \mathcal V_v^{N,W}\left(\{\mathcal F_u\}\right) =
\sum_{l=1,\mu(\tilde{\Omega}_l^v)<\infty}^{N_v}\int_{\tilde{\Omega}_l^v}
\left|\:L^v u_l^v(x^v)\:\right|^2\;dx^v \notag\\
&+\mathop{\sum_{\Gamma_{l,i}^v\subseteq\bar{\Omega}^v\setminus
\partial\Omega,}}_{\mu(\tilde{\Gamma}_{l,i}^v)<\infty}
\left(\left\|\:\sqrt{R_{l,i}^v}[u]\:\right\|_{0,\tilde{\Gamma}_{l,i}^v}^2
+\sum_{k=1}^3\left\|\:\sqrt{R_{l,i}^v}[u_{x_k^v}]\:\right\|_{1/2,{\tilde
{\Gamma}_{l,i}^v}}^2\right)\notag\\
&+\mathop{\sum_{{\Gamma_{l,i}^v}\subseteq{\Gamma^{[0]}},}}
_{\mu(\tilde{\Gamma}_{l,i}^v)<\infty}\left\|\:\sqrt{R_{l,i}^v} u_l^v\:
\right\|_{3/2,{\tilde{\Gamma}_{l,i}^v}}^2 +\mathop{\sum_{{\Gamma_{l,i}^v}
\subseteq{\Gamma^{[1]}},}}_{\mu(\tilde{\Gamma}_{l,i}^v)<\infty}\left\|\:
\sqrt{R_{l,i}^v}\left(\frac{\partial u_l^v}{\partial\nuw^v}\right)_{A^v}
\:\right\|_{1/2,{\tilde{\Gamma}_{l,i}^v}}^2.
\end{align}
The fractional Sobolev norms used above are as in (\ref{eq2.17}) and (\ref{eq2.18}).
Moreover $\mu$ denotes measure.
\newline
Finally, the quadratic form $\mathcal U_v^{N,W}(\{\mathcal{F}_u\})$ is given by
\begin{equation}\label{eq2.26}
\mathcal U_v^{N,W}(\{\mathcal F_u\}) = \sum_{l=1}^{N_v}
\int_{\tilde{\Omega}_l^v} e^{x_3^v}\sum_{|\alpha|\leq 2}
\:\left|\:D^\alpha_{x^v} u_l^v(x^v)\: \right|^2\;dx^v.
\end{equation}

We now define $\mathcal V_{vertex-edges}^{N,W}(\{\mathcal F_u\})$ and
$\mathcal U_{vertex-edges}^{N,W}(\{\mathcal F_u\})$. Let $v-e$ denote one of the vertex-edges
of $\Omega$. Here ${v-e \in \mathcal {V-E}}$, the set of vertex-edges of $\Omega$. Let
$\Omega^{v-e}$ denote the vertex-edge neighbourhood corresponding to the vertex-edge $v-e$.
We divide $\Omega^{v-e}$ into $N_{v-e}$ elements $\Omega^{v-e}_l$,\,$l=1,2,\ldots,N_{v-e}$,
using a geometric mesh.

Figure \ref{fig2.9} shows the vertex-edge neighbourhood $\Omega^{v-e}$ of the vertex $v$ and
the edge $e$. Now
\[\Omega^{v-e} =\left\{x \in \Omega:\; 0<x_3<\delta_v,\:0<\phi<\phi_v\:\right\}\;.\]
Here $\delta_v = \rho_v\cos\phi_v$. We impose a geometrical mesh on $\Omega^{v-e}$ as shown in
Figure \ref{fig2.9} by defining
$$(x_3)_0=0 \mbox{ and }(x_3)_i=\delta_v(\mu_v)^{N+1-i}$$
for $1\leq i\leq N+1$. Let
\[\zeta_i^{v-e} = \ln\left((x_3)_i\right)\]
for $0\leq i\leq N+1$.
\begin{figure}[!ht]
\centering
\includegraphics[scale = 0.60]{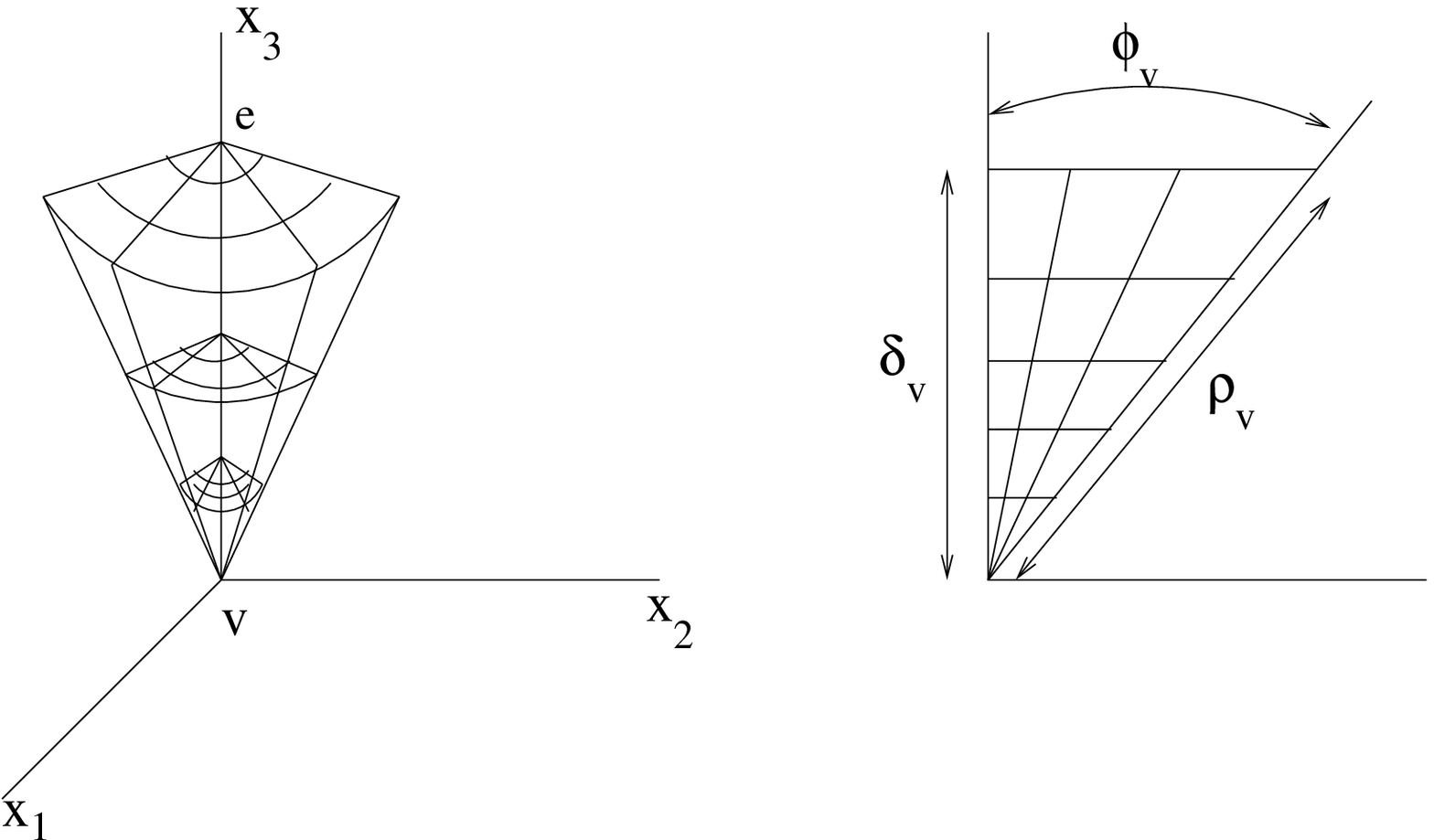}
\caption{ Geometrical mesh imposed on $\Omega^{v-e}$. }
\label{fig2.9}
\end{figure}

Let us introduce points $\phi_0^{v-e},\ldots,\phi_{N+1}^{v-e}$ such that
$\phi_{0}^{v-e} = 0$ and $\tan \phi_i^{v-e}=\mu_e^{N+1-i} \tan(\phi_v)$,
for $1\leq i\leq N+1$. Here $\mu_e$ is a positive constant less than one.
Thus we impose a geometrical mesh on $\phi$ with mesh ratio $\mu_e$.
Finally, $\theta_{v-e}^l<\theta<\theta_{v-e}^u$. A quasi-uniform mesh
\[\theta_{v-e}^l = \theta^{v-e}_0 < \theta^{v-e}_1
<\cdots<\theta^{v-e}_{I_{v-e}} = \theta_{v-e}^u\]
is imposed in $\theta$. We introduce new coordinates in $\Omega^{v-e}$ by
\begin{align*}
x_1^{v-e} &= \psi =\; \ln (\tan{\phi}) \\
x_2^{v-e} &= \theta  \\
x_3^{v-e} &= \zeta =\; \ln x_3\;.
\end{align*}
Let $\tilde{\Omega}^{v-e}$ be the image of $\Omega^{v-e}$ in $x^{v-e}$ coordinates. Thus
$\tilde{\Omega}^{v-e}$ is divided into $N_{v-e} = I_{v-e}(N+1)^2$ hexahedrons
$\{\tilde{\Omega}^{v-e}_n\}_{n=1,\ldots,N_{v-e}}$, where
\[\tilde{\Omega}^{v-e}_n =\left\{x^{v-e}:\;\psi_i^{v-e}<\psi<\psi_{i+1}^{v-e},
\;\theta_j^{v-e}<\theta<\theta_{j+1}^{v-e},\;\zeta_k^{v-e}<\zeta<\zeta_{k+1}^{v-e}\right\}.\]

We now define the spectral element functions on the elements in $\tilde{\Omega}^{v-e}$.
Consider an element
\[\tilde{\Omega}^{v-e}_n =\left\{x^{v-e}:\; \psi_i^{v-e}<\psi<\psi_{i+1}^{v-e},
\;\theta_j^{v-e}<\theta<\theta_{j+1}^{v-e},\;-\infty<\zeta<\zeta_{1}^{v-e}\right\}.\]
Then on $\tilde{\Omega}^{v-e}_n$
\[u_n^{v-e} = h_{v-e} = h_v\]
where $h_v$ is the same constant as for the spectral element function $u_m^v$ defined on the
corner element
\[\tilde{\Omega}^{v}_m = \left\{x^v:\;(\phi,\theta) \in
S_j^v,\; -\infty< \chi< \ln(\rho_1^v)\right\}.\]
Next, we consider the element
\[\tilde{\Omega}^{v-e}_p = \left\{x^{v-e}: -\infty<\psi<
\psi_{1}^{v-e},\;\theta_j^{v-e}<\theta<\theta_{j+1}
^{v-e},\;\zeta_k^{v-e}<\zeta<\zeta_{k+1}^{v-e}\right\}\,.\]
Here $k \geq 1$.
\newline
Then on $\tilde{\Omega}^{v-e}_p$ we define
$$u_p^{v-e}(x^{v-e}) = \sum_{l=0}^{W_p}\:\beta_l\:\zeta^l.$$
Here $1\leq W_p \leq W$. Moreover $W_p=[\mu_{2} k]$ for $1\leq k\leq N$, where $\mu_{2}>0$ is
a degree factor.
\newline
Now consider
\[\tilde{\Omega}^{v-e}_q = \left\{x^{v-e}:\;\psi_i^{v-e}<\psi<\psi_{i+1}^{v-e},\;
\theta_j^{v-e}<\theta<\theta_{j+1}^{v-e},\;\zeta_k^{v-e}<\zeta<\zeta_{k+1}^{v-e}\right\}\]
for $1\leq i \leq N$, $1\leq k \leq N$. Then on
$\tilde{\Omega}^{v-e}_q$ we define
\[u_q^{v-e}(x^{v-e})=\sum_{r=0}^{W_q}\sum_{s=0}^{W_q}\sum_{t=0}^{V_q}\;\gamma_{r,s,t}
\;\psi^r \theta^s \zeta^t.\]
Here $1\leq W_q \leq W$ and $1\leq V_q \leq W$. Moreover $W_q=[\mu_{1}i], V_q=[\mu_{2}k]$
for $1\leq i,k\leq N$, where $\mu_{1},\mu_{2}>0$ are degree factors \cite{G1}.

Let $\tilde{\Gamma}_{n,i}^{v-e}$ be one of the faces of $\tilde{\Omega}_{n}^{v-e}$ such
that $\mu(\tilde{\Gamma}_{n,i}^{v-e})<\infty$, where $\mu$ denotes measure. We introduce
a norm $|||\,u\,|||_{\tilde{\Gamma}_{n,i}^{v-e}}^2$ as follows:
\newline
Let $E_{n,i}^{v-e}=\underset{x^{v-e}\in{\tilde{\Gamma}_{n,i}
^{v-e}}}{sup}(\sin\phi)$ and $F_{n,i}^{v-e}=\underset{x^{v-e}
\in{\tilde{\Gamma}_{n,i}^{v-e}}}{sup}(e^{x_3^{v-e}})$. We
also define $G_{n,i}^{v-e}$ which is used in $(\ref{eq2.29})$.
\\
$1)$ If $\tilde{\Gamma}_{n,i}^{v-e} = \left\{ x^{v-e}:\;
\alpha_{0}< x_1^{v-e} <\alpha_1,\;\beta_{0}< x_2^{v-e}
<\beta_1,\;x_3^{v-e}= \gamma_{0}\right\}$ then define
$G_{n,i}^{v-e}=E_{n,i}^{v-e}$ and
\begin{subequations} \label{eq2.27}
\begin{align}\label{eq2.27a}
|||\,u\,|||_{\tilde{\Gamma}_{n,i}^{v-e}}^2 & = E_{n,i}^{v-e}
F_{n,i}^{v-e}\left(\int_{\beta_0}^{\beta_1}
\int_{\alpha_0}^{\alpha_1}u^2(\psi,\theta,\gamma_0)
\:d\psi\:d\theta\right. \notag\\
&+ \int_{\beta_0}^{\beta_1}\:d\theta \int_{\alpha_0}^{\alpha_1}
\int_{\alpha_0}^{\alpha_1} \frac{\left(u(\psi,\theta,\gamma_0)
-u(\psi^{\prime},\theta,\gamma_0)\right)^2}{(\psi-\psi^{\prime}
)^2}\:d\psi\:d\psi^{\prime} \notag\\
&+\left. \int_{\alpha_0}^{\alpha_1} \:d\psi
\int_{\beta_0}^{\beta_1}\int_{\beta_0}^{\beta_1}
\frac{(u(\psi,\theta,\gamma_0)-u(\psi,\theta^{\prime},
\gamma_0))^2}{(\theta-\theta^{\prime})^2}\:d\theta
\:d\theta^{\prime}\right).
\end{align}
$2)$ If $\tilde{\Gamma}_{n,i}^{v-e} = \left\{ x^{v-e}:
\;x_1^{v-e}=\alpha_0,\;\beta_{0}<x_2^{v-e}<\beta_1,\;
\gamma_{0}<x_3^{v-e}<\gamma_{1}\right\}$ then define
$G_{n,i}^{v-e}=1$ and
\begin{align}\label{eq2.27b}
|||\,u\,|||_{\tilde{\Gamma}_{n,i}^{v-e}}^2 &=F_{n,i}^{v-e}
\left(\int_{\gamma_0}^{\gamma_1} \int_{\beta_0}^{\beta_1}
u^2(\alpha_0,\theta,\zeta) \:d\theta\:d\zeta \right.\notag\\
&+ \int_{\gamma_0}^{\gamma_1} \:d\zeta
\int_{\beta_0}^{\beta_1}\int_{\beta_0}^{\beta_1}
\frac{(u(\alpha_0,\theta,\zeta)-u(\alpha_0,\theta^{\prime},
\zeta))^2}{(\theta-\theta^{\prime})^2}\:d\theta\:d\theta^
{\prime}\notag\\
&+\left.E_{n,i}^{v-e}\int_{\beta_0}^{\beta_1}
\:d\theta\int_{\gamma_0}^{\gamma_1}\int_{\gamma_0}
^{\gamma_1}\frac{(u(\alpha_0,\theta,\zeta)-u(\alpha_0,\theta,
\zeta^{\prime}))^2}{(\zeta-\zeta^{\prime})^2}\:d\zeta
\:d\zeta^{\prime}\right).
\end{align}
$3)$ If $\tilde{\Gamma}_{n,i}^{v-e} = \left\{ x^{v-e}
:\;\alpha_0<x_1^{v-e}<\alpha_1,\;x_2^{v-e} = \beta_0,
\;\gamma_{0}<x_3^{v-e}< \gamma_{1}\right\}$ then define
$G_{n,i}^{v-e}=1$ and
\begin{align}\label{eq2.27c}
|||\,u\,|||_{\tilde{\Gamma}_{n,i}^{v-e}}^2 &= F_{n,i}^{v-e}
\left(\int_{\gamma_0}^ {\gamma_1}\int_{\alpha_0}^{\alpha_1}
u^2(\psi,\beta_0,\zeta) \:d\psi\:d\zeta \right.\notag\\
&+ \int_{\gamma_0}^{\gamma_1}\:d\zeta\int_{\alpha_0}^{\alpha_1}
\int_{\alpha_0}^{\alpha_1}\frac{(u(\psi,\beta_0,\zeta)-u
(\psi^{\prime},\beta_0,\zeta))^2}{(\psi-\psi^{\prime})^2}
\:d\psi\:d\psi^{\prime}\notag\\
&+\left.E_{n,i}^{v-e}\int_{\alpha_0}^{\alpha_1}
\:d\psi\int_{\gamma_0}^{\gamma_1}\int_{\gamma_0}
^{\gamma_1}\frac{(u(\psi,\beta_0,\zeta)-u(\psi,\beta_0,
\zeta^{\prime}))^2}{(\zeta-\zeta^{\prime})^2}
\:d\zeta\:d\zeta^{\prime}\right).
\end{align}
\end{subequations}
Let $L^{v-e}$ be a differential operator such that
$$\int_{\tilde{\Omega}_n^{v-e}}\left|\:L^{v-e}u(x^{v-e})
\:\right|^2\:dx^{v-e} = \int_{\Omega_l^v}\rho^2\sin^2\phi
\:|Lu(x)|^2\:dx\:.$$
Here
$dx^{v-e}$ denotes a volume element in $x^{v-e}$ coordinates and
$dx$ a volume element in $x$ coordinates. In Chapter $3$ it will be
been shown that
\begin{equation}\label{eq2.28}
L^{v-e} u(x^{v-e}) = -div_{x^{v-e}}\left(e^{\zeta/2}A^{v-e}\nabla_{x^{v-e}}
u\right) + \sum_{i=1}^3 \hat{b}_i^{v-e}u_{x_i^{v-e}} + \hat{c}^{v-e} u.
\end{equation}
Here $A^{v-e}$ is a symmetric, positive definite matrix.
\newline
We now define the quadratic form
\begin{align}\label{eq2.29}
\mathcal V_{v-e}^{N,W}(\{\mathcal F_u\}) &= \sum_{l=1,
\mu(\tilde{\Omega}_l^{v-e})<\infty}^{N_{v-e}}
\int_{\tilde{\Omega}_l^{v-e}}\left|\:L^{v-e} u_l^{v-e}(x^{v-e})
\:\right|^2 \:dx^{v-e} \notag\\
&+ \mathop{\sum_{\Gamma_{n,i}^{v-e}\subseteq\bar{\Omega}^{v-e}
\setminus\partial{\Omega},}}_{\mu(\tilde{\Gamma}_{n,i}^{v-e})
<\infty}\left(\left\|\:\sqrt{F_{n,i}^{v-e}G_{n,i}^{v-e}}\;[u]
\:\right\|_{0,\tilde{\Gamma}_{n,i}^{v-e}}^2
+\big|\big|\big|\;[u_{x_1^{v-e}}]\;\big|\big|\big|_{\tilde
{\Gamma}_{n,i}^{v-e}}^2\right. \notag\\
&+\big|\big|\big|\:[u_{x_2^{v-e}}]\:\big|\big|
\big|_{\tilde{\Gamma}_{n,i}^{v-e}}^2+
\big|\big|\big|\:E_{n,i}^{v-e}\,[u_{x_3^{v-e}}]
\:\big|\big|\big|_{\tilde{\Gamma}_{n,i}^{v-e}}^2\Bigg) \notag\\
&+\mathop{\sum_{{\Gamma_{n,i}^{v-e}}\subseteq{\Gamma^{[0]}},}}
_{\mu(\tilde{\Gamma}_{n,i}^{v-e})<\infty}
\left(\left\|\,\sqrt{F_{n,i}^{v-e}}\;u_n^{v-e}\,
\right\|_{0,{\tilde{\Gamma}_{n,i}^{v-e}}}^2
+\big|\big|\big|\:u_{x_1^{v-e}}\:\big|\big|\big|^2_{\tilde
{\Gamma}_{n,i}^{v-e}}\right. \notag\\
&+\big|\big|\big|\:E_{n,i}^{v-e}
\,u_{x_3^{v-e}}\:\big|\big|\big|^2_{\tilde{\Gamma}_{n,i}
^{v-e}}\Bigg)+\mathop{\sum_{{\Gamma_{n,i}^{v-e}}\subseteq
{\Gamma^{[1]}},}}_{\mu(\tilde{\Gamma}_{n,i}^{v-e})<\infty}
\Big|\Big|\Big|\:\left(\frac{\partial u}{\partial\nuw^{v-e}}
\right)_{A^{v-e}}\:\Big|\Big|\Big|_{\tilde
{\Gamma}_{n,i}^{v-e}}^2.
\end{align}
Once more $\mu$ denotes measure.\\
Here the term
$\left(\frac{\partial u}{\partial \nuw^{v-e}}\right)_{A^{v-e}}$ is defined as
follows. Let $\tilde{\Gamma}_{n,i}^{v-e}$ be a face of $\tilde{\Omega}_{n,i}^{v-e}$,
$\tilde P$ be a point belonging to $\tilde{\Gamma}_{n,i}^{v-e}$ and $\nuw^{v-e}$
denote the unit normal to $\tilde{\Gamma}_{n,i}^{v-e}$ at the point $\tilde P$. Then
\begin{equation}\label{eq2.30}
\left(\frac{\partial u}{\partial \nuw^{v-e}}\right)_{A^{v-e}}
(\tilde P)= (\nuw^{v-e})^T A^{v-e}\nabla_{x^{v-e}}u.
\end{equation}
Now the quadratic form $\mathcal V_{vertex-edges}^{N,W}
(\{\mathcal F_u\})$ is given by
\begin{equation}\label{eq2.31}
\mathcal V_{vertex-edges}^{N,W}(\{\mathcal F_u\})=\sum_{{v-e}\in\mathcal {V-E}}
\mathcal V_{v-e}^{N,W} (\{\mathcal {F}_u\}).
\end{equation}
Next, we define the quadratic form
$\mathcal U_{v-e}^{N,W}(\{\mathcal F_u\})$. Let $w^{v-e}(x_1^{v-e})$ be a positive, smooth
weight function such that
\begin{eqnarray*}
w^{v-e}(x_1^{v-e}) = 1 \quad for \quad  x_1^{v-e} \geq
\zeta_1^{v-e}=\ln(\tan\phi_1^{v-e})
\end{eqnarray*}
and which satisfies
\[\int_{-\infty}^{\zeta_1^{v-e}}w^{v-e}
(x_1^{v-e})\:dx_1^{v-e} = 1.\]
We shall choose
$$w^{v-e}(x_1^{v-e}) = 1 \quad for \quad  x_1^{v-e} \geq \zeta_1^{v-e}-1$$
and
$$w^{v-e}(x_1^{v-e}) = 0 \quad for \quad  x_1^{v-e} < \zeta_1^{v-e}-1\:.$$
Then
\begin{align}\label{eq2.32}
\mathcal U_{v-e}^{N,W}(\{\mathcal F_u\}) &=\sum_{l=1,\mu(\tilde
{\Omega}_l^{v-e})<\infty}^{N_{v-e}}
\int_{\tilde{\Omega}_l^{v-e}}e^{x_3^{v-e}}
\left(\sum_{i,j=1,2}\left(\frac{\partial^2u_l^{v-e}} {\partial
x_i^{v-e}\partial x_j^{v-e}}\right)^2 \right. \notag\\
&+ \sum_{i=1}^2\sin^2\phi\left(\frac{\partial^2u_l^{v-e}}
{\partial x_i^{v-e} \partial x_3^{v-e}}\right)^2
+\sin^4\phi\left(\frac{\partial^2 u_l^{v-e}}{\left (\partial
x_{3}^{v-e}\right)^2}\right)^2 + \sum_{i=1}^2\left(\frac{\partial
u_l^{v-e}}{\partial x_i^{v-e}}\right)^2 \notag\\
&\left.+\sin^2\phi\left(\frac{\partial
u_l^{v-e}}{\partial x_3^{v-e}}\right)^2
+(u_l^{v-e})^2\right)\:dx^{v-e} \notag\\
&+\mathop{\sum_{l=1,}^{N_{v-e}}}_{\mu(\tilde{\Omega}_l^{v-e})=\infty}
\int_{\tilde{\Omega}_l^{v-e}}(u_l^{v-e})^2e^{x_3^{v-e}}
\:w^{v-e}(x_1^{v-e})\:dx^{v-e}\,.
\end{align}
The quadratic form $\mathcal U_{vertex-edges}^{N,W}
(\{\mathcal F_u\})$ is then given by
\begin{equation}\label{eq2.33}
\mathcal U_{vertex-edges}^{N,W}(\{\mathcal F_u\}) = \sum_{{v-e} \in
\mathcal {V-E}}\mathcal U_{v-e}^{N,W}(\{\mathcal F_u\}).
\end{equation}

Finally, we define the quadratic forms
$\mathcal V_{edges}^{N,W} \left(\{\mathcal F_u\}\right)$ and
$\mathcal U_{edges}^{N,W} \left(\{\mathcal F_u\}\right)$.
\begin{figure}[!ht]
\centering
\includegraphics[scale = 0.60]{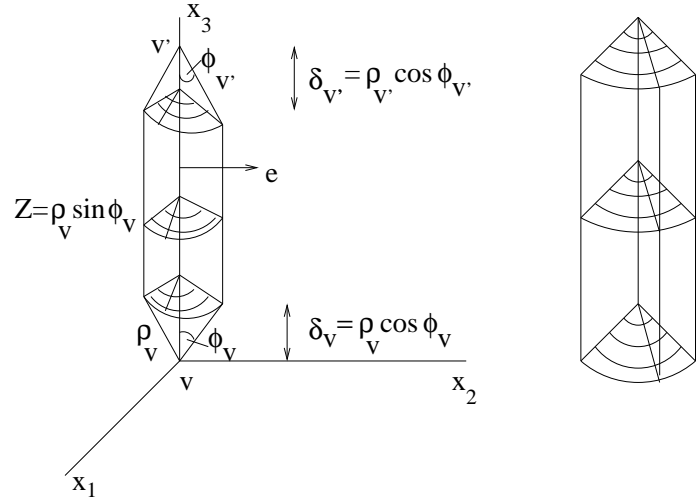}
\caption{ Geometrical mesh imposed on $\Omega^e$. }
\label{fig2.10}
\end{figure}
Consider the edge $e$ whose end points are $v$ and $v^{\prime}$. The edge $e$ coincides with
the $x_3$ axis and the vertex $v$ with the origin. Let the length of the edge $e$ be $l_e$.
Now the edge neighbourhood $\Omega^e$ is defined as
\[\Omega^e = \left\{x\in\Omega:\: 0<r<Z=\rho_v\sin\phi_v,\:\theta_{v-e}^l
<\theta<\theta_{v-e}^u,\:\delta_v<x_3<l_e-\delta_v^{\prime}\right\}\,.\]
Here $(r,\theta,x_3)$ denote cylindrical coordinates with origin at $v$,
$\delta_v=\rho_v\cos\phi_v$ and $\delta_v^{\prime}=\rho_v^{\prime}\cos\phi_v^{\prime}$
are as shown in Figure~\ref{fig2.10}.

A geometrical mesh in $r$ is imposed by defining $r_{0}^e=0$ and $r_j^e=Z(\mu_e)^{N+1-j}$
for $j=1,2,\ldots,N+1.$ We impose the same quasi-uniform mesh on $\theta$ as we did in the
vertex-edge neighbourhood, viz.
\[\theta_{v-e}^l = \theta^{v-e}_0 < \theta^{v-e}_1 <\cdots<
\theta^{v-e}_{I^{e}} = \theta_{v-e}^u\,.\] Here $I_e=I_{v-e}$ and
$\theta_k^e=\theta_k^{v-e}$ for $0\leq k\leq I_e$. A
quasi-uniform mesh is defined in $x_3$,\;by choosing
\[\delta_v=Z_0^e<Z_1^e<\cdots<Z_{J_e}^e
= l_e-\delta_v^{\prime}\,.\] Thus $\Omega^e$ is divided into $N_e=
I_e\,J_e\,(N+1)$ elements. We introduce new coordinates in $\Omega^e$ by
\begin{align*}
x_1^e &= \tau= \;\ln r\\
x_2^e &= \theta \\
x_3^e &= x_3 \,.
\end{align*}
Let $\tilde{\Omega}^e$ be the image of $\Omega^e$ in $x^e$ coordinates. Thus
$\tilde{\Omega}^e$ is divided into $N_e$ hexahedrons $\{\tilde{\Omega}^e_m\}_{m=1,\ldots,N_e}$
where
$$\tilde{\Omega}^e_m=\left\{x^e:\:\ln(r_i^e)<x_1^e
<\ln(r_{i+1}^e),\:\theta_j^e<x_2^e<\theta_{j+1}^e,
\:Z_k^e<x_3^e<Z_{k+1}^e\right\}\;.$$
We now define the spectral element functions on the elements in $\tilde{\Omega}^e$.
Consider an element
\[\tilde{\Omega}^e_p = \left\{x^e:\: -\infty<x_1^e<\ln(r_{1}^e),
\:\theta_j^e<x_2^e<\theta_{j+1}^e,\:Z_n^e<x_3^e<Z_{n+1}^e\right\}\;.\]
Then
\[u_p^e(x^e) = \sum_{t=0}^W\:\alpha_r(x_3^e)^t\:.\]
This representation is valid for all $j$ for fixed $n$.
\newline
Next, consider the element
\[\tilde{\Omega}^e_q = \left\{x^e:\: \ln (r_i^e)<x_1^e<\ln (r_{i+1}^e),
\:\theta_j^e<x_2^e<\theta_{j+1}^e,\:Z_n^e<x_3^e<Z_{n+1}^e\right\}\]
for $1\leq i\leq N\,.$
\newline
Then we define
\[u_q^e(x^e) = \sum_{r=0}^{W_q} \sum_{s=0}^{W_q}\sum_{t=0}^W \:\alpha_{r,s,t}
\:(x_1^e)^r(x_2^e)^s(x_3^e)^t\;.\]
Here $1\leq W_q\leq W$. Moreover $W_q=[\mu_{1}i]$ for all $1\leq i\leq N$, where
$\mu_{1}>0$ is a degree factor~\cite{G1}.
\newline
Let $\tilde{\Gamma}^e_{m,i}$ be one of the faces of
$\tilde{\Omega}^e_m$ such that $\mu(\tilde{\Gamma}_{m,i}^e)
<\infty$, where $\mu$ denotes measure. We define a norm
$|||\,u\,|||^2_{\tilde{\Gamma}^e_{m,i}}$ as follows:
\newline
Let $G_{m,i}^{e}=\underset{x^e\in{\tilde{\Gamma}_{m,i}^e}}{sup}\left(e^\tau\right)$.
We also define $H_{m,i}^e$ which is needed in (\ref{eq2.38}).
\newline
$1)$ If $\tilde{\Gamma}_{m,i}^e = \left\{ x^e:\:\alpha_{0}<x_1^e
<\alpha_1,\:\beta_{0}< x_2^e <\beta_1,\: x_3^e=\gamma_{0}\right\}$
then define $H_{m,i}^{e}=G_{m,i}^{e}$ and
\begin{subequations}\label{eq2.34}
\begin{align}\label{eq2.34a}
|||\,u\,|||_{\tilde{\Gamma}_{m,i}^e}^2
&=G_{m,i}^{e}\left(\int_{\beta_0}^{\beta_1} \int_{\alpha_0}
^{\alpha_1}u^2(\tau,\theta,\gamma_0)\:d\tau\:d\theta \right. \notag\\
&+ \int_{\beta_0}^{\beta_1} d\theta \int_{\alpha_0}^{\alpha_1}
\int_{\alpha_0}^{\alpha_1}\frac{(u(\tau,\theta,\gamma_0)-u(\tau^{\prime},
\theta,\gamma_0))^2}{(\tau-\tau^{\prime})^2}\:d\tau \:d\tau^{\prime} \notag\\
&+ \left.\int_{\alpha_0}^{\alpha_1}\:d\tau\int_{\beta_0}^{\beta_1}
\int_{\beta_0}^{\beta_1}\frac{(u(\tau,\theta,\gamma_0)
-u(\tau,\theta^{\prime},\gamma_0))^2}{(\theta-\theta^{\prime})^2}
\:d\theta\:d\theta^{\prime}\right)\;.
\end{align}
$2)$ If $\tilde{\Gamma}_{m,i}^{e} =\{ x^e:\; x_1^e =
\alpha_0,\;\beta_{0}< x_2^e <\beta_1,\; \gamma_{0}<
x_3^e<\gamma_{1}\}$ then define $H_{m,i}^{e}=1$ and
\begin{align}\label{eq2.34b}
|||\,u\,|||_{\tilde{\Gamma}_{m,i}^e}^2 &=
\left(\int_{\gamma_0}^{\gamma_1} \int_{\beta_0}^{\beta_1}
u^2(\alpha_0,\theta,x_3) \:d\theta\:dx_3\right. \notag\\
&+ \int_{\gamma_0}^{\gamma_1} \:dx_3\int_{\beta_0}^{\beta_1}
\int_{\beta_0}^{\beta_1}\frac{\left(u(\alpha_0,\theta,x_3)
-u(\alpha_0,\theta^{\prime},x_3)\right)^2}{(\theta-\theta
^{\prime})^2}\:d\theta\:d\theta^{\prime} \notag\\
&+G_{m,i}^{e}\left.\int_{\beta_0}^{\beta_1}
\:d\theta \int_{\gamma_0}^{\gamma_1}\int_{\gamma_0}^
{\gamma_1}\frac{(u(\alpha_0,\theta,x_3)-u(\alpha_0,
\theta,x_3^{\prime}))^2}{(x_3-x_3^{\prime})^2}
\:dx_3\:dx_3^{\prime}\right).
\end{align}
$3)$ If $\tilde{\Gamma}_{m,i}^e = \left\{ x^e: \;\alpha_0<
x_1^e<\alpha_1,\:x_2^e=\beta_0,\:\gamma_{0}< x_3^e<
\gamma_{1}\right\}$ then define $H_{m,i}^{e}=1$ and
\begin{align}\label{eq2.34c}
|||\,u\,|||_{\tilde{\Gamma}_{m,i}^e}^2 &=
\left(\int_{\gamma_0}^{\gamma_1} \int_{\alpha_0}
^{\alpha_1}u^2(\tau,\beta_0,x_3)\:d\tau\:dx_3 \right. \notag\\
&+ \int_{\gamma_0}^{\gamma_1}\:dx_3
\int_{\alpha_0}^{\alpha_1}\int_{\alpha_0}^{\alpha_1}
\frac{(u(\tau,\beta_0,x_3)-u(\tau^{\prime},
\beta_0,x_3))^2}{(\tau-\tau^{\prime})^2}\:d\tau\:d\tau^{\prime} \notag\\
&+\left. G_{m,i}^{e}\int_{\alpha_0}^{\alpha_1}\:d\tau
\int_{\gamma_0}^{\gamma_1} \int_{\gamma_0}^{\gamma_1}
\frac{(u(\tau,\beta_0,x_3)-u(\tau,\beta_0,x_3^{\prime}))^2}
{(x_3-x_3^{\prime})^2}\:dx_3\:dx_3^{\prime}\right)\;.
\end{align}
\end{subequations}
Let $L^e$ be a differential operator such that
\begin{equation}\label{eq2.35}
\int_{\tilde{\Omega}_m^e} |L^eu(x^e)|^2 dx^e = \int_
{\Omega_m^e} r^2 |Lu(x)|^2 dx\:.
\end{equation}
Here $dx^e$ denotes a volume element in $x^e$ coordinates and
$dx$ a volume element in $x$ coordinates.
\newline
In Chapter $3$ it will be shown that
\begin{equation}\label{eq2.36}
L^e u(x^e) = -div_{x^e}\left(A^e\nabla_{x^e}u\right)
+\sum_{i=1}^3 \hat{b}_i^e u_{x_i^e} + \hat{c}^e u
\end{equation}
where $A^e$ is a symmetric, positive definite matrix.
\newline
Let $\tilde{\Gamma}_{m,i}^e$ be one of the sides of
$\tilde{\Omega}_{m}^e$ and $\tilde P$ a point belonging to
$\tilde{\Gamma}_{m,i}^e$. Let $\nuw^e$ be the normal to
$\tilde{\Gamma}_{m,i}^e$ at $\tilde P$. Then
\begin{equation}\label{eq2.37}
\left(\frac {\partial u}{\partial \nuw^e}\right)_{A^e}
(\tilde P)=(\nu^e)^T A^e\nabla_{x^e}u(\tilde P)\:.
\end{equation}
We now define the quadratic form
\begin{align}\label{eq2.38}
\mathcal V_{e}^{N,W}(\{\mathcal F_u\}) &=
\sum_{l=1,\mu(\tilde{\Omega}_l^{e})<\infty}^{N_e}
\int_{\tilde{\Omega}_{l}^e}|L^e u_l^e(x^e)|^2\:dx^e \notag\\
&+\mathop{\sum_{\Gamma_{l,i}^e\subseteq\bar{\Omega}^e\setminus
\partial\Omega,}}_{\mu(\tilde{\Gamma}_{l,i}^{e})<\infty}
\left(\left\|\,\sqrt{H_{l,i}^e}\:[u]\,\right\|_{0,\tilde
{\Gamma}_{l,i}^e}^2
+\big|\big|\big|\,[u_{x_1^e}]\,\big|\big|\big|_{\tilde
{\Gamma}_{l,i}^e}^2 \right. \notag\\
&\left.+\big|\big|\big|\,[u_{x_2^e}]\,\big|\big|\big|
_{\tilde{\Gamma}_{l,i}^e}^2+\big|\big|\big|\,G_{l,i}^{e}
[u_{x_3^e}]\,\big|\big|\big|_{\tilde{\Gamma}_{l,i}^e}^2\right) \notag\\
&+ \mathop{\sum_{{\Gamma_{l,i}^e}\subseteq{\Gamma^{[0]}},}}_{\mu
(\tilde{\Gamma}_{l,i}^{e})<\infty}
\left(\left\|\:u_l^e\,\right\|_{0,{\tilde{\Gamma}
_{l,i}^e}}^2+\big|\big|\big|\,u_{x_1^e}\,\big|\big|\big|^2_{\tilde
{\Gamma}_{l,i}^e}+\big|\big|\big|\,G_{l,i}^{e}\,u_{x_3^e}\,
\big|\big|\big|^2_{\tilde{\Gamma}_{l,i}^e}\right) \notag\\
&+\mathop{\sum_{{\Gamma_{l,i}^e}\subseteq{\Gamma^{[1]}},}}_{\mu
(\tilde{\Gamma}_{l,i}^{e})<\infty}\Big|\Big|\Big|\left(\frac{\partial u}
{\partial\nuw^e}\right)_{A^e}\Big|\Big|\Big|_{\tilde{\Gamma}_{l,i}^e}^2.
\end{align}
The quadratic form $\mathcal V_{edges}^{N,W}(\{\mathcal F_u\})$
is given by
\begin{equation}\label{eq2.39}
\mathcal V_{edges}^{N,W}(\{\mathcal F_u\})=\sum_{e\in\mathcal E}
\mathcal V_e^{N,W}(\{\mathcal F_u\}).
\end{equation}
Next, let us define the quadratic form
$\mathcal U_e^{N,W}(\{\mathcal F_u\})$. Let $w^e(x_1^e)$ be a positive, smooth
weight function such
that
\begin{eqnarray*}
w^e(x_1^e) = 1 \quad {\it for} \quad  x_1^e \geq \ln (r_1^e)
\end{eqnarray*}
and
$$\int_{-\infty}^{\ln (r_1^e)}w^e(x_1^e)\; dx_1^e = 1\;.$$
We shall choose
$$w^e(x_1^e)=1\quad{\it for}\quad x_1^e\geq\ln(r_1^e)-1$$
and
$$w^e(x_1^e)=0\quad{\it for}\quad  x_1^e <\ln (r_1^e)-1\:.$$
\newline
Then
\begin{align}\label{eq2.40}
\mathcal U_e^{N,W}(\{\mathcal F_u\}) &= \sum_{l=1,\mu(\tilde
{\Omega}_l^{e})<\infty}^{N_e}
\int_{\tilde{\Omega}_l^e}\left(\sum_{i,j=1,2}\left(\frac
{\partial^2 u_l^e} {\partial x_i^e \partial x_j^e}\right)^2 +
e^{2\tau}\sum_{i=1}^2\left(\frac{\partial^2
u_l^e}{\partial x_i^e \partial x_3^e}\right)^2 \right. \notag\\
&+\left. e^{4\tau}\left(\frac{\partial^2 u_l^e}{\left(\partial
x_{3} ^e\right)^2}\right)^2 + \sum_{i=1}^2\left(\frac{\partial
u_l^e}{\partial x_i^e}\right)^2+ e^{2\tau}\left(\frac{\partial
u_l^e}{\partial x_3^e}\right)^2+(u_l^e)^2\right)\:dx^e \notag\\
&+\mathop{\sum_{l=1,}^{N_e}}_{\mu(\tilde{\Omega}_l^{e})=\infty}
\int_{\tilde{\Omega}_l^e}(u_l^e)^2\:w^e(x_1^e)\:dx^e\:.
\end{align}
Here $\mu$ denotes measure.
\newline
We define
\begin{equation}\label{eq2.41}
\mathcal U_{edges}^{N,W}(\{\mathcal F_u\}) = \sum_{e \in \mathcal
E} \mathcal U_e^{N,W}(\{\mathcal F_u\})\:.
\end{equation}
Finally, using $(\ref{eq2.14})$ and $(\ref{eq2.15})$ we can define the
quadratic forms
$\mathcal V^{N,W}(\{\mathcal F_u\})$ and $\mathcal U^{N,W}
(\{\mathcal F_u\}).$
\newline
We now state the main result of this chapter. It is assumed that $N$ is
proportional to $W$.
\begin{theo}\label{thm2.3.1}
Consider the elliptic boundary value problem $(\ref{eq2.1})$. Suppose
the boundary conditions are Dirichlet. Then
\begin{align}\label{eq2.42}
\mathcal U^{N,W}(\{\mathcal F_u\}) \leq C (\ln W)^2 \mathcal V^{N,W}
(\{\mathcal F_u\})\:.
\end{align}
\end{theo}
Next, we state the corresponding result for general boundary conditions.
\begin{theo}\label{thm2.3.2}
If the boundary conditions for the elliptic boundary value problem
$(\ref{eq2.1})$ are mixed then
\begin{align}\label{eq2.43}
\mathcal U^{N,W}(\{\mathcal F_u\}) \leq C N^4 \mathcal V^{N,W}
(\{\mathcal F_u\})
\end{align}
provided $W=O(e^{N^\alpha})$ for $\alpha<1/2$.
\end{theo}
The rapid growth of the factor $CN^4$ with $N$ creates difficulties in
parallelizing the numerical scheme. To overcome this problem we state a
version of Theorem \ref{thm2.3.2} when the spectral element functions vanish
on the wirebasket. Let $W\!B$ denotes the wirebasket along which the
spectral element functions need to be conforming. Here the wirebasket
denotes the union of the vertices and edges of the elements.
\begin{theo}\label{thm2.3.3}
If the spectral element functions $(\{\mathcal F_u\})$ are conforming on the wire
basket $W\!B$ and vanish on $W\!B$ then
\begin{align}\label{eq2.44}
\mathcal U^{N,W}(\{\mathcal F_u\}) \leq C (\ln W)^2 \mathcal V^{N,W}
(\{\mathcal F_u\})
\end{align}
provided $W=O(e^{N^\alpha})$ for $\alpha<1/2$.
\end{theo}

\chapcleardoublepage

\chapter{Proof of the Stability Theorem}
\section{Introduction}
In this chapter we prove the stability estimate, which we use to formulate our numerical scheme.

In Chapter $2$ we have introduced a set of local coordinate systems in various neighbourhoods of
the polyhedron $\Omega$, which we referred to as \textit{modified coordinates}, in the vertex,
vertex-edge and edge neighbourhoods. These local coordinates are modified versions of the spherical
and cylindrical co-ordinate systems in vertex and edge neighbourhoods respectively. In vertex-edge
neighbourhoods modified coordinates are a combination of spherical and cylindrical coordinates.
Away from the corners and edges in the regular region of the domain we use standard Cartesian
coordinates. The \textit{differentiability estimates} for the solution of the elliptic boundary
value problem in a polyhedral domain in terms of these new local coordinates were obtained. We
also stated a \textit{stability estimate} for a \textit{non-conforming} spectral element
representation of the solution. For problems with mixed Neumann and Dirichlet boundary conditions
the spectral element functions may be chosen to be continuous on the wirebasket of the elements.

We now briefly describe the contents of this chapter. In Sections $3.2$ and $3.3$, we derive
estimates for the second order derivatives and the lower order derivatives of the solution
respectively. Estimates for terms in the interior and on the boundary of the polyhedron
$\Omega$ are obtained in Sections $3.4$ and $3.5$. In Section $3.6$, we combine all the
results of sections $3.2,3.3,3.4$ and $3.5$ to complete the proof of the stability estimate
which has been stated in Chapter $2$.

In Chapter $4$ we shall give the numerical scheme and error estimates. It will be shown that
the error between the exact and the approximate solution is exponentially small in $N$, the
number of layers in the geometrical mesh. Optimal rate of convergence with respect to the number
of degrees of freedom is also provided.

\section{Estimates for the Second Derivatives of Spectral Element Functions}
\subsection{Estimates for the second derivatives in the interior}
We first obtain  estimates for elements in the regular region $\Omega^r$ of $\Omega$.
Divide $\Omega^r$ into $N_r$ elements, which may consist of curvilinear cubes, prisms
and tetrahedrons, $\Omega_l^r$ for $1\leq l \leq N_r$.

Let
\begin{equation}\label{eq3.1}
Lu = \sum_{i,j=1}^{3} -\left(a_{i,j} u_{x_j}\right)_{x_i} +
\sum_{i=1}^{3} b_i u_{x_i} + cu
\end{equation}
be a strongly elliptic operator that satisfies the Lax-Milgram conditions. Hence there
exists a positive constant $\mu_0$ such that
\[\sum_{i,j=1}^3 a_{i,j}\zeta_i\zeta_j\geq \mu_0|\zeta|^2\:.\]
We consider the mixed boundary value problem
\begin{subequations}\label{eq3.2}
\begin{equation}\label{eq3.2a}
Lw = f \;{\mbox{ in }} \Omega,
\end{equation}
\begin{equation}\label{eq3.2b}
\gamma_0 w = w|_{\Gamma^{[0]}} = g^{[0]},
\end{equation}
\begin{equation}\label{eq3.2c}
\gamma_1 w=\left.\left(\frac{\partial w}{\partial\nuw}\right)_{A}\right|_{\Gamma^{[1]}}=g^{[1]}.
\end{equation}
\end{subequations}
Let
\begin{equation}\label{eq3.3}
Mu = \sum_{i,j = 1}^{3} \left(a_{i,j} u_{x_j}\right)_{x_i}.
\end{equation}
Consider an element $\Omega_l^r$. Then $\Omega_l^r$ has $n_l$ faces
$\{\Gamma_{l,i}^r\}_{1 \leq i \leq n_l}$. Let $\partial\Gamma_{l,i}^r$ denote the boundary of
the face $\Gamma_{l,i}^r$.

To proceed we need to review some material in~\cite{G}. Let $O$ be a bounded open subset of
$R^3$ with a Lipschitz boundary $\partial O$. Assume in addition that $\partial O$ is piecewise
$C^2$. Let $P$ be a point on $\partial O$ in a neighbourhood of which $\partial O$ is $C^2$. It
is possible to find two curves of class $C^2$ in a neighbourhood of $P$, passing through $P$ and
being orthogonal there. Let us denote these curves by ${\mathcal C}_1, {\mathcal C}_2$ and by
{\mbox{\boldmath $\tau_1$}}, {\mbox{\boldmath $\tau_2$}} the unit tangent vectors to
${\mathcal C}_1, {\mathcal C}_2$ respectively and by $s_1,s_2$ the arc lengths along these curves.
We assume that {\mbox{\boldmath $\tau_1$}},{\mbox{\boldmath $\tau_2$}} has the direct orientation
at $P$. Let {\mbox{\boldmath $\nu$}} be the unit normal at $P$ defined as {\mbox{\boldmath $\nu$}}
$=$ {\mbox{\boldmath $\tau_1$}} $\times$ {\mbox{\boldmath $\tau_2$}}.
Then at $P$, ${\mathcal B}_P$, the second fundamental form at $P$ is the bilinear form
\begin{equation}\label{eq3.4}
(\zetau,\etav) |\rightarrow
-\sum_{j,k=1}^{2}\frac{\partial \nuw}
{\partial s_j} \cdot \taux_k \zeta_j \eta_k
\end{equation}
where $\zetau, \etav$ are the tangent vectors to $\partial O$ at $P$
and $\zetau = (\zeta_1, \zeta_2)$ and $\etav = (\eta_1, \eta_2)$ in
the basis $\{\taux_1, \taux_2 \}$. In other words,
\begin{equation}\label{eq3.5}
{\mathcal B}(\zetau, \etav) = -\frac{\partial \nuw}{\partial \zetau}\cdot \etav
\end{equation}
where $\frac{\partial}{\partial \zetau}$ denotes differentiation in the $\zetau$ direction.

Let $\wa$ be a vector field defined in a neighbourhood of $O$. If $P$ is a point on
$\partial O$ then by $\wa_{\nu}$ we shall denote the component of $\wa$ in the direction
of $\nuw$, while we shall denote by $\wa_T$, the projection of $\wa$ on the tangent
hyperplane to $\partial O$, i.e.
\begin{equation}\label{eq3.6}
\wa_{\nu} = \wa \cdot \nuw,
\end{equation}
\begin{equation}\label{eq3.7}
\wa_T = \wa-\wa_{\nu} \nuw = \wa_{\tau_1}\taux_1+\wa_{\tau_2}\taux_2.
\end{equation}
Here $\wa_{\tau_i} = \wa \cdot \taux_i$ for $i=1,2$.

We shall denote by $\nabla_T$ the projection of the gradient vector on the tangent
hyperplane
\begin{equation}\label{eq3.8}
\nabla_T u = \nabla u - \frac{\partial u}{\partial \nuw} \nuw =
\sum_{j=1}^2 \frac{\partial u}{\partial s_j} \taux_j.
\end{equation}
Let $\hd$ be a vector field defined on $\partial O$ such that $\hd$ is tangent to $O$
except on a set of zero measure. Then
\begin{equation}\label{eq3.9}
div_T(\hd)=\sum_{j=1}^2\left(\frac{\partial {\hd}}{\partial s_j}
\right) \cdot \taux_j.
\end{equation}
We now cite Theorem $3.1.1.2$ of~\cite{G}.
\begin{theo}\label{thm1}
Let $O$ be a bounded open subset of $R^3$ with a Lipschitz boundary $\partial O$. Assume,
in addition that $\partial O$ is piecewise $C^2$. Then for all $\wa \in (H^2(O))^3$ we have
\begin{align}\label{eq3.10}
\int_O \left(div (\wa)\right)^2 dx - \sum_{i,j=1}^{3} \int_O \frac
{\partial w_i}{\partial x_j}\frac{\partial w_j}{\partial x_i}\:dx
=\int_{\partial O} \left\{div_T(\wa_{\nu}\wa_{T}) - 2 \wa_T \cdot
\nabla_T \wa_{\nu} \right\} \:d\sigma  \notag\\
-\int_{\partial O} \left\{(tr {\mathcal B})\wa_{\nu}^2 + {\mathcal
B}(\wa_T, \wa_T) \right\} \:d\sigma.
\end{align}
Here $dx$ denotes a volume element and $d\sigma$ an element of surface area.
\end{theo}
Consider an element $\Omega^r_l$ which is assumed to be a curvilinear cube as shown in
Figure \ref{fig3.1}. Let $\Gamma_{l,i} ^r$ denote one of the faces of $\Omega^r_l$. Let
$Q$ be a point inside $\Gamma^r_{l,i}$. The unit tangent vectors $\taux_1,\taux_2$ and
the unit normal vector $\nuw$ at $Q$ are shown in Figure \ref{fig3.1}. Consider a point
$P \in \partial \Gamma^r_{l,i}$ and assume that $P$ is not a vertex of $\Omega^r_l$. Then
we can define the vector $\nb$ at $P$ as the vector belonging to the tangent hyperplane
which is orthogonal to the tangent vector to the curve $\partial\Gamma^r_{l,i}$ at $P$.
\begin{figure}[!ht]
\centering
\includegraphics[scale = 0.60]{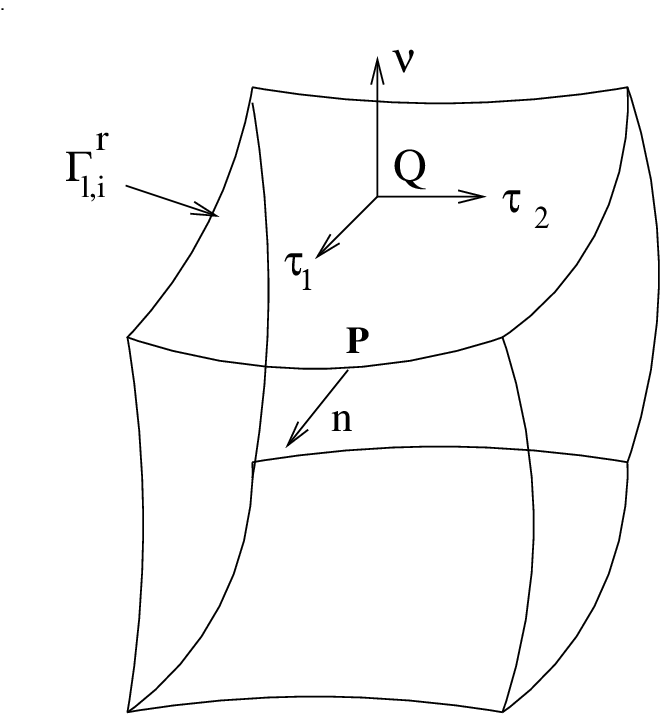}
\caption{The element $\Omega_l^r$.}
\label{fig3.1}
\end{figure}
Moreover $\nb$ is chosen to point out of $\Gamma_{l,i}^r$. Recall $A$ is the matrix
$(a_{i,j})$. Define
\begin{align}\label{eq3.11}
\left(\frac{\partial u}{\partial \nb}\right)_A(P)
&= \left(\nb\cdot A \nabla u \right)(P),\notag \\
\left(\frac{\partial u}{\partial \taux_1}\right)_A(P)
&= \left(\taux_1 \cdot A \nabla u \right)(P), \notag\\
\left(\frac{\partial u}{\partial \taux_2}\right)_A(P)
&= \left(\taux_2 \cdot A \nabla u\right)(P), \notag\\
\left(\frac{\partial u}{\partial \nuw}\right)_A(P) &= \left(\nuw
\cdot A \nabla u\right)(P).
\end{align}
Let $s_1,s_2$ denote arc lengths along $\taux_1$ and $\taux_2$ and $s$ denote arc
length measured along $\partial \Gamma_{l,i}^r$.

We can now prove the following result.
\begin{lem}\label{lem1}
Let $u \in H^3(\Omega^r_l)$. Then
\begin{align}\label{eq3.12}
& \frac{\mu_0^2}{2} \rho_v^2 \sin^2(\phi_v)\int_{\Omega^r_l}
\sum_{i,j=1}^{3} \left|\:\frac{\partial^2 u}{\partial x_i
\partial x_j}\:\right|^2\:dx \notag\\
&\quad\quad\quad\leq \rho_v^2\sin^2(\phi_v)\int_{\Omega^r_l}|L
u|^2 d x - \rho_v^2\sin^2(\phi_v)\left\{\sum_{i}\oint_{\partial
\Gamma_{l,i}^r}\left(\frac{\partial u}{\partial \nb}\right)_A
\left(\frac{\partial u}{\partial \nuw}\right)_A \:ds\right. \notag\\
&\quad\quad\quad-2 \left.\sum_{i} \int_{\Gamma_{l,i}^r}
\sum_{j=1}^2 \left(\frac{\partial u}{\partial \taux_j}\right)_A
\frac{\partial}{\partial s_j}\left(\left(\frac{\partial u
}{\partial \nuw}\right)_A\right)\:d\sigma\right\} \notag\\
&\quad\quad\quad+ C\,\int_{\Omega^r_l}\sum_{|\alpha| \leq 1}
\left|\:D^{\alpha}_x u\:\right|^2 \:d x\:.
\end{align}
Here $C$ denotes a constant. Moreover $dx$ denotes a volume element, $d\sigma$
an element of surface area and $ds$ an element of arc length.
\end{lem}
\begin{proof}
The proof is similar to the proof of Lemma $3.1$ in~\cite{DT}. We define the vector
field $\wa$ by $\wa= A \nabla_x u$. Then
\begin{subequations}\label{eq3.13}
\begin{align}\label{eq3.13a}
Mu=\sum_{i,j=1}^3 \frac{\partial}{\partial x_i}\left(a_{i,j}\frac{\partial u}
{\partial x_j}\right)=div(\wa),
\end{align}
\begin{align}\label{eq3.13b}
\left(\frac{\partial u}{\partial \nuw}\right)_A=\sum_{i,j=1}^{3}
\nu_i a_{i,j} \frac{\partial u}{\partial x_j}=\wa_{\nu}\;\mbox{and}
\end{align}
\begin{align}\label{eq3.13c}
\left(\frac{\partial u}{\partial \taux_j}\right)_A=
\sum_{i,k=1}^{3}(\taux_j)_i a_{i,k} \frac{\partial u}{\partial x_k}=\wa_{\tau_j}.
\end{align}
\end{subequations}
Applying Theorem \ref{thm1}
\begin{align}\label{eq3.14}
&\quad \int_{\Omega^r_l} |Mu|^2 d x - \sum_{i,j=1}^3
\int_{\Omega^r_l}\frac{\partial w_i}{\partial x_j}\frac {\partial
w_j}{\partial x_i}\:dx = \left\{ \sum_{i} \int_{\Gamma_{l,i}^r}
div_T(\wa_{\nu}\wa_{T}) d\sigma\right. \notag\\
&\quad\quad-2 \left.\sum_i \int_{\Gamma_{l,i}^r} \sum_{j=1}^2
\left(\frac{\partial u}{\partial \taux_j}\right)_A \frac{\partial}
{ \partial s_j}\left(\left(\frac{\partial u}{\partial
\nuw}\right)_A\right) \:d\sigma \right\} - \sum_{i}
\int_{\Gamma_{l,i}^r}\left\{\left(tr {\mathcal B}\right)
\left(\frac{\partial u}{\partial \nuw}\right)_A^2\right. \notag\\
&\quad\quad+\left.{\mathcal B}\left(\sum_{j=1}^2
\left(\frac{\partial u}{\partial \taux_j}\right)_A
\taux_j,\sum_{j=1}^2 \left(\frac{\partial u}{\partial
\taux_j}\right)_A \taux_j\right)\right\}\:d\sigma.
\end{align}
Now by Lemma $3.1.3.4$ of~\cite{G} the following inequality holds
for all $u \in H^2(\Omega)$:
\begin{eqnarray*}
\mu_0^2 \sum_{i,j=1}^3 \left|\:\frac{\partial^2 u}{\partial x_i
\partial x_j}\:\right|^2 \leq \sum_{i,j=1}^3 \frac{\partial
w_i}{\partial x_j}\frac{\partial w_j}{\partial x_i} + 2
\sum_{i,j,k,l=1}^3 \left|\:a_{i,k} \frac{\partial^2 u}{\partial
x_j \partial x_k} \frac{\partial a_{j,l}}{\partial x_i}
\frac{\partial u}{\partial x_l}\:\right|dx\;.
\end{eqnarray*}
a.e. in $\Omega$. Integrating we have
\begin{align}
\mu^2_0 \sum_{i,j=1}^3 \int_{\Omega_l^r} \left|\:\frac{\partial^2
u} {\partial x_i \partial x_j}\:\right|^2 d x &\leq \sum_{i,j=1}^3
\int_{\Omega_l^r} \frac{\partial w_i}{\partial x_j}
\frac{\partial w_j}{\partial x_i} \:dx \notag\\
&+C \int_{\Omega^r_l} \sum_{i=1}^3 \left|\:\frac{\partial u}
{\partial x_i}\:\right| \sum_{i,j=1}^3 \left|\:\frac{\partial^2 u}
{\partial x_i \partial x_j}\:\right| \:dx\;. \notag
\end{align}
The constant $C$ in the above inequality depends on $M$ where $M$
is a common bound for all $C^1$ norms of the $a_{i,j}$. Hence
\begin{align}\label{eq3.15}
\frac{\mu_0^2}{2} \sum_{i,j=1}^3 \int_{\Omega_l^r} \left|\:\frac
{\partial^2 u}{\partial x_i \partial x_j}\:\right|^2\:dx &\leq
\sum_{i,j=1}^3 \int_{\Omega_l^r} \frac{\partial w_i}{\partial x_j}
\frac{\partial w_j}{\partial x_i} \:dx \notag\\
&+C \int_{\Omega_l^r} \sum_{|\alpha|=1} |D_x^{\alpha} u |^2\:dx.
\end{align}
Here $C$ denotes a generic constant. Now from Lemma \ref{lem2} we have
\begin{align}
\sum_{i}\int_{\Gamma_{l,i}^r} div_T(\wa_{\nu} \wa_T) d\sigma=
\sum_{i}\int_{\partial \Gamma_{l,i}^r} \wa_{\nu}\wa_n \:ds\;. \notag
\end{align}
Here $\wa_n = \wa \cdot \nb$ and $\nb$ is the vector depicted in Figure \ref{fig3.1}
lying on the tangent hyperplane at the point $P$ and orthogonal to the tangent vector
to the curve $\partial\Gamma_{l,i}^r$. Now
$$Mu= Lu - \sum_{i=1}^3 b_i u_{x_i} - c u.$$
Combining (\ref{eq3.14}) and (\ref{eq3.15}) and proceeding as in Lemma $3.4$ in~\cite{DT}
we obtain
\begin{align}
& \frac{\mu_0^2}{2} \rho_v^2 \sin^2(\phi_v)\int_{\Omega^r_l}
\sum_{i,j=1}^{3} \left|\:\frac{\partial^2 u}{\partial x_i
\partial x_j}\:\right|^2\:dx \notag\\
&\quad\quad\quad\leq \rho_v^2\sin^2(\phi_v)\int_{\Omega^r_l}|L
u|^2 d x - \rho_v^2\sin^2(\phi_v)\left\{\sum_{i}\oint_{\partial
\Gamma_{l,i}^r}\left(\frac{\partial u}{\partial \nb}\right)_A
\left(\frac{\partial u}{\partial \nuw}\right)_A \:ds\right. \notag\\
&\quad\quad\quad-2 \left.\sum_{i} \int_{\Gamma_{l,i}^r}
\sum_{j=1}^2 \left(\frac{\partial u}{\partial \taux_j}\right)_A
\frac{\partial}{\partial s_j}\left(\left(\frac{\partial u
}{\partial \nuw}\right)_A\right)\:d\sigma\right\} \notag\\
&\quad\quad\quad+ C\,\int_{\Omega^r_l}\sum_{|\alpha| \leq 1}
\left|\:D^{\alpha}_x u\:\right|^2 \:d x
+D \sum_{i}\int_{\Gamma_{l,i}^r} \sum_{|\alpha|=1}
\left|\:D^{\alpha}_x u\:\right|^2 \:d\sigma\:. \label{eq3.16}
\end{align}
Now for any $\epsilon>0$ there exists a constant $K_\epsilon$ such that
\[\int_{\Gamma_{l,i}^r} \sum_{|\alpha|=1}
\left|\:D^{\alpha}_x u\:\right|^2 \:d\sigma\leq
\epsilon\int_{\Omega^r_l}\sum_{|\alpha|=2}
\left|\:D^{\alpha}_x u\:\right|^2 \:d x
+K_\epsilon\int_{\Omega_l^r}\sum_{|\alpha|=1}
\left|\:D^{\alpha}_x u\:\right|^2 \:d x\:.\]
Using above in (\ref{eq3.16}) and choosing $\epsilon$ small enough (\ref{eq3.12}) follows.
\end{proof}
We now prove the following result which we have used in the proof of Lemma \ref{lem1}.
\begin{lem}\label{lem2} Let $w \in H^2(\Omega^r_l)$ and
$\Gamma_{l,k}^r$ denote one of the faces of $\Omega^r_l$. Then
\begin{align}\label{eq3.17}
\int_{\Gamma_{l,k}^r} div_T \left(\wa_{\nu}\wa_T \right)\:d\sigma =
\int_{\partial \Gamma_{l,k}^r}\wa _{\nu} \wa_n \:ds.
\end{align}
Here $d\sigma$ denotes an element of surface area and $ds$ an
element of arc length.
\end{lem}

\begin{proof}
We shall use geodesic coordinates to prove the result. For any
point $P \in$ closure($\Gamma_{l,k}^r$) there is an open subset
$U$ of $R^2$ containing $(0,0)$ such that $\pi: U \rightarrow R^3$
is an allowable surface patch for $\Gamma_{l,k}^r$ in a
neighbourhood of $P$. Moreover the first fundamental form of $\pi$
is $d \zeta^2 + G(\zeta, \eta) d \eta^2$ where $G$ is a smooth
function on $U$ such that $G(0, \eta) = 1$ and $G_{\zeta}(0, \eta)
= 0$ whenever $(0, \eta) \in U$. Hence for any $\epsilon > 0$ we
can choose a fine enough triangulation of $\Gamma_{l,k}^r$ so that
on each triangle there is a set of geodesic coordinates such that
$|G_{\zeta}/G| \leq \epsilon$ for all $(\zeta, \eta) \in U_i $ for
all $i$. Here the curve corresponding to $\zeta = 0$ is chosen to
be a geodesic. All the curves $\eta =$ constant are geodesics
orthogonal to the curve $\zeta = 0$. This can always be done if
the surface patch is small enough. Such a system of coordinates is
called \textit{Fermi coordinates}~\cite{WK}.
\newline
Now integrating over one such triangle we obtain
\begin{align}\label{eq3.18}
\int_{\pi_i(U_i)} div_T(\wa_\nu \wa_T) d \sigma &=\int_{\pi_i(U_i)}
\sum_{j=1}^2 \frac{\partial}{\partial s_j} \left(\wa_{\nu}(\wa_T
\cdot \taux_j)\right)\:d\sigma \notag\\
&-\int_{\pi_i(U_i)} \sum_{j=1}^2 \wa_{\nu} \left(\wa_T \cdot
\frac{\partial \taux_j}{\partial s_j}\right) \:d\sigma.
\end{align}
Clearly
\begin{eqnarray*}
\int_{\pi_i(U_i)} \sum_{j=1}^2 \frac{\partial}{\partial s_j}
\left(\wa_{\nu}(\wa_T \cdot \taux_j)\right)\:d\sigma &=& \int_{U_i}
\left\{ \frac{\partial}{\partial \zeta}\left(\wa_{\nu}(\wa_T \cdot
\taux_1)\right) \sqrt{G} \right.\\
 &+& \left.\frac{\partial}{\partial \eta} \left(\wa_{\nu} (\wa_T
\cdot \taux_2)\right) \right\} \:d\zeta \:d\eta \\
&=& \int_{U_i} \left\{ \frac{\partial}{\partial
\zeta}\left(\wa_{\nu}
(\wa_T \cdot \taux_1)\sqrt{G}\right)\right. \\
&+& \left.\frac{\partial}{\partial \eta}\left(\wa_{\nu}(\wa_T \cdot
\taux_2)
\right) \right\} \:d\zeta \:d\eta \\
&-& \int_{U_i} \left\{\wa_{\nu}(\wa_T \cdot \taux_1) \frac{\partial
\sqrt{G}}{\partial \zeta}/ \sqrt{G} \right\} \sqrt{G} \:d\zeta
\:d\eta.
\end{eqnarray*}
Hence
\begin{eqnarray*}
\int_{\pi_i(U_i)} \sum_{j=1}^2 \frac{\partial}{\partial
s_j}\left(\wa_{\nu} (\wa_T \cdot \taux_j)\right) \:d\sigma &=&
\int_{\partial U_i} \left(\wa_{\nu}(\wa_T \cdot \taux_1)\sqrt{G}
d\eta-\wa_{\nu}(\wa_{T} \cdot \taux_2)\:d \zeta\right)\\
&-& \int_{\pi_i(U_i)} \wa_{\nu}(\wa_T \cdot
\taux_1)\left(\frac{\partial \sqrt{G}}{\partial
\zeta}/\sqrt{G}\right) \:d\sigma\;.
\end{eqnarray*}
Now
\begin{eqnarray*}
\frac{d x}{d s} = x_{\zeta} \frac{d \zeta}{d s} + x_{\eta} \frac{d
\eta}{d s} =\taux_1 \frac{d \zeta}{d s} + \taux_2 \sqrt{G} \frac{d
\eta}{d s}.
\end{eqnarray*}
Hence $\nb=\taux_1 \sqrt{G} \frac{d \eta}{d s} - \taux_2 \frac{d
\zeta}{ds}$ is the unit outward normal to $\partial \pi_i(U_i)$.
And so
\begin{eqnarray*}
& & \int_{\pi_i(U_i)} \sum_{j=1}^2 \frac{\partial}{\partial s_j}
\left(\wa_{\nu}(\wa_T \cdot \taux_j)\right)\:d\sigma \\
&=& \int_{\partial \pi_i(U_i)} \wa_{\nu} \wa_n \:d s -
\int_{\pi_i(U_i)} \wa_{\nu}(\wa_T \cdot \taux_1) \left(\frac
{\partial \sqrt{G}}{\partial \zeta}/\sqrt{G}\right)\:d\sigma\;.
\end{eqnarray*}
Summing over all triangular elements of the form $\pi_i(U_i)$ we obtain using
(\ref{eq3.18})
\begin{align}\label{eq3.19}
\int_{\Gamma_{l,k}^r} div_T(\wa_{\nu} \wa_T) d \sigma &=
\int_{\partial \Gamma_{l,k}^r} \wa_{\nu} \wa_n \:ds - \sum_{i}
\int_{\pi_i(U_i)} \wa_{\nu} (\wa_T \cdot \taux_1)\left(\frac
{\partial \sqrt{G}}{\partial \zeta}/ \sqrt{G}\right) d \sigma \notag\\
&-\sum_{i}\int_{\pi_i(U_i)}\sum_{j=1}^2 \wa_{\nu}\left(\wa_T\cdot
\frac{\partial \taux_j}{\partial s_j}\right)\:d\sigma.
\end{align}
Now
\begin{subequations}\label{eq3.20}
\begin{align}\label{eq3.20a}
\left|\:\sum_{i} \int_{\pi_i(U_i)} \wa_{\nu}(\wa_T \cdot
\taux_1)\left(\frac{\partial \sqrt{G}}{\partial \zeta}/
\sqrt{G}\right) \:d\sigma\:\right| \leq \epsilon
\int_{\Gamma^r_{l,k}} |\wa|^2 \:d\sigma
\end{align}
Next at the point $P$ the the $\zeta$ parameter curves and the
$\eta$ parameter curves are geodesics. Hence at $P$,
$\frac{\partial \taux_j}{\partial s_j} \cdot T^{\prime} = 0$ for
any vector $T^{\prime}$ which lies on the tangent plane at $P$.
Thus for any $\epsilon > 0$ we can choose a fine enough
triangulation so that
\begin{align}\label{eq3.20b}
\left|\:\sum_{i} \int_{\pi_i(U_i)} \wa_{\nu}(\wa_T \cdot
\frac{\partial \taux_j}{\partial s_j})d \sigma\:\right| &\leq
\epsilon \int_{\Gamma_{l,i}^r}\left(\wa_{\nu}^2 + |\wa_T|^2\right)
\:d\sigma \notag\\
&\leq \epsilon \int_{\Gamma_{l.i}^r} |\wa|^2\:d\sigma.
\end{align}
\end{subequations}
Now from (\ref{eq3.19}) and (\ref{eq3.20}) we obtain the result since $\epsilon$ is arbitrary.
\end{proof}

\subsection{Estimates for second derivatives in vertex neighbourhoods}
In Figure \ref{fig2.2} the intersection of $\Omega$ with a sphere of radius $\rho_v$ with
center at the vertex $v$ is shown. On removing the vertex-edge neighbourhoods, which are
shaded, we obtain the vertex neighbourhood $\Omega^v$ where $v\in \mathcal V$ and $\mathcal V$
denotes the set of vertices. Choose $\rho_v $ and $\phi_v$ so that $\rho_v \sin(\phi_v) =Z$, a
constant, for all $v \in V$. Let $S^v$ denote the intersection of the closure of $\Omega^v$
with $B_{\rho_v}(v)=\left\{ x:\:|x-v|\leq \rho_v \right\}$. $S^v$ is divided into triangular
and quadrilateral elements $S_j^v$ for $1 \leq j \leq I_v$, where $I_v$ is a fixed constant.
We define a geometric mesh in the vertex neighbourhood $\Omega^v$ of the vertex $v$ as shown
in Figure \ref{fig2.8}. Thus $\Omega^v$ is divided into $N_v $ curvilinear cubes and prisms
$\{\Omega^v_l \}_{1 \leq l \leq N_v}$. Let us introduce the new coordinate system
\begin{align}
x_1^v &= \phi \notag\\
x_2^v &= \theta \notag\\
x_3^v &= \ln \rho = {\mathcal X}\:.
\end{align}
Let $\tilde{\Omega}_l^v$ be the image of $\Omega^v_l$ in $(x_1^v,x_2^v, x_3^v)$ coordinates.
Now
\begin{align}\label{eq3.22}
\int_{\Omega^v_l} \rho^2 |Lu|^2 \:d x = \int_{\tilde{\Omega}^v_l}
e^{{\mathcal X}} \sin \phi | e^{2 {\mathcal X}} Lu|^2 d \phi d
\theta \:d {\mathcal X}\:.
\end{align}
Define
\begin{equation}\label{eq3.23}
L^vu(x^v) = e^{\frac{\mathcal X}{2}} \sqrt{\sin \phi}\left(e^{2 {\mathcal X}} L u\right).
\end{equation}
We have the relation
\begin{equation}\label{eq3.24}
\rho \nabla_x u = Q^v \nabla_{x^v}u\:.
\end{equation}
Here
\begin{subequations}\label{eq3.25}
\begin{equation}\label{eq3.25a}
Q^v = O^v P^v
\end{equation}
where $O^v$ is the orthogonal matrix
\begin{align}
O^v  = \left[ \begin{array}{ccc}
\cos \phi \cos \theta   &-\sin \theta & \sin \phi \cos \theta \\
\cos \phi \sin \theta & \cos \theta & \sin \phi \sin \theta \\
-\sin \phi & 0 & \cos \phi
\end{array} \right]
\end{align}
and
\begin{align}
P^v =  \left[ \begin{array}{ccc}
1 & 0 & 0\\
0 & 1/\sin \phi & 0 \\
0 & 0 & 1
\end{array} \right].
\end{align}
Define
\begin{equation}\label{eq3.25d}
A^v = (Q^v)^T A Q^v.
\end{equation}
\end{subequations}
Since $\phi_0 < \phi < \pi - \phi_0$ where $\phi_0$ denotes a positive constant,
$\mu_0 I \leq A^v \leq \mu_1 I $ for some positive constants $\mu_0$ and $\mu_1$.
Let $\tilde {\Omega}^v_l$ be a curvilinear cube and let its faces be denoted by
$\{\tilde{\Gamma} _{l,i}^v \}$. We now prove the following result.
\begin{lem}\label{lem3}
There exist positive constants $C_v$ such that
\begin{align}\label{eq3.26}
& \frac{\mu_0^2}{2} \sin^2(\phi_v) \int_{\tilde{\Omega}_l^v}
e^{x_3^v} \sum_{r,s=1}^3 \left|\:\frac{\partial^2 u}{\partial
x^v_r \partial x^v_s}\:\right|^2 \:dx^v \notag\\
&\quad\leq \sin^2(\phi_v)\int_{\tilde{\Omega}^v_l}|L^v u(x^v)|^2 dx^v \notag\\
&\quad- \sin^2(\phi_v) \left\{ \sum_{i} \oint_{\partial
\tilde{\Gamma}^v_{l,i}} e^{x_3^v} \sin(x_1^v) \left(\frac
{\partial u}{\partial \nb^v}\right)_{A^v} \left(\frac{\partial
u}{\partial \nuw^v}\right)_{A^v} \:ds^v\right. \notag\\
&\quad- \left.2 \sum_{i} \int_{\tilde{\Gamma}_{l,i}^v} e^{x_3^v}
\sin(x_1^v) \sum_{j=1}^2 \left(\frac{\partial u}{\partial
\taux_j^v}\right)_{A^v} \frac{\partial}{\partial s_j^v}
\left(\left(\frac{\partial u}{\partial \nuw^v}\right)
_{A^v}\right) d \sigma^v \right\} \notag\\
&\quad+ C_v \int_{\tilde{\Omega}^v_l} \sum_{|\alpha| \leq 1}
e^{x_3^v}\left|D^{\alpha}_{x^v}u\right|^2 \:dx^v
\end{align}
for $u \in H^3( \tilde{\Omega}^v_l)$. Here $dx^v$ denotes a volume element in $x^v$
coordinates, $d\sigma^v$ an element of surface area and $ds^v$ an element of arc
length in $x^v$ coordinates.
\end{lem}
\begin{proof}
We let $\fe$ denote a vector field. Then
\[ \rho\:div_x (\fe) = \frac{e^{-2 {\mathcal X}}}{\sin
\phi} div_{x^v}\left(e^{2 {\mathcal X}} \sin \phi (Q^v)^T
\fe\right).\]
Take $\fe= A \nabla_x u$. Hence
\begin{eqnarray*}
\int_{\Omega^v_l} \left|\rho\:div_x (A \nabla_x u )\right|^2 \:dx
=\int_{\tilde{\Omega}^v_l} \frac{e ^{-{\mathcal X}}}{\sin \phi}
\left| div_{x^v} (e^{\mathcal X} \sin \phi (Q^v)^T A Q^v
\nabla_{x^v}u)\right|^2 \:dx^v.
\end{eqnarray*}
Define
\begin{equation}\label{eq3.27}
M^v u (x^v) = div_{x^v}\left(e^{{\mathcal X}/2}
\sqrt{\sin \phi} A^v \nabla_{x^v} u \right).
\end{equation}
Then
\begin{eqnarray*}
\frac{e^{-{\mathcal X}/2}}{\sqrt{\sin \phi}} \:div_{x^v}
\left(e^{{\mathcal X}}
\sin \phi A^v \nabla_{x^v} u\right) &=& M^v u(x^v) + \frac{1}{2}
e^{{\mathcal X}/2} \sqrt{\sin \phi} \sum_{j=1}^3 a^v_{3,j}
\frac{\partial u}{\partial x^v_j} \\
&+& \frac{1}{2} e^{{\mathcal X}/2} \frac{\cos \phi}{\sqrt{\sin
\phi}} \sum_{j=1}^3 a^v_{1,j} \frac{\partial u}{\partial x^v_j}\:.
\end{eqnarray*}
Here $A^v$ is as defined in (\ref{eq3.25d}). \\
Define the vector field $\wa$ by
\begin{subequations}\label{eq3.28}
\begin{equation}\label{eq3.28a}
\wa = e^{{\mathcal X}/2} \sqrt{\sin \phi} A^v \nabla_{x^v} u\:.
\end{equation}
Then
$$L^v u (x^v) = div_{x^v}(\wa) + \eta^v u(x^v)$$
where
\begin{align}\label{eq3.28b}
\eta^v u (x^v) &= -\frac{1}{2} e^{{\mathcal X}/2} \sqrt{\sin \phi}
\sum_{j=1}^3 a_{3,j}^v \frac{\partial u}{\partial x^v_j}
-\frac{1}{2} e^{{\mathcal X}/2} \frac{\cos \phi}{\sqrt{\sin \phi}}
\sum_{j=1}^3 a_{1,j}^v \frac{\partial u}{\partial x^v_j} \notag\\
&+ \sum_{i=1}^3 b_i^v \frac{\partial u}{\partial x^v_i} + c^v u\:.
\end{align}
\end{subequations}
Here
\begin{subequations}\label{eq3.29}
\begin{align}\label{eq3.29a}
\| b_i^v \|_{0, \infty, \tilde{\Omega}^v_l}=O(e^{3 {\mathcal X}/2}),
\end{align}
\begin{align}\label{eq3.29b}
\| c^v \|_{0, \infty, \tilde{\Omega}^v_l}=O(e^{5 {\mathcal X}/2}),
\end{align}
and
\begin{align}\label{eq3.29c}
\left\|\:(e^{{\mathcal X}/2} \sqrt{\sin \phi} A^v)
\:\right\|_{1, \infty,\tilde{\Omega}^v_l}=O(e^{{\mathcal X}/2})\:.
\end{align}
\end{subequations}
To obtain (\ref{eq3.26}) we shall use Theorem \ref{thm1} applied to the vector field
$\wa$ along with Lemma \ref{lem2}. Now
\begin{align}\label{eq3.30}
& 2 \int_{\tilde{\Gamma}_{l,i}^v} \wa_T \cdot \nabla_T \wa_{\nu}\:d\sigma^v \notag \\
&\quad= 2 \int_{\tilde{\Gamma}_{l,i}^v} \sum_{j=1}^2 e^{{\mathcal
X}/2} \sqrt{\sin \phi} \left( \frac{\partial u}{\partial
\taux^v_j}\right)_{A^v} \frac{\partial}{\partial
s^v_j}\left(e^{{\mathcal X}/2} \sqrt{\sin \phi}
\left(\frac{\partial u}{\partial \nuw^v}\right)_{A^v}\right)\:d\sigma^v \notag \\
&\quad=2\int_{\tilde{\Gamma}_{l,i}^v}\sum_{j=1}^2 e^{{\mathcal X}}
\sin\phi \left(\frac{\partial u}{\partial \taux_j^v}\right)_{A^v}
\frac{\partial}{\partial s^v_j}\left(\left(\frac{\partial u}{\partial
\nuw^v}\right)_{A^v}\right)\:d\sigma^v \notag \\
&\quad+ 2 \int_{\tilde{\Gamma}_{l,i}^v} \sum_{j=1}^2 e^{{\mathcal
X}/2} \sqrt{\sin \phi} \left(\frac{\partial u}{\partial \taux^v_j}
\right)_{A^v}\left(\frac{\partial u}{\partial \nuw^v}\right)_{A^v}
\frac{\partial}{\partial s^v_j} \left(e^{{\mathcal X}/2}
\sqrt{\sin \phi}\right)\:d\sigma^v\:.
\end{align}
And so using (\ref{eq3.27}), (\ref{eq3.28}) and (\ref{eq3.30}) we obtain
\begin{align*}
& \frac{\mu_0^2}{2} \sin^2(\phi_v) \int_{\tilde{\Omega}_l^v}
e^{x_3^v} \sum_{r,s=1}^3 \left|\:\frac{\partial^2 u}{\partial x^v_r \partial x^v_s}
\:\right|^2 \:dx^v \\
&\quad\leq \sin^2(\phi_v)\int_{\tilde{\Omega}^v_l}|L^v u(x^v)|^2 d x^v \\
&\quad- \sin^2(\phi_v) \left\{ \sum_{i} \oint_{\partial
\tilde{\Gamma}^v_{l,i}} e^{x_3^v} \sin(x_1^v) \left(\frac
{\partial u}{\partial \nb^v}\right)_{A^v} \left(\frac{\partial
u}{\partial \nuw^v}\right)_{A^v} \:ds^v\right.
\end{align*}
\begin{align*}
&\quad- \left.2 \sum_{i} \int_{\tilde{\Gamma}_{l,i}^v} e^{x_3^v}
\sin(x_1^v) \sum_{j=1}^2 \left(\frac{\partial u}{\partial
\taux_j^v}\right)_{A^v} \frac{\partial}{\partial s_j^v}
\left(\left(\frac{\partial u}{\partial \nuw^v}\right)
_{A^v}\right) d \sigma^v \right\}\\
&\quad+ C_v \int_{\tilde{\Omega}^v_l} \sum_{|\alpha| \leq 1}
e^{x_3^v}\left|D^{\alpha}_{x^v}u\right|^2 \:dx^v+ D_v \sum_i
\int_{\tilde{\Gamma}_{l,i}^v}\sum_{|\alpha|
=1}e^{x_3^v}\left|D^{\alpha}_{x^v}u\right|^2 \:d\sigma^v\;.
\end{align*}
Now for any $\epsilon>0$ there exists a constant $K_\epsilon$ such that
\begin{align*}
\int_{\tilde{\Gamma}_{l,i}^v} \sum_{|\alpha|=1}
e^{x_3^v}\left|\:D^{\alpha}_x u\:\right|^2 \:d\sigma^v
&\leq\epsilon\int_{\tilde{\Omega}^v_l}\sum_{|\alpha|=2}
e^{x_3^v}\left|\:D^{\alpha}_x u\:\right|^2 \:d x^v\\
&+K_\epsilon\int_{\tilde{\Omega}_l^v}\sum_{|\alpha|=1}
e^{x_3^v}\left|\:D^{\alpha}_x u\:\right|^2 \:d x^v\:.
\end{align*}
Choosing $\epsilon$ small enough $(\ref{eq3.26})$ follows from the above equation.
\end{proof}
We now show that the boundary integrals in Lemma \ref{lem1} and Lemma \ref{lem3}
coincide  when $\Gamma^v_{k,i} = \Gamma_{l,m}^r$ is a portion of the sphere
$B_{\rho_v} (v) = \left\{x:\:|x-v|=\rho_v \right\}$, except that they have opposite
signs. Let $Q^v$ be the matrix defined in (\ref{eq3.25a}). Then by (\ref{eq3.24})
$$\rho \nabla_x u = Q^v \nabla_{x^v}u\:.$$
Now if $\bf dx$ is a tangent vector to a curve in $x$ coordinates then its image
in $x^v$ coordinates is given by ${\bf dx}^v$ where
\begin{align}\label{eq3.31}
{\bf dx}^v = \frac{(Q^v)^T}{\rho} {\bf dx}\;.
\end{align}
Clearly the first fundamental form $d s^2$ in $x$ coordinates is
\begin{align}\label{eq3.32}
d s^2 &= {\bf dx}^T {\bf dx} \notag\\
&= \rho^2 ({\bf dx}^v)^T\left[(Q^v)^{-1}(Q^v)^{-T}\right]\:{\bf dx}^v \notag\\
&= e^{2 {\mathcal X}}(d \phi^2 + \sin^2 \phi d \theta^2 + d{\mathcal X}^2)\:.
\end{align}
Moreover on $\Gamma_{l,m}^r$
\begin{align}\label{eq3.33}
d \sigma = \rho_v^2 \sin\phi\:d\phi\:d\theta\:.
\end{align}
Choose $\taux_1^v = (1,0,0)^T$ and $\taux_2^v = (0,1,0)^T$. These are then orthogonal
unit tangent vectors on $\tilde{\Gamma}_{k,i}^v$ since $(d s^v)^2=d\phi^2+d\theta^2
+d{\mathcal X }^2$. Define
\begin{align}\label{eq3.34}
\taux_1 &= - (Q^v)^{-T}\taux_1^v = - e_{\phi}, \notag\\
\taux_2 &= \frac{(Q^v)^{-T}}{\sin \phi} \taux_2^v = e_{\theta}\:.
\end{align}
Let $\nuw^v = (0,0,1)^T$ denote the unit normal vector on $\tilde{\Gamma}_{k,i}^v.$ Then
\begin{equation}\label{eq3.35}
\nuw = - (Q^v)^{-T} \nuw^v
\end{equation}
denotes the unit normal to $\Gamma_{l,m}^r$. Finally let
${\bf d s}^v = (d \phi, d \theta,0)^T $ denote a tangent vector
field on $\tilde{\Gamma}_{k,i}^v$. Define
\[d s^v = \sqrt{d \phi^2 + d \theta^2},\]
\begin{equation}\label{eq3.36}
{\bf d s} = \rho_v (Q^v)^{-T} {\bf d s}^v
\end{equation}
and
\[d s = \rho_v \sqrt{d \phi^2 + \sin^2 (\phi) d \theta^2}\;.\]
Let
\begin{equation}\label{eq3.37}
\nb^v = \frac{(-d \theta, d \phi, 0 )^T}{\sqrt{d \theta^2 + d\phi^2}}
\end{equation}
be the unit outward normal to $\partial\tilde{\Gamma}_{k,i}^v$. Define
\begin{subequations}\label{eq3.38}
\begin{equation}\label{eq3.38a}
{\mc}^v = \left(-\sin \phi d \theta, \frac{d \phi}{\sin \phi},
0 \right)^T / \sqrt{d \phi^2 + \sin^2 \phi d \theta^2}\;.
\end{equation}
Then
\begin{equation}\label{eq3.38b}
\nb = (Q^v)^{-T} {\mc}^v
\end{equation}
\end{subequations}
is the unit normal vector to $\partial {\Gamma}_{l,m}^r$. We now prove
the following result.
\begin{lem}\label{lem4}
Let $\Gamma_{k,i}^v = \Gamma_{l,m}^r$. Then the following
identities hold.
\begin{align}\label{eq3.39}
& \rho_v^2 \sin^2(\phi_v) \oint_{\partial \Gamma_{l,m}^r} \left
(\frac{\partial u}{\partial \nb}\right)_A \left(\frac{\partial u}
{\partial \nuw}\right)_A \:ds \notag\\
&\quad=- \sin^2 (\phi_v) \oint_{\partial \tilde{\Gamma}_{k,i}^v}
e^{x_3^v} \sin (x_1^v) \left(\frac{\partial u}{\partial
\nb^v}\right) _{A^v}\left(\frac{\partial u}{\partial
\nuw^v}\right)_{A^v} \:ds^v
\end{align}
and
\begin{align}\label{eq3.40}
& \rho_v^2 \sin^2(\phi_v) \int_{ \Gamma_{l,m}^r} \sum_{j=1}^2
\left(\frac{\partial u}{\partial \taux_j}\right)_A \frac{\partial}
{\partial s_j}\left(\left(\frac{\partial u}{\partial
\nuw}\right)_A\right) \:d\sigma \notag\\
&\quad= - \sin^2 (\phi_v) \int_{ \tilde{\Gamma}_{k,i}^v} e^{x_3^v}
\sin (x_1^v) \sum_{j=1}^2 \left(\frac{\partial u}{\partial
\taux_j^v} \right)_{A^v}\frac{\partial }{\partial s^v_j}
\left(\left(\frac{\partial u} {\partial\nuw^v}\right)_{A^v}\right)
\:d\sigma^v\;.
\end{align}
\end{lem}
\begin{proof}
We first evaluate $\rho_v^2 \sin^2(\phi_v) \oint_{\partial
\Gamma_{l,m}^r} \left(\frac{\partial u}{\partial \nb}\right)_A
\left(\frac{\partial u}{\partial \nuw}\right)_A \:ds$. By (\ref{eq3.24})
and (\ref{eq3.35})
\begin{eqnarray*}
\left(\frac{\partial u}{\partial \nuw}\right)_A &=& \nuw^T A
\nabla_x u
= - \frac{(\nuw^v)^T}{\rho_v} \left(( Q^v)^{-1} A Q^v\right)
\nabla_{x^v} u\\
&=& \frac{-\left((\nuw^v)^T (Q^v)^{-1}(Q^v)^{-T}\right)}{\rho_v}
A^v \nabla_{x^v}u\:.
\end{eqnarray*}
Now by (\ref{eq3.32}) and (\ref{eq3.35})
$$(\nuw^v)^T (Q^v)^{-1}(Q^v)^{-T} = (\nuw^v)^T.$$
Hence we can conclude that
\begin{equation}\label{eq3.41}
\left(\frac{\partial u}{\partial \nuw}\right)_A = \frac{-1}{\rho_v}
\left(\frac{\partial u}{\partial \nuw^v}\right)_{A^v}.
\end{equation}
Now by (\ref{eq3.38b})
\[ \left(\frac{\partial u}{\partial \nb}\right)_A = n^T A \nabla_x u
= \frac{1}{\rho_v}\left((\mc^v)^T (Q^v)^{-1}(Q^v)^{-T} A^v
\nabla_{x^v} u\right)\:. \]
Clearly
$$ \left( (\mc^v)^T (Q^v)^{-1}(Q^v)^{-T}\right) = \frac{\sin \phi
(- d\theta, d\phi, 0)}{\sqrt{d \phi^2 + \sin^2 \phi d
\theta^2}}\:.$$
Using (\ref{eq3.35}), (\ref{eq3.36}) and (\ref{eq3.37}) gives
$$\frac{1}{\rho_v}\left((\mc^v)^T (Q^v)^{-1}(Q^v)^{-T}\right)
= \sin \phi \frac{d s^v}{d s} (\nb^v)^T.$$ Hence
\begin{equation}\label{eq3.42}
\left(\frac{\partial u}{\partial \nb}\right)_A d s = \sin
\phi\left(\frac{\partial u}{\partial \nb^v}\right)_{A^v}
\:d s^v.
\end{equation}
Thus from (\ref{eq3.41}) and (\ref{eq3.42}) we obtain (\ref{eq3.39}).
Finally we evaluate the term
\[{\rho}_{v}^2 \sin^2(\phi_v)\int_{\Gamma_{l,m}^r} \sum_{j=1}^2
\left(\frac{\partial u}{\partial\taux_j}\right)_A \frac{\partial}
{\partial s_j}\left(\left(\frac{\partial u}{\partial \nuw}
\right)_A \right)\:d\sigma\;.\]
Using (\ref{eq3.34}) it is easy to show that
\begin{align}\label{eq3.43}
\left(\frac{\partial u}{\partial \taux_1}\right)_A &=
\frac{-1}{\rho_v} \left(\frac{\partial u}{\partial \taux_1^v}
\right)_{A^v},\:\mbox{ and } \notag\\
\left(\frac{\partial u}{\partial \taux_2}\right)_A &= \frac{\sin
\phi} {\rho_v} \left(\frac{\partial u}{\partial
\taux_2^v}\right)_{A^v}.
\end{align}
Moreover by (\ref{eq3.41})
\[\left(\frac{\partial u}{\partial \nuw}\right)_A = \frac{-1}
{\rho_v}\left(\frac{\partial u}{\partial \nuw^v}\right)_{A^v}.\]
Hence
\begin{align}\label{eq3.44}
\frac{\partial}{\partial s_1}\left(\left(\frac{\partial
u}{\partial \nuw} \right)_A\right) &= \frac{1}{\rho^2_v}
\frac{\partial}{\partial s_1^v} \left(\left(\frac{\partial
u}{\partial \nuw^v}\right)_{A^v}\right), \mbox{ and } \notag\\
\frac{\partial}{\partial s_2}\left(\left(\frac{\partial
u}{\partial \nuw}\right)_A\right)&= -\frac{1}{\rho^2_v \sin \phi }
\frac{\partial} {\partial s_2^v}\left(\left(\frac{\partial
u}{\partial \nuw^v}\right) _{A^v}\right).
\end{align}
Moreover from (\ref{eq3.32})
\begin{equation}\label{eq3.45}
d \sigma = \rho_v^2 \sin \phi \,d \sigma^v.
\end{equation}
Combining (\ref{eq3.43}), (\ref{eq3.44}) and (\ref{eq3.45}) we obtain (\ref{eq3.40}).
\end{proof}

\subsection{Estimates for second derivatives in vertex-edge neighbourhoods}
Figure \ref{fig2.4} shows the vertex-edge neighbourhood $\Omega^{v-e}$ of the vertex
$v$ and the edge $e$. As before we have $\rho_v \sin \phi_v = Z$. Now
\[\Omega^{v-e}=\left\{x \in \Omega: 0< x_3 < \delta_{v},\:0<\phi<\phi_v \right\}.\]
A geometric mesh is imposed on $\Omega^{v-e}$ as shown in Figure \ref{fig2.9}.
To proceed further we introduce new coordinates in $\Omega^{v-e}$ by
\begin{align}\label{eq3.46}
x_1^{v-e} &= \psi = \ln (\tan \phi) \notag \\
x_2^{v-e} &= \theta \notag \\
x_3^{v-e} &= \zeta = \ln x_3 = {\mathcal X} + \ln (\cos\phi)\:.
\end{align}
Here
\begin{align}\label{eq3.47}
x_1^{v} &= \phi \notag \\
x_2^{v} &= \theta \notag \\
x_3^{v} &= {\mathcal X} = \ln \rho
\end{align}
are the coordinates which have been introduced in the vertex
neighbourhood $\Omega^v$, adjoining $\Omega^{v-e}.$ Let $\tilde
{\Omega}^{v-e}$ be the image of $\Omega^{v-e}$ in $x^{v-e}$
coordinates. Thus $\tilde{\Omega}^{v-e}$ is divided into $N^{v-e}
= I^{v-e} (N+1)^2 $ hexahedrons  $\tilde{\Omega}^{v-e}_n$. Now
\begin{equation}\label{eq3.48}
\nabla_{x^v} u = J^{v-e} \nabla_{x^{v-e}} u
\end{equation}
where
\begin{equation}\label{eq3.49}
J^{v-e} = \left[ \begin{array}{ccc}
\sec^2 \phi \cot \phi  &0& -\tan \phi \\
0 & 1 & 0 \\
0 & 0 & 1
\end{array} \right].
\end{equation}
We now need to evaluate
\[\int_{\Omega^{v-e}_n} \rho^2 \sin^2 \phi | Lu(x)|^2 d x\:. \]
Let $\hat{\Omega}^{v-e}_n$ denote the image of $\Omega^{v-e}_n$ in
$x^v$ coordinates. Then
\[ \int_{\Omega_n^{v-e}} \rho^2 \sin^2 \phi |Lu(x)|^2 dx =
\int_{\hat{\Omega}^{v-e}_n} \sin^2 \phi | L^v u (x^v)|^2 d x^v.\]
Now
\begin{equation}\label{eq3.50}
d x^v = \sin \phi \cos \phi\:d x^{v-e}.
\end{equation}
Let $\fe$ denote a vector field. Then
\begin{equation}\label{eq3.51}
div_{x^v}(\fe) = \frac{1}{\sin \phi \cos \phi} div_{x^{v-e}}
\left(\sin \phi \cos \phi (J^{v-e})^T \fe\right)\;.
\end{equation}
Using $(\ref{eq3.50})$ we obtain
$$
\int_{\Omega^{v-e}_n} \rho^2 \sin^2 \phi \left|\:Lu(x)\:\right|^2
\:dx = \int_{\tilde{\Omega}^{v-e}_n} \sin ^3 \phi \cos \phi |L^v
u(x^v)|^2 \:d x^{v-e}\:.
$$
Define
\begin{equation}
L^{v-e}u(x^{v-e}) = (\sin \phi)^{3/2}(\cos \phi)^{1/2} L^v u(x^v)\:.
\end{equation}
Then using (\ref{eq3.27}), (\ref{eq3.48}) and (\ref{eq3.51})
\begin{align}\label{eq3.53}
L^{v-e} u(x^{v-e}) = (\sin \phi)^{1/2} (\cos \phi)^{-1/2}div_{x^{v-e}}
\left(e^{{\mathcal X}/2} (\sin \phi)^{3/2} \cos \phi \right. \notag\\
(J^{v-e})^T A^v J^{v-e} \nabla_{x^{v-e}}u\Big) + \sum_{i=1}^2
b_i^{v-e} u_{x_i^{v-e}} + c^{v-e} u\:.
\end{align}
Define
\[ M^{v-e}u(x^{v-e}) = div_{x^{v-e}}\left(e^{{\mathcal X}/2}
(\cos \phi)^{1/2} \sin^2 \phi(J^{v-e})^T A^v J^{v-e}
\nabla_{x^{v-e}} u \right).\]
Or
\begin{equation}\label{eq3.54}
M^{v-e} u (x^{v-e}) = div_{x^{v-e}}\left(e^{\zeta/2} A^{v-e}
\nabla_{x^{v-e}} u\right).
\end{equation}
Here
\[A^{v-e}= \sin^2\phi(J^{v-e})^T A^v J^{v-e}\;.\]
Or using $(\ref{eq3.25})$
\begin{equation}\label{eq3.55}
A^{v-e} = (K^{v-e})^T A K^{v-e}
\end{equation}
where
\[K^{v-e} = O^v R^{v-e} \]
and
\[
R^{v-e} = \left[ \begin{array}{ccc}
\frac{1}{\cos \phi}  &0& \frac{-\sin^2 \phi}{\cos \phi} \\
0 & 1 & 0 \\
0 & 0 &  \sin \phi
\end{array} \right].
\]
Now
\begin{eqnarray*}
(\tan \phi)^{1/2} div_{x^{v-e}} \left(e^{{\mathcal X}/2} (\sin
\phi)^{3/2} \cos \phi
(J^{v-e})^T A^v J^{v-e} \nabla_{x^{v-e}}u\right)\\
= M^{v-e} u(x^{v-e}) - \frac{1}{2} e^{\zeta/2} \sum_{j=1}^3
\hat{a}_{1,j}^{v-e} \frac{\partial u}{\partial x_j^{v-e}}\:.
\end{eqnarray*}
Hence using (\ref{eq3.53})
\begin{equation}\label{eq3.56}
L^{v-e} u (x^{v-e}) = M^{v-e} u (x^{v-e})+\eta^{v-e}u(x^{v-e})\:.
\end{equation}
Here
\begin{eqnarray*}
\eta^{v-e} u(x^{v-e}) &=& \frac{-1}{2}\:e^{\zeta/2} \sum_{j=1}^3
\hat{a}_{1,j}^{v-e} \frac{\partial u}{\partial x^{v-e}_j} +
(\sin\phi)^{3/2}(\cos\phi)^{1/2}\eta^v u(x^v)\\
&=& e^{\zeta/2} \sum_{j=1}^3 \hat{a}_{1,j}^{v-e} \frac{\partial
u}{\partial x^{v-e}_j} + \sum_{i=1}^3 \hat{b}_i^{v-e} \frac{\partial
u}{\partial x^{v-e}_i} + \hat{c}^{v-e}u\;.
\end{eqnarray*}
Moreover using (\ref{eq3.28}), (\ref{eq3.29}) and (\ref{eq3.48}),
it can be shown that
\begin{align}\label{eq3.57}
\|\:\hat{b}^{v-e}_i\:\|_{0, \infty, \tilde{\Omega}^{v-e}_n} &=
O(e^{\zeta/2}) \rm{\;\;for\;} \mbox{ i }= 1,2 , \notag \\
\|\:\hat{b}^{v-e}_3\:\|_{0, \infty, \tilde{\Omega}^{v-e}_n} &=
O(e^{\zeta/2} \sin\phi), \rm{\;\;and \;} \notag \\
\|\:\hat{c}^{v-e}\:\|_{0, \infty, \tilde{\Omega}^{v-e}_n} &=
O(e^{3{\zeta/2}} \sin^{\frac{3}{2}} \phi)\:.
\end{align}
Now consider the matrix $A^{v-e}$ defined in $(\ref{eq3.55})$. We note
that the matrix $A^{v-e}$ becomes singular as $\phi \rightarrow0$. To
overcome this problem we introduce a new set of local variables
$y=(y_1, y_2, y_3)$ in
$$
\Omega^{v-e}_n = \left\{ x:\:\phi_l^{v-e}< \phi < \phi_{l+1}^{v-e},
\: \theta_j^{v-e} < \theta < \theta^{v-e}_{j+1},\: \delta_{v}(\mu_v)^k
< x_3 < \delta_{v}(\mu_v)^{k-1} \right\}
$$
by
\begin{align}\label{eq3.58}
y_1 &= x_1^{v-e} \notag \\
y_2 &= x_2^{v-e} \notag \\
y_3 &= \frac{x_3^{v-e}}{\sin\left(\phi^{v-e}_{l+1}\right)}\:.
\end{align}
In making this transformation $\tilde{\Omega}^{v-e}_n$ is mapped to a
hexahedron $\hat{\Omega}^{v-e}_n$ such that the length of the $y_3$
side becomes large as $\Omega_n^{v-e}$ approaches the edge of the
domain $\Omega$. It is important to note that the trace and embedding
theorems in the theory of Sobolev spaces remain valid with a uniform
constant for all the domains $\hat{\Omega}^{v-e}_n$. Now
\begin{align}\label{eq3.59}
\nabla_{x^{v-e}} u =  \left[ \begin{array}{ccc}
1 & 0 & 0 \\
0 & 1 & 0 \\
0 & 0 & \frac{1}{\sin(\phi^{v-e}_{l+1})}
\end{array} \right] \nabla_y u,
\end{align}
\begin{equation}\label{eq3.60}
d x^{v-e} = \sin \left(\phi^{v-e}_{l+1}\right) d y,
\end{equation}
and
\begin{align}\label{eq3.61}
div_{x^{v-e}}(f) = div_y\left(\left[ \begin{array}{ccc}
1 & 0 & 0 \\
0 & 1 & 0 \\
0 & 0 & \frac{1}{\sin(\phi^{v-e}_{l+1})}
\end{array} \right]f \right).
\end{align}
Hence
\begin{subequations}\label{eq3.62}
\begin{equation}\label{eq3.62a}
M^{v-e}u(x^{v-e}) = div_y \left(e^{\zeta/2} A^y \nabla_y u\right)
\end{equation}
where
\begin{equation}\label{eq3.62b}
A^y = (N^y)^T A N^y.
\end{equation}
\end{subequations}
Here by (\ref{eq3.54}) and (\ref{eq3.55})
\begin{equation}\label{eq3.63}
N^y = O^v Q^y
\end{equation}
where
\[Q^y = \left[
\begin{array}{ccc}
\frac{1}{\cos(\phi)}   &0&  \frac{-\sin^2 \phi} {\sin
\left(\phi^{v-e}_{l+1}\right)
\cos (\phi)}\\0 & 1 & 0 \\
0 & 0 & \frac{\sin(\phi)}{ \sin \left(\phi^{v-e}_{l+1}\right)}
\end{array} \right]\:.\]
Clearly, there exist positive constants $\mu_0$ and $\mu_1$ such that
\begin{equation}\label{eq3.64}
\mu_0 I \leq A^y \leq \mu_1 I
\end{equation}
for all elements $\Omega^{v-e}_n.$ Moreover there exists a constant $C$ such
that $a^y_{i,j}$ and its derivatives with respect to $y$ are uniformly bounded
in $\hat{\Omega}^{v-e}_n$. Hence we obtain the following result.
\begin{lem}
Let $w^{v-e}(x_1^{v-e})$ be a smooth, positive weight function such that
$w^{v-e}(x_1^{v-e})=1$ for all $x_1^{v-e}$ such that
$$
x_1^{v-e}\geq \psi_1^{v-e}= \ln\left(\tan(\phi_1^{v-e})\right)
\;and\; \int_{-\infty}^{\psi_1^{v-e}} w^{v-e}(x_1^{v-e})=1.
$$
Then there exists a positive constant $C_{v-e}$ such that the estimate
\begin{align}\label{eq3.65}
&\frac{\mu_0^2}{2} \int_{\tilde{\Omega}^{v-e}_n}
e^{x_3^{v-e}}\left(\sum_{i,j=1}^2 \left(\frac{\partial^2 u}
{\partial x_i^{v-e}\partial x_j^{v-e}}\right)^2
+\sum_{i=1}^2\sin^2(\phi)\left(\frac{\partial^2 u}
{\partial x_i^{v-e}\partial x_3^{v-e}}\right)^2\right. \notag \\
&\quad\quad \left.+\sin^4(\phi)\left(\frac{\partial^2 u}
{\partial (x_3^{v-e})^2} \right)^2\right)
\:d x^{v-e}\leq \int_{\tilde{\Omega}^{v-e}_n} \left|\:L^{v-e}
u(x^{v-e})\:\right|^2 \:d x^{v-e} \notag \\
&\quad\quad-\left\{\sum_k\oint_{\partial \tilde
{\Gamma}_{n,k}^{v-e}}e^{x_3^{v-e}}\left(\frac{\partial u}
{\partial \nb^{v-e}}\right)_{A^{v-e}} \left(\frac{\partial u}
{\partial\nuw^{v-e}}\right)_{A^{v-e}}\:d s^{v-e}\right. \notag \\
&\qquad-\left. 2\sum_k\int_{\tilde{\Gamma}_{n,k}^{v-e}}
e^{x^{v-e}_3}\sum_{l=1}^2 \left(\frac{\partial u} {\partial
\taux_l^{v-e}}\right)_{A^{v-e}}\frac{\partial} {\partial
s^{v-e}_l}\left(\frac{\partial u} {\partial
\nuw^{v-e}}\right)_{A^{v-e}}\:d \sigma^{v-e}\right\} \notag \\
&\qquad+ C_{v-e}\Bigg(\int_{\tilde{\Omega}^{v-e}_n} e^{x^{v-e}_3}
\left(\sum_{i=1}^2 \left(\frac{\partial u}{\partial
x_i^{v-e}}\right)^2 + \sin^2 (\phi) \left(\frac{\partial u}
{\partial x_3^{v-e}}\right)^2\right)\:d x^{v-e}\notag \\
&\qquad+ \int_{\tilde{\Omega}_n^{v-e}}e^{x_3^{v-e}}u^2
\:w^{v-e}(x_1^{v-e})\:dx^{v-e}\Bigg)
\end{align}
holds for all $\Omega_n^{v-e}\subseteq\Omega^{v-e}$.
\end{lem}
\begin{proof}
Since the spectral element function $u(x^{v-e})$ is a function of only
$x_3^{v-e}$ if $\Omega_n^{v-e}\subseteq\left\{x:0<\phi<\phi_1^{v-e}
\right\}$ the result follows. Here we have used the fact that the second
fundamental form is identically zero.
\end{proof}
Now we can prove the following result.
\begin{lem}\label{lem3.2.6}
Let $\Gamma_{k,i}^v = \Gamma_{q,r}^{v-e}$. Then the following
identity holds.
\begin{align}\label{eq3.66}
&\sin^2 (\phi_v) \oint_{\partial \tilde{\Gamma}^v_{k,i}} e^{x_3^v}
\sin(x_1^v) \left(\frac{\partial u}{\partial \nb^v}\right)_{A^v}
\left(\frac{\partial u}{\partial \nuw^v}\right)_{A^v} \:d s^v \notag \\
&\quad - 2 \sin^2 (\phi_v) \int_{ \tilde{\Gamma}^v_{k,i}}
e^{x_3^v} \sin(x_1^v) \sum_{j=1}^2 \left(\frac{\partial
u}{\partial \taux_j^v} \right)_{A^v} \frac{\partial}{\partial
s^v_j} \left(\left(\frac {\partial u}{\partial\nuw^v}\right)_{A^v}\right)
\:d\sigma^v \notag \\
&= -\oint_{\partial \tilde{\Gamma}^{v-e}_{q,r}} e^{x_3^{v-e}}
\left(\frac{\partial u}{\partial \nb^{v-e}}\right)_{A^{v-e}}
\left(\frac{\partial u}{\partial \nuw^{v-e}}\right)_{A^{v-e}}
\:ds^{v-e} \notag \\
&\quad + 2\int_{ \tilde{\Gamma}^{v-e}_{q,r}} e^{x_3^{v-e}}
\sum_{j=1}^2 \left(\frac{\partial u}{\partial
\taux_j^{v-e}}\right)_{A^{v-e}} \frac{\partial}{\partial
s^{v-e}_j} \left(\left(\frac{\partial u}{\partial
\nuw^{v-e}}\right)_{A^{v-e}}\right) \:d \sigma^{v-e}.
\end{align}
\end{lem}
\begin{proof}
The proof is provided in Appendix B.1.
\end{proof}

\subsection{Estimates for second derivatives in edge neighbourhoods}
Consider the edge $e$ whose end points are $v$ and $v^{\prime}$. Let the length of
$e$ be $l_e$. We define the edge neighbourhood (Figure \ref{fig2.3})
\begin{eqnarray*}
\Omega^e = \left\{ x \in \Omega: 0<r< \rho_v \sin \phi_v = Z,\:
\theta^{l}_{v-e} < \theta < \theta^u_{v-e},\:\delta_v < x_3< l_e -
\delta_{v^{\prime}} \right\}.
\end{eqnarray*}
Here $(r,\theta,x_3)$ are polar coordinates with origin at $v$. We define a
geometrical mesh on $\Omega^e$ as in Figure \ref{fig2.10}.

Let $\Omega_u^e$ denote an element
\begin{equation}\label{eq3.67}
\Omega_u^e =\left\{ x: r^e_j < r < r^e_{j+1},\:
\theta^e_k < \theta < \theta^e_{k+1},\: Z_m^e < x_3 < Z_{m+1}^e
\right\}
\end{equation}
of the geometrical mesh imposed on the edge
neighbourhood $\Omega^e$. We now introduce a new set of
coordinates in the edge neighbourhood
\begin{align}\label{eq3.68}
x^e_1 &= \tau = \ln r \notag\\
x^e_2 &= \theta \notag\\
x_3^e &= x_3\:.
\end{align}
Let $\tilde\Omega_u^{e}$ denote the image of $\Omega_u^{e}$ in
$x^e$ coordinate. Now
\begin{equation}\label{eq3.69}
\nabla_x u = R^e \nabla_{x^e} u
\end{equation}
where
\begin{align}\label{eq3.70}
R^e = \left[ \begin{array}{ccc}
e^{-\tau} \cos \theta  & - e^{-\tau} \sin \theta & 0\\
e^{-\tau} \sin \theta  &   e^{-\tau} \cos \theta & 0 \\
0 & 0 & 1
\end{array} \right].
\end{align}
Hence
\begin{equation}\label{eq3.71}
d x = e^{2 \tau} d x^e\:.
\end{equation}
Moreover
\begin{equation}\label{eq3.72}
div_x(\fe)= e^{-2 \tau} div_{x^e} (e^{2 \tau}(R^e)^T \fe)\:.
\end{equation}
Here $\fe$ denotes a vector field.
\newline
We need to evaluate
\[\int_{\Omega^e_u} \left(\rho^2 \sin^2 \phi\right) |
Lu(x)|^2 d x = \int_{\Omega^e_u} r^2 |Lu(x)|^2 d x\:.\]
Clearly
\begin{equation}\label{eq3.73}
\int_{\Omega_u^{e}} r^2 |Lu(x)|^2 dx = \int_{\tilde\Omega_u^{e}}
|e^{2 \tau} Lu(x) |^2 d x^e.
\end{equation}
Let
\[Mu(x)=div(A\nabla_xu)\]
Now
\[e^{2\tau} M u(x) = div_{x^e} \left(e^{2 \tau} (R^e)^T A R^e
\nabla_{x^{e}} u\right).\]
Or
\begin{equation}\label{eq3.74}
e^{2\tau} M u(x) = div_{x^e}\left(A^e \nabla_{x^e} u\right).
\end{equation}
Here
\begin{equation}\label{eq3.75}
A^e = (S^e)^T A S^e
\end{equation}
and
\begin{align}\label{eq3.76}
S^e=\left[ \begin{array}{ccc}
\cos \theta  & -\sin \theta & 0\\
\sin \theta  &  \cos \theta & 0 \\
0 & 0 & e^{\tau}
\end{array} \right].
\end{align}
Hence
\begin{equation}\label{eq3.77}
e^{2\tau} Lu(x) = div_{x^e}\left(A^e \nabla_{x^e}u\right)
+ \sum_{i=1}^3 \hat{b}_i^e u_{x^e_i} + \hat{c}^e u\:.
\end{equation}
Now the matrix $A^e$ becomes singular as the element
$\Omega_u^e$ approaches the edge $e$. We note that
\begin{align}\label{eq3.78}
\| \hat{b}^e \|_{0, \infty, \tilde{\Omega}^e} &= O(e^{\tau})
\mbox{ and } \notag \\
\| \hat{c}^e \|_{0, \infty, \tilde{\Omega}^e} &= O(e^{2\tau})\:.
\end{align}
Now
\[\Omega^e_u = \left\{ x: r^e_j < r < r^e_{j+1},\: \theta^e_k <
\theta < \theta^e_{k+1},\: \delta_v < x_3 < l_e - \delta_{v^{\prime}}
\right\}.\]
To overcome the singular nature of $A^e$ as $j \rightarrow 0$ we once
again introduce a set of local coordinates $z$ in $\tilde{\Omega}^e_u$
defined as
\begin{align}\label{eq3.79}
z_1 &= x_1^e \notag\\
z_2 &= x^e_2 \notag\\
z_3 &= \frac{x_3^e}{r^e_{j+1}}\:.
\end{align}
Then $\tilde{\Omega}^e_u$ is mapped onto the hexahedron
$\hat{\Omega}^e_u$ such that the length of the $z_3$ side becomes
large as $\Omega^e_u$ approaches the edge of the domain $\Omega$.
\newline
Define
\begin{equation}\label{eq3.80}
M^e u (x^e) = e^{2 \tau} M u(x) = div_{x^e}\left(A^e \nabla_{x^e}u\right).
\end{equation}
Then
\begin{equation}\label{eq3.81}
M^e u(x^e) = M^z u(z) = div_z\left(A^z \nabla_z u\right).
\end{equation}
Here
\begin{equation}\label{eq3.82}
A^z = (T^z)^T A T^z
\end{equation}
where
\begin{align}\label{eq3.83}
T^z = \left[ \begin{array}{ccc}
\cos \theta  & \sin \theta & 0\\
-\sin \theta & \cos \theta & 0 \\
0 & 0 & \frac{e^{\tau}}{r^e_{j+1}}
\end{array} \right].
\end{align}
Clearly there exist positive constants $\mu_0$ and $\mu_1$ such that
\begin{equation}\label{eq3.84}
\mu_0 I \leq A^z \leq \mu_1 I\:.
\end{equation}
Moreover there exists a constant $C$ such that $a_{i,j}^z$ and its derivatives
with respect to $z$ are uniformly bounded in $\hat{\Omega}_u^e$.  Here
$a_{i,j}^z$ denotes the elements of the matrix $A^z$. Hence we obtain the
following result.
\begin{lem}
Let $w^e(x_1^e)$ be a smooth, positive weight function such that
$w^e(x_1^e)=1$ for all \;$x_1^{e}\geq \tau_1^e=\ln(r_1^{e})$
and $\int_{-\infty}^{\tau_1^e} w^e(x_1^{e}) \:d x_1^{e}=1.$ Then
there exists a positive constant $C_e$ such that
\begin{align}\label{eq3.85}
&\frac{\mu_0^2}{2} \int_{\tilde{\Omega}^e_u} \left(\sum_{i,j=1}^2
\left(\frac{\partial^2 u}{\partial x_i^e \partial x_j^e}\right)^2
+ e^{2 \tau} \sum_{i=1}^2 \left(\frac{\partial^2 u}{\partial x_i^e
\partial x_3^e}\right)^2 + e^{4 \tau}\left(\frac{\partial^2 u}
{\partial {x_3^e}^2 }\right)^2\right) \:d x^e \notag \\
&\quad\leq \int_{\tilde{\Omega}^e_u} | L^e u (x^e)|^2
\:d x^e - \left(\sum_k \oint_{\partial
\tilde{\Gamma}_{u,k}^e} \left(\frac {\partial u}{\partial
\nb^e}\right)_{A^e} \left(\frac{\partial u}{\partial
\nuw^e}\right)_{A^e} \:d s^e \right. \notag \\
&\quad- \left. 2 \sum_k \int_{\tilde{\Gamma}^e_{u,k}} \sum_{l=1}^2
\left(\frac {\partial u}{\partial \taux_l^e}\right)_{A^e}
\frac{\partial} {\partial s^e_l} \left(\left(\frac{\partial
u}{\partial \nuw^e}\right)_{A^e}\right)\:d \sigma^e\right) \notag \\
&\quad+ C_e \left(\int_{\tilde{\Omega}^e_u} \left(\sum_{i=1}^2
\left(\frac{\partial u}{\partial x^e_i}\right)^2 + e^{2 \tau}
\left(\frac{\partial u} {\partial x_3^e}\right)^2\right)\:dx^e
+ \int_{\tilde{\Omega}_u^e}u^2w^e(x_1^{e})\:d x^e\right)
\end{align}
holds for all $\Omega_u^{e}\subseteq \Omega^{e}$.
\end{lem}
\begin{proof}
Since the spectral element function $u(x^e)$ is only a function of $x_3^{e}$ if
$\Omega_u^{e}\subseteq\{x:r<r_1^{e}\}$ the result follows. Here we have used the
fact that the second fundamental form is zero.
\end{proof}
Finally we state the following results.
\begin{lem}\label{lem3.2.8}
Let $\Gamma_{u,k}^e = \Gamma_{n,l}^{v-e}$. Then
\begin{align}\label{eq3.86}
& \oint_{\partial \tilde{\Gamma}_{n,l}^{v-e}} e^{x_3^{v-e}}
\left(\frac{\partial u}{\partial \nb^{v-e}}\right)_{A^{v-e}}
\left(\frac{\partial u}{\partial \nuw^{v-e}}\right)_{A^{v-e}}
\:ds^{v-e} \notag \\
&\quad -2\sum_{j=1}^2 \int_{\tilde{\Gamma}^{v-e}_{n,l}}
e^{x_3^{v-e}} \left(\frac{\partial u}{\partial
\taux_j^{v-e}}\right)_{A^{v-e}} \frac{\partial}{\partial
s^{v-e}_j} \left(\left(\frac{\partial u} {\partial
\nuw^{v-e}}\right)_{A^{v-e}}\right)\:d\sigma^{v-e} \notag \\
&\quad\quad\quad= -\oint_{\partial \tilde{\Gamma}_{u,k}^{e}}
\left(\frac{\partial u}{\partial \nb^{e}}\right)_{A^{e}}
\left(\frac{\partial u}{\partial \nuw^{e}}\right)_{A^{e}}
ds^{e} \notag \\
&\quad\quad\quad+2 \sum_{j=1}^2 \int_{\tilde{\Gamma}^{e}_{u,k}}
\left(\frac{\partial u}{\partial \taux_j^{e}}\right)_{A^{e}}
\frac{\partial}{\partial s^{e}_j}\left(\left(\frac{\partial
u}{\partial \nuw^e}\right)_{A^{e}}\right)d\sigma^{e}\:.
\end{align}
\end{lem}
\begin{proof}
The proof is provided in Appendix B.2.
\end{proof}
\begin{lem}\label{lem3.2.9}
Let $\Gamma^e_{u,k} = \Gamma^r_{l,j}$. Then
\begin{align}\label{eq3.87}
\oint_{\partial \tilde{\Gamma}^e_{u,k}} \left(\frac{\partial
u}{\partial \nuw^e}\right)_{A^e} \left(\frac{\partial u}{\partial
\nb^e} \right)_{A^e} \:d s^e = -\rho_v^2 \sin^2(\phi_v)
\oint_{\partial \Gamma^r_{l,j}} \left(\frac{\partial u}{\partial
\nb}\right)_A \left(\frac{\partial u}{\partial \nuw}\right)_A \:d s,
\end{align}
and
\begin{align}\label{eq3.88}
& \sum_{m=1}^2 \int_{\tilde{\Gamma}^e_{u,k}} \left(\frac{\partial
u} {\partial \taux^e_m}\right)_{A^e} \frac{\partial}{\partial
s^e_m}\left(\frac{\partial u}{\partial \nuw^e}\right)_{A^e}
d \sigma^e \notag \\
&= - \rho_v^2 \sin^2(\phi_v)\left(\sum_{m=1}^2\int_{\Gamma_{l,j}^r}
\left(\frac{\partial u}{\partial\taux_m}\right)_A\frac{\partial}
{\partial s_m}\left(\frac{\partial u}{\partial \nuw}\right)_{A}
\:d \sigma \right).
\end{align}
\end{lem}
\begin{proof}
The proof is provided in Appendix B.3.
\end{proof}

\section{Estimates for Lower Order Derivatives}
\begin{lem}\label{lem3.3.1}
We can define a set of corrections $\{\eta_l^r\}_{l=1,\ldots,N_r}$,
$\{\eta_l^v\}_{l=1,\ldots,N_v}$ for $v\in \mathcal V$,
$\{\eta_l^{v-e}\} _{l=1,\ldots,N_{v-e}}$ for\, ${v-e}\in \mathcal{V-E}$
and $\{\eta_l^e\} _{l=1,\ldots,N_e}$ for\, $e\in \mathcal E$ such that
the corrected spectral element function $p$ defined as
\begin{align}\label{eq3.89}
p_l^r=u_l^r+\eta_l^r \hspace{1.4cm} & {\it for} \ l=1,\ldots,N_r, \notag \\
p_l^v=u_l^v+\eta_l^v \hspace{1.4cm} & {\it for}\ l=1,\ldots,N_v
\quad and\quad v\in\mathcal V, \notag \\
p_l^{v-e}=u_l^{v-e}+\eta_l^{v-e} \hspace{0.3cm} & {\it for} \
l=1,\ldots,N_{v-e}\quad and\quad {v-e}\in\mathcal {V-E}, \notag \\
p_l^e=u_l^e+\eta_l^e \hspace{1.4cm} & {\it for }\ l=1,\ldots,N_e
\quad and \quad e\in\mathcal E,
\end{align}
is conforming and $p\in H_0^1(\Omega)$ i.e. $p\in H^1(\Omega)$
and $p$ vanishes on $\Gamma^{[0]}$. Define
\begin{align}\label{eq3.90}
\mathcal U^{N,W}_{(1)}(\{\mathcal F_s\})&=\sum_{l=1}^{N_r} \left\|
\:s_l^r(x_1,x_2,x_3)\:\right\|
^2_{1,\Omega_l^r}+\sum_{v\in \mathcal V}\sum_{l=1}^{N_v}
\left\|\:s_l^v(x_1^v,x_2^v,x_3^v) e^{{x_3^v}/2}
\:\right\|^2_{1,\tilde{\Omega}_l^v} \notag \\
&\quad+\sum_{{v-e}\in \mathcal {V-E}}\Bigg(\mathop{\sum_{l=1}}
_{\mu(\tilde{\Omega}_l^{v-e})<\infty}^{N_{v-e}}
\int_{\tilde{\Omega}_l^{v-e}}e^{x_3^{v-e}}\left(\sum_{i=1}^2
\left(\frac{\partial s_l^{v-e}}{\partial x_i^{v-e}}\right)^2
+\sin^2\phi\left(\frac{\partial s_l^{v-e}}{\partial x_3^{v-e}}
\right)^2\right. \notag \\
&\quad+ (s_l^{v-e})^2\Bigg)\:dx^{v-e}
+\mathop{\sum_{l=1}}_{\mu(\tilde{\Omega}_l^{v-e})=\infty}
^{N_{v-e}}\int_{\tilde{\Omega}_l^{v-e}}e^{x_3^{v-e}}
(s_l^{v-e})^2\:w^{v-e}(x_{1}^{v-e})\:dx^{v-e}\Bigg) \notag \\
&\quad +\sum_{e\in \mathcal E}\Bigg(\mathop{\sum_{l=1}} _{\mu
(\tilde{\Omega}_l^e)<\infty}^{N_e}\int_{\tilde{\Omega}_l^e}
\left(\sum_{i=1}^2\left(\frac{\partial s_l^e}{\partial
x_i^e}\right)^2\right.+\left.e^{2\tau} \left(\frac
{\partial s_l^e}{\partial x_3^e}\right)^2+(s_l^e)^2\right)
\:dx^e \notag \\
&\quad+\mathop{\sum_{l=1}}_{\mu(\tilde{\Omega}_l^e)=\infty}
^{N_e}\int_{\tilde{\Omega}_l^e}(s_l^e)^2\:w^e(x_1^e)\:dx^e\Bigg)
\end{align}
Then the estimate
\begin{equation}\label{eq3.91}
\mathcal U^{N,W}_{(1)}(\{\mathcal F_\eta\})
\leq C_W \mathcal V^{N,W}(\{\mathcal F_u\})
\end{equation}
holds. Here $C_W$ is a constant, if the spectral element functions
are conforming on the wirebasket $W\!B$ of the elements, otherwise
$C_W=C(\ln W)$, where $C$ is a constant.
\end{lem}
\begin{proof}
The proof is provided in Appendix C.1.
\end{proof}
\begin{theo}\label{thm3.3.1}
The following estimate for the spectral element functions holds
\begin{equation}\label{eq3.92}
\mathcal U^{N,W}_{(1)}(\{\mathcal F_u\})
\leq K_{N,W} \mathcal V^{N,W}(\{\mathcal F_u\})
\end{equation}
Here $K_{N,W}=CN^4$, when the boundary conditions are mixed and
$K_{N,W}=C(\ln W)^2$ when the boundary conditions are Dirichlet.
\newline
If the spectral element functions vanish on the wirebasket $W\!B$ of
the elements then $K_{N,W}=C(\ln W)^2$, where $C$ is a constant.
\end{theo}
\begin{proof}
The proof is similar to the proof of Theorem $3.1$ in~\cite{DT}
and is provided in Appendix C.2.
\end{proof}

\section{Estimates for Terms in the Interior}
\subsection{Estimates for terms in the interior of $\Omega^r$}
\begin{lem}\label{lem3.4.1}
Let $\Omega_m^r$ and $\Omega_p^r$ be elements in the regular
region $\Omega^r$ of $\Omega$ and $\Gamma_{m,i}^r$ be a face of
$\Omega_m^r$ and $\Gamma_{p,j}^r$ be a face of $\Omega_p^r$ such
that $\Gamma_{m,i}^r=\Gamma_{p,j}^r$. Then for any $\epsilon>0$
there exists a constant $C_{\epsilon}$ such that for $W$ large
enough
\begin{align}\label{eq3.93}
&\left|\int_{\partial{\Gamma}_{m,i}^r}\left(\left(\frac{\partial
u_m^r}{\partial \nuw}\right)_A\left(\frac{\partial u_m^r}
{\partial\nb}\right)_A-\left(\frac{\partial u_p^r}{\partial\nuw}
\right)_A\left(\frac{\partial u_p^r}{\partial\nb}\right)_A\right)
\:ds\right| \notag \\
&\leq C_{\epsilon}(\ln W)^2\sum_{k=1}^3 \left\|[u_{x_k}]
\right\|^2_{1/2,{\Gamma_{m,i}^r}}
+\epsilon\sum_{1\leq |\alpha|\leq 2} \left(\|D_x^{\alpha}u_m^r
\|^2_{0,\Omega_m^r}+\|D_x^{\alpha}u_p^r \|^2_{0,\Omega_p^r}\right)\:.
\end{align}
\end{lem}
\begin{proof}
The proof is similar to the proof of Lemma $3.1$ in~\cite{DTK1}
and is provided in Appendix D.1.
\end{proof}
\begin{lem}\label{lem3.4.2}
Let $\Omega_m^r$ and $\Omega_p^r$ be elements in the regular
region $\Omega^r$ of $\Omega$ and $\Gamma_{m,i}^r$ be a face of
$\Omega_m^r$ and $\Gamma_{p,j}^r$ be a face of $\Omega_p^r$ such
that $\Gamma_{m,i}^r=\Gamma_{p,j}^r$. Then for any $\epsilon>0$
there exists a constant $C_{\epsilon}$ such that for $W$ large
enough
\begin{align}\label{eq3.94}
&\left|\sum_{j=1}^2\left(\int_{{\Gamma}_{m,i}^r}\left(\frac
{\partial u_m^r}{\partial \taux_j}\right)_A\frac{\partial}
{\partial s_j}\left(\frac{\partial u_m^r}{\partial\nuw}\right)
_A\: d\sigma-\int_{{\Gamma}_{p,j}^r}\left(\frac{\partial u_p^r}
{\partial\taux_j}\right)_A\frac{\partial}{\partial s_j}
\left(\frac{\partial u_p^r}{\partial\nuw}\right)_A
d\sigma\right)\right| \notag \\
&\leq C_{\epsilon}(\ln W)^2 \sum_{k=1}^3\left\|[u_{x_k}]
\right\|^2_{1/2,{{\Gamma}_{p,j}^r}}
+\epsilon\sum_{1\leq|\alpha|\leq 2}\left(\|D_x^{\alpha}
u_m^r\|^2_{0,\Omega_m^r}+\|D_x^{\alpha}u_p^r\|^2_{0,\Omega_p^r}\right)\:.
\end{align}
\end{lem}
\begin{proof}
The proof is similar to the proof of Lemma $3.3$ in~\cite{DTK1}
and is provided in Appendix D.2.
\end{proof}

\subsection{Estimates for terms in the interior of $\Omega^e$}
\begin{lem}\label{lem3.4.3}
Let $\Omega_m^e$ and $\Omega_p^e$ be elements in the edge
neighbourhood $\Omega^e$ of $\Omega$ and $\Gamma_{m,i}^e$
be a face of $\Omega_m^e$ and $\Gamma_{p,j}^e$ be a face
of $\Omega_p^e$ such that $\Gamma_{m,i}^e=\Gamma_{p,j}^e$
and $\mu{(\tilde{\Gamma}_{m,i}^e)<\infty}$.
Then for any $\epsilon>0$ there exists a constant
$C_{\epsilon}$ such that for $W$ large enough
\begin{subequations}\label{eq3.95}
\begin{align}\label{eq3.95a}
&\left|\oint_{\partial\tilde{\Gamma}_{m,i}^e}\left(\left
(\frac{\partial u_m^e}{\partial \nb^e}\right)_{A^e}\left(
\frac{\partial u_m^e}{\partial\nuw^e}\right)_{A^e}-\left(
\frac{\partial u_p^e}{\partial\nb^e}\right)_{A^e}\left(
\frac{\partial u_p^e}{\partial\nuw^e}\right)_{A^e}\right)
\:d s^e\right| \notag \\
&\quad\quad\leq C_{\epsilon}(\ln W)^2\left(\big|\big|\big|
\:[u_{x_1^e}]\:\big|\big|\big|^2_{\tilde{\Gamma}_{m,i}^e}
+\big|\big|\big|\:[u_{x_2^e}]\:\big|\big|\big|^2_{\tilde
{\Gamma}_{m,i}^e}+\big|\big|\big|\:G_{m,i}^e[u_{x_3^e}]
\:\big|\big|\big|^2_{\tilde{\Gamma}_{m,i}^e}\right) \notag \\
&\quad\quad+\epsilon\sum_{k=m,p}\left(\int_{\tilde
{\Omega}_k^e}\left(\sum_{i,j=1,2}\left(\frac{\partial^2
u_k^e}{\partial x_i^e\partial x_j^e} \right)^2
+e^{2\tau}\sum_{i=1}^2\left(\frac{\partial^2 u_k^e}
{\partial x_i^e\partial x_3^e}\right)^2+e^{4\tau}\left
(\frac{\partial^2 u_k^e}{\left(\partial x_3^e\right)^2}
\right)^2\right.\right. \notag \\
&\quad\quad\left.\left.+\sum_{i=1}^2\left(\frac {\partial
u_k^e}{\partial x_i^e}\right)^2+e^{2\tau}\left(\frac
{\partial u_k^e}{\partial x_3^e}\right)^2\right)
\;dx^e\right)\:.
\end{align}
Here $C_\epsilon$ is a constant which depends on $\epsilon$
but is uniform for all $\tilde{\Gamma}_{m,i}^e\subseteq
\tilde{\Omega}^e$, and $G_{m,i}^e=\underset{x^e\in
{\tilde{\Gamma}_{m,i}^e}}{sup}(e^\tau)$.\\
If $\mu(\tilde{\Gamma}_{m,i}^e)=\infty$ then for any
$\epsilon>0$
\begin{align}\label{eq3.95b}
&\left|\oint_{\partial\tilde{\Gamma}_{m,i}^e}\left(\left
(\frac{\partial u_m^e}{\partial \nb^e}\right)_{A^e}\left(
\frac{\partial u_m^e}{\partial\nuw^e}\right)_{A^e}-\left(
\frac{\partial u_p^e}{\partial\nb^e}\right)_{A^e}\left(
\frac{\partial u_p^e}{\partial\nuw^e}\right)_{A^e}\right)
\:d s^e\right| \notag \\
&\quad\leq\epsilon\left(\int_{\tilde{\Omega}_m^e}(u_m^e)^2
w^e(x_1^e)\:dx^e+\int_{\tilde{\Omega}_p^e}(u_p^e)^2
w^e(x_1^e)\:dx^e\right)
\end{align}
\end{subequations}
provided $W=O(e^{{N}^\alpha})$ with $\alpha<1/2$.
\end{lem}
\begin{proof}
The proof is provided in Appendix D.3.
\end{proof}
\begin{lem}\label{lem3.4.4}
Let $\Omega_m^e$ and $\Omega_p^e$ be elements in the edge
neighbourhood $\Omega^e$ of $\Omega$ and $\Gamma_{m,i}^e$
be a face of $\Omega_m^e$ and $\Gamma_{p,j}^e$ be a face
of $\Omega_p^e$ such that $\Gamma_{m,i}^e=\Gamma_{p,j}^e$
and $\mu{(\tilde{\Gamma}_{m,i}^e)<\infty}$.
Then for any $\epsilon>0$ there exists a constant
$C_{\epsilon}$ such that for $W$ large enough
\begin{subequations}\label{eq3.96}
\begin{align}\label{eq3.96a}
&\left|\int_{\tilde{\Gamma}_{m,i}^e}\sum_{l=1}^2\left(\left(
\frac{\partial u_m^e}{\partial \taux_l^e}\right)_{A^e}\frac
{\partial}{\partial s_l^e}\left(\left(\frac{\partial u_m^e}
{\partial\nuw^e}\right)_{A^e}\right)
-\left(\frac{\partial u_p^e}
{\partial\taux_l^e}\right)_{A^e}\frac{\partial}{\partial s_l^e}
\left(\left(\frac{\partial u_p^e}{\partial\nuw^e}\right)_{A^e}
\right)\right)\: d \sigma^e\right| \notag \\
&\quad\quad\leq C_{\epsilon}(\ln
W)^2\left(\big|\big|\big|\:[u_{x_1^e}]
\:\big|\big|\big|^2_{\tilde{\Gamma}_{m,i}^e}+\big|\big|\big|\:
[u_{x_2^e}]\:\big|\big|\big|^2_{\tilde{\Gamma}_{m,i}^e}+
\big|\big|\big|\:G_{m,i}^e[u_{x_3^e}]\:\big|\big|\big|^2_{\tilde
{\Gamma}_{m,i}^e}\right) \notag \\
&\quad\quad+\epsilon\sum_{k=m,p}\left(\int_{\tilde
{\Omega}_k^e}\left(\sum_{i,j=1,2}\left(\frac{\partial^2 u_k^e}
{\partial x_i^e\partial x_j^e}\right)^2
+e^{2\tau}\sum_{i=1}^2\left(\frac{\partial^2 u_k^e}
{\partial x_i^e \partial x_3^e}\right)^2+e^{4\tau}
\left(\frac{\partial^2 u_k^e}{\left(\partial x_3^e
\right)^2}\right)^2\right.\right. \notag \\
&\quad\quad\left.\left.+\sum_{i=1}^2\left(\frac
{\partial u_k^e}{\partial x_i^e}\right)^2+ e^{2\tau}
\left(\frac{\partial u_k^e}{\partial x_3^e}\right)^2\right)\;dx^e\right)\:.
\end{align}
If $\mu(\tilde{\Gamma}_{m,i}^e)=\infty$ then for any $\epsilon>0$
\begin{align}\label{eq3.96b}
&\left|\int_{\tilde{\Gamma}_{m,i}^e}\sum_{l=1}^2\left(\left(
\frac{\partial u_m^e}{\partial \taux_l^e}\right)_{A^e}\frac
{\partial}{\partial s_l^e}\left(\left(\frac{\partial u_m^e}
{\partial\nuw^e}\right)_{A^e}\right)-\left(\frac{\partial u_p^e}
{\partial\taux_l^e}\right)_{A^e}\frac{\partial}{\partial s_l^e}
\left(\left(\frac{\partial u_p^e}{\partial\nuw^e}\right)_{A^e}
\right)\right)\: d \sigma^e\right| \notag \\
&\quad\leq\epsilon\left(\int_{\tilde{\Omega}_m^e}(u_m^e)^2
w^e(x_1^e)\:dx^e+\int_{\tilde{\Omega}_p^e}(u_p^e)^2 w^e(x_1^e)\:dx^e\right)
\end{align}
\end{subequations}
provided $W=O(e^{{N}^\alpha})$ with $\alpha<1/2$.
\end{lem}
\begin{proof}
The proof is provided in Appendix D.4.
\end{proof}
We now state estimates for terms in the interior of vertex neighbourhoods and
vertex-edge neighbourhoods the proofs of which are similar to those for Lemma
\ref{lem3.4.1} to Lemma \ref{lem3.4.4}.

\subsection{Estimates for terms in the interior of {$\Omega^v$}}
\begin{lem}
Let $\Omega_m^v$ and $\Omega_p^v$ be elements in the vertex neighbourhood
$\Omega^v$ of $\Omega$ and $\Gamma_{m,i}^v$ be a face of $\Omega_m^v$ and
$\Gamma_{p,j}^v$ be a face of $\Omega_p^v$ such that $\Gamma_{m,i}^v=\Gamma_{p,j}^v$.
Then for any $\epsilon>0$ there exists a constant $C_{\epsilon}$ such that
for $W$ large enough
\begin{align}\label{eq3.97}
&\left|\oint_{\partial\tilde{\Gamma}_{m,i}^v}e^{x_3^v}
\sin(x_1^v)\left(\left (\frac{\partial u_m^v}{\partial
\nb^v}\right)_{A^v} \left(\frac{\partial
u_m^v}{\partial\nuw^v}\right)_{A^v} -\left(\frac{\partial
u_p^v}{\partial\nb^v}\right)_{A^v} \left(\frac{\partial
u_p^v}{\partial\nuw^v}\right)_{A^v}\right)\:d s^v\right| \notag \\
&\quad\leq C_\epsilon(\ln W)^2\sum_{k=1}^3R_{m,i}^v
\left\|[u_{x_k}]\right\|^{2}_{1/2,\tilde{\Gamma}_{m,i}^v} \notag \\
&\quad+\epsilon\sum_{1\leq|\alpha|\leq 2}
\left(\|e^{{x_3^v}/2}D_{x^v}^\alpha u_m^v\|_{0,\tilde
{\Omega}_m^v}^2+\|e^{{x_3^v}/2}D_{x^v}^\alpha
u_p^v\|_{0,\tilde{\Omega}_p^v}^2\right).
\end{align}
Here $R_{m,i}^v=\underset{x^v\in \tilde{\Gamma}_{m,i}^v}{sup}(e^{x_3^v})$.
\newline
If $\mu(\tilde{\Gamma}_{m,i}^v)=\infty$ then the integral in left hand side
of (\ref{eq3.97}) is zero.
\end{lem}
\begin{lem}
Let $\Omega_m^v$ and $\Omega_p^v$ be elements in the vertex neighbourhood
$\Omega^v$ of $\Omega$ and $\Gamma_{m,i}^v$ be a face of $\Omega_m^v$ and
$\Gamma_{p,j}^v$ be a face of $\Omega_p^v$ such that $\Gamma_{m,i}^v=\Gamma_{p,j}^v$.
Then for any $\epsilon>0$ there exists a constant $C_{\epsilon}$ such that
for $W$ large enough
\begin{align}\label{eq3.98}
&\left|\sum_{l=1}^2\left(\int_{{\Gamma}_{m,i}^v}e^{x_3^v}
\sin(x_1^v)\left(\left(\frac {\partial u_m^v}{\partial
\taux_l^v}\right)_{A^v}\frac{\partial} {\partial
s_l^v}\left(\left(\frac{\partial u_m^v}{\partial\nuw^{v}}
\right)_{A^v}\right)\:d\sigma^v\right.\right.\right. \notag \\
&\quad\left.\left.\left.-\int_{{\Gamma}_{p,j}^v}\left(\frac
{\partial u_p^v}{\partial\taux_l^v}\right)_{A^v}\frac
{\partial}{\partial s_l^v}\left(\left(\frac{\partial u_p^v}
{\partial\nuw^{v}}\right)_{A^v}\right)\:d\sigma^v\right)\right)
\right| \notag \\
&\quad\leq C_\epsilon(\ln W)^2\sum_{k=1}^3R_{m,i}^v
\left\|[u_{x_k}]\right\|^{2}_{1/2,\tilde{\Gamma}_{m,i}^v} \notag\\
&\quad+\epsilon\sum_{1\leq|\alpha|\leq 2}
\left(\|e^{{x_3^v}/2} D_{x^v}^\alpha u_m^v\|_{0,\tilde
{\Omega}_m^v}^2+\|e^{{x_3^v}/2} D_{x^v}^\alpha
u_p^v\|_{0,\tilde{\Omega}_p^v}^2\right).
\end{align}
If $\mu(\tilde{\Gamma}_{m,i}^v)=\infty$ then the integral in left hand side
of (\ref{eq3.98}) is zero.
\end{lem}

\subsection{Estimates for terms in the interior of {$\Omega^{v-e}$}}
\begin{lem}
Let $\Omega_m^{v-e}$ and $\Omega_p^{v-e}$ be elements in the vertex-edge
neighbourhood $\Omega^{v-e}$ of $\Omega$ and $\Gamma_{m,i}^{v-e}$ be a face
of $\Omega_m^{v-e}$ and $\Gamma_{p,j}^{v-e}$ be a face of $\Omega_p^{v-e}$
such that $\Gamma_{m,i}^{v-e}=\Gamma_{p,j}^{v-e}$. Then for any $\epsilon>0$
there exists a constant $C_{\epsilon}$ such that for $W$ large enough
\begin{subequations}\label{eq3.99}
\begin{align}\label{eq3.99a}
&\left|\oint_{\partial\tilde{\Gamma}_{m,i}^{v-e}}e^{x_3^{v-e}}
\left(\left(\frac{\partial u_m^{v-e}}{\partial
\nb^{v-e}}\right)_{A^{v-e}} \left(\frac{\partial
u_m^{v-e}}{\partial\nuw^{v-e}}\right)_{A^{v-e}}
\right.\right. \notag \\
&\hspace{2.0cm}\left.\left. -\left(\frac{\partial
u_p^{v-e}}{\partial\nb^{v-e}}\right)_{A^{v-e}}
\left(\frac{\partial u_p^{v-e}}{\partial\nuw^{v-e}}
\right)_{A^{v-e}}\right)\:d s^{v-e}\right| \notag \\
&\quad\leq C_{\epsilon}(\ln W)^2\left(\big|\big|\big|
\:[u_{x_1^{v-e}}]\:\big|\big|\big|^2_{\tilde{\Gamma}_{m,i}
^{v-e}}+\big|\big|\big|\:[u_{x_2^{v-e}}]\:
\big|\big|\big|^2_{\tilde{\Gamma}_{m,i}^{v-e}}\right. \notag \\
&\quad\left.+\big|\big|\big|\:E_{m,i}^{v-e}[u_{x_3^{v-e}}]
\:\big|\big|\big|^2_{\tilde{\Gamma}_{m,i}^{v-e}}\right)
+\epsilon\sum_{k=m,p}\left(\int_{\tilde{\Omega}_k^{v-e}}
\left(\sum_{i,j=1,2}\left(\frac{\partial^2 u_k^{v-e}}
{\partial x_i^{v-e}\partial x_j^ {v-e}}\right)^2\right.
\right. \notag \\
&\quad+\sin^2\phi\sum_{i=1}^2\left(\frac{\partial^2 u_k^{v-e}}
{\partial x_i^{v-e} \partial x_3^{v-e}} \right)^2
+\sin^4\phi\left(\frac{\partial^2 u_k^{v-e}}{\left(\partial
x_3^{v-e}\right)^2} \right)^2+\sum_{i=1}^2\left(\frac{\partial
u_k^{v-e}}{\partial x_i^{v-e}}\right)^2 \notag \\
&\quad\left.\left.+\sin^2\phi\left(\frac{\partial
u_k^{v-e}}{\partial x_3^{v-e}}\right)^2\right)e^{x_3^{v-e}}
\;dx^{v-e}\right)\:.
\end{align}
Here $C_\epsilon$ is a constant which depend on $\epsilon$
but is uniform for all $\tilde{\Gamma}_{m,i}^{v-e}\subseteq
\tilde{\Omega}^{v-e}$, and $E_{m,i}^{v-e}=\underset{x^{v-e}\in
{\tilde{\Gamma}_{m,i}^{v-e}}}{sup}(\sin\phi)$.
\newline
If $\mu(\tilde{\Gamma}_{m,i}^{v-e})=\infty$ then for any $\epsilon>0$
\begin{align}\label{eq3.99b}
&\left|\oint_{\partial\tilde{\Gamma}_{m,i}^{v-e}}
e^{x_3^{v-e}}\left(\left(\frac{\partial u_m^{v-e}}
{\partial\nb^{v-e}}\right)_{A^{v-e}} \left(\frac
{\partial u_m^{v-e}}{\partial\nuw^{v-e}}\right)_{A^{v-e}}
\right.\right. \notag \\
&\hspace{2.0cm}\left.\left. -\left(\frac{\partial
u_p^{v-e}}{\partial\nb^{v-e}}\right)_{A^{v-e}}
\left(\frac{\partial u_p^{v-e}}{\partial\nuw^{v-e}}
\right)_{A^{v-e}}\right)\:d s^{v-e}\right| \notag \\
&\quad\leq\epsilon\left(\sum_{k=m,p}\int_{\tilde{\Omega}
_k^{v-e}}(u_k^{v-e})^2e^{x_3^{v-e}}w^{v-e}(x_1^{v-e})\:dx^{v-e}\right)
\end{align}
\end{subequations}
provided $W=O(e^{{N}^\alpha})$ with $\alpha<1/2$.
\end{lem}
\begin{lem}
Let $\Omega_m^{v-e}$ and $\Omega_p^{v-e}$ be elements in the edge neighbourhood
$\Omega^{v-e}$ of $\Omega$ and $\Gamma_{m,i}^{v-e}$ be a face of $\Omega_m^{v-e}$
and $\Gamma_{p,j}^{v-e}$ be a face of $\Omega_p^{v-e}$ such that
$\Gamma_{m,i}^{v-e}=\Gamma_{p,j}^{v-e}$. Then for any $\epsilon>0$ there exists
a constant $C_{\epsilon}$ such that for $W$ large enough
\begin{subequations}\label{eq3.100}
\begin{align}\label{eq3.100a}
&\left|\int_{\tilde{\Gamma}_{m,i}^{v-e}}e^{x_3^{v-e}}\left(
\sum_{l=1}^2\left(\left(\frac{\partial u_m^{v-e}}{\partial\taux_l^{v-e}}
\right)_{A^{v-e}}\frac{\partial}{\partial s_l^{v-e}}
\left(\left(\frac{\partial u_m^{v-e}}{\partial\nuw^{v-e}}
\right)_{A^{v-e}}\right)\right.\right.\right. \notag \\
&\quad\quad\quad\quad\quad-\left.\left.\left.\left(\frac{\partial
u_p^{v-e}}{\partial\taux_l^{v-e}}\right)_{A^{v-e}}\frac
{\partial}{\partial s_l^{v-e}} \left(\left(\frac{\partial
u_p^{v-e}}{\partial\nuw^{v-e}}\right)_{A^{v-e}}
\right)\right)\right)\: d \sigma^{v-e}\right| \notag \\
&\quad\leq C_{\epsilon}(\ln W)^2\left(\big|\big|\big|
\:[u_{x_1^{v-e}}]\:\big|\big|\big|^2_{\tilde{\Gamma}_{m,i}
^{v-e}}+\big|\big|\big|\:[u_{x_2^{v-e}}]\:
\big|\big|\big|^2_{\tilde{\Gamma}_{m,i}^{v-e}}\right.\notag \\
&\quad\left.+\big|\big|\big|\:E_{m,i}^{v-e}[u_{x_3^{v-e}}]
\:\big|\big|\big|^2_{\tilde{\Gamma}_{m,i}^{v-e}}\right)
+\epsilon\sum_{k=m,p}\left(\int_{\tilde{\Omega}_k^{v-e}}
\left(\sum_{i,j=1,2}\left(\frac{\partial^2 u_k^{v-e}}{\partial
x_i^{v-e}\partial x_j^ {v-e}}\right)^2\right.\right. \notag \\
&\quad+\sin^2\phi\sum_{i=1}^2\left(\frac{\partial^2 u_k^{v-e}}
{\partial x_i^{v-e} \partial x_3^{v-e}} \right)^2
+\sin^4\phi\left(\frac{\partial^2 u_k^{v-e}}{\left(\partial
x_3^{v-e}\right)^2} \right)^2+\sum_{i=1}^2\left(\frac{\partial
u_k^{v-e}}{\partial x_i^{v-e}}\right)^2 \notag \\
&\quad\left.\left.+\sin^2\phi\left(\frac{\partial u_k^{v-e}}
{\partial x_3^{v-e}}\right)^2\right)e^{x_3^{v-e}}\;dx^{v-e}\right)\:.
\end{align}
If $\mu(\tilde{\Gamma}_{m,i}^{v-e})=\infty$ then for any $\epsilon>0$
\begin{align}\label{eq3.100b}
&\left|\int_{\tilde{\Gamma}_{m,i}^{v-e}}e^{x_3^{v-e}}\left(
\sum_{l=1}^2\left(\left(\frac{\partial u_m^{v-e}}{\partial\taux_l^{v-e}}
\right)_{A^{v-e}}\frac{\partial}{\partial s_l^{v-e}}
\left(\left(\frac{\partial u_m^{v-e}}{\partial\nuw^{v-e}}
\right)_{A^{v-e}}\right)\right.\right.\right. \notag \\
&\quad\quad\quad\quad\quad-\left.\left.\left.\left(\frac{\partial
u_p^{v-e}}{\partial\taux_l^{v-e}}\right)_{A^{v-e}}\frac
{\partial}{\partial s_l^{v-e}} \left(\left(\frac{\partial
u_p^{v-e}}{\partial\nuw^{v-e}}\right)_{A^{v-e}}
\right)\right)\right)\: d \sigma^{v-e}\right| \notag \\
&\quad\leq\epsilon\left(\sum_{k=m,p}\int_{\tilde{\Omega}
_k^{v-e}}(u_k^{v-e})^2e^{x_3^{v-e}}w^{v-e}(x_1^{v-e})\:dx^{v-e}\right)
\end{align}
\end{subequations}
provided $W=O(e^{{N}^\alpha})$ with $\alpha<1/2$.
\end{lem}

\section{Estimates for Terms on the Boundary}
\subsection{Estimates for terms on the boundary of $\Omega^r$}
To simplify the presentation we assume the face constituting part of the boundary of
$\Omega$ lies on the $x_2-x_3$ plane. The contributions from the boundary which have
to be estimated will then consist of terms from the regular region, the vertex region,
the vertex-edge region and the edge region as shown in Figure \ref{fig3.2}.
\begin{figure}[!ht]
\centering
\includegraphics[scale = 0.60]{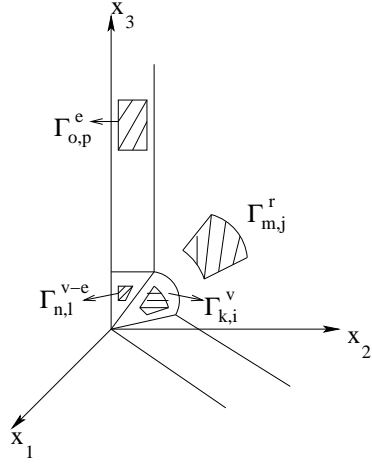}
\caption{Boundary terms.}
\label{fig3.2}
\end{figure}
We first examine how to estimate the terms on the boundary of
$\Omega$ for the regular region, terms from the other regions can
be estimated similarly.
\begin{lem}\label{lem3.5.1}
Let $\Gamma_{m,j}^r$ be part of the boundary of the element
$\Omega_m^r$ which lies on the $x_2-x_3$ axis. Define the
contributions from $\Gamma_{m,j}^r$ by
\begin{align}\label{eq3.101}
(BT)_{m,j}^r&=\rho_v^2\sin^2(\phi_v)\left(-\oint_{\partial
\Gamma_{m,j}^r}\left(\frac{\partial u}{\partial \nb} \right)_A
\left(\frac{\partial u}{\partial \nuw}\right)_A\: d s\right. \notag\\
&+\left. 2\int_{\Gamma_{m,j}^r}\sum_{j=1}^2\left(\frac{\partial
u}{\partial \taux_j}\right)_A\frac{\partial}{\partial s_j}\left(
\left(\frac{\partial u}{\partial \nuw}\right)_A\right)\:d\sigma\right)\:.
\end{align}
If Dirichlet boundary conditions are imposed on $\Gamma_{m,j}^r$
then for any $\epsilon>0$ there exist constants $C_\epsilon$ and
$K_\epsilon$ such that
\begin{align}\label{eq3.102}
|(BT)_{m,j}^r|&\leq C_{\epsilon}(\ln W)^2\|u_m^r\|^2_{3/2,
\Gamma_{m,j}^r}+K_{\epsilon}\sum_{|\alpha|=1}
\|D_x^\alpha u_m^r\|^2_{0,\Omega_m^r} \notag \\
&+\epsilon\sum_{|\alpha|=2}\|D_x^\alpha u_m^r\|^2_{0,\Omega_m^r}\:.
\end{align}
If Neumann boundary conditions are imposed on $\Gamma_{m,j}^r$
then for any $\epsilon>0$ there exists a constant $C_\epsilon$
such that
\begin{align}\label{eq3.103}
\left|(BT)_{m,j}^r\right|\leq C_{\epsilon}(\ln W)^2\left\|\left(
\frac{\partial u _m^r}{\partial \nuw}\right)_A\right\|^2_{1/2,\Gamma_{m,j}^r}
+\epsilon\sum_{1\leq|\alpha|\leq 2}\|D_x^{\alpha} u_m^r\|^2_{2,\Omega_m^r}\:.
\end{align}
\end{lem}
\begin{proof}
The proof is provided in Appendix D.5.
\end{proof}

\subsection{Estimates for terms on the boundary of $\Omega^e$}
\begin{lem}\label{lem3.5.2}
Let $\Gamma_{m,j}^e$ be part of the boundary of the element
$\Omega_m^e$ which lies on the $x_2-x_3$ axis. Define the
contributions from $\Gamma_{m,j}^e$ by
\begin{align}\label{eq3.104}
(BT)_{m,j}^e&=-\oint_{\partial\tilde{\Gamma}_{m,j}^e}
\left(\frac{\partial u}{\partial
\nb^e}\right)_{A^e}\left(\frac{\partial u}
{\partial \nuw^e}\right)_{A^e}\: d s^e \notag \\
&-2\int_{\tilde{\Gamma}_{m,j}^e}\sum_{l=1}^2\left(\frac
{\partial u}{\partial \taux_l^e}\right)_{A^e}\frac{\partial}
{\partial s_l^e}\left(\left(\frac{\partial u}{\partial
\nuw^e}\right)_{A^e}\right)\:d\sigma^e\:.
\end{align}
If Dirichlet boundary conditions are imposed on
$\Gamma_{m,j}^e$ and $\mu{(\tilde{\Gamma}_{m,j}^e)}<\infty$
then for any $\epsilon>0$ there exists constants\,
$C_\epsilon, K_\epsilon$ such that for $W$ large enough
\begin{align}\label{eq3.105}
|(BT)_{m,j}^e|&\leq C_{\epsilon}(\ln W)^2
\left(\|u_m^e\|_{0,\tilde{\Gamma}_{m,i}^e}^2+\Big|\Big|\Big|
\left(\frac{\partial u_m^e}{\partial x_1^e}\right)
\Big|\Big|\Big|^2_{\tilde{\Gamma}_{m,j}^e}+
\Big|\Big|\Big|G_{m,j}^e\left( \frac{\partial u_m^e}{\partial
x_3^e}\right)\Big|\Big|\Big|^2_{\tilde{\Gamma}
_{m,j}^e}\right) \notag \\
&+K_\epsilon\int_{\tilde{\Omega}_m^e}\left(\sum_{i=1}^2
\left(\frac{\partial u_m^e}{\partial x_i^e}\right)^2
+e^{2\tau}\left(\frac{\partial u_m^e}{\partial x_3^e}
\right)^2\right)\:dx^e \notag \\
&+\epsilon\int_{\tilde{\Omega}_m^e}\left(\sum_{i,j=1,2}
\left(\frac{\partial^2 u_m^e}{\partial x_i^e\partial x_j^e}
\right)^2+e^{2\tau}\sum_{i=1}^2\left(\frac{\partial^2 u_m^e}
{\partial x_i^e\partial x_3^e}\right)^2
+e^{4\tau}\left(\frac{\partial^2 u_m^e}{\left(\partial x_3^e
\right)^2}\right)^2\right)dx^e
\end{align}
If Neumann boundary conditions are imposed on $\Gamma_{m,j}^e$
and $\mu{(\tilde{\Gamma}_{m,j}^e)}<\infty$ then for any
$\epsilon>0$ there exists a constant $C_\epsilon$ such that
\begin{align}\label{eq3.106}
\left|(BT)_{m,j}^e\right|&\leq C_{\epsilon}(\ln
W)^2\Big|\Big|\Big|\left( \frac{\partial u}{\partial
\nuw^e}\right)_{A^e}\Big|\Big|\Big|^2_{\tilde{
\Gamma}_{m,j}^e}+\epsilon\left(\int_{\tilde{\Omega}_m^e}
\left(\sum_{i,j=1}^2 \left(\frac{\partial^2 u_m^e}{\partial x_i^e
\partial x_j^e}\right)^2\right.\right. \notag \\
&\left.\left.+ e^{2\tau}\sum_{i=1}^2\left(\frac {\partial^2
u_m^e}{\partial x_i^e \partial x_3^e}\right)^2
+e^{4\tau}\left(\frac{\partial^2 u_m^e}{(\partial
x_3^e)^2}\right)^2 + \sum_{i=1}^2\left(\frac{\partial u_m^e}
{\partial x_i^e}\right)^2\right.\right. \notag \\
&\left.\left.+ e^{2\tau}\left(\frac{\partial u_m^e} {\partial
x_3^e}\right)^2\right)\:dx^e\right)\:.
\end{align}
If $\mu{(\tilde{\Gamma}_{m,j}^e)}=\infty$ then for any $\epsilon>0$ for
$N,W$ large enough
\begin{align}\label{eq3.107}
\left|(BT)_{m,j}^e\right|&\leq \epsilon
\int_{\tilde{\Omega}_m^e}(u_m^e)^2w^e(x_1^e)\:dx^e
\end{align}
provided $W=O(e^{N^\alpha})$ for $\alpha<1/2$.
\end{lem}
\begin{proof}
The proof is provided in Appendix D.6.
\end{proof}

We now state estimates for terms on the boundary of vertex neighbourhoods
and vertex-edge neighbourhoods.

\subsection{Estimates for terms on the boundary of {$\Omega^v$}}
\begin{lem}\label{3.5.3}
Let $\Gamma_{m,j}^v$ be part of the boundary of the element
$\Omega_m^v$ which lies on the $x_2-x_3$ axis. Define the
contributions from $\Gamma_{m,j}^v$ by
\begin{align}\label{eq3.108}
(BT)_{m,j}^v&=\sin^2(\phi_v)\left(-\oint_{\partial\tilde
{\Gamma}_{m,j}^v}e^{x_3^v}\sin(x_1^v)\left(\frac{\partial
u}{\partial \nb^v}\right)_{A^v}\left(\frac{\partial u}
{\partial \nuw^v}\right)_{A^v}\: d s^v\right. \notag \\
&\left. -2\int_{\tilde{\Gamma}_{m,j}^v}
e^{x_3^v}\sin(x_1^v)\sum_{l=1}^2\left(\frac{\partial u}
{\partial\taux_l^v}\right)_{A^v}\frac{\partial}{\partial s_l^v}
\left(\left(\frac{\partial u}{\partial
\nuw^v}\right)_{A^v}\right)\:d \sigma^v \right).
\end{align}
If Dirichlet boundary conditions are imposed on $\Gamma_{m,j}^v$ and
$\mu{(\tilde{\Gamma}_{m,j}^v)}<\infty$ then for any $\epsilon>0$ there
exists constants\, $C_\epsilon$ and $K_\epsilon$ such that
\begin{align}\label{eq3.109}
|(BT)_{m,j}^v|&\leq C_{\epsilon}(\ln W)^2 R_{m,j}^v
\|u_m^v\|^2_{3/2,\tilde{\Gamma}_{m,j}^v}
+K_{\epsilon}\sum_{|\alpha|=1}\|e^{{x_3^v}/2}
D_{x^v}^\alpha u_m^v\|^2_{0,\tilde{\Omega}_m^e} \notag \\
&+\epsilon\sum_{|\alpha|=2} \|e^{{x_3^v}/2}D_{x^v}^\alpha
u_m^v\|^2_{0,\tilde{\Omega}_m^v}\:.
\end{align}
If Neumann boundary conditions are imposed on $\Gamma_{m,j}^v$ then
\begin{align}\label{eq3.110}
\left|(BT)_{m,j}^v\right|&\leq C_{\epsilon}(\ln W)^2R_{m,j}^v
\left\|\left(\frac{\partial u_m^v}{\partial\nuw^v}\right)_{A^v}
\right\|^2_{1/2,\tilde{\Gamma}_{m,j}^v} \notag \\
&+\epsilon\sum_{1\leq|\alpha|\leq2}\|e^{{x_3^v}/2}D_{x^v}^\alpha
u_m^v\|^2_{0,\tilde{\Omega}_m^v}\:.
\end{align}
If $\mu{(\tilde{\Gamma}_{m,j}^v)}=\infty$ then $(BT)_{m,j}^v=0$.
\end{lem}

\subsection{Estimates for terms on the boundary of {$\Omega^{v-e}$}}
\begin{lem}\label{lem3.5.4}
Let $\Gamma_{m,j}^{v-e}$ be part of the boundary of the element
$\Omega_m^{v-e}$ which lies on the $x_2-x_3$ axis. Define the
contributions from $\Gamma_{m,j}^{v-e}$ by
\begin{align}\label{eq3.111}
(BT)_{m,j}^{v-e}&=-\oint_{\partial\tilde{\Gamma}_{m,j}^{v-e}}
e^{x_3^{v-e}}\left(\frac{\partial u}{\partial
\nb^{v-e}}\right)_{A^{v-e}}\left(\frac{\partial u}
{\partial \nuw^{v-e}}\right)_{A^{v-e}}\: d s^{v-e} \notag \\
&-2\int_{\tilde{\Gamma}_{m,j}^{v-e}}e^{x_3^{v-e}}
\sum_{l=1}^2\left(\frac {\partial u}{\partial
\taux_l^{v-e}}\right)_{A^{v-e}}\frac{\partial}{\partial
s_l^{v-e}}\left( \left(\frac{\partial u}{\partial
\nuw^{v-e}}\right)_{A^{v-e}}\right)\:d\sigma^{v-e}\:.
\end{align}
If Dirichlet boundary conditions are imposed on
$\Gamma_{m,j}^{v-e}$ and $\mu{(\tilde{\Gamma}_{m,j}^{v-e})}
<\infty$ then for any $\epsilon>0$ there exists constants
\, $C_\epsilon$ and $K_\epsilon$ such that
\begin{align}\label{eq3.112}
|(BT)_{m,j}^{v-e}|&\leq C_{\epsilon}(\ln
W)^2\left(\|u_m^{v-e}\|_{0,\tilde{\Gamma}_{m,i}^{v-e}}^2
+\Big|\Big|\Big|\left(\frac{\partial u_m^{v-e}}
{\partial x_1^{v-e}}\right) \Big|\Big|\Big|^2_{\tilde{\Gamma}
_{m,j}^{v-e}}+\Big|\Big|\Big|E_{m,j}^{v-e} \left(\frac{\partial
u_m^{v-e}}{\partial x_3^{v-e}}\right)\Big|\Big|\Big|^2_{\tilde
{\Gamma}_{m,j}^{v-e}}\right) \notag \\
&+K_\epsilon\int_{\tilde{\Omega}_m^{v-e}}\left(\sum_{i=1}^2
\left(\frac{\partial u_m^{v-e}}{\partial x_i^{v-e}}\right)^2
+\sin^2\phi\left(\frac{\partial u_m^{v-e}}{\partial x_3^{v-e}}
\right)^2\right)e^{x_3^{v-e}}\:dx^{v-e} \notag \\
&+\epsilon\int_{\tilde{\Omega}_m^{v-e}}\left(\sum_{i,j=1,2}
\left(\frac{\partial^2 u_m^{v-e}}{\partial x_i^{v-e}\partial x_j^{v-e}}
\right)^2+\sin^2\phi\sum_{i=1}^2\left(\frac{\partial^2 u_m^{v-e}}
{\partial x_i^{v-e}\partial x_3^{v-e}}\right)^2\right. \notag \\
&\left.+\sin^4\phi\left(\frac{\partial^2 u_m^{v-e}}{\left(\partial
x_3^{v-e}\right)^2}\right)^2\right)e^{x_3^{v-e}}\:dx^{v-e}.
\end{align}
If Neumann boundary conditions are imposed on $\Gamma_{m,j}^{v-e}$
and $\mu(\tilde{\Gamma}_{m,j}^{v-e})<\infty$ then for any $\epsilon>0$
there exists a constant $C_\epsilon$ such that
\begin{align}\label{eq3.113}
\left|(BT)_{m,j}^{v-e}\right|&\leq C_{\epsilon}(\ln
W)^2\Big|\Big|\Big|\left( \frac{\partial u}{\partial
\nuw^{v-e}}\right)_{A^{v-e}}\Big|\Big|\Big|^2_{\tilde{
\Gamma}_{m,j}^{v-e}}+\epsilon\left(\int_{\tilde{\Omega}_m^{v-e}}
\left(\sum_{i,j=1}^2 \left(\frac{\partial^2 u_m^{v-e}}{\partial
x_i^{v-e}\partial x_j^{v-e}}\right)^2\right.\right. \notag \\
&\left.\left.+ \sin^2\phi\sum_{i=1}^2\left(\frac {\partial^2
u_m^{v-e}}{\partial x_i^{v-e} \partial x_3^{v-e}}\right)^2
+\sin^4\phi\left(\frac{\partial^2 u_m^{v-e}}{(\partial
x_3^{v-e})^2}\right)^2 + \sum_{i=1}^2\left(\frac{\partial
u_m^{v-e}}{\partial x_i^{v-e}}\right)^2\right.\right. \notag \\
&\left.\left.+ \sin^2\phi\left(\frac{\partial u_m^{v-e}}{\partial
x_3^{v-e}}\right)^2\right)e^{x_3^{v-e}}\:dx^{v-e}\right).
\end{align}
If $\mu{(\tilde{\Gamma}_{m,j}^{v-e})}=\infty$ then for any $\epsilon>0$
for $N,W$ large enough
\begin{align}
\left|(BT)_{m,j}^{v-e}\right|&\leq \epsilon
\int_{\tilde{\Omega}_m^{v-e}}(u_m^{v-e})^2e^{x_3^{v-e}}w^{v-e}(x_1^{v-e})
\:dx^{v-e} \notag
\end{align}
provided $W=O(e^{N^\alpha})$ for $\alpha<1/2$.
\end{lem}

\section{Proof of the Stability Theorem}
We are now in a position to prove the stability estimates stated in Theorems
\ref{thm2.3.1}, \ref{thm2.3.2} and \ref{thm2.3.3}. We recall them again.
\begin{thm2.3.1}
Consider the elliptic boundary value problem (\ref{eq2.1}). Suppose the boundary
conditions are Dirichlet. Then
\[\mathcal U^{N,W}(\{\mathcal F_u\})\leq C(\ln W)^2\mathcal V^{N,W}(\{\mathcal F_u\})\:.\]
\end{thm2.3.1}
\begin{thm2.3.2}
If the boundary conditions for the elliptic boundary value problem (\ref{eq2.1}) are
mixed then
\[\mathcal U^{N,W}(\{\mathcal F_u\})\leq CN^4 \mathcal V^{N,W}(\{\mathcal F_u\})\]
provided $W=O(e^{N^\alpha})$ for $\alpha<1/2$.
\end{thm2.3.2}
\begin{thm2.3.3}
If the boundary conditions are mixed and the spectral element functions $(\{\mathcal F_u\})$
are conforming on the wirebasket $W\!B$ and vanish on $W\!B$ then
\[\mathcal U^{N,W}(\{\mathcal F_u\}) \leq C (\ln W)^2\mathcal V^{N,W}(\{\mathcal F_u\})\]
provided $W=O(e^{N^\alpha})$ for $\alpha<1/2$.
\end{thm2.3.3}
\begin{proof}
Define
\begin{align*}
\mathcal U^{N,W}_{(2)}(\{\mathcal F_u\})&=\sum_{l=1}^{N_r}
\int_{\Omega_l^r}\sum_{i,j=1}^3\left(\frac{\partial^2u_l^r}
{\partial x_i\partial x_j}\right)^2 \:dx \\
&+\sum_{v\in \mathcal V}\left(\sum_{l=1,\mu{(\tilde{\Omega}_l^v)
<\infty}}^{N_v}\int_{\tilde{\Omega}_l^v}e^{x_3^v}\sum_{i,j=1}^3
\left(\frac{\partial^2u_l^v} {\partial x_i^v\partial x_j^v}
\right)^{2}\:dx^v\right)\\
&+\sum_{{v-e}\in \mathcal V-E}\left(\sum_{l=1,\mu{(\tilde
{\Omega}_l^{v-e})<\infty}}^{N_{v-e}}\int_{\tilde{\Omega}_l^{v-e}}
e^{x_3^{v-e}}\left(\sum_{i,j=1,2}
\left(\frac{\partial^2 u_l^{v-e}}{\partial x_i^{v-e}
\partial x_j^{v-e}}\right)^2\right.\right.\\
&\left.\left.+\sin^2\phi\sum_{i=1}^2\left(\frac{\partial^2
u_l^{v-e}}{\partial x_i^{v-e}\partial x_3^{v-e}}\right)^2
+\sin^4\phi\left(\frac{\partial^2 u_l^{v-e}}
{(\partial x_3^{v-e})^2}\right)^2\right)\:dx^{v-e}\right)\\
&+\sum_{e\in \mathcal E}\left(\int_{\tilde{\Omega}_l^e}\left(
\sum_{i,j=1,2}\left(\frac{\partial^2 u_l^e}{\partial x_i^e \partial
x_j^e}\right)^2+e^{2\tau}\sum_{i=1}^2\left(\frac{\partial^2 u_l^e}
{\partial x_i^e\partial x_3^e}\right)^2\right.\right.\\
&\left.\left.+e^{4\tau}\left(\frac{\partial^2 u_l^e}
{\left(\partial x_3^e\right)^2}\right)^2\right)dx^e\right)\:.
\end{align*}
Then combining the results in Sections $3.1, 3.2, 3.3$ and $3.4$ we
obtain that for any $\epsilon>0$ there exist constants $C_\epsilon$
and $K_\epsilon$ such that
\begin{align}
\mathcal U^{N,W}_{(2)}(\{\mathcal F_u\})&\leq C_\epsilon(\ln W)^2
\mathcal V^{N,W}(\{\mathcal F_u\})+\epsilon\:\mathcal U^{N,W}_{(2)}
(\{\mathcal F_u\}) \notag \\
&+K_\epsilon\:\mathcal U^{N,W}_{(1)}(\{\mathcal F_u\})\:.
\end{align}
Here $\mathcal U^{N,W}_{(1)}(\{\mathcal F_u\})$ is as defined in
$(\ref{eq3.90})$. Hence choosing $\epsilon$ small enough we obtain
\begin{equation}\label{eq3.115}
\mathcal U^{N,W}_{(2)}(\{\mathcal F_u\})\leq 2\left(C_\epsilon(\ln W)^2
\mathcal V^{N,W}(\{\mathcal F_u\})+K_\epsilon\:\mathcal U^{N,W}_{(1)}
(\{\mathcal F_u\})\right)\:.
\end{equation}
At the same time using Theorem \ref{thm3.3.1} we have
\begin{equation}\label{eq3.116}
\mathcal U^{N,W}_{(1)}(\{\mathcal F_u\})\leq K_{N,W}\mathcal V^{N,W}
(\{\mathcal F_u\})\:.
\end{equation}
Now
\begin{equation}\label{eq3.117}
\mathcal U^{N,W}(\{\mathcal F_u\})=\mathcal U^{N,W}_{(1)}(\{\mathcal F_u\})
+\mathcal U^{N,W}_{(2)}(\{\mathcal F_u\})\:.
\end{equation}
Combining $(\ref{eq3.115}), (\ref{eq3.116})$ and $(\ref{eq3.117})$ the
result follows.
\end{proof}

\chapcleardoublepage

\chapter{The Numerical Scheme and Error Estimates}
\section{Introduction}
In Chapters $2$ and $3$ we have described a method for solving three dimensional
elliptic boundary value problems on non-smooth domains using $h-p$ spectral
element methods.

In this chapter we provide the numerical scheme based on the stability estimates
of Chapters $2$ and $3$. We shall define a functional which is closely related
to the quadratic forms defined in Chapter $2$ in the regular region and in the
various neighbourhoods of vertices, edges and vertex-edges. We seek a solution
which minimizes this functional.

The functional which we minimize is the sum of a weighted squared norm of the
residuals in the partial differential equations and the squared norm of the
residuals in the boundary conditions in fractional Sobolev spaces and the sum
of the squares of the jumps in the function and its derivatives across inter-element
boundaries in fractional Sobolev norms. The Sobolev spaces in vertex-edge and
edge neighbourhoods are anisotropic and become singular at the corners and edges.

In section $4.3$ we obtain error estimates and show that the error is exponentially
small in terms of the number of degrees of freedom and number of layers in the
geometric mesh.


\section{The Numerical Scheme}
In Chapter $2$, we had divided the polyhedral domain $\Omega$ into a regular
region $\Omega^r$, a set of vertex neighborhoods $\Omega^v$, a set of edge
neighborhoods $\Omega^e$ and a set of vertex-edge neighborhoods $\Omega^{v-e}$.
We had further divided each of these sub domains into still smaller elements as
curvilinear hexahedrons, tetrahedrons and prisms and by virtue of the fact that
a tetrahedron can be split into four hexahedrons and a prism can be split into
three hexahedrons we can assume that all our elements are hexahedrons to keep
the exposition simple (and to keep the programming simple).

In Chapter $2$, we had introduced a spectral element representation of the
function $u$ on each of these elements in various parts of the domain $\Omega$.
To formulate the numerical scheme we will define a functional
$\mathcal R^{N,W}(\{\mathcal F_u\})$, closely related to the quadratic form
$\mathcal V^{N,W}(\{\mathcal F_u\})$ as follows:
\begin{align}\label{eq4.1}
\mathcal R^{N,W}\left(\{\mathcal F_u\}\right)
&=\mathcal R_{regular}^{N,W}\left(\{\mathcal F_u\}\right)
+\mathcal R_{vertices}^{N,W}\left(\{\mathcal F_u\}\right)
+\mathcal R_{vertex-edges}^{N,W}\left(\{\mathcal F_u\}\right)\notag\\
&+\mathcal R_{edges}^{N,W}(\{\mathcal F_u\}).
\end{align}

Let us first consider the regular region $\Omega^r$ of $\Omega$ and define the
functional $\mathcal R_{regular}^{N,W}(\{\mathcal F_u\})$. $\Omega^r$ is
divided into $N_r$ curvilinear hexahedrons $\Omega^r_l, l=1,2,\ldots,N_r$ and we
map each of these $\Omega^r_l$ onto the master cube
\[Q=(-1,1)^3=\{\lambda=(\lambda_1,\lambda_2,\lambda_3)|
-1\leq\lambda_i\leq 1,1\leq i\leq 3\}\]
using an analytic map $M^r_l$ having an analytic inverse.

We now define a non-conforming spectral element representation on each of these
elements as follows:
\[u_l^r(\lambda) = \sum_{i=0}^W \sum_{j=0}^W \sum_{k=0}^W
\:\alpha_{i,j,k}\:\lambda_1^i\lambda_2^j\lambda_3^k.\]

Let $\lambda=(\lambda_1,\lambda_2,\lambda_3)$ and let
$f^r_l(\lambda)=f(M^r_l(\lambda_1,\lambda_2,\lambda_3))$
where $\lambda\in Q$ for $l=1,2,\ldots,N_r$ and let $J^r_l(\lambda)$
denote the Jacobian of the mapping $M^r_l$. Define
\[F_{l}^{r}(\lambda)=f^r_l(\lambda)\sqrt{ J^r_l(\lambda)}\;,\]
and
\[L_{l}^{r}u_{l}^{r}(\lambda)=Lu_{l}^{r}(M_{l}^{r}(\lambda))\sqrt{J^r_l(\lambda)}\;.\]

Now consider the boundary conditions $w=g_k$ on $\Gamma_k$ for
$k\in \mathcal D=\Gamma^{[0]}$ and $\left(\frac{\partial w}{\partial \nuw}\right)_A=h_k$
on $\Gamma_k$ for $k\in \mathcal N=\Gamma^{[1]}$. Let
$\Gamma^r_{i,k}=\Gamma_k\cap\partial \Omega^r_i$ be the image of the mapping $M^r_i$
corresponding to $\lambda_1=-1$. Let $g^r_{i,k}=g_k(M^r_i(-1,\lambda_2,\lambda_3))$ and
$h^r_{i,k}=h_k(M^r_i(-1,\lambda_2,\lambda_3))$ where $-1\leq\lambda_2,\lambda_3\leq 1$.

We now define
\begin{align}\label{eq4.2}
\mathcal R_{regular}^{N,W}(\{\mathcal F_u\})&=\sum_{l=1}^{N_r}\int_{Q=(M_l^r)^{-1}
(\Omega_l^r)}\left|\:L_{l}^{r}u_l^r(\lambda)-F^r_l(\lambda)\:\right|^2\:d\lambda\notag\\
&+\sum_{\Gamma_{l,i}^r\subseteq\bar{\Omega}^r\setminus\partial\Omega}
\left(\|[u]\|_{0,\Gamma_{l,i}^r}^2 +\sum_{k=1}^3\left \|[u_{x_k}]
\right\|_{1/2,{\Gamma_{l,i}^r}}^2\right) \notag \\
&+\mathop{\sum_{{\Gamma_{l,i}^r}\subseteq{\Gamma^{[0]}}}}
\left\|u_l^r-g^r_{l,i}\right\|_{3/2,{\Gamma_{l,i}^r}}^2 \notag \\
&+\sum_{{\Gamma_{l,i}^r}\subseteq{\Gamma^{[1]}}}\left\|\left(\frac{\partial u_l^r}
{\partial \nuw}\right)_A-h^r_{l,i}\right\|_{1/2,{\Gamma_{l,i}^r}}^2.
\end{align}

Let $v$ be one of the vertices of $\Omega$. Consider the vertex neighborhood $\Omega^v$
of the vertex $v \in \mathcal V$, the set of vertices (defined in Chapter $2$).
$\Omega^v$ is divided into $N_v$ curvilinear hexahedrons $\Omega^v_l, l=1,2,\ldots,N_v$.
We define a non-conforming spectral element representation as follows:
Let $\tilde{\Omega}^v_l$ be the image of ${\Omega}^v_l$ in $x^v$ coordinates. Then there
is an analytic map $M^v_l:Q\rightarrow\tilde{\Omega}^v_l$ having an analytic inverse.
If $\tilde{{\Omega}}^v_l$ is a corner element of the form
\[\tilde{{\Omega}}^v_l=\{x^v:(\phi,\theta)\in S_j^v,-\infty<\chi<\ln(\rho_1^v)\}\]
then we define $u_l^v=h_v$, where $h_v$ is a constant.

If ${\tilde{\Omega}}^v_l$ is of the form
\[\tilde{{\Omega}}^v_l=\{x^v:(\phi,\theta)\in S_j^v, \ln(\rho_i^v)<\chi<\ln(\rho_{i+1}^v)\}\]
we define
\[u_l^v(x^v)=\sum_{t=0}^{W_l}\sum_{s=0}^{W_l}\sum_{r=0}^{W_l}\;\beta_{r,s,t}\:\phi^r\theta^s\chi^t\;.\]
Here $1\leq W_l\leq W$. Moreover as in~\cite{G1}, $W_l=[\mu_{1}i]$ for all $1\leq i\leq N$,
where $\mu_{1}>0$ is a degree factor. Hereafter $[a]$ denotes the greatest positive integer
$\leq a$.

Let $A_v=(x_1^v,x_2^v,x_3^v)$ denote one of the vertices of $\Omega$. Let $F_l^v(x^v)=e^{5/2\chi}
\sqrt{\sin\phi} f(x(x^v))$ for $x_l^v\in\tilde{\Omega}_l^v, 1\leq l\leq N_v$.

We now consider the boundary conditions $w=g_i$ on $\Gamma_i$ for $i\in \mathcal D=\Gamma^{[0]}$
and $\left(\frac{\partial w}{\partial \nuw}\right)_{A}=h_i$ on $\Gamma_i$ for
$i\in\mathcal N =\Gamma^{[1]}$. Let $\Gamma^v_{l,i}=\Gamma_i\cap\partial\Omega^v_l$ and suppose
$\Omega^v_l$ is not a corner element. Moreover it is assumed that $\Gamma^v_{l,i}$ lies on the
$x_2-x_3$ plane for simplicity. Define
\begin{align*}
g^v_{l,i}(x^v)&=w=g_i(x(x^v)) \;\mbox{for}\; \Gamma^v_{l,i}\subseteq\Gamma^{[0]},\\
h^v_{l,i}(x^v)&=\left(\frac{\partial w}{\partial \nuw^v}\right)_{A^{v}}
=\frac{e^\chi}{\sin\phi} h_i(x(x^v)) \;\mbox{for}\; \Gamma^v_{l,i}\subseteq\Gamma^{[1]}.
\end{align*}
Let
$$R_{l,i}^v=\underset {x^v\in\tilde{\Gamma}_{l,i}^v}{sup}(e^{x_3^v}).$$
We now define the functional
\begin{align}\label{eq4.3}
\mathcal R_v^{N,W}\left(\{\mathcal F_u\}\right) &=\sum_{l=1,\mu(\tilde{\Omega}_l^v)<\infty}^{N_v}
\int_ {\tilde{\Omega}_l^v}\left|\:L^v u_l^v(x^v)-F_l^v(x^v)\:\right|^2\;dx^v \notag \\
&+ \mathop{\sum_{\Gamma_{l,i}^v\subseteq\bar{\Omega}^v\setminus\partial\Omega}}_{\mu(\tilde
{\Gamma}_{l,i}^v)<\infty}\left(\left\|\:\sqrt{R_{l,i}^v}[u]\:\right\|_{0,\tilde{\Gamma}_{l,i}^v}^2
+\sum_{k=1}^3\left\|\:\sqrt{R_{l,i}^v}[u_{x_k^v}]\:\right\|_{1/2,{\tilde{\Gamma}_{l,i}^v}}^2\right)\notag\\
&+\mathop{\sum_{{\Gamma_{l,i}^v}\subseteq{\Gamma^{[0]}},}}_{\mu(\tilde{\Gamma}_{l,i}^v)<\infty}
\left\|\:\sqrt{R_{l,i}^v}\left( u_l^v-g^v_{l,i}\right)\:\right\|_{3/2,{\tilde{\Gamma}_{l,i}^v}}^2\notag\\
&+\mathop{\sum_{{\Gamma_{l,i}^v}\subseteq{\Gamma^{[1]}},}}_{\mu(\tilde{\Gamma}_{l,i}^v)<\infty}
\left\|\:\sqrt{R_{l,i}^v}\left(\left(\frac{\partial u_l^v}{\partial\nuw^v}\right)_{A^v}-h^v_{l,i}\right)\:
\right\|_{1/2,{\tilde{\Gamma}_{l,i}^v}}^2.
\end{align}
The functional $\mathcal R_{vertices}^{N,W}(\{\mathcal F_u\})$ is then given by
\begin{equation}\label{eq4.4}
\mathcal R_{vertices}^{N,W}(\{\mathcal F_u\})=\sum_{v\in\mathcal V}
\mathcal R_v^{N,W}(\{\mathcal F_u\})\;.
\end{equation}

Next, we define $\mathcal R_{vertex-edges}^{N,W}(\{\mathcal F_u\})$. Let $\Omega^{v-e}$ denote
the vertex-edge neighborhood of the vertex-edge $v-e\in\mathcal{V-E}$. $\Omega^{v-e}$ is divided
into $N_{v-e}$ elements $\Omega^{v-e}_l, l=1,2,\ldots,N_{v-e}$. As in Chapter $2$, let us define
a non-conforming spectral element representation on each of these sub-domains as follows:

Let $\tilde{\Omega}^{v-e}_l$ be the image of ${\Omega}^{v-e}_l$ in $x^{v-e}$ coordinates.
If $\tilde{\Omega}^{v-e}_n$ is a corner element of the form
\[\tilde{\Omega}^{v-e}_n=\{x^{v-e}:\;\psi_i^{v-e}<\psi<\psi_{i+1}^{v-e},\;\theta_j^{v-e}<\theta
<\theta_{j+1}^{v-e},\;-\infty<\zeta<\zeta_{1}^{v-e}\}\]
then on $\tilde{\Omega}^{v-e}_n$ we define
\[u_n^{v-e} = h_{v-e} = h_v\;.\]

Next, suppose $\tilde{\Omega}^{v-e}_p$ is a corner element of the form
\[\tilde{\Omega}^{v-e}_p = \left\{x^{v-e}: -\infty<\psi<\psi_{1}^{v-e},\;\theta_j^{v-e}<
\theta<\theta_{j+1}^{v-e},\;\zeta_n^{v-e}<\zeta<\zeta_{n+1}^{v-e}\right\}\]
with $k \geq 1$. Then on $\tilde{\Omega}^{v-e}_p$ we define
$$u_p^{v-e}(x^{v-e}) = \sum_{l=0}^{W_p}\:\beta_l\:\zeta^l.$$
Here $1\leq W_p \leq W$. Moreover $W_p=[\mu_{2}n]$ for $1\leq n\leq N$, where $\mu_{2}>0$
is a degree factor.
\newline
Finally, suppose $\tilde{\Omega}^{v-e}_q$ is of the form
\[\tilde{\Omega}^{v-e}_q = \left\{x^{v-e}:\;\psi_i^{v-e}<\psi<\psi_{i+1}^{v-e},\;\theta_j^{v-e}
<\theta<\theta_{j+1}^{v-e},\;\zeta_n^{v-e}<\zeta<\zeta_{n+1}^{v-e}\right\}\]
with $i \geq 1$, $n \geq 1$. Then on $\tilde{\Omega}^{v-e}_q$ we define
\[u_q^{v-e}(x^{v-e})=\sum_{r=0}^{W_q}\sum_{s=0}^{W_q}\sum_{t=0}^{V_q}\;\gamma_{r,s,t}
\;\psi^r \theta^s \zeta^t.\]
Here $1\leq W_q \leq W$ and $1\leq V_q \leq W$. Moreover $W_q=[\mu_{1}i], V_q=[\mu_{2}n]$ for
$1\leq i,n\leq N$, where $\mu_{1},\mu_{2}>0$ are degree factors~\cite{G1}.

Let $F_l^{v-e}(x^{v-e})=e^{2x_1^{v-e}}e^{\frac{5}{2}x_3^{v-e}}f(x(x^{v-e}))$ for $x^{v-e}\in
\tilde{\Omega}_l^{v-e}, 1\leq l\leq N_{v-e}$.

We now consider the boundary conditions $w=g_k$ on $\Gamma_k$ for $k\in \mathcal D=\Gamma^{[0]}$
and $\left(\frac{\partial w}{\partial\nuw}\right)_{A}=h_k$ on $\Gamma_k$ for $k\in\mathcal N=\Gamma^{[1]}$.
Then $\left(\frac{\partial w}{\partial\nuw^{v-e}}\right)_{A^{{v-e}}}=e^{x_3^{v-e}}e^{x_1^{v-e}}h_k(x(x^{v-e}))$.
Let $\Gamma^{v-e}_{l,k}=\Gamma_k\cap\partial \Omega^{v-e}_l$ and suppose $\Omega^{v-e}_{l}$ is not
a corner element. Moreover it is assumed that $\Gamma^{v-e}_{l,k}$ lies on the $x_2-x_3$ plane for
simplicity and $\mu(\tilde{\Gamma}^{v-e}_{l,k})<\infty$. Define
\begin{align*}
g^{v-e}_{l,k}(x^{v-e})&=w=g_k(x(x^{v-e})) \;\mbox{for}\;\Gamma^{v-e}_{l,k}\subseteq\Gamma^{[0]},\\
h^{v-e}_{l,k}(x^{v-e})&=\left(\frac{\partial w}{\partial \nuw^{v-e}}\right)_{A^{v-e}}
=e^{x_3^{v-e}}e^{x_1^{v-e}} h_k(x(x^{v-e})) \;\mbox{for}\;\Gamma^{v-e}_{l,k}\subseteq\Gamma^{[1]}.
\end{align*}

We define the functional
\begin{align}\label{eq4.5}
\mathcal R_{v-e}^{N,W}(\{\mathcal F_u\}) &= \sum_{l=1,\mu(\tilde{\Omega}_l^{v-e})<\infty}^{N_{v-e}}
\int_{\tilde{\Omega}_l^{v-e}}\left|\:L^{v-e} u_l^{v-e}(x^{v-e})-F_l^{v-e}(x^{v-e})\:\right|^2\:dx^{v-e}\notag\\
&+ \mathop{\sum_{\Gamma_{l,k}^{v-e}\subseteq\bar{\Omega}^{v-e}
\setminus\partial{\Omega},}}_{\mu(\tilde{\Gamma}_{l,k}^{v-e})<\infty}\left(\left\|\:\sqrt{F_{l,k}^{v-e}
G_{l,k}^{v-e}}\;[u]\:\right\|_{0,\tilde{\Gamma}_{l,k}^{v-e}}^2
+\big|\big|\big|\;[u_{x_1^{v-e}}]\;\big|\big|\big|_{\tilde{\Gamma}_{l,k}^{v-e}}^2\right. \notag \\
&+\big|\big|\big|\:[u_{x_2^{v-e}}]\:\big|\big|\big|_{\tilde{\Gamma}_{l,k}^{v-e}}^2
+\big|\big|\big|\:E_{l,k}^{v-e}\,[u_{x_3^{v-e}}]\:\big|\big|\big|_{\tilde{\Gamma}_{l,k}^{v-e}}^2\Bigg)\notag\\
&+\mathop{\sum_{{\Gamma_{l,k}^{v-e}}\subseteq{\Gamma^{[0]}},}}_{\mu(\tilde{\Gamma}_{l,k}^{v-e})<\infty}
\left(\left\|\,\sqrt{F_{l,k}^{v-e}}\;\left(u_l^{v-e}-g_{l,k}^{v-e}\right)\,
\right\|_{0,{\tilde{\Gamma}_{l,k}^{v-e}}}^2\right. \notag \\
&+\big|\big|\big|\:\left(u_{x_1^{v-e}}-{(g_{l,k}^{v-e})}_{x_1^{v-e}}\right)
\:\big|\big|\big|^2_{\tilde{\Gamma}_{l,k}^{v-e}} \notag \\
&+\big|\big|\big|\:E_{l,k}^{v-e}\,\left(u_{x_3^{v-e}}-{(g_{l,k}^{v-e})}_{x_3^{v-e}}
\right)\:\big|\big|\big|^2_{\tilde{\Gamma}_{l,k}^{v-e}}\Bigg) \notag \\
&+\mathop{\sum_{{\Gamma_{l,k}^{v-e}}\subseteq
{\Gamma^{[1]}},}}_{\mu(\tilde{\Gamma}_{l,k}^{v-e})<\infty}
\Big|\Big|\Big|\:\left(\frac{\partial u}{\partial\nuw^{v-e}}
\right)_{A^{v-e}}-h_{l,k}^{v-e}\:\Big|\Big|\Big|_{\tilde{\Gamma}_{l,k}^{v-e}}^2.
\end{align}
Here $E_{l,k}^{v-e}=\underset{x^{v-e}\in{\tilde{\Gamma}_{l,k}^{v-e}}}{sup}(\sin\phi)$ and
$G_{l,k}^{v-e}=\underset{x^{v-e}\in{\tilde{\Gamma}_{l,k}^{v-e}}}{sup}(e^{x_3^{v-e}})$. Moreover,
$G_{l,k}^{v-e}$ is defined in (\ref{eq2.27}). 
\newline
Then the functional $\mathcal R_{vertex-edges}^{N,W}(\{\mathcal F_u\})$ is defined as follows:
\begin{equation}\label{eq4.6}
\mathcal R_{vertex-edges}^{N,W}(\{\mathcal F_u\})=\sum_{v-e\in
\mathcal V-E}\mathcal R_{v-e}^{N,W}(\{\mathcal F_u\})\;.
\end{equation}

Finally, we define the functional $\mathcal R_{edges}^{N,W}(\{\mathcal F_u\})$. Let
$\Omega^{e}$ denote the edge neighborhood of the edge $e\in\mathcal{E}$. $\Omega^{e}$
is divided into $N_{e}$ elements $\Omega^{e}_l, l=1,2,\ldots,N_{e}$. Once again we define
a non-conforming spectral element representation on each of these sub-domains as follows:

Let $\tilde{\Omega}^{e}_{m}$ be the image of ${\Omega}^{e}_{m}$ in $x^{e}$ coordinates.
If $\tilde{\Omega}^{e}_p$ is a corner element of the form
\[\tilde{\Omega}^e_p = \left\{x^e:\: -\infty<x_1^e<\ln(r_{1}^e),
\:\theta_j^e<x_2^e<\theta_{j+1}^e,\:Z_n^e<x_3^e<Z_{n+1}^e\right\}\]
then we define
\[u_p^e(x^e) = \sum_{t=0}^W\:\alpha_t(x_3^e)^t\:.\]

Next, let $\tilde{\Omega}^{e}_q$ be of the form
\[\tilde{\Omega}^e_q = \left\{x^e:\: \ln (r_i^e)<x_1^e<\ln (r_{i+1}^e),
\:\theta_j^e<x_2^e<\theta_{j+1}^e,\:Z_n^e<x_3^e<Z_{n+1}^e\right\}\]
with $1\leq i\leq N,\,0\leq j\leq I_e,\,0\leq n\leq J_e\,.$\\
Then we define
\[u_q^e(x^e) = \sum_{r=0}^{W_q} \sum_{s=0}^{W_q}\sum_{t=0}^W \:\alpha_{r,s,t}
\:(x_1^e)^r(x_2^e)^s(x_3^e)^t\;.\]
Here $1\leq W_q\leq W$. Moreover $W_q=[\mu_{1}i]$ for $1\leq i\leq N$, $\mu_{1}>0$ is
a degree factor.

Let $F_l^{e}(x^{e})=e^{2x_1^{e}}f(x(x^{e}))$ for $x^e\in\tilde{\Omega}_l^e$, $1\leq l\leq N_{e}$.

We now consider the boundary conditions $w=g_k$ on $\Gamma_k$ for $k\in \mathcal D=\Gamma^{[0]}$ and
$\left(\frac{\partial w}{\partial\nuw}\right)_{A}=h_k$ on $\Gamma_k$ for $k\in \mathcal N=\Gamma^{[1]}$.
Then $\left(\frac{\partial w}{\partial \nuw^{e}}\right)_{A^{{e}}}=e^{x_1^e} h_k(x(x^{e}))$. Let
$\Gamma^e_{m,k}=\Gamma_k\cap\partial \Omega^{e}_m$ and suppose $\Omega_m^e$ is not a corner element.
Moreover it is assumed that $\Gamma^{e}_{m,k}$ lies on the $x_2-x_3$ plane for simplicity and
$\mu(\tilde{\Gamma}^{e}_{m,k})<\infty$. Define
\begin{align*}
g^{e}_{m,k}(x^{e})&=w=g_k(x(x^{e})) \;\mbox{for}\; \Gamma^{e}_{m,k}\subseteq\Gamma^{[0]},\\
h^{e}_{m,k}(x^{e})&=\left(\frac{\partial w}{\partial \nuw^{e}}\right)_{A^{e}}=e^{x_1^e} h_k(x(x^{e}))
\;\mbox{for}\; \Gamma^{e}_{m,k}\subseteq\Gamma^{[1]}\:.
\end{align*}

We define
\begin{align}\label{eq4.7}
\mathcal R_{e}^{N,W}(\{\mathcal F_u\}) &= \sum_{l=1,\mu(\tilde{\Omega}_l^{e})<\infty}^{N_{e}}
\int_{\tilde{\Omega}_l^{e}}\left|\:L^{e} u_l^{e}(x^{e})-F_l^{e}(x^{e})\:\right|^2\:dx^{e}\notag\\
&+ \mathop{\sum_{\Gamma_{m,k}^{e}\subseteq\bar{\Omega}^{e}
\setminus\partial{\Omega},}}_{\mu(\tilde{\Gamma}_{m,k}^{e})<\infty}\left(\left\|
\:\sqrt{H_{m,k}^{e}}\;[u]\:\right\|_{0,\tilde{\Gamma}_{m,k}^{e}}^2
+\big|\big|\big|\;[u_{x_1^{e}}]\;\big|\big|\big|_{\tilde{\Gamma}_{m,k}^{e}}^2\right. \notag \\
&+\big|\big|\big|\:[u_{x_2^{e}}]\:\big|\big|\big|_{\tilde{\Gamma}_{m,k}^{e}}^2+
\big|\big|\big|\:G_{m,k}^{e}\,[u_{x_3^{e}}]\:\big|\big|\big|_{\tilde{\Gamma}_{m,k}^{e}}^2\Bigg)\notag\\
&+\mathop{\sum_{{\Gamma_{m,k}^{e}}\subseteq{\Gamma^{[0]}},}}_{\mu(\tilde{\Gamma}_{m,k}^{e})<\infty}
\left(\left\|\;\left(u_m^{e}-g_{m,k}^{e}\right)\,\right\|_{0,{\tilde{\Gamma}_{m,k}^{e}}}^2
+\big|\big|\big|\:\left(u_{x_1^{e}}-{(g_{m,k}^{e})}_{x_1^{e}}\right)
\:\big|\big|\big|^2_{\tilde{\Gamma}_{m,k}^{e}} \right.\notag\\
&+\big|\big|\big|\:G_{m,k}^{e}\,\left(u_{x_3^{e}}-{(g_{m,k}^{e})}_{x_3^{e}}
\right)\:\big|\big|\big|^2_{\tilde{\Gamma}_{m,k}^{e}}\Bigg) \notag\\
&+\mathop{\sum_{{\Gamma_{m,k}^{e}}\subseteq{\Gamma^{[1]}},}}_{\mu(\tilde{\Gamma}_{m,k}^{e})<\infty}
\Big|\Big|\Big|\:\left(\frac{\partial u}{\partial\nuw^{e}}\right)_{A^{e}}-h_{m,k}^{e}
\:\Big|\Big|\Big|_{\tilde{\Gamma}_{m,k}^{e}}^2.
\end{align}
Here $G_{m,k}^{e}=\underset{x^e\in{\tilde{\Gamma}_{m,k}^e}}{sup}\left(e^\tau\right)$ and $H_{m,k}^{e}$
is as defined in (\ref{eq2.34}).
\newline
The functional $\mathcal R_{edges}^{N,W}(\{\mathcal F_u\})$ can now
be defined as:
\begin{equation}\label{eq4.8}
\mathcal R_{edges}^{N,W}(\{\mathcal F_u\})=\sum_{e\in\mathcal E}
\mathcal R_{e}^{N,W}(\{\mathcal F_u\})\;.
\end{equation}
Finally using (\ref{eq4.2}), (\ref{eq4.4}), (\ref{eq4.6}) and (\ref{eq4.8})
in (\ref{eq4.1}) we can define $\mathcal R^{N,W}(\{\mathcal F_u\})$.

Our numerical scheme may now be formulated as follows:

\textit{Find $\mathcal F_s\in {\mathcal S}^{N,W}$ which minimizes the functional
$\mathcal R^{N,W}(\{\mathcal F_u\})$ over all $\mathcal F_u\in{\mathcal S}^{N,W}$.
Here ${\mathcal S}^{N,W}$ denotes the space of spectral element functions
$\mathcal F_u$.}

\section{Error Estimates}
It is well known that for three dimensional elliptic problems containing singularities in
the form of vertices and edges, the geometric mesh and a proper choice of element degree
distribution leads to exponential convergence and efficiency of computations
(see~\cite{BG4,BG5,G1,G2} and references therein).

Let $N$ denote the number of refinements in the geometrical mesh and $W$ denote an upper
bound on the degree of the polynomial representation of the spectral element functions. We
assume that $N$ is proportional to $W$.

In this Section we show that the error obtained from the proposed method is exponentially
small in $N$. The optimal rate of convergence with respect to $N_{dof}$, the number of degrees
of freedom is also provided.

Our analysis of error estimates is similar to that in two dimensions (see~\cite{BG6,BG9,KDU,SKT2}
and references therein). Here, we briefly describe the main steps of the proof, for details one
may refer to~\cite{KDU}.

Let ${\mathcal S}^{N,W}(\Omega^v)$, ${\mathcal S}^{N,W}(\Omega^{v-e})$, ${\mathcal S}^{N,W}
(\Omega^{e})$, and ${\mathcal S}^{N,W}(\Omega^r)$ denote the space of spectral element functions
(SEF) over the set of vertex neighbourhoods $\Omega^v$, vertex-edge neighbourhoods $\Omega^{v-e}$,
edge neighbourhoods $\Omega^{e}$ and regular region $\Omega^{r}$ respectively. Let us denote by
$S^{N,W}(\Omega)$ the space of SEF over the whole domain $\Omega$ whose restrictions to $\Omega^v$,
$\Omega^{v-e}$, $\Omega^{e}$ and $\Omega^{r}$ belong to ${\mathcal S}^{N,W}(\Omega^v)$,
${\mathcal S}^{N,W}(\Omega^{v-e})$, ${\mathcal S}^{N,W}(\Omega^{e})$, and ${\mathcal S}^{N,W}(\Omega^r)$
respectively.



Let $\{\mathcal F_{z}\}$ minimize $\mathcal R^{N,W}(\{\mathcal F_u\})$ over all
$\{\mathcal F_{u}\}\in {\mathcal S}^{N,W}(\Omega)$, the space of spectral element functions.
We write one more representation for $\{\mathcal F_{z}\}$ as follows:-
\begin{align*}
\{\mathcal F_{z}\}=\left\{\left\{z_l^{r}(\lambda_{1},\lambda_{2},\lambda_{3})\right\}_{l=1}^{N_{r}},
\left\{z_l^{v}(\phi,\theta,\chi)\right\}_{l=1}^{N_{v}},\left\{z_l^{v-e}(\psi,\theta,\zeta)
\right\}_{l=1}^{N_{v-e}},\left\{z_l^{e}(\tau,\theta,x_{3})\right\}_{l=1}^{N_{e}}\right\}\:.
\end{align*}

Here $z_{l}^{r}(\lambda_{1},\lambda_{2},\lambda_{3})$ is a polynomial of degree $W$ in each of its
variables.

$z_l^{v}=a_{v}$, where $a_v$ is a constant, on corner elements $\tilde{\Omega}_{l}^{v}$ with
$\mu(\tilde{\Omega}_{l}^{v})=\infty$. In all other elements in the vertex neighbourhoods,
$z_l^{v}(\phi,\theta,\chi)$ is a polynomial of degree $W_{l}, 1\leq W_{l}\leq W, W_{l}=[\mu_{1}i]$
for all $1\leq i\leq N+1$, in $\phi,\theta$ and $\chi$ variables separately, where $\mu_{1}>0$ is a
degree factor.

$z_l^{v-e}=a_{v-e}=a_{v}$, on corner elements $\tilde{\Omega}_{l}^{v-e}$ of the form
\[\tilde{\Omega}^{v-e}_l=\{x^{v-e}: \; \psi_i^{v-e}<\psi<\psi_{i+1}^{v-e},\;\theta_j^{v-e}<\theta<
\theta_{j+1}^{v-e},\;-\infty<\zeta<\zeta_{1}^{v-e}\}\]
and $z_l^{v-e}$ is a polynomial of degree $V_l$ in $\zeta$, $1\leq V_{l}\leq W, V_{l}=[\mu_{2}n]$
for all $1\leq n\leq N, \mu_{2}>0$, on corner elements $\tilde{\Omega}_{l}^{v-e}$ of the form
\[\tilde{\Omega}^{v-e}_l = \left\{x^{v-e}: -\infty<\psi<\psi_{1}^{v-e},\;\theta_j^{v-e}<\theta<
\theta_{j+1}^{v-e},\;\zeta_n^{v-e}<\zeta<\zeta_{n+1}^{v-e}\right\}\]
with $n \geq 1$.

On the remaining elements in the vertex-edge neighbourhoods, $z_{l}^{v-e}(\psi,\theta,\zeta)$
is a polynomial of degree $W_{l}, 1\leq W_{l}\leq W, W_{l}=[\mu_{1}i]$ for all $1\leq i\leq N$,
in $\psi,\theta$ variables and of degree $V_l$, $1\leq V_{l}\leq W, V_{l}=[\mu_{2}n]$ for all
$1\leq n\leq N$, in $\zeta$ variable with $\mu_{1}>0,\mu_{2}>0$.

Finally, on corner elements $\tilde{\Omega}_{l}^{e}$ with $\mu(\tilde{\Omega}_{l}^{e})=\infty$,
$z_l^{e}$ is a polynomial of degree $W$ in $x_{3}$ and on the remaining elements $\tilde{\Omega}_{l}^{e}$
away from edges $z_{l}^{e}(\tau,\theta,x_{3})$ is a polynomial of degree $W_{l},1\leq W_{l}\leq W, W_{l}
=[\mu_{1}i],1\leq i\leq N, \mu_{1}>0$ in $\tau,\theta$ variables and of degree $W$ in the $x_{3}$
variable.
\newline
\textbf{Approximation in the regular region:}

Let us first consider the regular region $\Omega^{r}$ of $\Omega$. In Chapter $2$, $\Omega^{r}$
has been divided into $\Omega_{l}^{r}$, $l=1,\ldots,N_{r}$ curvilinear hexahedrons, tetrahedrons
and prisms. Let $M_{l}^{r}$ be the analytic map from $Q$ to $\Omega_{l}^{r}$.

Let $\Pi^{W,W,W}(w(M_{l}^{r}(\lambda)))$ denote the projection of the solution $w$ into the space
of polynomials of degree $N$ in each of its variables with respect to the usual inner product in
$H^2(Q)$. Then
on $\Omega_{l}^{r}$ we define
$$s_l^r(\lambda)=\Pi^{W,W,W}(w(M_{l}^{r}(\lambda)))=\Pi^{W,W,W}(w(\lambda)),\:\mbox{for}\:\lambda\in Q.$$
\textbf{Approximation in vertex neighbourhoods:}

Let us now consider the vertex neighbourhood $\Omega^v$ of the vertex $v\in\mathcal V$, where
$\mathcal V$ denotes the set of vertices of $\Omega$ (see Figure \ref{fig2.8}). We had divided
$\Omega^v$ into $\Omega_{l}^{v},l=1,\ldots,N_{v}$ elements (Chapter $2$).
If $\tilde{{\Omega}}^v_l$ is a corner element of the form
$$\tilde{{\Omega}}^v_l=\{x^v:(\phi,\theta)\in S_j^v,-\infty<\chi<\ln(\rho_1^v)\}$$
then on $\tilde{{\Omega}}^v_l$ we define
$$s_l^v=w_v,$$
where $w_v=w(v)$ denotes the value of $w$ at the vertex $v$ defined as in our differentiability
estimates of Chapter $2$.

If ${\tilde{\Omega}}^v_l$ is of the form
\[\tilde{{\Omega}}^v_l=\{x^v:(\phi,\theta)\in S_j^v,\ln(\rho_i^v)<\chi<\ln(\rho_{i+1}^v)\}\]
then on $\tilde{{\Omega}}^v_l$ we approximate $(w(x^v)-w_v)$ by its projection, denoted by
$\Pi^{W_{l},W_{l},W_{l}}$, into the space of polynomials of degree $N$ in each of its variables
separately with respect to the usual inner product in $H^{2}(\tilde{{\Omega}}^v_l)$ and define
$$s_l^v(x^v)=\Pi^{W_{l},W_{l},W_{l}}(w(x^v)-w_v)+w_v.$$
Here $1\leq W_l\leq W$, $W_l=[\mu_{1} i]$ for all $1\leq i\leq N$, where $\mu_{1}>0$ is a degree
factor~\cite{G1}.
\newline
\textbf{Approximation in vertex-edge neighbourhoods:}

We now consider the vertex-edge neighborhood $\Omega^{v-e}$ of the vertex-edge
$v-e\in\mathcal{V-E}$ (see Figure \ref{fig2.9}). Here, as earlier, $\mathcal{V-E}$
denotes the set of vertex-edges of the domain $\Omega$. $\Omega^{v-e}$ is divided into
$\Omega_{q}^{v-e},q=1,\ldots,N_{v-e}$ elements using a geometric mesh in $\phi,x_{3}$
variables and a quasi-uniform mesh in $\theta$ variable.

Let $\tilde{\Omega}^{v-e}_q$ be the image of ${\Omega}^{v-e}_q$ in $x^{v-e}$ coordinates.
If $\tilde{\Omega}^{v-e}_q$ is a corner element of the form
\[\tilde{\Omega}^{v-e}_q=\{x^{v-e}: \; \psi_i^{v-e}<\psi<\psi_{i+1}^{v-e},\;\theta_j^{v-e}<
\theta<\theta_{j+1}^{v-e},\;-\infty<\zeta<\zeta_{1}^{v-e}\}\]
then on $\tilde{\Omega}^{v-e}_q$ we define
\[s_l^{v-e} = w_{v-e} = w_v\;.\]
Here $w_{v}$ is the value of $w$ at the vertex $v$ (Chapter $2$).

Next, suppose $\tilde{\Omega}^{v-e}_q$ is a corner element of the form
\[\tilde{\Omega}^{v-e}_q = \left\{x^{v-e}: -\infty<\psi<\psi_{1}^{v-e},\;\theta_j^{v-e}<
\theta<\theta_{j+1}^{v-e},\;\zeta_n^{v-e}<\zeta<\zeta_{n+1}^{v-e}\right\}\]
with $n \geq 1$. Let $s(x_3^{v-e})=w(x_1,x_2,x_3)|_{\left(x_1=0,x_2=0\right)}$ be the value of
$w$ along the edge $e$. Define
$$\sigma(x_3^{v-e})=s(x_3^{v-e})-w_v.$$
Let $\Pi^{V_q}(\sigma(x_3^{v-e}))$ be the orthogonal projection of $\sigma(x_3^{v-e})$ into
the space of polynomials in $H^{2}(I)$. Then we define
$$s_l^{v-e}(x_3^{v-e}) = \Pi^{V_q}(\sigma(x_3^{v-e}))+w_v=\Pi^{V_q}s(x_3^{v-e}).$$
Here $1\leq V_q \leq W$. Moreover $W_q=[\mu_{2} n]$ for all $1\leq n\leq N$, where $\mu_{2}>0$
is a degree factor.

The remaining elements $\tilde{\Omega}^{v-e}_q$ in $\tilde{\Omega}^{v-e}$ are of the form
\[\tilde{\Omega}^{v-e}_q = \left\{x^{v-e}:\;\psi_i^{v-e}<\psi<\psi_{i+1}^{v-e},
\;\theta_j^{v-e}<\theta<\theta_{j+1}^{v-e},\;\zeta_n^{v-e}<\zeta<\zeta_{n+1}^{v-e}\right\}\]
with $i \geq 1$, $k \geq 1$. Let us write $\alpha(x^{v-e})=w(x^{v-e})-s(x_3^{v-e})$. Then on
$\tilde{\Omega}^{v-e}_q$ we approximate $\alpha(x^{v-e})$ by its projection, denoted by
$\Pi^{W_{q},W_{q},V_{q}}$, into the space of polynomials with respect to the usual inner
product in $H^{2}(\tilde{\Omega}^{v-e}_q)$. We now define
\[s_l^{v-e}(x^{v-e})=\Pi^{W_{q},W_{q},V_{q}}(\alpha(x^{v-e}))+\Pi^{V_q}(s(x_3^{v-e})).\]
Here $1\leq W_q \leq W$ and $1\leq V_q \leq W$. Moreover $W_q=[\mu_{1} i], V_q=[\mu_{2} n]$ for
all $1\leq i,n\leq N$, where $\mu_{1},\mu_{2}>0$ are degree factors~\cite{G1}.


Finally, we discuss approximation in the edge neighbourhood elements
and define comparison functions there.
\newline
\textbf{Approximation in edge neighbourhoods:}

Consider the edge neighborhood $\Omega^{e}$ of the edge $e\in\mathcal{E}$ (see Figure
\ref{fig2.10}). Here, as before, $\mathcal E$ denotes the set of edges of the domain
$\Omega$. In Chapter $2$, we had divided $\Omega^{e}$ into $\Omega_{p}^{e},p=1,\ldots,N_{e}$
elements.

Let $\tilde{\Omega}^{e}_{p}$ be the image of ${\Omega}^{e}_{p}$ in $x^{e}$ coordinates.
Let $\tilde{\Omega}^{e}_p$ be a corner element of the form
\[\tilde{\Omega}^e_p = \left\{x^e:\: -\infty<x_1^e<\ln(r_{1}^e),
\:\theta_j^e<x_2^e<\theta_{j+1}^e,\:Z_n^e<x_3^e<Z_{n+1}^e\right\}\:.\]

Let $s(x_3^{e})=w\left(x_1,x_2,x_3\right)|_{\left(x_1=0,x_2=0\right)}$. Then on
$\tilde{\Omega}^{e}_{p}$ we approximate $s(x_3^{e})$ by its projection onto the space of
polynomials with respect to the usual inner product in $H^{2}(I)$. Let $\Pi^{W}(s(x_3^{e}))$
denote this projection, then we define
$$s_l^{e}(x_3^{e})=\Pi^{W}(s(x_3^{e})).$$

Next, let $\tilde{\Omega}^{e}_p$ be of the form
\[\tilde{\Omega}^e_p = \left\{x^e:\: \ln (r_i^e)<x_1^e<\ln (r_{i+1}^e),
\:\theta_j^e<x_2^e<\theta_{j+1}^e,\:Z_n^e<x_3^e<Z_{n+1}^e\right\}\]
with $1\leq i\leq N,\,0\leq j\leq I_e,\,0\leq n\leq J_e$. Let us write $\beta(x^{e})=w(x^{e})-s(x_3^{e})$.
Then on $\tilde{\Omega}^{e}_p$ we approximate $\beta(x^{e})$ by its projection, denoted by
$\Pi^{W_{p},W_{p},W}$, into the space of polynomials with respect to the usual inner product in
$H^{2}(\tilde{\Omega}^e_p)$. Define
\[s_l^{e}(x^{e}) = \Pi^{W_{p},W_{p},W}(\beta(x^{e}))+\Pi^{W}(s(x_3^{e})).\]
Here $1\leq W_p\leq W$. Moreover $W_p=[\mu_{1} i]$ for all $1\leq i\leq N$, where $\mu_{1}>0$
is a degree factor~\cite{G1}.
\newline

Now consider the set of functions $\left\{\{s_{l}^{r}\}_{l=1}^{N_{r}},\{s_{l}^{v}\}_{l=1}^{N_{v}},
\{s_{l}^{v-e}\}_{l=1}^{N_{v-e}},\{s_{l}^{e}\}_{l=1}^{N_{e}}\right\}$ and denote it by $\{\mathcal{F}_{s}\}$.

We will show that the functional defined by
\begin{align}\label{eq4.10}
\mathcal R^{N,W}\left(\{\mathcal F_s\}\right)&=\mathcal R_{regular}^{N,W}\left(\{\mathcal F_s\}\right)
+\mathcal R_{vertices}^{N,W}\left(\{\mathcal F_s\}\right) \notag\\
&+\mathcal R_{vertex-edges}^{N,W}\left(\{\mathcal F_s\}\right)+\mathcal R_{edges}^{N,W}(\{\mathcal F_s\})
\end{align}
is exponentially small in $N$. Here the functionals $\mathcal R_{regular}^{N,W}\left(\{\mathcal F_s\}\right)$,
$\mathcal R_{vertices}^{N,W}\left(\{\mathcal F_s\}\right)$ etc. are similar to those as defined in Section $4.2$.

Using results on approximation theory in~\cite{BG9,BG4,KDU} it follows that there exist constants $C$ and $b>0$
such that the estimate
\begin{align}\label{eq4.11}
\mathcal R^{N,W}\left(\{\mathcal F_s\}\right)\leq Ce^{-bN}
\end{align}
holds.

Now $\{\mathcal F_z\}$ minimizes $\mathcal R^{N,W}\left(\{\mathcal F_u\}\right)$ over all $\{\mathcal F_u\}
\in\mathcal{S}^{N,W}(\Omega)$, the space of spectral element functions. Then from (\ref{eq4.11}), we have
\begin{align}\label{eq4.12}
\mathcal R^{N,W}\left(\{\mathcal F_z\}\right)\leq Ce^{-bN}.
\end{align}
Let $\mathcal V^{N,W}$ be the quadratic form as defined in Chapter $2$. Then from (\ref{eq4.11}) and
(\ref{eq4.12}) we can conclude that
\begin{align}\label{eq4.13}
\mathcal V^{N,W}\left(\{\mathcal F_{(s-z)}\}\right)\leq Ce^{-bN}
\end{align}
where $C$ and $b$ are generic constants.
\newline
Hence using the Stability Theorem \ref{thm2.3.1} we obtain
\begin{align}\label{eq4.14}
\mathcal U^{N,W}\left(\{\mathcal F_{(s-z)}\}\right)\leq Ce^{-bN}.
\end{align}
Here the quadratic form $\mathcal U^{N,W}\left(\{\mathcal F_{(s-z)}\}\right)$ is defined similar to the
quadratic form $\mathcal U^{N,W}\left(\{\mathcal F_u\}\right)$ as in (\ref{eq2.15}) of Chapter $2$.

Let $U_l^r(\lambda)=w(X_{l}^{r}(\lambda_{1},\lambda_{2},\lambda_{3}))\equiv w(M_{l}^{r} (\lambda))$ for
$\lambda\in Q$, $U_l^v(x^{v})=w(x^{v})$ for $x^v\in\tilde{\Omega}_l^{v}$, $U_l^{v-e}(x^{v-e})=w(x^{v-e})$
for $x^{v-e}\in\tilde{\Omega}_l^{v-e}$ and $U_l^{e}(x^{e})=w(x^{e})$ for $x^e\in\tilde{\Omega}_l^{e}$.
Here $w$ is the solution of the boundary value problem (\ref{eq2.1}).

We now define another quadratic form $\mathcal{E}^{N,W}(\{z-U\})$ as follows:-
\begin{align}\label{eq4.15}
\mathcal E^{N,W}(\{z-U\})&=\mathcal E_{regular}^{N,W}(\{z_{l}^{r}-U_{l}^{r}\})
+\mathcal E_{vertices}^{N,W}(\{z_{l}^{v}-U_{l}^{v}\}) \notag\\
&+\mathcal E_{vertex-edges}^{N,W}(\{z_{l}^{v-e}-U_{l}^{v-e}\})
+\mathcal E_{edges}^{N,W}(\{z_{l}^{e}-U_{l}^{e}\})\:,
\end{align}
where
\begin{align}
\mathcal E_{regular}^{N,W}(\{z_{l}^{r}-U_{l}^{r}\})&=\sum_{l=1}^{N_{r}}\int_{Q=(M_{l}^{r})^{-1}(\Omega_{l}^{r})}
\sum_{|\alpha|\leq 2}\left|D_{\lambda}^{\alpha}\left(z_{l}^{r}-U_{l}^{r}\right)(\lambda)\right|^{2}d\lambda\:,\notag\\
\mathcal E_{vertices}^{N,W}(\{z_{l}^{v}-U_{l}^{v}\})&=\sum_{v\in\mathcal{V}}\mathcal E_{v}^{N,W}
\left(z_{l}^{v}-U_{l}^{v}\right)\:,\notag\\
\mathcal E_{v}^{N,W}(\{z_{l}^{v}-U_{l}^{v}\})&=\sum_{l=1}^{N_{v}}\int_{\tilde{\Omega}_{l}^{v}}e^{x_{3}^{v}}
\sum_{|\alpha|\leq 2}\left|D_{x^{v}}^{\alpha}\left(z_{l}^{v}-U_{l}^{v}\right)(x^{v})\right|^{2}dx^{v}\:,\notag\\
\mathcal E_{vertex-edges}^{N,W}(\{z_{l}^{v-e}-U_{l}^{v-e}\})&=\sum_{v-e\in\mathcal{V-E}}\mathcal E_{v-e}^{N,W}
(z_{l}^{v-e}-U_{l}^{v-e})\:,\notag\\
\mathcal E_{v-e}^{N,W}(\{z_{l}^{v-e}-U_{l}^{v-e}\})&=\sum_{l=1}^{N_{v-e}}\int_{\tilde{\Omega}_{l}^{v-e}}
e^{x_{3}^{v-e}}\sum_{|\alpha|\leq 2}\left|D_{x^{v-e}}^{\alpha}\left(z_{l}^{v-e}-U_{l}^{v-e}\right)(x^{v-e})
\right|^{2}dx^{v-e}\:,\notag\\
\mathcal E_{edges}^{N,W}(\{z_{l}^{e}-U_{l}^{e}\})&=\sum_{e\in\mathcal{E}}\mathcal E_{e}^{N,W}(z_{l}^{e}-U_{l}^{e})
\:,\notag\\
\mathcal E_{e}^{N,W}(\{z_{l}^{e}-U_{l}^{e}\})&=\sum_{l=1}^{N_{e}}\int_{\tilde{\Omega}_{l}^{e}}
\sum_{|\alpha|\leq 2}\left|D_{x^{e}}^{\alpha}\left(z_{l}^{e}-U_{l}^{e}\right)(x^{e})\right|^{2}dx^{e}\:.\notag
\end{align}
It is easy to show that
\begin{align}\label{eq4.16}
\mathcal E_{regular}^{N,W}(\{s_{l}^{r}-U_{l}^{r}\}) &\leq Ce^{-bN}\:,\notag\\
\mathcal E_{vertices}^{N,W}(\{s_{l}^{v}-U_{l}^{v}\}) &\leq Ce^{-bN}\:,\notag\\
\mathcal E_{vertex-edges}^{N,W}(\{s_{l}^{v-e}-U_{l}^{v-e}\}) &\leq Ce^{-bN}\:,\notag\\
\mathcal E_{edges}^{N,W}(\{s_{l}^{e}-U_{l}^{e}\}) &\leq Ce^{-bN}
\end{align}
where the quadratic forms $\mathcal E_{regular}^{N,W}\left(\{s_{l}^{r}-U_{l}^{r}\}\right)$,
$\mathcal E_{vertices}^{N,W}\left(\{s_{l}^{v}-U_{l}^{v}\}\right)$ etc. are defined similar to those in
(\ref{eq4.15}). Now define
\begin{align}
\mathcal E^{N,W}(\{s-U\})&=\mathcal E_{regular}^{N,W}(\{s_{l}^{r}-U_{l}^{r}\})
+\mathcal E_{vertices}^{N,W}(\{s_{l}^{v}-U_{l}^{v}\}) \notag\\
&+\mathcal E_{vertex-edges}^{N,W}(\{s_{l}^{v-e}-U_{l}^{v-e}\})
+\mathcal E_{edges}^{N,W}(\{s_{l}^{e}-U_{l}^{e}\})\:. \notag
\end{align}
Then from (\ref{eq4.16}) it follows that
\begin{align}\label{eq4.17}
\mathcal E^{N,W}\left(\{s-U\}\right)\leq Ce^{-bN}.
\end{align}
Finally, using estimates (\ref{eq4.14}) and (\ref{eq4.17}), we obtain
\begin{align}
\mathcal U^{N,W}\left(\{\mathcal F_{(z-U)}\}\right)\leq Ce^{-bN}. \notag
\end{align}
Our main theorem on error estimates is now stated
\begin{theo}\label{thm4.3.1}
Let $\{\mathcal F_{z}\}$ minimize $\mathcal R^{N,W}(\{\mathcal F_u\})$ over all
$\{\mathcal F_{u}\}\in {\mathcal S}^{N,W}(\Omega)$. Then there exist constants $C$
and $b$ (independent of $N$) such that
\begin{align}\label{eq4.18}
\mathcal U^{N,W}\left(\{\mathcal F_{(z-U)}\}\right)\leq Ce^{-bN}.
\end{align}
Here $\mathcal U^{N,W}\left(\{\mathcal F_{(z-U)}\}\right)$ is as defined in (\ref{eq2.15}) of Chapter $2$.
\end{theo}
\begin{rem}\label{rem4.3.0}
After having obtained the non-conforming spectral element solution we can make a correction to it
so that the corrected solution is conforming and the error in the $H^1$ norm is exponentially small
in $N$. These corrections are similar to those described in~\cite{KDU,SKT1,SKT2} in two dimensions
and are explained in the proof of Lemma \ref{lem3.3.1} in Appendix C.1.
\end{rem}

To end this chapter, let us estimate the error in terms of number of degrees of freedom in various
subregions of the domain $\Omega$.
\newline
\textbf{The regular region $\Omega^r$:}

The regular region $\Omega^r$ contains no vertices and edges of the domain $\Omega$. Here the
solution $w$ has no singularity and is analytic.

There are $O(1)$ number of elements in this region and each element has $O(W^{3})$ degrees of
freedom. Let $N_{dof}(\Omega^{r})$ denotes the number of degrees of freedom in $\Omega^r$. Then
$$N_{dof}(\Omega^{r})=O(W^{3})=O(N^{3})\:.$$
\textbf{The vertex neighbourhoods $\Omega^v$:}

In a vertex neighbourhood $\Omega^v$ there are $O(N)$ elements with $O(W^{3})$ degrees of freedom
in each element. If $N_{dof}(\Omega^{v})$ denotes the number of degrees of freedom in $\Omega^v$. Then
$$N_{dof}(\Omega^{v})=O(NW^{3})=O(N^{4}).$$
\textbf{The vertex-edge neighbourhoods $\Omega^{v-e}$:}

There are $O(N^2)$ number of elements in each of the vertex-edge neighbourhoods $\Omega^{v-e}$ and
each element has $O(W^{3})$ degrees of freedom. Then
$$N_{dof}(\Omega^{v-e})=O(N^2W^{3})=O(N^{5})\:.$$
Here $N_{dof}(\Omega^{v-e})$ denotes the number of degrees of freedom in $\Omega^{v-e}$.
\newline
\textbf{The edge neighbourhoods $\Omega^{e}$:}

An edge neighbourhood $\Omega^{e}$ has $O(N)$ elements with $O(W^{3})$ degrees of freedom within
each element. Let $N_{dof}(\Omega^{e})$ be the number of degrees of freedom in $\Omega^{e}$. Then
$$N_{dof}(\Omega^{e})=O(NW^{3})=O(N^{4})\:.$$
Hence, the error estimate Theorem \ref{thm4.3.1} in terms of number of degrees of freedom
assumes the form
\begin{theo}\label{thm4.3.2}
Let $\{\mathcal F_{z}\}$ minimizes $\mathcal R^{N,W}(\{\mathcal F_u\})$ over all
$\{\mathcal F_{u}\}\in {\mathcal S}^{N,W}(\Omega)$. Then there exist constants $C$
and $b$ (independent of $N$) such that
\begin{align}\label{eq4.19}
\mathcal U^{N,W}\left(\{\mathcal F_{(z-U)}\}\right)\leq Ce^{-bN_{dof}^{1/5}}.
\end{align}
Here $\mathcal U^{N,W}\left(\{\mathcal F_{(z-U)}\}\right)$ is as defined in (\ref{eq2.15}) and
$N_{dof}=dim(\mathcal{S}^{N,W}(\Omega))$ is the number of degrees of freedom.
\end{theo}
\begin{proof}
Follows from Theorem \ref{thm4.3.1}.
\end{proof}
\begin{rem}\label{rem4.3.1}
From the above theorem it is clear that the exponential rate of
convergence will be visible only for a large value of $N_{dof}$, as
a result we need to sufficiently refine the geometric mesh both in
the direction of edges and in the direction perpendicular to the
edges.
\end{rem}
\begin{rem}\label{rem4.3.4}
Since the majority of degrees of freedom is present in the vertex-edge neighbourhoods the factor $N_{dof}^{1/5}$
in the theorem is due to the vertex-edge singularity in the solution.
Hence the optimal convergence rate will be $e^{-bN_{dof}^{1/5}}$.
\end{rem}
\begin{rem}\label{rem4.3.3}
It was conjectured in~\cite{BG5,G1} that for $h-p$ version of the finite element method in $\mathbb{R}^3$
the optimal convergence rate will be $e^{-bN_{dof}^{1/5}}$, and it can not be improved further with any
mesh and any anisotropic polynomial order within the elements.
\end{rem}

\chapcleardoublepage

\chapter{Solution Techniques on Parallel Computers}
\section{Introduction}
In Chapter 4 we had described the numerical scheme and the error estimates for the
numerical solution were stated.

The method is essentially a \textit{least-squares} collocation method. To solve
the minimization problem we need to solve the \textit{normal equations} arising
from the minimization problem. We use \textit{preconditioned conjugate gradient
method} (PCGM) to solve the normal equations using a \textit{block diagonal
preconditioner} which is nearly optimal as the condition number of the
preconditioned system is polylogarithmic in $N$, the number of refinements in
the geometric mesh. Moreover, to compute the residuals in the normal equations
we do not need to compute mass and stiffness matrices.

We use only non-conforming spectral element functions. We shall examine spectral
element functions which are conforming on the wirebasket in future work.

It is shown that we can define a preconditioner for the minimization problem which
allows the problem to decouple. Construction of preconditioner on each element in
various regions of the polyhedron is also provided. It will be shown as in~\cite{DBR}
that there exists a new preconditioner which can be diagonalized in a new set of
basis functions using separation of variables technique.

The coefficients in the differential operator and the data are projected into the
polynomial spaces (we call them \textit{filtered coefficients} of the differential
operator and \textit{filtered data}) so that the integrands involved in the numerical
formulation are exactly evaluated using Gauss-Lobatto-Legendre (GLL) quadrature.
However, in our computations we do not need to filter the coefficients of the
differential operator and the data. Of course, in doing so we commit an error and it
can be argued as in~\cite{SKT2} that this error is spectrally small.

In section $5.3$ we shall describe the steps involved in computing the residuals in the
normal equations without computing mass and stiffness matrices on a distributed memory
parallel computer. The evaluation of these residuals on each processor requires the
interchange of boundary values between neighbouring processors. Hence the communication
involved is quite small. Thus we can compute the residuals in the normal equations
inexpensively and efficiently. Finally, we shall discuss computational complexity of
our method.

\section{Preconditioning}
We now come to the aspect of preconditioning. Our construction of preconditioners is
similar to that for elliptic problems in two dimensions (see~\cite{DTK1,DT,SKT1,SKT2}).

\subsection{Preconditioners on the regular region}
In the regular region the preconditioner which needs to be examined corresponds to the
quadratic form
\begin{align}\label{eq5.7}
\mathcal B({u}) =||{u}||^{2}_{H^{2}(Q)}
\end{align}
where $Q=(-1,1)^{3}=$master cube, ${u}={u}(\lambda)={u}(\lambda_1,\lambda_2,\lambda_3)$ is
a polynomial of degree $W$ in $\lambda_1,\lambda_2$ and $\lambda_3$ separately.

Let ${u}(\lambda_1,\lambda_2,\lambda_3)$ be the spectral element function, defined on
$Q=(-1,1)^{3}$, as
\begin{equation}\label{eq5.8}
u(\lambda_1,\lambda_2,\lambda_3)=\sum_{i=0}^W\sum_{j=0}^W\sum_{k=0}^W a_{i,j,k}L_{i}(\lambda_1)
L_{j}(\lambda_2)L_{k}(\lambda_3).
\end{equation}
\noindent
Here $L_{i}(\cdot)$ denotes the Legendre polynomial of degree $i$.
\newline
The quadratic form $\mathcal B({u})$ can be written as
\begin{equation}\label{eq5.9}
\mathcal{B}({u})=\int_{Q}\sum_{|\alpha|\le 2}{|D^{\alpha}_{\lambda}{u}|}^2d\lambda.
\end{equation}
\noindent We shall show that there is another quadratic form
$\mathcal{C}(u)$ which is spectrally equivalent to $\mathcal{B}(u)$
and which can be easily diagonalized using separation of variables.

Let $I$ denote the interval $(-1,1)$ and
\begin{equation}\label{eq5.10}
v(\lambda_1) = \sum_{i=0}^{W}\beta_{i}L_{i}(\lambda_1) \;.
\end{equation}
Moreover $b={(\beta_{0},\beta_{1},\ldots,\beta_{W})}^T $. We now define the quadratic
form
\begin{equation}\label{eq5.11}
\mathcal{E}(v) = \int_{I}(v_{\lambda_1\lambda_1}^2+v_{\lambda_1}^2 )d\lambda_1
\end{equation}
and
\begin{equation}\label{eq5.12}
\mathcal{F}(v) = \int_{I} v^2 d\lambda_1 \;.
\end{equation}
Clearly there exist $(W+1)\times(W+1)$ matrices $E$ and $F$ such that
\begin{equation}\label{eq5.13}
\mathcal{E}(v) = b^T E b
\end{equation}
and
\begin{equation}\label{eq5.14}
\mathcal{F}(v) = b^T F b.
\end{equation}
Here the matrices $E$ and $F$ are symmetric and $F$ is positive definite.
\newline
Hence there exist $W+1$ eigenvalues $0 \le \mu_0 \le \mu_1\le\cdots\le \mu_W$ and $W+1$
eigenvectors $b_0, b_1,\ldots,b_W$ of the symmetric eigenvalue problem
\begin{equation}\label{eq5.15}
(E-\mu F)b=0  \;.
\end{equation}
Here
$$(E-\mu_i F)b_i=0\;.$$
The eigenvectors $b_i$ are normalized so that
\begin{subequations}\label{eq5.16}
\begin{equation}\label{eq5.16a}
b_i^T F b_j= \delta^i_j  \;.
\end{equation}
Moreover the relations
\begin{align}\label{eq5.16b}
b_i^T E b_j=\mu_i \delta^i_j
\end{align}
\end{subequations}
hold. Let $b_i = (b_{i,0}, b_{i,1},\ldots, b_{i,W})$. We now define the polynomial
\begin{equation}\label{eq5.17}
\phi_i(\lambda_1)=\sum_{j=0}^{W}b_{i,j}L_j(\lambda_1)\:\:\mbox{for}\:\:0\le i\le W\:.
\end{equation}
Next, let $\psi_{i,j,k}$ denote the polynomial
\begin{equation}\label{eq5.18}
\psi_{i,j,k}(\lambda_1,\lambda_2,\lambda_3)=\phi_i(\lambda_1)\phi_j(\lambda_2)
\phi_k(\lambda_3)
\end{equation}
for\quad $0 \leq i \le W,\;0 \le j\le W,\;0 \le k\le W $.

Let $u(\lambda_1,\lambda_2,\lambda_3)$ be a polynomial as in (\ref{eq5.8}). Define the
quadratic form
\begin{equation}\label{eq5.19}
\mathcal{C}(u)=\int_Q(u_{\lambda_1\lambda_1}^2+u_{\lambda_2\lambda_2}^2
+u_{\lambda_3\lambda_3}^2+u_{\lambda_1}^2+u_{\lambda_2}^2+u_{\lambda_3}^2+u^2)
d\lambda_1 d\lambda_2 d\lambda_3 \;.
\end{equation}
Then the quadratic form $\mathcal{C}(u)$ is spectrally equivalent to the quadratic form
$\mathcal{B}(u)$, defined in (\ref{eq5.7}). Moreover, the quadratic form $\mathcal{C}(u)$
can be diagonalized in the basis $\psi_{i,j,k}(\lambda_1,\lambda_2,\lambda_3)$. Note that
${\{\psi_{i,j,k}(\lambda_1,\lambda_2,\lambda_3)\}}_{i,j,k}$ is the tensor product of the
polynomials $\phi_i(\lambda_1)$, $\phi_j(\lambda_2)$ and $\phi_k(\lambda_3)$. The eigenvalue
$\mu_{i,j,k}$ corresponding to the eigenvector $\psi_{i,j,k}$ is given by the relation
\begin{equation}\label{eq5.20}
\mu_{i,j,k}=\mu_i+\mu_j+\mu_k+1\;.
\end{equation}
Hence the matrix corresponding to the quadratic form $\mathcal{C}(u)$ is easy to invert.

Using the extension theorems in~\cite{ADAM} and Lemma $2.1$
in~\cite{DBR} we can extend $u(\lambda_1,\lambda_2,\lambda_3)$
defined in (\ref{eq5.8}) to $U(\lambda_1,\lambda_2,\lambda_3)$ by
the method of reflection (see Theorem $4.26$ of~\cite{ADAM}). This
extension $U(\lambda_1,\lambda_2,\lambda_3)$ of
${u}(\lambda_1,\lambda_2,\lambda_3)$ is such that
$U(\lambda_1,\lambda_2,\lambda_3)\in H^{2}(\mathbb{R}^3)$ and
satisfies the estimate
\begin{align}
\int_{{\mathbb{R}}^3}\left(U^2_{\lambda_1\lambda_1}+U^2_{\lambda_2\lambda_2}+U^2_{\lambda_3
\lambda_3}+U^2\right)d\lambda\le K\int_Q\left({u}^2_{\lambda_1\lambda_1}+u^2_{\lambda_2\lambda_2}
+ {u}^2_{\lambda_3\lambda_3} + {u}^2\right)d\lambda\;.\notag
\end{align}
Here $K$ is a constant independent of $W$. Now making use of Theorem $2.1$ of~\cite{DBR}
and extending it to three dimensions it follows that there exists a constant $L$ (independent
of $W$) such that
\begin{align}
\frac{1}{L}||{u}||^{2}_{H^{2}(Q)} &\leq\int_{Q}\left(|{u}_{\lambda_1\lambda_1}|^{2}
+|{u}_{\lambda_2\lambda_2}|^{2}+|{u}_{\lambda_3\lambda_3}|^{2}+|{u}_{\lambda_1}|^2
+|{u}_{\lambda_2}|^2+|{u}_{\lambda_3}|^2+|{u}|^{2}\right)d\lambda \notag\\
&\leq ||{u}||^{2}_{H^{2}(Q)}\:.\notag
\end{align}
i.e. the quadratic forms $\mathcal{B}({u})$ and $\mathcal{C}({u})$ are spectrally equivalent.
\begin{theo}\label{thm5.1}
The quadratic forms $\mathcal{B}({u})$ and $\mathcal{C}({u})$ are spectrally equivalent.
\end{theo}
We now show that the quadratic form $\mathcal{C}({u})$ defined in (\ref{eq5.19}) as
\begin{equation}
\mathcal{C}({u})=\int_Q\left({u}_{\lambda_1\lambda_1}^2+{u}_{\lambda_2\lambda_2}^2
+{u}_{\lambda_3\lambda_3}^2+{u}_{\lambda_1}^2+{u}_{\lambda_2}^2+{u}_{\lambda_3}^2+{u}^2\right)
d\lambda_1 d\lambda_2 d\lambda_3 \notag
\end{equation}
can be diagonalized in the basis $\{\psi_{i,j,k}\}_{i,j,k}$. Here ${u}$ is a polynomial
in $\lambda_1$, $\lambda_2$ and $\lambda_3$ as defined in (\ref{eq5.8}). Let $\widetilde
{\mathcal{C}}(f,g)$ denote the bilinear form induced by the quadratic form $\mathcal{C}({u})$.
Then
\begin{align}\label{eq5.21}
\widetilde{\mathcal{C}}(f,g)=\int_Q\left(f_{\lambda_1\lambda_1}g_{\lambda_1\lambda_1}
+f_{\lambda_2\lambda_2}g_{\lambda_2\lambda_2}+f_{\lambda_3\lambda_3}g_{\lambda_3\lambda_3}
+f_{\lambda_1}g_{\lambda_1}+f_{\lambda_2}g_{\lambda_2}\right. \notag\\
\left.+f_{\lambda_3}g_{\lambda_3}+fg\right)\:d\lambda_1 d\lambda_2 d\lambda_3 \;.
\end{align}
Let $\mathcal{E}(v)$ and $\mathcal{F}(v)$ be the quadratic forms defined in (\ref{eq5.11})
and (\ref{eq5.12}) and let $\widetilde{\mathcal{E}}(f,g)$ and $\widetilde{\mathcal{F}}(f,g)$
denote the bilinear forms induced by $\mathcal{E}(v)$ and $\mathcal{F}(v)$ respectively. Then
\begin{subequations}\label{eq5.22}
\begin{align}\label{eq5.22a}
\widetilde{\mathcal{E}}(f,g)=\int_{I}(f_{\lambda_1\lambda_1}g_{\lambda_1\lambda_1}
+f_{\lambda_1}g_{\lambda_1})d\lambda_1
\end{align}
and
\begin{align}\label{eq5.22b}
\widetilde{\mathcal{F}}(f,g)=\int_{I}fg\:d\lambda_1\:.
\end{align}
\end{subequations}
Here $I$ denotes the unit interval and $f(\lambda_1)$, $g(\lambda_1)$ are polynomials
of degree $W$ in $\lambda_1$.

Finally, let $\phi_i(\lambda_1)$ be the polynomial as defined in (\ref{eq5.17}). Then relation
(\ref{eq5.16a}) may be written as
\begin{subequations}\label{eq5.23}
\begin{align}\label{eq5.23a}
\widetilde{\mathcal{F}}(\phi_i,\phi_j)=\int_{I}\phi_i(\lambda_1)\phi_j(\lambda_1)
\:d\lambda_1=\delta_{j}^{i}\:.
\end{align}
Moreover, relation (\ref{eq5.16b}) may be stated as
\begin{align}\label{eq5.23b}
\widetilde{\mathcal{E}}(\phi_i,\phi_j)=\int_{I}((\phi_i)_{\lambda_1\lambda_1}
(\phi_j)_{\lambda_1\lambda_1}+(\phi_i)_{\lambda_1}(\phi_j)_{\lambda_1})
\:d\lambda_1=\mu_{i}\delta_{j}^{i}\:.
\end{align}
\end{subequations}
Recalling that $\psi_{i,j,k}(\lambda_1,\lambda_2,\lambda_3)=\phi_{i}(\lambda_1)\phi_{j}
(\lambda_2)\phi_{k}(\lambda_3)$ and using (\ref{eq5.23}) in (\ref{eq5.21}) it is easy to
show that
\begin{align}
\widetilde{\mathcal{C}}(\psi_{i,j,k},\psi_{l,m,n})&=(\mu_i+\mu_j+\mu_k+1)\delta_{l}^{i}
\delta_{m}^{j}\delta_{n}^{k} \notag\\
&=\mu_{i,j,k}\delta_{l}^{i}\delta_{m}^{j}\delta_{n}^{k}\:. \notag
\end{align}
Hence the eigenvectors of the quadratic form $\mathcal{C}(u)$ are $\{\psi_{i,j,k}\}_{i,j,k}$
and the eigenvalues are $\{\mu_{i,j,k}\}_{i,j,k}$. Moreover the quadratic form
$\mathcal{C}({u})$ can be diagonalized in the basis ${\{\psi_{i,j,k}\}}_{i,j,k}$ and
consequently the matrix corresponding to $\mathcal{C}({u})$ is easy to invert.

Let
$${u}(\lambda_1,\lambda_2,\lambda_3)=\sum_{i=0}^{W}\sum_{j=0}^{W}\sum_{k=0}^{W}
\beta_{i,j,k}L_i(\lambda_1)L_j(\lambda_2)L_k(\lambda_3)$$
and $\beta$ denotes the column vector whose components are $\beta_{i,j,k}$ arranged in
lexicographic order. Then there is a ${(W+1)}^3\times{(W+1)}^3$ matrix $C$ such that
$$\mathcal{C}({u})=\beta^T C \beta\;.$$
As in~\cite{DBR} it can be shown that the system of equations
\begin{equation}
C\beta =\rho \notag
\end{equation}
can be solved in $O(W^4)$ operations. Therefore the quadratic form $\mathcal{C}({u})$
can be inverted in $O(W^4)$ operations.

Let $\kappa$ denote the condition number of the preconditioned system obtained by using the
quadratic form $\mathcal{C}(u)$ as a preconditioner for the quadratic form $\mathcal{B}(u)$.
Then the values of $\kappa$ as a function of $W$ are shown in Table \ref{tab5.1}.
\pagebreak
\begin{table}[!ht]
\caption{Condition number $\kappa$ as a function of W}
\label{tab5.1}
\begin{center}
         \begin{tabular}{|c|c|}
            \hline
             \text{W} & \text{$\kappa$}\\
            \hline
            \hline
             2 & 3.69999999999999 \\
            \hline
             4 & 4.90406593328559 \\
            \hline
             6 & 5.27448215795748 \\
            \hline
             8 & 5.48239323328901 \\
            \hline
             10 & 5.62480021244268 \\
            \hline
             12 & 5.72673215953223 \\
            \hline
             14 & 5.80192403338903 \\
            \hline
             16 & 5.85907843805046 \\
            \hline
        \end{tabular}
\end{center}
\end{table}

In Figure \ref{fig5.1}, the condition number $\kappa$ is plotted against the polynomial
order $W$.
\begin{figure}[!ht]
\centering
\input{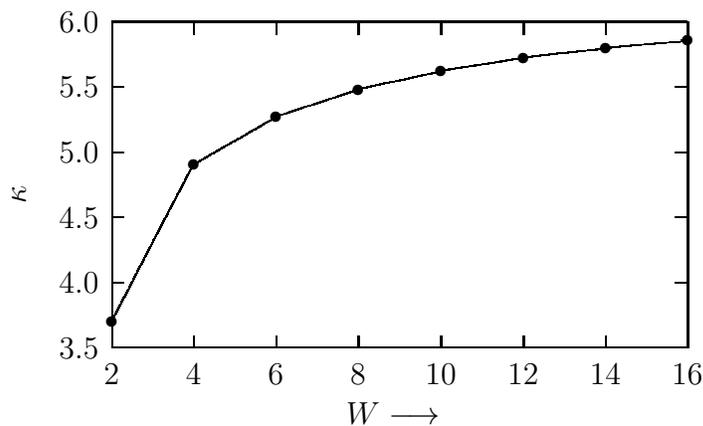}
\caption[$\kappa$ vs. $W$.]
{Condition number $\kappa$ vs. $W$.}
\label{fig5.1}
\end{figure}

\subsection{Preconditioners on singular regions}
In Chapter $2$ we had divided the polyhedral domain $\Omega$ into a regular region
$\Omega^r$, a set of vertex neighbourhoods $\Omega^v$, a set of edge neighbourhoods
$\Omega^e$ and a set of vertex-edge neighbourhoods $\Omega^{v-e}$. $\Omega^r$ is
divided into a set of curvilinear hexahedrons, tetrahedrons and prisms and the
singular regions in the neighbourhoods of vertices, edges and vertex-edges are
divided into hexahedrons and prisms using a geometric mesh.

A set of spectral element functions has been defined on all elements in the regular region
and various singular regions.

Let $N$ denote the number of refinements in the geometric mesh. We
shall assume that $N$ is proportional to $W$. Then the proposed
method gives exponentially accurate solution in $N$ provided the
data satisfy usual regularity conditions~\cite{G1,GOH1,SKT2}.

We choose our spectral element functions to be fully non-conforming. Let ${\mathcal F}_u$
denote the spectral element representation of the function $u$.
We define the quadratic form
\begin{align}\label{eq5.24}
{\mathcal W}^{N,W}(\{{\mathcal F}_u\})&={\mathcal W}^{N,W}_{regular}(\{{\mathcal F}_u\})
+{\mathcal W}^{N,W}_{vertices}(\{{\mathcal F}_u\})+{\mathcal W}^{N,W}_{vertex-edges}
(\{{\mathcal F}_u\})\notag \\
&+{\mathcal W}^{N,W}_{edges}(\{{\mathcal F}_u\})\:.
\end{align}
Here ${\mathcal W}^{N,W}_{regular}(\{{\mathcal F}_u\})$, ${\mathcal W}^{N,W}_{vertices}
(\{{\mathcal F}_u\})$, ${\mathcal W}^{N,W}_{vertex-edges}(\{{\mathcal F}_u\})$ and
${\mathcal W}^{N,W}_{edges}(\{{\mathcal F}_u\})$ are defined similar to the quadratic
forms ${\mathcal U}^{N,W}_{regular}(\{{\mathcal F}_u\})$, ${\mathcal U}^{N,W}_{vertices}
(\{{\mathcal F}_u\})$, ${\mathcal U}^{N,W}_{vertex-edges}(\{{\mathcal F}_u\})$ and
${\mathcal U}^{N,W}_{edges}(\{{\mathcal F}_u\})$ as in (\ref{eq2.19}), (\ref{eq2.21}),
(\ref{eq2.33}) and (\ref{eq2.41}) respectively. Then for problems with Dirichlet boundary
conditions by Theorem \ref{thm2.3.1} it follows that the following estimate is true
\begin{equation}\label{eq5.25}
\mathcal W^{N,W}(\{\mathcal F_u\}) \leq C (\ln W)^2 \mathcal V^{N,W}(\{\mathcal F_u\})
\end{equation}
provided $W=O(e^{N^{\alpha}})$ for $\alpha<1/2$.

At the same time using trace theorems for Sobolev spaces there exists a constant $k$ such
that
\begin{equation}\label{eq5.26}
\frac{1}{k}\mathcal V^{N,W}(\{\mathcal F_u\}) \leq \mathcal W^{N,W}(\{\mathcal F_u\})\:.
\end{equation}
Hence using (\ref{eq5.25}) and (\ref{eq5.26}) we conclude that the two quadratic forms
$\mathcal W^{N,W}(\{\mathcal F_u\})$ and $\mathcal V^{N,W}(\{\mathcal F_u\})$ are spectrally
equivalent and there exists a constant $K$ such that
\begin{equation}\label{eq5.27}
\frac{1}{K}\mathcal V^{N,W}(\{\mathcal F_u\}) \leq \mathcal W^{N,W}(\{\mathcal F_u\})
\leq K (\ln W)^2 \mathcal V^{N,W}(\{\mathcal F_u\})
\end{equation}
provided $W=O(e^{N^{\alpha}})$ for $\alpha<1/2$.

We can now use the quadratic form $\mathcal W^{N,W}(\{\mathcal F_u\})$ which consists
of a decoupled set of quadratic forms on each element as a preconditioner. It follows
that the condition number of the preconditioned system is $O(\ln W)^2$.

The other case is when the boundary conditions are of mixed Neumann and Dirichlet type.
In this case, as above, using Theorem \ref{thm2.3.2} and trace theorems for Sobolev spaces
it follows that for $W$ and $N$ large enough the following estimate holds
\begin{equation}
\frac{1}{K}\mathcal V^{N,W}(\{\mathcal F_{u}\})\leq\mathcal W^{N,W}(\{\mathcal F_{u}\})
\leq K N^4 \mathcal V^{N,W}(\{\mathcal F_{u}\})\:. \notag
\end{equation}
Here $K$ is a constant. It is clear that the quadratic form
$\mathcal W^{N,W}(\{\mathcal F_{u}\})$ can be used as a preconditioner and the
condition number of the preconditioned system is $O(N^4)$.
\newline
%

We will now construct preconditioners on each of the elements in the neighbourhoods of vertices,
edges, vertex-edges and the regular region. Here $u$ denotes the spectral element function which
is a polynomial of degree $W$ in each of its variables separately defined in various regions of
the polyhedron.

The quadratic forms which need to be examined are
\begin{align}\label{eq5.28}
\mathcal{B}_{regular}(u)&={\Arrowvert u \Arrowvert}^2_{H^{2}(Q)}=\int_{Q}\sum_{|\alpha|\le 2}
{|D^{\alpha}_{\lambda_1,\lambda_2,\lambda_3}u|}^2\:d\lambda_1 d\lambda_2 d\lambda_3\:,\\
\mathcal{B}_{vertices}(u)&={\Arrowvert e^{\chi/2}u \Arrowvert}^2_{H^{2}(\tilde{\Omega}^{v}_{l})}
=\int_{\tilde{\Omega}^{v}_{l}}e^{\chi}\sum_{|\alpha|\le 2}{|D^{\alpha}_{\phi,\theta,\chi}u|}^2
\:d\phi d\theta d\chi\:, \\
\mathcal{B}_{vertex-edges}(u)&=\int_{\tilde{\Omega}^{v-e}_{l}}e^{\zeta}(u^2_{\psi\psi}+u^2_{\theta
\theta}+u^2_{\psi\theta}+\sin^{2}\phi u^2_{\phi\zeta}+\sin^{2}\phi u^2_{\theta\zeta}\notag\\
&+\sin^{4}\phi u^2_{\zeta\zeta}+u^2_{\psi}+u^2_{\theta}+\sin^{2}\phi u^2_{\zeta}
+u^2)\:d\psi d\theta d\zeta\:,\\
\mathcal{B}_{edges}(u)&=\int_{\tilde{\Omega}^{e}_{l}}(u^2_{\tau\tau}+u^2_{\theta\theta}
+u^2_{\tau\theta}+e^{2\tau}u^2_{\tau x_3}+e^{2\tau}u^2_{\theta x_3}+e^{4\tau}u^2_{x_3x_3}\notag\\
&+u^2_{\tau}+u^2_{\theta}+e^{2\tau}u^2_{x_3}+u^2)\:d\tau d\theta dx_3\:.
\end{align}
Here $(\phi,\theta,\chi), (\psi,\theta,\zeta)$ and $(\tau,\theta,x_3)$ denote the modified systems
of coordinates introduced in Chapter $2$ in vertex neighbourhoods, vertex-edge neighbourhoods
and edge neighbourhoods respectively. Moreover $\tilde{\Omega}^{v}_{l}$, $\tilde{\Omega}^{v-e}_{l}$
and $\tilde{\Omega}^{e}_{l}$ denote elements in the vertex neighbourhood, vertex-edge neighbourhood
and edge neighbourhood respectively.

The construction of preconditioners corresponding to the quadratic
forms $\mathcal{B}_{regular}(u)$ and $\mathcal{B}_{vertices}(u)$ is
similar to the case of a smooth domain already discussed so we omit
the details. It follows that there exist quadratic forms
$\mathcal{C}_{regular}(u)$ and $\mathcal{C}_{vertices}(u)$ which are
spectrally equivalent to $\mathcal{B}_{regular}(u)$ and
$\mathcal{B}_{vertices}(u)$ respectively and which can be
diagonalized using separation of variables technique. We will now
obtain preconditioners for edge and vertex-edge neighbourhood
elements. For this purpose we observe that for quadratic forms in
edge and vertex-edge neighbourhoods it is enough to examine the
quadratic form
\begin{align}\label{eq5.29}
\mathcal{B}^{\star}(u)&=\int_{-1}^{1}\int_{-1}^{1}\int_{-1}^{1}(u^2_{xx}+u^2_{yy}
+\eta^{2}u^2_{xz}+\eta^{2}u^2_{yz}+\eta^{4}u^2_{zz}\notag\\
&\hspace{3.0cm}+u^2_{x}+u^2_{y}+\eta^{2}u^2_{z}+u^2)\:dx dy dz\:.
\end{align}
Here $\eta=\sin\phi$ and $\eta=e^{\tau}$ for vertex-edge and edge neighbourhood elements respectively.
We remark that the factor $\eta$ becomes smaller towards the vertices and edges of the domain $\Omega$.

Making the transformation $\tilde{z}=\frac{z}{\eta}$, so that $\frac{d}{dz}=\frac{1}{\eta}\frac{d}
{d\tilde{z}}$, the quadratic form $\mathcal{B}^{\star}(u)$ assumes the form
\begin{align}
\mathcal{B}^{\star}(u)&=\int_{-{\frac{1}{\eta}}}^{\frac{1}{\eta}}\int_{-1}^{1}\int_{-1}^{1}
(u^2_{xx}+u^2_{yy}+u^2_{xy}+u^2_{x\tilde{z}}+u^2_{y\tilde{z}}+u^2_{\tilde{z}\tilde{z}}+u^2_{x}
+u^2_{y}+u^2_{\tilde{z}}+u^2)\:dxdyd\tilde{z}\:.\notag
\end{align}
Let us define the quadratic form
\begin{align}\label{eq5.30}
\mathcal{C}^{\star}(u)&=\int_{-1}^{1}\int_{-1}^{1}\int_{-1}^{1}(u^2_{xx}+u^2_{yy}
+\eta^{4}u^2_{zz}+u^2_{x}+u^2_{y}+\eta^{2}u^2_{z}+u^2)\:dx dy dz\:.
\end{align}

We now show that the quadratic form $\mathcal{C}^{\star}(u)$ is
spectrally equivalent to the quadratic form
$\mathcal{B}^{\star}(u)$, defined in $(\ref{eq5.29})$. Moreover,
$\mathcal{C}^{\star}(u)$ can be diagonalized using separation of
variables technique.

Let
\begin{equation}
v(x) = \sum_{i=0}^{W}\beta_{i}L_{i}(x) ,\quad\mbox{and}\quad
v(z) = \sum_{i=0}^{W}\gamma_{i}L_{i}(z) \;. \notag
\end{equation}
Moreover $b={(\beta_{0},\beta_{1},\ldots,\beta_{W})}^T $ and
$d={(\gamma_{0},\gamma_{1},\ldots,\gamma_{W})}^T$.
\newline
We now define the quadratic form
\begin{equation}\label{eq5.31}
\mathcal{E}(v) = \int_{I}(v_{xx}^2+v_{x}^2 )dx
\end{equation}
and
\begin{equation}\label{eq5.32}
\mathcal{F}(v) = \int_{I} v^2 dx \;.
\end{equation}
We also define the quadratic form
\begin{equation}\label{eq5.33}
\mathcal{G}(v) = \int_{I}(\eta^{4}v_{zz}^2+\eta^{2}v_{z}^2 )dz
\end{equation}
and
\begin{equation}\label{eq5.34}
\mathcal{H}(v) = \int_{I} v^2 dz \;.
\end{equation}
Here $I$ denotes the unit interval $(-1,1)$. 
Clearly there exist $(W+1)\times(W+1)$ matrices $E,G$ and $F,H$ such that
\begin{equation}\label{eq5.35}
\mathcal{E}(v) = b^T E b\:,\quad\mathcal{G}(v) = d^T G d
\end{equation}
and
\begin{equation}\label{eq5.36}
\mathcal{F}(v) = b^T F b\:,\quad\mathcal{H}(v) = d^T H d.
\end{equation}
Here the matrices $E,G$ and $F,H$ are symmetric and $F,H$ are positive definite.
\newline
Hence there exist $W+1$ eigenvalues $0 \le \mu_0 \le \mu_1\le\cdots\le \mu_W$ and $W+1$
eigenvectors $b_0, b_1,\ldots,b_W$ of the symmetric eigenvalue problem
\begin{equation}\label{eq5.37}
(E-\mu F)b=0\;.
\end{equation}
Here
$$(E-\mu_i F)b_i=0\;.$$
Similarly, there exist $W+1$ eigenvalues $0 \le \nu_0 \le \nu_1\le\cdots\le\nu_W$ and
$W+1$ eigenvectors $d_0, d_1,\ldots,d_W$ of the symmetric eigenvalue problem
\begin{align}\label{eq5.38}
(G-\nu H)d=0\;.
\end{align}
Here
\[(G-\nu_i H)d_i=0\;.\]
The eigenvectors $b_i$ and $c_i$ are normalized so that
\begin{subequations}\label{eq5.39}
\begin{equation}\label{eq5.39a}
b_i^T F b_j= \delta^i_j,\:\:\mbox{and}\:\:
d_i^T H d_j= \delta^i_j \;.
\end{equation}
Moreover the relations
\begin{align}\label{eq5.39b}
b_i^T E b_j=\mu_i \delta^i_j\:\:\mbox{and}\:\:
d_i^T G d_j=\nu_i \delta^i_j
\end{align}
\end{subequations}
hold. Let $b_i = (b_{i,0}, b_{i,1},\ldots, b_{i,W})$ and $d_i = (d_{i,0},d_{i,1},\ldots,d_{i,W})$.
We now define the polynomials
\begin{subequations}\label{eq5.40}
\begin{align}
\phi_i(x)=\sum_{m=0}^{W}b_{i,m}L_m(x)\:\:\mbox{for}\:\:0\le i\le W\:,\\
\phi_j(y)=\sum_{m=0}^{W}b_{j,m}L_m(y)\:\:\mbox{for}\:\:0\le j\le W\:,\\
\theta_k(z)=\sum_{m=0}^{W}d_{k,m}L_m(z)\:\:\mbox{for}\:\:0\le k\le W\:.
\end{align}
\end{subequations}
Next, let $\chi_{i,j,k}$ denote the polynomial
\begin{equation}\label{eq5.41}
\chi_{i,j,k}(x,y,z)=\phi_i(x)\phi_j(y)\theta_k(z)
\end{equation}
for\quad $0 \leq i \le W,\;0 \le j\le W,\;0 \le k\le W $.
\newline
Note that ${\{\chi_{i,j,k}(x,y,z)\}}_{i,j,k}$ is the tensor product of the polynomials
$\phi_i(x)$, $\phi_j(y)$ and $\theta_k(z)$. The eigenvalue $\sigma_{i,j,k}$ corresponding to
the eigenvector $\chi_{i,j,k}$ is given by the relation
\begin{equation}\label{eq5.42}
\sigma_{i,j,k}=\mu_i+\mu_j+\nu_k+1\;.
\end{equation}
Let $\widetilde{\mathcal{C}}^{\star}(f,g)$ be the bilinear form induced by the quadratic
form $\mathcal{C}^{\star}(u)$. Then
\begin{align}
\widetilde{\mathcal{C}}^{\star}(f,g)&=\int_{-1}^{1}\int_{-1}^{1}\int_{-1}^{1}(f_{xx}f_{xx}
+f_{yy}g_{yy}+\eta^{4}f_{zz}g_{zz}+f_{x}g_{x}+f_{y}g_{y}\notag\\
&\hspace{2.7cm}+\eta^{2}f_{z}g_{z}+fg)\:dx dy dz\:. \notag
\end{align}
It is easy to show that
\begin{align}
\widetilde{\mathcal{C}}^{\star}(\chi_{i,j,k},\chi_{l,m,n})&=(\mu_i+\mu_j+\nu_k+1)
\delta_{l}^{i}\delta_{m}^{j}\delta_{n}^{k} \notag\\
&=\sigma_{i,j,k}\delta_{l}^{i}\delta_{m}^{j}\delta_{n}^{k}\:. \notag
\end{align}
Hence the eigenvectors of the quadratic form $\mathcal{C}^{\star}(u)$ are $\{\chi_{i,j,k}\}_{i,j,k}$
and the eigenvalues are $\{\sigma_{i,j,k}\}_{i,j,k}$. Thus, the quadratic form $\mathcal{C}^{\star}({u})$
can be diagonalized in the basis ${\{\chi_{i,j,k}\}}_{i,j,k}$. Therefore, the matrix corresponding to the
quadratic form $\mathcal{C}^{\star}(u)$ is easy to invert.

Now proceeding as earlier, it can be shown that the quadratic forms $\mathcal{B}^{\star}(u)$
and $\mathcal{C}^{\star}(u)$ are spectrally equivalent. Moreover, the quadratic form
$\mathcal{C}^{\star}(u)$ can be inverted in $O(W^{4})$ operations.

Thus, it follows that there exist quadratic forms
$\mathcal{C}_{vertex-edges}(u)$ and $\mathcal{C}_{edges}(u)$ which
are spectrally equivalent to $\mathcal{B}_{vertex-edges}(u)$ and
$\mathcal{B}_{edges}(u)$ respectively and which can be diagonalized
using separation of variables technique.


\section{Parallelization Techniques}
In minimizing the functional $\mathcal R^{N,W}(\{\mathcal F_v\})$ we seek a solution which
minimizes the sum of weighted norms of the residuals in the partial differential equation and
a fractional Sobolev norm of the residuals in the boundary conditions and enforce continuity
by adding a term which measures the sum of squares of the jumps in the function and its
derivatives at inter-element boundaries in appropriate Sobolev norms, suitably weighted in
various regions of the polyhedron.

In order to obtain a solution using PCGM we must be able to compute residuals in the normal
equations inexpensively. Since we are minimizing $\mathcal R^{N,W}(\{\mathcal F_v\})$ over
all $\{\mathcal F_v\}\in \mathcal S^{N,W}$ (space of spectral element functions) we have
\[\mathcal R^{N,W}(U+\epsilon V)=\mathcal R^{N,W}(U)+2\epsilon V^{t}(XU-YG)+O(\epsilon^2)\]
for all $V$, where $U$ is a vector assembled from the values of
\[\left\{\left\{u_l^{r}(\lambda)\right\}_{l=1}^{N_r},\left\{u_l^{v}(x^{v})\right\}_{l=1}^{N_v},
\left\{u_l^{v-e}(x^{v-e})\right\}_{l=1}^{N_{v-e}},\left\{u_l^{e}(x^{e})\right\}_{l=1}^{N_e}\right\}.\]
$V$ is a vector similarly assembled and $G$ is assembled from the data. Here $X$ and $Y$ denote
matrices. Thus we have to solve $XU-YG=0$ and so we must be able to compute $XV-YG$ economically
during the iterative process. We will now explain in brief how this can be done. The idea is very
similar to the case of two dimensional problems so we refer the reader to~\cite{SKT1} for details.

The above minimization amounts to an overdetermined system of
equations consisting of collocating the residuals in the partial
differential equation, the residuals in the boundary conditions and
jumps in the function and its derivatives at inter-element
boundaries at an over determined set of collocation points, weighted
suitably. In fact, we collocate the partial differential equation on
a finer grid of Gauss-Lobatto-Legendre (GLL) points and then we
apply the adjoint differential operator to these residuals and
project these values back to the original grid. Such a treatment
obviously involves integration by parts and hence leads to
evaluation of terms at the boundaries. These boundary terms can be
evaluated be a collocation procedure and the other boundary terms
corresponding to jump terms at the inter-element boundaries can be
easily calculated.

\subsection{Integrals on the element domain}
We first show how to compute the integrals on the element domain. Since each element
is mapped onto the master cube $Q$=$(-1,1)^{3}=(-1,1)\times(-1,1)\times(-1,1)$, so we
consider the computations only on $Q$.
\newline
Consider the integral
\begin{equation}\label{eq5.44}
{\Arrowvert{({\mathcal{L}})}^a {{u}}-{\widehat{F}}\Arrowvert}^2_{Q}\,.
\end{equation}
Its variation is
\begin{equation}\label{eq5.45}
\int_{Q}{(\mathcal{L})}^{a}{v}\,({(\mathcal{L})}^{a}{u}-\widehat{F})\,d\xi_1\,d\xi_2\,d\xi_3.
\end{equation}
Here
\begin{align}
({\mathcal{L}})^{a}{w}
&=\widehat{A}w_{\xi_1\xi_1}+\widehat{B}w_{\xi_2\xi_2}+\widehat{C}w_{\xi_3\xi_3}
+\widehat{D}w_{\xi_1\xi_2}+\widehat{E}w_{\xi_2\xi_3}+\widehat{F}w_{\xi_3\xi_1}\notag\\
&+\widehat{G}w_{\xi_1}+\widehat{H}w_{\xi_2}+\widehat{I}w_{\xi_3}+\widehat{J}{w},\notag
\end{align}
where the coefficients $\widehat{A}, \widehat{B}, \widehat{C}$ etc. are polynomials of degree
$N$ and $\widehat{F}$ is a filtered representation of $F$.
\newline
Let,
\[{u}(\xi_1,\xi_2,\xi_3) = \sum_{i=0}^{2N}\sum_{j=0}^{2N}\sum_{k=0}^{2N} \alpha_{i,j,k}
{\xi_1}^i{\xi_2}^j{\xi_3}^k\]
and
\[{v}(\xi_1,\xi_2,\xi_3) = \sum_{i=0}^{2N}\sum_{j=0}^{2N}\sum_{k=0}^{2N} \beta_{i,j,k}
{\xi_1}^i{\xi_2}^j{\xi_3}^k.\]
The differential operator $({\mathcal{L}})^{a}$ is obtained from the differential operator
\begin{align}
L{w}&={A}{w}_{\xi_1\xi_1}+{B}{w}_{\xi_2\xi_2}+{C}{w}_{\xi_3\xi_3}+{D}{w}_{\xi_1\xi_2}
+{E}{w}_{\xi_2\xi_3}+{F}{w}_{\xi_3\xi_1} \notag \\
&+{G}{w}_{\xi_1}+{H}{w}_{\xi_2}+{I}{w}_{\xi_3}+{J}{w}, \notag
\end{align}
with analytic coefficients $A,B,C$ etc. We choose the coefficients $\widehat{A}, \widehat{B},
\widehat{C}$ etc. so that they are the orthogonal projections of $A,B,C$ etc. into the space
of polynomials of degree $N$ with respect to the usual inner product in $H^2(Q)$.
\newline
Let $(\mathcal{L}^{a})^{T}$ denote the formal adjoint of the differential operator
${(\mathcal{L})}^{a}$.
\newline
Then
\begin{align}
({\mathcal{L}}^{a})^{T}{w}
&=\left(\widehat{A}{w}\right)_{\xi_1\xi_1}+\left(\widehat{B}{w}\right)_{\xi_2\xi_2}
+\left(\widehat{C}{w}\right)_{\xi_3\xi_3}+\left(\widehat{D}{w}\right)_{\xi_1\xi_2}
+\left(\widehat{E}{w}\right)_{\xi_2\xi_3} \notag \\
&+\left(\widehat{F}{w}\right)_{\xi_3\xi_1}-\left(\widehat{G}{w}\right)_{\xi_1}
-\left(\widehat{H}{w}\right)_{\xi_2}-\left(\widehat{I}{w}\right)_{\xi_3}+\widehat{J}{w},
\notag
\end{align}
Let $\mathfrak{z}$ denote $(\mathcal{L})^{a}{u}-\widehat{F}$. Integrating (\ref{eq5.45}) by
parts we obtain
\begin{align}\label{eq5.46}
&\int_{Q}{(\mathcal{L})}^{a}{v}\mathfrak{z}\,d\xi_1\,d\xi_2\,d\xi_3\notag\\
&=\int_{Q}{v}(\mathcal{L}^{a})^{T}\mathfrak{z}\,d\xi_1\,d\xi_2\,d\xi_3\notag\\
&+\int_{-1}^{1}\int_{-1}^{1}\left((\widehat{A}v_{\xi_1}+\widehat{G}v)\mathfrak{z}
-v\left((\widehat{A}\mathfrak{z})_{\xi_1}+(\widehat{F}\mathfrak{z})_{\xi_3}
+(\widehat{D}\mathfrak{z})_{\xi_2}\right)\right)(1,\xi_2,\xi_3)d\xi_2\,d\xi_3\notag\\
&-\int_{-1}^{1}\int_{-1}^{1}\left((\widehat{A}v_{\xi_1}+\widehat{D}v_{\xi_2}+\widehat{G}v)
\mathfrak{z}-v\left((\widehat{A}\mathfrak{z})_{\xi_1}+(\widehat{F}\mathfrak{z})_{\xi_3}\right)
\right)(-1,\xi_2,\xi_3)d\xi_2\,d\xi_3\notag\\
&+\int_{-1}^{1}v(\widehat{D}\mathfrak{z})(1,1,\xi_3)d\xi_3-\int_{-1}^{1}v
(\widehat{D}\mathfrak{z})(1,-1,\xi_3)d\xi_3\notag\\
&+\int_{-1}^{1}\int_{-1}^{1}\left((\widehat{B}v_{\xi_2}+\widehat{H}v)\mathfrak{z}
-v\left((\widehat{B}\mathfrak{z})_{\xi_2}+(\widehat{D}\mathfrak{z})_{\xi_1}
+(\widehat{E}\mathfrak{z})_{\xi_3}\right)\right)(\xi_1,1,\xi_3)d\xi_1\,d\xi_3\notag\\
&-\int_{-1}^{1}\int_{-1}^{1}\left((\widehat{B}v_{\xi_2}+\widehat{E}v_{\xi_3}+\widehat{H}v)
\mathfrak{z}-v\left((\widehat{B}\mathfrak{z})_{\xi_2}+(\widehat{D}\mathfrak{z})
_{\xi_1}\right)\right)(\xi_1,-1,\xi_3)d\xi_1\,d\xi_3\notag\\
&+\int_{-1}^{1}v(\widehat{E}\mathfrak{z})(\xi_1,1,1)d\xi_1-\int_{-1}^{1}v(\widehat{E}
\mathfrak{z})(\xi_1,1,-1)d\xi_1\notag\\
&+\int_{-1}^{1}\int_{-1}^{1}\left((\widehat{C}v_{\xi_3}+\widehat{I}v)\mathfrak{z}
-v\left((\widehat{C}\mathfrak{z})_{\xi_3}+(\widehat{E}\mathfrak{z})_{\xi_2}+(\widehat{F}
\mathfrak{z})_{\xi_1}\right)\right)(\xi_1,\xi_2,1)d\xi_1\,d\xi_2\notag\\
&-\int_{-1}^{1}\int_{-1}^{1}\left((\widehat{C}v_{\xi_3}+{F}v_{\xi_1}+{I}v)
\mathfrak{z}-v\left((\widehat{C}\mathfrak{z})_{\xi_3}+(\widehat{E}\mathfrak{z})
_{\xi_2}\right)\right)(\xi_1,\xi_2,-1)d\xi_1\,d\xi_2\notag\\
&+\int_{-1}^{1}v(\widehat{F}\mathfrak{z})(1,\xi_2,1)d\xi_2-\int_{-1}^{1}v
(\widehat{F}\mathfrak{z})(-1,\xi_2,1)d\xi_2
\end{align}
Now $u,v$
and $\widehat{A}$, $\widehat{B}$, $\widehat{C}$ etc. are polynomials of degree $N$ in each of
the variables $\xi_1$, $\xi_2$ and $\xi_3$. Moreover $\widehat{F}$ is a polynomial of degree
$2N$. Hence all the integrands are polynomials of degree $4N$ in each of the variables $\xi_1$,
$\xi_2$ and $\xi_3$ and so the integral may be exactly evaluated by the Gauss-Lobatto-Legendre
(GLL) quadrature formula with $2N+1$ points. Let $\xi_{1,0}^{2N},\ldots,\xi_{1,2N}^{2N}$,
$\xi_{2,0}^{2N},\ldots,\xi_{2,2N}^{2N}$ and $\xi_{3,0}^{2N},\ldots,\xi_{3,2N}^{2N}$ represent
the quadrature points in $\xi_1$, $\xi_2$ and $\xi_3$ directions respectively and $w_{0}^{2N},
\ldots,w_{2N}^{2N}$ the corresponding weights.
\begin{prop}
Let the matrix $D^{2N}=d_{i,j}^{2N}$ denote the differentiation matrix. Then
\begin{equation}\label{eq5.47}
\frac{dl}{d\xi}\left(\xi_{i}^{2N}\right)=\sum_{j=0}^{2N}d_{i,j}^{2N}l\left(\xi_{j}^{2N}\right)
\end{equation}
if $l$ is a polynomial of degree less than or equal to $2N$.
\end{prop}
Using the GLL quadrature rule we obtain
\begin{align}
&\int_{Q}(\mathcal{L})^{a}{v}((\mathcal{L})^{a}{u}-\tilde{F})d\xi_1 d\xi_2 d\xi_3
\notag\\
= &\sum_{i=0}^{2N}\sum_{j=0}^{2N}\sum_{k=0}^{2N}{v}\left(\xi_{1,i}^{2N},\xi_{2,j}^{2N}
,\xi_{3,k}^{2N}\right)\left(w_{i}^{2N}w_{j}^{2N}w_{k}^{2N}{L}^{T}\mathfrak{z}
\left(\xi_{1,i}^{2N},\xi_{2,j}^{2N},\xi_{3,k}^{2N}\right)\right)   \notag \\
& +\sum_{i=0}^{2N}\sum_{j=0}^{2N}\sum_{k=0}^{2N}{v}\left(\xi_{1,i}^{2N},\xi_{2,j}^{2N},\xi_{3,k}^{2N}
\right)\left(w_{j}^{2N}w_{k}^{2N}d_{2N,i}^{2N}{A}(1,\xi_{2,j}^{2N},\xi_{3,k}^{2N})
\mathfrak{z}\left(1,\xi_{2,j}^{2N},\xi_{3,k}^{2N}\right)\right) \notag \\
& +\sum_{j=0}^{2N}\sum_{k=0}^{2N}{v}\left(1,\xi_{2,j}^{2N},\xi_{3,k}^{2N}\right)
\left(w_{j}^{2N}w_{k}^{2N}\left({G}\mathfrak{z}-({A}\mathfrak{z})_{\xi_1}
-({F}\mathfrak{z})_{\xi_3}-({D}\mathfrak{z})_{\xi_2}\right)\left(1,\xi_{2,j}^{2N},
\xi_{3,k}^{2N}\right)\right) \notag \\
& -\sum_{i=0}^{2N}\sum_{j=0}^{2N}\sum_{k=0}^{2N}{v}\left(\xi_{1,i}^{2N},\xi_{2,j}^{2N},\xi_{3,k}^{2N}
\right)\left(w_{j}^{2N}w_{k}^{2N}\left(d_{0,i}^{2N}{A}(-1,\xi_{2,j}^{2N},\xi_{3,k}^{2N})
\mathfrak{z}(-1,\xi_{2,j}^{2N},\xi_{3,k}^{2N})\right.\right. \notag \\
&\hspace{2.4cm}\left.\left.+d_{j,i}^{2N}{D}(-1,\xi_{2,j}^{2N},\xi_{3,k}^{2N})
\mathfrak{z}(-1,\xi_{2,j}^{2N},\xi_{3,k}^{2N})\right)\right) \notag \\
& -\sum_{j=0}^{2N}\sum_{k=0}^{2N}{v}\left(-1,\xi_{2,j}^{2N},\xi_{3,k}^{2N}\right)
\left(w_{j}^{2N}w_{k}^{2N}\left({G}\mathfrak{z}-({A}\mathfrak{z})_{\xi_1}
-({F}\mathfrak{z})_{\xi_3}\right)\left(-1,\xi_{2,j}^{2N},\xi_{3,k}^{2N}\right)\right) \notag \\
& +\sum_{i=0}^{2N}\sum_{j=0}^{2N}\sum_{k=0}^{2N}{v}\left(\xi_{1,i}^{2N},\xi_{2,j}^{2N},\xi_{3,k}^{2N}
\right)\left(w_{i}^{2N}w_{k}^{2N}d_{2N,j}^{2N}{B}(\xi_{1,i}^{2N},1,\xi_{3,k}^{2N})
\mathfrak{z}\left(\xi_{1,i}^{2N},1,\xi_{3,k}^{2N}\right)\right) \notag\\ 
& +\sum_{i=0}^{2N}\sum_{k=0}^{2N}{v}\left(\xi_{1,i}^{2N},1,\xi_{3,k}^{2N}\right)
\left(w_{i}^{2N}w_{k}^{2N}\left({H}\mathfrak{z}-({B}\mathfrak{z})_{\xi_2}
-({D}\mathfrak{z})_{\xi_1}-({E}\mathfrak{z})_{\xi_3}\right)\left(\xi_{1,i}^{2N},1,
\xi_{3,j}^{2N}\right)\right) \notag \\
& -\sum_{i=0}^{2N}\sum_{j=0}^{2N}\sum_{k=0}^{2N}{v}\left(\xi_{1,i}^{2N},\xi_{2,j}^{2N},\xi_{3,k}^{2N}
\right)\left(w_{i}^{2N}w_{k}^{2N}\left(d_{0,j}^{2N}{B}(\xi_{1,i}^{2N},-1,\xi_{3,k}^{2N})
\mathfrak{z}(\xi_{1,i}^{2N},-1,\xi_{3,k}^{2N})\right.\right. \notag \\
&\hspace{2.4cm}\left.\left.+d_{k,j}^{2N}{E}(\xi_{1,i}^{2N},-1,\xi_{3,k}^{2N})
\mathfrak{z}(\xi_{1,i}^{2N},-1,\xi_{3,k}^{2N})\right)\right) \notag \\
& -\sum_{i=0}^{2N}\sum_{k=0}^{2N}{v}\left(\xi_{1,i}^{2N},-1,\xi_{3,k}^{2N}\right)
\left(w_{i}^{2N}w_{k}^{2N}\left({H}\mathfrak{z}-({B}\mathfrak{z})_{\xi_2}
-({D}\mathfrak{z})_{\xi_1}\right)(\xi_{1,i}^{2N},-1,\xi_{3,k}^{2N})\right) \notag 
\end{align}
\begin{align}\label{eq5.48}
& +\sum_{i=0}^{2N}\sum_{j=0}^{2N}\sum_{k=0}^{2N}{v}\left(\xi_{1,i}^{2N},\xi_{2,j}^{2N},\xi_{3,k}^{2N}
\right)\left(w_{i}^{2N}w_{j}^{2N}d_{2N,k}^{2N}{C}(\xi_{1,i}^{2N},\xi_{2,j}^{2N},1)
\mathfrak{z}\left(\xi_{1,i}^{2N},\xi_{2,j}^{2N},1\right)\right) \notag \\
& +\sum_{i=0}^{2N}\sum_{j=0}^{2N}{v}\left(\xi_{1,i}^{2N},\xi_{2,j}^{2N},1\right)
\left(w_{i}^{2N}w_{j}^{2N}\left({I}\mathfrak{z}-({C}\mathfrak{z})_{\xi_3}
-({E}\mathfrak{z})_{\xi_2}-({F}\mathfrak{z})_{\xi_1}\right)\left(\xi_{1,i}^{2N},
\xi_{2,j}^{2N},1\right)\right) \notag \\
& -\sum_{i=0}^{2N}\sum_{j=0}^{2N}\sum_{k=0}^{2N}{v}\left(\xi_{1,i}^{2N},\xi_{2,j}^{2N},\xi_{3,k}^{2N}
\right)\left(w_{i}^{2N}w_{j}^{2N}\left(d_{0,k}^{2N}{C}(\xi_{1,i}^{2N},\xi_{2,j}^{2N},-1)
\mathfrak{z}(\xi_{1,i}^{2N},\xi_{2,j}^{2N},-1)\right.\right. \notag \\
&\hspace{2.4cm}+\left.\left.d_{i,k}^{2N}{F}(\xi_{1,i}^{2N},\xi_{2,j}^{2N},-1)
\mathfrak{z}(\xi_{1,i}^{2N},\xi_{2,j}^{2N},-1)\right)\right) \notag \\
& -\sum_{i=0}^{2N}\sum_{j=0}^{2N}{v}\left(\xi_{1,i}^{2N},\xi_{2,j}^{2N},-1\right)
\left(w_{i}^{2N}w_{j}^{2N}\left({I}\mathfrak{z}-({C}\mathfrak{z})_{\xi_3}
-({E}\mathfrak{z})_{\xi_2}\right)(\xi_{1,i}^{2N},\xi_{2,j}^{2N},-1)\right) \notag \\
&+ \sum_{k=0}^{2N}{v}(1,1,\xi_{3,k}^{2N})w_{k}^{2N}{D}(1,1,\xi_{3,k}^{2N})\mathfrak{z}
(1,1,\xi_{3,k}^{2N})-\sum_{k=0}^{2N}{v}(1,-1,\xi_{3,k}^{2N})w_{k}^{2N}{D}(1,-1,\xi_{3,k}^{2N})\notag \\
& \times \mathfrak{z}(1,-1,\xi_{3,k}^{2N}) + \sum_{i=0}^{2N}{v}(\xi_{1,i}^{2N},1,1)w_{i}^{2N}
{E}(\xi_{1,i}^{2N},1,1)\mathfrak{z}(\xi_{1,i}^{2N},1,1)\notag \\
&-\sum_{i=0}^{2N}{v}(\xi_{1,i}^{2N},1,-1)w_{i}^{2N}{E}(\xi_{1,i}^{2N},1,-1)
\mathfrak{z}(\xi_{1,i}^{2N},1,-1) \notag \\
&+ \sum_{j=0}^{2N}{v}(1,\xi_{2,j}^{2N},1)w_{j}^{2N}{F}(1,\xi_{2,j}^{2N},1)\mathfrak{z}
(1,\xi_{2,j}^{2N},1)-\sum_{k=0}^{2N}{v}(-1,\xi_{2,j}^{2N},1)w_{j}^{2N}{F}(-1,\xi_{2,j}^{2N},1)
 \notag \\
& \times \mathfrak{z}(-1,\xi_{2,j}^{2N},1)
\end{align}
Note that we have used unfiltered coefficients $A,B,C$ etc. in the right hand side.
\begin{rem}
The stability estimate is stated in terms of unfiltered coefficients of the differential operator,
and we use unfiltered coefficients in our computations. Of course, in writing the above we commit
an error. It can be argued as in \cite{DT,DTK1,SKT1,SKT2} that this error is spectrally small. In
fact, if we assume that the boundary of the domain $\Omega$ is analytic, the coefficients of the
differential operator are analytic and the data is analytic then the error committed is
exponentially small in $N$. Hence we do not need to filter the coefficients of the differential
and boundary operators or the data in any of our computations.
\end{rem}
Rewrite ${u}\left(\xi_{1,i}^N,\xi_{2,j}^N,\xi_{3,k}^N \right)$ by arranging them in lexicographic
order and denote
$$U_{(N+1)^2k+(N+1)j+i}^{N}={u}\left(\xi_{1,i}^N,\xi_{2,j}^N,\xi_{3,k}^N\right)\qquad \text{for}
\;0\leq i,j,k \le N\:,$$
and let
\begin{equation}
U_{(2N+1)^2k+(2N+1)j+i}^{2N}={u}\left(\xi_{1,i}^{2N},\xi_{2,j}^{2N},\xi_{3,k}^{2N}\right)
\qquad \text{for}\;0\leq i,j,k \le 2N.\notag \\
\end{equation}
Similarly
\begin{equation}
Z_{(2N+1)^2k+(2N+1)j+i}^{2N}=({L}{u}-F)\left(\xi_{1,i}^{2N},\xi_{2,j}^{2N},
\xi_{3,k}^{2N}\right),\notag
\end{equation}
for $0\leq i,j,k \le 2N$, are arranged in lexicographic order.
\newline
Then we may write
$$
\int_{Q}(\mathcal{L})^{a}{v}\,((\mathcal{L})^{a}{u}-\widehat{F})\, d\xi_1 d\xi_2 d\xi_3=
\left(V^{2N}\right)^{T}RZ^{2N},
$$
where $R$ is a matrix such that $RZ^{2N}$ is easily computed.

\subsection{Integrals on the boundary of the elements}
We now show how to evaluate the boundary terms. For this we have to examine the norm
$H^{1/2}(E)$, here $E$ denote either $S$ or $T$.
Now
\begin{align}
\|l\|_{1/2, E}^2 = \|l\|_{0,E}^2 + \int_E\int_E\frac{\left(l(\xi_1,\xi_2)-l(\xi_1^{\prime},
\xi_2^{\prime})\right)^2}{\left((\xi_1-\xi_1^{\prime})^2+(\xi_2-\xi_2^{\prime})^2\right)^{3/2}}
d\xi_1\:d\xi_2\:d\xi_1^{\prime}\:d\xi_2^{\prime}\:.\notag
\end{align}
\newline
However if $E$ is $S$, which is always the case, then we prefer to use the equivalent norm~\cite{LM}
\begin{align}
\|l\|_{1/2, S}^2 = \|l\|_{0,S}^2 + \int_{-1}^{1}\int_{-1}^{1}\int_{-1}^{1}\frac{(l(\xi_1,\xi_2)
-l(\xi_1^{\prime},\xi_2))^2}{(\xi_1-\xi_1^{\prime})^{2}}\,d\xi_1 d\xi_1^{\prime}d\xi_2 \notag\\
+\int_{-1}^{1}\int_{-1}^{1}\int_{-1}^{1}\frac{(l(\xi_1,\xi_2)-l(\xi_1,\xi_2^{\prime}))^2}
{(\xi_2-\xi_2^{\prime})^{2}}\,d\xi_2 d\xi_2^{\prime}d\xi_1\,.\notag
\end{align}
Let $l(\xi_1,\xi_2)$ be a polynomial of degree less than or equal to $2N$. Then $\frac{(l(\xi_1,
\xi_2)-l(\xi_1^{\prime},\xi_2))}{(\xi_1-\xi_1^{\prime})}$ is a polynomial of degree less than or
equal to $2N$ in $\xi_1$ and $\xi_1^{\prime}$. Now using the GLL quadrature rule we can write
\begin{align}
\|l\|_{1/2, S}^2 &=\sum_{i=0}^{2N}\sum_{j=0}^{2N}w_{i}^{2N}w_{j}^{2N}l^2(\xi_{1,i}^{2N},
\xi_{2,j}^{2N})\notag\\
&+ \sum_{j=0}^{2N}\sum_{{i=0},{i^{\prime}\neq i}}^{2N}\sum_{i^{\prime}=0}^{2N}w_{j}^{2N}w_{i}^{2N}
w_{i^{\prime}}^{2N}\left(\frac{{l(\xi_{1,i}^{2N},\xi_{2,j}^{2N})-l(\xi_{1,i^{\prime}}^{2N},
\xi_{2,j}^{2N})}}{({\xi_{1,i}^{2N}-\xi_{1,i^{\prime}}^{2N}})}\right)^{2} \notag\\
&+ \sum_{i=0}^{2N}\sum_{{j=0},{j^{\prime}\neq j}}^{2N}\sum_{j^{\prime}=0}^{2N}w_{i}^{2N}w_{j}^{2N}
w_{j^{\prime}}^{2N}\left(\frac{{l(\xi_{1,i}^{2N},\xi_{2,j}^{2N})-l(\xi_{1,i}^{2N},
\xi_{2,j^{\prime}}^{2N})}}{({\xi_{2,j}^{2N}-\xi_{2,j^{\prime}}^{2N}})}\right)^{2} \notag\\
&+ \sum_{i=0}^{2N}\sum_{j=0}^{2N}(w_{i}^{2N})^{2}w_{j}^{2N}\left(\sum_{k=0}^{2N}d_{i,k}
l(\xi_{1,k}^{2N},\xi_{2,j}^{2N})\right)^{2} \notag\\
&+ \sum_{j=0}^{2N}\sum_{i=0}^{2N}(w_{j}^{2N})^{2}w_{i}^{2N}\left(\sum_{k=0}^{2N}d_{j,k}
l(\xi_{1,i}^{2N},\xi_{2,k}^{2N})\right)^{2} \notag
\end{align}
Thus there is a symmetric positive definite matrix $H^{2N}$ such that
\begin{equation}\label{eq5.49}
\Arrowvert l \Arrowvert_{1/2, S}^2=\sum_{i=0}^{2N}\sum_{j=0}^{2N}\sum_{i^{\prime}=0}^{2N}
\sum_{j^{\prime}=0}^{2N}l(\xi_{i}^{2N},\eta_{j}^{2N})H_{i,j,i^{\prime},j^{\prime}}^{2N}
l(\xi_{i^{\prime}}^{2N},\eta_{j^{\prime}}^{2N})
\end{equation}
A typical boundary term will be of the form
\begin{align}
&{\Arrowvert ({\widehat{P}}{u}_{\xi_1}+{\widehat{Q}}{u}_{\xi_2}+{\widehat{R}}{u}_{\xi_3})
(\xi_1,\xi_2,1)-{\widehat{g}(\xi_1,\xi_2)}\Arrowvert}_{1/2,S} \notag
\end{align}
and its variation is given by
\begin{align}\label{eq5.50}
\sum_{i=0}^{2N}\sum_{j=0}^{2N}\sum_{i^{\prime}=0}^{2N}\sum_{j^{\prime}=0}^{2N}
({\widehat{P}}{v}_{\xi_1}+{\widehat{Q}}{v}_{\xi_2}+{\widehat{R}}{v}_{\xi_3})
(\xi_{1,i}^{2N},\xi_{2,j}^{2N},1)H_{i,j,i^{\prime},j^{\prime}}^{2N} \notag\\
\times(({\widehat{P}}{u}_{\xi_1}+{\widehat{Q}}{u}_{\xi_2}+{\widehat{R}}{u}_{\xi_3})
(\xi_{1,i^{\prime}}^{2N},\xi_{2,j^{\prime}}^{2N},1)-\widehat{g}(\xi_{1,i^{\prime}}^{2N},
\xi_{2,j^{\prime}}^{2N}))\:.
\end{align}
Let
\begin{equation}
\sigma_{i,j}^{2N}=\sum_{i^{\prime}=0}^{2N}\sum_{j^{\prime}=0}^{2N}H_{i,j,i^{\prime},
j^{\prime}}^{2N}(\left({\widehat{P}}{u}_{\xi_1}+{\widehat{Q}}{u}_{\xi_2}+{\widehat{R}}
{u}_{\xi_3}\right)(\xi_{1,i^{\prime}}^{2N},\xi_{2,j^{\prime}}^{2N},1)
-\widehat{g}(\xi_{1,i^{\prime}}^{2N},\xi_{2,j^{\prime}}^{2N})).\notag
\end{equation}
Then
\begin{align}
&\sum_{i=0}^{2N}\sum_{j=0}^{2N}({\widehat{P}}{v}_{\xi_1}+{\widehat{Q}}{v}_{\xi_2}
+{\widehat{R}}{v}_{\xi_3})(\xi_{1,i}^{2N},\xi_{2,j}^{2N},1)\sigma_{i,j}^{2N}\notag\\
&=\sum_{i=0}^{2N}\sum_{i^{\prime}=0}^{2N}\sum_{j=0}^{2N}{v}(\xi_{1,i^{\prime}}^{2N},\xi_{2,j}^{2N},1)
d_{i,i^{\prime}}^{2N}\left({\widehat{P}}(\xi_{1,i}^{2N},\xi_{2,j}^{2N},1)\sigma_{i,j}^{2N}\right)\notag\\
&+\sum_{i=0}^{2N}\sum_{j=0}^{2N}\sum_{j^{\prime}=0}^{2N}{v}(\xi_{1,i}^{2N},\xi_{2,j^{\prime}}^{2N},1)
d_{j,j^{\prime}}^{2N}\left({\widehat{Q}}(\xi_{1,i}^{2N},\xi_{2,j}^{2N},1)\sigma_{i,j}^{2N}\right).\notag\\
&+\sum_{i=0}^{2N}\sum_{j=0}^{2N}\sum_{k=0}^{2N}{v}(\xi_{1,i}^{2N},\xi_{2,j}^{2N},\xi_{3,k}^{2N})
d_{2N,k}^{2N}\left({\widehat{R}}(\xi_{1,i}^{2N},\xi_{2,j}^{2N},1)\sigma_{i,j}^{2N}\right).\notag
\end{align}
\begin{prop}
Let
$$
Y_{i,j}^{2N}=({P}{u}_{\xi_1}+{Q}{u}_{\xi_2}+{R}{u}_{\xi_3}-{g})(\xi_{1,i}^{2N},\xi_{2,j}^{2N},1)
,\quad X^{2N}=H^{2N}Y^{2N},
$$
and let ${\widehat{P}}$, ${\widehat{Q}}$ and ${\widehat{R}}$ be filtered representations of $P$,
$Q$ and $R$. Similarly $\widehat{g}$ is a filtered representation of the actual boundary data $g$.
\newline
Then we may write
\begin{align}
&\sum_{i=0}^{2N}\sum_{j=0}^{2N}\sum_{i^{\prime}=0}^{2N}\sum_{j^{\prime}=0}^{2N}
({\widehat{P}}{v}_{\xi_1}+{\widehat{Q}}{v}_{\xi_2}+{\widehat{R}}{v}_{\xi_3})
(\xi_{1,i}^{2N},\xi_{2,j}^{2N},1)H_{i,j,i^{\prime},j^{\prime}}^{2N} \notag\\
&\times(({\widehat{P}}{u}_{\xi_1}+{\widehat{Q}}{u}_{\xi_2}+{\widehat{R}}{u}_{\xi_3})
(\xi_{1,i^{\prime}}^{2N},\xi_{2,j^{\prime}}^{2N},1)-\widehat{g}(\xi_{1,i^{\prime}}^{2N},
\xi_{2,j^{\prime}}^{2N})) \notag\\
&=\sum_{i=0}^{2N}\sum_{i^{\prime}=0}^{2N}\sum_{j=0}^{2N}{v}(\xi_{1,i^{\prime}}^{2N},\xi_{2,j}^{2N},1)
d_{i,i^{\prime}}^{2N}\left({P}(\xi_{1,i}^{2N},\xi_{2,j}^{2N},1)X_{i,j}^{2N}\right)\notag\\
&+\sum_{i=0}^{2N}\sum_{j=0}^{2N}\sum_{j^{\prime}=0}^{2N}{v}(\xi_{1,i}^{2N},\xi_{2,j^{\prime}}^{2N},1)
d_{j,j^{\prime}}^{2N}\left({Q}(\xi_{1,i}^{2N},\xi_{2,j}^{2N},1)X_{i,j}^{2N}\right) \notag\\
&+\sum_{i=0}^{2N}\sum_{j=0}^{2N}\sum_{k=0}^{2N}{v}(\xi_{1,i}^{2N},\xi_{2,j}^{2N},\xi_{3,k}^{2N})
d_{2N,k}^{2N}\left({R}(\xi_{1,i}^{2N},\xi_{2,j}^{2N},1)X_{i,j}^{2N}\right)\notag\\
&=\left(V^{2N}\right)^{t}TX^{2N}. \notag
\end{align}
\end{prop}
Here $T$ is a $\left(2N+1\right)^{3}\times\left(2N+1\right)^{2}$ matrix and $TX^{2N}$
can be easily computed. In writing the above we are again committing an error and this error can be
shown to be exponentially small in $N$ once more. Hence there is no need to filter the coefficients
of the boundary operators or the data.
\newline
Adding all the terms we obtain
\begin{align}
&\int_{Q}(\mathcal{L})^{a}{v}\,((\mathcal{L})^{a}{u}-\widehat{F})\, d\xi_1 d\xi_2 d\xi_3
\notag \\
&+ \sum_{i=0}^{2N}\sum_{j=0}^{2N}\sum_{i^{\prime}=0}^{2N}\sum_{j^{\prime}=0}^{2N}
({\widehat{P}}{v}_{\xi_1}+{\widehat{Q}}{v}_{\xi_2}+{\widehat{R}}{v}_{\xi_3})
(\xi_{1,i}^{2N},\xi_{2,j}^{2N},1)H_{i,j,i^{\prime},j^{\prime}}^{2N} \notag\\
& \times(({\widehat{P}}{u}_{\xi_1}+{\widehat{Q}}{u}_{\xi_2}+{\widehat{R}}{u}_{\xi_3})
(\xi_{1,i^{\prime}}^{2N},\xi_{2,j^{\prime}}^{2N},1)-\widehat{g}(\xi_{1,i^{\prime}}^{2N},
\xi_{2,j^{\prime}}^{2N})) + \cdots \notag \\
&= \left(V^{2N}\right)^{t}\left(RZ^{2N}+TX^{2N}+\cdots\right)
=\left(V^{2N}\right)^{t}O^{2N} \notag
\end{align}
where $O^{2N}=RZ^{2N}+TX^{2N}+\cdots$ is a $\left(2N+1\right)^{3}$ vector which can be easily
computed. Now there exists a matrix $G^{N}$ such that $V^{2N}=G^{N}V^{N}$.
Hence
$$
\left(V^{2N}\right)^{t}O^{2N}=\left(V^{N}\right)^{T}\left(\left(G^{N}\right)^{T}O^{2N}\right).
$$
In~\cite{SKT1,SKT2} it has been shown how $\left(G^{N}\right)^{T}O^{2N}$ can be computed for
two dimensional problems without having to compute and store any mass and stiffness matrices.
Same idea extends here and so we just describe the steps involved in brief and refer the
reader to~\cite{SKT1,SKT2} for details.
\newline
Let $\gamma_{m}^{N}$ be the normalizing factors used in computing the \textit{discrete Legendre
transform} as
\begin{eqnarray*}
&  & \gamma_{m}^{N}=\left\{ \begin{array}{c}
(m+1/2)^{-1}\,\,\,\textrm{if}\,\, m<N,\\
2/N\hspace{1.8cm}\textrm{if}\,\, m=N\end{array}\right.\;.
\end{eqnarray*}
Let $\left\{O_{i,j,k}\right\} _{0\leq i,j,k\leq2N}$ be defined as $O_{i,j,k}=O_{k{(2N+1)}^2
+j(2N+1)+i}^{2N}.$ Now we perform the following set of operations.
\begin{enumerate}
\item Define $O_{i,j,k}\leftarrow O_{i,j,k}/w_{i}^{2N}w_{j}^{2N}w_{k}^{2N}.$
\item Calculate $\left\{ \Delta_{i,j,k}\right\} _{0\leq i,j,k\leq 2N}$ the Legendre
transform of $\left\{ O_{i,j,k}\right\} _{0\leq i,j,k\leq2N}.$ Define
$$
\Delta_{i,j,k}\leftarrow\gamma_{i}^{2N}\gamma_{j}^{2N}\gamma_{k}^{2N}\Delta_{i,j,k}.
$$
\item Calculate $\mu_{i,j,k}\leftarrow\Delta_{i,j,k}/\gamma_{i}^{N}\gamma_{j}^{N}\gamma_{k}^{N},
\:0\leq i,j,k\leq N.$
\item Compute $\varepsilon$, the \textit{inverse Legendre transform} of $\mu$. Define
$$
\varepsilon_{i,j,k}\leftarrow w_{i}^{N}w_{j}^{N}w_{k}^{N}\varepsilon_{i,j,k},\:0\leq i,j,k
\leq N.
$$
\item Define a vector $J$ of dimension $\left(N+1\right)^{3}$ as
$$
J_{k{(N+1)}^2+j(N+1)+i}=\varepsilon_{i,j,k}\;\mathrm{for}\;0\leq i,j,k\leq N.
$$
\end{enumerate}
Then $J=\left(G^{N}\right)^{t}O^{2N}$. Thus we see to compute $J$ we do not need to compute
and store any matrices such as the mass and stiffness matrices.
\newline
Clearly, $O(N^{4})$ operations are required to compute the residual vector on a parallel
computer. In order to compute the residual vector we need to communicate the values of the
function and its derivatives on the boundary of the element between neighbouring elements.
Moreover, when computing the two global scalars, needed to update the approximate solution
and the search direction during each step of PCGM, some scalars have to be exchanged among
all the processors. Thus we see that inter processor communication is small.

\chapcleardoublepage

\chapter{Numerical Results}
\section{Introduction}
In this chapter we present the results of the numerical simulations which have been
performed to validate the asymptotic error estimates and estimates of computational
complexity that we have obtained. First, we consider test problems on non-smooth domains
with smooth solutions (i.e. solutions without singularities) which include Laplace,
Poisson and Helmholtz problems with different types of homogeneous and non-homogeneous
boundary conditions (e.g. Dirichlet, mixed, Neumann and Robin). Next, we consider a
general elliptic equation with mixed boundary conditions. Numerical results for a
variable coefficient problem on $L$-shaped domain are also presented which confirm
the theoretical estimates obtained.

The method works for non self-adjoint problems too. Computational results for a non
self-adjoint problem with variable coefficients are provided.

In order to show the effectiveness of the method in dealing with elliptic problems
on non-smooth domains containing singularities we shall consider various examples of
Laplace and Poisson equations on a number of polyhedral domains containing all three
types of singularities namely, vertex, edge and vertex-edge singularities with
Dirichlet and mixed type of boundary conditions to include all possible cases of
singularities that may arise due to the presence of corners and edges. We show that
the error between the exact and the approximate solution is exponentially small.

The spectral element functions are nonconforming. The method we have described is a
least-squares collocation method and to obtain our approximate solution we solve the
normal equations using Preconditioned Conjugate Gradient Method (PCGM). The residuals
in the normal equations can be obtained without computing and storing mass and stiffness
matrices. A preconditioner is defined for the method which is a block diagonal matrix,
where each diagonal block corresponds to an element. It is also shown in Chapter $5$
that there exists a new preconditioner which is a diagonal preconditioner on each
element. Let $N$ denote the number of refinements in the geometrical mesh and $W$ denote
an upper bound on the polynomial degree. We shall assume that $N$ is proportional to $W$.
Then the condition number of the preconditioned system is $K_{N,W}$,
where $K_{N,W}=O((lnW)^2)$ for Dirichlet problems and $K_{N,W}=O(N^4)$ for mixed problems
with Neumann and Dirichlet boundary conditions, provided $W=O(e^{N^{\alpha}})$ for
$\alpha<1/2$.

As stated in previous chapters after obtaining our approximate solution consisting of
non-conforming spectral element functions we can make a correction to it so that the
corrected solution is conforming and is an exponentially accurate approximation to the
true solution in the $H^1$ norm over the whole domain. These corrections are
explained in the proof of Lemma \ref{lem3.3.1} in Appendix C.1.

We perform the PCGM until a stopping criterion on the relative norm of the residual
vector for the normal equations becoming less than $\epsilon$, a specified number,
is satisfied. Since we would need to perform
$O\left(\frac{\sqrt{\kappa}}{2}|\log\left(\frac{2}{\epsilon}\right)|\right)$ iterations
of the PCGM to satisfy the stopping criterion, we would need to perform $O(N(ln(N))$ and
$O(N^3)$ iterations of the PCGM for Dirichlet and mixed problems respectively to obtain
the approximate solution. Here $\kappa$ denotes condition number of the preconditioner.

We shall denote by $\Omega,\Omega^{(v)},\Omega^{(e)},\Omega^{(v-e)}$ etc. polyhedral
domains in $R^{3}$. Throughout this chapter, $(x_{1},x_{2},x_{3})$, $(\phi,\theta,\rho)$
and $(r,\theta,x_{3})$ will denote the Cartesian, the spherical and the cylindrical coordinates,
respectively of points in these domains. In case the solution is smooth we shall denote
the Cartesian coordinates by $(x,y,z)$.

Let $u_{appx.}$ be the spectral element solution obtained from the minimization
problem and $u_{ex.}$ be the exact solution. Then the relative error is defined as
\begin{align}
||E||_{rel}=\frac{||u_{appx.}-u_{ex.}||_{H^{1}}}{||u_{ex.}||_{H^{1}}}\,.\notag
\end{align}

The computations are carried out on a distributed memory parallel computer which consists
of V20Z and V40Z Sun Fire Servers. Each of the Sun Fire Servers is a 2 CPU 64 bit 2.4 GHZ
Oeptron AMD machine with 4GB RAM. The library used for inter processor communication is
Message Passing Interface (MPI).

\section{Test Problems in Regular Regions}
We start with the numerical results for various test problems on polyhedral domains on
which the solution is analytic. Let $\Omega = Q = {(-1,1)}^3$ denote the standard cube
in $\mathbb{R}^{3}$ with boundary $\partial \Omega$. Let $N$ denote the number of
refinements in each direction and $W$ denote the degree of the polynomials used for
approximation. In each of the examples that will follow we have used three different
meshes with uniform mesh size $h=2.0,1.0$ and $0.67$ in each direction which corresponds
to $N=1,2$ and $3$ respectively (Figure \ref{fig6.1}).
\begin{figure}[!ht]
\centering
\subfigure[]{
\includegraphics[scale = 0.55]{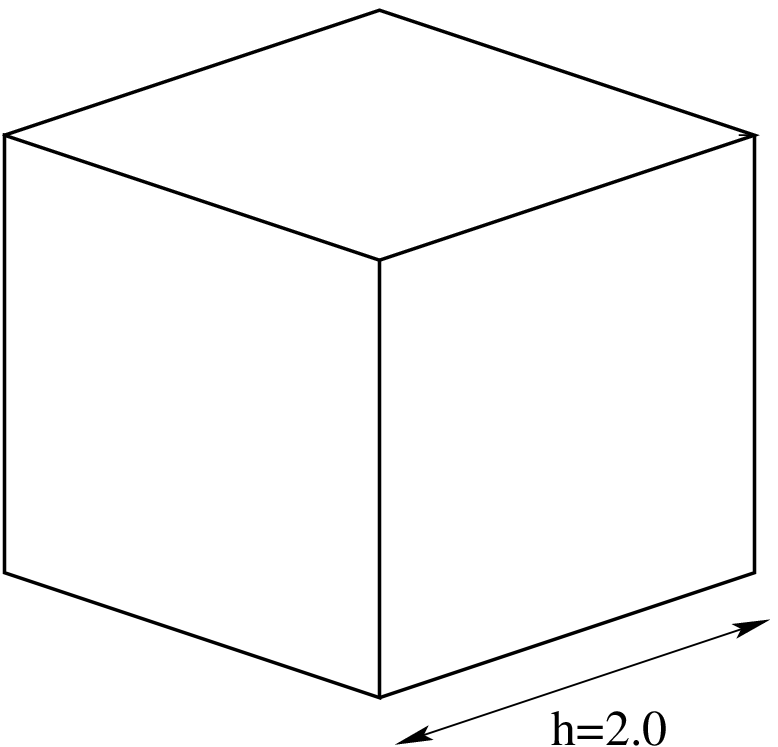}
\label{fig6.1a}
}
\hspace{0.5cm}
\subfigure[]{
\includegraphics[scale = 0.55]{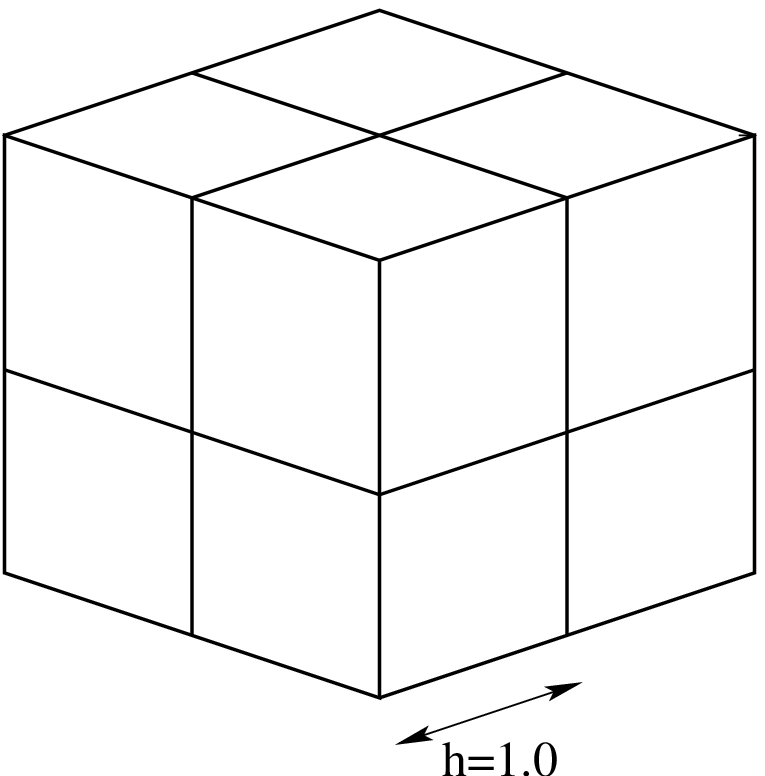}
\label{fig6.1b}
}
\hspace{0.5cm}
\subfigure[]{
\includegraphics[scale = 0.55]{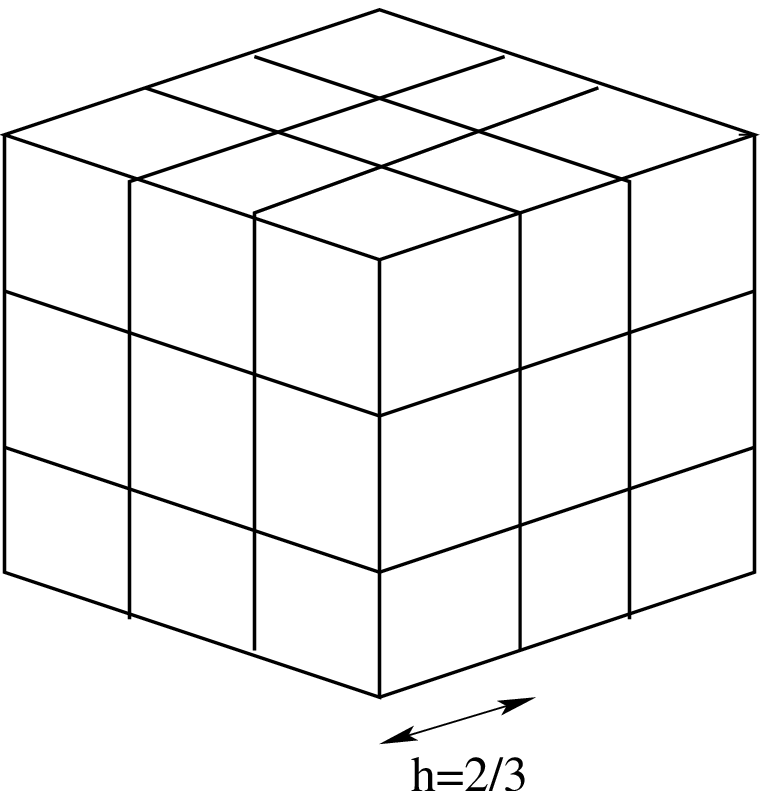}
\label{fig6.1c}
}
\caption[The domain $\Omega=Q=$ standard cube with uniform mesh refinements.]
{The domain $\Omega=Q=(-1,1)^{3}=$ standard cube with uniform mesh refinements.
{\subref{fig6.1a} Mesh 1, \subref{fig6.1b} Mesh 2 and \subref{fig6.1c} Mesh 3.}}
\label{fig6.1}
\end{figure}

It is known from Chapter $4$ that the error in the regular (smooth) region obeys
\begin{align}\label{eq6.1}
||u_{appx.}-u_{ex.}||_{H^{1}}\leq C e^{-bN_{dof}^{1/3}}\,.
\end{align}
Here $N_{dof}$ denotes the number of degrees of freedom (DOF).

Thus, in case the solution is analytic on $\bar{\Omega}$, exponential convergence can
be achieved by increasing the polynomial order and keeping the number of elements fixed.
Hence, for practical implementation it is enough to compute the error for $p-$version
of the method. In what follows by iterations we always mean the total number of iterations
required to compute the solution upto desired accuracy by PCGM.
\begin{rem}
For computational simplicity, in all our computations we assume that the degrees of the
approximating polynomials are uniform within each element.
\end{rem}
\begin{guess8}
\bf{(Laplace equation with Dirichlet boundary conditions)}:
\end{guess8}
Our first example is the Laplace equation in the unit cube $\Omega=(0,1)^3$ shown
in Figure \ref{fig6.2}, with Dirichlet boundary conditions:
\begin{align}
\triangle{u} & = u_{xx}+u_{yy}+u_{zz} = 0 \qquad \text{in} \quad \Omega\,,\notag \\
u & = g \qquad \text{on} \quad \partial \Omega \notag
\end{align}
where the data $g$ is chosen so that the exact solution is
$$u(x,y,z)=\frac{1}{\pi^2\sinh\sqrt{2}\pi}\sin(\pi x)\sin(\pi y)\sinh(\sqrt{2}\pi z).$$
\begin{figure}[!ht]\label{fig6.2}
\centering
\includegraphics[scale = 0.60]{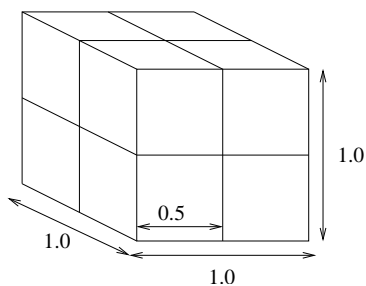}
\caption[Mesh imposed on $\Omega=(0,1)^{3}$ with mesh size $h=0.5$.]
{Mesh imposed on $\Omega=(0,1)^{3}$ with mesh size $h=0.5$.}
\end{figure}
The results are given in Table \ref{tab6.1}. The relative error (in $\%$) against polynomial
order $p$ and iterations against $p$ are plotted in Figure \ref{fig6.3a} and \ref{fig6.3b}
respectively. The error curve is a straight line and this shows the exponential rate
of convergence. The relative error is plotted on a $\log$-scale in each case.
\pagebreak
\begin{table}[!ht]
    \caption{Performance of the $p-$version for \textbf{Laplace} equation with Dirichlet boundary
             conditions}
    \begin{center}
        \label{tab6.1}
        \begin{tabular}{|c|c|c|c|}
          \hline
        \text{$p=W$} & \text{DOF} & \text{Iterations} & \text{$||E||_{rel}$($\%$)}\\
          \hline
             2 & 64 & 51 & 0.380275E+02 \\
            \hline
             4 & 512 & 213 & 0.204875E+01 \\
            \hline
             6 & 1728 & 303 & 0.269917E-01 \\
            \hline
             8 & 4096 & 380 & 0.221613E-03 \\
            \hline
             10 & 8000 & 456 & 0.106885E-05 \\
            \hline
             12 & 13824 & 523 & 0.452056E-08 \\
            \hline
        \end{tabular}
    \end{center}
\end{table}
\vspace{1mm}
\begin{figure}[!ht]
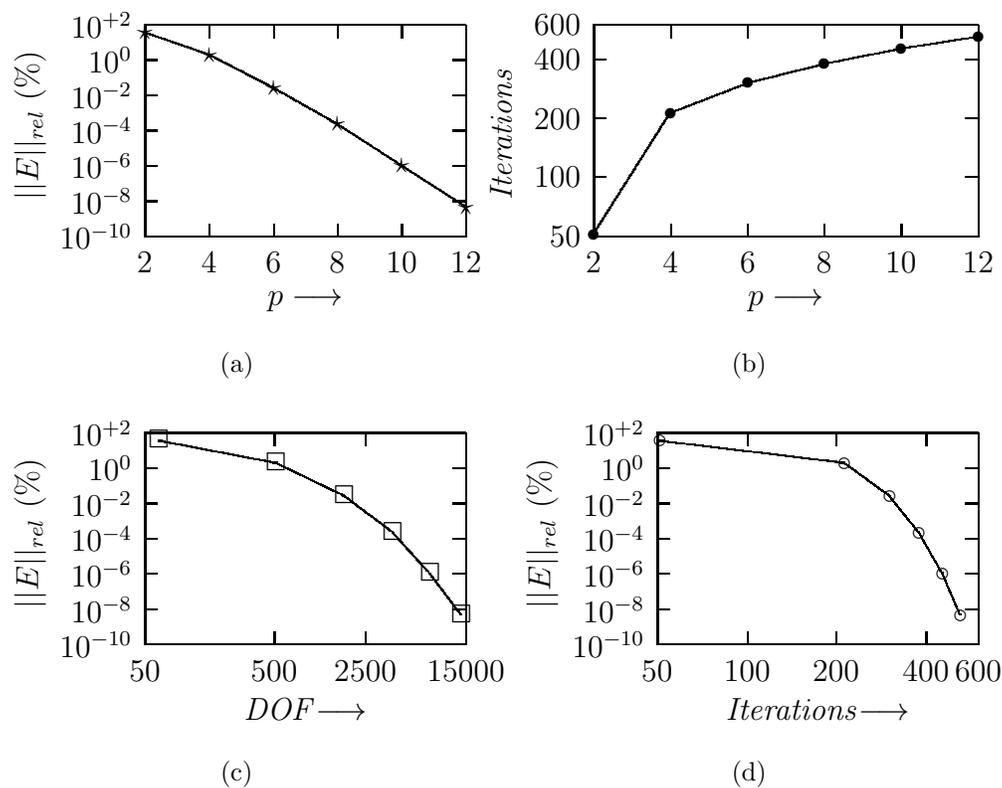

\centering
\subfigure[]{
\input{./3D_dirichlet_Laplace_err}
\label{fig6.3a}
}
\hspace{-1.0cm}
\subfigure[]{
\input{./3D_dirichlet_Laplace_itr}
\label{fig6.3b}
}
\vspace{1.0mm}
\subfigure[]{
\input{./3D_dirichlet_Laplace_dof}
\label{fig6.3c}
}
\hspace{-1.0cm}
\subfigure[]{
\input{./3D_dirichlet_Laplace_err_itr}
\label{fig6.3d}
}
\caption[Error vs. $p$, Iterations vs. $p$, Error vs. DOF and Error vs. Iterations for
Laplace equation.]
{\subref{fig6.3a} Error, \subref{fig6.3b} Iterations vs. $p$, \subref{fig6.3c} Error vs.
DOF and \subref{fig6.3d} Error vs. Iterations for Laplace equation with Dirichlet boundary
conditions.}
\label{fig6.3}
\end{figure}
\pagebreak

In Figure \ref{fig6.3c} a graph is drawn for $||E||_{rel}$ against degrees of
freedom on a $\log-\log$ scale. It is clear that the error obeys (\ref{eq6.1}).
\begin{guess8}
\bf{(Poisson equation with homogeneous boundary conditions)}:
\end{guess8}
We now consider the Poisson equation posed on $\Omega=Q=(-1,1)^3$ with homogeneous
Dirichlet boundary conditions as follows:
\begin{align}
Lu & =-\triangle{u} = f \qquad \text{in} \quad \Omega\,,\notag \\
u & = 0 \qquad \text{on} \quad \partial \Omega \notag
\end{align}
where the right hand side function $f$ is chosen so that the exact solution is
$$u(x,y,z)=\sin(\pi x)\sin(\pi y)\sin(\pi z).$$
\begin{table}[!ht]
    \caption{Performance of the $p-$version for \textbf{Poisson} equation with homogeneous
             boundary conditions}
    \begin{center}
        \label{tab6.2}
        \begin{tabular}{|c|c|c|c|}
            \hline
        \text{$p=W$} & \text{DOF} & \text{Iterations} & \text{$||E||_{rel}$($\%$)}\\
            \hline
             2 & 64 & 10 & 0.131056E+02 \\
            \hline
             4 & 512 & 56 & 0.178835E+01 \\
            \hline
             6 & 1728 & 87 & 0.607343E-01 \\
            \hline
             8 & 4096 & 100 & 0.872751E-03 \\
            \hline
             10 & 8000 & 109 & 0.717092E-05 \\
            \hline
             12 & 13824 & 116 & 0.384300E-07 \\
            \hline
        \end{tabular}
    \end{center}
\end{table}
The performance of the $p-$version on Mesh $2$ of Figure \ref{fig6.1b} is given in
Table \ref{tab6.2} for different values of the polynomial order $W$.
\begin{figure}[!ht]
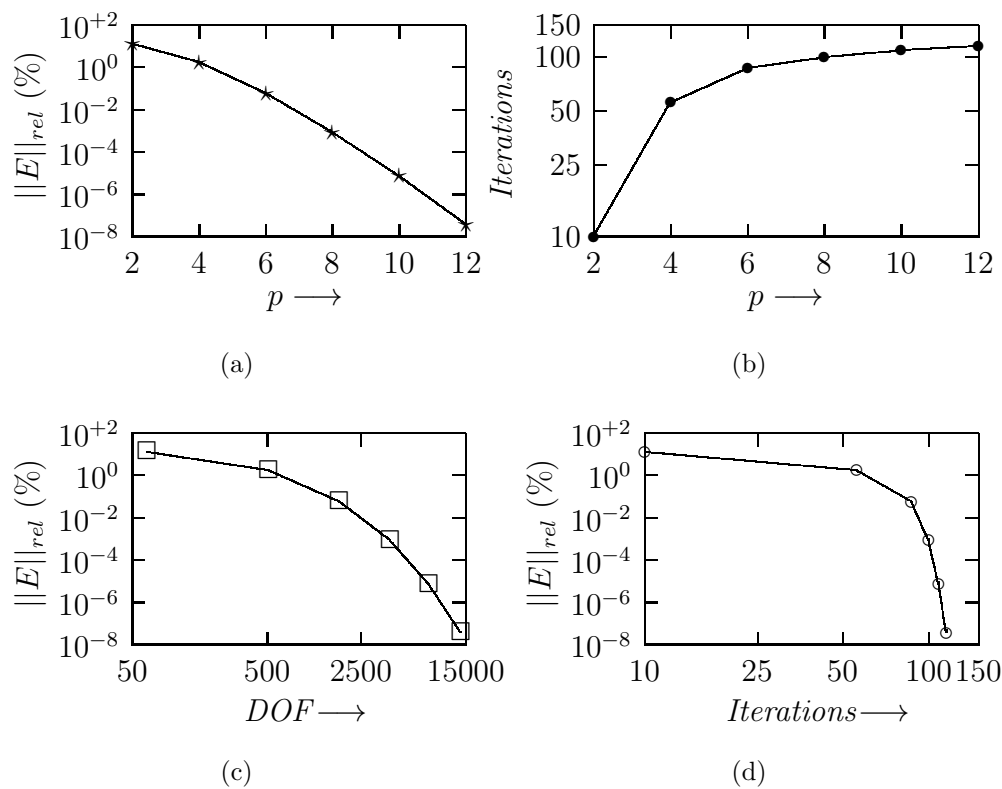

\centering
\subfigure[]{
\input{./3D_dirichlet_Poisson_err}
\label{fig6.4a}
}
\hspace{-1.0cm}
\subfigure[]{
\input{./3D_dirichlet_Poisson_itr}
\label{fig6.4b}
}
\vspace{1.0mm}
\subfigure[]{
\input{./3D_dirichlet_Poisson_dof}
\label{fig6.4c}
}
\hspace{-1.0cm}
\subfigure[]{
\input{./3D_dirichlet_Poisson_err_itr}
\label{fig6.4d}
}
\caption[Error vs. $p$, Iterations vs. $p$, Error vs. DOF and Error vs. Iterations for
Poisson equation.]
{\subref{fig6.4a} Error, \subref{fig6.4b} Iterations as a function of $p$, \subref{fig6.4c}
Error vs. DOF and \subref{fig6.4d} Error vs. Iterations for Poisson equation with homogeneous
boundary conditions.}
\label{fig6.4}
\end{figure}
In Figure \ref{fig6.4a}, the relative error (in \%) is plotted against $p$. The error
decays to $\approx 10^{-8}\,\%$ with $p=12$. Figure \ref{fig6.4b} depicts iterations
against $p$. The relative error is plotted on a $\log$-scale. The relative error against
degrees of freedom and iterations is plotted on a $\log-\log$ scale in Figure \ref{fig6.4c}
and Figure \ref{fig6.4d} respectively.
\begin{guess8}\label{ex6.3}
\bf{(Poisson equation with mixed boundary conditions)}:
\end{guess8}
Now, let us consider the Poisson equation posed on $\Omega=(-1,1)^3$ with mixed
Dirichlet and Neumann boundary conditions as follows:
\begin{align}
Lu & =-\triangle{u} = f \qquad \text{in} \quad \Omega\,,\notag \\
u & = g \qquad \text{on} \quad {\mathcal{D}}\subset\partial \Omega \notag \\
\frac{\partial{u}}{\partial\nu} & = h \qquad \text{on} \quad {\mathcal{N}}
=\partial\Omega\setminus{\mathcal{D}}\:.\notag
\end{align}
Here $\mathcal{D}=\Gamma_{1}\cup\Gamma_{2}\cup\Gamma_{3}$, where $\Gamma_{1},\Gamma_{2}$
and $\Gamma_{3}$ are the faces corresponding to $x=-1,x=1$ and $y=-1$ respectively.
$\mathcal{N}=\Gamma_{4}\cup\Gamma_{5}\cup\Gamma_{6}$, where $\Gamma_{4},\Gamma_{5}$ and
$\Gamma_{6}$ are the faces corresponding to $y=1,z=-1$ and $z=1$ respectively. Moreover,
$\nu$ denotes the outer unit normal to the faces where Neumann boundary conditions are imposed.
\pagebreak
\begin{table}[!ht]
    \caption{Performance of the $p-$version for \textbf{Poisson} equation with mixed boundary
             conditions}
    \begin{center}
        \label{tab6.3}
        \begin{tabular}{|c|c|c|c|}
            \hline
        \text{$p=W$} & \text{DOF} & \text{Iterations} & \text{$||E||_{rel}$($\%$)}\\
            \hline
             2 & 64 & 97 & 0.103976E+02 \\
            \hline
             4 & 512 & 159 & 0.540972E+00 \\
            \hline
             6 & 1728 & 183 & 0.614615E-02 \\
            \hline
             8 & 4096 & 201 & 0.284761E-04 \\
            \hline
             10 & 8000 & 217 & 0.716345E-07 \\
            \hline
             12 & 13824 & 232 & 0.876167E-09 \\
            \hline
        \end{tabular}
    \end{center}
\end{table}
\vspace{1.5mm}
\begin{figure}[!ht]
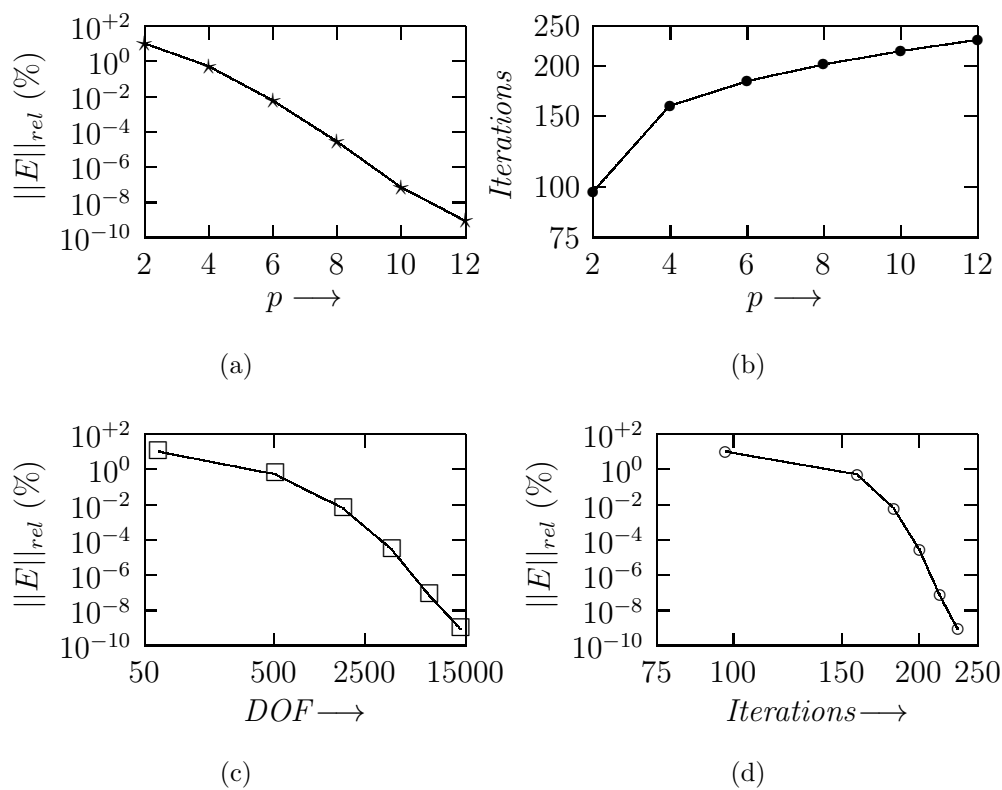

\centering
\subfigure[]{
\input{./3D_mixed_Poisson_err}
\label{fig6.5a}
}
\hspace{-1.0cm}
\subfigure[]{
\input{./3D_mixed_Poisson_itr}
\label{fig6.5b}
}
\vspace{1.0mm}
\subfigure[]{
\input{./3D_mixed_Poisson_dof}
\label{fig6.5c}
}
\hspace{-1.0cm}
\subfigure[]{
\input{./3D_mixed_Poisson_err_itr}
\label{fig6.5d}
}
\caption[Error vs. $p$, Iterations vs. $p$, Error vs. DOF and Error vs. Iterations for Poisson
equation with mixed boundary conditions.]
{\subref{fig6.5a} Error vs. $p$, \subref{fig6.5b} Iterations vs. $p$, \subref{fig6.5c} Error
vs. DOF and \subref{fig6.5d} Error vs. Iterations for Poisson equation with mixed boundary
conditions.}
\label{fig6.5}
\end{figure}
\pagebreak

The right hand side function $f$ and the data $g$, $h$ are chosen so that the true solution
is
$$u(x,y,z)=\frac{2}{\pi^2}\sin\left(\frac{\pi x}{2}\right)\left(\cos\left(\frac{\pi y}{2}
\right)-\sin\left(\frac{\pi z}{2}\right)\right).$$

Results for $p-$version of the method for the mesh in Figure \ref{fig6.1b} are provided in
Table \ref{tab6.3}.
The relevant curves are given in Figure \ref{fig6.5}. In Figure \ref{fig6.5a}, $\log||E||_{rel}$
is plotted against $W$ and the graph is almost linear.

In our next example, we consider a constant coefficient elliptic problem of Helmholtz
type with mixed boundary conditions and examine the performance of the $h-$version for
the three meshes shown in Figure \ref{fig6.1} i.e. for $h=2,1$ and $0.67$. Polynomials
with a uniform degree $W=4$ are used. Comparision with $p-$version of the spectral
element method is provided to show its effectiveness over $h-$version.
\begin{guess8}
\bf{(Helmholtz equation with constant coefficients)}:
\end{guess8}
Consider the elliptic boundary value problem
\begin{align}
Lu & =-\triangle{u} + u = f \qquad \text{in} \quad \Omega\,,\notag \\
Bu & = g \qquad \text{on} \quad \partial \Omega \notag
\end{align}
where $\Omega=(-1,1)^{3}$ and let us choose the data $f$ and $g$ so that the exact solution
is
$$u(x,y,z)=e^{x}\cos{y}\sin{z}.$$
Different choices of the operator $B$, give rise to different boundary conditions. For our
analysis we impose mixed Dirichlet and Neumann boundary conditions on $\partial\Omega$ by
choosing $B$ as follows:
\begin{align}
&B_{1}(u)=u(-1,y,z), B_{2}(u)=u(1,y,z), B_{3}(u)=u(x,-1,z),\notag\\
&B_{4}(u)=\frac{\partial u}{\partial y}(x,1,z), B_{5}(u)=-\frac{\partial u}
{\partial z}(x,y,-1)\:\text{and}\:B_{6}(u)=\frac{\partial u}{\partial z}(x,y,1).\notag
\end{align}
\textbf{The $h-$version:}

Table \ref{tab6.6} shows relative errors ($\%$) against different mesh sizes $h$ for
$W=4$. The relative error decays to $\approx 0.05 \%$ when $h=2/3$ and $N_{dof}\approx 1800$.
\newpage
\textbf{The $p-$version:}

Performance of the $p-$version is given in Table \ref{tab6.7}. A uniform mesh of
size $h=1$ is used which corresponds to the Mesh $2$ (Figure \ref{fig6.1}).

In Figure \ref{fig6.8a}, the relative error is plotted against $p$. Figure \ref{fig6.8b},
contains iterations vs. $p$. The relative error ($\%$) is plotted on a $\log$-scale.
\begin{table}[!ht]
    \caption{Performance of the $h-$version for \textbf{Helmholtz} problem}
    \begin{center}
        \label{tab6.6}
        \begin{tabular}{|c|c|c|c|}
            \hline
        \text{$h$} & \text{$N_{dof}$} &  \text{Iterations} & $||E||_{rel}(\%)$ \\
            \hline
             2 & 64 & 54  & 0.839583E+00 \\
            \hline
             1 & 512 & 76 & 0.114335E+00 \\
            \hline
             2/3 & 1728 & 216 & 0.377972E-01 \\
            \hline
        \end{tabular}
    \end{center}
\end{table}
\begin{table}[!ht]
    \caption{Performance of the $p-$version for \textbf{Helmholtz} problem on Mesh $2$}
    \begin{center}
        \label{tab6.7}
        \begin{tabular}{|c|c|c|c|}
            \hline
        \text{$p$=$N$} & \text{$N_{dof}$} & Iterations & {$||E||_{rel}$($\%$)} \\
            \hline
             2 & 64 & 25 & 0.696654E+01 \\
            \hline
             4 & 512 & 76 & 0.114335E+00 \\
            \hline
             6 & 1728 & 135 & 0.521586E-03 \\
            \hline
             8 & 4096 & 222 & 0.993449E-06 \\
            \hline
             10 & 8000 & 294 & 0.109613E-08 \\
            \hline
             12 & 13824 & 302 & 0.149336E-08 \\
            \hline
        \end{tabular}
    \end{center}
\end{table}

Comparison between the $h$ and $p$ versions are shown in Figure \ref{fig6.9}, where
the relative errors are plotted against the number of degrees of freedom on a
$log-log$ scale. It is clear that the $p-$version is superior to the $h-$version as
expected.
\pagebreak
\begin{figure}[!ht]
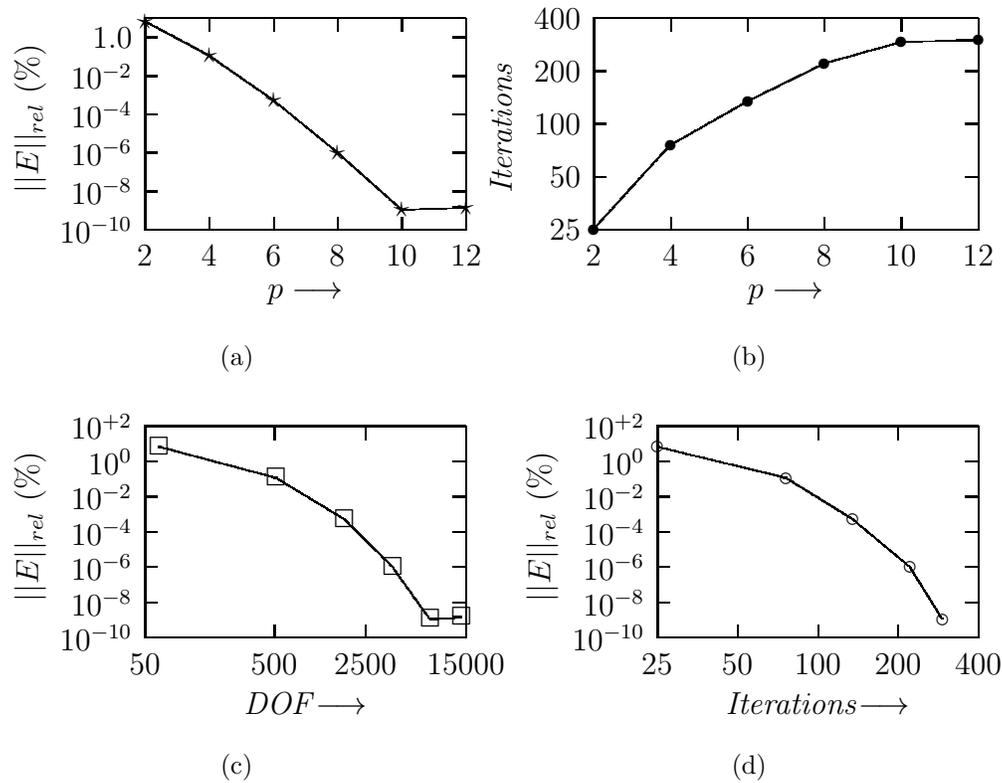

\centering
\subfigure[]{
\input{./3D_mixed_elliptic_err}
\label{fig6.8a}
}
\hspace{-1.0cm}
\subfigure[]{
\input{./3D_mixed_elliptic_itr}
\label{fig6.8b}
}
\vspace{1.0mm}
\subfigure[]{
\input{./3D_mixed_elliptic_dof}
\label{fig6.8c}
}
\hspace{-1.0cm}
\subfigure[]{
\input{./3D_mixed_elliptic_err_itr}
\label{fig6.8d}
}
\caption[Error vs. $p$, Iterations vs. $p$, Error vs. DOF and Error vs. Iterations for
Helmholtz problem.]
{\subref{fig6.8a} Error vs. $p$, \subref{fig6.8b} Iterations vs. $p$, \subref{fig6.8c}
Error vs. DOF and \subref{fig6.8d} Error vs. Iterations for Helmholtz problem with mixed
boundary conditions.}
\label{fig6.8}
\end{figure}
\begin{figure}[!ht]
\centering
\input{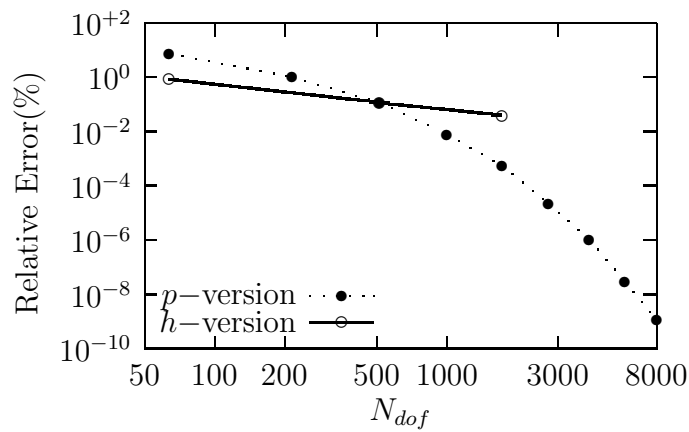}
\caption[Comparison between $h$ and $p$ versions: Error vs. DOF.]
{Comparison between $h$ and $p$ versions: Error vs. Degrees of Freedom.}
\label{fig6.9}
\end{figure}

The examples presented so far deal with constant coefficient differential operators.
However, the method is equally efficient for a general elliptic problem too. To
illustrate this we now give more examples which include problems with variable
coefficients, non self-adjoint problems and a general elliptic problem on $L$-shaped
domain.
\begin{guess8}
\bf{(Elliptic equation with variable coefficients)}:
\end{guess8}
Let us consider the following variable coefficient boundary value problem with Robin
boundary conditions:
\begin{align}
a(x,y,z){u_{xx}}+b(x,y,z){u_{yy}}+c(x,y,z){u_{zz}}+d(x,y,z)u & = f \quad
\text{in} \quad \Omega=(-1,1)^3 ,\notag \\
\frac{\partial u}{\partial \nu}+u & = g \quad\text{on}\quad\partial\Omega\notag
\end{align}
where $\nu$ denotes the outward unit normal to the boundary $\partial\Omega$. We choose
the coefficients as:
\newline
$a(x,y,z)=-(1.00+0.01y\sin x), b(x,y,z)=-(2.50+0.02\cos(x^2+z))$,
\newline
$c(x,y,z)=-(3.00+0.03ye^{z})$ and $d(x,y,z)=0.15\sin(2\pi(y+z))$.
\newline
Moreover, the right hand side function $f$ and the data $g$ are chosen such that the exact
solution is
$$u(x,y,z)=\cos(\pi(x+y))\exp(z).$$
\begin{table}[!ht]
    \caption{Performance of the $p-$version for elliptic problem with \textbf{variable}
                 coefficients}
    \begin{center}
        \label{tab6.12}
        \begin{tabular}{|c|c|c|c|}
            \hline
        \text{$p=W$} & DOF & \text{Iterations} & \text{Relative Error($\%$)} \\
            \hline
             2 & 216 & 231 & 0.284009E+02 \\
            \hline
             3 & 729 & 332 & 0.737578E+01 \\
            \hline
             4 & 1728 & 369 & 0.203203E+01 \\
            \hline
             5 & 3375 & 397 & 0.318021E+00 \\
            \hline
             6 & 5832 & 414 & 0.352678E-01 \\
            \hline
             7 & 9216 & 428 & 0.337654E-02 \\
            \hline
             8 & 13824 & 437 & 0.263263E-03 \\
            \hline
             9 & 19683 & 445 & 0.185170E-04 \\
            \hline
        \end{tabular}
    \end{center}
\end{table}
We examine the $p$-version of the spectral element method with fixed mesh size
$h=2/3$ that corresponds to the Mesh $3$ in Figure \ref{fig6.1}. Results are given
in Table \ref{tab6.12}.

Figure \ref{fig6.11a} shows a graph between $\log||E||_{rel}$ against polynomial
degree $p$. The relationship is almost linear demonstrating that the error decays
exponentially.
In Figure \ref{fig6.11b}, $\log(iterations)$ is drawn against $p$
while in Figure \ref{fig6.11c} error against degrees of freedom is plotted on a
$log-log$ scale.
\vspace{6mm}
\begin{figure}[!ht]
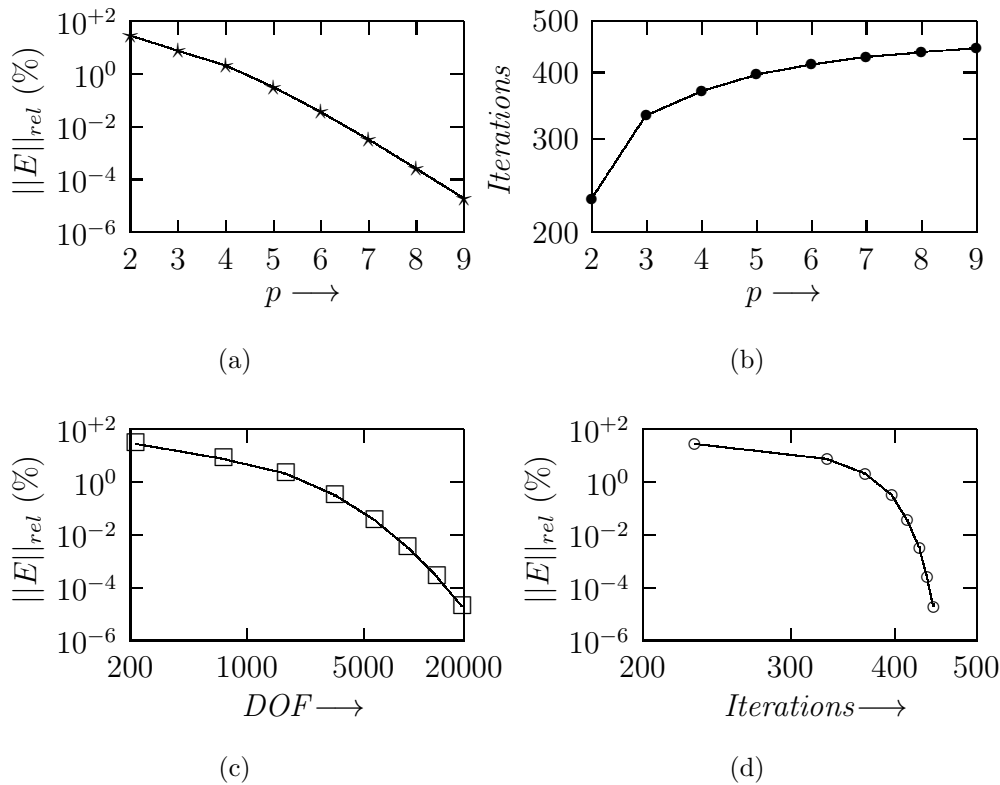

\centering
\subfigure[]{
\input{./3D_var_robin_err}
\label{fig6.11a}
}
\hspace{-1.0cm}
\subfigure[]{
\input{./3D_var_robin_itr}
\label{fig6.11b}
}
\vspace{1.0mm}
\subfigure[]{
\input{./3D_var_robin_dof}
\label{fig6.11c}
}
\hspace{-1.0cm}
\subfigure[]{
\input{./3D_var_robin_err_itr}
\label{fig6.11d}
}
\caption[Error vs. $p$, Iterations vs. $p$, Error vs. DOF and Error vs. Iterations for
elliptic problem with variable coefficients.]
{\subref{fig6.11a} Error vs. $p$, \subref{fig6.11b} Iterations vs. $p$, \subref{fig6.11c}
Error vs. DOF and \subref{fig6.11d} Error vs. Iterations for elliptic problem with variable
coefficients.}
\label{fig6.11}
\end{figure}
\begin{guess8}
\bf{(Variable coefficient problem on L shaped domain)}:
\end{guess8}
In our next example we consider the following variable coefficients problem with
Dirichlet boundary conditions on the solution domain $\Omega$ shown in Figure
\ref{fig6.12}:
\begin{align}
a(x,y,z){u_{xx}}+b(x,y,z){u_{yy}}+c(x,y,z){u_{zz}}+d(x,y,z)u_x\notag \\
+e(x,y,z)u_y+h(x,y,z)u_z+l(x,y,z)u & = f \quad \text{in} \quad \Omega ,\notag \\
u & = g \quad \text{on} \quad {\mathcal{D}}=\partial \Omega \notag
\end{align}
where the choice of the coefficients is as follows:
\newline
$a(x,y,z)=-(0.50+0.01y\sin x), b(x,y,z)=-(1.50+0.02\cos(x^2+z))$,\newline
$c(x,y,z)=-(2.00+0.03ye^{z})$, $d(x,y,z)=0.25\sin(2\pi(y+z))$,\newline
$e(x,y,z)=0.25\sin(2\pi(z+x))$, $h(x,y,z)=0.25\sin(2\pi(x+y))$,\newline
and $l(x,y,z)=2.50-0.025\exp\left(\frac{\pi(x+y+z)}{2}\right)$.
\newline
Moreover, the right hand side function $f$ and the data $g$ are chosen such that the
exact solution is
$$u(x,y,z)=\sin\left(\frac{\pi(x+y+z)}{2}\right)\exp\left(\frac{\pi z}{2}\right).$$
\begin{figure}[!ht]
\centering
\includegraphics[scale = 0.65]{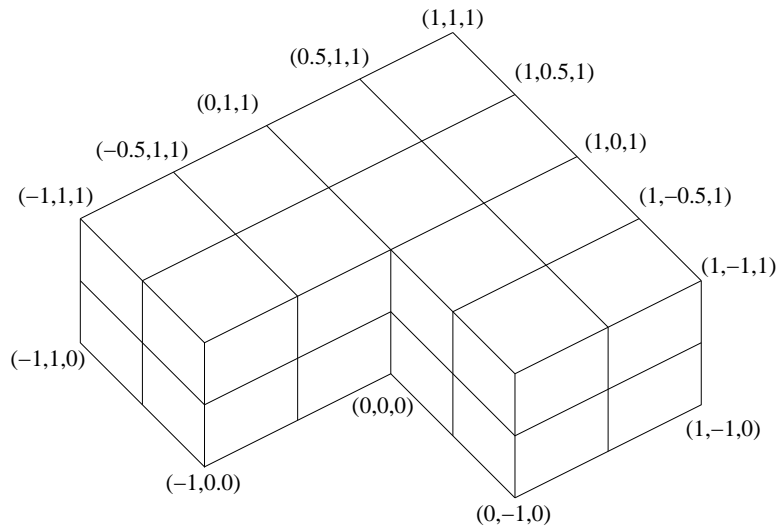}
\caption{Mesh imposed on domain $\Omega$ containing 24 brick elements}
\label{fig6.12}
\end{figure}

Performance of the $p-$version on mesh (containing brick elements) shown in Figure
\ref{fig6.12} is provided in Table \ref{tab6.13}. The relative error decays to
$\approx 10^{-8}$ when $p=10$ and $N_{dof}=240000$.

Relative error and iterations are plotted against polynomial degree $p$ in Figure
\ref{fig6.13a} and \ref{fig6.13b} respectively.
In Figure \ref{fig6.13c} we plot the relative error as a function of degrees of freedom.
The graph would be a straight line if the error obeys (\ref{eq6.1}) exactly.

\pagebreak
\begin{table}[!ht]
    \caption{Performance of the $p-$version for elliptic problem on \textbf{L-shaped} domain}
    \begin{center}
        \label{tab6.13}
        \begin{tabular}{|c|c|c|c|}
            \hline
        \text{$p=W$} & \text{DOF} & Iterations & \text{Relative Error($\%$)}\\
            \hline
             2 & 384 & 54 & 0.143522E+02 \\
            \hline
             3 & 1944 & 143 & 0.131046E+01 \\
            \hline
             4 & 6144 & 192 & 0.209803E+00 \\
            \hline
             5 & 15000 & 243 & 0.149981E-01 \\
            \hline
             6 & 31104 & 308 & 0.101674E-02 \\
            \hline
             7 & 57624 & 336 & 0.785544E-04 \\
            \hline
             8 & 98304 & 393 & 0.611724E-05 \\
            \hline
             9 & 157464 & 420 & 0.549101E-06 \\
            \hline
             10 & 240000 & 488 & 0.447650E-07 \\
            \hline
        \end{tabular}
    \end{center}
\end{table}
\begin{figure}[!ht]
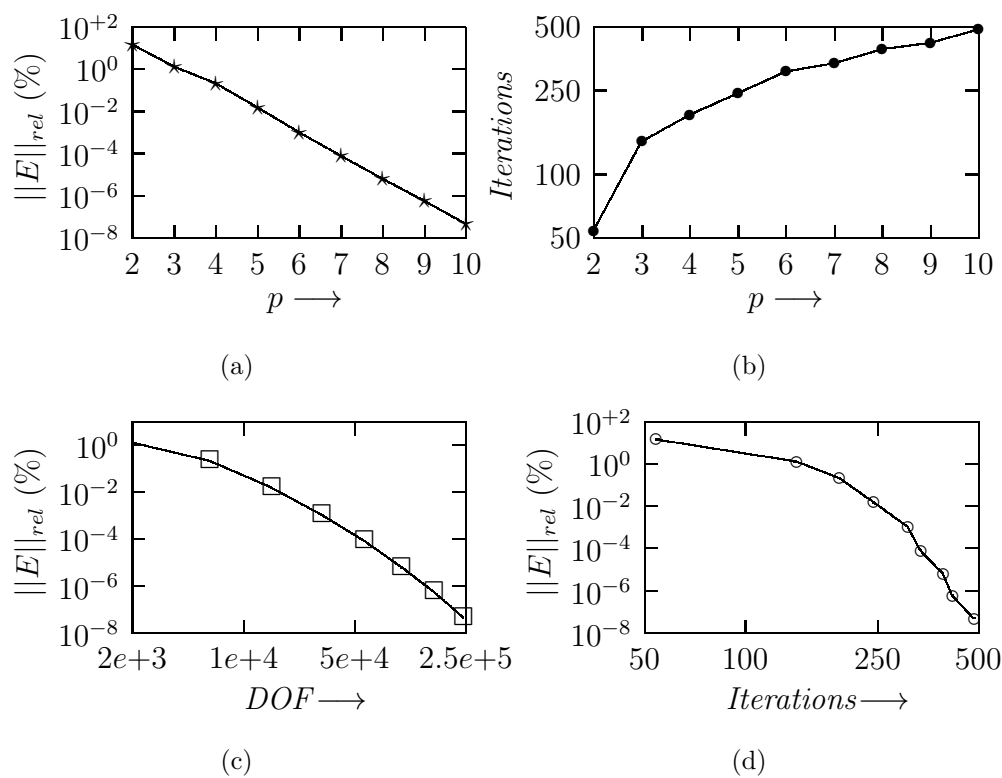

\centering
\subfigure[]{
\input{./3D_var_Lshape_err}
\label{fig6.13a}
}
\hspace{-1.0cm}
\subfigure[]{
\input{./3D_var_Lshape_itr}
\label{fig6.13b}
}
\subfigure[]{
\input{./3D_var_Lshape_dof}
\label{fig6.13c}
}
\hspace{-1.0cm}
\subfigure[]{
\input{./3D_var_Lshape_err_itr}
\label{fig6.13d}
}
\caption[Error vs. $p$, Iterations vs. $p$, Error vs. DOF and Error vs. Iterations
for elliptic problem on L-shaped domain.]
{\subref{fig6.13a} Error vs. $p$, \subref{fig6.13b} Iterations vs. $p$, \subref{fig6.13c}
Error vs. DOF and \subref{fig6.13d} Error vs. Iterations for elliptic problem on
L-shaped domain.}
\label{fig6.13}
\end{figure}
\pagebreak

\begin{guess8}
\bf{(General elliptic equation with variable coefficients: A non self-adjoint problem)}:
\end{guess8}
The method works for non self-adjoint problems too. To verify this, let us consider
the following non self-adjoint general elliptic problem with mixed boundary conditions.
\begin{align}
a(x,y,z){u_{xx}}+b(x,y,z){u_{yy}}+c(x,y,z){u_{zz}}\notag\\
+d(x,y,z)(u_{xy}+u_{yz}+u_{zx})+e(x,y,z)u & = f \quad\text{in}
\quad\Omega=(-1,1)^3\,,\notag \\
u & = g \quad \text{on} \quad {\mathcal{D}}\subset\partial \Omega \notag \\
\frac{\partial u}{\partial \nu} & = h \quad \text{on}\quad {\mathcal{N}}=\partial
\Omega \setminus{\mathcal{D}}\:. \notag
\end{align}

Here $\mathcal{D}$ and $\mathcal{N}$ denote the Dirichlet and Neumann boundary part of
$\partial\Omega$ respectively and are chosen similar to those in example \ref{ex6.3}.
Further, we choose the coefficients of the problem as follows:
\newline
$a(x,y,z)=-(0.50+0.05\exp(xyz)), b(x,y,z)=-(1.00+0.015\cos(x+y))$,
\newline
$c(x,y,z)=-(2.50+0.02\exp(y+z)), d(x,y,z)=-0.001\sin(\pi(x+y+z))$
\newline
and $e(x,y,z)=\Bigg(4.05+0.045\cos\Big(\frac{\pi(x+y+z)}{2}\Big)\Bigg)$.

Moreover, the right hand side function $f$ and the data $g$ and $h$ are chosen such
that the true solution is
$$u(x,y,z)=\Big(\sin(\pi x)+\sin\left(\frac{\pi y}{2}\right)\Big)\cos(\pi z).$$

We examine the $p-$version of the method on different meshes in Table \ref{tab6.14}
for polynomial degree $p=2,\ldots,10$. It is clear that the method performs best on
Mesh 3 and the error reduces to approximately $10^{-6}\%$. However, on Mesh 1 the
relative error decays slowly.
Figure \ref{fig6.15} shows $\log||E||_{rel}$ plotted against $p$ for different meshes.

In Figure \ref{fig6.16a} we plot error against polynomial order $p$. Error as a function
of degrees of freedom is plotted in Figure \ref{fig6.16c} on a $log-log$ scale.
\pagebreak
\begin{table}[!ht]
    \caption{Performance of the $p-$version on different meshes for general
                 elliptic (non self-adjoint) problem with \textbf{variable} coefficients}
    \begin{center}
        \label{tab6.14}
        \begin{tabular}{|c|c|c|c|}
            \hline
        \text{$p=W$} & \text{Mesh 1} & \text{Mesh 2} & \text{Mesh 3}\\
            \hline
             2 & 0.498666E+02 & 0.485737E+02 & 0.291328E+02 \\
            \hline
             3 & 0.507328E+02 & 0.107829E+02 & 0.586317E+01 \\
            \hline
             4 & 0.227058E+02 & 0.517751E+01 & 0.149092E+01 \\
            \hline
             5 & 0.143968E+02 & 0.814482E+00 & 0.237686E+00 \\
            \hline
             6 & 0.408230E+01 & 0.234810E+00 & 0.292025E-01 \\
            \hline
             7 & 0.184822E+01 & 0.279261E-01 & 0.289156E-02 \\
            \hline
             8 & 0.311612E+00 & 0.430208E-02 & 0.236661E-03 \\
            \hline
             9 & 0.104765E+00 & 0.377498E-03 & 0.168223E-04 \\
            \hline
             10 & 0.136308E-01 & 0.431639E-04 & 0.352480E-05 \\
            \hline
        \end{tabular}
    \end{center}
\end{table}
\vspace{8mm}
\begin{figure}[!ht]
\centering
\input{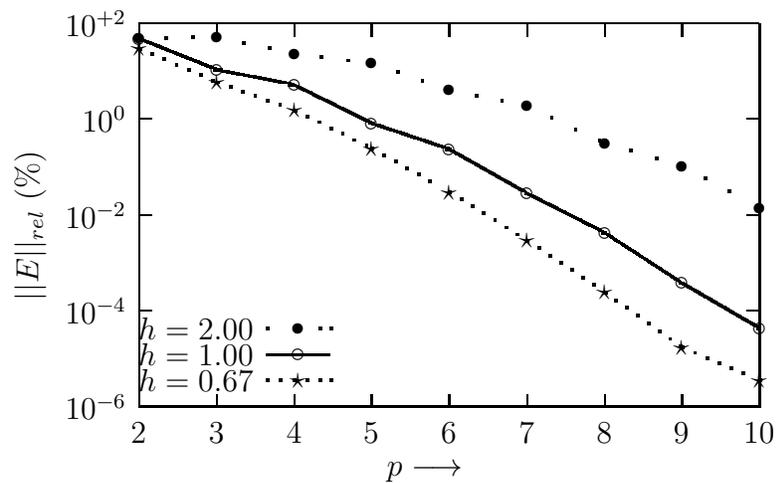}
\caption[Error as a function of $W$ for different values of $h$ for general elliptic
(non self-adjoint) problem.]
{Error as a function of $W$ for different values of $h$ for general elliptic
(non self-adjoint) problem.}
\label{fig6.15}
\end{figure}
\begin{table}[!ht]
    \caption{Performance of the $p-$version for \textbf{non self-adjoint} problem.}
    \begin{center}
        \label{tab6.15}
        \begin{tabular}{|c|c|c|c|}
            \hline
        \text{$p=W$} & \text{DOF} & Iterations & \text{Relative Error($\%$)}\\
            \hline
             2 & 64 & 207 & 0.670184E+02 \\
            \hline
             3 & 216 & 364 & 0.741499E+01 \\
            \hline
             4 & 512 & 464 & 0.536971E+01 \\
            \hline
             5 & 1000 & 519 & 0.725120E+00 \\
            \hline
             6 & 1728 & 547 & 0.228276E+00 \\
            \hline
             7 & 2744 & 605 & 0.267290E-01 \\
            \hline
             8 & 4096 & 642 & 0.420000E-02 \\
            \hline
             9 & 5832 & 671 & 0.364194E-03 \\
            \hline
             10 & 8000 & 694 & 0.424524E-04 \\
            \hline
        \end{tabular}
    \end{center}
\end{table}
\begin{figure}[!ht]
\centering
\subfigure[]{
\input{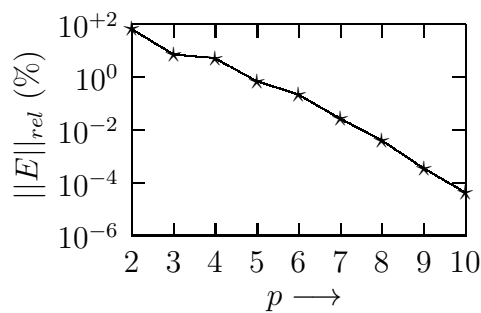}
\label{fig6.16a}
}
\hspace{-1.0cm}
\subfigure[]{
\input{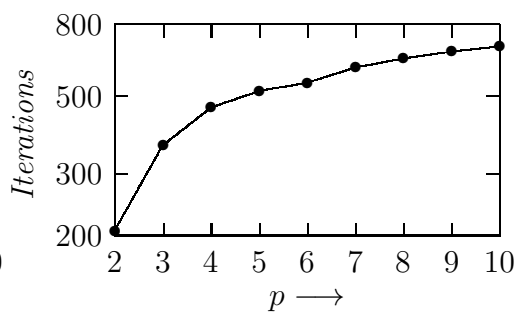}
\label{fig6.16b}
}
\subfigure[]{
\input{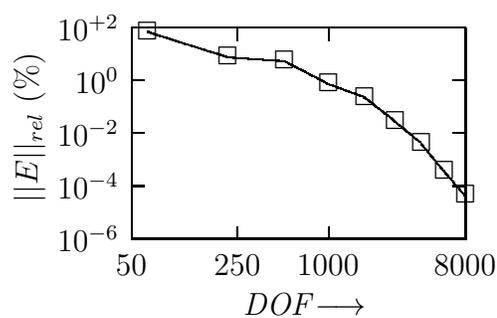}
\label{fig6.16c}
}
\hspace{-1.0cm}
\subfigure[]{
\input{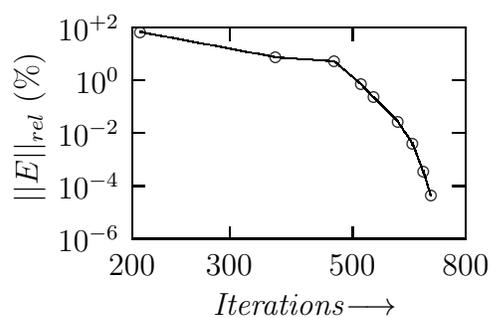}
\label{fig6.16d}
}
\caption[Error vs. $p$, Iterations vs. $p$, Error vs. DOF and Error vs. Iterations
for general elliptic (non self-adjoint) problem.]
{\subref{fig6.16a} Error vs. $p$, \subref{fig6.16b} Iterations vs. $p$, \subref{fig6.16c}
Error vs. DOF and \subref{fig6.16d} Error vs. Iterations for general elliptic (non
self-adjoint) equation with variable coefficients.}
\label{fig6.16}
\end{figure}

\section{Test Problems in Singular Regions}
In order to show the effectiveness of the proposed method in dealing with three dimensional
elliptic boundary value problems containing singularities we now present results of numerical
simulations for model problems on various polyhedral domains containing different types of
singularities discussed in the previous chapters which aim to verify the validity of the
asymptotic error estimates and estimates of computational complexity that have been obtained.

Hereafter \textbf{$N$ will always denote the number of layers in the geometric mesh and $W$,
the polynomial order used. It is assumed that $N$ is proportional to $W$}. In case of examples
with vertex and edge singularities all our calculations are based on a parallel computer
with $O(N)$ processors and each element is being mapped onto a separate processor. However,
in case of examples with vertex-edge singularities we have used a parallel computer with $O(N^{2})$
processors. The geometric mesh factors in the neighbourhoods of vertices and edges are chosen
as $\mu_{v}=0.15$ and $\mu_{e}=0.15$ which give optimal results.

The examples reported here will include all three types of singularities. In addition, to
compare our results with those existing in the literature we shall consider a few examples
similar to those of Guo and Oh~\cite{GOH1} as well.

Our first example is the Poisson equation containing only a vertex singularity with Dirichlet
boundary conditions. For computational simplicity we shall assume that the singularity arises
only at one vertex of the domain under consideration. However, in general, a typical domain
may contain singularities at more than one corner in which case the same technique can be
applied to each vertex containing a singularity.
\begin{guess8}
\bf{(Poisson equation with vertex singularity)}
\end{guess8}
Consider the domain $\Omega^{(v)}$ shown in Figure \ref{fig6.17} defined by
\[\Omega^{(v)}=\{(\phi,\theta,\rho)\:|\:\pi/6\leq\phi\leq\pi/3,\:0\leq\theta\leq 3\pi/2,
\:\rho\leq 1\}.\]
We consider the following model problem:
\begin{align}\label{eq6.2}
-\triangle{u} & = f \quad\text{in} \quad\Omega^{(v)}, \notag \\
            u & = g \quad \text{on} \quad \partial \Omega^{(v)}.
\end{align}
\begin{figure}[!ht]
\centering
\includegraphics[scale = 0.75]{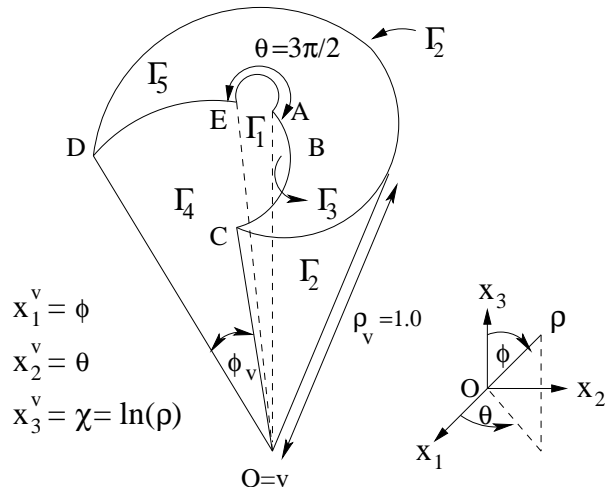}
\caption[The domain $\Omega^{(v)}$ containing a vertex singularity.]
{The domain $\Omega^{(v)}$ containing a vertex singularity.}
\label{fig6.17}
\end{figure}
Let $w(\phi,\theta,\rho)=\rho^{1/2}\sin(\frac{\phi}{2})$. Then this axially
symmetric function $w$ is the true solution of the model problem \ref{eq6.2}.
Note that $w$ has a vertex singularity at the origin and none other. Let
$x_{1}^{v},x_{2}^{v}$ and $x_{3}^{v}$ denote the modified coordinates in the
vertex neighbourhood introduced in Chapter $2$.
\begin{figure}[!ht]
\centering
\includegraphics[scale = 0.65]{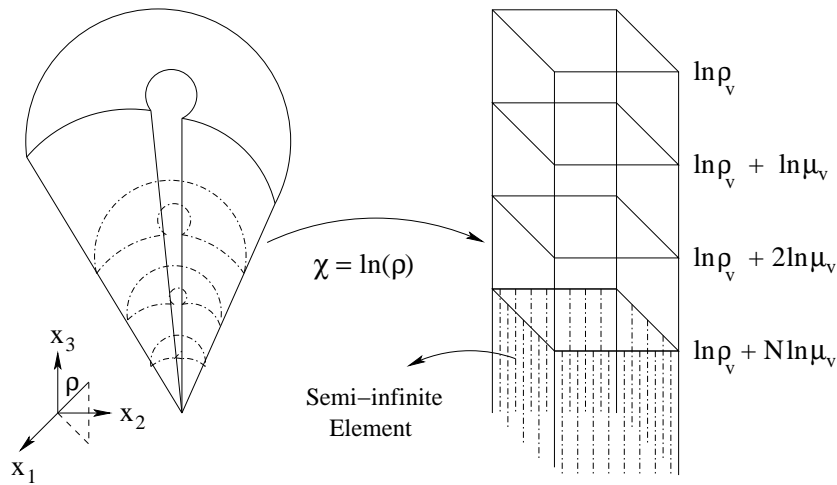}
\caption[Geometric mesh imposed on $\Omega^{(v)}$ and elements after mapping.]
{Geometric mesh imposed on $\Omega^{(v)}$ (front view) and elements after mapping.}
\label{fig6.30}
\end{figure}
\pagebreak

Performance of the $h-p$ version of the spectral element method is analyzed
on geometrically refined meshes shown in Figure \ref{fig6.30} with mesh ratio
$\mu_{v}=0.15$ (optimal for geometric mesh refinement~\cite{BG7,BG5,BGOH,G1,GOH1}
for the $h-p$ version of the finite element method). Results are given in Table
\ref{tab6.16}.
\begin{table}[!ht]
    \caption[Performance of the $h-p$ version for \textbf{Dirichlet} problem on
                 $\Omega^{(v)}$ containing a \textbf{vertex singularity}.]
                {Performance of the $h-p$ version for Dirichlet problem on domain
                 $\Omega^{(v)}$ containing a vertex singularity}
    \begin{center}
        \label{tab6.16}
        \begin{tabular}{|c|c|c|c|}
            \hline
        \text{$p=W$} & \text{DOF($N_{dof}$)} & Iterations & \text{Relative Error($\%$)}\\
            \hline
             2 & 9 & 46 & 0.100821E+01 \\
            \hline
             3 & 55 & 56 & 0.841462E-01 \\
            \hline
             4 & 193 & 86 & 0.309263E-02 \\
            \hline
             5 & 501 & 109 & 0.273612E-03 \\
            \hline
             6 & 1081 & 123 & 0.733920E-04 \\
            \hline
             7 & 2059 & 143 & 0.180276E-04 \\
            \hline
             8 & 3585 & 158 & 0.503131E-05 \\
            \hline
             9 & 5833 & 179 & 0.984438E-06 \\
            \hline
             10 & 9001 & 191 & 0.480789E-06 \\
            \hline
        \end{tabular}
    \end{center}
\end{table}
\begin{figure}[!ht]
\centering
\subfigure[]{
\input{./vertex_dirichlet_err}
\label{fig6.18a}
}
\hspace{-1.0cm}
\subfigure[]{
\input{./vertex_dirichlet_itr}
\label{fig6.18b}
}
\vspace{1.0mm}
\subfigure[]{
\input{./vertex_dirichlet_dof}
\label{fig6.18c}
}
\hspace{-1.0cm}
\subfigure[]{
\input{./vertex_dirichlet_err_itr}
\label{fig6.18d}
}
\caption[Error vs. $p$, Iterations vs. $N$, Error vs. $N_{dof}$ and Error vs. Iterations
for Poisson equation containing a vertex singularity.]
{\subref{fig6.18a} Error vs. $p$, \subref{fig6.18b} Iterations vs. $N$, \subref{fig6.18c}
Error vs. $N_{dof}$ and \subref{fig6.18d} Error vs. Iterations for Poisson equation
containing a vertex singularity.}
\label{fig6.18}
\end{figure}

We know that the error in the neighbourhoods of vertices satisfies
\begin{align}\label{eq6.3}
||u_{appx.}-u_{ex.}||_{H^{1}}\leq C e^{-bN_{dof}^{1/4}}\,.
\end{align}
Here $N_{dof}$ denotes the number of degrees of freedom.

In figure \ref{fig6.18a} the relative error versus polynomial degree $p$ is drawn.
Iterations as a function of $N$, the number of layers are given in Figure \ref{fig6.18b}.
Relative error against degrees of freedom is depicted on a $\log-\log$ scale in Figure
\ref{fig6.18c}, the curve is a straight line which confirms the estimate (\ref{eq6.3}).
The error decays to $\approx 10^{-7}$ when $p=10$. Further, the number of iterations
required to obtain the solution to the desired accuracy is nearly 200 which is fairly low.
\pagebreak

The second example illustrates the effectiveness of the method for problems with mixed
boundary conditions containing a vertex singularity and is similar to that of Guo and
Oh reported in~\cite{GOH1}.
\begin{guess8}
\bf{(Mixed problem containing vertex singularity)}:
\end{guess8}
Consider the axisymmetric Poisson equation with mixed boundary conditions:
\begin{align}\label{eq6.4}
-\triangle{u} & = f \quad\text{in} \quad\Omega^{(v)}, \notag \\
u & = g \quad \text{on}\quad{\mathcal D}\subset\partial \Omega^{(v)},\notag\\
\frac{\partial u}{\partial \nu} &= h \quad \text{on} \quad
{\mathcal N}=\partial \Omega^{(v)}\setminus{\mathcal D},
\end{align}
where the domain $\Omega^{(v)}$ is shown in Figure \ref{fig6.17} and
\begin{align}
{\mathcal D} & =\Gamma_{1}\cup\Gamma_{2}\cup\Gamma_{3}\cup\Gamma_{4}=
\Gamma^{1}_{\mathcal D}\cup\Gamma^{2}_{\mathcal D}\:, \notag\\
\Gamma^{1}_{\mathcal D} & =\{(\phi,\theta,\rho):\:\phi=\pi/6,\pi/3,\:0\leq
\theta\leq 3\pi/2,\:0\leq\rho\leq 1\}\:,\notag\\
\Gamma^{2}_{\mathcal D} & =\{(\phi,\theta,\rho):\:\pi/6\leq\phi\leq\pi/3,
\:\theta=0,3\pi/2,\:0\leq\rho\leq 1\}\:, \notag\\
{\mathcal N} & =\Gamma_{5}=\{(\phi,\theta,\rho):\:\pi/6\leq\phi\leq\pi/3,
\:0\leq\theta\leq 3\pi/2,\:\rho=1\}\:. \notag
\end{align}
We choose data $f$, $g$ and $h$ such that the function
$w=\rho^{0.1}(1-\rho)\sin2\phi$
is the true solution of (\ref{eq6.4}) satisfying prescribed boundary conditions.
Here $\nu$ denotes the exterior unit normal to the part of the boundary where we
impose Neumann boundary conditions.
\begin{table}[!ht]
    \caption[Performance of the $h-p$ version for \textbf{mixed} problem on $\Omega^{(v)}$
                 containing a \textbf{vertex singularity}.]
                {Performance of the $h-p$ version for mixed problem on polyhedral
                 domain $\Omega^{(v)}$ containing a vertex singularity}
    \begin{center}
        \label{tab6.17}
        \begin{tabular}{|c|c|c|c|}
            \hline
        \text{$p=W$} & \text{DOF($N_{dof}$)} & Iterations & \text{Relative Error($\%$)}\\
            \hline
             2 & 9 & 16 & 0.962637E+01 \\
            \hline
             3 & 55 & 39 & 0.252012E+01 \\
            \hline
             4 & 193 & 128 & 0.191490E+00 \\
            \hline
             5 & 501 & 176 & 0.212320E-01 \\
            \hline
             6 & 1081 & 314 & 0.192391E-02 \\
            \hline
             7 & 2059 & 409 & 0.884830E-03 \\
            \hline
             8 & 3585 & 743 & 0.412629E-03 \\
            \hline
             9 & 5833 & 814 & 0.470681E-04 \\
            \hline
        \end{tabular}
    \end{center}
\end{table}

Table \ref{tab6.17} contains the relative error obtained by applying the method
on geometrically refined mesh in $\rho$. The relative error versus degrees of
freedom is depicted in Figure \ref{fig6.18}. It is clear that the method gives
exponential accuracy.
\begin{figure}[!ht]
\centering
\subfigure[]{
\input{./vertex_mixed_err}
\label{fig6.19a}
}
\hspace{-1.0cm}
\subfigure[]{
\input{./vertex_mixed_itr}
\label{fig6.19b}
}
\vspace{1.0mm}
\subfigure[]{
\input{./vertex_mixed_dof}
\label{fig6.19c}
}
\hspace{-1.0cm}
\subfigure[]{
\input{./vertex_mixed_err_itr}
\label{fig6.19d}
}
\caption[Error vs. $p$, Iterations vs. $N$, Error vs. $N_{dof}$ and Error vs. Iterations
for mixed problem containing a vertex singularity.]
{\subref{fig6.19a} Error vs. $p$, \subref{fig6.19b} Iterations vs. $N$, \subref{fig6.19c}
Error vs. $N_{dof}$ and \subref{fig6.19d} Error vs. Iterations for mixed problem containing
a vertex singularity.}
\label{fig6.19}
\end{figure}

Next, we apply our method to an elliptic problem containing an edge singularity.
\pagebreak

\begin{guess8}
\bf{(Elliptic equation containing edge singularity)}
\end{guess8}
Consider the following Laplace equation:
\begin{align}\label{eq6.5}
-\triangle{u} & = 0 \quad\text{in} \quad\Omega^{(e)}_{1}, \notag \\
            u & = g \quad \text{on} \quad \partial\Omega^{(e)}_{1}.
\end{align}
where the domain $\Omega^{(e)}_{1}$ is shown in Figure \ref{fig6.20} and is given
by
\[\Omega^{(e)}_{1}=\{(r,\theta,x_{3}):\: 0\leq r\leq 1,\:0\leq\theta\leq\pi/2,\:
0\leq x_{3}\leq1\}.\]
We impose Dirichlet boundary conditions on all the faces marked as
$\Gamma_{i},i=1,\ldots,5$.

Let $w(r,\theta,x_{3})=r^{\frac{1}{3}}\sin(\frac{\theta}{3})x_{3}$. Then $w$ is the
exact solution of (\ref{eq6.5}) satisfying the Dirichlet boundary conditions
$u|_{\partial \Omega^{(e)}_{1}}=w$.
\begin{figure}[!ht]
\centering
\subfigure[]{
\includegraphics[scale = 0.65]{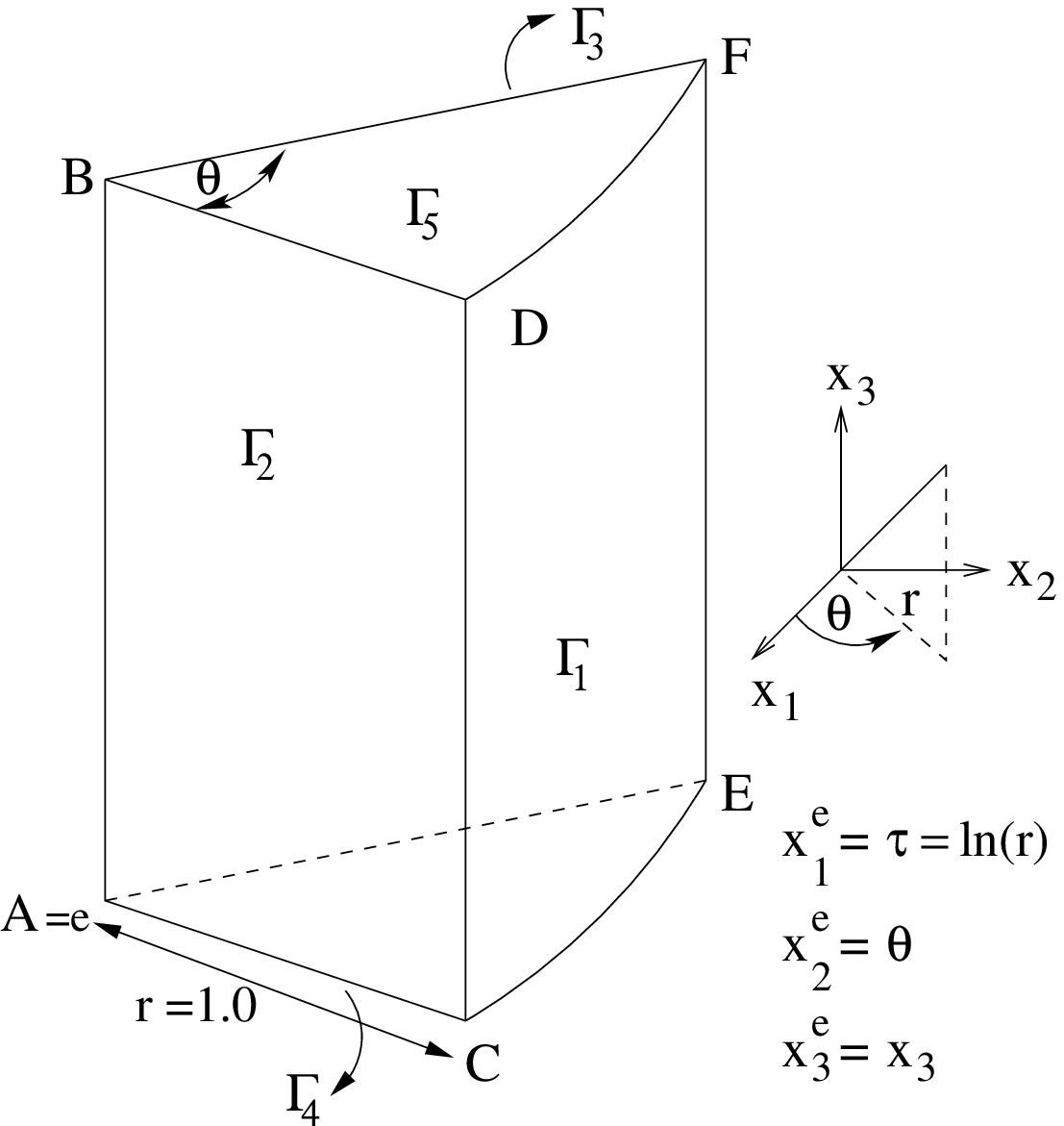}
\label{fig6.20a}
}
\hspace{1.0cm}
\subfigure[]{
\includegraphics[scale = 0.70]{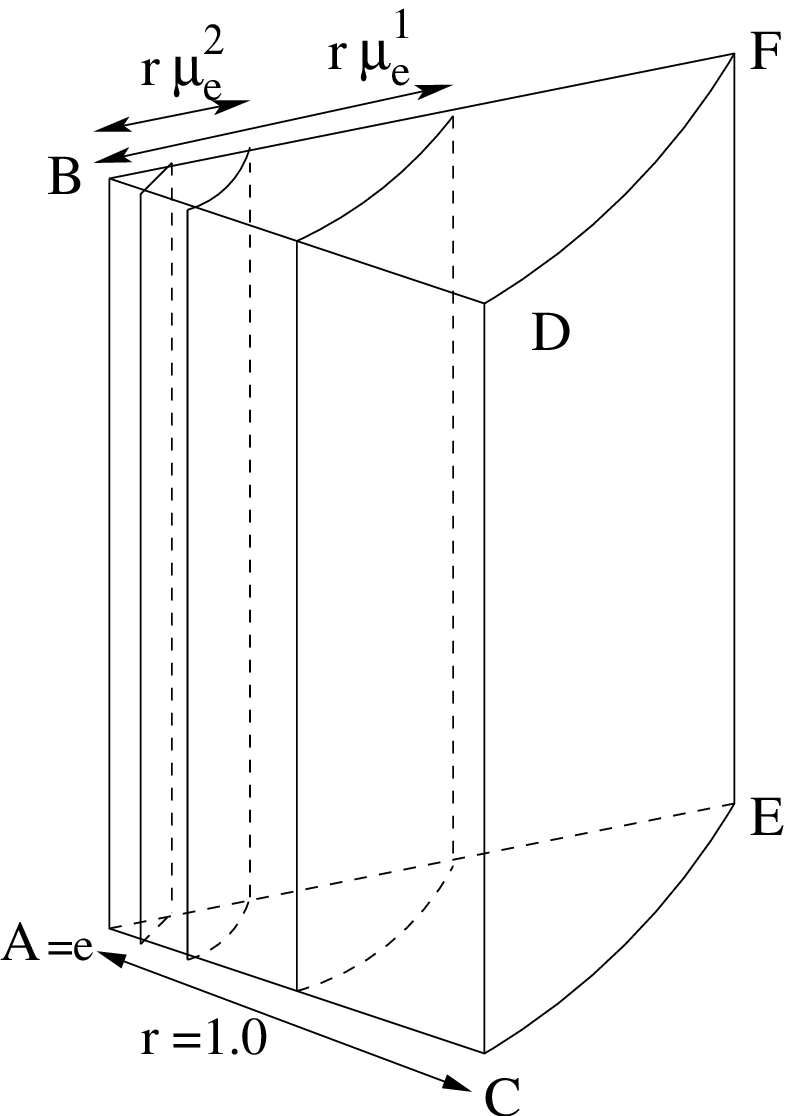}
\label{fig6.20b}
}
\caption[The domain $\Omega^{(e)}_{1}$ containing an edge singularity.]
{\subref{fig6.23a} The domain $\Omega^{(e)}_{1}$ containing an edge singularity,
\subref{fig6.23b} Geometrical mesh imposed on $\Omega^{(e)}_{1}$.}
\label{fig6.20}
\end{figure}
Note that $w$ has an edge singularity. Let $x_{1}^{e},x_{2}^{e}$ and $x_{3}^{e}$
denote the modified coordinates in the edge neighbourhood introduced in Chapter $2$.

Table \ref{tab6.18} contains the numerical results and it shows that $\approx 10^{-4} (\%)$
of relative error in the $H^{1}$-norm is achieved with $p=8$ and $N_{dof}\approx 4000$.
\begin{table}[!ht]
    \caption[Performance of the $h-p$ version for \textbf{Laplace} equation on
                 $\Omega^{(e)}_{1}$.]
                {Performance of the $h-p$ version for Laplace equation on $\Omega^{(e)}_{1}$}
    \begin{center}
        \label{tab6.18}
        \begin{tabular}{|c|c|c|c|}
            \hline
        \text{$p=W$} & \text{DOF($N_{dof}$)} & Iterations & \text{Relative Error($\%$)}\\
            \hline
             2 & 10 & 24 & 0.464542E+00 \\
            \hline
             3 & 57 & 34 & 0.131359E+00 \\
            \hline
             4 & 196 & 38 & 0.402204E-01 \\
            \hline
             5 & 504 & 49 & 0.123974E-01 \\
            \hline
             6 & 1085 & 58 & 0.364617E-02 \\
            \hline
             7 & 2064 & 68 & 0.107525E-02 \\
            \hline
             8 & 3592 & 75 & 0.315249E-03 \\
            \hline
             9 & 5840 & 85 & 0.920349E-04 \\
            \hline
             10 & 9001 & 98 & 0.279848E-04 \\
            \hline
        \end{tabular}
    \end{center}
\end{table}

The relative error against polynomial degree $p$ for $p=2,\ldots,10$ is drawn in Figure
\ref{fig6.21a}. It follows that the error decays exponentially. In Figures \ref{fig6.21c}
and \ref{fig6.21d} error as a function of degrees of freedom and iterations is plotted on
a $\log-\log$ scale.
\begin{figure}[!ht]
\centering
\subfigure[]{
\input{./edge_dirichlet_err}
\label{fig6.21a}
}
\hspace{-1.0cm}
\subfigure[]{
\input{./edge_dirichlet_itr}
\label{fig6.21b}
}
\vspace{1.0mm}
\subfigure[]{
\input{./edge_dirichlet_dof}
\label{fig6.21c}
}
\hspace{-1.0cm}
\subfigure[]{
\input{./edge_dirichlet_err_itr}
\label{fig6.21d}
}
\caption[Error vs. $p$, Iterations vs. $N$, Error vs. $N_{dof}$ and Error vs. Iterations
for Laplace equation containing an edge singularity.]
{\subref{fig6.21a} Error vs. $p$, \subref{fig6.21b} Iterations vs. $N$, \subref{fig6.21c}
Error vs. $N_{dof}$ and \subref{fig6.21d} Error vs. Iterations for Laplace equation containing
an edge singularity.}
\label{fig6.21}
\end{figure}

In the following example, we consider Laplace equation containing an edge
singularity with mixed boundary conditions which is similar to that examined
by Guo~\cite{GOH1}.
\begin{guess8}
\bf{(Laplace equation with mixed boundary conditions containing edge singularity:
A crack problem)}
\end{guess8}
Let us consider the Laplace equation:
\begin{align}\label{eq6.7}
-\triangle{u} & = 0 \quad\text{in} \quad\Omega^{(e)}_{2}, \notag \\
u & = 0 \quad\text{on} \quad {\Gamma^{[0]}}\subset\partial\Omega^{(e)}_{2},\notag\\
\frac{\partial u}{\partial \nu} & = h \quad \text{on} \quad {\Gamma^{[1]}}=\partial
\Omega^{(e)}_{2}\setminus{\Gamma^{[0]}},
\end{align}
where
\begin{align}
\Omega^{(e)}_{2} & =\{(r,\theta,x_{3}):\: 0\leq r\leq 1,\:0<\theta< 2\pi,\:
0\leq x_{3}\leq 1\},\notag\\
{\Gamma^{[0]}} & = \{(r,\theta,x_{3}):\: 0\leq r\leq 1,\:\theta=0,2\pi,\:0\leq x_{3}\leq 1\},
\notag\\
{\Gamma^{[1]}} & = {\Gamma^{[1]}_{1}}\cup{\Gamma^{[1]}_{2}} = {\Gamma_{2}}\cup
{\Gamma_{5}}\cup{\Gamma_{6}},\notag\\
\Gamma^{[1]}_{1} & = \{(r,\theta,x_{3}):\: r=1,\:0<\theta<2\pi,\:0<x_{3}<1\},\notag\\
{\Gamma^{[1]}_{2}} & = \{(r,\theta,x_{3}):\: 0<r<1,\:0<\theta<2\pi,\:x_{3}=0,1\},
\notag\\
h & = \left\{\begin{array}{ccc}
      \frac{1}{2}r^{\frac{1}{2}}\sin(\frac{\theta}{2}) & \mbox{on} \quad
      {\Gamma^{[1]}_{1}} \\
                                                    0  & \mbox{on} \quad
      {\Gamma^{[1]}_{2}}.
      \end{array}\right.
\end{align}
Then the function $w(r,\theta,x_{3})=r^{\frac{1}{2}}\sin(\frac{\theta}{2})$ is the exact
solution of (\ref{eq6.7}) satisfying the given boundary conditions $w|_{\Gamma^{[0]}}=0$
and $w|_{\Gamma^{[1]}}=h$. Note that $w$ has an edge singularity along $x_{3}$-axis. The
domain $\Omega^{(e)}_{2}$ and its plane section are shown in Figure \ref{fig6.22}.
\begin{figure}[!ht]
\centering
\subfigure[]{
\includegraphics[scale = 0.60]{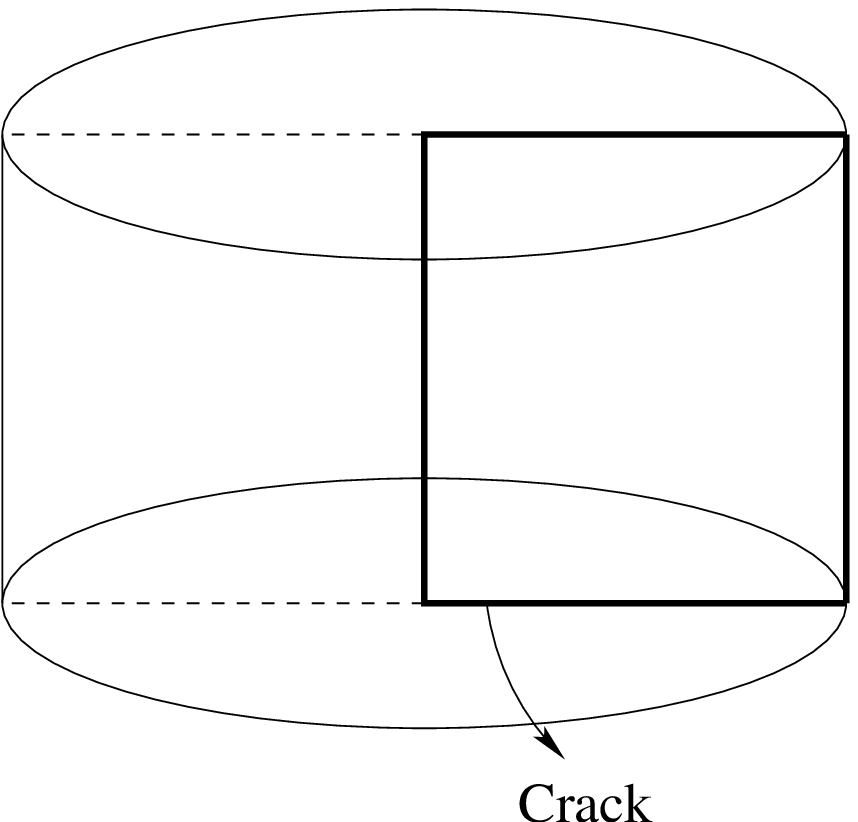}
\label{fig6.22a}
}
\hspace{1.0cm}
\subfigure[]{
\includegraphics[scale = 0.60]{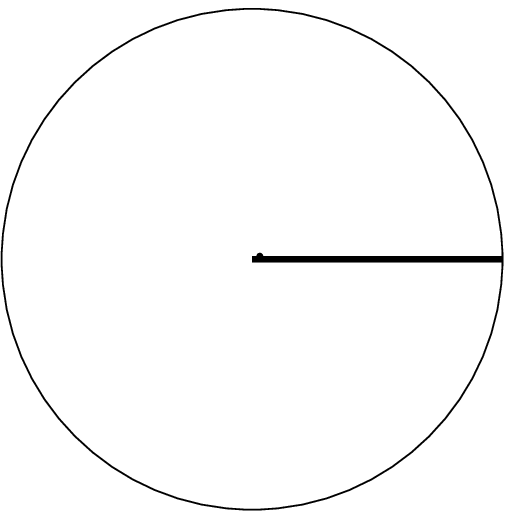}
\label{fig6.22b}
}
\caption[The domain $\Omega^{(e)}_{2}$ containing an edge singularity.]
{\subref{fig6.22a} The domain $\Omega^{(e)}_{2}$ with a crack, \subref{fig6.22b} Plane section
of the domain $\Omega^{(e)}_{2}$.}
\label{fig6.22}
\end{figure}

Note that the domain ${\Omega}^{(e)}_{2}$ is obtained by removing a vertical slit of unit
height from $r=0$ to $r=1$ i.e. a crack along $\theta=0$, from a cylinder of unit height and
radius $1$ with centre at the origin.

We impose a geometric mesh on ${\Omega}^{(e)}_{2}$ as shown in Figure \ref{fig6.23a} which
is refined geometrically in $r$ with mesh ratio $\mu_{e}=0.15$. All the elements in the
geometric mesh are of unit height. Cross section of the mesh on ${\Omega}^{(e)}_{2}$ is
shown in Figure \ref{fig6.23b}.
\begin{figure}[!ht]
\centering
\subfigure[]{
\includegraphics[scale = 0.60]{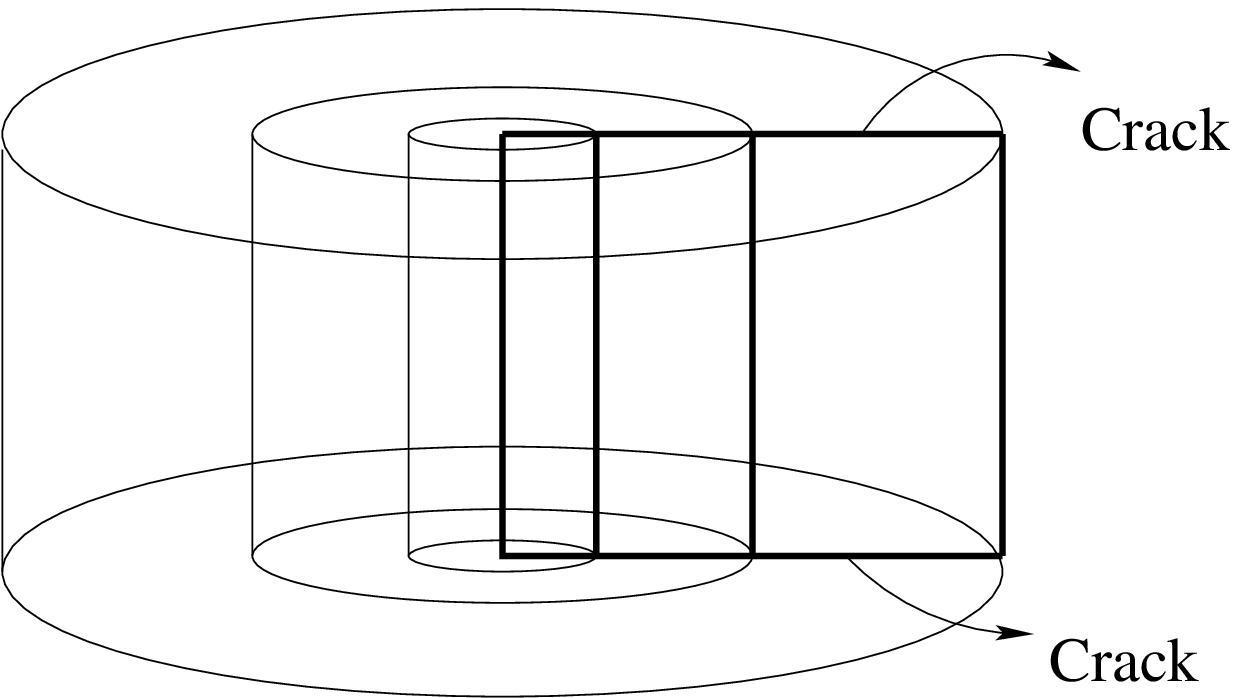}
\label{fig6.23a}
}
\hspace{1.0cm}
\subfigure[]{
\includegraphics[scale = 0.60]{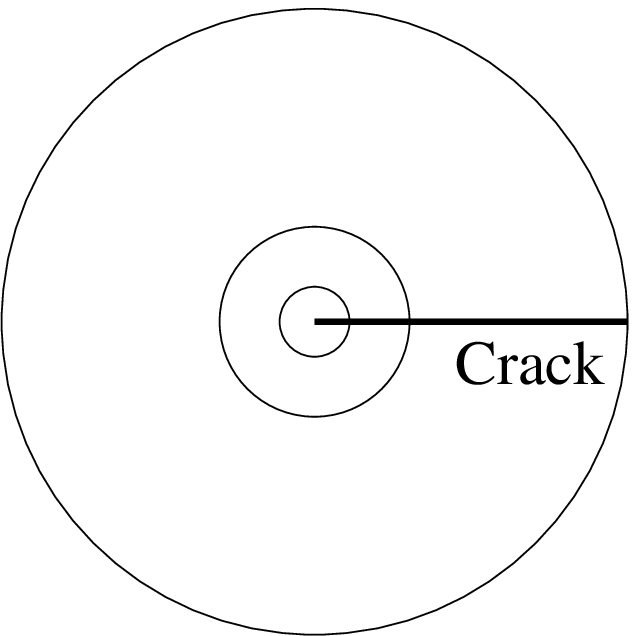}
\label{fig6.23b}
}
\caption[Geometric mesh imposed on ${\Omega}^{(e)}_{2}$.]
{\subref{fig6.23a} Geometric mesh imposed on ${\Omega}^{(e)}_{2}$,
\subref{fig6.23b} Cross section of the mesh imposed on ${\Omega}^{(e)}_{2}$.}
\label{fig6.23}
\end{figure}

Table \ref{tab6.19} contains the relative error and iterations for different values of
the polynomial order used in computations.

The relative error is plotted against polynomial degree $p$ in Figure \ref{fig6.24a}.
It is clear that the error decays rapidly with increase in the polynomial order.
In Figures \ref{fig6.24c} and \ref{fig6.24d} error as a function of degrees of
freedom and iterations is plotted on a $\log-\log$ scale.

\begin{table}[!ht]
    \caption[Performance of the method for \textbf{Laplace} equation with mixed boundary
                 conditions on $\Omega^{(e)}_{2}$ containing an \textbf{edge singularity}.]
                {Performance of the $h-p$ version for Laplace equation with mixed boundary
                 conditions on $\Omega^{(e)}_{2}$ containing an edge singularity}
    \begin{center}
        \label{tab6.19}
        \begin{tabular}{|c|c|c|c|}
            \hline
        \text{$p=W$} & \text{$N_{dof}$} & Iterations & \text{Relative Error($\%$)}\\
            \hline
             2 & 26 & 129 & 0.121468E+02 \\
            \hline
             3 & 138 & 167 & 0.979093E+01 \\
            \hline
             4 & 452 & 225 & 0.395164E+00 \\
            \hline
             5 & 1129 & 244 & 0.234473E+00  \\
            \hline
             6 & 2382 & 400 & 0.571179E-02 \\
            \hline
             7 & 4466 & 462 & 0.326770E-02 \\
            \hline
             8 & 7688 & 768 & 0.610394E-04 \\
            \hline
        \end{tabular}
    \end{center}
\end{table}
\begin{figure}[!ht]
\centering
\subfigure[]{
\input{./edge_mixed_old_err}
\label{fig6.24a}
}
\hspace{-1.0cm}
\subfigure[]{
\input{./edge_mixed_old_itr}
\label{fig6.24b}
}
\vspace{1.0mm}
\subfigure[]{
\input{./edge_mixed_old_dof}
\label{fig6.24c}
}
\hspace{-1.0cm}
\subfigure[]{
\input{./edge_mixed_old_err_itr}
\label{fig6.24d}
}
\caption[Error vs. $p$, Iterations vs. $N$, Error vs. $N_{dof}$ and Error vs. Iterations
for mixed problem containing an edge singularity.]
{\subref{fig6.24a} Error vs. $p$, \subref{fig6.24b} Iterations vs. $p$, \subref{fig6.24c}
Error vs. $N_{dof}$ and \subref{fig6.24d} Error vs. Iterations for mixed problem containing
an edge singularity.}
\label{fig6.24}
\end{figure}

\pagebreak

In the next example, we consider Poisson equation on a polyhedral domain containing vertex,
edge and vertex-edge singularities and analyze the performance of the method.
\begin{guess8}
\bf{(Poisson equation containing vertex-edge singularity)}
\end{guess8}
Let us consider the problem:
\begin{align}\label{eq6.8}
-\triangle{u} & = f \quad\text{in} \quad\Omega^{(v-e)}, \notag \\
            u & = g \quad\text{on} \quad\partial\Omega^{(v-e)}.
\end{align}
Let $w(\phi,\theta,\rho)=\rho^{3/4}(\sin\phi)^{1/2}\sin\left(\frac{\theta}{2}\right)$ and
$f=-\triangle w$. Then $w$ is the exact solution of (\ref{eq6.8}) in $\Omega^{(v-e)}$
satisfying the Dirichlet boundary conditions $u|_{\partial\Omega^{(v-e)}}=g$, where
$\Omega^{(v-e)}$ is the domain in Figure \ref{fig6.25} defined by
\[\Omega^{(v-e)}=\{(\phi,\theta,\rho):\: 0\leq\phi\leq\pi/6,\:0\leq\theta\leq
3\pi/2,\:\rho\leq 1\}.\]
Let us note that $w$ has a vertex singularity at the origin, an edge singularity along the
$z$-axis and a vertex-edge singularity. A geometrical mesh is imposed on $\Omega^{(v-e)}$
(Figure \ref{fig6.31}) in both $\phi$ (angular direction) and $x_{3}$ (radial direction)
variables with geometric mesh factors $\mu_{e}$ and $\mu_{v}$ respectively.
\begin{figure}[!ht]
\centering
\includegraphics[scale = 0.75]{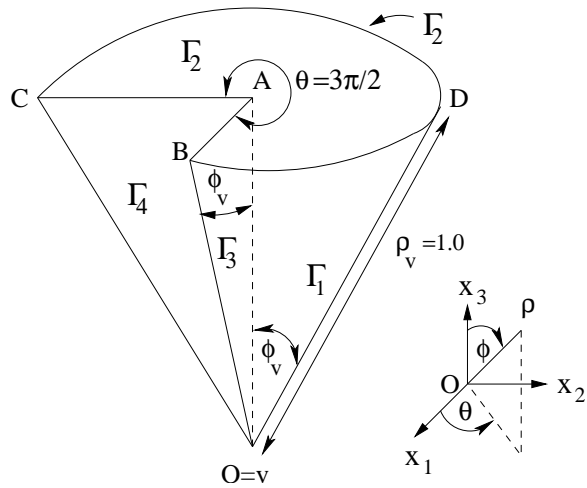}
\caption[The domain $\Omega^{(v-e)}$.]
{The domain $\Omega^{(v-e)}$.}
\label{fig6.25}
\end{figure}
\begin{figure}[!ht]
\centering
\includegraphics[scale = 0.75]{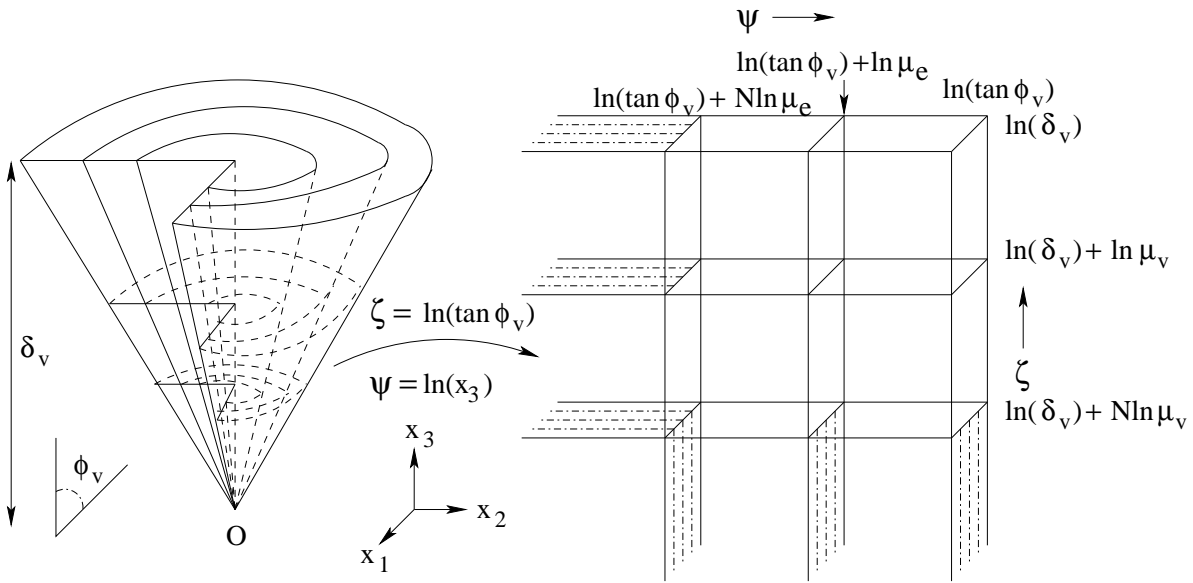}
\caption[Geometrical mesh imposed on $\Omega^{(v-e)}$ and the elements after mapping.]
{Mesh imposed on $\Omega^{(v-e)}$ and the elements after mapping.}
\label{fig6.31}
\end{figure}

To obtain the exponential convergence and efficiency of computations it is essential
to refine the mesh geometrically in both angular and radial direction along with
\begin{table}[!ht]
    \caption[Performance of the method for \textbf{Poisson} equation on $\Omega^{(v-e)}$
                 containing a \textbf{vertex-edge singularity}.]
                {Performance of the $h-p$ version for Poisson equation on $\Omega^{(v-e)}$
                 containing a vertex-edge singularity}
    \begin{center}
        \label{tab6.20}
        \begin{tabular}{|c|c|c|c|}
            \hline
        \text{$p=W$} & \text{$N_{dof}$} & Iterations & \text{Relative Error($\%$)}\\
            \hline
             2 & 19 & 15 & 0.473188E+01 \\
            \hline
             3 & 117 & 51 & 0.106067E+01 \\
            \hline
             4 & 592 & 61 & 0.363064E+00 \\
            \hline
             5 & 2025 & 69 & 0.165742E+00 \\
            \hline
             6 & 5436 & 77 & 0.693299E-01 \\
            \hline
        \end{tabular}
    \end{center}
\end{table}
\begin{figure}[!ht]
\centering
\subfigure[]{
\input{./vertex-edge_err}
\label{fig6.26a}
}
\hspace{-1.0cm}
\subfigure[]{
\input{./vertex-edge_itr}
\label{fig6.26b}
}
\subfigure[]{
\input{./vertex-edge_dof}
\label{fig6.26c}
}
\hspace{-1.0cm}
\subfigure[]{
\input{./vertex-edge_err_itr}
\label{fig6.26d}
}
\caption[Error vs. $p$, Iterations vs. $N$, Error vs. $N_{dof}$ and Error vs. Iterations
for Poisson equation containing a vertex-edge singularity.]
{\subref{fig6.26a} Error vs. $p$, \subref{fig6.26b} Iterations vs. $N$, \subref{fig6.26c}
Error vs. $N_{dof}$ and \subref{fig6.26d} Error vs. Iterations for Poisson equation
containing a vertex-edge singularity.}
\label{fig6.26}
\end{figure}
a proper choice of polynomial degree distribution because the solution possesses the combined
effect of a vertex singularity and an edge singularity.

The numerical results are given in Table \ref{tab6.20}. The relative error reduces to
nearly $0.07\%$ with $p=6$.

Figure \ref{fig6.26a} contains the relative error versus polynomial order $p$. The error
is plotted on a $\log-$scale. The curve is almost a straight line and it confirms the
theoretical estimates obtained. Further in Figure \ref{fig6.26b} the error is drawn as
a function of $N_{dof}$ on a $\log-\log$ scale.

\section{Conclusions and Future Work}

\subsection{Summary and conclusions}
There are three different types of singularities for elliptic problems on non-smooth
domains in $R^{3}$, namely, the vertex, the edge and the vertex-edge. In addition,
the solutions to these problems are anisotropic in the neighbourhoods of edges and
vertex-edges.

Among many approaches that have been attempted over the past three decades to provide
accurate and economical solutions to the elliptic boundary value problems containing
singularities, the two principal approaches are the adaptive mesh refinement (very small
elements near the singularities and large elements elsewhere) and the use of singular
basis functions. In three dimensional problems, the first approach suffers the problems
of limited storage capacity and high computational cost while ill-conditioning occur in
the latter in presence of singularities. As an extension of the second approach, the method
of auxiliary mapping (MAM) was introduced by Babu\v{s}ka and Oh (see~\cite{BABU,LUCA,OH1,OH2}
and references therein). The MAM has proven to be highly successful in dealing with
elliptic boundary value problems containing singularities in $R^{2}$~\cite{BABU,LUCA,OH1,OH2}
and $R^{3}$~\cite{GOH1,LEE} using $p$-version of the finite element method. However, the
success of this method depends on optimal choices of auxiliary mappings.

In the present work we have proposed an exponentially accurate $h-p$ spectral
element method for three dimensional elliptic problems on non-smooth domains
using parallel computers.

We choose as our solution the spectral element function which minimizes the sum
of a weighted squared norm of the residuals in the partial differential equations
and the squared norm of the residuals in the boundary conditions in fractional
Sobolev spaces and enforce continuity by adding a term which measures the jump
in the function and its derivatives at inter-element boundaries in fractional
Sobolev norms, to the functional being minimized.

The present study can be summarized as follows:-
\begin{itemize}
\item To resolve the singularities which arise in the neighbourhoods of vertices,
edges and vertex-edges we use modified systems of coordinates that are modified
versions of spherical and cylindrical coordinate systems. These modified coordinates
serve as our auxiliary mappings in different neighbourhoods.

\item The Sobolev spaces in vertex-edge and edge neighbourhoods are anisotropic and
become singular at the corners and edges.




\item We choose spectral element functions (SEF) which are fully non-conforming. i.e.
there are no set of common boundary values, and hence we do not need to solve the Schur
complement system.

\item The method is essentially a least-squares method and to solve the minimization
problem we need to solve the normal equations for the least-squares problem.

\item The residuals in the normal equations can be obtained without computing and storing
mass and stiffness matrices.

\item We use Preconditioned Conjugate Gradient Method (PCGM) to solve normal equations
using a block diagonal preconditioner. Moreover, there exists a new preconditioner which
can be diagonalized in a new set of basis functions, and hence it is easily inverted
on each element.

\item For Dirichlet problems the condition number of the preconditioner is $O((lnW)^2)$,
provided $W=O(e^{N^{\alpha}})$ for $\alpha<1/2$. However, it grows like $O(N^4)$ for
mixed problems.

\item For Dirichlet problems the overall complexity of the method is $O(N^{5}ln(N))$
operations on a parallel computer with $O(N^{2})$ processors to compute the solution.
For mixed problems it is equal to $O(N^{7})$ operations on a parallel computer with
$O(N^{2})$ processors.

\item Computational results for a number of test problems on smooth as well as non-smooth
domains confirm the estimates obtained for the error and computational complexity.

\item In each iteration of the PCGM we need to communicate the values of the function
and its derivatives on the boundary of the element between neighbouring elements.
Moreover, we need to compute two global scalars to update the approximate solution and
search direction. Thus, the inter processor communication involved is small.

\item Though the method is efficient and delivers exponential accuracy, more experience
is needed in computations and implementation. e.g. how to implement the distribution of
element degrees, what are the optimal degree factors, etc.
\end{itemize}

\subsection{Proposed future work}
Rapid growth of the factor $N^{4}$ for mixed problems creates difficulty in parallelizing
the numerical scheme. To overcome this difficulty another version of the method can be
defined in which we choose spectral element functions to be conforming only on the wirebasket
of the elements.

The values of the spectral element functions at the wirebasket of the elements constitute
the set of common boundary values and we need to solve the Schur complement matrix system
corresponding to the common boundary values. We intend to consider this in future work.

\chapcleardoublepage


\renewcommand{\theequation}{A.\arabic{equation}}
\setcounter{equation}{0}
\addcontentsline{toc}{chapter}{Appendix A}
\markboth{\small Appendix A}{\small Appendix A}
\begin{appendix}
\section*{Appendix A}
\subsection*{A.1}
\bf{Proof of Proposition \ref{prop2.2.1}}
\begin{prop2.2.1}
There exists a constant $\beta_v\in(0,{1/2})$ such that for all $0<\nu\leq \rho_v$
the estimate
\begin{align}\label{A.1}
\underset{\tilde{\Omega}^v\cap \{x^v:\: x_3^v\leq\ln(\nu)\}}\int\sum_{|\alpha|\leq m}
e^{x_3^v}\left|\:D_{x_v}^\alpha\left(w(x^v)-w_v\right)\:\right|^2\;dx^v
\leq C\,(d^m\;m!)^2\,\nu^{(1-2\beta_v)}
\end{align}
{\it holds for all integers $m\geq 1$}.
\end{prop2.2.1}
\begin{proof}
By Theorem $3.19$ of~\cite{BG1} for $\beta_v \in (0,1/2)$,
${\bf H}_{\beta_v}^{2,2}(\Omega^v)$ is embedded in ${\bf C}^0(\bar{\Omega}^v)$ and
\[\left\|w\right\|_{{\bf C}^0(\bar{\Omega}^v)}\leq C\;\left
\|w\right\|_{{\bf H}_{\beta_v}^{2,2}(\Omega^v)}\;.\]
Here $C$ denotes a constant and $\bar{\Omega}^v$ the closure of $\Omega^v$.
\newline
Hence we can define $w_v=w(v)$, the value of $w$ at the vertex $v$ and
$$|w_v|\leq C\;\|w\|_{{\bf H}_{\beta_v}^{2,2}(\Omega^v)}\:.$$
Let $\mathcal D^\alpha u = u_{\phi^{\alpha_1}\theta^{\alpha_2}\rho^{\alpha_3}}$.
Here $\alpha = \left(\alpha_1,\alpha_2,\alpha_3\right)$ and
$\alpha^{\prime} = \left(\alpha_1,\alpha_2\right)$.
\newline
Define
$$\Phiy_{\beta_v}^{\alpha,2}(x) =
\left\{\begin{array}{ccc}
\rho^{{\beta_v}+|\alpha|-2}\ &  for \;|\alpha|\geq 2\\
1 &  for \;|\alpha|< 2
\end{array}\right.
$$
as in $(2.3)$ of~\cite{BG1}. Let
\begin{eqnarray*}
{\mathcal H}_{\beta_v}^{k,2}(\Omega^v) &=&
\left\{u|\:\left\|u\right\| ^2_{{\bf \mathcal H}_{\beta_v}
^{k,2}(\Omega^v)} = \underset{|\alpha|\leq k} \sum
\left\|\Phiy_{\beta_v}^{\alpha,2}\:\rho^{-|\alpha^{\prime}|}
\mathcal D^{\alpha} u\:\right\|^2_{L^2{({\Omega}^v)}}<\infty\right\}
\end{eqnarray*}
and
\begin{eqnarray*}
\mathcal B^2_{\beta_v}(\Omega^v) &=& \left\{u|\:u \in
\mathcal H_{\beta_v} ^{k,2}(\Omega^v)\:\mbox{ for all }
\:k\geq 2, \left\|\Phiy_{\beta_v}^{\alpha,2}
\rho^{-|\alpha^{\prime}|}{\mathcal D}^{\alpha} u\;\right\|
^2_{L^2{({\Omega}^v)}}\leq C\,d^{\alpha}\,\alpha!
\right\}.
\end{eqnarray*}
Then from Theorem $4.13$ of~\cite{BG1} we have that $w\in\mathcal B^2_{\beta_v}(\Omega^v)$
iff $w\in{\bf B}^2_{\beta_v}(\Omega^v)$.
\newline
Hence
\begin{align}\label{A.2}
\underset{2\leq |\alpha|\leq m}{\sum}\underset{{\Omega}^{v}}{\int}
\left|\:\rho^{\beta_v-2}\rho^{\alpha_3}w_{\phi^{\alpha_1}\theta^{\alpha_2}
\rho^{\alpha_3}}\:\right|^2\rho^2\:\sin\phi\:d\rho\,d\phi\,d\theta
\leq (C\,d^m\,m!)^2\:. 
\end{align}
Define $\chi=\ln\rho$. Then
\[\frac{\partial}{\partial\chi}=\rho\frac{\partial}
{\partial\rho} \mbox{ and } \frac{d\rho}{\rho}=d\chi\:.\]
Now as in~\cite{BG1} it can be shown using (\ref{A.2}) that
\[\underset{2 \leq |\alpha| \leq m}{\sum}\underset{\tilde{\Omega}^{v}}
{\int} e^{(2\beta_v-1)\chi}\left|\:w_{\phi^{\alpha_1} \theta^{\alpha_2}
\chi^{\alpha_3}}\:\right|^2\:d\chi\,d\phi\,d\theta\leq(C\,d^m\,m!)^2.\]
Here $C$ and $d$ denote generic constants.
\newline
Hence
\begin{equation}\label{A.3}
\underset{2 \leq |\alpha| \leq m}{\sum}
\underset{\tilde{\Omega}^{v} \cap \{x^v:\:\chi\leq
\ln\nu\}}{\int}\left|\:D^{\alpha}_{x^v}w\:\right|^2\:dx^v
\leq(C\,d^m\,m!)^2\:\nu^{1-2\beta_v}.
\end{equation}

We now obtain estimates for $0\leq|\alpha|\leq 1$. By Lemma $5.5$ of~\cite{BG1}
since $w \in {\bf H}_{\beta_v}^{2,2}(\Omega^v)$ the estimate
\[\underset{\Omega^v}{\int}\rho^{2(\beta_v-2)}\left|w-w(v)\right|^2\:dx
\leq C\:\left\|u\right\|^2_{{\bf H}_{\beta_v}^{2,2}(\Omega^v)}\]
holds.
\newline
Hence
\[\underset{\tilde{\Omega}^v}{\int} e^{(2\beta_v-1)\chi}\left
|w-w_v\right|^2 \:d\chi\,d\phi\,d\theta \leq (C\,d^m\,m!)^2.\]
Thus we conclude that
\begin{equation}\label{A.4}
\underset{\tilde{\Omega}^{v} \cap \{x^v:\: \chi\leq\ln\nu\}}{\int}
|w-w_v|^2\:dx^v \leq C\,(d^m\,m!)^2\nu^{1-2\beta_v}.
\end{equation}

Finally, let ${\bf C}^2_{\beta_v}(\Omega^v)$ denote the set of functions
$u(x)\in{\bf C}^0(\bar{\Omega}^v)$ such that for all $|\alpha| \geq 0$
\begin{equation}\label{A.5}
\left|\:D_x^{\alpha}(u(x)-u(v))\:\right|\leq C\,d^{\alpha}\,
\alpha!\,\rho^{-(\beta_v+|\alpha|-{1/2})}(x).
\end{equation}
Then by Theorem $5.6$ of~\cite{BG1}, ${\bf B}^2_{\beta_v}(\Omega^v)\subseteq
{\bf C}^2_{\beta_v}(\Omega^v)$.
\newline
Now
\begin{equation}\label{A.6}
\rho \nabla_xu = Q^v \nabla_{x_v}u.
\end{equation}
Here \[Q^v = O^v P^v\]
where $O^v$ is the orthogonal matrix
\[O^v = \left[ \begin{array}{ccc}
\cos \phi \cos \theta   & -\sin \theta & \sin \phi \cos \theta \\
\cos \phi \sin \theta & \cos \theta & \sin \phi \sin \theta \\
-\sin \phi & 0 & \cos \phi
\end{array} \right]\]
and
\[P^v = \left[ \begin{array}{ccc}
1 & 0  & 0\\
0 & \frac{1}{\sin\phi} & 0 \\
0 & 0 & 1
\end{array} \right].\]
Now in $\Omega^{v}$
\[\phi^{v}<\phi<\pi-\phi^{v}.\]
Hence from (\ref{A.5}) and (\ref{A.6}) we can conclude that
\begin{equation}\label{A.7}
|\:\nabla_{x^v}w\:| \leq C\,\rho^{-\beta_v+{1/2}}\:.
\end{equation}
Using (\ref{A.7}) the estimate
\begin{equation}\label{A.8}
\underset{\tilde{\Omega}^{v} \cap \{x^v:\:\chi\leq
\ln\nu\}}{\int}\underset{|\alpha|=1}{\sum}\left|\:D_{x^v}^
{\alpha}w\:\right|^2\:dx^v \leq C\,\int_{-\infty}^{\ln\nu}
e^{(-2\beta_v+1)\chi} \:d\chi \leq C\,\nu^{1-2\beta_v}.
\end{equation}
follows.
\newline
Combining (\ref{A.3}), (\ref{A.4}) and (\ref{A.8}) we obtain the result.
\end{proof}

\subsection*{A.2}
\textbf{Proof of Proposition \ref{prop2.2.2}}
\begin{prop2.2.2}
Let $s(x_3)=w(x_1,x_2,x_3)|_{(x_1=0,x_2=0)}$. Then
\begin{equation}\label{A.9}
\int_{\delta_v}^{l_e-\delta_{v^{\prime}}} \underset
{k\leq m}{\sum}\left|\:\frac{d^k}{\left(dx_3^e\right)^k}
s(x_3^e)\right|^2\:dx_3^e \leq (C\,d^m\,m!)^2.
\end{equation}
Moreover there exists a constant $\beta_e \in (0,1)$ such that for $\mu\leq Z$
\begin{equation}\label{A.10}
\underset{\tilde{\Omega}^{e} \cap \{x^e: x_1^e\leq\ln\mu\}}{\int}\underset{|\alpha|
\leq m}{\sum}\left|\:D_{x^e}^{\alpha}\left(w(x^e)-s(x_3^e)\right)
\right|^2 \:dx^e\leq\:C\,(d^m\,m!)^2\mu^{2(1-\beta_e)}
\end{equation}
for all integers $m\geq 1$.
\end{prop2.2.2}
\begin{proof}
We denote by ${\bf C}^2_{\beta_e}(\Omega^e),\ \beta_e \in (0,1)$, the set of functions
$u \in {\bf C}^0(\bar{\Omega}^e)$ such that for $|\alpha| \geq 0$
\begin{equation}\label{A.11}
\left\|\:r^{\beta_e+\alpha_1+\alpha_2-1}D_x^{\alpha}
\left(u(x)-u(0,0,x_3)\right)\right\|_{{\bf C}^0(\bar{\Omega}^e)}
\leq C\,d^{\alpha}\,\alpha!
\end{equation}
and for $k\geq 0$
\begin{equation}\label{A.12}
\left\|\frac{d^k}{(dx_3)^k}u(0,0,x_3)\right\|_{{\bf C}^0
(\bar{\Omega}^e\cap\{x:\:x_1=x_2=0\})}\leq C\,d^k\,k!
\end{equation}
as in $(5.1)$ and $(5.2)$ of~\cite{BG1}. Here $\bar{\Omega}^e$ denotes the closure
of $\Omega^e$.
\newline
Then by Theorem $5.3$ of~\cite{BG1}, ${\bf B}^2_{\beta_e}(\Omega^e)\subseteq
{\bf C}^2_{\beta_e}(\Omega^e)$.
\newline
Now
\begin{align*}
x_1^e &= \tau \ =\:\ln r \\
x_2^e &= \theta \\
x_3^e &= x_3 \: .
\end{align*}
Define $s(x_3)=w(x_1,x_2,x_3)|_{(x_1=0,x_2=0)}$. Then (\ref{A.9}) follows
immediately from (\ref{A.12}) since $w\in{\bf B}^2_{\beta_e}(\Omega^e)$ and
hence $w \in{\bf C}^2_{\beta_e}(\Omega^e)$.
\newline
Let
$$p(x)=w(x)-s(x_3)\:.$$
Then by (\ref{A.11}) we have that
\begin{equation}\label{A.13}
\left\|\:r^{\beta_e+\alpha_1+\alpha_2-1}D_x^{\alpha}p(x)
\:\right\|_{{\bf C}^0(\tilde{\Omega}^e)}\leq C\,d^{\alpha}
\,\alpha!
\end{equation}
Now we can show just as in Theorem $4.1$ of~\cite{BG1} that
\begin{equation}\label{A.14}
\left\|r^{\beta_e-1}\:D_{x^e}^{\alpha}p(x^e)\:\right\|_{{\bf C}^0
(\tilde{\Omega}^e)}\leq C\,d^{\alpha}\,\alpha!
\end{equation}
using the estimate (\ref{A.13}).
\newline
Hence
\[\left|\:D_{x^e}^{\alpha} p(x^e)\:\right| \leq
C\,d^{\alpha}\,\alpha!\:e^{(1-\beta_e)x_1^{e}}\]
for $x^e \in \tilde{\Omega}^e$.
\newline
Using the above we conclude that
\begin{align*}
\underset{\tilde{\Omega}^{e} \cap \{x^e: x_1^e\leq \ln\mu\}}
{\int}\underset{|\alpha|\leq m}{\sum}\left|\:D_{x^e}^{\alpha}
p(x^e)\:\right|^2 dx^e \leq C\,(d^m\,m!)^2\:\int_{-\infty}
^{\ln\mu}e^{2(1-\beta_e)\tau}\,d\tau
\end{align*}
and this gives the required estimate (\ref{A.10}).
\end{proof}

\subsection*{A.3}
\textbf{Proof of Proposition \ref{prop2.2.3}}
\begin{prop2.2.3}
Let $w_v=w(v)$, the value of $w$ evaluated at the vertex $v$, and
$s(x_3)=w(x_1,x_2,x_3)|_{(x_1=0,x_2=0)}$. Then there exists a constant
$\beta_v \in (0,1/2)$ such that for any $0<\nu\leq\delta_v$
\begin{equation}\label{A.15}
\underset{-{\infty}}\int^{\ln\nu} e^{x_3^{v-e}}
\sum_{k\leq m}\left|\:D_{x_3^{v-e}}^k(s(x_3^{v-e})-w_v)\:\right|^2 dx_3^{v-e}
\leq C\,(d^m\,m!)^2\,\nu^{(1-2\beta_v)}\,.
\end{equation}
Moreover there exists a constant $\beta_e \in (0,1)$ such that for any
$0<\alpha \leq \tan\phi_v$ and $0<\nu \leq \delta_v$
\begin{align}\label{A.16}
&\underset{{\tilde{\Omega}^{v-e}} \cap \{x^{v-e}:
x_1^{v-e}<\ln\alpha, \ x_3^{v-e}<\ln\nu\}} \int e^{x_3^{v-e}}
\sum_{|\gamma|\leq m}\left|\:D_{x^{v-e}}^{\gamma}
\left(\:w(x_3^{v-e})-s(x_3^{v-e})\right)\:\right|^2\:dx_3^{v-e} \notag \\
&\hspace{4.5cm}\leq C\,(d^m\,m!)^2\,\alpha^{2(1-\beta_e)}\,\nu^{(1-2\beta_v)}
\end{align}
for all integers $m \geq 1$.
\end{prop2.2.3}
\begin{proof}
By ${\bf C}^2_{\beta_{v-e}}(\Omega^{v-e})$, where $\beta_{v-e}=(\beta_v,\beta_e),
 \ \beta_v \in(0,1/2)$ and $\beta_e \in (0,1)$
we denote the set of functions $u(x)\in {\bf C}^0(\bar{\Omega}^{v-e})$ such that
\begin{equation}\label{A.17}
\left\|\:\rho^{\beta_v+|\alpha|-{1/2}}\:
(\sin\phi)^{\beta_e+\alpha_1+\alpha_2-1}\:D^{\alpha}_x\left(u(x)-u(0,0,x_3)\right)
\:\right\|_{{\bf C}^0({\bar{\Omega}^{v-e}})}\leq C\,d^{\alpha}\,\alpha!
\end{equation}
and
\begin{equation}\label{A.18}
\left|\:|x_3|^{\beta_v+k-{1/2}} \frac{d^k}{dx_3^k}
\left(u(0,0,x_3)-u(v)\right)\right|_{{\bf C}^0\left(\bar{\Omega}
^{v-e}\cap\{x:\:x_1=x_2=0 \}\right)} \leq C\,d^k\,k!
\end{equation}
as described in $(5.26)$ and $(5.27)$ of~\cite{BG1}. Here $\bar{\Omega}^{v-e}$ denotes
the closure of $\Omega^{v-e}$.

Now by Theorem $5.9$ of~\cite{BG1}, ${\bf B}^2_{\beta_{v-e}}
(\Omega^{v-e})\subseteq {\bf C}^2_{\beta_{v-e}}(\Omega^{v-e})$.
Since $w\in {\bf B}^2_{\beta_{v-e}}(\Omega^{v-e})$ we conclude that
$w\in {\bf C}^2_{\beta_{v-e}}(\Omega^{v-e})$. Let $s(x_3)=w(0,0,x_3)$ and $w_v=w(v)$.
Then
\[\left|\:|x_3|^{\beta_v+k-{1/2}} \frac{d^k}{dx_3^k}\left(s(x_3)
-w_v\right)\:\right|\leq C\,d^k\,k!.\]
Now $x_3^{v-e}=\ln x_3$. Hence it can be shown as before that
\[\int_{-\infty}^{\ln\nu}\underset{k\leq m}{\sum}\left|\:
D^k_{x_3^{v-e}}\left(s(x_3^{v-e})-w_v\right)\:\right|^2
\:dx_3^{v-e}\leq C\,(d^m\,m!)^2\,\nu^{(1-2\beta_v)}.\]
Let $p(x)=w(x)-s(x_3)$. Then by (\ref{A.17}) we have that
\[\left\|\:\rho^{\beta_v+|\alpha|-{1/2}} (\sin\phi)^{\beta_e+
\alpha_1+\alpha_2-1}\:D^{\alpha}_x p(x)\:\right\|_{{\bf C}^0
({\bar{\Omega}}^{v-e})}\leq C\,d^{\alpha}\,\alpha!\,.\]
It can be shown as in Theorem $4.8$ of~\cite{BG1} that
\[\left\|\:\left(\rho^{\beta_v-{1/2}}(\sin\phi)^{\beta_e-1}\right)\,\rho^{\alpha_3}
(\sin\phi)^{\alpha_1}\,p_{\phi^{\alpha_1}\theta^{\alpha_2}\rho^{\alpha_3}}
\:\right\|_{{\bf C}^0(\breve{\Omega}^{v-e})} \leq C\,d^{\alpha}\,\alpha!\,.\]
Here $\breve{\Omega}^{v-e}$ is the image of $\bar{\Omega}^{v-e}$ in $(\phi,\theta,\rho)$
coordinates.
\newline
From the above the estimate
\begin{equation}\label{A.19}
\left\|\:e^{(\beta_v-{1/2})\chi}\:(\sin\phi)^{\beta_e-1}\,(\sin\phi)^{\alpha_1}
\,p_{\phi^{\alpha_1}\theta^{\alpha_2}\chi^{\alpha_3}}
\:\right\|_{{\bf C}^0(\widehat{\Omega}^{v-e})}\leq C\,d^{\alpha}\,\alpha!
\end{equation}
follows. Here $\widehat{\Omega}^{v-e}$ is the image of $\bar{\Omega}^{v-e}$ in $x^v$
coordinates, $x^v=(\phi,\theta,\chi)$ and $\chi=\ln\rho$.
\newline
Now the vertex-edge coordinates are defined as
\begin{align*}
x_1^{v-e} &= \psi \ =\;\ln(\tan\phi) \\
x_2^{v-e} &= \theta \\
x_3^{v-e} &= \zeta \ =\;\ln(x_3) \ =\;\chi+\ln(\cos\phi).
\end{align*}
Thus
\[\nabla_{x^v}u = J^{v-e}\nabla_{x^{v-e}}u\]
where
\[J^{v-e} =  \left[ \begin{array}{ccc}
\sec^2 \phi \cot \phi   &0&  -\tan \phi\\
0 & 1 & 0 \\
0 & 0 &  1
\end{array} \right].\]
Hence
\[\nabla_{x^{v-e}}u = (J^{v-e})^{-1} \nabla_{x^v}u.\]
Here
\[(J^{v-e})^{-1} =  \left[ \begin{array}{ccc}
\cos\phi \sin\phi  & 0 &  \sin^{2} \phi\\
0 & 1 & 0 \\
0 & 0 &  1
\end{array} \right].\]
Hence
\begin{align*}
\frac{\partial u}{\partial \psi} &= \cos\phi \sin\phi
\frac{\partial u}{\partial\phi}+\sin^2\phi\frac{\partial u}
{\partial\chi},\\
\frac{\partial u}{\partial \zeta} &= \frac{\partial u}
{\partial\chi}.
\end{align*}
From the above we obtain
\begin{equation}\label{A.20}
\frac{\partial^mu}{\partial\psi^m}=\sum_{k=1}^m\sum_{\alpha_1+\alpha_2=k}
\left(\sum_{j_1+j_2=2m-\alpha_1} a^m_{\alpha_1,\alpha_2,j_1,j_2}
\:(\cos\phi)^{j_1}\:(\sin\phi)^{j_2}\:((\sin\phi)^{\alpha_1})
u_{\phi^{\alpha_1}\chi^{\alpha_2}}\right)
\end{equation}
It can be shown that the coefficients
$a^m_{\alpha_1,\alpha_2,j_1,j_2}$ satisfy the recurrence relation
\begin{align}\label{A.21}
a^{m+1}_{\alpha_1,\alpha_2,j_1,j_2} &=a^m_{{\alpha_1-1},\alpha_2,{j_1-1},j_2}
+(\alpha_1+j_2)a^m_{\alpha_1,\alpha_2,{j_1-2},j_2}\\
&- j_1a^m_{\alpha_1,\alpha_2,j_1,{j_2-2}} + a^m_{\alpha_1,{\alpha_2-1},j_1,{j_2-2}}
\end{align}
for $|\alpha| \leq m.$
\newline
For $|\alpha|=m+1$
\begin{align}
a^{m+1}_{\alpha_1,\alpha_2,j_1,j_2} = \left\{\begin{array}{ccc}
1 & if \:j_1=\alpha_1, j_2=2m+2-\alpha_1 \notag \\
0 & \:otherwise\,.
\end{array}\right.
\end{align}
Since $0\leq \phi \leq \phi_v$, where $\phi_v<{\pi/2}$, we can conclude from (\ref{A.19})
that
\begin{equation}\label{A.23}
\left\|\:e^{(\beta_v-{1/2})\zeta}\:e^{(\beta_e-1)
\psi}\:(\sin\phi)^{\alpha_1}\:p_{\phi^{\alpha_1}\theta^{\alpha_2}
\chi^{\alpha_3}}\:\right\|_{{\bf C}^0_{(\widehat{\Omega}^{v-e})}}
\leq C\,d^{\alpha}\,\alpha! \:.
\end{equation}
Here $\widehat{\Omega}^{v-e}$ denotes the image of $\bar{\Omega}^{v-e}$ in $x^{v-e}$
coordinates. From (\ref{A.23}) the estimate
\begin{equation}\label{A.24}
\left \|\:e^{(\beta_v-{1/2})x_3^{v-e}}\:e^{(\beta_e-1)x_1^{v-e}}
\:D^{\alpha}_{x^{v-e}}\:p\:\right\|_{{\bf C}^0_{(\tilde
{\Omega}^{v-e})}} \leq C\,d^{\alpha}\,\alpha!\:.
\end{equation}
follows.

As in~\cite{BG1} we show (\ref{A.24}) for the cases $\alpha=(m,0,0)$, $\alpha=(0,m,0)$
and $\alpha=(0,0,m)$ since the general case can be shown in the same way. It is enough
to prove (\ref{A.24}) for $\alpha=(m,0,0)$ since the other two cases are trivial.
\newline
Let
\begin{equation}\label{A.25}
A_k^m = \sum_{\alpha_1+\alpha_2=k} \sum_{j_1+j_2=2m-\alpha_1}\left|
a^m_{\alpha_1,\alpha_2,j_1,j_2}\right|.
\end{equation}
Then
\begin{equation}\label{A.26}
A_m^m \leq 4^m.
\end{equation}
Moreover for $k<m$
\begin{equation}\label{A.27}
A_k^m \leq 4^m \frac{m!}{k!}.
\end{equation}
The proof is by induction. Using the recurrence relation (\ref{A.21}) we obtain
\begin{align}\label{A.28}
A_k^{m+1} &\leq 2m A_k^m + 2A_{k-1}^m \notag \\
& \leq 2m\left(\frac{4^m m!}{k!}\right)+2\left(\frac{4^m m!}
{(k-1)!}\right) \notag \\
& \leq 4^{m+1}\frac{(m+1)!}{k!}\left(\frac{2m}{4(m+1)}+
\frac{2k}{4(m+1)}\right) \notag \\
& \leq \frac{4^{m+1}}{k!}(m+1)!
\end{align}
Now using (\ref{A.19}), (\ref{A.20}) and (\ref{A.28}) it can
be shown that
\begin{align*}
&\left\|\:e^{(\beta_v-{1/2})x_3^{v-e}}\:e^{(\beta_e-1)x_1^{v-e}}
\:\frac{\partial^m p}{\partial \psi^m}\:\right\|_{{\bf C}^0(\tilde{\Omega}^{v-e})}\\
&\leq \sum_{k=1}^m\sum_{\alpha_1+\alpha_2=k}\sum_{j_1+j_2=2m-\alpha_1}
\left|a^m_{\alpha_1,\alpha_2,j_1,j_2}\right|
\:C\,d^{\alpha_1+\alpha_2}\,{\alpha_1}!{\alpha_2}!\\
&\leq \sum_{k=1}^m A^m_k(C\,d^k\,k!)\\
&\leq C\,d^m\,m!
\end{align*}
Here $C$ and $d$ denote generic constants.
\newline
The inequality (\ref{A.24}) is obtained in the same way. Now the estimate (\ref{A.16})
follows immediately from (\ref{A.24}).
\end{proof}

\subsection*{A.4}
\textbf{Proof of Proposition \ref{prop2.2.4}}
\begin{prop2.2.4}The estimate
\begin{equation}\label{A.29}
\underset{{\Omega^r}}\int\sum_ {|\alpha|\leq m}\left|\:D_x^\alpha w(x)\:\right|^2
\leq C\,(d^m\,m!)^2
\end{equation}
holds for all integers $m\geq 1$.
\end{prop2.2.4}
\begin{proof}
Now $w(x)$ is analytic in an open neighbourhood of $\bar{\Omega}^r$. Hence (\ref{A.29})
follows.
\end{proof}
\end{appendix}
\thispagestyle{empty}
\chapcleardoublepage

\renewcommand{\theequation}{B.\arabic{equation}}
\renewcommand{\thefigure}{B.\arabic{figure}}
\setcounter{equation}{0}
\setcounter{figure}{0}
\addcontentsline{toc}{chapter}{Appendix B}
\markboth{\small Appendix B}{\small Appendix B}
\begin{appendix}
\section*{Appendix B}
\subsection*{B.1}
\textbf{Proof of Lemma \ref{lem3.2.6}}
\begin{lem3.2.6}
Let $\Gamma_{k,i}^v = \Gamma_{q,r}^{v-e}$. Then the following
identity holds.
\begin{align}
&\sin^2 (\phi_v) \oint_{\partial \tilde{\Gamma}^v_{k,i}} e^{x_3^v}
\sin(x_1^v) \left(\frac{\partial u}{\partial \nb^v}\right)_{A^v}
\left(\frac{\partial u}{\partial \nuw^v}\right)_{A^v} \:d s^v \notag \\
&- 2\sin^2 (\phi_v) \int_{ \tilde{\Gamma}^v_{k,i}} e^{x_3^v}
\sin(x_1^v) \sum_{j=1}^2 \left(\frac{\partial u}{\partial
\taux_j^v}\right)_{A^v} \frac{\partial}{\partial s^v_j}
\left(\left(\frac{\partial u}
{\partial \nuw^v}\right)_{A^v}\right) \:d\sigma^v \notag \\
&= -\oint_{\partial \tilde{\Gamma}^{v-e}_{q,r}} e^{x_3^{v-e}}
\left(\frac{\partial u}{\partial \nb^{v-e}}\right)_{A^{v-e}}
\left(\frac{\partial u}{\partial \nuw^{v-e}}\right)_{A^{v-e}}
\:ds^{v-e} \notag\\
&+ 2\int_{ \tilde{\Gamma}^{v-e}_{q,r}} e^{x_3^{v-e}} \sum_{j=1}^2
\left(\frac{\partial u}{\partial \taux_j^{v-e}}\right)_{A^{v-e}}
\frac{\partial}{\partial s^{v-e}_j} \left(\left(\frac{\partial u}
{\partial \nuw^{v-e}}\right)_{A^{v-e}}\right) \:d\sigma^{v-e}.\notag
\end{align}
\end{lem3.2.6}
\begin{proof}
Let
\begin{align}\label{B.1}
& x_1^v = \phi \notag \\
& x_2^v = \theta \notag \\
& x_3^v = {\mathcal X} = \ln\rho\:.
\end{align}
\begin{figure}[!ht]
\centering
\includegraphics[scale = 0.60]{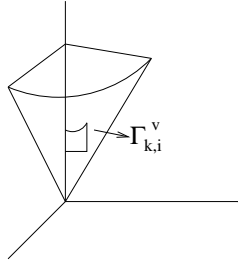}
\caption{Interior common boundary face $\Gamma_{k,i}^v={\Gamma}_{q,r}^{v-e}$.}
\label{figB.1}
\end{figure}
and
\begin{align} \label{B.2}
& x_1^{v-e} = \psi = \ln (\tan x_1^v) \notag \\
& x_2^{v-e} = \theta \notag \\
& x_3^{v-e} = \zeta = \ln x_3 = x_3^v + \ln \cos (x_1^v)\:.
\end{align}
Clearly
\begin{align} \label{B.3}
\nabla_{x^v} u = J^{v-e} \nabla_{x^{v-e}} u
\end{align}
where
\begin{align} \label{B.4}
J^{v-e} =  \left[ \begin{array}{ccc}
\sec^2 \phi \cot \phi   &0&  -\tan \phi\\
0 & 1 & 0 \\
0 & 0 &  1
\end{array} \right].
\end{align}
Now if ${\bf dx}^v$ is a tangent vector to a curve in $x^v$ coordinates then its image
in $x^{v-e}$ coordinates is given by ${\bf dx}^{v-e}$ where
\begin{equation} \label{B.5}
{\bf dx}^{v-e} = (J^{v-e})^T {\bf dx}^v.
\end{equation}
Clearly the first fundamental form $(d s^v)^2$ in $x^v $ coordinates is
\begin{equation} \label{B.6}
(d s^v)^2 = ({\bf dx}^v)^T ({\bf dx}^v) = ({\bf dx}^{v-e})^T
(J^{v-e})^{-1} (J^{v-e})^{-T} {\bf dx}^{v-e}.
\end{equation}
Now
\begin{align*}
(J^{v-e})^{-1}  =
\left[ \begin{array}{ccc}
\cos \phi \sin \phi   &0&  \sin^2 \phi\\
0 & 1 & 0 \\
0 & 0 &  1
\end{array} \right].
\end{align*}
Hence
\begin{align} \label{B.7}
(ds^v)^2 &= \sin^2 \phi (dx_1^{v-e})^2 + (d x_2^{v-e})^2 + (d
x_3^{v-e})^2 +2 \sin^2 \phi dx_1^{v-e} d x_3^{v-e} \notag \\
&= \sin^2 \phi d \psi^2 + d \theta^2 + d \zeta^2 + 2 \sin^2 \phi
\:d\psi \:d\zeta.
\end{align}
Moreover, on $\tilde{\Gamma}^v_{k,i}$,
\begin{equation} \label{B.8}
d \sigma^v = d \theta d \zeta
\end{equation}
since by (\ref{B.7})
\[(ds^v)^2 = d \theta^2 + d\zeta^2\]
on $\tilde{\Gamma}^v_{k,i}$. Choose ${\bf\tau}_1^{v-e}=(0,1,0)^T$ and
${\bf\tau}_2^{v-e}= -(0,0,1)^T$. These are orthogonal unit vectors on
$\tilde{\Gamma}^{v-e}_{q,r}$ since
\begin{equation} \label{B.9}
(ds^{v-e})^2 = d \psi^2 +d\theta^2+d\zeta^2.
\end{equation}
Define
\begin{equation} \label{B.10}
\taux_1^v = (J^{v-e})^{-T} \taux_1^{v-e}
\end{equation}
\begin{equation} \label{B.11}
\taux_2^v = - (J^{v-e})^{-T} \taux_2^{v-e}.
\end{equation}
Let $\nuw^{v-e} = -(1,0,0)^{T}$ denote the unit normal vector on
$\tilde{\Gamma}^{v-e}_{q,r}$ and
\begin{equation} \label{B.12}
\mc^{v-e} = (\sec^2\phi\cot\phi, 0, -\tan\phi)^T.
\end{equation}
Then
\begin{equation} \label{B.13}
\nuw^v = (1,0,0)^T = (J^{v-e})^{-T} \mc^{v-e}
\end{equation}
is the unit normal to $\partial \tilde{\Gamma}^v_{k,i}$. Finally let
${\bf ds}^{v-e} = (0, d \theta, d \zeta)^T$ denote a tangent vector field on
$\tilde{\Gamma}^{v-e}_{q,r}$. Define
\begin{equation} \label{B.14}
ds^{v-e} = \sqrt{d \theta^2 + d \zeta^2}
\end{equation}
\begin{equation} \label{B.15}
{\bf ds}^v = (J^{v-e})^{-T} {\bf ds}^{v-e} =
(0, d \theta, d {\mathcal X})^T
\end{equation}
and
\begin{equation} \label{B.16}
ds^v = \sqrt{d \theta^2 + d {\mathcal X}^2}.
\end{equation}
Let
\begin{equation} \label{B.17}
\nb^{v-e} = (0, d \zeta, - d \theta)^T/\sqrt{d \zeta^2 + d
\theta^2}
\end{equation}
be the unit outward normal to $\partial \tilde{\Gamma}^{v-e}_{q,r}$. Then
\begin{equation} \label{B.18}
\nb^v= (J^{v-e})^{-T} \nb^{v-e}
\end{equation}
is the unit outward normal to $\partial \tilde{\Gamma}^v_{k,i}$.
\newline
Now
\begin{align}
\left(\frac{\partial u}{\partial \nb^v}\right)_{A^v}
&= (\nb^v)^T (A^v) \nabla_{x^v}u \notag \\
&= (\nb^{v-e})^T(J^{v-e})^{-1}(A^v) \nabla_{x^v}u \notag \\
&=\frac{(\nb^{v-e})^T}{\sin^2 \phi}
\left((J^{v-e})^{-1}(J^{v-e})^{-T}
\right)(\sin^2\phi (J^{v-e})^T A^v J^{v-e}) \nabla_{x^{v-e}}u \notag \\
&= \frac{(\nb^{v-e})^T}{\sin^2 \phi}  \left[ \begin{array}{ccc}
\sin^2 \phi   &0&  \sin^2 \phi \\
0 & 1 & 0 \\
\sin^2 \phi & 0 &  1
\end{array} \right] A^{v-e} \nabla_{x^{v-e}}u \notag \\
&= \frac{(-d\theta\sin^2\phi, d\zeta, -d\theta)^T}{\sin^2 \phi
\sqrt{d \zeta^2 + d \theta^2}} A^{v-e} \nabla_{x^{v-e}}u. \notag
\end{align}
Hence referring to Figure \ref{figB.2},
\begin{figure}[!ht]
\centering
\includegraphics[scale = 0.60]{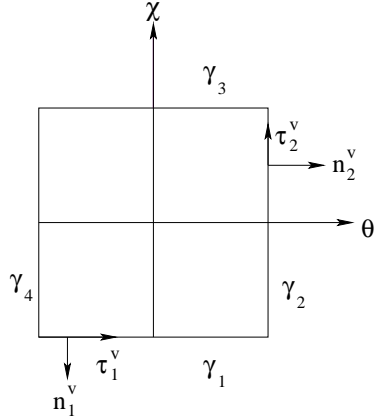}
\caption{The face $\tilde{\Gamma}_{k,i}^{v}$}
\label{figB.2}
\end{figure}
\begin{align} \label{B.19}
\left(\frac{\partial u}{\partial \nb_1^v}\right)_{A^{v}} &=
\left(\frac{\partial u}{\partial \nuw^{v-e}}\right)_{A^{v-e}} +
\frac{1}{\sin^2 \phi}\left(\frac{\partial u}{\partial
\taux_2^{v-e}}\right)_{A^{v-e}}
\end{align}
and
\begin{align} \label{B.20}
\left(\frac{\partial u}{\partial \nb_2^v}\right)_{A^{v}} &=
\frac{1}{\sin^2 \phi}\left(\frac{\partial u}{\partial
\taux_1^{v-e}}\right)_{A^{v-e}}\:.
\end{align}
Moreover
\[\left(\frac{\partial u}{\partial \nuw^v}\right)_{A^v}
= \frac{(\sec^2 \phi \cot \phi, 0, -\tan \phi)}{\sin^2 \phi}
 \left[ \begin{array}{ccc}
\sin^2 \phi   &0&  \sin^2 \phi\\
0 & 1 & 0 \\
\sin^2 \phi & 0 &  1
\end{array} \right] A^{v-e} \nabla_{x^{v-e}}u\:. \]
Hence
\begin{equation} \label{B.21}
\left(\frac{\partial u}{\partial \nuw^v}\right)_{A^v}
= -\cot \phi \left(\frac{\partial u}{\partial \nuw^{v-e}}
\right)_{A^{v-e}}.
\end{equation}
Finally
\begin{align} \label{B.22}
\left(\frac{\partial u}{\partial \taux_1^v}\right)_{A^v} &=
\frac{(0,1,0)^T}{\sin^2 \phi} \left[ \begin{array}{ccc}
\sin^2 \phi   &0&  \sin^2 \phi\\
0 & 1 & 0 \\
\sin^2 \phi & 0 &  1
\end{array} \right] A^{v-e} \nabla_{x^{v-e}}u \notag\\
&= \frac{1}{\sin^2 \phi}\left(\frac{\partial u}{\partial
\taux_1^{v-e}}\right)_{A^{v-e}}.
\end{align}
And
\begin{align} \label{B.23}
\left(\frac{\partial u}{\partial \taux_2^v}\right)_{A^v} &=
\frac{(0,0,1)^T}{\sin^2 \phi} \left[
\begin{array}{ccc}
\sin^2 \phi   &0&  \sin^2 \phi\\
0 & 1 & 0 \\
\sin^2 \phi & 0 &  1
\end{array} \right] A^{v-e} \nabla_{x^{v-e}}u \notag \\
&= - \left(\frac{\partial u}{\partial \nuw^{v-e}}\right)_{A^{v-e}}
- \frac{1}{\sin^2 \phi} \left(\frac{\partial u}{\partial
\taux_2^{v-e}}\right)_{A^{v-e}}.
\end{align}
Now referring to Figure \ref{figB.3}
\begin{figure}[!ht]
\centering
\includegraphics[scale = 0.60]{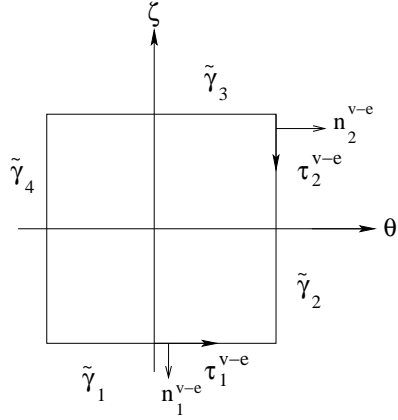}
\caption{The face $\tilde{\Gamma}_{q,r}^{v-e}$}
\label{figB.3}
\end{figure}
\begin{align}
&\oint_{\partial\tilde{\Gamma_{k,i}^v}} \sin^2 (\phi_v)
e^{x_3^{v}} \sin(x_1^v) \left(\frac{\partial u}{\partial
\nb^v}\right)_{A^v}\left(\frac{\partial u}{\partial \nuw^v}
\right)_{A^v}\:ds^v \notag \\
&\quad =\int_{\gamma_1} \rho \sin^3 (\phi_v) \left(\frac{\partial
u} {\partial \nb_1^v}\right)_{A^v}\left(\frac{\partial u}{\partial
\nuw^v}\right)_{A^v} \:d\theta \notag \\
&\quad + \int_{\gamma_2} \rho \sin^3 (\phi_v) \left(\frac{\partial
u} {\partial \nb_2^v}\right)_{A^v}\left(\frac{\partial u}{\partial
\nuw^v}\right)_{A^v} \:d {\mathcal X} \notag \\
&\quad - \int_{\gamma_3} \rho \sin^3 (\phi_v) \left(\frac{\partial
u} {\partial \nb_1^v}\right)_{A^v}\left(\frac{\partial u}{\partial
\nuw^v}\right)_{A^v} \:d\theta \notag \\
&\quad - \int_{\gamma_4} \rho \sin^3 (\phi_v) \left(\frac{\partial
u} {\partial \nb_2^v}\right)_{A^v}\left(\frac{\partial u}{\partial
\nuw^v}\right)_{A^v} \:d{\mathcal X} \notag \\
&\quad= -\int_{\tilde{\gamma}_1} \rho \sin^3 (\phi_v) \cot(\phi_v)
\left(\frac{\partial u}{\partial \nuw^{v-e}}\right)^2_{A^{v-e}}
\:d\theta \notag \\
&\quad - \int_{\tilde{\gamma}_1} \rho \cos(\phi_v)
\left(\frac{\partial u} {\partial
\taux_2^{v-e}}\right)_{A^{v-e}}\left(\frac{\partial u}{\partial
\nuw^{v-e}}\right)_{A^{v-e}} \:d\theta \notag 
\end{align}
\begin{align}\label{eqB.24}
&\quad - \int_{\tilde{\gamma}_2} \rho \cos (\phi_v)
\left(\frac{\partial u}{\partial\taux_1^{v-e}}\right)_{A^{v-e}}
\left(\frac{\partial u}{\partial \nuw^{v-e}}\right)_{A^{v-e}}
\:d\zeta \notag \\
&\quad + \int_{\tilde{\gamma}_3}\rho \sin^3(\phi_v)
\cot(\phi_v)\left(\frac{\partial u}{\partial
\nuw^{v-e}}\right)^2_{A_{v-e}}
\:d {\theta} \notag \\
&\quad + \int_{\tilde{\gamma}_3} \rho \cos(\phi_v)
\left(\frac{\partial u}{\partial
\taux_2^{v-e}}\right)_{A^{v-e}}\left(\frac{\partial u}{\partial
\nuw^{v-e}}\right)_{A^{v-e}} \:d{\theta} \notag \\
&\quad + \int_{\tilde{\gamma}_4} \rho \cos(\phi_v)
\left(\frac{\partial u}{\partial
\taux_1^{v-e}}\right)_{A^{v-e}}\left(\frac{\partial u}{\partial
\nuw^{v-e}}\right)_{A^{v-e}} \:d{\zeta}\;.
\end{align}
Hence
\begin{align}\label{eqB.25}
&\oint_{\partial\tilde\Gamma_{k,i}^v}\sin^2(\phi_v)e^{x_3^v}
\sin(x_1^v)\left(\frac{\partial u}{\partial \nb^v}\right)_{A^v}
\left(\frac{\partial u}{\partial \nuw^v}\right)_{A^v}\:d s^v \notag \\
&\quad=-\oint_{\partial\tilde\Gamma_{q,r}^{v-e}}e^{x_3^{v-e}}\left
(\frac{\partial u}{\partial \nb^{v-e}}\right)_{A^{v-e}}\left
(\frac{\partial u}{\partial \nuw^{v-e}}\right)_{A^{v-e}}\:d s^{v-e}
\notag \\
&\quad -\int_{\tilde\gamma_1}\rho\sin^3(\phi^v)\cot(\phi_v)\left
(\frac{\partial u}{\partial \nuw^{v-e}}\right)^2_{A^{v-e}}
\:d {\theta} \notag \\
&\quad+\int_{\tilde\gamma_3}\rho\sin^3(\phi_v)\cot(\phi_v)\left
(\frac{\partial u}{\partial \nuw^{v-e}}\right)^2 \:d{\theta}
\end{align}
Next,
\begin{align}\label{eqB.26}
&2 \sin^2 (\phi_v) \int_{\tilde{\Gamma}^v_{k,i}}
e^{x_3^{v}} \sin(x_1^v) \left(\frac{\partial u}{\partial
\taux_1^v}\right)_{A^v} \frac{\partial} {\partial
s_1^v}\left(\left(\frac{\partial u}{\partial \nuw^v}\right)
_{A^v}\right) \:d \sigma^v \notag \\
&\quad= 2 \sin^2 (\phi_v) \int_{\tilde{\Gamma}^v_{k,i}}
\rho \sin(\phi_v) \left(\frac{\partial u}{\partial
\taux_1^v}\right)_{A^v} \frac{\partial} {\partial
s_1^v}\left(\left(\frac{\partial u}{\partial \nuw^v}\right)
_{A^v}\right) \:d \sigma^v \notag \\
&\quad= -2 \int_{\tilde{\Gamma}^{v-e}_{q,r}}\rho\sin^3
(\phi_v)\frac{1}{\sin^2 (\phi_v)} \left(\frac{\partial u}
{\partial \taux_1^{v-e}}\right)_{A^{v-e}} \notag \\
&\quad\quad\frac{\partial}
{\partial \theta}\left(\cot(\phi_v)\left(\frac{\partial u}
{\partial \nuw^{v-e}}\right)_{A^{v-e}}\right) \:d\theta
\:d\zeta \notag \\
&\quad= -2 \int_{\tilde\Gamma_{q,r}^{v-e}} e^{x_3^{v-e}}\left
(\frac{\partial u}{\partial \taux_1^{v-e}}\right)_{A^{v-e}}
\frac{\partial}{\partial s_1^{v-e}}\left(\left(\frac{\partial
u}{\partial \nuw^{v-e}}\right)_{A^{v-e}}\right) \:d \sigma^{v-e}.
\end{align}
Moreover
\begin{align}\label{eqB.27}
&2 \sin^2 (\phi_v) \int_{\tilde{\Gamma}^v_{k,i}} e^{x_3^{v}}
\sin(x^v_1) \left(\frac{\partial u}{\partial
\taux_2^{v}}\right)_{A^v} \frac{\partial} {\partial
s_2^{v}}\left(\left(\frac{\partial u}{\partial
\nuw^v}\right)_{A^{v}}\right) \:d \sigma^{v} \notag \\
&\quad= 2 \sin^2 (\phi_v) \int_{\tilde{\Gamma}^v_{q,r}}
\rho\sin(\phi_v) \left(\frac{\partial u}{\partial
\taux_2^{v}}\right)_{A^v} \frac{\partial}{\partial s_2^{v}}
\left(\left(\frac{\partial u}{\partial
\nuw^v}\right)_{A^{v}}\right)\:d \sigma^{v} \notag \\
&\quad= 2 \int_{\tilde{\Gamma}^{v-e}_{q,r}}\left(\rho \sin^3
(\phi_v)\left(\left(\frac{\partial u}{\partial \nuw^{v-e}}\right)
_{A^{v-e}}\right.\right. \notag \\
&\quad+ \frac{1}{\sin^2 (\phi_v)} \left.\left.\left(\frac{\partial u}
{\partial\taux_2^{v-e}}\right)_{A^{v-e}}\right)\right)\frac
{\partial}{\partial {\zeta}}\left(\cot(\phi_v) \left(\frac
{\partial u}{\partial\nuw^{v-e}}\right)_{A^{v-e}}\right)
\:d{\theta}\:d {\zeta} \notag 
\end{align}
\begin{align}
&= \int_{\tilde{\Gamma}^{v-e}_{q,r}}\rho\sin^3
(\phi_v)\cot(\phi_v) \frac{\partial}{\partial \zeta}\left(\frac
{\partial u}{\partial \nuw^{v-e}}\right)^2_{A^{v-e}} \:d
\theta \:d \zeta \notag \\
&+ 2 \int_{\tilde{\Gamma}^{v-e}_{q,r}} e^{\zeta}
\left(\frac{\partial u} {\partial \taux_2^{v-e}}
\right)_{A^{v-e}}\frac{\partial }{\partial \zeta}
\left(\frac{\partial u}{\partial\nuw^{v-e}}\right)
_{A^{v-e}}\:d \theta d \zeta \notag \\
&= - \int_{\tilde{\gamma}_1} \rho \sin^3 (\phi_v) \cot
(\phi_v)\left(\frac{\partial u}{\partial \nuw^{v-e}}\right)
_{A^{v-e}}^2 \:d\theta \notag \\
&+ \int_{\tilde{\gamma}_3} \rho \sin^3 (\phi_v) \cot
(\phi_v)\left(\frac{\partial u}{\partial \nuw^{v-e}}\right)
_{A^{v-e}}^2 \:d\theta \notag \\
&- 2 \int_{\tilde{\Gamma}_{q,r}^{v-e}}{e^{{x_3}^{v-e}}}
\left(\frac{\partial u}{\partial\taux_2^{v-e}}\right)_{A^{v-e}}
\frac{\partial}{\partial s_2^{v-e}}\left(\frac{\partial u}
{\partial \nuw^{v-e}}\right)_{A^{v-e}}\:d\sigma^{v-e} \:.
\end{align}
Combining (\ref{eqB.24})-(\ref{eqB.27}) we obtain the result.
\end{proof}

\subsection*{B.2}
\textbf{Proof of Lemma \ref{lem3.2.8}}
\begin{lem3.2.8}
Let $\Gamma_{u,k}^e = \Gamma_{n,l}^{v-e}$. Then
\begin{align}
& \oint_{\partial \tilde{\Gamma}_{n,l}^{v-e}} e^{x_3^{v-e}}
\left(\frac{\partial u}{\partial \nb^{v-e}}\right)_{A^{v-e}}
\left(\frac{\partial u}{\partial \nuw^{v-e}}\right)_{A^{v-e}}
w^{v-e}(x_1^{v-e})\:ds^{v-e} \notag \\
&-2\sum_{j=1}^2 \int_{\tilde{\Gamma}^{v-e}_{n,l}} e^{x_3^{v-e}}
\left(\frac{\partial u}{\partial \taux_j^{v-e}}\right)_{A^{v-e}}
\frac{\partial}{\partial s^{v-e}_j} \left(\left(\frac{\partial u}
{\partial \nuw^{v-e}}\right)_{A^{v-e}}\right)w^{v-e}(x_1^{v-e})
\:d\sigma^{v-e} \notag \\
&= -\oint_{\partial \tilde{\Gamma}_{u,k}^{e}} \left(\frac{\partial
u}{\partial \nb^{e}}\right)_{A^{e}}\left(\frac{\partial
u}{\partial \nuw^{e}}\right)_{A^{e}}w^e(x_1^e) \:ds^{e} \notag \\
&+2 \sum_{j=1}^2 \int_{\tilde{\Gamma}^{e}_{u,k}}
\left(\frac{\partial u}{\partial \taux_j^{e}}\right)_{A^{e}}
\frac{\partial}{\partial s^{e}_j}\left(\left(\frac{\partial
u}{\partial \nuw^e}\right)_{A^{e}} \right)w^e(x_1^e)\:d\sigma^{e}.
\notag
\end{align}
\end{lem3.2.8}
\begin{proof}
We have
\begin{align}\label{B.28}
x_1^e &= \tau = \ln r \notag \\
x_2^e &= \theta \notag \\
x_3^e &= x_3\,.
\end{align}
\begin{figure}[!ht]
\centering
\includegraphics[scale = 0.60]{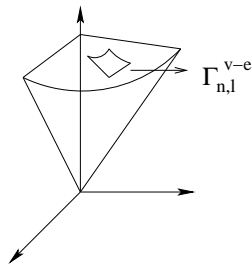}
\caption{ Interior common boundary face ${\Gamma}_{u,k}^e=
{\Gamma}_{n,l}^{v-e}$ }
\label{figB.4}
\end{figure}
and
\begin{align}\label{B.29}
x_1^{v-e} &= \psi = \ln \tan \phi \notag \\
x_2^{v-e} &= \theta \notag \\
x_3^{v-e} &= \zeta = \ln x_3\,.
\end{align}
Hence
\begin{align}\label{B.30}
x_1^e &= x_1^{v-e} + x_3^{v-e} = \psi + \zeta \notag \\
x_2^e &= x_2^{v-e} \notag \\
x_3^e &= e^{x_3^{v-e}} = e^{\zeta}\;.
\end{align}
Clearly
\begin{equation}\label{B.31}
\nabla_{x^{v-e}}u = J^e \nabla_{x^e}u\,.
\end{equation}
Here
\begin{align}\label{B.32}
J^e=\left[ \begin{array}{ccc}
1 & 0 & 0 \\
0 & 1 & 0 \\
1 & 0 & x_3
\end{array} \right] \:.
\end{align}
Now if ${\bf dx}^{v-e}$ is a tangent vector to a curve in $x^{v-e}$ coordinates
then its image in $x^e$ coordinates is given by ${\bf dx}^e$ where
\begin{equation}\label{B.33}
{\bf dx}^e = (J^e)^T {\bf dx}^{v-e}.
\end{equation}
Clearly the first fundamental form $(ds^{v-e})^2$ in $x^{v-e}$ coordinates is
\begin{equation}\label{B.34}
(ds^{v-e})^2 = ({\bf dx}^{v-e})^T ({\bf dx}^{v-e}) =
({\bf dx}^e)^T((J^e)^{-1}(J^e)^{-T}){\bf dx}^{v-e}.
\end{equation}
Now
\begin{align}\label{B.35}
(J^e)^{-1}= \left[ \begin{array}{ccc}
1              & 0 & 0 \\
0              & 1 & 0 \\
-\frac{1}{x_3} & 0 & \frac{1}{x_3}
\end{array} \right].
\end{align}
Hence
\begin{align}\label{B.36}
(ds^{v-e})^2 &= (dx_1^e)^2 + (dx_2^e)^2 + \frac{2}{x_3^2}
(dx_3^e)^2 - \frac{2}{x_3} dx_1^e d x_3^e \notag \\
&=d\tau^2 + d \theta^2 + \frac{2}{x_3^2} d x_3^2-\frac{2}{x_3}
d \tau d x_3.
\end{align}
Moreover on $\tilde{\Gamma}^{v-e}_{n,l}$
$$d \sigma^{v-e}=d\tau d \theta$$
since
$$(d s^{v-e})^2 = d \tau^2 + d \theta^2$$
on $\tilde{\Gamma}^{v-e}_{n,l}$. Choose $\taux^e_1=(1,0,0)^T$ and $\taux_2^e=-(0,1,0)^T$.
These are orthogonal unit vectors on $\tilde{\Gamma}^e_{u,k}$ since
\begin{equation}\label{B.37}
(ds^e)^2 = d \tau^2 + d \theta^2 + d x_3^2\:.
\end{equation}
Define
\begin{align}\label{B.38}
\taux_1^{v-e} &= (J^e)^{-T} \taux_1^e,\rm{\;and\;} \notag\\
\taux_2^{v-e} &= -(J^e)^{-T} \taux^e_2 \:.
\end{align}
Let $\nuw^e=-(0,0,1)^T$ denote the unit normal vector to $\tilde{\Gamma}^e_{u,k}$. Let
\begin{equation}\label{B.39}
\mc^e = (1,0, x_3)^T.
\end{equation}
Then
\begin{equation}\label{B.40}
\nuw^{v-e} = (0,0,1)^T = (J^e)^{-T} \mc^e
\end{equation}
is the unit normal to $\tilde{\Gamma}^{v-e}_{n,l}$. Finally let
$${\bf ds}^e = (d \tau, d\theta, 0)^T$$
denote a tangent vector field on $\tilde{\Gamma}^e_{u,k}$. Define
\begin{equation}\label{B.41}
d s^e = \sqrt{d \tau^2 + d \theta^2}
\end{equation}
\begin{equation}\label{B.42}
{\bf ds}^{v-e} = (J^e)^{-T} {\bf ds}^e =
(d \psi, d \theta, 0)^{T}
\end{equation}
\begin{equation}\label{B.43}
d s^{v-e}=\sqrt{d\theta^2 + d\psi^2}=\sqrt{d\tau^2 + d\theta^2}=d s^e\,.
\end{equation}
Let
\begin{equation}\label{B.44}
\nb^e = \frac{(d\theta, - d\tau, 0)^T}{\sqrt{d \tau^2 + d
\theta^2}}
\end{equation}
be the unit outward normal to $\partial\tilde{\Gamma}^e_{u,k}$. Then the unit outward
normal $\nb^{v-e}$ to $\partial\tilde{\Gamma}^{v-e}_{n,l}$ is
\begin{equation}\label{B.45}
\nb^{v-e} = (J^e)^{-T} \nb^e\,.
\end{equation}
Now
\begin{align*}
\left(\frac{\partial u}{\partial \nb^{v-e}}\right)_{A^{v-e}}
&=(\nb^{v-e})^T A^{v-e}\nabla_{x^{v-e}}u \\
&=(\nb^e)^T (J^e)^{-1} A^{v-e} J^e \nabla_{x^e} u \\
&=(\nb^e)^{T}((J^e)^{-1}(J^e)^{-T})\left((J^e)^T A^{v-e}
J^e\right)\nabla_{x^e}u \\
&=(\nb^e)^T  \left[ \begin{array}{ccc}
1              & 0 & -\frac{1}{x_3}  \\
0              & 1 & 0               \\
-\frac{1}{x_3} & 0 & \frac{2}{x^2_3}
\end{array} \right] A^e \nabla_{x^e} u \\
&=\frac{(d\theta, -d\tau, -\frac{d\theta}{x_3})^T}
{\sqrt{d\tau^2 + d \theta^2}} A^e \nabla_{x^e} u\:.
\end{align*}
Here
\begin{align*}
\nb^e=\frac{(d \theta, -d \tau, 0)^T}{\sqrt{d \theta^2 + d \tau^2}}\:.
\end{align*}
It should be noted that $A^e=(J^e)^T A^{v-e} J^e$. Hence referring to Figure \ref{figB.5}
\begin{figure}[!ht]
\centering
\includegraphics[scale = 0.60]{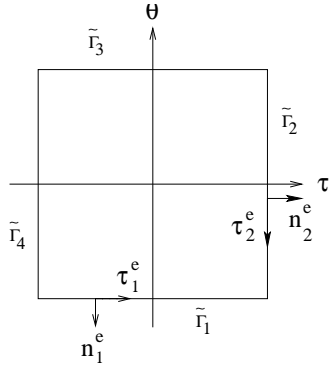}
\caption{The face $\tilde{\Gamma}_{u,k}^e$}
\label{figB.5}
\end{figure}
\begin{align}\label{B.46}
\left(\frac{\partial u}{\partial \nb_1^{v-e}}\right)_{A^{v-e}}
&=\left(\frac{\partial u}{\partial \taux^e_2}\right)_{A^e} \:,
\end{align}
and
\begin{align}\label{B.47}
\left(\frac{\partial u}{\partial \nb_2^{v-e}}\right)_{A^{v-e}}
&=\left(\frac{\partial u}{\partial \taux^e_1}\right)_{A^e} +
\frac{1}{x_3}\left(\frac{\partial u}{\partial \nuw^e}
\right)_{A^e}.
\end{align}
Moreover
\begin{align}\label{B.48}
\left(\frac{\partial u}{\partial \nuw^{v-e}}\right)_{A^{v-e}}
&=
\left[\begin{array}{ccc} 1 & 0 & x_3
\end{array} \right]
\left[ \begin{array}{ccc}
1 & 0 & -\frac{1}{x_3} \\
0 & 1 & 0              \\
-\frac{1}{x_3} & 0 &  \frac{2}{x_3^2}
\end{array} \right] A^e \nabla_{x^e} u \notag \\
&= \left[0,0,\frac{1}{x_3}\right]A^e\nabla_{x^e}u \notag \\
&= -\frac{1}{x_3} \left(\frac{\partial u} {\partial
\nuw^e}\right)_{A^e}\,.
\end{align}
\begin{figure}[!ht]
\centering
\includegraphics[scale = 0.60]{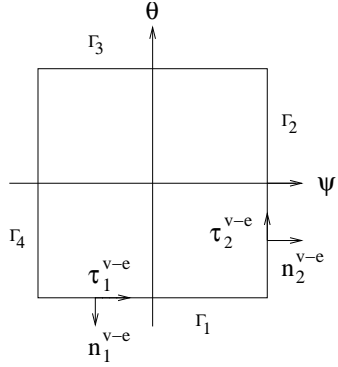}
\caption{The face $\tilde{\Gamma}_{n,l}^{v-e}$}
\label{figB.6}
\end{figure}
Now
\begin{align}\label{B.49}
& \oint_{\partial \tilde{\Gamma}^{v-e}_{n,l}}e^{x_3^{v-e}}
\left(\frac{\partial u}{\partial\nb^{v-e}}\right)_{A^{v-e}}
\left(\frac{\partial u}{\partial \nuw^{v-e}}\right)_{A^{v-e}}
\:d s^{v-e} \notag \\
&= \int_{\Gamma_1} x_3 \left(\frac{\partial u} {\partial
\nb_1^{v-e}}\right)_{A^{v-e}} \left(\frac{\partial u}{\partial
\nuw^{v-e}}\right)_{A^{v-e}} \:d \psi \notag \\
&+ \int_{\Gamma_2} x_3\left(\frac{\partial u}{\partial
\nb_2^{v-e}}\right)_{A^{v-e}}\left(\frac{\partial u}{\partial
\nuw^{v-e}}\right)_{A^{v-e}} \:d \theta \notag \\
&- \int_{\Gamma_3} x_3\left(\frac{\partial u}{\partial
\nb_1^{v-e}}\right)_{A^{v-e}}\left(\frac{\partial u}{\partial
\nuw^{v-e}}\right)_{A^{v-e}} \:d \psi \notag \\
&- \int_{\Gamma_4} x_3\left(\frac{\partial u}{\partial
\nb_2^{v-e}}\right)_{A^{v-e}}\left(\frac{\partial u}{\partial
\nuw^{v-e}}\right)_{A^{v-e}}\:d \theta \notag \\
&= -\int_{\tilde{\Gamma}_1} \left(\frac{\partial u} {\partial
\taux^e_2}\right)_{A^e}\left(\frac{\partial u}
{\partial \nuw^e}\right)_{A^e}\:d \tau \notag 
\end{align}
\begin{align}
&- \int_{\tilde{\Gamma}_2}\left(\left(\frac{\partial u}
{\partial \taux^e_1}\right)_{A^e} + \frac{1}{x_3}\left(\frac
{\partial u}{\partial \nuw^e}\right)_{A^e}\right)
\left(\frac{\partial u}{\partial \nuw^e}\right)_{A^e}
\:d\theta \notag \\
&+ \int_{\tilde{\Gamma}_3}\left(\frac{\partial u}{\partial
\taux^e_2}\right)_{A^e}\left(\frac{\partial u}{\partial
\nuw^e}\right)_{A^e} \:d\tau \notag \\
&+ \int_{\tilde{\Gamma}_4} \left(\left(\frac{\partial u}
{\partial \taux^e_1}\right)_{A^e} + \frac{1}{x_3}\left(\frac
{\partial u}{\partial \nuw^e}\right)_{A^e}\right)
\left(\frac{\partial u}{\partial \nuw^e}\right)_{A^e}\:d\theta \notag \\
&=- \oint_{\partial \tilde{\Gamma}_{u,k}^e}\left(\frac
{\partial u}{\partial \nb^e}\right)_{A^e}\left(\frac{\partial u}
{\partial \nuw^e}\right)_{A^e} \:d s^e \notag \\
&- \int_{\tilde{\Gamma_2}}\frac{1}{x_3}\left(\frac{\partial u}
{\partial \nuw^e}\right)^2_{A^e}\:d\theta +
\int_{\tilde{\Gamma_4}}\frac{1}{x_3}\left(\frac{\partial u}
{\partial \nuw^e}\right)^2_{A^e}\:d\theta
\end{align}
Now $\taux_1^{v-e} = (1,0,0)^T$ and $\taux_1^e = (J^e)^{-T}\taux_1^{v-e}$. Moreover
$\taux_2^{v-e} = (0,1,0)^T$ and $\taux_2^e = -(J^e)^T \taux_2^{v-e}$.
\newline
Hence
\begin{align}\label{B.50}
&  2 \int_{\tilde{\Gamma}^{v-e}_{n,l}} e^{x_3^{v-e}}
\left(\frac{\partial u}{\partial \taux_1^{v-e}}\right)
_{A^{v-e}}\frac{\partial}{\partial s_1^{v-e}}\left(\left
(\frac{\partial u}{\partial \nuw^{v-e}}\right)_{A^{v-e}}
\right) \:d \sigma^{v-e} \notag \\
&= -2 \int_{\tilde{\Gamma}^{e}_{u,k}}\left(\left(\frac
{\partial u}{\partial \taux_1^{e}}\right)_{A^{e}} + \frac{1}
{x_3}\left(\frac{\partial u}{\partial \nuw^{e}}\right)_{A^{e}}
\right)\frac{\partial}{\partial \tau}\left(\left(\frac
{\partial u}{\partial \nuw^{e}}\right)_{A^{e}}\right)
\:d \tau \:d \theta \notag \\
&=  -2 \int_{\tilde{\Gamma}^{e}_{u,k}} \left(\frac {\partial
u}{\partial \taux_1^{e}}\right)_{A^{e}} \frac {\partial}{\partial
s_1^{e}}\left(\left(\frac{\partial u}
{\partial \nuw^{e}}\right)_{A^{e}}\right) \:d \sigma^{e} \notag \\
&- \int_{\tilde{\Gamma}_2} \frac{1}{x_3} \left(\frac{\partial
u}{\partial \nuw^e}\right)^2_{A^e} \:d\theta
+\int_{\tilde{\Gamma}_4} \frac{1}{x_3} \left(\frac{\partial
u}{\partial \nuw^e}\right)^2_{A_e} d \theta\:.
\end{align}
And
\begin{align}\label{B.51}
& 2 \int_{\tilde{\Gamma}^{v-e}_{n,l}} e_3^{x^{v-e}}
\left(\frac{\partial u}{\partial \taux_2^{v-e}}\right)_{A^{v-e}}
\frac{\partial}{\partial s_2^{v-e}}\left(\left(\frac{\partial u}
{\partial \nuw^{v-e}}\right)_{A^{v-e}}\right) \:d \sigma^{v-e}
\notag \\
&=-2\int_{\tilde{\Gamma}^{e}_{u,k}}\left(\frac{\partial u}
{\partial \taux_2^{e}}\right)_{A^{e}} \frac{\partial}{\partial
\theta}\left(\left(\frac{\partial u}
{\partial \nuw^{e}}\right)_{A^{e}}\right) \:d\tau \:d\theta \notag \\
&= -2  \int_{\tilde{\Gamma}^{e}_{u,k}}\left(\frac{\partial u}
{\partial \taux_2^{e}}\right)_{A^{e}} \frac{\partial}{\partial
s_2^e}\left(\left(\frac{\partial u}{\partial \nuw^{e}}\right)
_{A^{e}}\right) \:d\sigma^e\:.
\end{align}
Combining (\ref{B.49}), (\ref{B.50}) and (\ref{B.51}) we obtain the result.
\end{proof}

\subsection*{B.3}
\textbf{Proof of Lemma \ref{lem3.2.9}}
\begin{lem3.2.9}
Let $\Gamma^e_{u,k} = \Gamma^r_{l,j}$. Then
\begin{align}
\rho_v^2 \sin^2(\phi_v)
\oint_{\partial \Gamma^r_{l,j}} \left(\frac{\partial u}{\partial
\nb}\right)_A\left(\frac{\partial u}{\partial \nuw}\right)_A\:d s
=-\oint_{\partial \tilde{\Gamma}^e_{u,k}} \left(\frac{\partial
u}{\partial \nuw^e}\right)_{A^e} \left(\frac{\partial u}{\partial
\nb^e} \right)_{A^e} \:d s^e \notag
\end{align}
and
\begin{align}
\rho_v^2 \sin^2(\phi_v)\left(\sum_{m=1}^2
\int_{\Gamma_{l,j}^r} \left(\frac{\partial u}{\partial\taux_m}
\right)_A\frac{\partial}{\partial s_m}\left(\frac{\partial u}
{\partial \nuw}\right)_{A} \:d \sigma \right) \notag \\
\quad=-\sum_{m=1}^2 \int_{\tilde{\Gamma}^e_{u,k}} \left(\frac
{\partial u}{\partial \taux^e_m}\right)_{A^e} \frac{\partial}
{\partial s^e_m}\left(\frac{\partial u}{\partial \nuw^e}
\right)_{A^e} d \sigma^e \notag
\end{align}
\end{lem3.2.9}
\begin{proof}
We have
\begin{align}\label{B.52}
x_1^e &= \tau = \ln r \notag \\
x_2^e &= \theta \notag \\
x_3^e &= x_3\,.
\end{align}
Clearly,
\begin{align}\label{B.53}
\nabla_{x} u = R^e \nabla_{x^e} u
\end{align}
where
\begin{align}\label{B.54}
R^e =  \left[ \begin{array}{ccc}
e^{-\tau}\cos\theta & -e^{-\tau}\sin\theta & 0 \\
e^{-\tau}\sin\theta &  e^{-\tau}\cos\theta & 0 \\
0                   & 0                    & 1
\end{array} \right].
\end{align}
Now if ${\bf dx}$ is a tangent vector to a curve in $x^v$
coordinates then its image in $x^e$ coordinates is given by
${\bf dx}^e$ where
\begin{equation}\label{B.55}
{\bf dx}^e = (R^e)^T {\bf dx}.
\end{equation}
Clearly, the first fundamental form $(d s)^2$ in $x$ coordinates
is
\begin{equation}\label{B.56}
(d s)^2 = ({\bf dx})^T ({\bf dx}) = ({\bf dx}^e)^T
(R^e)^{-1}(R^e)^{-T} {\bf dx}^e.
\end{equation}
Now
\begin{align}\label{B.57}
{(R^e)}^{-1} =  \left[ \begin{array}{ccc}
e^{\tau}\cos\theta  & e^{\tau}\sin\theta & 0 \\
-e^{\tau}\sin\theta & e^{\tau}\cos\theta & 0 \\
0                   & 0                  &  1
\end{array} \right].
\end{align}
Hence
\begin{equation}\label{B.58}
(d s)^2 = e^{2\tau}(d \tau)^2+e^{2\tau}(d \theta)^2+(d {x_3^e})^2\:.
\end{equation}
Moreover on $\Gamma_{l,j}^r$
\begin{equation} \label{B.59}
d \sigma = e^\tau d\theta d x_3^e
\end{equation}
since
\[(d s)^2 = e^{2\tau}(d \theta)^2+(d {x_3^e})^2\]
on $\Gamma_{l,j}^r$. Choose $\taux_1^e=(0,1,0)^T$ and $\taux_2^e=(0,0,1)^T$. These
are orthogonal unit tangent vectors on $\tilde{\Gamma}_{u,k}^e$ since
$(d s)^2 = e^{2\tau}(d \tau)^2+e^{2\tau}(d \theta)^2+(d {x_3^e})^2$.
\newline
Define
\begin{align} \label{B.60}
\taux_1 & = e^{-\tau}(R^e)^{-T}\taux_1^e\:, \notag \\
\taux_2 & = (R^e)^{-T}\taux_2^e \:.
\end{align}
Let $\nuw^e=(1,0,0)^T$ denotes the unit normal vector on $\tilde{\Gamma}_{u,k}^e$.
Then
\[\nuw = -e^{-\tau}(R^e)^{-T}\nuw^e\]
denotes the unit normal to $\Gamma_{l,j}^r$. Finally let
\[{\bf d s}^e = (0,d \theta, d x_3^e)\]
denotes a tangent vector field on $\tilde{\Gamma}_{u,k}^e$. Define
\begin{align}\label{B.61}
ds^e & = \sqrt{(d \theta)^2+(d {x_3^e})^2}\:,
\end{align}
\begin{align}\label{B.62}
{\bf ds} &= (R^e)^{-T}{\bf d s}^e
\end{align}
and
\begin{equation}\label{B.63}
d s = \sqrt{e^{2\tau}(d \theta)^2+(d {x_3^e})^2}\:.
\end{equation}
Let
\begin{equation}\label{B.64}
\nb^e =\frac{(0,-d x_3^e, d\theta)^T}{\sqrt{(d \theta)^2
+(d {x_3^e})^2}}
\end{equation}
be the outward unit normal to
$\partial\tilde{\Gamma}_{u,k}^e$. Define
\begin{equation} \label{B.65}
\mc^e =\frac{(0,-e^{-\tau}d x_3^e, e^{\tau}d\theta)^T}
{\sqrt{e^{2\tau}(d\theta)^2+(d {x_3^e})^2}}\;.
\end{equation}
Then
\begin{equation}\label{B.66}
\nb = (R^e)^{-T}\mc^e
\end{equation}
is the outward unit normal to $\partial\Gamma_{l,j}^r$.
\newline
Now
\begin{align*}
\left(\frac{\partial u}{\partial\nuw}\right)_A &= \nuw^T
A\nabla_x u\\
&=-e^{-\tau}{(\nuw^e)}^T(R^e)^{-1} A (R^e)\nabla_x^e u\\
&=-e^{-\tau}{(\nuw^e)}^T e^{-2\tau}((R^e)^{-1}(R^e)^{-T})
A^e\nabla_x^e u \:.
\end{align*}
And
\[{(\nuw^e)}^T e^{-2\tau}((R^e)^{-1}(R^e)^{-T})={(\nuw^e)}^T\:.\]
Hence we conclude that
\begin{align} \label{B.67}
\left(\frac{\partial u}{\partial\nuw}\right)_A
&=-e^{-\tau}\left(\frac{\partial u}{\partial\nuw^e}\right)_{A^e}.
\end{align}
Also from (\ref{B.65})
\begin{align*}
\left(\frac{\partial u}{\partial\nb}\right)_A&=\nb^TA\nabla_x u
=((\mc^e)^T(R^e)^{-1}e^{-2\tau}(R^e)^{-T})A^e\nabla_x^e u\:.
\end{align*}
Clearly
\[((\mc^e)^T(R^e)^{-1}e^{-2\tau}(R^e)^{-T}))=e^{-\tau}\frac{d s^e}{d s}(\nb^e)^T\:.\]
Hence
\begin{equation} \label{B.68}
\left(\frac{\partial u}{\partial\nb}\right)_A\:d s
=e^{-\tau}\left(\frac{\partial u}{\partial\nb^e}\right)_{A^e}ds^e\:.
\end{equation}
Thus from (\ref{B.67}) and (\ref{B.68}) we get
\begin{align} \label{B.69}
\rho_v^2 \sin^2(\phi_v)
&\oint_{\partial \Gamma^r_{l,j}} \left(\frac{\partial u}{\partial
\nb}\right)_A\left(\frac{\partial u}{\partial \nuw}\right)_A
\:d s \notag \\
\quad\quad&=-\rho_v^2 \sin^2(\phi_v)\oint_{\partial \tilde{\Gamma}
^e_{u,k}}e^{-2\tau}\left(\frac{\partial u}{\partial \nuw^e}\right)
_{A^e}\left(\frac{\partial u}{\partial\nb^e} \right)_{A^e}
\:d s^e \notag \\
\quad\quad&=-\oint_{\partial \tilde{\Gamma}^e_{u,k}}\left(\frac
{\partial u}{\partial \nuw^e}\right)_{A^e}\left(\frac{\partial u}
{\partial\nb^e}\right)_{A^e}\:d s^e\:.
\end{align}
Using (\ref{B.60}) it is easy to show that
\begin{align} \label{B.70}
\left(\frac{\partial u}{\partial\taux_1}\right)_A
&=e^{-\tau}\left(\frac{\partial u}{\partial\taux_1^e}\right)_{A^e},
\notag \\ 
\left(\frac{\partial u}{\partial\taux_1}\right)_A
&=e^{-2\tau}\left(\frac{\partial u}{\partial\taux_1^e}\right)_{A^e}
\end{align}
Moreover using (\ref{B.67}) we get
\begin{align} \label{B.71}
\frac{\partial}{\partial s_1}\left(\left(\frac{\partial u}{\partial
\nuw}\right)_A\right)= e^{-2\tau}\frac{\partial}{\partial s_1^e}
\left(\left(\frac{\partial u}{\partial\nuw^e}\right)_{A^e}\right),
\notag \\ 
\frac{\partial}{\partial s_2}\left(\left(\frac{\partial
u}{\partial \nuw}\right)_A\right)=e^{-\tau}\frac{\partial}{\partial s_2^e}
\left(\left(\frac{\partial u}{\partial\nuw^e}\right)_{A^e}\right)\:.
\end{align}
Combining (\ref{B.69}), (\ref{B.70}) and (\ref{B.71}) we obtain the result.
\end{proof}
\end{appendix}
\thispagestyle{empty}
\chapcleardoublepage

\renewcommand{\theequation}{C.\arabic{equation}}
\renewcommand{\thefigure}{C.\arabic{figure}}
\setcounter{equation}{0}
\setcounter{figure}{0}
\addcontentsline{toc}{chapter}{Appendix C}
\markboth{\small Appendix C}{\small Appendix C}
\begin{appendix}
\section*{Appendix C}
\subsection*{C.1}
\textbf{Proof of Lemma \ref{lem3.3.1}}
\begin{lem3.3.1}
We can define a set of corrections
$\{\eta_l^r\}_{l=1,\ldots,N_r}$, $\{\eta_l^v\}_{l=1,\ldots,N_v}$
for $v\in \mathcal V$, $\{\eta_l^{v-e}\}_{l=1,\ldots,N_{v-e}}$
for ${v-e}\in \mathcal {V-E}$ and $\{\eta_l^e\} _{l=1,\ldots,N_e}$
for $e\in \mathcal E$ such that the corrected spectral element
function $p$ defined as
\begin{align}
p_l^r=u_l^r+\eta_l^r \hspace{1.4cm} & {\it for} \ l=1,\ldots,N_r,
\notag \\
p_l^v=u_l^v+\eta_l^v \hspace{1.4cm} & {\it for}\ l=1,\ldots,N_v
\quad and\quad v\in\mathcal V, \notag \\
p_l^{v-e}=u_l^{v-e}+\eta_l^{v-e} \hspace{0.3cm} & {\it for} \
l=1,\ldots,N_{v-e}\quad and\quad {v-e}\in\mathcal {V-E}, \notag \\
p_l^e=u_l^e+\eta_l^e \hspace{1.4cm} & {\it for }\ l=1,\ldots,N_e
\quad and \quad e\in\mathcal E, \notag
\end{align}
is conforming and $p\in H_0^1(\Omega)$. i.e. $p\in H^1{(\Omega)}$
and $p$ vanishes on $\Gamma^{[0]}$. Define
\begin{align} \label{eqC.1}
\mathcal U^{N,W}_{(1)}(\{\mathcal F_s\})&=\sum_{l=1}^{N_r}
\left\|\:s_l^r(x_1,x_2,x_3)\:\right\|
^2_{1,\Omega_l^r}+\sum_{v\in \mathcal V}\sum_{l=1}^{N_v}
\left\|\:s_l^v(x_1^v,x_2^v,x_3^v) e^{{x_3^v}/2}
\:\right\|^2_{1,\tilde{\Omega}_l^v} \notag \\
&\quad+\sum_{{v-e}\in \mathcal {V-E}}\Bigg(\mathop{\sum_{l=1}}
_{\mu(\tilde{\Omega}_l^{v-e})<\infty}^{N_{v-e}}
\int_{\tilde{\Omega}_l^{v-e}}e^{x_3^{v-e}}\left(\sum_{i=1}^2
\left(\frac{\partial s_l^{v-e}}{\partial x_i^{v-e}}\right)^2
+\sin^2\phi\left(\frac{\partial s_l^{v-e}}{\partial x_3^{v-e}}
\right)^2\right. \notag \\
&\quad+ (s_l^{v-e})^2\Bigg)\:dx^{v-e}
+\mathop{\sum_{l=1}}_{\mu(\tilde{\Omega}_l^{v-e})=\infty}
^{N_{v-e}}\int_{\tilde{\Omega}_l^{v-e}}e^{x_3^{v-e}}
(s_l^{v-e})^2\:w^{v-e}(x^{v-e})\:dx^{v-e}\Bigg) \notag \\
&\quad +\sum_{e\in \mathcal E}\Bigg(\mathop{\sum_{l=1}}
_{\mu(\tilde{\Omega}_l^e)<\infty}^{N_e}\int_{\tilde{\Omega}_l^e}
\left(\sum_{i=1}^2\left(\frac{\partial s_l^e}{\partial
x_i^e}\right)^2\right.+\left.e^{2\tau} \left(\frac
{\partial s_l^e}{\partial x_3^e}\right)^2+(s_l^e)^2\right)
\:dx^e \notag \\
&\quad+\mathop{\sum_{l=1}}_{\mu(\tilde{\Omega}_l^e)=\infty}
^{N_e}\int_{\tilde{\Omega}_l^e}(s_l^e)^2\:w^e(x_1^e)\:dx^e\Bigg)\:.
\end{align}
Then the estimate
\begin{align}\label{C.2}
\mathcal U^{N,W}_{(1)}(\{\mathcal F_\eta\})\leq C_W \mathcal V^{N,W}(\{\mathcal F_u\})
\end{align}
holds. Here $C_W$ is a constant, if the spectral element functions are conforming
on the wirebasket $W\!B$ of the elements, otherwise $C_W=C(\ln W)$, where $C$ is a
constant.
\end{lem3.3.1}
\begin{proof}
We first consider the case when the spectral element functions are conforming on the
wirebasket. Consider an element $\Omega_l^r\subseteq\Omega^r$ which, for simplicity
is assumed to be a hexahedral domain. Let $Q$ denote the cube $(-1,1)^3$ and
$M_l^r(\lambda_1,\lambda_2,\lambda_3)$ denote an analytic mapping from $Q$ to
$\Omega_l^r$. Let $f_1$ be the face corresponding to $\lambda_3=-1$, $f_2$ the face
corresponding to $\lambda_3=1$ and $\left.{[u(\lambda_1,\lambda_2)]}\right|_{f_1}$
denote the jump in the value of $u$ across the face $f_1$ so that $u_l^r(\lambda_1,\lambda_2)
+{1/2}\left.{[u(\lambda_1,\lambda_2)]}\right|_{f_1}=$ average value of $u$ on the face
$f_1$. Similarly by $\left.{[u(\lambda_1,\lambda_2)]}\right|_{f_2}$ we denote the jump
in the value of $u$ across the face $f_2$ so that $u_l^r(\lambda_1,\lambda_2)+{1/2}
\left.{[u(\lambda_1,\lambda_2)]}\right|_{f_2}=$ average value of $u$ on the face $f_2$.

Let $f_3$ be the face corresponding to $\lambda_2=-1$ and $f_4$ the face corresponding
to $\lambda_2=1$. We can now define $\left.{[u(\lambda_1,\lambda_3)]}\right|_{f_3}$, the
jump in $u$ across the face $f_3$, and $\left.{[u(\lambda_1,\lambda_3)]}\right|_{f_4}$,
the jump in $u$ across the face $f_4$.
\begin{figure}[!ht]
\centering
\includegraphics[scale = 0.60]{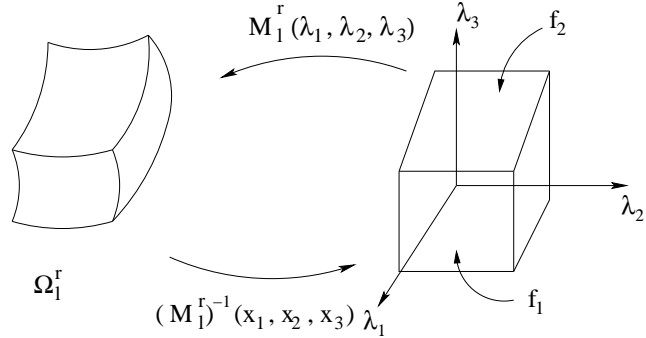}
\caption{The element $\Omega_l^r$ and the master cube.}
\label{figC.1}
\end{figure}
Finally, let $f_5$ denote the face corresponding to $\lambda_1=-1$ and $f_6$ the face
corresponding to $\lambda_1=1$. Once again, we define $\left.{[u(\lambda_2,\lambda_3)]}
\right|_{f_5}$, the jump in $u$ across the face $f_5$, and $\left.{[u(\lambda_2,\lambda_3)]}
\right|_{f_6}$, the jump in $u$ across the face $f_6$. Let us now define the correction
\begin{align}
\nu_l^r(\lambda_1,\lambda_2,\lambda_3)
&=\left.{\frac{(1-\lambda_3)}{2}
{\frac{[u(\lambda_1,\lambda_2)]}{2}}}\right|_{f_1}+
\left.{\frac{(1+\lambda_3)}{2}
{\frac{[u(\lambda_1,\lambda_2)]}{2}}}\right|_{f_2} \notag \\
&+ \left.{\frac{(1-\lambda_2)}{2}
{\frac{[u(\lambda_1,\lambda_3)]}{2}}}\right|_{f_3}
+\left.{\frac{(1+\lambda_2)}{2} {\frac{[u(\lambda_1,
\lambda_3)]}{2}}}\right|_{f_4} \notag \\
&+\left.{\frac{(1-\lambda_1)}{2}
{\frac{[u(\lambda_2,\lambda_3)]}{2}}}\right|_{f_5}
+\left.{\frac{(1+\lambda_1)}{2}
{\frac{[u(\lambda_2,\lambda_3)]}{2}}}\right|_{f_6}. 
\end{align}
Then
\[\eta_l^r(x_1,x_2,x_3) = \nu_l^r\left((M_l^r)^{-1}
(x_1,x_2,x_3)\right)\:.\]

Next, we define the corrections in the vertex neighbourhood $\Omega^v$ for
$v \in \mathcal V$, the set of vertex neighbourhoods. Suppose $\Omega_l^v$ is
not a corner element, nor does it have a face in common with a corner element.
Then the correction $\eta_l^v(x_1^v,x_2^v,x_3^v)$ is defined in the same way as
in the regular region.
\newline
If $\Omega_l^v$ is a corner element then define
$$\eta_l^v (x_1^v,x_2^v,x_3^v)\equiv 0.$$
If $\Omega_l^v$ is an element which has a face in common with a corner element, i.e.
$$\tilde{\Omega}_l^v=\left\{x^v:\:(\theta,\phi) \in
S_j^v,\:\ln(\rho_1^v)<\chi<\ln(\rho_2^v)\:\right\}$$
then the corrected element function will assume the constant value $h_v$ on the side
$\chi=\ln(\rho_1^v)$. Hence the correction will assume the value
$[u(\phi,\theta,\ln(\rho_1^v))]$ instead of
$\frac{1}{2}\left[u(\phi,\theta,\ln(\rho_1^v))\right]$ on the side $\chi=\ln(\rho_1^v)$.
In the vertex neighbourhood $\Omega^v$ the vertex coordinates are
\begin{align*}
x_1^v &= \phi \\
x_2^v &= \theta \\
x_3^v &= \chi \ =\:\ln\rho\;.
\end{align*}
Moreover
\[\int_{\Omega_l^v}\left|\:\nabla_x u_l^v\:\right|^2 dx =
\int_{\tilde{\Omega}_l^v} \left(u_{\phi}^2+\frac{1}{\sin^2
\phi}\: u_\theta^2+(\rho u_{\rho})^2\right)\:\sin\phi\:d
\rho\:d\phi\:d\theta.\]
Hence
\[\int_{\Omega_l^v}|\nabla_{x} u_l^v|^2 dx\]
is uniformly equivalent to
\[\int_{\tilde{\Omega}_l^v}e^{x_3^v}\left|\:\nabla_{x^v}
u_l^v\:\right|^2\:dx^v\]
for all $\Omega_l^v, v \in \mathcal V.$
\newline
Here we have used the fact that
\[0<\phi_0\leq \phi_v\leq \pi-\phi_0\]
for all $v \in \mathcal V$.

Next, we define the corrections in $\Omega^{v-e}$, a vertex-edge neighbourhood,
where $v-e \in \mathcal {V-E}$, the set of vertex edge neighbourhoods. Let
$\Omega_l^{v-e}$ be an element which is not a corner element, nor which has
a face in common with a corner vertex-edge element. Then
\[\tilde{\Omega}_l^{v-e}=\left\{x^{v-e}:\:\psi_i^{v-e}<
\psi<\psi_{i+1}^{v-e},\:\theta_j^{v-e}<\theta<
\theta_{j+1}^{v-e},\:\zeta_k^{v-e}<\zeta<\zeta_{k+1}^{v-e}
\right\}\]
with $2\leq i,2\leq k.$ If $\Omega_l^{v-e}$ is a corner element then
$\eta_l^{v-e}(x^{v-e})=0$. Here
\begin{align*}
x_1^{v-e} &= \psi=\:\ln(\tan\phi)\\
x_2^{v-e} &= \theta \\
x_3^{v-e} &= \zeta=\:\ln x_3=\chi+\ln(\cos(\phi)).
\end{align*}

We introduce local variables
\begin{align*}
y_1 &= x_1^{v-e}\\
y_2 &= x_2^{v-e}\\
y_3 &= \frac{x_3^{v-e}}{\sin\left(\phi_{i+1}^{v-e}\right)}\:.
\end{align*}

Then $\Omega_l^{v-e}$ is mapped to a rectangular hexahedron $\widetilde{\Omega}_l^{v-e}$
in $x^{v-e}$ coordinates and to a long, thin hexahedron $\widehat{\Omega}_l^{v-e}$ in the
local $y$ coordinates such that the length of the $y_3$ side becomes large as
$\Omega_l^{v-e}$ approaches the edge of the domain $\Omega^{v-e}$, as shown in Figure
\ref{figC.2}. Clearly
$$\int_{\Omega_l^{v-e}}\left|\:\nabla_x u\:\right|^2\:dx$$
is uniformly equivalent to
$$\int_{\widehat{\Omega}_l^{v-e}}\left(u_{\phi}^2+\frac{1}{\sin^2\phi}
u_\theta^2 + u_{\chi}^2\right)
\:e^{\chi}\sin\phi\:d\phi\:d\theta\:d\chi$$ for all $\Omega_l^{v-e}
\subseteq \Omega^{v-e}, v-e \in \mathcal {V-E}$. Here $\widehat{\Omega}_l^{v-e}$ denotes
the image of $\Omega_l^{v-e}$ in $x^v$ coordinates.

Now $$\nabla_{x^v}u=J^{v-e}\nabla_{x^{v-e}}u$$ where
$$J^{v-e} =  \left[ \begin{array}{ccc}
\sec^2 \phi \cot \phi   &0&  -\tan \phi\\
0 & 1 & 0 \\
0 & 0 &  1
\end{array} \right]$$
and
$$dx^v=\sin\phi \cos\phi \:dx^{v-e}.$$
Using the above we can show that
$$\int_{\Omega_l^{v-e}}\left|\:\nabla_x u\:\right|^2 dx$$
is uniformly equivalent to
$$\int_{\widetilde{\Omega}_l^{v-e}}\left(\sum_{m=1}^2 \left(\frac
{\partial u}{\partial x_m^{v-e}}\right)^2 + \sin^2\phi\left
(\frac{\partial u}{\partial x_3^{v-e}}\right)^2\right)
\:e^{x_3^{v-e}}\:dx^{v-e}$$
for all $\Omega_l^{v-e} \subseteq \Omega^{v-e}, v-e \in
\mathcal {V-E}$.
\newline
Hence $$\int_{\Omega_l^{v-e}}\left|\:\nabla_x u\:\right|^2 dx$$
is uniformly equivalent to
$$\int_{\widehat{\Omega}_{l}^{v-e}} e^{\sin\left
(\phi_{i+1}^{v-e}\right)y_3}\sin\left(\phi_{i+1}^{v-e}\right)
|\nabla_y u|^2\:d y.$$
\begin{figure}[!ht]
\centering
\includegraphics[scale = 0.60]{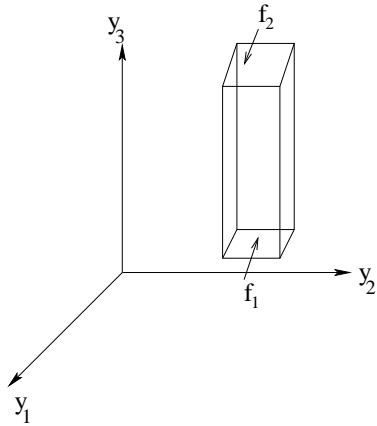}
\caption{Image of ${\Omega}_{l}^{v-e}$ in $y$-variables.}
\label{figC.2}
\end{figure}
Let $f_1$ denote the face $y_3=\frac{\zeta_{k}}{\sin\left(\phi_{i+1}^{v-e}\right)}
=\kappa^i_{k}$ and $f_2$ the face $y_3=\frac{\zeta_{k+1}}{\sin\left(\phi_{i+1}^{v-e}\right)}
=\kappa_{k+1}^i.$

Let $\left.{[u(y_1,y_2)]}\right|_{f_1}$ denote the jump in $u$ across the face $f_1$ so
that $u(y_1,y_2)\left|\right._{f_1}+1/2[u(y_1,y_2)]\left|\right._{f_1}$ denotes the
average value of $u$ on the face $f_1$. Similarly let $[u(y_1,y_2)]\left|\right._{f_2}$
denote the jump in the value of $u$ across the face $f_2$. We now define the correction
\[\nu_l^{v-e}(y_1,y_2,y_3)=\alpha_l^{v-e}(y_1,y_2,y_3)+
\beta_l^{v-e}(y_1,y_2,y_3)+\gamma_l^{v-e}(y_1,y_2,y_3).\]
$\alpha_l^{v-e}(y_1,y_2,y_3)$ is first defined as follows:
\newline
Let
\[h^i_k = min\left(1,\frac{\kappa_{k+1}^i-\kappa_{k}^i}{2}\right)\,.\]
\begin{figure}[!ht]
\centering
\includegraphics[scale = 0.60]{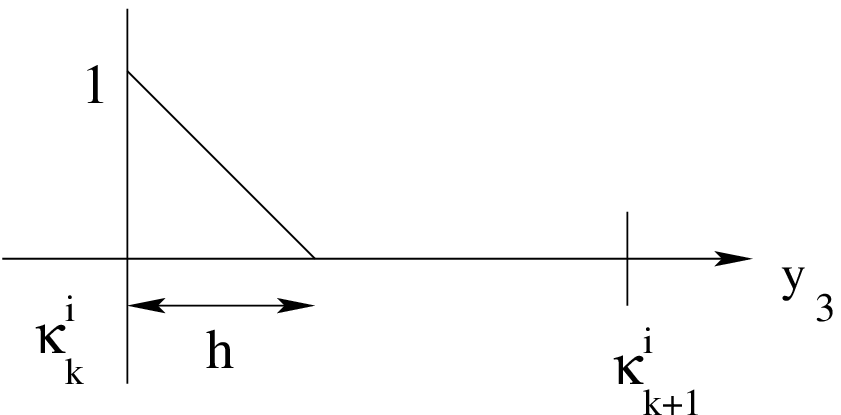}
\caption{The function $r(y_3)$.}
\label{figC.3}
\end{figure}
Let $r(y_3)$ be the function
\begin{align}\label{eqC.4}
r(y_3)&= \left\{\begin{array}{ccc}
\frac{(\kappa_{k}^i-y_3)}{h_k^i}+1 & {\it for} \; \kappa_{k}^i\leq y_3\leq
\kappa_{k}^i+h_k^i\\
0 & {\it for} \; \kappa_{k}^i+h_k^i\leq y_3\leq \kappa_{k+1}^i\,.
\end{array} \right.
\end{align}
Let $s(y_3)$ be the function
\begin{align}\label{eqC.5}
s(y_3)&= \left\{\begin{array}{ccc}
0  & {\it for} \; \kappa_{k}^i \leq y_3\leq \kappa_{k+1}^i-h_k^i\\
\frac{y_3-(\kappa_{k+1}^i-h_k^i)}{h_k^i} & {\it for}
\;\kappa_{k+1}^i-h_k^i\leq y_3\leq \kappa_{k+1}^i\,.
\end{array}\right.
\end{align}
\begin{figure}[!ht]
\centering
\includegraphics[scale = 0.60]{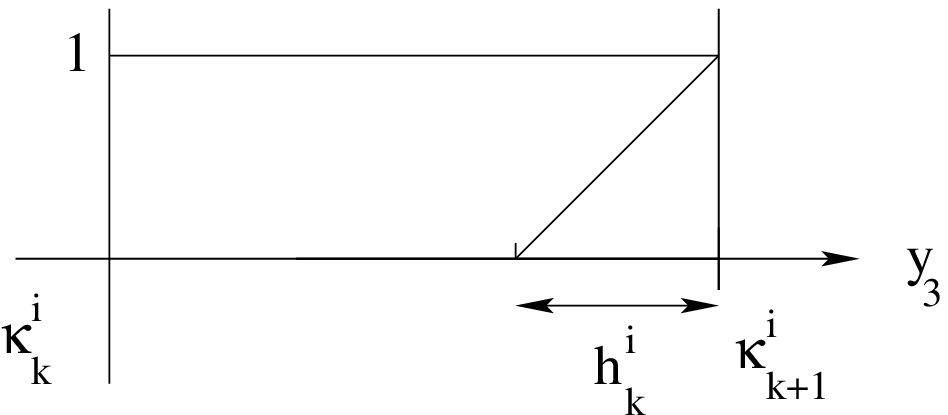}
\caption{The function $s(y_3)$.}
\label{figC.4}
\end{figure}
Then we define the correction
\[\alpha_l^{v-e}(y_1,y_2,y_3)=\left.{1/2[u(y_1,y_2)]}
\right|_{f_1} r(y_3)+1/2[u(y_1,y_2)]|_{f_2} s(y_3).\]

If $\Omega^{v-e}_l$ is a corner element, i.e. the face $f_1$ corresponds to
$x_3^{v-e}=\zeta_1^{v-e}$ then
\[\alpha_l^{v-e}(y_1,y_2,y_3)=\left.{[u(y_1,y_2)]}
\right|_{f_1} r(y_3)+1/2[u(y_1,y_2)]|_{f_2} s(y_3).\]

Next, let $f_3$ denote the side $\theta=\theta_j^{v-e}$ and $f_4$
the side $\theta=\theta_{j+1}^{v-e}$. Let $[u(y_1,y_3)]|_{f_3}$
denote the jump across the side $f_3$ and $[u(y_1,y_3)]|_{f_4}$
the jump across the side $f_4$. Define
\begin{eqnarray*}
\beta_l^{v-e}(y_1,y_2,y_3) &=& \left(\frac{y_2-\theta_{j+1}^{v-e}}
{\theta_{j}^{v-e}-\theta_{j+1}^{v-e}}\right)\left.{\frac{[u(y_1,
y_3)]}{2}}\right|_{f_3} \\
&+&\left(\frac{y_2-\theta_{j}^{v-e}}{\theta_{j+1}^{v-e}-\theta_{j}
^{v-e}}\right)\left.{\frac{[u(y_1,y_3)]}{2}}\right|_{f_4}
\end{eqnarray*}
It is assumed here that $\theta=\theta_j^{v-e}$ and $\theta=\theta
_{j+1}^{v-e}$ do not correspond to a face of the domain $\Omega$.
\newline
Finally, let $f_5$ denote the face $\psi=\psi_i^{v-e}$ and $f_6$
the face $\psi=\psi_{i+1}^{v-e}$. Define
\begin{eqnarray*}
\gamma_l^{v-e}(y_1,y_2,y_3) &=& \left(\frac{y_1-\psi_{i+1}^{v-e}}
{\psi_{i}^{v-e}-\psi_{i+1}^{v-e}}\right)\left.{\frac{[u(y_2,y_3)]}
{2}}\right|_{f_5}\\
&+&\left(\frac{y_1-\psi_{i}^{v-e}}{\psi_{i+1}^{v-e}-\psi_{i}
^{v-e}}\right)\left.{\frac{[u(y_2,y_3)]}{2}}\right|_{f_6}
\end{eqnarray*}
If $i=1$ then the first term in the right hand side is multiplied
by a factor of two.
\newline
Define the final correction as
\[\nu_l^{v-e}(y_1,y_2,y_3)=\alpha_l^{v-e}(y_1,y_2,y_3)+
\beta_l^{v-e}(y_1,y_2,y_3)+\gamma_l^{v-e}(y_1,y_2,y_3).\]
Then
\[\eta_l^{v-e}(x^{v-e})=\nu_l^{v-e}\left(x_1^{v-e},x_2^{v-e},
\frac{x_3^{v-e}}{\sin(\phi_{i+1}^{v-e})}\right)\:.\]

If $\Omega_l^{v-e}$ has a face in common with a corner element
$\Omega_m^{v-e}$ then on the common face $\Gamma_{l,i}^{v-e}$ we
define the correction $\eta_l^{v-e}\left(x^{v-e}\right)$ so that
the corrected value of the spectral element function $p_l^{v-e}=
u_l^{v-e}(x^{v-e})+\eta_l^{v-e}(x^{v-e})$ assumes the value of
$u_m^{v-e}$ on the common face $\Gamma_{l,i}^{v-e}$.
\newline
If $\Omega_m^{v-e}$ is a corner element then we define the
correction $\eta_m^{v-e}(x^{v-e})\equiv 0.$

Finally, we define the corrections in $\Omega^e$, an edge
neighbourhood where $e \in \mathcal E$, the set of edge
neighbourhoods. Let $\Omega_l^e$ be an edge element. Then in
$\Omega^e$ the edge coordinates are
\begin{align*}
x_1^e &= \tau =\ \ln r \\
x_2^e &= \theta \\
x_3^e &= x_3\:.
\end{align*}
Let $\tilde{\Omega}_l^e$ denote the image of $\Omega^e$
in $x^e$ coordinates. Then
$$\tilde{\Omega}^e_l = \left\{x^e:\;\ln(r_j^e)<x_1^e
<\ln(r_{j+1}^e),\:\theta_k^e<x_2^e<\theta_{k+1}^e,\:
Z_n^e<x_3^e<Z_{n+1}^e\right\}.$$
Once again we introduce a set of local coordinates $z$ in
$\tilde{\Omega}^e_l$ defined as
\begin{align*}
z_1 &= x_1^e \\
z_2 &= x_2^e \\
z_3 &= \frac{x_3^e}{r_{j+1}^e}\:.
\end{align*}
Then $\tilde{\Omega}_l^e$ is mapped onto the hexahedron
$\widehat{\Omega}_l^e$ such that the length of the $z_3$ side
becomes large as $\Omega_l^e$ approaches the edge of the domain.
\newline
Now
$$\int_{\Omega_u^e}|\nabla_x u|^2 dx=\int_{\tilde{\Omega}_l^e}
\left(\left(\frac{\partial u}{\partial \tau}\right)^2+
\left(\frac{\partial u}{\partial\theta}\right)^2+
e^{2\tau}\left(\frac{\partial u}{\partial
x_3}\right)^2\right)\:d\tau \:d\theta\:dx_3\:.$$
Hence
$$\int_{\Omega_u^e}|\nabla_x u|^2 dx$$
is uniformly equivalent to
$$\int_{\widehat{\Omega}_u^e}|\nabla_z u|^2 r_{j+1}^e\,dz\:.$$

Here $u$ denotes the spectral element function defined on $\Omega_u^e\,.$ Now we can
define corrections to the spectral element functions so that the corrected element
functions are conforming as we have for vertex-edge elements.

Finally, a further set of corrections can be made so that the corrected element function
$p(x) \in H_0^1(\Omega)$, i.e. $p$ vanishes on $\Gamma^{[0]}$, the Dirichlet portion of
the boundary $\partial{\Omega}$ of $\Omega$. Moreover the estimate (\ref{eqC.1}) holds.

We now briefly indicate how the corrections need to be defined in case the spectral element
functions are nonconforming.

The corrections are first described for an element $\Omega_l^r\subseteq \Omega^r$, the
regular region of $\Omega$. Let $M_l^r\left(\lambda_1,\lambda_2,\lambda_3\right)$ denote
the map from the cube $Q=(-1,1)^3$ to $\Omega_l^r$ as shown in Figure \ref{figC.1}.

We first make a correction
$\delta_l^r(\lambda_1,\lambda_2,\lambda_3)=\sum_{i=0}^1\sum_{j=0}^1\sum_{k=0}^1
a_{i,j,k}\lambda_1^i\lambda_2^j \lambda_3^k$ such that $(u_l^r+\delta_l^r)(n)=\bar{u}(n)$,
where $n$ denotes a node of $Q=(M_l^r)^{-1}(\Omega_l^r)$ and $\bar{u}(n)$ the average value
of $u$ at the node $n$. Hence the corrected spectral element function is conforming at the
nodes of $\Omega_l^r$. Here we use the estimate
\[\left\|w\right\|^2_{L^{\infty}(R)}\leq C\left\|w\right\|^2_{H^{3/2}(R)}\]
where $R$ is a rectangle in the plane, each of whose sides is of length $=\Theta(1).$

Next, we define a correction $\epsilon_l^r(\lambda_1,\lambda_2,\lambda_3)$ such that if
$t$ is a point on a side $S$ of $Q=\left(M_l^r\right)^{-1}(\Omega_l^r)$ then
$(u_l^r+\delta_l^r+\epsilon_l^r)(t)=\overline{u+\delta}(t)$ where $\overline{u+\delta}(t)$
denotes the average value of $u+\delta$ at $t$. It should be noted that $\epsilon_l^r(t)=0$
if $t$ is a node of $Q$. Consider for example the side corresponding to $\lambda_1=-1,\:
\lambda_2=-1,\:-1<\lambda_3<1$ and let the correction corresponding to the side be
$g(\lambda_3)$. Then the contribution of the correction for this side to $\epsilon_l^r$ would
be $\frac{(1-\lambda_1)}{2}\frac{(1-\lambda_2)}{2} g(\lambda_3)$. The contributions of the
corrections for other sides would be similarly defined and $\epsilon_l^r$ would be the sum of
all these contributions. Here we use the estimate
$$\|w\|^2_{H^1(S)}\leq C \ln W \|w\|^2_{H^{3/2}(R)}$$
where $R$ is a rectangle in the plane, each of whose sides is of length $\Theta(1),$ and $S$
is a side of the rectangle $R$. Moreover $w$ is a polynomial of degree $W$ in each of its two
arguments.

Now $u_l^r+\delta_l^r+\epsilon_l^r$, would be conforming on the portion of the wirebasket
contained in $\Omega^r$, the regular region of $\Omega$. We shall next make corrections in
$\Omega^v$ for $v\in \mathcal V,\Omega^{v-e}$ for $v-e \in \mathcal {V-E}$ and $\Omega^e$
for $e\in \mathcal E$, so that the corrected spectral element functions are conforming on
the wire basket of the elements contained in $\Omega$. Once this has been done a further set
of corrections can be made so that the corrected spectral element functions would be
conforming on the faces as well as the wire basket and vanish on the Dirichlet portion of
the boundary $\Gamma^{[0]}$ of $\partial{\Omega}$ as has already been described.
\begin{figure}[!ht]
\centering
\includegraphics[scale = 0.60]{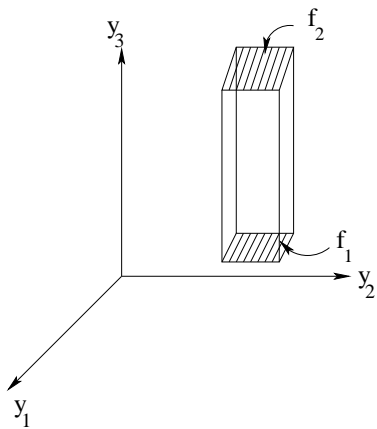}
\caption{The element $\hat{\Omega}_l^{v-e}$.}
\label{figC.5}
\end{figure}
The corrections $\delta_l^v+\epsilon_l^v$ that need to be made in the vertex neighbourhood
$\Omega^v$ of $\Omega$ so that $u_l^v+\delta_l^v+\epsilon_l^v$ would be conforming on the portion
of the wirebasket contained in $\Omega^v$ are very similar to those described for the regular
region $\Omega^r$ of $\Omega$ and hence are not discussed any further. We now describe the
corrections $\delta_l^{v-e}+\epsilon_l^{v-e}$ that need to be made in the vertex-edge
neighbourhood $\Omega^{v-e}$ of $\Omega$ so that $u_l^{v-e}+\delta_l^{v-e}+\epsilon_l^{v-e}$
would be conforming on the portion of the wirebasket contained in $\Omega^{v-e}$.

Let $\Omega_l^{v-e}$ be an element in the vertex-edge neighbourhood $\Omega^{v-e}$, and let
$\widehat{\Omega}_l^{v-e}$ be its image in $y$ coordinates where
\begin{align*}
y_1 &= \psi =\: \ln(\tan\phi) \\
y_2 &= \theta \\
y_3 &= \frac{\zeta}{\sin\left(\phi_{i+1}^{v-e}\right)} =\:
\frac{\ln(x_3)}{\sin\left(\phi_{i+1}^{v-e}\right)}\:.
\end{align*}
Then the length of the $y_3$ side becomes large as the elements approach the edge of the domain
$\Omega$. We define
$$\sigma_1(y_1,y_2)=\sum_{i,j=0,1}h_{i,j}^{(1)} y_1^i y_2^j,$$
a bilinear function of $y_1$ and $y_2$ so that
$$\left.{u_l^{v-e}(y_1,y_2,y_3)+\sigma_1(y_1,y_2)}\right|_{f_1}=w_1(y_1,y_2)$$
is conforming at the nodes of the face $f_1$. Here $w_1(y_1,y_2)$ assumes the average value of
the spectral element functions at the nodes of $f_1$.

In the same way we define
$$\sigma_2(y_1,y_2)=\sum_{i,j=0,1}h_{i,j}^{(2)}y_1^iy_2^j,$$
a bilinear function of $y_1$ and $y_2$, so that
$$u_l^{v-e}(y_1,y_2,y_3)+\sigma_2(y_1,y_2)|_{f_2}=w_2(y_1,y_2)$$
is conforming at the nodes of the face $f_2$.
\begin{figure}[!ht]
\centering
\includegraphics[scale = 0.60]{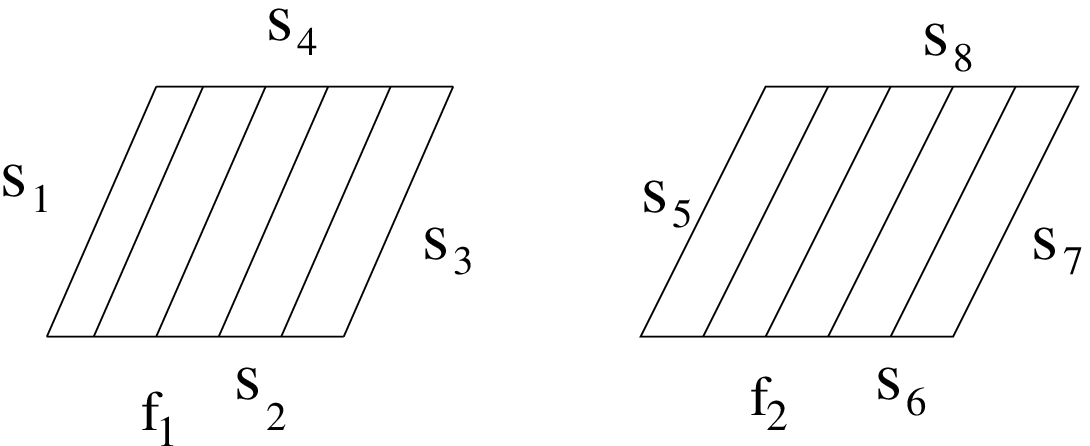}
\caption{Faces $f_1$ and $f_2$.}
\label{figC.6}
\end{figure}
The correction $\delta_l^{v-e}(y_1,y_2,y_3)$ is then given by
\[\delta_l^{v-e}(y_1,y_2,y_3)=\sigma_1(y_1,y_2)r(y_3)+\sigma_2(y_1,y_2)s(y_3)\:.\]
Here $r(y_3)$ and $s(y_3)$ are defined in (\ref{eqC.4}) and (\ref{eqC.5}).

Let
\[\nu_1(y_1,y_2)=\sum_{i=0}^W \sum_{j=0}^W o_{i,j}^{(1)}y_1^i y_2^j\]
be a polynomial of degree $N$ in $y_1$ and $y_2$ separately such that
\[{u_l^{v-e}(y_1,y_2,y_3)+\delta_l^{v-e}(y_1,y_2,y_3)}+\nu_1(y_1,y_2)\]
is conforming at the sides of the face $f_1$.

In the same way we define
\[\nu_2(y_1,y_2)=\sum_{i=0}^W \sum_{j=0}^W o_{i,j}^{(2)}y_1^i y_2^j\]
such that
\[{u_l^{v-e}(y_1,y_2,y_3)+\delta_l^{v-e}(y_1,y_2,y_3)}+\nu_2(y_1,y_2)\]
is conforming at the sides of the face $f_2$.

Define
\[\omega_1(y_1,y_2,y_3)=\nu_1(y_1,y_2)r(y_3)+\nu_2(y_1,y_2)s(y_3).\]
Here $r(y_3)$ and $s(y_3)$ are defined in (\ref{eqC.4}) and (\ref{eqC.5}).

Now $u_l^{v-e}+\delta_l^{v-e}+\omega_1$ is conforming on the face $f_1$ and $f_2$.
We can define corrections $\sigma_9(y_3),\ \sigma_{10}(y_3),\ \sigma_{11}(y_3)$
and $\sigma_{12}(y_3)$ such that
\[\left(u_l^{v-e}+\delta_l^{v-e}+\omega_1\right)(y_1,y_2,y_3)+\sigma_i{(y_3)}\]
is conforming on the sides $s_i$ as shown in Figure \ref{figC.7}. Note that $\sigma_i(y_3)$
is zero at the end points of the side $s_i$.
\begin{figure}[!ht]
\centering
\includegraphics[scale = 0.68]{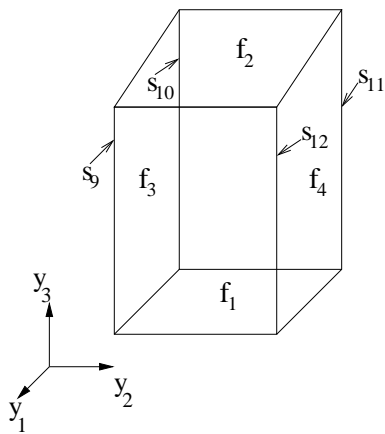}
\caption{Sides $s_i,i=1,\cdots,12$ and faces $f_i,i=1,\cdots,6$ of
a typical element in vertex-edge neighbourhood.}
\label{figC.7}
\end{figure}
On the face $f_3$ define $l_3(y_1,y_3)$ to be a
linear function of $y_1$ which assumes the value $\sigma_i(y_3)$ on the side $s_i$ for $i=9,10$.
Similarly on the face $f_4$ define $l_4(y_1,y_3)$ to be a linear function of $y_1$ which assumes
the value $\sigma_i(y_3)$ on the side $s_i$ for $i=11,12$. We now define $\omega_2(y_1,y_2,y_3)$
to be a linear function of $y_2$ such that $\omega_2$ assumes the value $l_3(y_1,y_3)$ on the
face $f_3$ and the value $l_4(y_1,y_3)$ on the face $f_4$. The correction
$\beta_l^{v-e}(y_1,y_2,y_3)$ is then given by
\[\beta_l^{v-e}(y_1,y_2,y_3)=\delta_l^{v-e}(y_1,y_2,y_3)+\omega_1(y_1,y_2,y_3)+\omega_2(y_1,y_2,y_3).\]
\end{proof}

\subsection*{C.2}
\textbf{Proof of Theorem \ref{thm3.3.1}}
\begin{thm3.3.1}
The following estimate for the spectral element functions holds
\begin{align}\label{eqC.6}
\mathcal U^{N,W}_{(1)}(\{\mathcal F_u\}) \leq K_{N,W}
\mathcal V^{N,W}(\{\mathcal F_u\})\:.
\end{align}
Here $K_{N,W}=CN^4$, when the boundary conditions are mixed and
$K_{N,W}=C(\ln W)^2$ when the boundary conditions are Dirichlet.
\newline
If the spectral element functions vanish on the wirebasket $W\!B$ of the elements then
$K_{N,W}=C(\ln W)^2$, where $C$ is a constant.
\end{thm3.3.1}
\begin{proof}
Let $p$ denote the corrected spectral element function defined in Lemma \ref{lem3.3.1}
and
\begin{align}\label{eqC.7}
B(p,\lambda)=\int_{\Omega}\left(\sum_{i,j=1}^3 a_{i,j}p_{x_i}\lambda_{x_j}
+\sum_{i=1}^3 b_i p_{x_i}\lambda+c p\lambda\right) dx
\end{align}
be a bilinear form for $p,\lambda \in H_0^1(\Omega)$. Here $H_0^1(\Omega)$ denotes the
subspace of functions $\subseteq$ $H^1(\Omega)$ which vanish on $\Gamma^{[0]}$. Moreover
there is a constant $K$ such that
\begin{equation}\label{eqC.8}
K\|q\|_{H^1(\Omega)}^2 \leq B(q,q)
\end{equation}
for all $q \in H_0^1(\Omega)$.
\newline
Now
\begin{equation}\label{eqC.9}
B(p,\lambda)=B_{regular}(p,\lambda)+B_{vertices}(p,\lambda)+
B_{vertex-edge}(p,\lambda)+B_{edges}(p,\lambda)\,.
\end{equation}
Here
\[B_{regular}(p,\lambda)=\sum_{l=1}^{N_r}B(p,\lambda)_{{\Omega}_l^r}\]
and
\[B(p,\lambda)_{{\Omega}_l^r}=\int_{{\Omega}_l^r}\sum_{i,j=1}
^3\left(a_{i,j} p_{x_i}\lambda_{x_j}+\sum_{i=1}^3b_ip_{x_i}
\lambda+cp\lambda\right) d x\:.\]
The other terms are similarly defined.
\newline
Now
$$p_l^r=u_l^r+\eta_l^r.$$
Hence
$$B(p,\lambda)_{{\Omega}_l^r}=B(u_l^r,\lambda)_{{\Omega}_l^r}
+B(\eta_l^r,\lambda)_{{\Omega}_l^r}.$$
Integrating by parts we obtain,
\[B(u_l^r,\lambda)_{{\Omega}_l^r}=\int_{{\Omega}_l^r} Lu_l^r
\:\lambda\:d x + \int_{\partial{{\Omega}_l^r}}\left(\frac
{\partial u_l^r}{\partial\nuw}\right)_A\lambda\:d\sigma\:.\]
Hence
\begin{align}\label{eqC.10}
B_{regular}(p,p)=\sum_{l=1}^{N_r}\int_{{\Omega}_l^r} Lu_l^r \:p\:d
x + \sum_{\Gamma_{l,i}^r\subseteq\Omega^r}\int_{\Gamma_{l,i}^r}
\left(\frac{\partial u}{\partial\nuw}\right)_A\:p\: d\sigma \notag \\
+ \sum_{\Gamma_{l,i}^r\subseteq\partial{\Omega}^r}\int_{\Gamma_{l,i}^r}
\left(\frac{\partial u}{\partial\nuw}\right)_A\:p\:d\sigma
+\sum_{l=1}^{N_r}B(\eta_l^r,p)_{{\Omega}_l^r}\:.
\end{align}
Now
$$B_{vertices}(p,\lambda)=\sum_{v \in \mathcal V}
\sum_{l=1}^{N_v}B(p,\lambda)_{\Omega_l^v}\:.$$
Clearly
$$B(p,\lambda)_{{\Omega}_l^v}=B(u_l^v,\lambda)_{{\Omega}_l^v}
+B(\eta_l^v,\lambda)_{{\Omega}_l^v}\:.$$
Once more using integration by parts
\begin{align*}
B(u_l^v,\lambda)_{{\Omega}_l^v}=\int_{\tilde{\Omega}_l^v} L^v
u(x^v)\:\lambda\:e^{\chi/2}{\sqrt{\sin\phi}}\:d x^v+
\int_{\partial{\tilde{\Omega}_l^v}}\left(\frac{\partial
u_l^v}{\partial\nuw^v}\right)_{A^v}\:\lambda\:e^{\chi}\sin\phi
\:d\sigma^v\:.
\end{align*}
Hence
\begin{align}\label{eqC.11}
B_{vertices}(p,p)&=\sum_{v \in \mathcal V}\left(\sum_{l=1}^{N_v}
\int_{\tilde{\Omega}_l^v}L^v u(x^v)\:\lambda\:e^{\chi/2}{\sqrt
{\sin\phi}}\:d x^v\right. \notag\\
&+\left.\sum_{\Gamma_{l,i}^v\subseteq\Omega^v}\int_{\tilde
{\Gamma}_{l,i}^v}\left[\left(\frac{\partial u}{\partial\nuw^v}
\right)_{A^v}\right] p\,e^\chi\sin\phi\:d\sigma^v\right.\notag \\
&+\left.\sum_{\Gamma_{l,i}^v\subseteq\partial{\Omega}^v}\int_
{\tilde{\Gamma}_{l,i}^v}\left(\frac{\partial u}{\partial\nuw^v}
\right)_{A^v} p\,e^\chi\sin\phi\: d \sigma^v +\sum_{l=1}^{N_v}
B(\eta_l^v,p)_{\Omega_l^v}\right).
\end{align}
Now if ${\Gamma}_{l,i}^v$ is a face of ${\Omega}_l^v$ and ${\Gamma}_{j,k}^r$ is a
face of ${\Omega}_j^r$ such that ${\Gamma}_{l,i}^v={\Gamma}_{j,k}^r$ then
\begin{equation}\label{eqC.12}
\int_{\Gamma_{j,k}^r}\left(\frac{\partial u_l^r}{\partial\nuw}
\right)_{A^v}\,p\:d\sigma=-\int_{\Gamma_{l,i}^v}\left(\frac
{\partial u_l^v}{\partial \nuw^v}\right)_{A^v}p e^\chi\sin\phi
\:d\sigma^v\,.
\end{equation}
Next
$$B_{vertex-edge}(p,\lambda)=\sum_{{v-e}\in\mathcal {V-E}}\sum
_{l=1}^{N_{v-e}}B(p,\lambda)_{{\Omega}_l^{v-e}}.$$
And
$$B(p,\lambda)_{{\Omega}_l^{v-e}}=B(u_l^{v-e},\lambda)_{{\Omega}
_l^{v-e}}+B(\eta_l^{v-e},\lambda)_{{\Omega}_l^{v-e}}.$$
Integration by parts gives
\begin{align*}
B(u_l^{v-e},\lambda)_{{\Omega}_l^{v-e}}&=\int_{\tilde{\Omega}_l^
{v-e}}L^{v-e} u(x^{v-e})\lambda(x^{v-e}) e^{\zeta/2} \:dx^{v-e}\\
&+\int_{\partial\tilde{\Omega}_l^{v-e}}\left(\frac{\partial u}
{\partial \nuw^{v-e}}\right)_{A^{v-e}}\lambda e^{\zeta}\:d\sigma^{v-e}\,.
\end{align*}
Hence
\begin{align}\label{eqC.13}
B_{vertex-edge}(p,p)&=\sum_{{v-e}\in\mathcal {V-E}}\left(
\sum_{l=1}^{N_{v-e}}\int_{\tilde{\Omega}_l^{v-e}}L^{v-e}
u(x^{v-e})\,p(x^{v-e})\,e^{\zeta/2}\:d x^{v-e}\right. \notag \\
&+\left. \sum_{{\Gamma}_{l,i}^{v-e}\subseteq\Omega^{v-e}}
\int_{\tilde{\Gamma}_{l,i}^{v-e}}\left[\left(\frac{\partial u}
{\partial\nuw^{v-e}}\right)_{A^{v-e}}\right]p\,e^\chi
\:d\sigma^{v-e}\right. \notag \\
&+ \left.\sum_{{\Gamma}_{l,i}^{v-e}\subseteq\partial
{\Omega}^{v-e}}\int_{\tilde{\Gamma}_{l,i}^{v-e}}\left(\frac{\partial u}
{\partial\nuw^{v-e}}\right)_{A^{v-e}}p\,e^\chi\:d\sigma^{v-e}\right. \notag \\
&+\left.\sum_{l=1}^{N_{v-e}}B(\eta_l^{v-e},p)_{\Omega_l^{v-e}}\right)\:.
\end{align}
Now if $\Gamma_{m,j}^v=\Gamma_{l,i}^{v-e}$ then
\begin{equation}\label{eqC.14}
\int_{\tilde{\Gamma}_{m,j}^v}\left(\frac{\partial u}{\partial\nuw^v}
\right)_{A^v}\:p\,e^\chi\sin\phi\: d\sigma^v=-\int_{\tilde
{\Gamma}_{l,i}^{v-e}}\left(\frac{\partial u}{\partial\nuw^{v-e}}
\right)_{A^{v-e}}\:p e^\zeta\:d\sigma^{v-e}.
\end{equation}
Finally
$$B_{edges}(p,\lambda)=\sum_{e\in\mathcal E}\sum_{l=1}^{N_e}
B(p,\lambda)_{\Omega_l^e}\:.$$
Here
$$B(p,\lambda)_{\Omega_l^e}=B(u_l^e,\lambda)_{\Omega_l^e}+
B(\eta_l^e,\lambda)_{\Omega_l^e}.$$
Integrating by parts we obtain
$$B(u_l^e,\lambda)_{\Omega_l^e}=\int_{\tilde{\Omega}_l^e}L^eu(x^e)
\lambda\:d x^e+\int_{\partial\tilde{\Omega}_l^e}\left(\frac
{\partial u}{\partial \nuw^e}\right)_{A^e}\lambda\:d \sigma^e.$$
Hence
\begin{align}\label{eqC.15}
B_{edges}(p,p)&=\sum_{e\in\mathcal E}\left(\sum_{l=1}^{N^e}
\int_{\tilde{\Omega}_l^e}L^eu(x^e)p\: d x^e\right. \notag \\
&+\sum_{{\Gamma}_{l,i}^{e}\subseteq\Omega^{e}}
\int_{\tilde{\Gamma}_{l,i}^{e}}\left[\left(\frac{\partial u}
{\partial\nuw^e}\right)_{A^e}\right]p\:d\sigma^e
+\sum_{{\Gamma}_{l,i}^e\subseteq\partial{\Omega}^e}\int_{\tilde
{\Gamma}_{l,i}^e}\left(\frac{\partial u}{\partial\nuw^e}\right)
_{A^e}p \:d\sigma^e \notag \\
&+ \left.\sum_{l=1}^{N_e}B(\eta_l^e,p)_{\Omega_l^e}\right)\:.
\end{align}
Let $\Gamma_{l,i}^{v-e}=\Gamma_{m,j}^e$. Then it can be shown that
\begin{equation}\label{eqC.16}
\int_{\tilde{\Gamma}_{l,i}^{v-e}}\left(\frac{\partial
u}{\partial \nuw^{v-e}}\right)_{A^{v-e}}p e^\zeta\:d
\sigma^{v-e}=-\int_{\tilde{\Gamma}_{m,j}^{e}}\left(\frac{\partial
u}{\partial \nuw^e} \right)_{A^{e}}p \:d \sigma^e
\end{equation}
Also if $\Gamma_{m,j}^{e}=\Gamma_{n,k}^{r}$ then
\begin{equation}\label{eqC.17}
\int_{\tilde{\Gamma}_{m,j}^{e}}\left(\frac{\partial u}{\partial \nuw^{e}}
\right)_{A^{e}}p \:d\sigma^{e}=-\int_{{\Gamma}_{n,k}^{r}}\left(\frac{\partial
u}{\partial \nuw}\right)_{A}p \:d \sigma
\end{equation}

Now
\begin{align}\label{eqC.18}
B(p,p)_{\Omega}=B_{regular}(p,p)+B_{vertices}(p,p)+B_{vertex-edges}(p,p)\notag \\
+B_{edges}(p,p)\:.
\end{align}
Substituting (\ref{eqC.10})$-$(\ref{eqC.17}) into the above we obtain
\begin{equation}\label{eqC.19}
B(p,p)_{\Omega}\leq cd\:\mathcal V^{N,W}\left({\mathcal \{F_u\}}\right)
+1/d\:{\mathcal G}(p)\:.
\end{equation}
In the above
\begin{align}\label{eqC.20}
{\mathcal G}(p)&=\|p(x)\|^2_{1,\Omega^r}+\sum_{v\in\mathcal V}
\|p(x^v)\:e^{x_3^{v-e}/2}\|^2_{1,\tilde{\Omega}^v}
+\sum_{{v-e}\in\mathcal {V-E}}\int_{\tilde{\Omega}^{v-e}}
\left(\sum_{i=1}^2\left(\frac{\partial p(x^{v-e})}{\partial x_i^
{v-e}}\right)^2\right. \notag\\
&\quad+\left.\sin^2\phi\left(\frac{\partial p(x^{v-e})} {\partial
x_3^{v-e}}\right)^2+\big(p\,(x^{v-e})\big)^2\right) e^{x_3^{v-e}}
w^{v-e}(x_1^{v-e})\: d x^{v-e} \notag \\
&\quad+\sum_{e\in \mathcal E}\int_{\tilde{\Omega}^e} \left
(\sum_{i=1}^2\left(\frac{\partial p(x^e)} {\partial
x_i^e}\right)^2 +e^{2\tau} \left(\frac{\partial p(x^e)} {\partial
x_3^e}\right)^2+\big(p\,(x^e)\big)^2\right)\:w^e(x_1^e)dx^e.
\end{align}
Now
$$K\|p\|^2_{H^1(\Omega)}\leq B(p,p)\,.$$
Moreover if the spectral element functions are not conforming on the wirebasket $WB$ then
\begin{equation}\label{eqC.21}
{\mathcal G}(p)\leq C_{N,W}\|p\|^2_{H^1(\Omega)}\,. 
\end{equation}
Here $C_{N,W}=EN^2$ if the boundary conditions are mixed and $C_{N,W}=E$ if the boundary
conditions are Dirichlet. Choosing $d=\frac{2E{N}^2}{K}$ in (\ref{eqC.19}) gives
$$B(p,p)_{\Omega}\leq \frac{c}{K}C_{N,W}\mathcal V^{N,W}(\{\mathcal F_u\})\,.$$
Hence combining the above and (\ref{eqC.20}) gives
\begin{equation}\label{eqC.22}
{\mathcal G}(p)\leq \frac{2c}{K^2}(C_{N,W})^{2}N^4\:{\mathcal V}^{N,W}
\left(\{{\mathcal F}_u\}\right).
\end{equation}
Now (\ref{eqC.1}) and (\ref{eqC.22}) yields (\ref{eqC.6}).

Finally, we consider the case when the spectral element functions are conforming on the
wirebasket $W\!B$ and vanish on it. Then $p$ vanishes on the wirebasket. Consider $p(x^e)$
on $\tilde{\Omega}^e_l$. Here
$$\tilde{\Omega}^e_l=\left\{x^e:\: \ln(r_i^e)<x_1^e<\ln(r_{i+1}^e),
\:\theta_j<x_2^e<\theta_{j+1},\:Z_k^e<x_3^e<Z_{k+1}^e\right\}\:.$$
Now if $s(y_1,y_2)$ is a polynomial of degree $W$ in each of its arguments on the unit
square $S$ which vanishes at one of the vertices of $S$ then using a \textit{scaling}
argument as in $\S$ 3.4 of~\cite{TW} we obtain
\[\int_{S}\left(s(y)\right)^2\:d y\leq F(\ln W)\: \int_{S}\left(\left(\frac{\partial s}
{\partial y_1}\right)^2+\left(\frac{\partial s}{\partial y_2}\right)^2\right)\:d y\:.\]
Hence, since $p(x_1^e,x_2^e,.)$ is a polynomial of degree $W$ in $x_1^e$ and $x_2^e$
$$\int_{\tilde{\Omega}_l^e}(p(x^e))^2\:d x^e\leq F(\ln W)
\int_{\tilde{\Omega}_l^e}\left(\sum_{i=1}^2\left(\frac{\partial
p(x^e)}{\partial x_i^e}\right)^2+e^{2\tau}\left(\frac{\partial
p(x^e)}{\partial x_3^e} \right)^2\right)\: d x^e\:.$$
And so it can be shown that
\begin{equation}\label{eqC.23}
{\mathcal G}(p)\leq F\ln W\|p\|^2_{H^1(\Omega)}\:. 
\end{equation}
Choosing $d=\frac{1}{\frac{K}{2F\ln W}}$ in (\ref{eqC.19}) gives
\begin{equation}\label{eqC.24}
{\mathcal G}(p)\leq \frac{2cF^2}{K^2}\left(\ln W\right)^2
{\mathcal V}^{N,W}\left(\{{\mathcal F}_u\}\right)\:. 
\end{equation}
Now combining (\ref{eqC.1}) and (\ref{eqC.24}) yields (\ref{eqC.6}).
\end{proof}
\end{appendix}
\thispagestyle{empty}
\chapcleardoublepage

\renewcommand{\theequation}{D.\arabic{equation}}
\renewcommand{\thefigure}{D.\arabic{figure}}
\setcounter{equation}{0}
\setcounter{figure}{0}
\addcontentsline{toc}{chapter}{Appendix D}
\markboth{\small Appendix D}{\small Appendix D}
\begin{appendix}
\section*{Appendix D}
\subsection*{D.1}
\textbf{Proof of Lemma \ref{lem3.4.1}}
\begin{lem3.4.1}
Let $\Omega_m^r$ and $\Omega_p^r$ be elements in the regular region $\Omega^r$ of $\Omega$
and $\Gamma_{m,i}^r$ be a face of $\Omega_m^r$ and $\Gamma_{p,j}^r$ be a face of $\Omega_p^r$
such that $\Gamma_{m,i}^r=\Gamma_{p,j}^r$. Then for any $\epsilon>0$ there exists a constant
$C_{\epsilon}$ such that for $W$ large enough
\begin{align}\label{eqD.1}
&\left|\int_{\partial{\Gamma}_{m,i}^r}\left(\left(\frac{\partial
u_m^r}{\partial \nuw}\right)_A\left(\frac{\partial u_m^r}
{\partial\nb}\right)_A-\left(\frac{\partial u_p^r}{\partial\nuw}
\right)_A\left(\frac{\partial u_p^r}{\partial\nb}\right)_A\right)
\:ds\right| \notag\\
&\leq C_{\epsilon}(\ln W)^2\sum_{k=1}^3\left\|[u_{x_k}]
\right\|^2_{1/2,{\Gamma_{m,i}^r}}+\epsilon\sum_{1\leq |\alpha|\leq 2}
\left(\|D_x^{\alpha}u_m^r\|^2_{0,\Omega_m^r}+\|D_x^{\alpha}u_p^r\|^2_{0,\Omega_p^r}\right)\:.
\end{align}
\end{lem3.4.1}
\begin{proof}
Let $\Gamma_{m,i}^r$ be the image of the mapping $M_m^r$ from $Q$ to $\Omega_m^r$ corresponding
to $\lambda_3=-1$ and $\Gamma_{p,j}^r$, the image of the mapping $M_m^r$ from $Q$ to $\Omega_p^r$
corresponding to $\lambda_3=1$. Then $\Gamma_{m,i}^r$ is the image of the square $S=(-1,1)^2$
under these mappings. However the proof remains valid if it is the image of $T$, the master
triangle, too.
\newline
Now
\begin{equation}\label{eqD.2}
\int_{\partial{\Gamma}_{m,i}^r}\left(\frac{\partial u}
{\partial \nuw}\right)_A\left(\frac{\partial u}{\partial \nb}
\right)_A\:d s=\int_{\partial S}\sum_{i,j=1}^3\alpha_{i,j}
\frac{\partial u}{\partial x_i}\frac{\partial u}{\partial x_j}\:ds\:.
\end{equation}
Here the element of arc length $d s$ in the right hand side denotes either $d\lambda_1$ or
$d\lambda_2$. Moreover $\alpha_{i,j}$ is an analytic function of its arguments. Now
\[\frac{\partial u_m^r}{\partial x_i}=\sum_{j=1}^3\beta_{i,j}\frac
{\partial u_m^r}{\partial \lambda_j}\]
where $\beta_{i,j}$ is an analytic function of $\lambda=(\lambda_1,\lambda_2,\lambda_3)$.

Let $\hat{\beta}_{i,j}$ be the projection of $\beta_{i,j}$ into the space of polynomials of
degree $W$ in $H^3(Q)$. Then for any $m>0$
\begin{equation}\label{eqD.3}
\|\beta_{i,j}-\hat\beta_{i,j}\|_{1,\infty,Q}\leq \frac{C_m}{W^m}
\end{equation}
as $W\rightarrow\infty$.
\newline
Define $\left(\frac{\partial u_m^r}{\partial x_i}\right)^a$, an approximate representation
for $\frac{\partial u_m^r}{\partial x_i}$, by
$$\left(\frac{\partial u_m^r}{\partial x_i}\right)^a=\sum_{j=1}^3
\hat\beta_{i,j}\frac{\partial u_m^r}{\partial \lambda_j}\:.$$
Moreover $\left(\frac{\partial u_m^r}{\partial x_i}\right)^a$ is a polynomial of degree $2W$
in each of its arguments $\lambda_1,\lambda_2$ and $\lambda_3$. Now
$$\left\|\frac{\partial u_m^r}{\partial x_i}-\left(\frac{\partial
u_m^r}{\partial x_i}\right)^a\right\|_{0,\partial S}\leq \frac{C}
{W^4}\sum_{j=1}^3\left\|\frac{\partial u_m^r}{\partial \lambda_j}
\right\|_{0,\partial S}\leq \frac{C}{W^5}\sum_{j=1}^3\left\|\frac
{\partial u_m^r}{\partial \lambda_j}(\lambda)\right\|_{5/4,Q}$$
by the trace theorem for Sobolev spaces.
\newline
Hence
\begin{equation}\label{eqD.4}
\left\|\frac{\partial u_m^r}{\partial x_i}-\left(\frac{\partial u_m^r}{\partial x_i}
\right)^a\right\|_{0,\partial S}\leq\frac{C}{W^4}\sum_{1\leq|\alpha|\leq2}
\left(\|D_{\lambda}^{\alpha}u^r_m\|^2_{0,Q}\right)^{1/2}\,.
\end{equation}
Here we have used the inverse inequality for differentiation
$$\left\|\frac{\partial u_m^r}{\partial \lambda_i}(\lambda)\right\|_{5/4,Q}\leq KW
\: \left\|\frac{\partial u_m^r}{\partial\lambda_i }(\lambda)\right\|_{1,Q}\:.$$
Now
$$\left\|\frac{\partial u_m^r}{\partial x_i}\right\|^2_{0,\partial S}\leq C\:
\sum_{i=1}^3\left\|\frac{\partial u_m^r}{\partial\lambda_i}\right\|^2_{0,\partial S}\:.$$
Moreover (see~\cite{TW}),
\begin{subequations}\label{eqD.5}
\begin{equation}\label{eqD.5a}
\left\|\frac{\partial u_m^r}{\partial x_i}\right\|^2_{0,\partial S}
\leq C(\ln W)\sum_{1\leq|\alpha|\leq2}\|D_{\lambda}^{\alpha}u^r_m\|^2_{0,Q}\:.
\end{equation}
In the same way it can be shown that
\begin{equation}\label{eqD.5b}
\left\|\left(\frac{\partial u_m^r}{\partial x_i}\right)^a\right\|^2_{0,\partial S}
\leq C(\ln W)\sum_{1\leq|\alpha|\leq2}\|D_{\lambda}^ {\alpha}u^r_m\|^2_{0,Q}\:.
\end{equation}
\end{subequations}
Now
$$\int_{\partial{\Gamma}_{m,i}^r}\left(\frac{\partial u_m^r}
{\partial \nuw}\right)_A\left(\frac{\partial u_m^r}{\partial \nb}
\right)_A \:d s=\int_{\partial S}\sum_{i,j=1}^3\alpha_{i,j}\left
(\frac{\partial u_m^r}{\partial x_i}\right)\left(\frac{\partial
u_m^r} {\partial x_j}\right)\, ds\,.$$
Using (\ref{eqD.4}) and (\ref{eqD.5}) gives
\begin{align}\label{eqD.6}
&\left|\int_{\partial S}\sum_{i,j=1}^3\alpha_{i,j}\left(\left
(\frac{\partial u_m^r}{\partial x_i}\right)^a\left(\frac{\partial
u_m^r} {\partial x_j}\right)^a - \frac{\partial u_m^r}{\partial
x_i}\frac{\partial u_m^r}{\partial x_j}\right)\: d s\right| \notag \\
&\quad \leq C \left(\sum_{i,j=1}^3\left\|\left(\frac{\partial
u_m^r}{\partial x_i}\right)^a\right\|_{0,\partial
S}\left\|\left(\frac{\partial u_m^r} {\partial x_j}\right)^a -
\left(\frac{\partial u_m^r}{\partial x_j}\right)\right\|_{0,\partial S}
\right. \notag \\
&\quad +\left.\sum_{i,j=1}^3\left\|\left(\frac{\partial
u_m^r}{\partial x_j}\right)\right\|_{0,\partial
S}\left\|\left(\frac{\partial u_m^r} {\partial x_i}\right)^a -
\left(\frac{\partial u_m^r}{\partial x_i}
\right)\right\|_{0,\partial S}\right) \notag \\
&\quad \leq C\frac{\sqrt{\ln W}}{W^4}\sum_{1\leq|\alpha|\leq2}
\|D_\lambda^ \alpha u^r_m\|^2_{0,Q}\:.
\end{align}
In the same way it can be shown that
\begin{align}\label{eqD.7}
&\left|\int_{\partial S}\sum_{i,j=1}^3\alpha_{i,j}\left(\left(\frac{\partial u_p^r}
{\partial x_i}\right)^a\left(\frac{\partial u_p^r}{\partial x_j}\right)^a
-\left(\frac{\partial u_p^r}{\partial x_i} \right)\left(\frac{\partial u_p^r}
{\partial x_j}\right)\right)\: d s\right| \notag \\
&\quad \leq C\frac{\sqrt{\ln W}} {W^4} \sum_{1\leq|\alpha|\leq2}\|D_\lambda^\alpha
u_p^r\|^2_{0,Q}\,.
\end{align}
Hence
\begin{align}\label{eqD.8}
&\left|\int_{\partial S}\sum_{i,j=1}^3\alpha_{i,j}\left(\left
(\frac{\partial u_m^r}{\partial x_i}\right)\left(\frac{\partial
u_m^r} {\partial x_j}\right) - \left(\frac{\partial
u_p^r}{\partial x_i} \right)\left(\frac{\partial u_p^r}{\partial
x_j}\right)\right)\,d s\right| \notag \\
&\quad \leq \left|\int_{\partial S} \sum_{i,j=1}^3
\alpha_{i,j}\left (\left(\frac{\partial u_m^r}{\partial
x_i}\right)^a\left(\frac {\partial u_m^r}{\partial x_j}\right)^a -
\left(\frac{\partial u_p^r} {\partial
x_i}\right)^a\left(\frac{\partial u_p^r}{\partial x_j}
\right)^a\right)\,d s\right| \notag \\
&\quad +\frac{C\sqrt{\ln W}}{W^4}\left(\sum_{1\leq|\alpha|\leq
2}\left( \|D_\lambda^\alpha u_m^r\|^2_{0,Q} + \|D_\lambda^\alpha
u_p^r\|^2_{0,Q} \right)\right)\,.
\end{align}
Now
\begin{align*}
&\left|\int_{\partial S}\sum_{i,j=1}^3\alpha_{i,j}
\left(\left(\frac{\partial u_m^r}{\partial x_i}
\right)^a\left(\frac{\partial u_m^r}{\partial x_j} \right)^a -
\left(\frac{\partial u_p^r}{\partial x_i}\right)^a\left(\frac
{\partial u_p^r}{\partial x_j}\right)^a\right)\, d s\right|\\
&\quad \leq C\,\left[\sum_{i,j=1}^3\left(\left\|\left(\frac
{\partial u_m^r}{\partial x_i}\right)^a\right\|_{0,
\partial S}\left\|\left(\frac{\partial u_m^r}{\partial
x_j}\right)^a - \left(\frac{\partial u_p^r}{\partial x_j}
\right)^a\right\|_{0,\partial S}\right.\right.\\
&\quad +\left.\left.\left\|\left(\frac{\partial u_p^r}{\partial
x_j}\right)^a\right\|_{0,\partial S}\left\|\left(\frac{\partial
u_m^r}{\partial x_i}\right)^a - \left(\frac{\partial
u_p^r}{\partial x_i}\right)^a \right\|_{0,\partial
S}\right)\right]
\end{align*}
Hence for any $\epsilon>0$ there is a constant $C_{\epsilon}$ such that
\begin{align}\label{eqD.9}
&\left|\int_{\partial S}\sum_{i,j=1}^3\alpha_{i,j}
\left(\left(\frac{\partial u_m^r}{\partial x_i}\right)^a
\left(\frac{\partial u_m^r}{\partial x_j}\right)^a -
\left(\frac{\partial u_p^r}{\partial x_i}\right)^a
\left(\frac{\partial u_p^r}{\partial x_j}\right)^a\right)
\,d s\right| \notag \\
&\quad \leq C_{\epsilon}(\ln W)\left(\sum_{i,j=1}^3
\left\|\left(\frac{\partial u_m^r}{\partial x_i}\right)^a -
\left(\frac{\partial u_p^r}{\partial x_i}\right)^a\right
\|^2_{0,\partial S}\right) \notag \\
&\quad +\epsilon/3\left(\sum_{1\leq |\alpha|\leq 2}\left(
\|D_\lambda ^\alpha u_m^r\|^2_{0,Q} + \|D_\lambda^\alpha
u_p^r\|^2_{0,Q} \right)\right)\,.
\end{align}
Here we have used the estimate (\ref{eqD.5})
$$\left\|\left(\frac{\partial u_m^r}{\partial x_i}\right)^a
\right\|^2_{0,\partial S}\leq C(\ln W)\left(\sum_{1\leq
|\alpha|\leq2}\|D_\lambda^\alpha u_m^r\|^2_{0,Q}\right)\,.$$
and
$$\left\|\left(\frac{\partial u_p^r}{\partial x_i}\right)^a
\right\|^2_{0,\partial S}\leq C(\ln W)\left(\sum_{1\leq
|\alpha|\leq2}\|D_\lambda^\alpha u_p^r\|^2_{0,Q}\right)\,.$$
Now as in~\cite{TW}
\begin{equation}\label{eqD.10}
\left\|\left(\frac{\partial u_m^r}{\partial x_i}
\right)^a - \left(\frac{\partial u_p^r}{\partial x_i}\right)
^a\right\|^2_{0,\partial S}\leq C(\ln W)\: \left(\left\|\left
(\frac{\partial u_m^r}{\partial x_i}\right)^a - \left(\frac
{\partial u_p^r}{\partial x_i}\right)^a\right\|^2_{1/2,S}\right)
\end{equation}
since $\left(\frac{\partial u_m^r}{\partial x_i}\right)$ and
$\left(\frac{\partial u_p^r}{\partial x_i}\right)$ are polynomials of degree $2W$ in
each of their arguments.
\newline
Now combining (\ref{eqD.9}) and (\ref{eqD.10}) gives
\begin{align}\label{eqD.11}
&\left|\int_{\partial S}\sum_{i,j=1}^3\alpha_{i,j}\left(\left
(\frac{\partial u_m^r}{\partial x_i}\right)^a\left(\frac{\partial
u_m^r} {\partial x_j}\right)^a - \left(\frac{\partial
u_p^r}{\partial x_i} \right)^a\left(\frac{\partial u_p^r}{\partial
x_j}\right)^a\right)\: d s\right| \notag \\
&\quad \leq C_{\epsilon}(\ln W)^2\: \left(\sum_{i,j=1} ^3\left\|
\left(\frac{\partial u_m^r}{\partial x_i}\right)^a -
\left(\frac{\partial u_p^r}{\partial x_i}\right)^a\right\|^2_{1/2,S}\right) \notag \\
&\quad +\epsilon/3 \left(\sum_{1\leq|\alpha|\leq2}\left(\|D_\lambda^\alpha u_m^r\|^2_{0,Q}
+\|D_\lambda^\alpha u_p^r\|^2_{0,Q}\right)\right).
\end{align}
Next
\begin{align*}
\left\|\left(\frac{\partial u_m^r}{\partial x_i}\right)^a
-\left(\frac{\partial u_m^r}{\partial x_i}\right)\right\|_{1/2,S}
&\leq\sum_{j=1}^3\left\|(\beta_{i,j}-\hat{\beta}_{i,j})\frac
{\partial u_m^r}{\partial\lambda_j}\right\|_{1/2,S}\\
&\leq \sum_{j=1}^3\left\|(\beta_{i,j} -
\hat{\beta}_{i,j})\right\|_{1,\infty,S}\left\| \frac{\partial
u_m^r}{\partial\lambda_j}\right\|_{1/2,S}\:.
\end{align*}
Hence
\begin{subequations}\label{eqD.12}
\begin{equation}\label{eqD.12a}
\left\|\left(\frac{\partial u_m^r}{\partial x_i}\right)^a-\left
(\frac{\partial u_m^r}{\partial x_i}\right)\right\|^2_{1/2,S} \leq
\frac{C}{W^8}\:\left(\sum_{1\leq|\alpha|\leq2}\|D_\lambda^\alpha
u_m^r\|^2_{0,Q}\right).
\end{equation}
In the same way it can be shown that
\begin{equation}\label{eqD.12b}
\left\|\left(\frac{\partial u_p^r}{\partial x_i}\right)^a-\left
(\frac{\partial u_p^r}{\partial x_i}\right)\right\|^2_{1/2,S} \leq
\frac{C}{W^8}\:\left(\sum_{1\leq|\alpha|\leq2}\|D_\lambda^\alpha
u_p^r\|^2_{0,Q}\right).
\end{equation}
\end{subequations}
Combining (\ref{eqD.11}) and (\ref{eqD.12}) and choosing $W$ large enough gives
\begin{align}\label{eqD.13}
&\left|\int_{\partial S}\sum_{i,j=1}^3\alpha_{i,j}\left(\left
(\frac{\partial u_m^r}{\partial x_i}\right)^a\left(\frac{\partial
u_m^r} {\partial x_j}\right)^a - \left(\frac{\partial
u_p^r}{\partial x_i} \right)^a\left(\frac{\partial u_p^r}{\partial
x_j}\right)^a\right)\: d s\right| \notag \\
&\quad \leq C_{\epsilon}(\ln W)^2\:\left(\sum_{i=1}^3\left\|\frac
{\partial u_m^r}{\partial x_i} - \frac{\partial u_p^r}{\partial x_i}
\right\|^2_{1/2,S}\right)+2\epsilon/3\left(\sum_{1\leq|\alpha|\leq2}
\|D_\lambda^\alpha u_p^r\|^2_{0,Q}\right).
\end{align}
Now substituting (\ref{eqD.13}) into (\ref{eqD.8}) and choosing $W$ large enough yields
\begin{align}\label{eqD.14}
&\left|\int_{\partial S}\sum_{i,j=1}^3\alpha_{i,j}\left(\left(\frac
{\partial u_m^r}{\partial x_i}\right)\left(\frac{\partial u_m^r}
{\partial x_j}\right) - \left(\frac{\partial u_p^r}{\partial x_i}
\right)\left(\frac{\partial u_p^r}{\partial x_j}\right)\right)
\: d s\right| \notag \\
&\quad \leq C_{\epsilon}(\ln W)^2\:\left(\sum_{i=1}^3\left\|\frac
{\partial u_m^r}{\partial x_i} - \frac{\partial u_p^r}{\partial x_i}
\right\|^2_{1/2,S}\right)+\epsilon\left(\sum_{1\leq|\alpha|\leq2}
\|D_\lambda^\alpha u_p^r\|^2_{0,Q}\right)\:.
\end{align}
And so the required result (\ref{eqD.1}) follows.
\end{proof}

\subsection*{D.2}
\textbf{Proof of Lemma \ref{lem3.4.2}}
\begin{lem3.4.2}
Let $\Omega_m^r$ and $\Omega_p^r$ be elements in the regular
region $\Omega^r$ of $\Omega$ and $\Gamma_{m,i}^r$ be a face of
$\Omega_m^r$ and $\Gamma_{p,j}^r$ be a face of $\Omega_p^r$ such
that $\Gamma_{m,i}^r=\Gamma_{p,j}^r$. Then for any $\epsilon>0$
there exists a constant $C_{\epsilon}$ such that for $W$ large
enough
\begin{align}\label{eqD.15}
&\left|\sum_{j=1}^2\left(\int_{{\Gamma}_{m,i}^r}\left(\frac
{\partial u_m^r}{\partial \taux_j}\right)_A\frac{\partial}
{\partial s_j}\left(\frac{\partial u_m^r}{\partial\nuw}\right)
_A-\int_{{\Gamma}_{p,j}^r}\left(\frac{\partial u_p^r}
{\partial\taux_j}\right)_A\frac{\partial}{\partial s_j}
\left(\frac{\partial u_p^r}{\partial\nuw}\right)_A\right)
d\sigma\right| \notag \\
&\hspace{1.0cm}\leq C_{\epsilon}(\ln W)^2\sum_{k=1}^3
\left\|[u_{x_k}]\right\|^2_{1/2,{{\Gamma}_{p,j}^r}} \notag \\
&\hspace{1.0cm}+\epsilon\sum_{1\leq|\alpha| \leq 2}
\left(\|D_x^{\alpha}u_m^r\|^2_{0,\Omega_m^r}+\|D_x^{\alpha}
u_p^r\|^2_{0,\Omega_p^r}\right)\:.
\end{align}
\end{lem3.4.2}
\begin{proof}
It is enough to show that
\begin{align}\label{eqD.16}
&\left|\int_{{\Gamma}_{m,i}^r}\left(\left(\frac{\partial u_m^r}
{\partial \taux_j}\right)_A\frac{\partial}{\partial s_j}\left
(\frac{\partial u_m^r}{\partial\nuw}\right)_A-\left(\frac{\partial
u_p^r} {\partial\taux_j}\right)_A\frac{\partial}{\partial s_j}
\left(\frac{\partial u_p^r}{\partial\nuw}\right)_A\right)
\: d \sigma\right| \notag \\
&\quad\quad\leq C_{\epsilon}(\ln W)^2\sum_{k=1}^3
\left\|[u_{x_k}]\right\|^2_{1/2,{{\Gamma}_{p,j}^r}} \notag \\
&\quad\quad+\epsilon\sum_{1\leq|\alpha|\leq 2}
\left(\|D_x^{\alpha}u_m^r\|^2_{0,\Omega_m^r}+\|D_x^{\alpha}
u_p^r\|^2_{0,\Omega_p^r}\right)
\end{align}
for $j=1$.

Let $\Gamma_{m,i}^r$ be the image of the mapping $M_m^r$ from $Q$
to $\Omega_m^r$ corresponding to $\lambda_3=-1$ and $\Gamma_{p,j}^r$,
the image of the mapping $M_m^r$ from $Q$ to $\Omega_p^r$
corresponding to $\lambda_3=1$. Then $\Gamma_{m,i}^r$ is the image
of the square $S=(-1,1)^2$ under these mappings. However the proof
remains valid if it is the image of $T$, the master triangle, too.
\newline
Now
\begin{align}\label{eqD.17}
&\int_{{\Gamma}_{m,i}^r}\left\{\left(\frac{\partial u}{\partial
\taux_1}\right)_A\frac{\partial}{\partial s_1}\left(\frac
{\partial u}{\partial\nuw}\right)_A\right\}\:d\sigma \notag \\
&\quad=\sum_{j=1}^2\sum_{i,k=1}^3\int_S\left\{\left(a_{i,j}\frac
{\partial u}{\partial x_i}\right ) \frac{\partial}
{\partial\lambda_j}\left(b_k\frac{\partial u}{\partial
x_k}\right)\right\} \: d\lambda_1 d\lambda_2
\end{align}
where $a_{i,j}(\lambda_1,\lambda_2)$ and $b_k(\lambda_1,\lambda_2)$
are analytic functions of $\lambda_1$ and $\lambda_2$.

Now as in Lemma \ref{lem3.4.1}, let $\hat{a}_{i,j}(\lambda_1,\lambda_2)$
and $\hat{b}_k(\lambda_1,\lambda_2)$ be the projections of
$a_{i,j} (\lambda_1,\lambda_2)$ and $b_k(\lambda_1,\lambda_2)$
into the space of polynomials of degree $W$ in $H^3(S)$. Moreover,
let $\left(\frac {\partial u_m^r}{\partial x_i}\right)^a$,
$\left(\frac{\partial u_p^r} {\partial x_i}\right)^a$ be the
approximations to $\frac{\partial u_m^r} {\partial x_i}$ and
$\frac{ \partial u_p^r}{\partial x_i}$ defined in Lemma \ref{lem3.4.1}.
Then $\left(\frac {\partial u_m^r}{\partial x_i}\right) ^a$ and
$\left(\frac{\partial u_p^r} {\partial x_i}\right)^a$ are
polynomials of degree $2W$ in each of the variables
$\lambda_1,\lambda_2$ and $\lambda_3$.
\newline
Now
\begin{align}\label{eqD.18}
&\left|\int_S\left(a_{i,j}\left(\frac{\partial u_m^r} {\partial
x_i}\right)\frac{\partial}{\partial\lambda_j}\left(b_k\frac
{\partial u_m^r} {\partial x_k}\right)-\hat{a}_{i,j}\left(\frac
{\partial u_m^r} {\partial x_i}\right)\frac{\partial}{\partial
\lambda_j}\left(\hat{b}_k\frac{\partial u_m^r}{\partial x_k}
\right)\right)\: d \lambda_1 d\lambda_2\right| \notag \\
&\qquad\leq \left|\int_S (a_{i,j}-\hat{a}_{i,j})\frac{\partial
u_m^r} {\partial x_i}\frac{\partial}{\partial\lambda_j}\left(b_k
\frac{\partial u_m^r}{\partial x_k}\right)\:d \lambda_1
d \lambda_2\right| \notag \\
&\qquad+ \left|\int_S \hat{a}_{i,j}\frac{\partial u_m^r}{\partial
x_i}\frac {\partial}{\partial\lambda_j}\left((b_k-\hat{b}_k)\frac
{\partial u_m^r} {\partial x_k}\right)\:d \lambda_1
d \lambda_2\right| \notag \\
&\qquad\leq \frac{C}{W^4}\sum_{l,m=1}^3\left\|\frac{\partial
u_m^r}{\partial\lambda_l}\right\|_{0,S}\left(\left\|
\frac{\partial u_m^r}{\partial \lambda_m}\right\|_{0,S}
+\left\|\frac{\partial^2 u_m^r}{\partial \lambda
_m\partial \lambda_j}\right\|_{0,S}\right) \notag \\
&\qquad\leq \frac{C}{W^4}\sum_{l=1}^3 \left\|\frac{\partial
u_m^r}{\partial \lambda_l}\right\|^2_{7/4,Q}\:.
\end{align}
To obtain the above inequality we have used the trace theorem for
Sobolev spaces. Hence by the inverse inequality for differentiation
\begin{align}\label{eqD.19}
&\left|\int_S\left(a_{i,j}\left(\frac{\partial u_m^r} {\partial
x_i}\right)\frac{\partial}{\partial\lambda_j}\left(b_k\frac
{\partial u_m^r} {\partial
x_k}\right)-\hat{a}_{i,j}\left(\frac{\partial u_m^r} {\partial
x_i}\right)\frac{\partial}{\partial \lambda_j}\left(\hat{b}_k
\frac{\partial u_m^r}{\partial x_k}\right)\right)\: d \lambda_1
d\lambda_2\right| \notag \\
&\qquad\leq\frac{C}{W}\sum_{1\leq|\alpha|\leq2}\left\|D
_{\lambda}^\alpha u_m^r\right\|^2_{2,Q}\:.
\end{align}
Combining (\ref{eqD.17}) and (\ref{eqD.19}) gives
\begin{align}\label{eqD.20}
&\left|\int_S\left(a_{i,j}\frac{\partial u_m^r}{\partial x_i}
\frac{\partial}{\partial \lambda_j}\left(b_k\frac{\partial
u_m^r}{\partial x_k}\right) - a_{i,j}\frac{\partial
u_p^r}{\partial x_i} \frac{\partial}{\partial
\lambda_j}\left(b_k\frac{\partial u_p^r} {\partial
x_k}\right)\right)\: d \lambda_1 d \lambda_2\right| \notag \\
&\quad\leq \left|\int_S\left(\hat{a}_{i,j}\frac{\partial u_m^r}
{\partial x_i} \frac{\partial}{\partial \lambda_j}\left
(\hat{b}_k\frac{\partial u_m^r} {\partial x_k}\right) -
\hat{a}_{i,j} \left(\frac{\partial u_p^r}{\partial
x_i}\right)\frac{\partial}{\partial \lambda_j}\left
(\hat{b}_k\frac{\partial u_p^r}{\partial x_k}\right)\right)\:
d\lambda_1 d \lambda_2\right| \notag \\
&\quad+\frac{C}{W}\left(\sum_{1\leq|\alpha|\leq2}\left(\left\|
D_{\lambda}^\alpha u_m^r\right\|^2_{2,Q}+\left\|D_{\lambda}^
\alpha u_p^r\right\|^2_ {2,Q}\right)\right).
\end{align}
Next
\begin{align}\label{eqD.21}
&\left|\int_S\left(\hat{a}_{i,j}\frac{\partial u_m^r}{\partial
x_i}\frac{\partial}{\partial\lambda_j}\left(\hat{b}_k\frac
{\partial u_m^r} {\partial x_k}\right)
-\hat{a}_{i,j}\left(\frac{\partial u_m^r}{\partial
x_i}\right)^a\frac{\partial}{\partial\lambda_j}\left(\hat{b}_k
\left(\frac{\partial u_m^r}{\partial x_k}\right)^a\right)\right)
\:d \lambda_1d\lambda_2\right| \notag \\
&\quad\quad\quad\leq \left|\int_S \hat{a}_{i,j}\frac{\partial
u_m^r}{\partial x_i}\frac{\partial}
{\partial\lambda_j}\left(\hat{b}_k \left(\frac{\partial
u_m^r}{\partial x_k}-\left(\frac{\partial u_m^r}{\partial
x_k}\right)^a\right)\right)\,d \lambda_1 d
\lambda_2\right| \notag \\
&\quad\quad\quad+ \left|\int_S\hat{a}_{i,j}
\left(\left(\frac{\partial u_m^r}{\partial
x_i}\right)^a-\frac{\partial u_m^r}{\partial x_i}
\right)\frac{\partial}{\partial \lambda_j}
\left(\left(\frac{\partial u_m^r}{\partial
x_k}\right)^a\right)\,d \lambda_1 d \lambda_2\right| \notag \\
&\quad\quad\quad\leq C\left(\left\|\frac{\partial u_m^r}{\partial
x_i} \right\|_{0,S} \left (\left\| \frac{\partial}
{\partial\lambda_j}\left (\frac{\partial u_m^r}{\partial
x_k}-\left(\frac{\partial u_m^r}
{\partial x_k}\right)^a\right)\right\|_{0,S}\right)\right. \notag \\
&\quad\quad\quad\left.+\left\|\left(\frac{\partial u_m^r}{\partial
x_i} \right)^a-\frac{\partial u_m^r}{\partial x_i} \right\|_{0,S}
\left\|\frac{\partial}{\partial\lambda_j}\left(\left(\frac
{\partial u_m^r}{\partial x_k} \right)^a
\right)\right\|_{0,S}\right).
\end{align}
Now
\begin{align}\label{eqD.22}
\left\|\frac{\partial u_m^r}{\partial x_i}\right\|_{0,S}
\leq C \sum_{|\alpha|\leq1}\|D_{\lambda}^{\alpha}
u_m^r\|_{1,Q}\:.
\end{align}
and
\begin{align}\label{eqD.23}
\left\|\frac{\partial}{\partial \lambda_j}\left(\frac
{\partial u_m^r}{\partial x_k}-\left(\frac{\partial u_m^r}
{\partial x_k}\right)^a\right)\right\|_{0,S}\leq\frac{C}{W^4}
\sum_{1\leq|\alpha|\leq2}\|D_{\lambda}^{\alpha}u_m^r\|_{0,S}\:.
\end{align}
Hence
\begin{align}\label{eqD.24}
\left\|\frac{\partial}{\partial \lambda_j}\left(\frac
{\partial u_m^r}{\partial x_k}-\left(\frac{\partial
u_m^r}{\partial x_k}\right)^a\right)\right\|_{0,S}\leq
\frac{C}{W^4}\sum_{|\alpha|=1}\|D_{\lambda}^{\alpha}
u_m^r\|_{2,Q}\:.
\end{align}
Here we have used the trace theorem for Sobolev spaces.
\newline
Now using the inverse inequality for differentiation in
(\ref{eqD.24}) gives
\begin{align}\label{eqD.25}
\left\|\frac{\partial}{\partial \lambda_j}\left
(\frac{\partial u_m^r}{\partial x_k}-\left(\frac{\partial
u_m^r}{\partial x_k}\right)^a\right)\right\|_{0,S}\leq
\frac{C}{W^2}\sum_{1\leq|
\alpha|\leq2}\|D_{\lambda}^{\alpha}u_m^r\|_{1,Q}\:.
\end{align}
Combining (\ref{eqD.22}) and (\ref{eqD.25}) gives
\begin{align}\label{eqD.26}
\left\|\frac{\partial u_m^r}{\partial x_i}\right
\|_{0,S}\left\|\frac{\partial}{\partial
\lambda_j}\left(\frac{\partial u_m^r} {\partial
x_k}-\left(\frac{\partial u_m^r}{\partial x_k}\right)^a
\right)\right\|_{0,S}\leq\frac{C}{W^2}
\sum_{1\leq|\alpha|\leq2}\|D_{\lambda}^{\alpha}u_m^r\|_{0,Q}\:.
\end{align}
Next
\[\left\|\left(\frac{\partial u_m^r}{\partial x_i}\right)^a
-\frac{\partial u_m^r}{\partial x_i}\right\|_{0,S}\leq \frac{C}
{W^4}\sum_{|\alpha|=1}\|D_{\lambda}^{\alpha}u_m^r\|_{0,S}\:.\]
Hence by the trace theorem for Sobolev spaces
\begin{align}\label{eqD.27}
\left\|\left(\frac{\partial u_m^r}{\partial x_i}\right)^a
-\frac{\partial u_m^r}{\partial x_i}\right\|_{0,S}\leq\frac{C}
{W^4}\sum_{1\leq|\alpha|\leq2}\|D_{\lambda}^{\alpha}u_m^r\|_{0,Q}\:.
\end{align}
And
$$\left\|\frac{\partial}{\partial\lambda_j}\left(\left(\frac
{\partial u_m^r}{\partial x_k}\right)^a\right)\right\|_{0,S}\leq C
\sum_{1\leq |\alpha|\leq
2}\|D_{\lambda}^{\alpha}u_m^r\|_{0,S}\,.$$
Using the trace theorem and inverse inequality for differentiation gives
\begin{align}\label{eqD.28}
\left\|\frac{\partial}{\partial\lambda_j}\left(\left(\frac
{\partial u_m^r}{\partial x_k}\right)^a\right)\right\|_{0,S}\leq
CW^3 \sum_{1\leq |\alpha|\leq 2}\|D_{\lambda}^{\alpha}u_m^r\|_{0,Q}\:.
\end{align}
Combining (\ref{eqD.27}) and (\ref{eqD.28}) we obtain
\begin{align}\label{eqD.29}
&\left\|\left(\frac{\partial u_m^r}{\partial x_i}\right)^a
-\frac{\partial u_m^r}{\partial x_i}\right\|_{0,S}\left\|\frac
{\partial} {\partial\lambda_j}\left(\left(\frac{\partial
u_m^r}{\partial x_k}\right)^a\right)\right\|_{0,S} \notag \\
&\quad\quad\leq\frac{C}{W}\left(\sum_{1\leq|\alpha|\leq2}
\|D_{\lambda}^{\alpha}u_m^r\|^2_{0,Q}\right).
\end{align}
Substituting (\ref{eqD.26}) and (\ref{eqD.29}) into (\ref{eqD.21})
yields
\begin{subequations}\label{eqD.30}
\begin{align}\label{eqD.30a}
&\left|\int_S\left(\hat{a}_{i,j}\frac{\partial u_m^r}{\partial
x_i} \frac{\partial}{\partial
\lambda_j}\left(\hat{b}_k\frac{\partial u_m^r} {\partial
x_k}\right) - \hat{a}_{i,j}\left(\frac{\partial u_m^r}{\partial
x_i}\right)^a\frac{\partial}{\partial
\lambda_j}\left(\hat{b}_k\left (\frac{\partial u_m^r}{\partial
x_k}\right)^a\right)\right)\:d\lambda_1d\lambda_2\right| \notag\\
&\hspace{3.0cm}\leq \frac{C}{W}\left(\sum_{1 \leq|\alpha|\leq 2}
\|D_{\lambda}^{\alpha}u_m^r\|^2_{0,Q}\right).
\end{align}
In the same way it can be shown that
\begin{align}\label{eqD.30b}
&\left|\int_S\left(\hat{a}_{i,j}\frac{\partial u_p^r}{\partial
x_i} \frac{\partial}{\partial
\lambda_j}\left(\hat{b}_k\frac{\partial u_p^r}{\partial
x_k}\right) - \hat{a}_{i,j}\left(\frac{\partial u_p^r}{\partial
x_i}\right)^a\frac{\partial}{\partial\lambda_j}\left(\hat{b}_k\left
(\frac{\partial u_p^r}{\partial x_k}\right)^a\right)\right)\:d
\lambda_1 d\lambda_2\right| \notag \\
&\hspace{3.0cm}\leq\frac{C}{W}\left(\sum_{1\leq|\alpha|\leq2}
\|D_{\lambda}^{\alpha}u_p^r\|^2_{0,Q}\right).
\end{align}
\end{subequations}
Combining (\ref{eqD.20}) and (\ref{eqD.30}) gives
\begin{align}\label{eqD.31}
&\left|\int_S\left(a_{i,j}\frac{\partial u_m^r}{\partial x_i}
\frac{\partial}{\partial \lambda_j}\left(b_k\frac{\partial u_m^r}
{\partial x_k}\right) - a_{i,j}\frac{\partial u_p^r}{\partial
x_i}\frac{\partial}{\partial \lambda_j}\left(b_k\left(\frac
{\partial u_p^r}{\partial x_k}\right)\right)\right)\: d
\lambda_1 d\lambda_2\right| \notag \\
&\quad\leq\left|\int_S\left(\hat{a}_{i,j}\left(\frac{\partial u_m^r}
{\partial x_i}\right)\frac{\partial}{\partial \lambda_j}\left(\hat{b}_k
\left(\frac{\partial u_m^r}{\partial x_k}\right)^a\right)\right.
\right. \notag \\
&\quad-\left.\left.\hat{a}_{i,j}\left(\frac{\partial u_p^r}
{\partial x_i}\right)^a\frac{\partial}{\partial \lambda_j}
\left(\hat{b}_k\left(\frac{\partial u_p^r}{\partial x_k}\right)^a
\right)\right)\: d \lambda_1 d\lambda_2\right| \notag \\
&\quad+\frac{C}{W}\left(\sum_{1\leq|\alpha|\leq2}\left(\|D_x^
{\alpha}u_m^r\|^2_{0,Q}+\|D_x^{\alpha}u_m^r\|^2_{0,Q}\right)\right)\:.
\end{align}
Clearly
\begin{align}\label{eqD.32}
&\left|\int_S\left(\hat{a}_{i,j}\left(\frac{\partial u_m^r}
{\partial x_i}\right)^a\frac{\partial}{\partial \lambda_j}
\left(\hat{b}_k\left(\frac{\partial u_m^r}{\partial x_k}\right)^a
\right)\right.\right. \notag \\
&\quad-\left.\left.\hat{a}_{i,j}\left(\frac{\partial u_p^r}
{\partial x_i}\right)^a\frac{\partial}{\partial \lambda_j}
\left(\hat{b}_k\left(\frac{\partial u_p^r}{\partial x_k}\right)^a
\right)\right)\: d \lambda_1 d\lambda_2\right| \notag \\
&\quad\leq\left|\int_S\hat{a}_{i,j}\left(\frac{\partial
u_m^r}{\partial x_i}\right)^a\frac{\partial}{\partial
\lambda_j}\left(\hat{b}_k\left(\left(\frac{\partial
u_m^r}{\partial x_k}\right)^a - \left(\frac{\partial u_p^r}
{\partial x_k}\right)^a\right)\right)\:d\lambda_1 d\lambda_2
\right| \notag \\
&\quad+\left|\int_S\left(\hat{a}_{i,j}\left(\left(\frac{\partial
u_m^r}{\partial x_i}\right)^a- \left(\frac{\partial u_p^r}
{\partial x_i}\right)^a\right)\frac{\partial}{\partial \lambda_j}
\left(\hat{b}_k\left(\frac{\partial u_p^r}{\partial x_k}\right)^a
\right)\right)\: d \lambda_1 d\lambda_2\right|\:.
\end{align}
Now if $e(\lambda_1,\lambda_2)$ and $f(\lambda_1,\lambda_2)$ are
polynomials of degree $W$ in $\lambda_1$ and $\lambda_2$ then
using Theorem $4.1$ of~\cite{DTK1}
\begin{equation}\label{eqD.33}
\left|\int_S e(\lambda_1,\lambda_2)\frac{\partial}
{\partial\lambda_j}\left(f(\lambda_1,\lambda_2)\right) \: d
\lambda_1 d\lambda_2\right| \leq C(\ln W)
\|e\|_{1/2,S}\|f\|_{1/2,S}\:.
\end{equation}
The estimate (\ref{eqD.32}) remains valid if we replace the square
$S$ by the master triangle $T$.
\newline
Hence for any $\eta>0$ there is a constant $C$ such that
\begin{align}\label{eqD.34}
&\left|\int_S\hat{a}_{i,j}\left(\frac{\partial u_m^r} {\partial
x_i}\right)^a\frac{\partial}{\partial \lambda_j}
\left(\hat{b}_k\left(\left(\frac{\partial u_m^r}{\partial x_k}
\right)^a - \left(\frac{\partial u_p^r}{\partial x_k}
\right)^a\right)\right)\:d\lambda_1 d\lambda_2\right| \notag \\
&\leq\frac{C}{\eta}(\ln W)^2\left\|\hat{b}_k
\left(\left(\frac{\partial u_m^r}{\partial x_k}\right)^a -
\left(\frac{\partial u_p^r}{\partial x_k}\right)^a\right)
\right\|^2_{1/2,S}+\eta\left\|\hat{a}_{i,j}\left(\frac
{\partial u_m^r}{\partial x_i}\right)^a\right\|^2_{1/2,S} \notag \\
&\leq K\left(\frac{1}{\eta}(\ln W)^2\left\|\left(\left(\frac
{\partial u_m^r}{\partial x_k}\right)^a - \left(\frac {\partial
u_p^r}{\partial x_k}\right)^a\right)\right\|^2
_{1/2,S}+\eta\left\|\left(\frac{\partial u_m^r}{\partial
x_i}\right)^a\right\|^2_{1/2,S}\right)\:.
\end{align}
Here we have used the estimate
\begin{equation}\label{eqD.35}
\|gh\|_{1/2,S}\leq C\|g\|_{1,\infty,S}\,\|h\|_{1/2,S}\:.
\end{equation}
Now
\begin{align}\label{eqD.36}
\left\|\left(\frac{\partial u_m^r}{\partial x_k} \right)^a -
\left(\frac{\partial u_p^r}{\partial x_k}\right)^a\right
\|^2_{1/2,S}&\leq\left(\left\|\frac{\partial
u_m^r}{\partial x_k} - \frac{\partial u_p^r} {\partial
x_k}\right\|^2_{1/2,S}+\left\|\frac{\partial u_p^r}{\partial
x_k}-\left(\frac{\partial u_p^r}
{\partial x_k}\right)^a\right\|^2_{1/2,S}\right. \notag \\
&\left.+\left\|\frac{\partial u_m^r} {\partial x_k}
-\left(\frac{\partial u_m^r} {\partial
x_k}\right)^a\right\|^2_{1/2,S}\right)\:.
\end{align}
By the trace theorem for Sobolev spaces
\begin{equation}\label{eqD.37}
\left\|\frac{\partial u_p^r}{\partial x_k}
- \left(\frac{\partial u_p^r}{\partial x_k}\right)^a
\right\|_{1/2,S}\leq C\left\|\frac{\partial u_p^r} {\partial x_k}
- \left(\frac{\partial u_p^r}{\partial x_k}
\right)^a\right\|_{1,Q}
\end{equation}
and
\begin{equation}\label{eqD.38}
\left\|\frac{\partial u_p^r}{\partial x_k}
-\left(\frac{\partial u_p^r} {\partial x_k}\right)^a
\right\|_{1,Q}\leq \frac{C}{W^4}\left(\sum_{1\leq|\alpha|\leq
2}\|D_{\lambda}^{\alpha}u_p^r\|_{1,Q}\right).
\end{equation}
Substituting (\ref{eqD.37}) and (\ref{eqD.38}) into (\ref{eqD.36}) gives
\begin{align}\label{eqD.39}
\left\|\left(\frac{\partial u_m^r}{\partial x_k}\right)^a -
\left(\frac{\partial u_p^r}{\partial x_k}\right)^a\right
\|^2_{1/2,S} &\leq C\left\|\frac{\partial u_m^r}{\partial x_k} -
\frac{\partial u_p^r}{\partial x_k}\right\|^2_{1/2,S} \notag\\
&+\frac{K}{W^8}\left(\sum_{1\leq|\alpha|\leq 2}\|D_x^{\alpha}
u_p^r\|^2_{0,\Omega_p^r}+\|D_x^{\alpha}u_m^r\|^2
_{0,\Omega_m^r}\right).
\end{align}
Next by the trace theorem for Sobolev spaces
$$\left\|\left(\frac{\partial u_m^r}{\partial x_i}\right)
^a\right\|^2_{1/2,S}\leq C \left\|\left(\frac{\partial u_m^r}
{\partial x_i}\right)^a\right\|_{1,Q}$$
and
$$\left\|\left(\frac{\partial u_m^r}{\partial x_i}\right)^a
\right\|_{1,Q}\leq C \left\|\frac{\partial u_m^r}{\partial
x_i}\right\|_{1,Q}+\frac{C}{W^4}\left(\sum_{1\leq\|\alpha\|\leq2}
\left\|D_x^\alpha u_m^r\right\|_{1,Q}\right).$$
Hence
\begin{equation}\label{eqD.40}
\left\|\left(\frac{\partial u_m^r}{\partial x_i}
\right)^a\right\|_{1/2,S}\leq C\left(\sum_{1\leq|\alpha|\leq 2}
\|D_x^{\alpha}u_m^r\|_{\Omega_m^r}\right).
\end{equation}
Substituting (\ref{eqD.39}) and (\ref{eqD.40}) into (\ref{eqD.34})
for any $\eta>0$
\begin{align}\label{eqD.41}
&\left|\int_S\hat{a}_{i,j}\left(\frac{\partial u_m^r}{\partial
x_i}\right)^a\frac{\partial}{\partial\lambda_j}\left(\hat{b}_k
\left (\left(\frac{\partial u_m^r}{\partial x_k}\right)^a
-\left(\frac {\partial u_p^r}{\partial x_k}\right)^a\right)\right)
\:d\lambda_1 d\lambda_2\right| \notag \\
&\quad\leq K\left(\frac{(\ln W)^2}{\eta}\left\|\frac {\partial
u_m^r}{\partial x_k} - \frac{\partial u_p^r}{\partial
x_k}\right\|^2_{1/2,S}+\eta\left(\sum_{1\leq|\alpha|\leq 2}
\|D_x^{\alpha}u_m^r\|^2_{0,\Omega_m^r}\right)\right)
\end{align}
by choosing $W$ large enough.
\newline
In the same way it can be shown that
\begin{align}\label{eqD.42}
&\left|\int_S\hat{a}_{i,j}\left(\left(\frac{\partial u_m^r}
{\partial x_i}\right)^a-\left(\frac{\partial u_p^r}{\partial
x_i}\right)^a \right)\frac{\partial}{\partial
\lambda_j}\left(\hat{b}_k\left(\frac {\partial u_p^r}{\partial
x_k}\right)^a\right)\:d\lambda_1d\lambda_2\right| \notag \\
&\quad\leq K\left(\frac{(\ln W)^2}{\eta}\left\|\frac{\partial
u_m^r}{\partial x_k} - \frac{\partial u_p^r}{\partial x_k}\right\|
^2_{1/2,S}+\eta\left(\sum_{1\leq|\alpha|\leq 2}\|D_x^{\alpha}u_p^r
\|^2_{0,\Omega_m^r}\right)\right)\:.
\end{align}
Substituting (\ref{eqD.41}) and (\ref{eqD.42}) into (\ref{eqD.31})
and combining it with (\ref{eqD.32}) gives
\begin{align}\label{eqD.43}
&\left|\int_S\left(a_{i,j}\frac{\partial u_m^r}{\partial x_i}
\frac{\partial}{\partial \lambda_j}\left(b_k\frac{\partial u_m^r}
{\partial x_k}\right) - a_{i,j}\frac{\partial u_p^r}{\partial x_i}
\frac{\partial}{\partial \lambda_j}\left(b_k\frac{\partial
u_p^r}{\partial x_k}\right)\right)\: d \lambda_1 d\lambda_2
\right| \notag\\
&\quad\leq K\left(\frac{(\ln W)^2}{\eta}\left(\sum_{k=1}^3
\left\|[u_{x_k}]\right\|^2_{1/2,\Gamma_{p,j}^r}\right)\right.
\notag \\
&\quad+\left.\eta\left(\sum_{1\leq|\alpha|\leq
2}\left(\|D_x^{\alpha}u_m^r\|^2_{0,Q}+\|D_x^{\alpha}u_p^r
\|^2_{0,Q}\right)\right)\right)\:.
\end{align}
Now choosing $\epsilon=K\eta$, (\ref{eqD.15}) follows.
\end{proof}

\subsection*{D.3}
\textbf{Proof of Lemma 3.4.3}
\begin{lem3.4.3}
Let $\Omega_m^e$ and $\Omega_p^e$ be elements in the edge
neighbourhood $\Omega^e$ of $\Omega$ and $\Gamma_{m,i}^e$
be a face of $\Omega_m^e$ and $\Gamma_{p,j}^e$ be a face
of $\Omega_p^e$ such that $\Gamma_{m,i}^e=\Gamma_{p,j}^e$
and $\mu(\tilde{\Gamma}_{m,i}^e)<\infty$.
Then for any $\epsilon>0$ there exists a constant
$C_{\epsilon}$ such that for $W$ large enough
\begin{subequations}\label{eqD.44}
\begin{align}
&\left|\oint_{\partial\tilde{\Gamma}_{m,i}^e}\left(\left
(\frac{\partial u_m^e}{\partial \nb^e}\right)_{A^e}\left(
\frac{\partial u_m^e}{\partial\nuw^e}\right)_{A^e}-\left(
\frac{\partial u_p^e}{\partial\nb^e}\right)_{A^e}\left(
\frac{\partial u_p^e}{\partial\nuw^e}\right)_{A^e}\right)
\:d s^e\right| \notag \\
&\quad\quad\quad\leq C_{\epsilon}(\ln
W)^2\left(\big|\big|\big|\:[u_{x_1^e}]
\:\big|\big|\big|^2_{\tilde{\Gamma}_{m,i}^e}+\big|\big|\big|
\:[u_{x_2^e}]\:\big|\big|\big|^2_{\tilde
{\Gamma}_{m,i}^e}+\big|\big|\big|\:G_{m,i}^e[u_{x_3^e}]
\:\big|\big|\big|^2_{\tilde{\Gamma}_{m,i}^e}\right)\\
&\quad\quad\quad+\epsilon\sum_{k=m,p}\left(\int_{\tilde
{\Omega}_k^e}\left(\sum_{i,j=1,2}\left(\frac{\partial^2 u_k^e}
{\partial x_i^e\partial x_j^e} \right)^2
+e^{2\tau}\sum_{i=1}^2\left(\frac{\partial^2 u_k^e}{\partial
x_k^e\partial x_3^e}\right)^2+e^{4\tau}\left(\frac{\partial^2
u_k^e}{\left(\partial x_3^e\right)^2}\right)^2\right.\right.
\notag \\
&\quad\quad\quad\left.\left.+\sum_{i=1}^2\left(\frac {\partial
u_k^e}{\partial x_i^e}\right)^2+e^{2\tau}\left(\frac {\partial
u_k^e}{\partial x_3^e}\right)^2\right)\:w^e(x_1^e)\;dx^e\right)\:.
\end{align}
Here $C_\epsilon$ is a constant which depend on $\epsilon$ but is uniform for all
$\tilde{\Gamma}_{m,i}^e\subseteq\tilde{\Omega}^e$, and $G_{m,i}^e=\underset{x^e\in
{\tilde{\Gamma}_{m,i}^e}}{sup}(e^\tau)$\,.
\newline
If $\mu(\tilde{\Gamma}_{m,i}^e)=\infty$ then for any $\epsilon>0$ for $W,N$ large
enough
\begin{align}
&\left|\oint_{\partial\tilde{\Gamma}_{m,i}^e}\left(\left(\frac{\partial u_m^e}
{\partial \nb^e}\right)_{A^e}\left(\frac{\partial u_m^e}{\partial\nuw^e}
\right)_{A^e}-\left(\frac{\partial u_p^e}{\partial\nb^e}\right)_{A^e}\left(
\frac{\partial u_p^e}{\partial\nuw^e}\right)_{A^e}\right)\:d s^e\right|\notag\\
&\quad\leq\epsilon\left(\int_{\tilde{\Omega}_m^e}(u_m^e)^2w^e(x_1^e)\:dx^e
+\int_{\tilde{\Omega}_m^e}(u_m^e)^2w^e(x_1^e)\:dx^e\right)
\end{align}
\end{subequations}
provided $W=O(e^{{N}^\alpha})$ with $\alpha<1/2$.
\end{lem3.4.3}
\begin{proof}
In the edge neighbourhood $\Omega^e$ the system of coordinates
used are
\begin{subequations}\label{eqD.45}
\begin{align}\label{eqD.45a}
x_1^e&=\tau \notag \\
x_2^e&=\theta \notag \\
x_3^e&=x_3\,.
\end{align}
There are two cases to be considered. the first case is when
$$\tilde{\Gamma}_{m,i}^e=\left\{x^e:\alpha_0<x_1^e<\alpha_1\:,
\beta_0<x_2^e<\beta_1\:,x_3^e=\gamma_0\right\}\:.$$
\begin{figure}[!ht]
\centering
\includegraphics[scale = 0.60]{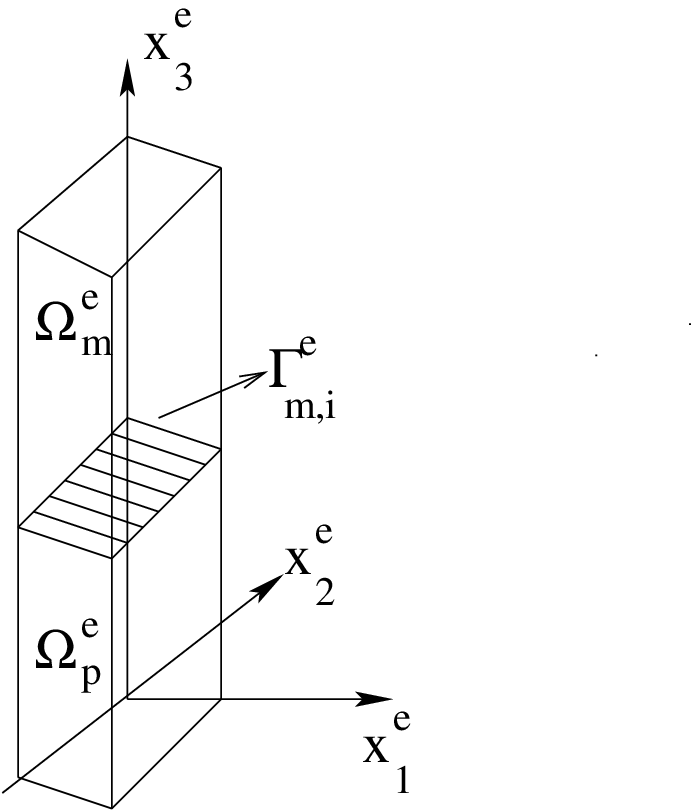}
\caption{A common face in the interior of the edge neighbourhood
in $x_1^e-x_2^e$ plane.}
\label{figD.1}
\end{figure}
We introduce local variables
\begin{align}\label{eqD.45b}
z_1&=x_1^e \notag\\
z_2&=x_2^e \notag\\
z_3&=\frac{x_3^e}{G_{m,i}^e}\,.
\end{align}
\end{subequations}
Now
\begin{align*}
\left(\frac{\partial u}{\partial x_1^e}\right)_{A^e} &= \left[
\begin{array}{ccc} 1 & 0 & 0 \end{array} \right]\:A^e \left[
\begin{array}{ccc}
\frac{\partial u}{\partial x_1^e}\\
\frac{\partial u}{\partial x_2^e}\\
\frac{\partial u}{\partial x_3^e}
\end{array} \right]\\
&=\left[ \begin{array}{ccc} 1 & 0 & 0 \end{array} \right] \left[
\begin{array}{ccc}
1 & 0 & 0\\
0 & 1 & 0\\
0 & 0 & \frac{1}{G_{m,i}^e}\end{array} \right] A^e\left[
\begin{array}{ccc}
1 & 0 & 0\\
0 & 1 & 0\\
0 & 0 & \frac{1}{G_{m,i}^e}\end{array} \right] \left[
\begin{array}{ccc}
\frac{\partial u}{\partial z_1}\\
\frac{\partial u}{\partial z_2}\\
\frac{\partial u}{\partial z_3}
\end{array} \right]
\end{align*}
Hence we obtain
\begin{subequations}\label{eqD.46}
\begin{align}
\left(\frac{\partial u}{\partial x_1^e}\right)_{A^e}&=
\left(\frac{\partial u}{\partial z_1}\right)_{\widehat{A}^e} \\
\left(\frac{\partial u}{\partial x_2^e}\right)_{A^e}&=
\left(\frac{\partial u}{\partial z_2}\right)_{\widehat{A}^e} \\
\left(\frac{\partial u}{\partial x_3^e}\right)_{A^e}&=
G_{m,i}^e\left(\frac{\partial u}{\partial z_3}\right)_{\widehat{A}^e}.
\end{align}
\end{subequations}
Here using (\ref{eq3.76})
\begin{subequations}\label{eqD.47}
\begin{equation}\label{eqD.47a}
\widehat{A}^e=(\widehat{S}^e)^TA\widehat{S}^e
\end{equation}
where
\begin{align}\label{eqD.47b}
\widehat{S}^e &=\left[ \begin{array}{ccc}
\cos\theta  & -\sin\theta & 0 \\
\sin\theta  & \cos\theta  & 0 \\
0           & 0           & \frac{e^\tau}{G_{m,i}^e}
\end{array} \right]\:.
\end{align}
\end{subequations}
Hence $\widehat{A}^e$ and its derivatives with respect to $z$ are uniformly bounded in
$\widehat{\Omega}_m^e$ and $\widehat{\Omega}_p^e$. Here $\widehat{\Omega}_m^e$ and
$\widehat{\Omega}_p^e$ are the images of ${\Omega}_m^e$ and ${\Omega}_p^e$ in $z$
variables respectively.
\begin{figure}[!ht]
\centering
\includegraphics[scale = 0.60]{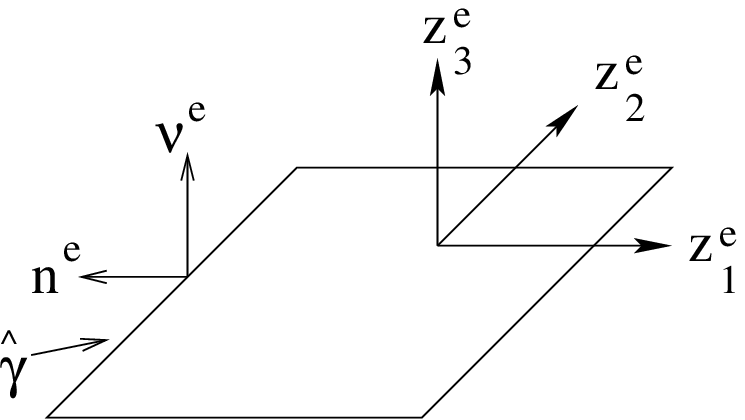}
\caption{$\widehat{\gamma}$, image of $\tilde{\gamma}$, in $z-$coordinates.}
\label{figD.2}
\end{figure}
Now
\begin{align}\label{eqD.48}
&\int_{\tilde\gamma}\left(\left(\frac{\partial
u_m^e}{\partial\nb^e} \right)_{A^e}\left(\frac{\partial
u_m^e}{\partial\nuw^e}\right) _{A^e}-\left(\frac{\partial
u_p^e}{\partial\nb^e}\right)_{A^e}\left(\frac{\partial u_p^e}
{\partial\nuw^e}\right)_{A^e}\right)\:d s^e \notag \\
&=\int_{\tilde\gamma}\left(\left(\frac{\partial u_m^e} {\partial
x_1^e}\right)_{A^e}\left(\frac{\partial u_m^e}{\partial
x_3^e}\right)_{A^e}-\left(\frac{\partial u_p^e}{\partial x_1^e}
\right)_{A^e}\left(\frac{\partial u_p^e}{\partial
x_3^e}\right)_{A^e}\right)\:d x_2^e \notag \\
&=G_{m,i}^e\int_{\hat\gamma}\left(\left(\frac{\partial u_m^e}
{\partial z_1}\right)_{\widehat{A}^e}\left(\frac{\partial
u_m^e}{\partial z_3}\right)_{\widehat{A}^e}-\left(\frac{\partial
u_p^e}{\partial z_1} \right)_{\widehat{A}^e}\left(\frac{\partial
u_p^e}{\partial z_3}\right) _{\widehat{A}^e}\right)\:d z_2\:.
\end{align}
Here $\widehat\gamma$ is the image of $\tilde\gamma$ in $z$ coordinates as shown in Figure
\ref{figD.2}.

Proceeding as in Lemma \ref{lem3.4.1} we can prove using the representation (\ref{eqD.48})
that for any $\eta>0$ there exists a constant $C_{\eta}$, which depends on $\eta$, but is
uniform for all $\tilde{\Gamma}_{m,i}^e \subseteq\tilde{\Omega}^e$ such that
\begin{align}\label{eqD.49}
&\left|\int_{\partial\tilde{\Gamma}_{m,i}^e}\left(\left(\frac
{\partial u_m^e}{\partial\nb^e}\right)_{A^e}\left(\frac{\partial
u_m^e}{\partial \nuw^e}\right)_{A^e}-\left(\frac{\partial
u_p^e}{\partial\nb^e}\right) _{A^e}\left(\frac{\partial
u_p^e}{\partial\nuw^e}\right)_{A^e}\right)\:d s^e\right| \notag\\
&\leq G_{m,i}^e\left(C_{\eta}(\ln W)^2\sum_{k=1}^3
\left\|[u_{z_k}]\right\|^2_{1/2,\hat{\Gamma}_{m,i}^e}\right.\notag \\
&+\left.\eta\left(\sum_{1\leq|\alpha|\leq 2}\left(\|D_z^{\alpha}
u_m^e\|^2_{0,\hat{\Omega}_m^e}+\|D_z^{\alpha}u_p^e\|^2_{0,\hat{\Omega}_p^e}
\right)\right)\right).
\end{align}
Here $\widehat{\Gamma}_{m,i}^e$ is the image of $\Gamma_{m,i}^e$ in $z$ coordinates and
$\widehat{\Omega}_m^e$ and $\widehat{\Omega}_p^e$ are the images of $\Omega_m^e$ and
$\Omega_p^e$ in $z$ coordinates. And by a proper choice of $\eta$ in (\ref{eqD.49}) we
obtain (\ref{eqD.44}).

The other case to be considered is when the boundary $\tilde{\Gamma} _{m,i}^e$ corresponds to
\[\tilde{\Gamma}_{m,i}^e=\left\{x^e:\alpha_0<x_1^e<\alpha_1,
\:x_2^e=\beta_0,\:\gamma_0<x_3^e<\gamma_1\right\}\:.\]
\begin{figure}[!ht]
\centering
\includegraphics[scale = 0.60]{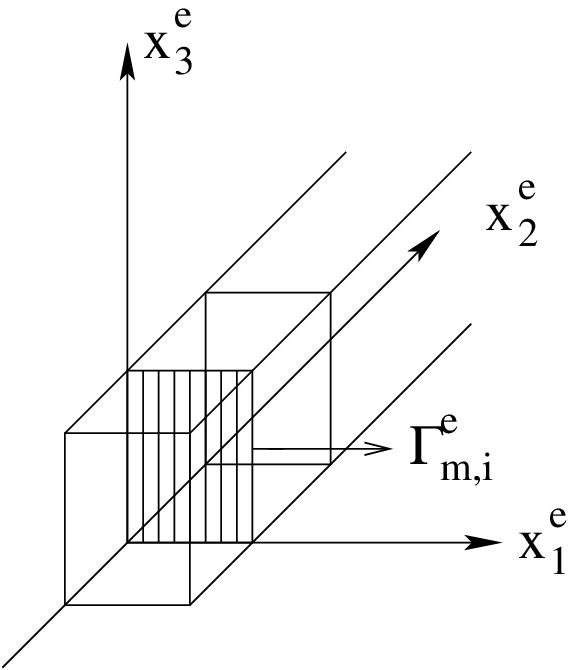}
\caption{A common face in the interior of the edge neighbourhood in $x_1^e-x_3^e$ plane.}
\label{figD.3}
\end{figure}
The case when
\[\tilde{\Gamma}_{m,i}^e=\left\{x^e:x_1^e=\alpha_0,\:\beta_0<x_2^e<
\beta_1,\:\gamma_0<x_3^e<\gamma_1\right\}\]
is essentially the same and hence is omitted.
\newline
Now
\begin{align}\label{eqD.50}
&\int_{\tilde\gamma}\left(\left(\frac{\partial
u_m^e}{\partial\nb^e} \right)_{A^e}\left(\frac{\partial
u_m^e}{\partial\nuw^e}\right) _{A^e}-\left(\frac{\partial
u_p^e}{\partial\nb^e}\right)_{A^e}\left(\frac{\partial u_p^e}
{\partial\nuw^e}\right)_{A^e}\right)\:d s^e \notag\\
&=\int_{\tilde\gamma}\left(\left(\frac{\partial u_m^e} {\partial
x_1^e}\right)_{A^e}\left(\frac{\partial u_m^e}{\partial
x_2^e}\right)_{A^e}-\left(\frac{\partial u_p^e}{\partial x_1^e}
\right)_{A^e}\left(\frac{\partial u_p^e}{\partial
x_2^e}\right)_{A^e}\right)\:d x_3^e \notag\\
&=G_{m,i}^e\int_{\hat\gamma}\left(\left(\frac{\partial u_m^e}
{\partial z_1}\right)_{\widehat{A}^e}\left(\frac{\partial
u_m^e}{\partial z_2}\right)_{\widehat{A}^e}-\left(\frac{\partial
u_p^e}{\partial z_1} \right)_{\widehat{A}^e}\left(\frac{\partial
u_p^e}{\partial z_2}\right) _{\widehat{A}^e}\right)\:dz_3\:.
\end{align}
\begin{figure}[!ht]
\centering
\includegraphics[scale = 0.60]{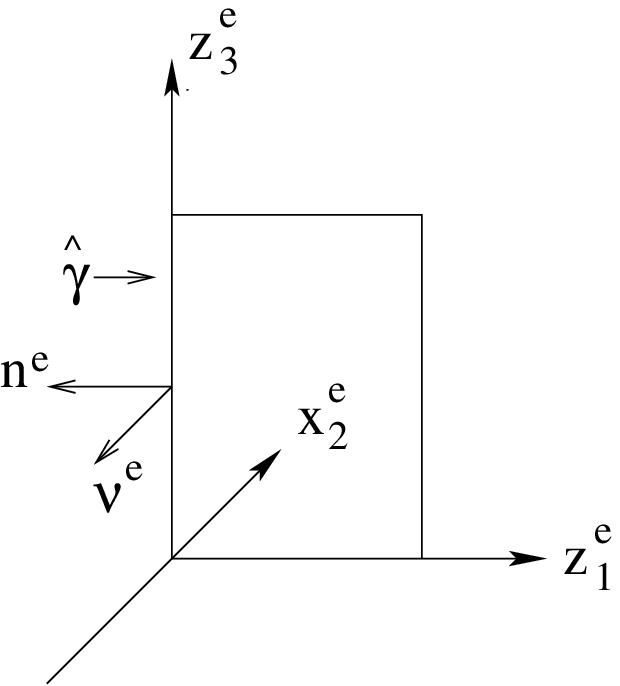}
\caption{$\widehat{\gamma}$, image of $\tilde{\gamma}$, in $z-$coordinates.}
\label{figD.4}
\end{figure}
Here $\widehat\gamma$ is the image of $\tilde\gamma$ in $z$ coordinates as shown in Figure \ref{figD.4}.
\newline
\begin{figure}[!ht]
\centering
\includegraphics[scale = 0.60]{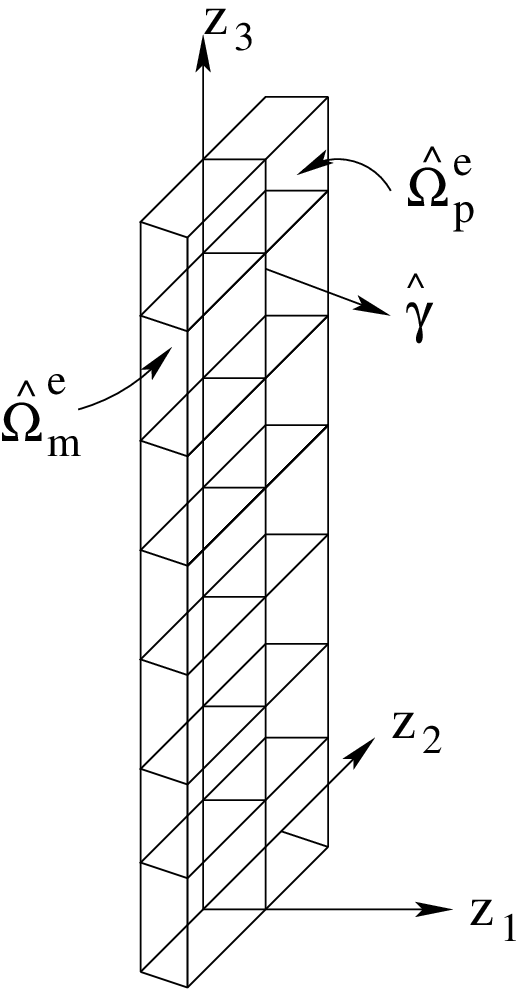}
\caption{Division of $\widehat{\Omega}_m^e$ into smaller rectangles.}
\label{figD.5}
\end{figure}
Let $\widehat{\Omega}_m^e$ and $\widehat{\Omega}_p^e$ be the images of $\Omega_m^e$ and $\Omega_p^e$
in $z$ coordinates. Then $\widehat{\Omega}_m^e$ and $\widehat{\Omega}_p^e$ are long thin rectangles
in $z$ coordinates with the length in the $z_3$ direction$=\Theta\left(\frac{1}{G_{m,i}^e}\right)$.
We divide these rectangles into $\Theta\left(\frac{1}{G_{m,i}^e}\right)$ rectangles as shown in the
Figure \ref{figD.5}, where each rectangle has length in the $z_3$ direction $=\Theta(1)$.

Proceeding as in Lemma \ref{lem3.4.1} we can prove using the representation (\ref{eqD.50}) that
\begin{align}\label{eqD.51}
&\left|\int_{\partial\tilde{\Gamma}_{m,i}^e}\left(\left(\frac
{\partial u_m^e}{\partial\nb^e}\right)_{A^e}\left(\frac{\partial
u_m^e}{\partial \nuw^e}\right)_{A^e}-\left(\frac{\partial u_p^e}
{\partial\nb^e}\right) _{A^e}\left(\frac{\partial u_p^e}
{\partial\nuw^e}\right)_{A^e}\right)\:d s^e\right| \notag \\
&\leq G_{m,i}^e\left(C_{\eta}(\ln W)^2\sum_{k=1}^3
\left\|[u_{z_k}]\right\|^2_{1/2,\hat{\Gamma}_{m,i}^e}\right. \notag \\
&+\left.\eta\left(\sum_{1\leq|\alpha|\leq 2}\left(\|D_z^{\alpha}
u_m^e\|^2_{0,\hat{\Omega}_m^e}+\|D_z^{\alpha}u_p^e\|^2_{0,\hat
{\Omega}_p^e}\right)\right)\right). 
\end{align}
Here $\widehat{\Gamma}_{m,i}^e$ is the image of $\Gamma_{m,i}^e$ in $z$ coordinates and
$\widehat{\Omega}_m^e$ and $\widehat{\Omega}_p^e$ are the images of $\Omega_m^e$ and $\Omega_p^e$
in $z$ coordinates. Now by a proper choice of $\eta$ in (\ref{eqD.51}) we obtain
(\ref{eqD.44}).
\end{proof}

\subsection*{D.4}
\textbf{Proof of Lemma \ref{lem3.4.4}}
\begin{lem3.4.4}
Let $\Omega_m^e$ and $\Omega_p^e$ be elements in the edge neighbourhood $\Omega^e$ of $\Omega$ and
$\Gamma_{m,i}^e$ be a face of $\Omega_m^e$ and $\Gamma_{p,j}^e$ be a face of $\Omega_p^e$ such that
$\Gamma_{m,i}^e=\Gamma_{p,j}^e$ and $\mu(\tilde{\Gamma}_{m,i}^e)<\infty$. Then for any $\epsilon>0$
there exists a constant $C_{\epsilon}$ such that for $W$ large enough
\begin{subequations}\label{eqD.52}
\begin{align}
&\left|\int_{\tilde{\Gamma}_{m,i}^e}\sum_{l=1}^2\left(\left(
\frac{\partial u_m^e}{\partial \taux_l^e}\right)_{A^e}\frac
{\partial}{\partial s_l^e}\left(\left(\frac{\partial u_m^e}
{\partial\nuw^e}\right)_{A^e}\right)-\left(\frac{\partial u_p^e}
{\partial\taux_l^e}\right)_{A^e}\frac{\partial}{\partial s_l^e}
\left(\left(\frac{\partial u_p^e}{\partial\nuw^e}\right)_{A^e}
\right)\right)\: d \sigma^e\right| \notag \\
&\quad\quad\leq C_{\epsilon}(\ln W)^2\left(\big|\big|\big|
\:[u_{x_1^e}]\:\big|\big|\big|^2_{\tilde{\Gamma}_{m,i}^e}
+\big|\big|\big|\:[u_{x_2^e}]\:\big|\big|\big|^2_{\tilde
{\Gamma}_{m,i}^e}+\big|\big|\big|\:G_{m,i}^e[u_{x_3^e}]
\:\big|\big|\big|^2_{\tilde{\Gamma}_{m,i}^e}\right) \notag\\
&\quad\quad+\epsilon\sum_{k=m,p}\left(\int_{\tilde
{\Omega}_k^e}\left(\sum_{i,j=1,2}\left(\frac{\partial^2 u_k^e}
{\partial x_i^e\partial x_j^e}\right)^2
+e^{2\tau}\sum_{i=1}^2\left(\frac{\partial^2 u_k^e}{\partial x_i^e
\partial x_3^e}\right)^2+e^{4\tau}\left(\frac {\partial^2 u_k^e}
{\left(\partial x_3^e\right)^2}\right)^2\right.\right. \notag\\
&\quad\quad\left.\left.+\sum_{i=1}^2\left(\frac{\partial u_k^e}
{\partial x_i^e}\right)^2+ e^{2\tau}\left(\frac{\partial
u_k^e}{\partial x_3^e}\right)^2\right)\:w^e(x_1^e)\;dx^e\right)\:.
\end{align}
If $\mu(\tilde{\Gamma}_{m,i}^e)=\infty$ then for any
$\epsilon>0$ for $W,N$ large enough
\begin{align}
&\left|\int_{\tilde{\Gamma}_{m,i}^e}\sum_{l=1}^2\left(\left(
\frac{\partial u_m^e}{\partial \taux_l^e}\right)_{A^e}\frac
{\partial}{\partial s_l^e}\left(\left(\frac{\partial u_m^e}
{\partial\nuw^e}\right)_{A^e}\right)-\left(\frac{\partial u_p^e}
{\partial\taux_l^e}\right)_{A^e}\frac{\partial}{\partial s_l^e}
\left(\left(\frac{\partial u_p^e}{\partial\nuw^e}\right)_{A^e}
\right)\right)\: d \sigma^e\right| \notag \\
&\quad\leq\epsilon\left(\int_{\tilde{\Omega}_m^e}(u_m^e)^2
w^e(x_1^e)\:dx^e+\int_{\tilde{\Omega}_m^e}(u_m^e)^2
w^e(x_1^e)\:dx^e\right)
\end{align}
\end{subequations}
provided $W=O(e^{{N}^\alpha})$ with $\alpha<1/2$.
\end{lem3.4.4}
\begin{proof}
Let us first consider the case when
\[\tilde{\Gamma}_{m,i}^e=\left\{x^e:\alpha_0<x_1^e<\alpha_1\:
\beta_0<x_2^e<\beta_1\:x_3^e=\gamma_0\right\}\]
which corresponds to Figure \ref{figD.1}.
\newline
Now using (\ref{eqD.47})
\begin{align}
&\int_{\tilde{\Gamma}_{m,i}^e}\sum_{l=1}^2\left(\left(
\frac{\partial u_m^e}{\partial \taux_l^e}\right)_{A^e}\frac
{\partial}{\partial s_l^e}\left(\left(\frac{\partial u_m^e}
{\partial\nuw^e}\right)_{A^e}\right)\right. \notag \\
&\quad\quad\quad\quad\quad\left. -\left(\frac{\partial u_p^e}
{\partial\taux_l^e}\right)_{A^e}\frac{\partial}{\partial s_l^e}
\left(\left(\frac{\partial u_p^e}{\partial\nuw^e}\right)_{A^e}
\right)\right)\: d \sigma^e \notag 
\end{align}
\begin{align}\label{eqD.53}
&=\int_{\tilde{\Gamma}_{m,i}^e}\sum_{l=1}^2\left(\left(
\frac{\partial u_m^e}{\partial x_l^e}\right)_{A^e}\frac
{\partial}{\partial x_l^e}\left(\left(\frac{\partial u_m^e}
{\partial x_3^e}\right)_{A^e}\right)\right. \notag \\
&\quad\quad\quad\quad\quad\left. -\left(\frac{\partial u_p^e}
{\partial x_l^e}\right)_{A^e}\frac{\partial}{\partial x_l^e}
\left(\left(\frac{\partial u_p^e}{\partial x_3^e}\right)_{A^e}
\right)\right)\: d x_1^e d x_2^e \notag \\
&=G_{m,i}^e\int_{\widehat{\Gamma}_{m,i}^e}\sum_{l=1}^2\left
(\left(\frac{\partial u_m^e}{\partial z_l}\right)_{\widehat{A}^e}
\frac{\partial}{\partial z_l}\left(\left(\frac{\partial u_m^e}
{\partial z_3}\right)_{\widehat{A}^e}\right)\right. \notag \\
&\quad\quad\quad\quad\quad\left. -\left(\frac{\partial u_p^e}
{\partial z_l}\right)_{\widehat{A}^e}\frac{\partial}{\partial
z_l}\left(\left(\frac{\partial u_p^e}{\partial z_3}\right)
_{\widehat{A}^e}\right)\right)\: d z_1 d z_2 \:.
\end{align}

Proceeding as in Lemma \ref{lem3.4.2} we can prove using the representation (\ref{eqD.53})
that for any $\eta>0$ there exists a constant $C_{\eta}$, which depends on $\eta$, but is
uniform for all $\tilde{\Gamma}_{m,i}^e\subseteq \tilde{\Omega}^e$ such that
\begin{align}\label{eqD.54}
&\int_{\tilde{\Gamma}_{m,i}^e}\sum_{l=1}^2\left(\left(
\frac{\partial u_m^e}{\partial \taux_l^e}\right)_{A^e}\frac
{\partial}{\partial s_l^e}\left(\left(\frac{\partial u_m^e}
{\partial\nuw^e}\right)_{A^e}\right)-\left(\frac{\partial u_p^e}
{\partial\taux_l^e}\right)_{A^e}\frac{\partial}{\partial s_l^e}
\left(\left(\frac{\partial u_p^e}{\partial\nuw^e}\right)_{A^e}
\right)\right)\: d \sigma^e \notag \\
&\quad\leq G_{m,i}^e\left(C_{\eta}(\ln W)^2\sum_{k=1}^3
\left\|[u_{z_k}]\right\|^2_{1/2,\hat{\Gamma}_{m,i}^e}\right.
\notag\\
&\hspace{1.8cm}+\left.\eta\left(\sum_{1\leq|\alpha|\leq
2}\left(\|D_z^{\alpha}u_m^e\|^2_{0,\hat{\Omega}_m^e}+
\|D_z^{\alpha}u_p^e\|^2_{0,\hat{\Omega}_p^e}\right)\right)\right).
\end{align}
And by a proper choice of $\eta$ in (\ref{eqD.54}) we obtain (\ref{eqD.52}).
\newline
The next case we consider is when the boundary $\tilde{\Gamma}_{m,i}^e$ corresponds to
\[\tilde{\Gamma}_{m,i}^e=\left\{x^e:\alpha_0<x_1^e<\alpha_1,\:
x_2^e=\beta_0,\:\gamma_0<x_3^e<\gamma_1\right\}\:.\]
The boundary $\tilde{\Gamma}_{m,i}^e$ is shown in Figure \ref{figD.3}. Now
\begin{align*}
&\int_{\tilde{\Gamma}_{m,i}^e}\sum_{l=1}^2\left(\left(
\frac{\partial u_m^e}{\partial \taux_l^e}\right)_{A^e}\frac
{\partial}{\partial s_l^e}\left(\left(\frac{\partial u_m^e}
{\partial\nuw^e}\right)_{A^e}\right)-\left(\frac{\partial u_p^e}
{\partial\taux_l^e}\right)_{A^e}\frac{\partial}{\partial s_l^e}
\left(\left(\frac{\partial u_p^e}{\partial\nuw^e}\right)_{A^e}
\right)\right)\: d \sigma^e\\
&=\int_{\tilde{\Gamma}_{m,i}^e}\left(\left(\frac{\partial
u_m^e}{\partial x_1^e}\right)_{A^e}\frac{\partial}{\partial
x_1^e}\left(\left(\frac{\partial u_m^e}{\partial x_2^e}\right)_
{A^e}\right)-\left(\frac{\partial u_p^e}{\partial x_1^e}\right)
_{A^e}\frac{\partial}{\partial x_1^e}\left(\left(\frac{\partial
u_p^e}{\partial x_2^e}\right)_{A^e}\right)\right)\:dx_1^e dx_3^e\\
&+\int_{\hat{\Gamma}_{m,i}^e}\left(\left(\frac{\partial u_m^e}
{\partial x_3^e}\right)_{A^e}\frac{\partial}{\partial x_3^e}\left
(\left(\frac{\partial u_m^e}{\partial x_2^e}\right)_{A^e}\right)
-\left(\frac{\partial u_p^e}{\partial x_3^e}\right)_{A^e}\frac
{\partial}{\partial x_3^e}\left(\left(\frac{\partial
u_p^e}{\partial x_2^e}\right)_{A^e}\right)\right)\:d x_1^e d x_3^e\:.
\end{align*}
Hence
\begin{align}\label{eqD.55}
&\int_{\tilde{\Gamma}_{m,i}^e}\sum_{l=1}^2\left(\left(
\frac{\partial u_m^e}{\partial \taux_l^e}\right)_{A^e}\frac
{\partial}{\partial s_l^e}\left(\left(\frac{\partial u_m^e}
{\partial\nuw^e}\right)_{A^e}\right)\right. \notag \\
&\quad\quad\quad\quad\quad-\left.\left(\frac{\partial u_p^e}
{\partial\taux_l^e}\right)_{A^e}\frac{\partial}{\partial s_l^e}
\left(\left(\frac{\partial u_p^e}{\partial\nuw^e}\right)_{A^e}
\right)\right)\: d \sigma^e \notag \\
&\quad=G_{m,i}^e\left(\int_{\widehat{\Gamma}_{m,i}^e}\left(\left
(\frac{\partial u_m^e}{\partial z_1}\right)_{\widehat{A}^e}
\frac{\partial}{\partial z_1}\left(\left(\frac{\partial u_m^e}
{\partial z_2}\right)_{\widehat{A}^e}\right)\right.\right. \notag \\
&\quad\quad\quad\quad\quad-\left.\left.\left(\frac{\partial
u_p^e}{\partial z_1}\right)_{\widehat{A}^e}\frac{\partial}{\partial
z_1}\left(\left(\frac{\partial u_p^e}{\partial
z_2}\right)_{\widehat{A}^e}\right)\right)\:d z_1 d z_3\right. \notag \\
&\quad+\left.\int_{\widehat{\Gamma}_{m,i}^e}\left(\left(\frac
{\partial u_m^e}{\partial z_3}\right)_{\widehat{A}^e}\frac
{\partial}{\partial z_3}\left(\left(\frac{\partial u_m^e}
{\partial z_2}\right)_{\widehat{A}^e}\right)\right.\right. \notag \\
&\quad\quad\quad\quad\quad-\left.\left.\left(\frac{\partial
u_p^e}{\partial z_3}\right)_{\widehat{A}^e}\frac{\partial}{\partial z_3}
\left(\left(\frac{\partial u_p^e}{\partial z_2}\right)_{\widehat{A}^e}\right)
\right)\:d z_1 d z_3\right).
\end{align}

Proceeding as in Lemma \ref{lem3.4.2} we can prove using the
representation (\ref{eqD.55}) that for any $\eta>0$ there exists
a constant $C_\eta$, which depends on $\eta$, but is uniform
for all $\tilde{\Gamma}_{m,i}^e \subseteq\tilde{\Omega}^e$ such that
\begin{align}\label{eqD.56}
&\left|\int_{\tilde{\Gamma}_{m,i}^e}\sum_{l=1}^2\left(\left(
\frac{\partial u_m^e}{\partial \taux_l^e}\right)_{A^e}\frac
{\partial}{\partial s_l^e}\left(\left(\frac{\partial u_m^e}
{\partial\nuw^e}\right)_{A^e}\right)-\left(\frac{\partial u_p^e}
{\partial\taux_l^e}\right)_{A^e}\frac{\partial}{\partial s_l^e}
\left(\left(\frac{\partial u_p^e}{\partial\nuw^e}
\right)_{A^e}\right)\right)\: d \sigma^e\right| \notag \\
&\quad\quad\quad\leq G_{m,i}^e\left(C_{\eta}(\ln W)^2\sum_{k=1}^3
\left\|[u_{z_k}]\right\|^2_{1/2,\hat{\Gamma}_{m,i}^e}\right.
\notag \\
&\hspace{2.4cm}+\left.\eta\left(\sum_{1\leq|\alpha|\leq
2}\left(\|D_z^{\alpha}u_m^e\|^2_{0,\hat{\Omega}_m^e}
+\|D_z^{\alpha}u_p^e\|^2_{0,\hat{\Omega}_p^e}\right)
\right)\right)\:.
\end{align}
Now by a proper choice of $\eta$, (\ref{eqD.52}) follows.
\end{proof}

\subsection*{D.5}
\textbf{Proof of Lemma \ref{lem3.5.1}}
\begin{lem3.5.1}
Let $\Gamma_{m,j}^r$ be part of the boundary of the element
$\Omega_m^r$ which lies on the $x_2-x_3$ axis. Define the
contributions from $\Gamma_{m,j}^r$ by
\begin{align}\label{eqD.57}
(BT)_{m,j}^r&=\rho_v^2\sin^2(\phi_v)\left(-\oint_{\partial
\Gamma_{m,j}^r}\left(\frac{\partial u}{\partial \nb} \right)_A
\left(\frac{\partial u}{\partial \nuw}\right)_A\:ds\right. \notag\\
&+\left. 2\int_{\Gamma_{m,j}^r}\sum_{j=1}^2\left(\frac{\partial
u}{\partial \taux_j}\right)_A\frac{\partial}{\partial s_j}\left(
\left(\frac{\partial u}{\partial \nuw}\right)_A\right)\:d\sigma\right)\:.
\end{align}
If Dirichlet boundary conditions are imposed on $\Gamma_{m,j}^r$
then
\begin{align}\label{eqD.58}
|(BT)_{m,j}^r|&\leq C_{\epsilon}(\ln W)^2\|u_m^r\|^2_{3/2,
\Gamma_{m,j}^r}+K_{\epsilon}\sum_{|\alpha|=1}
\|D_x^\alpha u_m^r\|^2_{0,\Omega_m^r} \notag\\
&+\epsilon\sum_{|\alpha|=2} \|D_x^\alpha u_m^r\|^2_{0,\Omega_m^r}\:.
\end{align}
If Neumann boundary conditions are imposed on $\Gamma_{m,j}^r$
then
\begin{align}\label{eqD.59}
\left|(BT)_{m,j}^r\right|\leq C_{\epsilon}(\ln W)^2\left\|\left(
\frac{\partial u _m^r}{\partial \nuw}\right)_A\right\|^2_{1/2,
\Gamma_{m,j}^r}+\epsilon\sum_{1\leq|\alpha|\leq 2}\|D_x^r u_m^r\|
^2_{2,\Omega_m^r}\:.
\end{align}
\end{lem3.5.1}
\begin{proof}
In case the boundary conditions on the face corresponding to the
$x_2-x_3$ axis are Neumann we can show that for any $\epsilon>0$
\begin{align}\label{eqD.60}
\left|(BT)_{m,j}^r\right|\leq C_{\epsilon}(\ln W)^2\left\|\left(
\frac{\partial u _m^r}{\partial \nuw}\right)_A\right\|^2_{1/2,
\Gamma_{m,j}^r}+\epsilon\sum_{1\leq|\alpha|\leq 2}\|D_x^r u_m^r\|
^2_{2,\Omega_m^r}\:.
\end{align}
\begin{figure}[!ht]
\centering
\includegraphics[scale = 0.60]{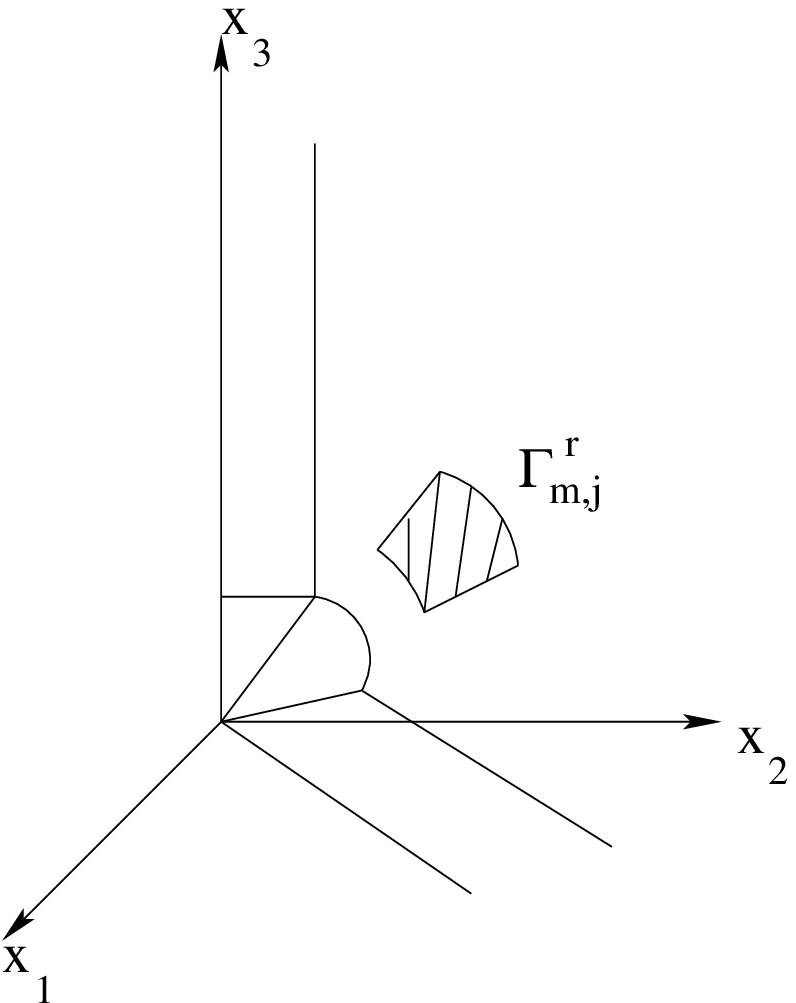}
\caption{The face $\Gamma_{m,j}^r$ of the boundary element $\Omega_m^r$.}
\label{figD.6}
\end{figure}
We now need to examine the case when Dirichlet boundary conditions are imposed
on the face corresponding to the $x_2-x_3$ plane.
\newline
We may choose
\begin{align}\label{eqD.61}
\taux_1 &=(0,0,1)\:, \notag \\
\taux_2 &=(0,1,0)\:, \notag \\
\nuw &=(-1,0,0)\:.
\end{align}
Now
\begin{align}\label{eqD.62}
A\taux_1=\alpha_{11}\taux_1+\alpha_{12}\taux_2+\alpha_{13}A\nuw \notag \\
A\taux_2=\alpha_{21}\taux_1+\alpha_{22}\taux_2+\alpha_{23}A\nuw
\end{align}
Here
\begin{align}\label{eqD.63}
\alpha_{13}=\frac{\nuw^tA\taux_1}{\nuw^tA\nuw},\:\quad
\alpha_{23}=\frac{\nuw^tA\taux_2}{\nuw^tA\nuw}\:,
\end{align}
and
\begin{align}
\alpha_{11}&=\taux_1^tA\taux_1-\frac{(\nuw^tA\taux_1)^2}
{\nuw^tA\nuw},\hspace{1.6cm}\alpha_{12}=\taux_2^tA\taux_1
-\frac{(\nuw^tA\taux_1)(\nuw^tA\taux_2)}{\nuw^tA\nuw}\:,\notag \\
\alpha_{21}&=\taux_1^tA\taux_2-\frac{(\nuw^tA\taux_2)
(\nuw^tA\taux_1)}{\nuw^tA\nuw},\hspace{0.5cm}\alpha_{22}
=\taux_2^tA\taux_2-\frac{(\nuw^tA\taux_2)^2}{\nuw^tA\nuw}\:.\notag
\end{align}
Hence we conclude that
\begin{align}\label{eqD.64}
\left(\frac{\partial u}{\partial\taux_1}\right)_A
=\alpha_{11}\frac{\partial u}{\partial\taux_1}
+\alpha_{12}\frac{\partial u}{\partial\taux_2}
+\alpha_{13}\left(\frac{\partial u}{\partial\nuw}\right)_A \notag\\
\left(\frac{\partial u}{\partial\taux_2}\right)_A
=\alpha_{21}\frac{\partial u}{\partial\taux_1}
+\alpha_{22}\frac{\partial u}{\partial\taux_2}
+\alpha_{23}\left(\frac{\partial u}{\partial\nuw}\right)_A\:.
\end{align}
Thus
\begin{align}\label{eqD.65}
&2\int_{\Gamma_{m,j}^r}\sum_{k=1}^2\left(\frac{\partial u}
{\partial\taux_k}\right)_A\frac{\partial}{\partial s_k}
\left(\frac{\partial u}{\partial \nuw}\right)_A\:d\sigma \notag\\
&=2\int_{\Gamma_{m,j}^r}\left(\sum_{k,l=1}^2
\alpha_{k,l}\left(\frac{\partial u}{\partial \taux_l}
\right)\frac{\partial}{\partial s_k}\left(\frac{\partial u}
{\partial\nuw}\right)_A\right)\:d\sigma \notag\\
&+2\int_{\Gamma_{m,j}^r}\left(\alpha_{13}\left(\frac {\partial
u}{\partial \nuw}\right)_A\frac{\partial} {\partial
s_1}\left(\frac{\partial u}{\partial \nuw}\right)_A
+\alpha_{23}\left(\frac{\partial u}{\partial \nuw}\right)_A
\frac{\partial}{\partial s_2}\left(\frac{\partial u}{\partial
\nuw}\right)_A\right)\:d\sigma\,.
\end{align}
Now
\begin{align}\label{eqD.66}
&2\int_{\Gamma_{m,j}^r}\left(\alpha_{13}\left(\frac{\partial u}
{\partial \nuw}\right)_A\frac{\partial}{\partial s_1}\left(\frac
{\partial u}{\partial \nuw}\right)_A+\alpha_{23}\left(\frac
{\partial u}{\partial \nuw}\right)_{A}\frac{\partial}{\partial s_2}
\left(\frac{\partial u}{\partial \nuw}\right)_A\right)\:d\sigma\notag \\
&=\int_{\Gamma_{m,j}^r}\left(\frac{\partial}{\partial s_1}\left
(\alpha_{13}\left(\frac{\partial u}{\partial \nuw}\right)_A^2
\right)+\frac{\partial}{\partial s_2}\left(\alpha_{23}\left(\frac
{\partial u}{\partial \nuw}\right)_A^2\right)\right)\:d \sigma\notag \\
&-\int_{\Gamma_{m,j}^r}\left(\frac{\partial}{\partial s_1}
(\alpha_{13})+\frac{\partial}{\partial s_2}(\alpha_{23})
\right)\left(\frac{\partial u}{\partial \nuw} \right)_A^2\:d\sigma\,.
\end{align}
Using the divergence theorem we conclude that
\begin{align}\label{eqD.67}
&2\int_{\Gamma_{m,j}^r}\left(\alpha_{13}\left(\frac{\partial
u}{\partial \nuw}\right)_A\frac{\partial}{\partial s_1}
\left(\frac{\partial u}{\partial \nuw}\right)_A+\alpha_{23}
\left(\frac{\partial u}{\partial\nuw}\right)_{A}\frac{\partial}
{\partial s_2}\left(\frac{\partial u}{\partial \nuw}\right)_A
\right)\:d\sigma \notag \\
&=\int_{\partial\Gamma_{m,j}^r}\left\{(\nb\cdot\taux_1)
\alpha_{13}+(\nb\cdot\taux_2)\alpha_{23}\right\}\left
(\frac{\partial u}{\partial\nuw}\right)_A^2\:d s \notag \\
&-\int_{\Gamma_{m,j}^r}\left(\frac{\partial}{\partial s_1}
(\alpha_{13})+\frac{\partial}{\partial s_2}(\alpha_{23})
\right)\left(\frac{\partial u}{\partial \nuw}\right)_A^2 \:d\sigma\,.
\end{align}
Here $\nb$ denotes the unit outward normal to $\partial\Gamma_{m,j}^r$.
\newline
Now
\begin{align*}
\{(\nb\cdot\taux_1)\alpha_{13}+(\nb\cdot\taux_2)\alpha_{23}\}
&=\frac{(\nb\cdot\taux_1)\nuw^tA\taux_1+(\nb\cdot\taux_2)
\nuw^tA\taux_2}{\nuw^tA\nuw}\\
&=\frac{\nuw^tA\{(\nb\cdot\taux_1)\taux_1+(\nb\cdot\taux_2)
\taux_2\}}{\nuw^tA\nuw}=\frac{\nuw^tA\nb}{\nuw^tA\nuw}\:.
\end{align*}
Hence
\begin{align}\label{eqD.68}
\int_{\partial\Gamma_{m,j}^r}\left\{(\nb\cdot\taux_1)\alpha_{13}
+(\nb\cdot\taux_2)\alpha_{23}\right\}\left(\frac{\partial u}
{\partial\nuw}\right)_A^2\:d s=\int_{\partial\Gamma_{m,j}^r}
\frac{\nuw^tA\nb}{\nuw^t A\nuw}\left(\frac{\partial u}{\partial
\nuw}\right)_A^2\:d s\,.
\end{align}
Thus
\begin{align}\label{eqD.69}
&2\rho_v^2\sin^2(\phi_v)\int_{\Gamma_{m,j}^r}\sum_{k=1}^2\left
(\frac{\partial u}{\partial\taux_k}\right)_A\frac{\partial}
{\partial s_k}\left(\frac{\partial u}{\partial \nuw}\right)
_A\:d\sigma \notag\\
&=\rho_v^2\sin^2(\phi_v)\left(\int_{\Gamma_{m,j}^r}\left(
\sum_{k,l=1}^22\alpha_{k,l}\left(\frac{\partial u}{\partial
\taux_l}\right)\frac{\partial}{\partial s_k}\left(\frac {\partial
u}{\partial \nuw}\right)_A\right)\:d\sigma\right. \notag\\
&-\left.\int_{\Gamma_{m,j}^r}\left(\sum_{k=1}^2\frac{\partial}
{\partial s_k}(\alpha_{k3})\right)\left(\frac{\partial u}
{\partial \nuw}\right)_A^2\:d\sigma
+\int_{\partial\Gamma_{m,j}^r}\frac{\nuw^tA\nb}{\nuw^tA\nuw}
\left(\frac{\partial u}{\partial\nuw}\right)_A^2\:d s\right).
\end{align}
We now examine the term
\begin{align}\label{eqD.70}
\rho_v^2\sin^2(\phi_v)\left(-\oint_{\partial\Gamma_{m,j}^r}
\left(\frac{\partial u}{\partial \nb}\right)_A\left(\frac
{\partial u}{\partial \nuw}\right)_A\:d s\right).
\end{align}
Now
\[\left(\frac{\partial u}{\partial \nb}\right)_A=\gamma_{11}
\frac{\partial u}{\partial \taux_1}+\gamma_{12}\frac{\partial
u}{\partial\taux_2}+\gamma_{13}\left(\frac{\partial u}{\partial\nuw}
\right)_A.\]
Here
\begin{align}\label{eqD.71}
\gamma_{13}&=\frac{\nuw^tA\nb}{\nuw^tA\nuw}\:, \notag \\
\gamma_{11}&=\taux_1^tA\nb-\frac{(\nuw^tA\nb)(\nuw^tA\taux_1)}
{\nuw^tA\nuw}\:, \notag \\
\gamma_{12}&=\taux_2^tA\nb-\frac{(\nuw^tA\nb)(\nuw^tA\taux_2)}
{\nuw^tA\nuw}\,.
\end{align}
Hence
\begin{align}\label{eqD.72}
&\rho_v^2\sin^2(\phi_v)\left(-\oint_{\partial\Gamma_{m,j}^r}\left
(\frac{\partial u}{\partial \nb}\right)_A\left(\frac{\partial u}
{\partial \nuw}\right)_A\:d s\right) \notag \\
&=\rho_v^2\sin^2(\phi_v)\left(-\oint_{\partial\Gamma_{m,j}^r}\sum_
{k=1}^2\gamma_{1,k}\frac{\partial u}{\partial \taux_k}\left(
\frac{\partial u}{\partial \nuw}\right)_A\:d s\right. \notag\\
&\hspace{2.2cm}-\left.\oint_{\partial\Gamma_{m,j}^r}
\left(\frac{\nuw^tA\nb} {\nuw^tA\nuw}\right)\left(\frac{\partial
u}{\partial \nuw}\right) _A^2\:d s\right)\,.
\end{align}
Combining (\ref{eqD.69}) and (\ref{eqD.72}) we obtain
\begin{align}\label{eqD.73}
(BT)_{m,j}^r&=\rho_v^2\sin^2(\phi_v)\left(\int_{\Gamma_{m,j}^r}
\left(\sum_{k,l=1}^2 2\alpha_{k,l}\frac{\partial
u}{\partial\taux_l} \frac{\partial}{\partial
s_k}\left(\frac{\partial u}{\partial
\nuw}\right)_A\right)\:d\sigma\right.\\
&-\left.\int_{\Gamma_{m,j}^r}\left(\sum_{k=1}^2\frac{\partial}
{\partial s_k}(\alpha_{k3})\right)\left(\frac{\partial u}
{\partial \nuw}\right)_A^2\:d\sigma-\oint_{\partial\Gamma_{m,j}^r}
\sum_{k=1}^2\gamma_{1,k}\frac{\partial u}{\partial \taux_k}\left
(\frac{\partial u}{\partial \nb}\right)_A\:d s\right).\tag{D.74}
\end{align}
Using (\ref{eqD.73}) we can show that
\begin{align}\label{eqD.74}
\left\|(BT)_{m,j}^r\right\|&\leq C_{\epsilon}(\ln W)^2
\|u_m^r\|^2_{3/2,\Gamma_{m,j}^r}+\epsilon/2\sum_{1\leq|\alpha|
\leq2}\|D_x^\alpha u_m^r\|^2_{0,\Omega_m^r} \notag \\
&+K|u_m^r|^2_{1,\Gamma_{m,j}^r}\:.
\end{align}
Now
\begin{align}\label{eqD.75}
|u_m^r|^2_{1,\Gamma_{m,j}^r}\leq K_{\epsilon}|u_m^r|^2_{1,\Omega_m^r}
+\epsilon/2|u_m^r|^2_{2,\Omega_m^r}.
\end{align}
Hence
\begin{align}\label{eqD.76}
\left\|(BT)_{m,j}^r\right\|&\leq C_{\epsilon}(\ln W)
^2\|u_m^r\|^2_{3/2,\Gamma_{m,j}^r}+K_{\epsilon}\sum_{|\alpha|=1}
\|D_x^\alpha u_m^r\|^2_{0,\Omega_m^r} \notag\\
&+\epsilon\sum_{|\alpha|=2} \|D_x^\alpha u_m^r\|^2_{0,\Omega_m^r}\:.
\end{align}
\end{proof}

\subsection*{D.6}
\textbf{Proof of Lemma \ref{lem3.5.2}}
\begin{lem3.5.2}
Let $\Gamma_{m,j}^e$ be part of the boundary of the element
$\Omega_m^e$ which lies on the $x_2-x_3$ axis. Define the
contributions from $\Gamma_{m,j}^e$ by
\begin{align}\label{eqD.77}
(BT)_{m,j}^e&=-\oint_{\partial\tilde{\Gamma}_{m,j}^e}
\left(\frac{\partial u}{\partial\nb^e}\right)_{A^e}
\left(\frac{\partial u}{\partial \nuw^e}\right)_{A^e}\:d s^e\notag \\
&-2\int_{\tilde{\Gamma}_{m,j}^e}\sum_{l=1}^2\left(\frac{\partial
u}{\partial \taux_l^e}\right)_{A^e}\frac{\partial}{\partial
s_l^e}\left(\left(\frac{\partial u}{\partial
\nuw^e}\right)_{A^e}\right)\:d\sigma^e\:.
\end{align}
If Dirichlet boundary conditions are imposed on
$\Gamma_{m,j}^e$ and $\mu{(\tilde{\Gamma}_{m,j}^e)}<\infty$
then for any $\epsilon>0$ there exists constants\,
$C_\epsilon, K_\epsilon$ such that for $W$ large enough
\begin{align}\label{eqD.78}
|(BT)_{m,j}^e|&\leq C_{\epsilon}(\ln W)^2
\left(\|u_m^e\|_{0,\tilde{\Gamma}_{m,i}^e}^2+\Big|\Big|\Big|
\left(\frac{\partial u_m^e}{\partial x_1^e}\right)
\Big|\Big|\Big|^2_{\tilde{\Gamma}_{m,j}^e}+
\Big|\Big|\Big|G_{m,j}^e\left( \frac{\partial u_m^e}{\partial
x_3^e}\right)\Big|\Big|\Big|^2_{\tilde{\Gamma}_{m,j}^e}\right)\notag\\
&+K_\epsilon\int_{\tilde{\Omega}_m^e}\left(\sum_{i=1}^2
\left(\frac{\partial u_m^e}{\partial x_i^e}\right)^2
+e^{2\tau}\left(\frac{\partial u_m^e}{\partial x_3^e}
\right)^2\right)\:dx^e +\epsilon\int_{\tilde{\Omega}_m^e}\left(\sum_{i,j=1,2}
\left(\frac{\partial^2 u_m^e}{\partial x_i^e\partial x_j^e}
\right)^2 \right.\notag \\
&\left.+e^{2\tau}\sum_{i=1}^2\left(\frac{\partial^2 u_m^e}
{\partial x_i^e\partial x_3^e}\right)^2
+e^{4\tau}\left(\frac{\partial^2 u_m^e}{\left(\partial x_3^e
\right)^2}\right)^2\right)dx^e\:.
\end{align}
If Neumann boundary conditions are imposed on $\Gamma_{m,j}^e$ and
$\mu{(\tilde{\Gamma}_{m,j}^e)}<\infty$ then for any $\epsilon>0$ there exists a
constant $C_\epsilon$ such that
\begin{align}\label{eqD.79}
\left|(BT)_{m,j}^e\right|&\leq C_{\epsilon}(\ln
W)^2\Big|\Big|\Big|\left( \frac{\partial u}{\partial
\nuw^e}\right)_{A^e}\Big|\Big|\Big|^2_{\tilde{
\Gamma}_{m,j}^e}+\epsilon\left(\int_{\tilde{\Omega}_m^e}
\left(\sum_{i,j=1}^2 \left(\frac{\partial^2 u_m^e}{\partial
x_i^e\partial x_j^e}\right)^2\right.\right. \notag \\
&\left.\left.+ e^{2\tau}\sum_{i=1}^2\left(\frac {\partial^2
u_m^e}{\partial x_i^e \partial x_3^e}\right)^2
+e^{4\tau}\left(\frac{\partial^2 u_m^e}{(\partial
x_3^e)^2}\right)^2 + \sum_{i=1}^2\left(\frac{\partial u_m^e}
{\partial x_i^e}\right)^2\right.\right. \notag \\
&\left.\left.+ e^{2\tau}\left(\frac{\partial u_m^e} {\partial
x_3^e}\right)^2\right)\:dx^e\right)\:.
\end{align}
If $\mu{(\tilde{\Gamma}_{m,j}^e)}=\infty$ then for any $\epsilon>0$ for $N,W$
large enough
\begin{align}
\left|(BT)_{m,j}^e\right|&\leq \epsilon \int_{\tilde{\Omega}_m^e}
(u_m^e)^2w^e(x_1^e)\:dx^e \notag
\end{align}
provided $W=O(e^{N^\alpha})$ for $\alpha<1/2$.
\end{lem3.5.2}
\begin{proof}
In the edge neighbourhood the coordinates used are
\begin{align}\label{eqD.80}
x_1^e&=\tau \notag \\
x_2^e&=\theta \notag \\
x_3^e&=x_3\:.
\end{align}
We introduce the variables as in (\ref{eqD.45})
\begin{align}\label{eqD.81}
z_1&=x_1^e \notag \\
z_2&=x_2^e \notag \\
z_3&=\frac{x_3^e}{G_{m,i}^e}\:.
\end{align}
Let us examine the term
\begin{align}\label{eqD.82}
(BT)_{m,j}^e&=-2\int_{\tilde{\Gamma}_{m,j}^e}\sum_{l=1}^2\left(
\frac{\partial u}{\partial \taux_l^e}\right)_{A^e}\frac{\partial}
{\partial s_l^e}\left(\frac{\partial u}{\partial \nuw^e}\right)
_{A^e}\: d\sigma^e \notag \\
&+\oint_{\partial\tilde{\Gamma}_{m,j}^e}\left(\frac{\partial u}
{\partial\nb^e}\right)_{A^e}\left(\frac{\partial u}
{\partial\nuw^e}\right)_{A^e}\: ds^e\,.
\end{align}
If the boundary conditions are Neumann we can prove as in Lemma \ref{lem3.4.4}
that
\begin{align}\label{eqD.83}
|(BT)_{m,j}^e|&\leq C_{\epsilon}(\ln W)^2|||\left(\frac {\partial
u_m^e}{\partial \nuw^e}\right)_{\tilde{A}^e}|||^2
_{\tilde{\Gamma}_{m,j}^e} \notag \\
&+\epsilon\left(\int_{\tilde{\Omega}_m^e}
\sum_{i,j=1,2}\left(\frac{\partial^2 u_m^e}{\partial x_i^e
\partial x_j^e} \right)^2+e^{2\tau}\sum_{i=1}^2\left(\frac
{\partial^2 u_m^e}{\partial x_i^e \partial x_3^e}\right)^2
+e^{4\tau}\left(\frac{\partial^2 u_m^e}{\left(\partial
x_3^e\right)^2}\right)^2\right. \notag \\
&+\left.\sum_{i=1}^2\left(\frac{\partial u_m^e} {\partial
x_i^e}\right)^2+e^{2\tau}\left(\frac{\partial u_m^e} {\partial
x_3^e}\right)^2\right)w^e(x_1^e)\: dx^e\,.
\end{align}
\begin{figure}[!ht]
\centering
\includegraphics[scale = 0.60]{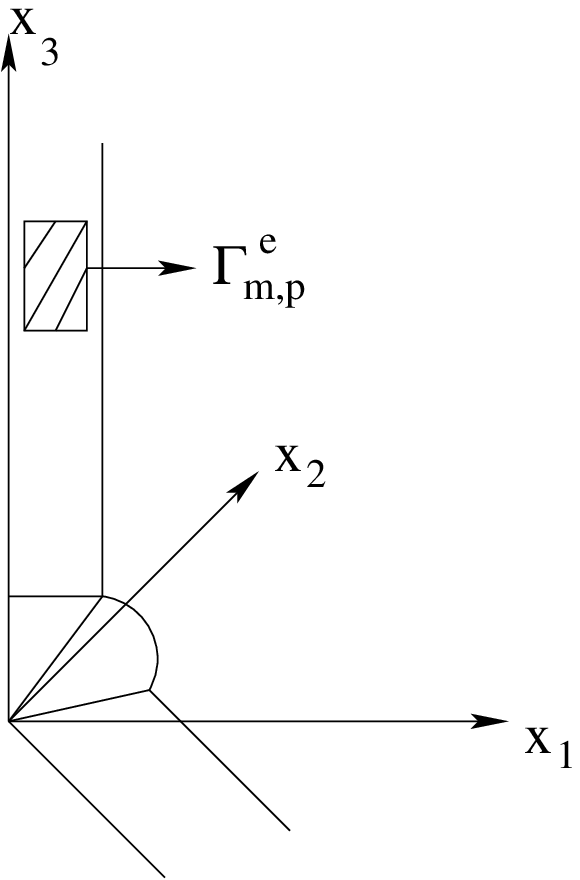}
\caption{The face $\Gamma_{m,p}^e$ of the boundary element $\Omega_m^e$.}
\label{figD.7}
\end{figure}
Let us now examine the case when the boundary conditions are Dirichlet.
\newline
Now choose
\begin{align}\label{eqD.84}
\taux_1^e&=(1,0,0) \notag \\
\taux_2^e&=(0,0,1) \notag \\
\nuw^e   &=(0,-1,0)\,.
\end{align}
Moreover
\begin{subequations}\label{eqD.85}
\begin{align}\label{eqD.85a}
&A^e\taux_1^e=\alpha_{11}^e\taux_1^e+\alpha_{12}^e\taux_2^e+
\alpha_{13}^eA^e\nuw^e\:, \notag\\
&A^e\taux_2^e=\alpha_{21}^e\taux_1^e+\alpha_{22}^e\taux_2^e+
\alpha_{23}^eA^e\nuw^e\:.
\end{align}
Here
\begin{align}\label{eqD.85b}
&\alpha_{13}^e=\frac{(\nuw^e)^tA^e\taux_1^e}{(\nuw^e)^tA^e
\nuw^e},\:\quad
\alpha_{23}^e=\frac{(\nuw^e)^tA^e\taux_2^e}{(\nuw^e)^tA^e
\nuw^e}\:, \notag \\
&\alpha_{11}^e=(\taux_1^e)^tA^e\taux_1^e-\frac{((\nuw^e)^t
A^e\taux_1^e)^2}{(\nuw^e)^tA^e\nuw^e}, \notag\\
&\alpha_{12}^e=(\taux_2^e)^tA^e
\taux_1^e-\frac{((\nuw^e)^tA^e\taux_1^e)((\nuw^e)^tA^e
\taux_2^e)}{(\nuw^e)^tA^e\nuw^e}\:, \notag \\
&\alpha_{21}^e=(\taux_1^e)^tA^e\taux_2^e-\frac{((\nuw^e)^t
A^e\taux_2^e)((\nuw^e)^tA^e\taux_1^e)}{(\nuw^e)^tA^e\nuw^e},
\notag \\
&\alpha_{22}^e=(\taux_2^e)^tA^e\taux_2^e-\frac{((\nuw^e)^tA^e
\taux_2^e)^2}{(\nuw^e)^tA^e\nuw^e}\:.
\end{align}
\end{subequations}
Now
\begin{subequations}\label{eqD.86}
\begin{align}\label{eqD.86a}
\left(\frac{\partial u}{\partial \nb^e}\right)_{A^e}
=\gamma_{11}^e\left(\frac{\partial u}{\partial
\taux_1^e}\right)+\gamma_{12}^e\left( \frac{\partial u}
{\partial\taux_2^e}\right)+\gamma_{13}^e\left(\frac{\partial
u}{\partial \nuw^e}\right)_{A^e}\,.
\end{align}
Here
\begin{align}\label{eqD.86b}
\gamma_{13}^e&=\frac{(\nuw^e)^tA^e\nb^e}{(\nuw^e)^tA^e\nuw^e}\ ,
\notag\\
\gamma_{11}^e&=(\taux_1^e)^tA^e\nb^e-\frac{((\nuw^e)^tA^e\nb^e)
((\nuw^e)^tA^e\taux_1^e)}{(\nuw^e)^tA^e\nuw^e}\ , \notag \\
\gamma_{12}^e&=(\taux_2^e)^tA^e\nb^e-\frac{((\nuw^e)^tA^e\nb^e)
((\nuw^e)^tA^e\taux_2^e)}{(\nuw^e)^tA^e\nuw^e}\,.
\end{align}
\end{subequations}
Then it can be shown that
\begin{align}\label{eqD.87}
(BT)_{m,j}^e&=\left(\int_{\tilde\Gamma_{m,j}^e}
\left(\sum_{k,l=1}^22\alpha_{k,l}^e\frac{\partial
u}{\partial\taux_l^e}\frac{\partial} {\partial
s_k^e}\left(\frac{\partial u}{\partial \nuw^e}\right)_{A^e}
\right)\:d\sigma^e\right. \notag \\
&\qquad-\left.\int_{\tilde\Gamma_{m,j}^e}\left(\sum_{k=1}^2\frac
{\partial}{\partial
s_k^e}(\alpha_{k3}^e)\right)\left(\frac{\partial u}{\partial
\nuw^e}\right)_{A^e}^2\:d\sigma^e \right. \notag \\
&\qquad\left. -\oint_{\partial\tilde\Gamma_{m,j}^e}
\sum_{k=1}^2\gamma_{1,k}^e\frac{\partial u}{\partial
\taux_k^e}\left( \frac{\partial u}{\partial \nb^e}
\right)_{A^e}\:ds^e\right)\,.
\end{align}
Now
\begin{align}\label{eqD.88}
&\alpha_{13}^e=\frac{(\nuw^e)^tA^e\taux_1^e}{(\nuw^e)^t
A^e\nuw^e},\:\quad \alpha_{23}^e=\frac{(\nuw^e)^tA^e\taux_2^e}
{(\nuw^e)^tA^e\nuw^e}\:, \notag \\
&\alpha_{11}^e=(\taux_1^e)^tA^e\taux_1^e-\frac{((\nuw^e)
^tA^e\taux_1^e)^2}{(\nuw^e)^tA^e\nuw^e}, \notag \\
&\alpha_{12}^e=(\taux_2^e)^tA^e\taux_1^e-\frac{((\nuw^e)
^tA^e\taux_1^e)((\nuw^e)^tA^e\taux_2^e)}{(\nuw^e)
^tA^e\nuw^e}\:, \notag \\
&\alpha_{21}^e=\alpha_{12}^e,\quad
\alpha_{22}^e=(\taux_2^e)^tA^e\taux_2^e-\frac{((\nuw^e)^t
A^e\taux_2^e)^2}{(\nuw^e)^tA^e\nuw^e}\:.
\end{align}
Hence
\begin{align}\label{eqD.89}
(\mbox{ I })&=\int_{\tilde\Gamma_{m,j}^e}\left(\sum_{k,l=1}^2
2\alpha_{k,l}^e\frac{\partial u}{\partial\taux_l^e}\frac
{\partial}{\partial s_k^e}\left(\left(\frac{\partial u}
{\partial\nuw^e}\right)_{A^e}\right)\right)\:d\sigma^e \notag\\
&=-\int_{\tilde\Gamma_{m,j}^e}2\left(\alpha_{1,1}^e\frac{\partial
u}{\partial x_1^e}\frac{\partial}{\partial x_1^e}\left(\left
(\frac{\partial u}{\partial x_2^e}\right)_{A^e}\right)
+\alpha_{1,2}^e\frac {\partial u}{\partial x_3^e}
\frac{\partial}{\partial x_1^e}\left(\left(\frac{\partial u}
{\partial x_2^e}\right)_{A^e}\right)\right. \notag \\
&+\left.\alpha_{2,1}^e\frac{\partial u}{\partial x_1^e}\frac
{\partial}{\partial x_3^e}\left(\left(\frac{\partial u}{\partial
x_2^e}\right)_ {A^e}\right)+\alpha_{2,2}^e\frac{\partial
u}{\partial x_3^e}\frac {\partial}{\partial
x_3^e}\left(\left(\frac{\partial u}{\partial x_2^e}
\right)_{A^e}\right)\right)\:dx_1^e\: dx_3^e \notag \\
&=-2G_{m,j}^e\int_{\widehat\Gamma_{m,j}^e} \left(\alpha_{1,1}^e\frac
{\partial u}{\partial z_1}\frac{\partial}{\partial z_1}\left(
\left(\frac{\partial u}{\partial z_2}\right)_{\widehat{A}^e}\right)
+\frac{\alpha_{1,2}^e}{G_{m,j}^e}\frac{\partial u}{\partial
z_3}\frac{\partial} {\partial z_1}\left(\left(\frac{\partial
u}{\partial z_2}\right)_{\widehat{A}^e}\right)\right. \notag \\
&+\left.\frac{\alpha_{2,1}^e}{G_{m,j}^e}\frac{\partial u}{\partial
z_1}\frac{\partial}{\partial z_3}\left(\left(\frac{\partial u}
{\partial z_2}\right)_{\widehat{A}^e} \right)
+\frac{\alpha_{2,2}^e}{(G_{m,j}^e)^2}\frac{\partial u}{\partial
z_3}\frac{\partial}{\partial z_3}\left(\left( \frac{\partial
u}{\partial z_2}\right)_{\widehat{A}^e}\right)\right)\:dz_1\:dz_3\,.
\end{align}
Here $\widehat{\Gamma}_{m,j}^e$ denotes the image of $\Gamma_{m,j}^e$
in $z$ coordinates.
\newline
Now from (\ref{eq3.79}), (\ref{eq3.80}) and (\ref{eqD.88})
$\alpha_{1,1}^e,\:\frac{\alpha_{1,2}^e}{G_{m,j}^e},\:\frac{\alpha_{2,1}^e}{G_{m,j}^e}
\mbox{ and }\frac{\alpha_{2,2}^e}{(G_{m,j}^e)^2}$ are bounded and have bounded
derivatives in $z$ coordinates.
\newline
Therefore
\begin{align}\label{eqD.90}
|(\mbox{ I })|&\leq\left(C_{\epsilon}(\ln W)^2\left(\left\|\frac
{\partial u_m^e}{\partial z_1}\right\|^2_{1/2,\widehat{\Gamma}
_{m,j}^e}+\left\|\frac{\partial u_m^e}{\partial z_3}\right\|^2_
{1/2,\widehat{\Gamma}_{m,j}^e}\right)\right. \notag \\
&+\left.\epsilon\left(\sum_{1\leq|\alpha|\leq2}\|D_z^\alpha u_m^e
\|^2_{0,\widehat{\Omega}_m^e}\right)\right)G_{m,j}^e\,.
\end{align}
Next, we examine
\begin{align}\label{eqD.91}
(\mbox{ II })&=-\int_{\tilde\Gamma_{m,j}^e}\left(\sum_{k=1}^2
\frac{\partial}{\partial s_k^e}(\alpha_{k3}^e)\right)\left(\frac
{\partial u}{\partial\nuw^e}\right)_{A^e}^2\:d\sigma^e \notag \\
&=-G_{m,j}^e\int_{\widehat{\Gamma}_{m,j}^e}\left(\left(\frac
{\partial}{\partial z_1}\alpha_{1,3}^e\right)\left(\frac {\partial
u}{\partial z_2}\right)_{\widehat{A}^e}^2\right. \notag \\
&\hspace{1.5cm}\left.+\frac{1}{G_{m,j}^e}\left(\frac{\partial}
{\partial z_3}\alpha_{2,3}^e\right)\left(\frac{\partial u}
{\partial z_2}\right)_{\widehat{A}^e}^2\right)\: dz_1 dz_3\,.
\end{align}
Now $\alpha_{1,3}^e$ and $\frac{\alpha_{2,3}^e}{G_{m,j}^e}$ are bounded
and have uniformly bounded derivatives in $z$ coordinates.
\newline
Hence
\begin{align}\label{eqD.92}
|(\mbox{ II })|\leq KG_{m,j}^e|u_m^e|^2_{1,\widehat{\Gamma}_m^e}
\leq G_{m,j}^e\left(K_\epsilon|u_m^e|^2_{1,\widehat{\Omega}_m^e}
+\epsilon|u_m^e|^2_{2,\widehat{\Omega}_m^e}\right).
\end{align}
Finally, we examine
\begin{align}\label{eqD.93}
(\mbox{ III })=\oint_{\partial\tilde\Gamma_{m,j}^e}
\sum_{k=1}^2\gamma_{1,k}^e\frac{\partial u}{\partial \taux_k^e}
\left(\frac{\partial u}{\partial \nb^e}\right)_{A^e}\:d s^e\:.
\end{align}
Consider
\[(\mbox{ IIIa })=\int_{\tilde\gamma}\sum_{k=1}^2\gamma_{1,k}^e
\frac{\partial u}{\partial \taux_k^e}\left(\frac{\partial u}
{\partial \nb^e}\right)_{A^e}\:d s^e\:.\]
Now $\nb^e=\taux_2^e$ and so from (\ref{eqD.86b})
\begin{figure}[!ht]
\centering
\includegraphics[scale = 0.60 ]{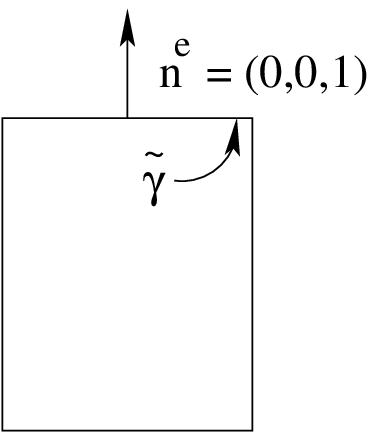}
\caption{The edge $\tilde\gamma$.}
\label{figD.8}
\end{figure}
\begin{align}\label{eqD.94}
\gamma_{11}^e&=\left((\taux_1^e)^t{A}^e\taux_2^e
-\frac{((\nuw^e)^t{A}^e\taux_2^e)((\nuw^e)^t{A}^e
\taux_1^e)}{(\nuw^e)^t{A}^e\nuw^e}\right), \notag\\
\gamma_{12}^e&=\left((\taux_2^e)^t{A}^e\taux_2
^e-\frac{((\nuw^e)^t{A}^e\taux_2^e)((\nuw^e)^tA^e\taux_2
^e)}{(\nuw^e)^t{A}^e\nuw^e}\right).
\end{align}
Hence
\begin{align*}
(\mbox{ IIIa }) &= \int_{\widehat\gamma}\left(\gamma_{11}^e
\left(\frac{\partial u}{\partial z_1}\right)\left(\frac{\partial
u}{\partial z_2}\right)_{\widehat{A}^e} +
\frac{\gamma_{12}^e}{G_{m,j}^e}\left(\frac{\partial u}{\partial
z_3}\right)\left(\frac{\partial
u}{\partial z_2}\right)_{\widehat{A}^e}\right)\: dz_1\\
&= G_{m,j}^e\int_{\widehat\gamma}\left(\frac{\gamma_{11}^e}
{G_{m,j}^e}\left(\frac{\partial u}{\partial
z_1}\right)\left(\frac{\partial u} {\partial
z_2}\right)_{\widehat{A}^e} + \frac{\gamma_{12}^e}{(G_{m,j}
^e)^2}\left(\frac{\partial u}{\partial
z_3}\right)\left(\frac{\partial u}{\partial
z_2}\right)_{\widehat{A}^e}\right)\: dz_1\:.
\end{align*}
Now from (\ref{eq3.79}), (\ref{eq3.80}) and (\ref{eqD.94}),
$\frac{\gamma_{11}^e}{G_{m,j}^e}$ and $\frac{\gamma_{12}^e}{(G_{m,j}^e)^2}$
are bounded and have uniformly bounded derivatives in $z$ coordinates.
\newline
And so we can show that
\begin{align}\label{eqD.95}
|(\mbox{ IIIa })|&\leq G_{m,j}^e\left(C_{\epsilon}(\ln W)^2
\left(\left\|\frac {\partial u_m^e}{\partial z_1}\right\|^2
_{1/2,\widehat{\Gamma}_{m,j}^e}+\left\|\frac{\partial u_m^e}
{\partial z_3}\right\|^2_{1/2,\widehat{\Gamma}_{m,j}^e}
\right)\right. \notag \\
&+\left.\epsilon\sum_{1\leq|\alpha|\leq2}\|D_z^\alpha
u_m^e\|^2_{0,\widehat{\Omega}_m^e}\right).
\end{align}
Next, consider
\[(\mbox{ IIIb })=\int_{\widehat\gamma}\sum_{k=1}^2\gamma_{1,k}^e
\frac{\partial u}{\partial \taux_k^e}\left(\frac{\partial u}
{\partial \nb^e}\right)_{A^e}\:d s^e\:.\]
Now $\nb^e=\taux_1^e$ and so from (\ref{eqD.86b})
\begin{figure}[!ht]
\centering
\includegraphics[ scale = 0.60 ]{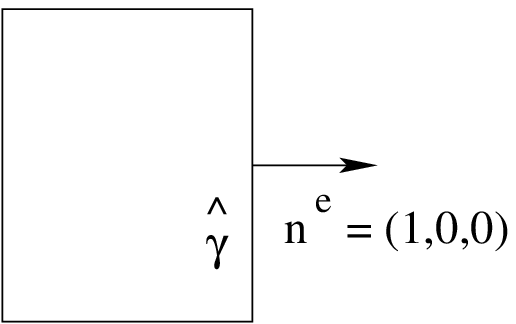}
\caption{The edge $\hat{\gamma}$.}
\label{figD.9}
\end{figure}
\begin{align}\label{eqD.96}
\gamma_{11}^e&=\left((\taux_1^e)^t{A}^e\taux_1^e
-\frac{((\nuw^e)^t{A}^e\taux_1^e)((\nuw^e)^t{A}^e
\taux_1^e)}{(\nuw^e)^t{A}^e\nuw^e}\right), \notag \\
\gamma_{12}^e&=\left((\taux_2^e)^t{A}^e\taux_1^e
-\frac{((\nuw^e)^t{A}^e\taux_1^e)((\nuw^e)^tA^e\taux_2
^e)}{(\nuw^e)^t{A}^e\nuw^e}\right)\,.
\end{align}
Hence
\[(\mbox{ IIIb }) = G_{m,j}^e\int_{\widehat\gamma}
\left(\gamma_{11}^e\left(\frac{\partial u}{\partial
z_1}\right)\left(\frac{\partial u} {\partial
z_1}\right)_{\widehat{A}^e} + \frac{\gamma_{12}^e}{G_{m,j}^e}
\left(\frac{\partial u}{\partial z_3}\right)\left(\frac{\partial
u}{\partial z_1}\right)_{\widehat{A}^e}\right)\: dz_3\,.\]
Now from (\ref{eq3.79}), (\ref{eq3.80}) and (\ref{eqD.96}), $\gamma_{11}^e$ and
$\frac{\gamma_{12}^e}{G_{m,j}^e}$ are bounded and have uniformly bounded derivatives
in $z$ coordinates.
\newline
So we can show that
\begin{align}\label{eqD.97}
|(\mbox{ IIIb })|&\leq G_{m,j}^e\left(C_{\epsilon}(\ln W)^2
\left(\left\|\frac{\partial u_m^e}{\partial z_1}\right\|^2
_{1/2,\widehat{\Gamma}_{m,j}^e}+\left\|\frac{\partial u_m^e}{\partial
z_3}\right\|^2_{1/2,\widehat{\Gamma}_{m,j}^e}\right)\right. \notag \\
&+\left.\epsilon\sum_{1\leq|\alpha|\leq2}\|D_z^\alpha
u_m^e\|^2_{0,\widehat{\Omega}_m^e}\right).
\end{align}
Combining (\ref{eqD.90}), (\ref{eqD.92}), (\ref{eqD.95}) and (\ref{eqD.97}), we
obtain
\begin{align}
|(BT)_{m,j}^e|&\leq G_{m,j}^e\left(C_{\epsilon}(\ln W)^2\left(\left\|\frac
{\partial u_m^e}{\partial z_1}\right\|^2_{1/2,\widehat{\Gamma}_{m,j}^e}
+\left\|\frac{\partial u_m^e}{\partial z_3}\right\|^2_{1/2,\widehat{\Gamma}_{m,j}^e}
\right)\right. \notag\\
&+\left.K_\epsilon\sum_{|\alpha|=1}\|D_z^\alpha u_m^e\|^2_{0,\widehat{\Omega}_m^e}
+\epsilon\sum_{|\alpha|=2}\|D_z^\alpha u_m^e\|^2_{0,\widehat{\Omega}_m^e}\right).
\end{align}
From this the result follows.
\end{proof}
\end{appendix}
\thispagestyle{empty}
\chapcleardoublepage

\addcontentsline{toc}{chapter}{Bibliography}
\renewcommand{\bibname}{Bibliography}
\bibliographystyle{mythesis}
\bibliography{bubble.bib}
\bibliographystyle{plain}

\chapcleardoublepage

\end{document}